\documentclass[12pt]{amsart}

\usepackage{amsmath,amsthm,amssymb,graphicx,color,setspace}
\usepackage[margin=1in]{geometry}

\newcommand{\pwFuk}{\mathcal{F}}

\usepackage{amssymb}
\usepackage{yfonts}

\def\hbar{\bar{h}}

\def\iso{\buildrel \sim\over\to}

\def\GS{{\mathfrak{S}}}

\def\Ggl{{\mathfrak{gl}}}

\def\Gsl{{\mathfrak{sl}}}

\def\CA{{\mathcal{A}}}
\def\CB{{\mathcal{B}}}
\def\CC{{\mathcal{C}}}
\def\CD{{\mathcal{D}}}
\def\CE{{\mathcal{E}}}
\def\CF{{\mathcal{F}}}

\def\CH{{\mathcal{H}}}

\def\CM{{\mathcal{M}}}

\def\CP{{\mathcal{P}}}

\def\CS{{\mathcal{S}}}

\def\CU{{\mathcal{U}}}
\def\CV{{\mathcal{V}}}
\def\CW{{\mathcal{W}}}
\def\CX{{\mathcal{X}}}

\def\CZ{{\mathcal{Z}}}

\def\BC{{\mathbf{C}}}

\def\BF{{\mathbf{F}}}

\def\BR{{\mathbf{R}}}

\def\BZ{{\mathbf{Z}}}
\def\Ba{{\mathbf{a}}}

\def\eps{\varepsilon}

\def\add{\operatorname{add}\nolimits}

\def\can{{\mathrm{can}}}

\def\colim{\operatorname{colim}\nolimits}

\def\cone{\operatorname{cone}\nolimits}

\def\diff{\operatorname{diff}\nolimits}
\def\mdiff{\operatorname{\!-diff}\nolimits}

\def\End{\operatorname{End}\nolimits}

\def\GL{\operatorname{GL}\nolimits}

\def\gr{{\operatorname{gr}\nolimits}}

\def\Hom{\operatorname{Hom}\nolimits}

\def\Id{\operatorname{Id}\nolimits}
\def\id{\operatorname{id}\nolimits}

\def\mMOD{\operatorname{\!-Mod}\nolimits}

\def\opp{{\operatorname{opp}\nolimits}}

\def\pt{\operatorname{pt}\nolimits}

\def\Sets{\operatorname{Sets}\nolimits}

\def\St{\operatorname{St}\nolimits}

\def\Sym{\operatorname{Sym}\nolimits}

\def\ie{{\em i.e.}}

\def\tE{{\tilde{E}}}

\def\CHom{{{\mathcal H}om}}

\newtheorem{thm}{Theorem}[subsection]
\newtheorem{lemma}[thm]{Lemma}
\newtheorem{cor}[thm]{Corollary}
\newtheorem{prop}[thm]{Proposition}
\newtheorem{defi}[thm]{Definition}
\theoremstyle{definition}
\newtheorem{rem}[thm]{Remark}
\newtheorem{example}[thm]{Example}
\newtheorem{convention}[thm]{Convention}
\newtheorem{properties}[thm]{Properties}
\numberwithin{equation}{subsection}

\usepackage{graphicx}
\usepackage{epstopdf}

\usepackage{stmaryrd}
\usepackage{imakeidx}
\usepackage{mathabx}

\input xypic
\xyoption{all}
\def\dm{\operatorname{dm}\nolimits}
\def\supp{\operatorname{supp}\nolimits}

\def\nil{\operatorname{nil}\nolimits}
\def\pt{\operatorname{pt}\nolimits}
\def\dotimes{\protect{\hspace{0.07cm}\otimes\hspace{-0.380cm}\bigcirc}}
\def\CHH{\CH\hspace{-0.20cm}\CH}

\def\bij#1{\chi(#1)}

\makeindex
\makeindex[name=ter,title=Terms,columns=2]
\def\indexnot#1#2{\index{#1@$#2$ |textbf  {\hskip0.5cm} \textbf }}

\begin{document}
\title{Higher representations and cornered Heegaard Floer homology}
\author{Andrew Manion}
\address{A.M.: Department of Mathematics, North Carolina State University, 2108 SAS Hall, Raleigh, NC 27695, USA.}
\email{ajmanion@ncsu.edu}
\author{Rapha\"el Rouquier}
\address{R.R.: UCLA Mathematics Department, Los Angeles, CA 90095-1555, USA.}
\email{rouquier@math.ucla.edu}

\thanks{The first and second author thank the NSF for its support (grant DMS-1702305).
The first author gratefully acknowledges support from the NSF (grant DMS-1502686).
The second author gratefully acknowledges support from the NSF
(grant DMS-1161999) and from the Simons Foundation (grant \#376202).
This material is based upon work supported by the NSF (Grant No. 1440140),
while the second author was in residence at the Mathematical Sciences Research
Institute in Berkeley, California, during the Spring 2018.}

\date\today

\maketitle
\begin{abstract}
We develop the $2$-representation theory of the odd one-dimensional super Lie algebra
	$\Ggl(1|1)^+$
	and show it controls the Heegaard Floer theory of surfaces of Lipshitz,
	Ozsv\'ath and Thurston \cite{LiOzTh1}.
	Our main tool is the construction of a tensor product for $2$-representations.
	We show it corresponds to a gluing operation for surfaces, or for the chord
	diagrams of arc decompositions. This provides an extension of
	Heegaard Floer theory to dimension one, expanding the work of Douglas, Lipshitz and Manolescu
	\cite{DouMa,DouLiMa}.
\end{abstract}

\setcounter{tocdepth}{3}
\tableofcontents



%





\section{Introduction}

\subsection{Higher representations}

While Lie algebra representations and their tensor products have long played an important role in
mathematics, their connection with low-dimensional topology is more recent. This involves quantum
groups, which provide a deformation of the classical Lie theory. Reshetikhin--Turaev's theory
give rise to invariants of links and
$3$-manifolds \cite{ReTu}.

Crane and Frenkel \cite{CrFr} conjectured that there should be a ``higher'' representation theory where vector spaces are replaced by categories, and this would provide invariants of $4$-manifolds. The notion of higher representations was introduced first for type $A$ \cite{ChRou} and then for general
Kac--Moody algebras \cite{Rou1,KhoLau}. In a work in preparation \cite{Rou3}, the second author gives a construction of a tensor product for higher representations of Kac--Moody algebras, in an $\infty$-categorical setting. An important feature is that the category underlying a tensor product of higher representations $\mathcal{V}$ and $\mathcal{V}'$ of $\mathfrak{g}$ depends on the action of the positive part $\mathfrak{g}^+$ of $\mathfrak{g}$ on $\mathcal{V}$ and $\mathcal{V}'$, and not just on the categories $\mathcal{V}$ and $\mathcal{V}'$ themselves.
Evidence for Crane and Frenkel's program has also been provided by the work of Khovanov
\cite{Kho1}, Webster \cite{We} and others.

In this article, we consider the case of the super Lie algebra $\mathfrak{gl}(1|1)$. We do not discuss the notion of higher representations of $\mathfrak{gl}(1|1)$ (cf \cite{Rou3}), but we focus on the positive part $\mathfrak{gl}(1|1)^+=\mathbb{C}e$, a one-dimensional odd super Lie algebra. The notion of a higher representation of $\mathfrak{gl}(1|1)^+$ is due to Khovanov \cite{Kho2}: it is the data of a differential category $\CV$
over $\mathbb{F}_2$ together with a differential endofunctor $E$ and an endomorphism $\tau$ of $E^2$ with $d(\tau)=1$ satisfying $\tau^2 = 0$ and braid relations. So, a higher representation provides
an endofunctor whose square is homotopic to $0$. An equivalent definition is that of an action of
the monoidal category $\CU$ generated by an object $E$ and a map $\tau:E^2\to E^2$ satisfying the
conditions above.

We will allow a more general type of action where $E$ is given by a $(\CV,\CV)$-bimodule.

\subsection{Higher tensor products}

    We define a notion of tensor product $\dotimes$ of higher representations of $\mathfrak{gl}(1|1)^+$,
    endowing the $2$-category of $2$-representations of $\Ggl(1|1)^+$ on differential categories
    with a structure of monoidal $2$-category.
    This does not require working in an $\infty$-categorical setting, as in the Kac--Moody case.

\medskip
    Given $\mathcal{V}_1$ and $\mathcal{V}_2$ two higher representations,
    we construct a higher representation $\mathcal{V}_1\dotimes\mathcal{V}_2$. A
    typical object of this category is a pair
    $(M_1\otimes M_2,\pi)$ where $M_1\otimes M_2$ is an
    object of the ordinary
    tensor product $\mathcal{V}_1\otimes\mathcal{V}_2$ and $\pi$ is a closed map
    $M_1\otimes E_2(M_2)\to E_1(M_1)\otimes M_2$ compatible with $\tau$.
    More generally, one considers objects
    obtained from those by taking cones and direct summands.
    The image by $E$ of the pair above is $(\mathrm{cone}(\pi),\pi')$ for some $\pi'$.

    \smallskip
    This construction generalizes immediately to differential categories endowed with
    two commuting structures of higher representations, but we need a more general
    construction dealing with two lax-commuting higher representations
    to handle general gluings of surfaces. We provide three increasingly subtle versions of such a
    construction.
In general, we obtain a differential category without the (full) structure of
higher representation.

\smallskip
Starting with two structures of higher representations given by endofunctors $E_1$ and
$E_2$ on a differential category
$\mathcal{W}$ and a map $\sigma:E_2E_1\to E_1E_2$ (suitably compatible with $\tau$'s),
we define a differential category
$\Delta_\sigma(\mathcal{W})$ by proceeding as in the tensor product case. It will 
have a structure of higher representation if $\sigma$ is invertible.

\smallskip
The notion of right higher representation coincides with that of (left)
higher representation, but it leads to a different version of the construction above.
We start with the same structures as above, but write $F_1$ instead of $E_1$ and
$\lambda$ instead of $\sigma$. We define a differential category 
$\Delta_\lambda(\mathcal{W})$ with
typical objects pairs $(M,(\upsilon)_{i\ge 1})$ where $M$ is an object of $\mathcal{W}$ and
$\upsilon_i:E_2^iF_2^i(M)\to M$ is a system of compatible maps with respect to
$\lambda$ and $\tau$. In order to define a structure of higher representation,
we need $F_1$ to have a right adjoint $E_1$. Using this adjunction, $\lambda$
gives rise to $\sigma:E_2E_1\to E_1E_2$ and, when $\sigma$ is invertible, we obtain
a structure of higher representation on $\Delta_\lambda(\mathcal{W})$.

\smallskip
Finally, starting with a lax action of $\mathcal{U}\times\mathcal{U}$ on $\mathcal{W}$,
we define a differential category $\Delta_E(\mathcal{W})$.

\smallskip
Our constructions extend an earlier construction of Douglas-Manolescu \cite{DouMa}. They provided a
construction of the category underlying a tensor product.

\smallskip
One of the applications of tensor products in higher representation theory is the construction of complicated categories from simpler ones. This is illustrated below in the reconstruction of partially wrapped Fukaya categories of symmetric powers of surfaces from more basic algebras. We formulate the problem in terms of surfaces with an arc decomposition, and then reformulate it again in terms of certain singular curves. We provide another example, the construction of nil affine Hecke algebras from nil Hecke algebras (in type $A$).








\subsection{Fukaya categories}\label{sec:FukayaCatsIntro}



The main examples of higher representations of $\mathfrak{gl}(1|1)^+$ that we introduce here are on partially wrapped Fukaya categories of symmetric powers of surfaces. These
are $A_\infty$-categories and the results of \S\ref{sec:FukayaCatsIntro} will
be made more precise in \S\ref{se:HFintro}, where we work with differential categories.

Let $\Sigma$ be a compact oriented surface with a finite collection $M$ of marked points in its boundary, and assume that each component of $F$ contains at least one point of $M$.

In the pictures below, we expand each point of $M$ to an interval in the boundary of $\Sigma$. We draw
the complement of these intervals in dotted light orange.
Here are two views of a genus-one surface with one boundary component and $M$ consisting of one point.

\begin{center}
\includegraphics[scale=0.75]{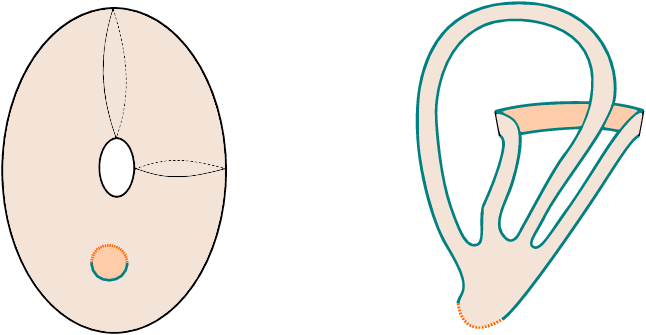}
\end{center}

For any $k \geq 0$, Auroux \cite[Section 3.1]{Au2} considers a partially wrapped Fukaya category $\pwFuk(\mathrm{Sym}^k(\Sigma),M)$ of $\mathrm{Sym}^k(\Sigma)$ with set of stops
$M\times \mathrm{Sym}^{k-1}(\Sigma)$.
We write $\pwFuk(\mathrm{Sym}^*(\Sigma),M)$ for the direct sum of these categories over all $k \geq 0$ (they vanish for $k$ large enough).

Given a component $I$ of $\partial \Sigma \setminus M$, we
define a higher action of $\mathfrak{gl}(1|1)^+$ on $\pwFuk(\mathrm{Sym}^*(\Sigma),M)$.





Consider $(\Sigma_1,M_1,I_1)$ and $(\Sigma_2,M_2,I_2)$ two surfaces with chosen intervals as above.
We form a new surface $(\Sigma,M)$ with a chosen interval $I$ by gluing $I_1$ and $I_2$ to the two legs of an open pair of pants.

\begin{thm}\label{thm:IntroMainPairing}
    There is an equivalence of triangulated categories
    \[
    \pwFuk(\mathrm{Sym}^*(\Sigma),M) \simeq \pwFuk(\mathrm{Sym}^*(\Sigma_1),M_1) \dotimes \pwFuk(\mathrm{Sym}^*(\Sigma_2),M_2)
    \]
    compatible with the structure of higher representations of $\mathfrak{gl}(1|1)^+$.
\end{thm}

Theorem~\ref{thm:IntroMainPairing} extends a result of Douglas--Manolescu \cite{DouMa}; their theorem corresponds to the special case of Theorem~\ref{thm:IntroMainPairing} in which $\Sigma_1$ and $\Sigma_2$ have only one boundary circle and one marked point each. They prove an equivalence of categories without
the statement on compatibility of higher actions.

More generally, given $(\Sigma,M)$ with two disjoint chosen intervals $I_1, I_2$ in $\partial \Sigma \setminus M$, we form a new surface $(\overline{\Sigma}, \overline{M})$ by gluing the two legs of an open pair of pants to $I_1$ and $I_2$. Theorem~\ref{thm:IntroMainPairing} generalizes to say that $\mathcal{F}(\mathrm{Sym}^*\overline{\Sigma},\overline{M})$ is equivalent to $\Delta \mathcal{F}(\mathrm{Sym}^*\Sigma,M)$ with its diagonal action.

This makes it possible to recover $\pwFuk(\mathrm{Sym}^*(\Sigma),M)$ for any $(\Sigma,M)$ from
the case of a disk with two points in the boundary.
As an illustration, consider the genus-one surface $(\Sigma,M)$ shown above; the partially wrapped Fukaya category $\pwFuk(\mathrm{Sym}^*(\Sigma),M)$ is described by the ``torus algebra'', a standard example in bordered Heegaard Floer homology. We can cut $(\Sigma,M)$ along three arcs as shown; we are left with two rectangles. Gluing the cuts back together, the torus algebra can then be recovered from one application of the tensor product followed by two applications of the more general $\Delta$ construction. 

\begin{center}
\includegraphics[scale=0.75]{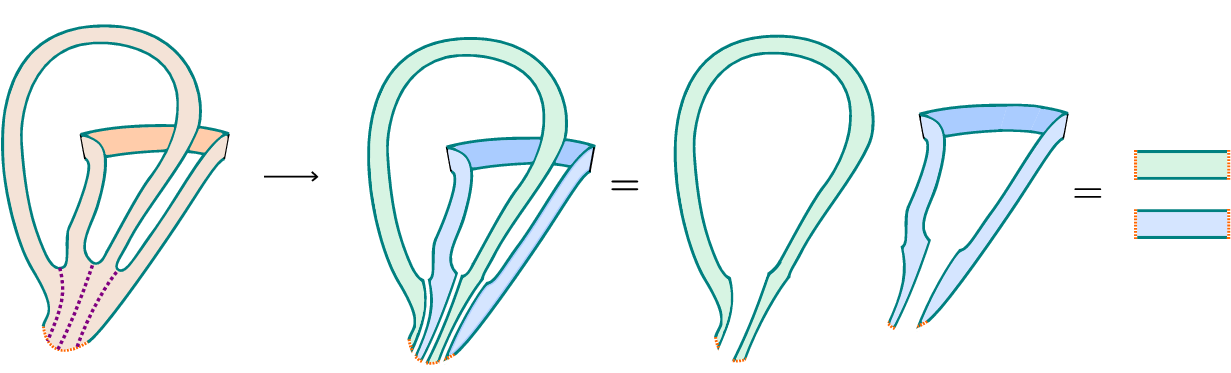}
\end{center}

\begin{center}
\includegraphics{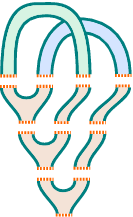}
\end{center}

In the case of the first symmetric power $\Sym^1(\Sigma)=\Sigma$, 
a general construction of Fukaya categories by a gluing procedure
is given by Haiden, Katzarkov and Kontsevich in \cite{HaiKaKon}.

A general theory of partially wrapped Fukaya categories and how they glue is provided by Ganatra,
Pardon and Shende in \cite{GaPaShe}, but this doesn't apply directly to our case.

\subsection{Heegaard Floer homology}\label{se:HFintro}





Heegaard Floer homology, defined by Ozsv{\'a}th--Szab{\'o} \cite{OsSz1,OsSz2,OsSz3},
is a set of invariants for 3- and 4-dimensional manifolds. For a 3-manifold $Y$, the Heegaard Floer invariant of $Y$ (an abelian group) is defined by choosing a Heegaard decomposition of $Y$ as two handlebodies glued along a genus-$g$ surface $\mathcal{H}$, then computing a Lagrangian intersection Floer homology group between two Lagrangian submanifolds in $\mathrm{Sym}^g(\mathcal{H})$ induced by the two handlebodies.

In bordered Heegaard Floer homology \cite{LiOzTh1, Za}, there are also extended Heegaard Floer invariants for 2d surfaces and 3d cobordisms. Let $F$ be the data of a surface $(\Sigma,M)$ with a set of
points $M\subset\partial \Sigma$ as above, equipped with a choice of arc decomposition. To
such a surface, bordered Heegaard Floer associates a differential algebra $A(F)$.
Auroux \cite{Au2} has shown that the algebra $A(F)$ is the endomorphism algebra of a generating object
of $\mathcal{F}(\mathrm{Sym}^*(\Sigma),M)$ determined
by the arc decomposition. 





    
   
   \smallskip
    Our constructions are based on the combinatorics of $A(F)$ and we do not work directly with
    $\mathcal{F}(\mathrm{Sym}^*(\Sigma),M)$. Given a component of
    $\partial \Sigma \setminus M$, we define a differential bimodule $E$ over $A(F)$, and a bimodule endomorphism $\tau$ of $E \otimes_{A(F)} E$, that yield a higher representation of
    $\mathfrak{gl}(1|1)^+$.




    Theorem \ref{thm:IntroMainPairing} follows now from the following result.

\begin{thm}
	\label{th:introglueAF}
    If $F_1$ and $F_2$ are surfaces with arc decompositions glued as in
	Theorem~\ref{thm:IntroMainPairing} to form $F$, then
    \[
    A(F) \cong A(F_1) \dotimes A(F_2)
    \]
    as higher representations of $\mathfrak{gl}(1|1)^+$.
\end{thm}




\subsection{Singular curves}

Let $Z$ be a singular oriented curve. We associate to $Z$ a differential algebra $A(Z)=\bigoplus_{i\ge 0}A_k(Z)$. 

The algebra $A_k(Z)$ has a basis given by ``braids": these are pairs $(I,([\zeta_i])_{i\in I})$, where $I$ is a set of $k$ singular points of $Z$ and $[\zeta_i]$ is a homotopy class of smooth oriented paths starting at $i$ and ending at a singular point. We require that the end points of $\zeta_i$ and $\zeta_j$ are distinct if $i\neq j$. 

We define $d(I,([\zeta_i]))$ to be the sum over intersection points between paths $\zeta_i$ of the 
braid obtained by resolving the intersection point.

The composition $(I',[\zeta'_{i'}])\circ (I,[\zeta_i])$ is $0$ or $(I,[\zeta'_{\zeta_i(1)}\circ\zeta_i])$ if a number of conditions are satisfied:
\begin{itemize}
    \item $\{\zeta_i(1)\}=I'$
    \item the paths $\zeta'_{\zeta_i(1)}\circ\zeta_i$ are smooth
    \item there are no representatives in the homotopy class of the concatenated paths with fewer intersections than $\{\zeta'_{\zeta_i(1)}\circ\zeta_i\}_{i\in I}$.
\end{itemize}

\smallskip
    \includegraphics{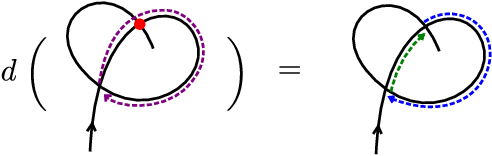}

A singular curve $Z$ with a worst ordinary double points gives rise to a sutured surface $F(Z)$ with an arc decomposition (cf the case of a torus below) and we have $A(F(Z))=A(Z)$: the algebras $A(Z)$ generalize
those of \cite{LiOzTh1, Za}.

    \includegraphics[width=0.7\linewidth]{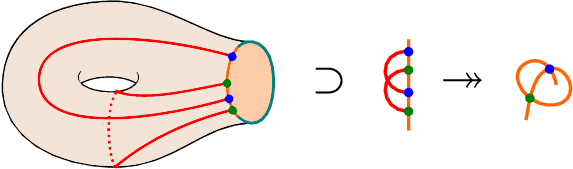}

Given a closed embedding of $(0,1]$ (resp. $[-1,0)$) in $Z$ avoiding singular points, we construct a higher representation of $\mathfrak{gl}(1|1)^+$ on $A(Z)$. The bimodule $E$ has a basis given by braids where one path starts at $1$ (resp. ends at $-1$).

Given embeddings of $[-1,0)$ and $(0,1]$ in $Z$ as above, we can construct a singular curve $\bar{Z}$ by attaching $[-1,1]$ to $Z$ along $[-1,0)\cup(0,1]$.
Our results on algebras associated to the gluing of surfaces is a consequence of the more general result below on singular curves.

\begin{thm}
    There is an isomorphism of higher representations $A(\bar{Z})\iso \Delta A(Z)$.
\end{thm}

 
A version of this result allowing partially oriented singular curves contains as a special case the construction of nil affine Hecke algebras from nil Hecke algebras.

The reconstruction of the partially wrapped Fukaya categories for the torus depicted in
\S \ref{sec:FukayaCatsIntro} corresponding to the following decomposition of the corresponding singular
curve:

\smallskip

    \includegraphics{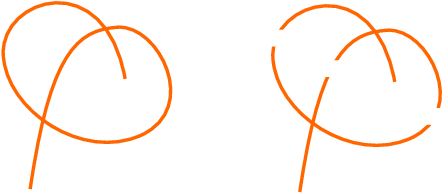}

\subsection{Extended TQFT and further remarks}

We expect that our constructions (in particular
Theorem \ref{thm:IntroMainPairing}) will be part of a $4$-dimensional TQFT.

This article is the first step towards a higher representation-theoretic reconstruction of
Heegaard Floer theory, which would be a fulfilment of the program of
Crane and Frenkel for $\Ggl(1|1)$. This article focuses on dimensions $1$ and $2$, where homotopic
phenomena can be avoided.

It is natural to ask if our constructions extend to dimension $0$.  Work in progress of the first author and Reeshad Arian  \cite{ArMa} shows this is possible at the decategorified level.

\smallskip
We are pursuing two directions for the extension to dimension $3$.

In a work in preparation, the first author extends the tensor product functoriality
to the $A_\infty$-setting. This is a key step to construct morphisms of $2$-representations from
Heegaard Floer diagrams by cutting them into pieces.

Work in preparation of the second author \cite{Rou4}
provides a construction of invariants of links in $S^3$. Appropriate $t$-structures are used
to handle the homotopic phenomena. 

Note that Ellis, Pevtkova and V\'ertesi in
 \cite{ElPeVe} construct homological invariants of tangles in the setting of bordered Heegaard Floer
theory in which $\Ggl(1|1)$-categorifications appear.
Note also that \cite{DouLiMa} considered $3$-manifolds with codimension-$2$ corners in the setting of
\cite{DouMa}.



\smallskip

We expect that our algebraic constructions will provide a blueprint for the construction of higher categorical structures and $3$ and $4$-dimensional invariants associated to (ordinary) simple Lie algebras.









\subsection{Structure of the article}
We gather in \S\ref{se:notations} a number of basic definitions and facts involving differential
categories and bimodules. Most differential vector spaces we encounter come with
bases, and we formalize this aspect in the notion of ``differential pointed sets"
and corresponding differential pointed categories.

We consider Hecke algebras in \S\ref{se:Hecke}.
We study in \S\ref{se:diffnilHecke} the differential
algebra structure on nil Hecke algebras of Coxeter groups over a field of characteristic $2$
and we describe adjunctions
for induction and restriction functors, in the case of finite Coxeter groups. An important
fact is that those Hecke algebras are the graded algebras associated with the filtration
of the group algebra with respect to the length function.
The remainder
of \S\ref{se:Hecke} is devoted to the case of symmetric groups and their affine versions.
We introduce in \S\ref{se:positive} positive submonoids of the affine symmetric groups and we
provide a description by generators and relations of their nil Hecke algebras.

Section \S\ref{se:2reptheory} is devoted to the development of the $2$-representation theory
of $\Ggl(1|1)^+$. We introduce the monoidal category $\CU$.
Our main construction is that of a tensor
product operation on $2$-representations, and more generally, of a diagonal
action given two (lax) commuting $2$-representation structures. We also consider a more complicated
``dual" construction in \S\ref{se:dualdiagonal}.
In \S\ref{se:bimodule2rep},
we recast
our functorial constructions into bimodule constructions. We formulate our constructions
in the differential ungraded setting.

In \S\ref{se:Hecke2rep}, we construct bimodules and $2$-representations associated with nil Hecke
algebras.
In \S\ref{se:regular2rep}, we describe explicitly the structures of $2$-representation coming
from the left and the right action of the monoidal category $\CU$ on itself and we show that
the diagonal category arising from these commuting left and right actions corresponds
to Hecke algebras of positive affine symmetric groups. It is a remarkable fact that those
can be recovered from the Hecke algebras of the ordinary symmetric groups.
We introduce
in \S\ref{se:nilHeckecat} a categorical version of affine symmetric groups and their Hecke
algebras.

We develop in \S\ref{se:strandalgebras} an extension of Lipshitz-Ozsv{\'a}th-Thurston 
\cite{LiOzTh1} and Zarev's \cite{Za} theory of
strand algebras associated with matched circles and intervals. Instead of considering
curves with matchings, we consider the corresponding quotient spaces, where the matched
points are identified. We start in \S\ref{se:1dimensionalspaces}
with $1$-dimensional spaces, which we define
as complements of a finite set of points in a $1$-dimensional finite CW-complex. In
\S\ref{se:curves}, we define our objects of interest, the singular curves. They
are $1$-dimensional spaces together with an additional structure at singular points, and
a partially defined orientation.
They arise as quotients of smooth curves, or, equivalently, as curves in $\BR^n$
with transverse intersections of branches. This leads to a notion of admissible paths,
those paths that lift to a smooth model for the curve (\S\ref{se:curvepaths}). We
introduce in \S\ref{se:strands} the differential categories of strands associated to a curve. They are
defined as graded categories associated with a filtered category, in a way similar
to the constructions of \S\ref{se:diffnilHecke}. We show in \S\ref{se:StrandsS1} that strand
categories
on unoriented $S^1$ correspond to the categories built from nil Hecke algebras of affine
symmetric groups.

The final section \S\ref{se:2repstrand} shows that the strand category of a glued curve is
obtained
as a tensor (or more general diagonal) construction from the strand category of the
original curve. This
provides some sort of $1$-dimensional field theory, which is really part of
a $2$-dimensional field theory for surfaces with extra structure.
This gives a categorical mechanism by which strand categories can be computed by 
cutting the curve into basic building blocks. We start in \S\ref{se:leftaction} by
constructing a structure of $2$-representation associated with an unoriented ``end" of a curve.
We describe in \S\ref{se:gluing} how the strand categories behave under the gluing of two ends
of a curve. This requires to solve a combinatorial generators and relation problem
generalizing Proposition \ref{le:generatorsrelationsHnhat+}.
When the gluing operation does not create an $S^1$, we show in
\S\ref{se:actiontensor} that the 
resulting $2$-representation is the one obtained from the diagonal action.

\subsection{Acknowledgments}

We thank Ciprian Manolescu for several useful conversations.

\section{Differential and pointed structures}
\label{se:notations}

\subsection{Differential algebras and categories}
\subsubsection{Categories}
\label{se:cat}
Let $\CC$ be a category. We denote by $\CC^\opp$ the opposite category. We identify $\CC$
with a full subcategory of
$\Hom(\CC^\opp,\mathrm{Sets})$ via the Yoneda embedding $c\mapsto \Hom(-,c)$.

Given $(L,R)$ a pair of adjoint functors, we denote the unit of the adjunction by
$\eta_{L,R}$\indexnot{eta}{\eta_{L,R}} and the counit by $\eps_{L,R}$\indexnot{eps}{\eps_{L,R}}.

When $\CC$ is enriched in abelian groups,
we denote by $\add(\CC)$ the smallest full subcategory of $\Hom(\CC^\opp,\mathrm{Sets})$
containing $\CC$ and closed under finite coproducts and isomorphisms.

\smallskip
Let $\CX$ be a $2$-category. We denote by $\CX^\opp$ the $2$-category with same objects
and $\CHom(x,y)=\CHom(x,y)^\opp$.
We denote by $\CX^{\mathrm{rev}}$ the $2$-category with the same
objects and with $\CHom(x,y)=\CHom(y,x)$ for $x$ and $y$ two objects of $\CX$ (so that the
composition of $1$-arrows is reversed).

\smallskip
Let $\CC{at}$ be the $2$-category of categories. There is an equivalence
$\CC{at}\iso\CC{at}^\opp$ sending a category $\CC$ to $\CC^\opp$.

Let $\CC{at}^r$ (resp. $\CC{at}^l$)
be the $2$-full $2$-subcategory of $\CC{at}$ with $1$-arrows
those functors that admit a left (resp. right) adjoint. There is an equivalence of
$2$-categories $\CC{at}^r\iso (\CC{at}^l)^{\mathrm{rev }\opp}$. It is the identity on objects and
sends a functor to a left adjoint.

\subsubsection{Differential categories}
Let $k$\indexnot{k}{k} be a field of characteristic $2$. We write
$\otimes$ for $\otimes_k$.

A {\em differential module}\index[ter]{differential module} is a $k$-vector space
$M$ endowed with an endomorphism
$d$ satisfying $d^2=0$. We put $Z(M)=\ker d$. An element $m$ of $M$ is said
to be {\em closed}\index[ter]{closed element} when $d(m)=0$.
We define $\Hom$-spaces in the category $k\mdiff$ of differential modules by
$\Hom_{k\mdiff}(M,M')=\Hom_{k\mMOD}(M,M')$. That $k$-module has a differential given by
$\Hom(d_M,M')+\Hom(M,d_{M'})$.  We define the category $Z(k\mdiff)$ as the subcategory of
$k\mdiff$ with same objects as $k\mdiff$ and $\Hom_{Z(k\mdiff)}(M,M')=Z(\Hom_{k\mdiff}(M,M'))$.

The tensor product of vector spaces and the permutation
of factors equip $k\mdiff$ and $Z(k\mdiff)$ with a structure of symmetric monoidal category.

A {\em differential category}\index[ter]{differential category} is a category enriched over 
$Z(k\mdiff)$.

\smallskip
Let $\CV$ and $\CV'$ be two differential categories.
We denote by $\Hom(\CV,\CV')$ the differential category
of ($k$-linear) differential functors $\CV\to\CV'$. Its $\Hom$ spaces are
$k$-linear natural transformations.

We denote
by $\CV\otimes\CV'$ the differential category with
set of objects $\mathrm{Obj}(\CV)\times\mathrm{Obj}(\CV')$ and with
$\Hom_{\CV\otimes\CV'}((v_1,v'_1),(v_2,v'_2))=\Hom_{\CV}(v_1,v_2)\otimes
\Hom_{\CV'}(v'_1,v'_2)$.

\medskip
We denote by $\CV\mdiff=\Hom(\CV,k\mdiff)$\indexnot{diff}{\CV-\mathrm{diff}} the category
of {\em $\CV$-modules}\index[ter]{$\CV$-modules}.
There is a fully faithful embedding $v\mapsto \Hom_\CV(-,v):\CV\to\CV^\opp\mdiff$
and we identify $\CV$ with its image.

Note that $\mathrm{add}(\CV)$\indexnot{add}{\mathrm{add}(\CV)} identifies
with the smallest full subcategory of $\CV^\opp\mdiff$ containing $\CV$ and closed under
finite direct sums and isomorphisms.

\smallskip
There is a differential functor $\otimes_\CV:\CV^\opp\mdiff\otimes\CV\mdiff\to k\mdiff$.
Given $M\in\CV^\opp\mdiff$ and $N\in\CV\mdiff$, there is an exact sequence of
differential $k$-modules
$$\bigoplus_{f\in\Hom_\CV(v_1,v_2)}M(v_2)\otimes N(v_1)
\xrightarrow{\substack{a\otimes b
\mapsto M(f)(a)\otimes b\\-a\otimes N(f)(b)}}
\bigoplus_{v\in\CV}M(v)\otimes N(v)\to M\otimes_\CV N\to 0.$$

Given $v\in\CV$, we have 
$\Hom(-,v)\otimes_\CV N=N(v)$ and $M\otimes_{\CV}\Hom(v,-)=M(v)$.

\medskip
Recall that a category is {\em idempotent complete}\index[ter]{idempotent complete category}
if all idempotent maps have images.

We denote by $\CV^i$\indexnot{i}{\CV^i} the {\em idempotent completion}\index[ter]{idempotent completion}
of $\CV$: this
is the smallest full subcategory of $\CV^\opp\mdiff$ containing $\CV$
and closed under direct summands and isomorphisms.
The $2$-functor $\CV\mapsto\CV^i$ is left adjoint to the embedding of
idempotent-complete differential categories in differential categories.

\subsubsection{Objects}
\label{se:objects}
Given $v_1,v_2$ two
objects of $\CV$ and given $f\in Z\Hom_\CV(v_1,v_2)$, the {\em cone}\index[ter]{cone} of $f$
is the
object $\mathrm{cone}(\Hom_{\CV}(-,f))$
of $\CV^\opp\mdiff$ denoted by ${\xy (0,0)*{v_1\oplus v_2},
\ar@/^/^{f}(-4,2)*{};(4,2)*{}, \endxy}$.
We say that $\CV$ is {\em strongly pretriangulated}\index[ter]{strongly pretriangulated
category} if the cone of
any map of $\CV$ is isomorphic to an object of $\CV$.
Note that $\CV^\opp\mdiff$ is strongly pretriangulated.

We denote by $\bar{\CV}$\indexnot{V}{\bar{\CV}}
the smallest full strongly pretriangulated subcategory of
$\CV^\opp\mdiff$ closed under taking isomorphic objects and containing $\CV$.
 Note that
$(\bar{\CV})^i$ is strongly pretriangulated. Note also that if $\CV$ is a full
subcategory of a strongly pretriangulated $\CV'$, then $\CV$ is strongly pretriangulated if
the cone in $\CV'$ of a map between objects of $\CV$ is isomorphic to
an object of $\CV$.

\smallskip
Let $v_1,\ldots,v_n$ be objects of $\CV$ and $f_{ij}\in\Hom_{\CV}(v_j,v_i)$
for $i<j$. Assume $d(f_{ij})=\sum_{i<r<j}f_{ir}\circ f_{rj}$ for all $i<j$.
We define the {\em twisted object}\index[ter]{twisted object}
$[v_n\oplus\cdots\oplus v_1,\left(\begin{matrix}
	0  \\
	f_{n-1,n}& \ddots \\
	\vdots &\ddots& 0 \\
	f_{1,n}&\hdots& f_{1,2} & 0
\end{matrix}\right)]$
of $\bar{\CV}$ inductively on $n$ as the cone
of 
$$(f_{n-1,n},\ldots,f_{1,n}):v_n\to
[v_{n-1}\oplus\cdots\oplus v_1,\left(\begin{matrix}
	0  \\
	f_{n-2,n-1}& \ddots \\
	\vdots &\ddots& 0 \\
	f_{1,n-1}&\hdots& f_{1,2} & 0
\end{matrix}\right)].$$

The objects of $\bar{\CV}$ are the objects of $\CV^\opp\mdiff$ isomorphic to
a twisted object of $\CV$.

\smallskip
If $\CV'$ is strongly pretriangulated, then the restriction functor
$\Hom(\bar{\CV},\CV')\to\Hom(\CV,\CV')$ is an equivalence. So,
$\CV\mapsto\bar{\CV}$ is left adjoint to the embedding of
strongly pretriangulated differential categories in differential categories.

\subsubsection{Algebras}
\label{se:Algebras}
Let $A$ be a differential algebra. We denote by $A\mdiff$\indexnot{diff}{A-\mathrm{diff}} the category
of (left) differential $A$-modules. Note that $\Hom_{A\mdiff}(M,M')$ is the differential
$k$-module of $A$-linear maps $M\to M'$. This is an idempotent-complete
strongly pretriangulated differential category. We say
that a differential $A$-module is {\em strictly perfect}\index[ter]{strictly perfect module}
if it is in $(\bar{A})^i$, where $A$ denotes the full subcategory of $A\mdiff$
with a unique object $A$.

\medskip

A differential category $\CC$ with one object $c$ is the
same as the data of a differential algebra $A=\End_\CC(c)$.
When $\CC$ has a unique object $c$ and $A=\End_\CC(c)$, then there is an isomorphism
$A\mdiff\iso\CC\mdiff,\ M\mapsto (c\mapsto M)$.

More generally, a differential category $\CC$ can be viewed as a
``differential algebra with several objects". More 
precisely, there is an equivalence from the category of 
differential categories $\CC$ with finitely many objects (arrows are differential functors)
to the category of differential algebras $A$ equipped with a finite set $I$ of orthogonal
idempotents with
sum $1$ (arrows $(A,I)\to (A',I')$ are non-unital morphisms of differential algebras $f:A\to A'$
such that $f(I)\subset I'$):
\begin{itemize}
	\item to $\CC$, we associate $A=\bigoplus_{c,c'\in\CC}\Hom_\CC(c,c')$ and $I$
		the set of projectors on objects of $\CC$;
	\item to $(A,I)$, 
we associate the differential category $\CC$ with set of objects $I$ and
$\Hom_\CC(e,f)=fAe$.
\end{itemize}

\subsubsection{$G$-graded differential structures}

We define a {\em $\BZ$-monoid}\index[ter]{$\BZ$-monoid} $G$ to be a
monoid $G$ endowed with an action of the group $\BZ$, denoted by
$g\mapsto g+n$ for
$g\in G$ and $n\in\BZ$, and such that $(g+n)(g'+n')=gg'+n+n'$. Note that
$e_G+\BZ$ is a central submonoid of $G$, where $e_G$ denotes the unit of $G$.
So, the data above is equivalent to the data of a morphism of monoids
$\BZ\to Z(G)$. This is itself determined by the image of $1$, a central invertible
element $\upsilon$ of $G$.

\smallskip
We define a {\em differential $G$-graded $k$-module}\index[ter]{differential $G$-graded
structure} to be a $G$-graded $k$-module $M$ together with a differential module structure
such that $d(M_g)\subset M_{g+1}$ (cf \cite[\S 2.5]{LiOzTh1}).

Given $g\in G$, we define $M\langle g\rangle$
\indexnot{Mg}{M\protect\langle g\protect\rangle, \protect\langle g\protect\rangle M}
to be the differential $G$-graded $k$-module
given by $(M\langle g\rangle)_h=M_{hg}$. Similarly, we define
$\langle g\rangle M$ by $(\langle g\rangle M)_h=M_{gh}$.

\smallskip
We define similarly the notion of {\em differential $G$-graded algebra},
of {\em differential $G$-graded category}, etc.

When $G=\BZ$ and $\upsilon=1$, we recover the usual notion of differential graded $k$-module,
etc.

\smallskip
Let $G_1$ and $G_2$ be two $\BZ$-monoids. We define
$G_1\times_{\BZ}G_2$ as the quotient of $G_1\times G_2$ by the
equivalence relation $(g_1,g_2+n)\sim (g_1+n,g_2)$ for $g_1,g_2\in G$ and $n\in\BZ$. Denote by 
$p:G_1\times G_2\to G_1\times_{\BZ}G_2$ the quotient map, a morphism of monoids.
There is a structure of $\BZ$-monoid on $G_1\times_{\BZ}G_2$ given by
$p(g_1,g_2)+1=p(g_1+1,g_2)=p(g_1,g_2+1)$.

Let $M_i$ be a differential $G_i$-graded $k$-module for $i\in\{1,2\}$.
We define a structure of differential $(G_1\times_{\BZ}G_2)$-module on
the differential module $M_1\otimes M_2$ by setting
$(M_1\otimes M_2)_g=\bigoplus_{(g_1,g_2)\in p^{-1}(g)} (M_1)_{g_1}\otimes (M_2)_{g_2}$.

\subsection{Bimodules}
\subsubsection{Algebras}
\label{se:diffalg}
Let $\mathrm{Alg}$ be the $2$-category with objects
the differential algebras, and $\Hom_{\mathrm{Alg}}(A,A')$ the
category of $(A',A)$-bimodules. The composition of $1$-arrows
is the tensor product of differential bimodules.

\medskip
Given $M$ an $(A',A)$-bimodule, we put
$M^\vee=\Hom_{A^\opp}(M,A)$\indexnot{Mv}{M^\vee}, an $(A,A')$-bimodule.

There is a morphism of $(A',A)$-bimodules
$$M\to \Hom_A(M^\vee,A),\ m\mapsto (\zeta\mapsto \zeta(m)).$$
It is an isomorphism if $M$ is finitely generated and projective as a
(non-differential) $A^\opp$-module.

There is a morphism of functors
$$\Hom_A(M^\vee,A)\otimes_A -\to \Hom_A(M^\vee,-),\
f\otimes r\mapsto (\zeta\mapsto f(\zeta)r).$$
It is an isomorphism if $M^\vee$ is finitely generated and projective
as a (non-differential) $A$-module.

Combining those two morphisms, we obtain a morphism of functors
$$M\otimes_A -\to \Hom_A(M^\vee,-)$$
that is an isomorphism  if $M$ is finitely generated and projective as a
(non-differential) $A^\opp$-module.
So, when this holds, we have an adjoint pair
$(M^\vee\otimes_{A'}-,M\otimes_A-)$,
with corresponding unit $\eta:A'\to M\otimes_A M^\vee$ and counit
$\eps:M^\vee\otimes_{A'}M\to A$. In other terms, the bimodule $M^\vee$ is a left dual of
$M$\index[ter]{left dual bimodule}.

\smallskip
Note conversely that given $M$ such that $(M^\vee\otimes_{A'}-,M\otimes_A-)$ is an adjoint
pair, then $M^\vee$ is a finitely generated projective $A$-module because
$\Hom_A(M^\vee,-)$ is exact and commutes with direct sums, hence $M\simeq
\Hom_A(M^\vee,A)$ is finitely generated and projective as an $A^\opp$-module.

We say that $M$ is {\em right finite} \index[ter]{right finite bimodule}
when it is finitely generated and projective as an
$A^\opp$-module. We say that $M$ is {\em left finite}  \index[ter]{left finite bimodule}
when it is finitely generated and projective as an $A'$-module.

\smallskip
Consider the $2$-full subcategory $\mathrm{Alg}^r$ (resp. $\mathrm{Alg}^l$)
of $\mathrm{Alg}$ with same objects
and $1$-arrows the right (resp. left) finite bimodules.
There is an equivalence of $2$-categories $\mathrm{Alg}^r\iso (\mathrm{Alg}^l)^{\mathrm{rev }\opp}$.
It is the identity on objects and sends a bimodule $M$ to $M^\vee$.

\subsubsection{Categories}

Let $\CC$ and $\CC'$ be differential categories.
A {\em $(\CC,\CC')$-bimodule} \index[ter]{$(\CC,\CC')$-bimodule} is a differential functor
$\CC\otimes\CC^{\prime\opp}\to k\mdiff$.
There is a $2$-category $\mathrm{Bimod}$\indexnot{Bimod}{\mathrm{Bimod}}
of differential categories and bimodules.
Its objects are differential categories and
$\CHom_{\mathrm{Bimod}}(\CC,\CC')$ is the differential category of
$(\CC',\CC)$-bimodules. Composition is given by tensor product: given $\CC''$ a 
differential category, $M$ a $(\CC,\CC')$-bimodule and $N$ a $(\CC',\CC'')$-bimodule,
we put
$$\bigl(M\otimes_{\CC'}N\bigr)(c,c'')=M(c,-)\otimes_{\CC'}N(-,c'').$$

There is an equivalence of $2$-categories
$\mathrm{Bimod}\iso\mathrm{Bimod}^{\mathrm{rev}}$ sending 
a differential category $\CC$ to $\CC^\opp$ and a $(\CC,\CC')$-bimodule to the
same functor, viewed as a $(\CC^{\prime\opp},\CC^\opp)$-bimodule.

\medskip
The bimodule $\Hom:\CC\otimes\CC^\opp\to k\mdiff,\ (c_1,c_2)\mapsto
\Hom_\CC(c_2,c_1)$ is an identity for the tensor product.
The canonical isomorphism of
$(\CC,\CC)$-bimodules $\Hom\otimes_\CC\Hom\iso\Hom$ is given by
$$\Hom_\CC(-,c_1)\otimes_\CC \Hom_\CC(c_2,-)\to \Hom(c_2,c_1),\
((f:d\to c_1)\otimes (g:c_2\to d)\mapsto f\circ g.$$

\medskip
Let $M$ be a $(\CC',\CC)$-bimodule. We define the $(\CC,\CC')$-bimodule
$M^\vee$ by 
$$M^\vee(c,c')=\Hom_{\CC^{\opp}\mdiff}(M(c',-),\Hom_{\CC}(-,c)).$$

There is a morphism of $(\CC,\CC)$-bimodules
$\eps_M:M^\vee\otimes_{\CC'}M\to \Hom$ given by
\begin{align*}
	\eps_M(c_1,c_2):M^\vee(c_1,-)\otimes_{\CC'}M(-,c_2)&\to \Hom(c_2,c_1)\\
	(M(c',-)\xrightarrow{f}\Hom(-,c_1))\otimes m&\mapsto f(c_2)(m)
	\text{ for }m\in M(c',c_2).
\end{align*}

Given $L\in\CC\mdiff$ and $L'\in\CC'\mdiff$, we have a morphism functorial in $L$ and $L'$
$$\Hom(L',M\otimes_{\CC}L)\xrightarrow{M^\vee\otimes -}
\Hom(M^\vee\otimes_{\CC'}L',M^\vee\otimes_{\CC'}M\otimes_{\CC}L)
\xrightarrow{\Hom(M^\vee\otimes_{\CC'}L',\eps_M)}
\Hom(M^\vee\otimes_{\CC'}L',L).$$

\smallskip
We say that $M$ is {\em right finite}\index[ter]{right finite bimodule}
if the morphism above is an isomorphism for all $L$ and $L'$. When this holds, the functor
$M^\vee\otimes_{\CC'}-$ is left adjoint to $M\otimes_{\CC}-$ and $M^\vee$ is
left dual to $M$ \index[ter]{left dual bimodule}. We also write
${^\vee N}=M$ where $N=M^\vee$.
We say that $M$ is {\em left finite}\index[ter]{left finite bimodule} if
it is a right finite $(\CC^{\prime\opp},\CC^\opp)$-bimodule.

\medskip
Let $M$ be a $(\CC,\CC)$-bimodule. We define the differential category 
$T_{\CC}(M)$\indexnot{TCM}{T_{\CC}(M)}. Its objects are
those of $\CC$ and
$$\Hom_{T_{\CC}(M)}(c_1,c_2)=\bigoplus_{i\ge 0}M^i(c_1,c_2).$$

\subsubsection{Bimodules and functors}
\label{se:bimodfunctors}
There is a $2$-functor from $\mathrm{Alg}$ to $\mathrm{Bimod}$: it sends $A$ to
the differential category $\CC_A$ with one object $c_A$ and $\End(c_A)=A$. It sends
an $(A',A)$-bimodule $M$ to the $(\CC_{A'},\CC_A)$-bimodule $\CC_M$ given by
$\CC_M(c_A,c_{A'})=M$. This $2$-functor provides isomorphisms of categories
$\Hom_{\mathrm{Alg}}(A,A')\iso\Hom_{\mathrm{Bimod}}(\CC_A,\CC_{A'})$.

\smallskip
There is a $2$-fully faithful $2$-functor from the $2$-category of differential
categories to $\mathrm{Bimod}^{\mathrm{rev}}$: it sends $\CC$ to $\CC$ and
$F:\CC\to\CC'$ to
the $(\CC,\CC')$-bimodule $(c,c')\mapsto\Hom(c',F(c))$.

There is a $2$-fully faithful $2$-functor from
$\mathrm{Bimod}$ to the $2$-category of differential categories:
it sends $\CC$ to $\CC\mdiff$ and $M$ a $(\CC',\CC)$-bimodule to $M\otimes_{\CC}-:\CC\mdiff
\to\CC'\mdiff$.

\smallskip
Composing the $2$-functor $\mathrm{Alg}\to\mathrm{Bimod}$ and the $2$-functor from
$\mathrm{Bimod}$ to the $2$-category of differential categories, we obtain
a differential $2$-functor from $\mathrm{Alg}$ to
the $2$-category of differential categories: it sends 
$A$ to $A\mdiff$ and
it sends an $(A',A)$-bimodule $M$ to the functor $M\otimes_A-:
A\mdiff\to A'\mdiff$. Note that this $2$-functor is $2$-fully faithful.

\subsection{Pointed sets and categories}
\subsubsection{Pointed sets}
\label{se:pointedsets}
A {\em pointed set} is a set with a distinguished element $0$. The category
$\Sets^\bullet$\indexnot{sets}{\Sets^\bullet} of pointed sets has objects pointed sets and
arrows those maps that
preserve the distinguished element. 

It has coproducts: $\bigvee S_i$ is the quotient of $\coprod S_i$ by
the relation identifying the $0$-objects of the $S_i$'s.

We define $\bigwedge S_i$ as the quotient of
$\prod S_i$ by the relation identifying an element with $(0)_i$ if
one of its components is $0$.
There is a canonical isomorphism $S\wedge \{0,\ast\}\iso S$.
This provides the category of pointed sets with a structure of symmetric monoidal
category (the tensor product of $S_1$ and $S_2$ is $S_1\wedge S_2$)
and there is a symmetric monoidal functor from the category of sets
to the category of pointed sets $E\mapsto E_+=E\sqcup\{0\}$.

\smallskip
Given $S$ a pointed set and $k$ a commutative ring, we denote by
$k[S]$ the quotient of the free $k$-module with basis $S$ by the $k$-submodule
generated by the distinguished element of $S$. This gives a coproduct preserving
monoidal functor from the category of pointed sets to the category of $k$-modules.

\smallskip
Assume $k$ is finite. Let $S$ and $S'$ be two pointed sets. We say that a $k$-linear
map $f:k[S]\to k[S']$ is {\em bounded} \index[ter]{bounded map} if
there is $N>0$ such that for all $s\in S$, the set of elements of $S'$ that have
a non-zero coefficient in $f(s)$ has fewer than $N$ elements.

The functor
$k[-]$ induces a bijection from $k[\Hom_{\Sets^\bullet}(S,S')]$ to
the subspace of bounded maps in $\Hom_{k\mMOD}(k[S],k[S'])$.

\subsubsection{Gradings and filtrations}
Let $G$ be a set.
A $G$-graded pointed set is a pointed set $S$ together with pointed subsets
$S_g$ for $g\in G$
such that $S=\bigcup_{g\in G} S_g$ and $S_g\cap S_h=\{0\}$ for $g\neq h$.

Given a map $f:G\to G'$ and $S$ a $G$-graded pointed set, we define a structure
of $G'$-graded pointed set on $S$ by setting $S_{g'}=\{0\}\cup\bigcup_{g\in f^{-1}(g')}S_g$.

Given $G_1$ and $G_2$ two sets and $S_i$ a $G_i$-graded pointed set for $i\in\{1,2\}$,
then $S_1\wedge S_2$ is a $(G_1\times G_2)$-graded pointed set with
$(S_1\wedge S_2)_{(g_1,g_2)}=(S_1)_{g_1}\wedge (S_2)_{g_2}$.

\smallskip
Assume $G$ is a monoid. Given two $G$-graded pointed sets $S$ and $T$, there
is a structure of $(G\times G)$-graded pointed set on $S\wedge T$. Via the multiplication
map, we obtain a structure of $G$-graded pointed set on $S\wedge T$. This
makes the category of $G$-graded pointed sets into a monoidal category with 
unit object the pointed set $S=\{0,\ast\}$ with $S_1=S$ and $S_g=\{0\}$ for $g\neq 1$.

\medskip
Let $G$ be a poset.
A $G$-filtered set (resp. pointed set) is a set (resp. a pointed set)
$S$ together with subsets (resp. pointed subsets)
$S_{\ge g}$ for $g\in G$ such that $S_{\ge g}\subset S_{\ge g'}$ if $g>g'$ and
such that given $s\in S$ (resp. $s\in S\setminus\{0\}$),
the set $\{g\in G\ |\ s\in S_{\ge g}\}$ is non-empty and has a maximal element,
which we denote by $\deg(s)$.

Note that a structure of $G$-filtered set on a set (resp. a pointed set)
$S$ is the same as the data of a map $S\to G$ (resp. a map $S\setminus\{0\}\to G$).

\smallskip
The associated $G$-graded pointed set is
$\gr S=\{0\}\sqcup S$ (resp. $\gr S=S$) with 
$$(\gr S)_g=\{0\}\sqcup \{s\in S\ |\ \deg(s)=g\}\ 
(\text{resp. }
(\gr S)_g=\{s\in S\setminus\{0\}\ |\ \deg(s)=g\}).$$

If $G$ is a (partially) ordered monoid, then the category of $G$-filtered sets
(resp. pointed sets)
is a monoidal category with $(S\wedge T)_{\ge g}$ the image of $\coprod_{g_1,g_2\in G,
g_1g_2\ge g}(S_{\ge g_1}\times T_{\ge g_2})$ in $S\wedge T$. Its unit object
is the set $S=\{\ast\}$ (resp. the pointed set $S=\{0,\ast\}$)
with $S_{\ge g}=S$ if $1\ge g$ 
and $S_{\ge g}=\emptyset$ (resp. $S_{\ge g}=\{0\}$) otherwise.

\smallskip
There is a monoidal functor
$S\mapsto \gr S$ from the monoidal category of
$G$-filtered sets (resp. pointed sets) to the monoidal category of $G$-graded
pointed sets. Given $f:S\to T$ a map between $G$-filtered sets (resp.
pointed sets), the map $\gr f:\gr S\to \gr T$ is given for $s\in (\gr S)_g$ by
$(\gr f)(s)=f(s)$ if $f(s)\in (\gr T)_g$ and
$(\gr f)(s)=0$ otherwise.

Note also that given a commutative ring $k$
there is a monoidal functor $S\mapsto k[S]$ from the category of $G$-graded pointed sets
to the category of $G$-graded $k$-modules.

\subsubsection{Pointed categories}
\label{se:pointedcat}
A {\em pointed category} is a category enriched in pointed sets. We define similarly
$G$-graded pointed categories, etc. The monoidal functors $\CV_1\to\CV_2$ defined above
provide a construction from a category enriched in $\CV_1$ of a category
enriched in $\CV_2$. Let us describe this more explicitly.

\smallskip
$\bullet\ $Given a $G$-filtered category (or a $G$-filtered pointed category) $\CC$, we
have a $G$-graded pointed category $\gr\CC$. Its objects are the same as those of 
$\CC$ and $\Hom_{\gr\CC}(c,c')=\gr\Hom_{\CC}(c,c')$.

\smallskip
$\bullet\ $Given a pointed category $\CC$, we denote by $k[\CC]$ the associated $k$-linear category:
its objects are those of $\CC$ and $\Hom_{k[\CC]}(c,c')=k[\Hom_\CC(c,c')]$.
If $\CC$ is a $G$-graded pointed category, then $k[\CC]$ is a $k$-linear $G$-graded category.

\smallskip
$\bullet\ $Given a category $\CC$, the associated pointed category $\CC_+$\indexnot{C+}{\CC_+} has the same
objects as $\CC$ and $\Hom_{\CC_+}(c,c')=\Hom_{\CC}(c,c')\sqcup\{0\}$.

\smallskip
Consider a family $\{\CC_i\}$ of pointed categories. We have a pointed category
$\bigwedge\CC_i$ with object set $\prod \mathrm{Obj}(\CC_i)$ and 
$\Hom_{\bigwedge\CC_i}((c_i),(c'_i))=\bigwedge \Hom_{\CC_i}(c_i,c'_i)$.
Similarly, we have a pointed category
$\bigvee\CC_i$ with object set $\coprod \mathrm{Obj}(\CC_i)$ and 
given $c\in\CC_r$ and $c'\in\CC_s$, we have
$$\Hom_{\bigwedge\CC_i}(c,c')=\begin{cases}
	\Hom_{\CC_r}(c,c') & \text{ if }r=s\\
	\{0\} & \text{ otherwise.}
\end{cases}$$

\smallskip
Note that the data of a structure of $G$-filtered pointed category on a pointed category
$\CC$ is the same as the data of a map $\deg$ from the set of non-zero maps
of $\CC$ to $G$ such that $\deg(\beta\circ\alpha)\ge\deg(\beta)\deg(\alpha)$ for any
two composable maps $\alpha$ and $\beta$ such that $\beta\circ\alpha\neq 0$.

Given a $G$-filtered pointed category $\CC$ with degree function $\deg$
and given a morphism of (partially) ordered 
monoids $f:G\to H$, we obtain a structure of $H$-filtered pointed category on $\CC$
with degree function $f\circ\deg$.

\medskip
Note that the category $\Sets^\bullet$ has a structure of pointed category: the
distinguished map between two pointed sets is the map with image $0$.

\subsubsection{Differential pointed categories}
We define a {\em differential pointed set}\index[ter]{differential pointed set}
to be a pointed set $S$ together with
a bounded endomorphism $d$ of $\BF_2[S]$ satisfying $d^2=0$.

\smallskip
Given $S$ and $S'$ two differential pointed sets, then $S\vee S'$ and
$S\wedge S'$ have structures of differential pointed sets coming from the canonical
isomorphisms $\BF_2[S\vee S']\iso \BF_2[S]\oplus \BF_2[S']$ and $\BF_2[S\wedge S']\iso
\BF_2[S]\otimes \BF_2[S']$. 

\smallskip
We define the category $\mathrm{diff}$\indexnot{diff}{\mathrm{diff}}
of differential pointed sets:  its objects are differential pointed sets and
maps the maps of pointed sets. 
There is a functor $\BF_2[-]:\mathrm{diff}\to \BF_2\mdiff$.
Let $S$ and $S'$ be two differential pointed sets. Because the differentials on $\BF_2[S]$
and $\BF_2[S']$ are bounded, the vector space
$\BF_2[\Hom_{\Sets^\bullet}(S,S')]$ identifies with a subspace of
$\Hom_{\BF_2\mMOD}(\BF_2[S],\BF_2[S'])$ that is stable under the
differential $\Hom(d_{\BF_2[S]},-)+\Hom(-,d_{\BF_2[S']})$.

\smallskip
We define $Z(\mathrm{diff})$ as the subcategory of
$\mathrm{diff}$ with same objects as $\mathrm{diff}$ and with
$\Hom_{Z(\mathrm{diff})}(S,S')$ the subset of maps in the kernel of $d$ (where
we view $\Hom_{\mathrm{diff}}(S,S')$ inside $\Hom_{\BF_2\mMOD}(\BF_2[S],\BF_2[S'])$).
The categories
$\mathrm{diff}$ and $Z(\mathrm{diff})$ have a structure of symmetric monoidal category
coming from those on pointed sets and differential modules.

\smallskip
We define a {\em differential pointed category}\index[ter]{differential pointed category}
to be a category enriched in $Z(\mathrm{diff})$.
This is the same as a pointed category $\CV$ together
with a differential on $\BF_2[\CV]$ endowing it with a structure of differential
category.
The $2$-functor $\CV\mapsto\BF_2[\CV]$ from the $2$-category of
differential pointed categories to the $2$-category of
differential categories is $2$-faithful and $2$-conservative.

\smallskip
Note that the category $\mathrm{diff}$ is a differential pointed category:

\smallskip
All our constructions below for differential pointed categories
are compatible with the corresponding constructions for differential categories,
via the $2$-functor $\BF_2[?]$.

\smallskip
Given $G$ a $\BZ$-monoid, we will also consider
{\em differential $G$-graded pointed sets}\index[ter]{differential $G$-graded pointed
strctures}: these are differential pointed
sets $S$
with a structure of $G$-graded pointed set such that $d(S_g)]\subset\BF_2[S_{g+1}]$ for
$g\in G$.
We have a corresponding notion of {\em differential $G$-graded pointed category}.

\medskip
Let $\CV$ be a differential pointed category. 
We say that a map of $\CV$ is {\em closed} if its image in $\BF_2[\CV]$ is closed.
Given $f:S\to S'$ a closed map of differential pointed sets, we define the
cone $\cone(f)$ of $f$ as the pointed set $S\vee S'$ with differential on
$\BF_2[S\vee S']=\BF_2[S]\oplus \BF_2[S']$ given by 
$\left(\begin{matrix}d_{\BF_2[S]} & 0 \\f & d_{\BF_2[S']}\end{matrix}\right)$.

\smallskip
We define a $\CV$-module to be a differential pointed functor (\ie, a functor
enriched in $Z(\mathrm{diff})$) $\CV\to \mathrm{diff}$.
We denote by $\CV\mdiff$ the category of $\CV$-modules.

Given $f:v_1\to v_2$ a closed map in $\CV$, we define 
$\cone(f)=\cone(\Hom_{\CV}(f,-))\in\CV\mdiff$.

\medskip
Let $M$ be a $\CV^\opp$-module and $N$ a $\CV$-module. We define the differential pointed
set $M\wedge_\CV N$ as the coequalizer of
$$\xymatrix{
\bigvee_{f\in\Hom_{\CV}(v_1,v_2)}(M(v_2)\wedge N(v_1))
\ar@<0.5ex>[rrr]^-{a\wedge b\mapsto M(f)(a)\wedge b}
\ar@<-0.5ex>[rrr]_-{a\wedge b\mapsto a\wedge N(f)(b)}
&&&
\bigvee_{v\in\CV}(M(v)\wedge N(v)).
}$$

Given $\CV'$ a differential pointed category, we define a
$(\CV,\CV')$-bimodule to be a differential pointed functor
$\CV\bigwedge\CV^{\prime\opp}\to\mathrm{diff}$.

Given $\CV''$ a differential pointed category, $N$ a $(\CV,\CV')$-bimodule and
$M$ a $(\CV',\CV'')$-bimodule, then $N\wedge_{\CV'}M$ is a $(\CV,\CV'')$-bimodule.
This gives rise to a $2$-category $\mathrm{Bimod}^\bullet$\indexnot{Bimod}{\mathrm{Bimod}^\bullet}
of differential pointed categories and bimodules, with
a $2$-fully faithful functor to the $2$-category of differential pointed categories and
a $2$-faithful functor $\BF_2[-]$ to the $2$-category $\mathrm{Bimod}$.


\medskip
Let $M$ be a $(\CV,\CV)$-bimodule. We define a differential pointed category 
$T_{\CV}(M)$\indexnot{TVM}{T_{\CV}(M)}. Its objects are those of $\CV$ and
$$\Hom_{T_{\CV}(M)}(v_1,v_2)=\bigvee_{i\ge 0}
M^i(v_1,v_2).$$

\subsubsection{Pointed structures as $\BF_2$-structures with a basis}

Let us reformulate the definitions of the previous sections in terms of $\BF_2$-vector
spaces with a basis.

The functor $\BF_2[-]$ gives an equivalence from the category of pointed sets to
the category with objects $\BF_2$-vector spaces with a basis and where maps
are $\BF_2$-linear maps sending a basis element to a basis element or $0$.

Under this equivalence, we have the following correspondences:
\begin{itemize}
	\item a coproduct of pointed spaces corresponds to a direct sum with basis the union
		of bases
	\item a wedge product of pointed spaces corresponds to a tensor product with basis
		the product of bases
	\item a $G$-graded pointed set corresponds to a $G$-graded $\BF_2$-vector space
		with a basis consisting of homogeneous elements
	\item a $G$-filtered pointed set corresponds to a $G$-filtered $\BF_2$-vector space
		$V$, ie a family $\{V_{\ge g}\}_{g\in G}$ of subspaces of $V$ with
		$V_{\ge g}\subset V_{\ge g'}$ if $g>g'$,
		with a basis $B$ such that $B\cap V_{\ge g}$ is a basis of $V_{\ge g}$
		for all $g\in G$ and such that given $v\in V\setminus\{0\}$, the set
		$\{g\in G\ |\ V_{\ge g}\neq 0\}$ is non-empty and has a maximal element
	\item a differential pointed set corresponds to 
an $\BF_2$-vector space with a basis together with a bounded differential.
\end{itemize}

\subsection{Symmetric powers}
\label{se:symmetricpowers}
Let $\CC$ be a pointed category. We define a pointed category
$S(\CC)$. Its objects are finite families $I$ of distinct objects of $\CC$.
We put
$$\Hom_{S(\CC)}(I,J)=\bigvee_{\phi}\bigwedge_{i\in I}\Hom_{\CC}(i,\phi(i))$$
where $\phi$ runs over the set of bijections $I\iso J$.

An element of $\Hom_{S(\CC)}(I,J)$ is a pair
$(\phi,f)$ where $\phi:I\iso J$ is a bijection and
$f\in\prod_{i\in I}\Hom_{\CC}(i,\phi(i))$. All pairs with 
$f_i=0$ for some $i$ are identified, and they form the $0$-element of
$\Hom_{S(\CC)}(I,J)$.
The composition is given by
$(\psi,g)\circ (\phi,f)=(\psi\phi,(g_{\phi(i)}\circ f_i)_{i\in I})$.

Given a functor $F:\CC\to\CC'$ of pointed categories that is injective
on the set of objects, we obtain a
functor $S(F):S(\CC)\to S(\CC')$ of pointed categories. If in addition $F$ is
faithful, then $S(F)$ is faithful.

\medskip
Given a commutative ring $k$ and a $k$-linear category $\CD$, we define a
$k$-linear category $S_k(\CD)$.
Its objects are finite families $I$ of distinct objects of $\CD$.
We put 
$$\Hom_{S_k(\CD)}(I,J)=\bigoplus_{\phi:I\iso J}\bigotimes_{i\in I}
\Hom_{\CD}(i,\phi(i)).$$

The composition is defined as in the case of pointed categories above.

Consider a functor $F:\CD\to\CD'$ of $k$-linear categories that is injective on
the set of objects. We obtain a
functor $S_k(F):S_k(\CD)\to S_k(\CD')$ of pointed categories.
If $\Hom$-spaces in $\CD$ and $\CD'$ are flat over $k$ and $F$ is faithful, then
$S_k(F)$ is faithful.

\medskip
Given a pointed category $\CC$, there is an isomorphism of
$k$-linear categories $k[S(\CC)]\iso S_k(k[\CC])$.

\section{Hecke algebras}
\label{se:Hecke}
In this section, we define and study variations of the nil affine Hecke algebra of
$\GL_n$. From \S \ref{se:diffHecke} onwards, all additive structures will be defined
over $k=\BF_2$.

\subsection{Differential graded nil Hecke algebras}
\label{se:diffnilHecke}
We discuss here the case of general Coxeter groups. The results will be used only for
types $A_n$ and $\tilde{A}_n$.

\subsubsection{Coxeter groups}
We refer to \cite[\S 5 and \S 7.1--7.3]{Hu} for basic properties of Coxeter groups and
Hecke algebras.
Recall that a {\em Coxeter group} $(W,S)$ is the data of a group $W$ with a
subset $S\subset W$ such that $W$ has a presentation with generating set $S$ and
relations
$$s^2=1,\ \underbrace{sts\cdots}_{m_{st}\text{ terms}} =
\underbrace{tst\cdots }_{m_{st}\text{ terms}}
\text{ when }st \text{ has order }m_{st}\ \text{ for }s,t\in S.$$

A {\em reduced expression} of an element $w\in W$ is a decomposition
$w=s_{i_1}\cdots s_{i_l}$ such that $s_{i_r}\in S$ for $r=1,\ldots,l$ and
such that $l$ is minimal with this property. The integer $l$ is
the {\em length} $\ell(w)$\indexnot{lw}{\ell(w)} of $w$.

The Chevalley-Bruhat (partial) order on $W$ is defined as follows\indexnot{<}{w<w'}.
Let $w',w\in W$ and let $w=s_{i_1}\cdots s_{i_l}$ be a reduced decomposition.
We say that $w'\le w$ if there is $l'\le l$ and an increasing injection
$f:\{1,\ldots,l'\}\to\{1,\ldots,l\}$ such that $w'=s_{i_{f(1)}}\cdots
s_{i_{f(l')}}$. This is independent of the choice of the reduced decomposition
of $w$.

\subsubsection{Hecke algebras}
\label{se:Heckedef}

Let $R=\BZ[\{a_s,b_s\}_{s\in S}]$ where $a_s$ and $b_s$ are indeterminates with
$a_s=a_{s'}$ and $b_s=b_{s'}$ if $s$ and $s'$ are conjugate in $W$.

The Hecke algebra 
$H=H(W)$ of $(W,S)$
is the $R$-algebra generated by $\{T_s\}_{s\in S}$ with relations
$$T_s^2+a_sT_s+b_s=0,\ \underbrace{T_sT_tT_s\cdots}_{m_{st}\text{ terms}} =
\underbrace{T_tT_sT_t\cdots }_{m_{st}\text{ terms}}
\text{ when }st \text{ has order }m_{st}.$$

Given a reduced decomposition $w=s_{i_1}\cdots s_{i_l}$, we put
$T_w=T_{s_{i_1}}\cdots T_{s_{i_l}}$\indexnot{Tw}{T_w}. This element
is independent of the choice of the
reduced decomposition of $w$. The set $\{T_w\}_{w\in W}$ is a basis of $H$.

\smallskip
Let $\iota:H\iso H^\opp$ be the algebra automorphism defined by $T_s\mapsto T_s$
for $s\in S$.

\medskip
Let $I$ be a subset of $S$. We denote by $W_I$ the subgroup of $W$ generated by
$I$. The group $W_I$, together with $I$, is a Coxeter group and the length function on
$W_I$ is the restriction of that on $W$ \cite[\S 1.10]{Hu}.

We put $R_I=\BZ[\{a_{s,I},b_{s,I}\}_{s\in I}]$ where $a_{s,I}$ and $b_{s,I}$ are
indeterminates with
$a_{s,I}=a_{s',I}$ and $b_{s,I}=b_{s',I}$ if $s$ and $s'$ are conjugate in $W_I$.
There is a morphism of rings $R_I\to R,\ a_{s,I}\mapsto a_s,\ b_{s,I}\mapsto b_s$.

We denote by $H_I=H_I(W)$ the $R$-subalgebra of $H$ generated by $\{T_s\}_{s\in I}$.
There is an isomorphism of $R$-algebras $R\otimes_{R_I}H(W_I)\iso H_I(W),\ T_w\mapsto T_w$.

\medskip
We assume for the remainder of \S\ref{se:Heckedef} that $W$ is finite.
In this case, there is a unique element $w_S$ of $W$ with maximal length
\cite[\S 1.8]{Hu} and we denote by $N$ its length. We have $w_S^2=1$ and $w_S Sw_S=S$.
There is an automorphism of algebras
$$\iota_S:H\iso H,\ T_v\mapsto T_{w_S\cdot v\cdot w_S}.$$

We denote by $w_I$ the longest element of $W_I$ and by $N_I$ its length.
We denote by $W^I$ (resp. ${^IW}$) the set of elements $v\in W$ such that $v$ has minimal
length in $vW_I$ (resp. $W_Iv$). Note that $W^I\iso W/W_I,\ v\mapsto vW_I$
\cite[Proposition 1.10]{Hu}.

\smallskip

\subsubsection{Traces}
\label{se:traces}

We assume in \S\ref{se:traces} that $W$ is finite.

Given $J\subset I$,
we define an $R$-linear map
$$t_{I,J}:H_I\to H_J,\
T_v\mapsto\begin{cases}
	T_{w_J w_I v} & \text{ if } v\in w_I\cdot W_J \\
	0 & \text{ otherwise.}
\end{cases}$$

The next proposition shows this is relative Frobenius form (cf eg \cite[\S 2.3.2]{Rou1}).

\begin{prop}
	\label{pr:trace}
	We have $t_{S,J}=t_{I,J}\circ t_{S,I}$.

	Given $h\in H$ and $x\in W_I$, we have
	$$t_{S,I}(hT_x)=t_{S,I}(h)T_x,\
	t_{S,I}(T_{w_Sw_I\cdot x\cdot w_Iw_S} h)=T_x t_{S,I}(h).$$
	Given $h'\in H$ commuting with $H_I$, we have
	$t_{S,I}(hh')=t_{S,I}(\iota_S(h')h)$.

	There is an isomorphism of $R$-modules
	$$\hat{t}_{S,I}:H\iso \Hom_{H_I^\opp}(H,H_I),\ h\mapsto
	(h'\mapsto t_{S,I}(hh'))$$
	with
	$$\hat{t}_{S,I}(T_{w_Sw_I\cdot x\cdot w_Iw_S} h T_y)=T_x\hat{t}_{S,I}(h)T_y
	\text{ for }x\in W_I \text{ and }y\in W.$$
\end{prop}

\begin{proof}

	Define $w^I=w_Sw_I$ and ${^Iw}=w_Iw_S$, so that ${^Iw}\cdot w^I=1$.
	We have $w^I\in W^I$.

	Let $v\in W$. There is a unique decomposition $v=v'v''$ where 
	$\ell(v)=\ell(v')+\ell(v'')$, $v''\in W_I$ and
	$v'\in W^I$ \cite[Proposition 1.10]{Hu}. Furthermore,
	$\ell(v')<\ell(w^I)$ unless $v'=w^I$. We have $t_{S,I}(T_v)=\delta_{v',w^I}T_{v''}$.

	\smallskip
	There is a unique decomposition $v''=v_1v_2$ with 
	$\ell(v'')=\ell(v_1)+\ell(v_2)$, $v_2\in W_J$ and $v_1$ has minimal
	length in $v''W_J$. We have $v=(v'v_1)v_2$ where $\ell(v)=\ell(v'v_1)+\ell(v_2)$
	and $v'v_1$ has minimal length in $vW_J$. Furthermore,
	$v'v_1=w^J$ if and only if $v'=w^I$ and $v_1=w_Iw_J$.
	It follows that 
	$$t_{I,J}\circ t_{S,I}(T_v)=\delta_{v',w^I}t_{I,J}(T_{v''})=
	\delta_{v',w^I}\delta_{v_1,w_Iw_J}T_{v_2}=t_{S,J}(T_v).$$
	This shows the first statement of the lemma.

	\smallskip
	We have $T_{v''}T_x\in H_I$, hence
	$$t_{S,I}(T_vT_x)=t_{S,I}(T_{v'}(T_{v''}T_x))=\delta_{v',w^I}
	T_{v''}T_x=t_{S,I}(T_v)T_x.$$
	This shows the second statement of the lemma.

	\smallskip
	Let $x'=w^I\cdot x\cdot {^Iw}$. We have $\ell(w^I\cdot x\cdot {^Iw})=\ell(x)$.
	Since $T_{x'}T_v$ is a linear combination of elements $T_{yz}$ with
	$y\le x'$ and $z\le v$, it follows that if 
	$v'\neq w^I$, then $T_{x'}T_{v'}$ is a linear combination of elements
	$T_{w^I\cdot y\cdot {^Iw}z}$ with $y\in W_I$ and $z{\not\in}w^IW_I$, hence
	of elements $T_u$ with $u{\not\in} w^I W_I$. So, if $v'\neq w^I$, then
	$t_{S,I}(T_{x'}T_v)=0$.

	Assume now $v'=w^I$. We have $T_{x'}T_v=T_{w^I\cdot x\cdot {^Iw}}T_{w^I}T_{v''}=
	T_{w^I\cdot x}T_{v''}=T_{w^I}T_xT_{v''}$ because 
	$\ell(x'\cdot w^I)=\ell(w^I\cdot x)=\ell(w^I)+\ell(x)=\ell(x')+\ell(w^I)$. We deduce
	that $t_{S,I}(T_{x'}T_v)=T_xT_{v''}=T_x t_{S,I}(T_{v})$.
	This shows the third statement of the lemma.

	\smallskip
	Let $v_0\in W^I$.  We have $\ell(w^I)=\ell(w^I v_0^{-1})+\ell(v_0)$.
	Let $v\in W^I$. Note that
	$T_{w^Iv_0^{-1}}T_v=T_{w^Iv_0^{-1}v}$ or
	$T_{w^Iv_0^{-1}}T_v$ is a linear combination of $T_w$'s with
	$\ell(w)<\ell(w^Iv_0^{-1})+\ell(v)$.
	It follows that if $t_{S,I}(T_{w^I v_0^{-1}}T_v)=0$ if $\ell(v)<\ell(v_0)$
	or $\ell(v)=\ell(v_0)$ and $v\neq v_0$. We have also 
	$t_{S,I}(T_{w^I v_0^{-1}}T_{v_0})=1$.

	Since $H$ is a free right $H_I$-module with basis $\{T_v\}_{v\in W^I}$,
	we deduce that $\hat{t}_{S,I}$ is surjective.
	Since $\hat{t}_{S,I}$ is an $R$-module morphism between free $R$-modules
	of the same finite rank, it follows that it is an isomorphism.
	This shows the fifth statement of the lemma.

	\smallskip
	Let $s\in S$ and $v\in W$. Let $s'=w_S\cdot s\cdot w_S\in S$.
	If $v{\not\in}\{w_S,w_S\cdot s\}$, then $t_{S,\emptyset}(T_vT_s)=0$ and
	$t_{S,\emptyset}(T_{s'}T_v)=0$. 
	If $v=w_S\cdot s$, then $T_vT_s=T_{w_S}=T_{s'}T_v$.
	If $v=w_S$, then $t_{S,\emptyset}(T_vT_s)=a_s=t_{S,\emptyset}(T_{s'}T_v)$.
	So, we have shown that $t_{S,\emptyset}(T_vT_s)=t_{S,\emptyset}(T_{s'}T_v)$.
	It follows by induction on $\ell(w)$ that
	$t_{S,\emptyset}(T_vT_w)=t_{S,\emptyset}(T_{w_S\cdot w\cdot w_S}T_v)$ for all
	$w\in W$.
	
	Consider now $h'\in H$ commuting with $H_I$. Let $h''\in H_I$.
	We have 
\begin{multline*}
t_{I,\emptyset}(t_{S,I}(hh')h'')=t_{I,\emptyset}(t_{S,I}(hh'h''))=
t_{S,\emptyset}(hh''h')=t_{S,\emptyset}(\iota_S(h')hh'')= \\
=t_{I,\emptyset}(t_{S,I}(\iota_S(h')hh''))=t_{I,\emptyset}(t_{S,I}(\iota_S(h')h)h'').
\end{multline*}
It follows that $\hat{t}_{I,\emptyset}(t_{S,I}(hh'))=\hat{t}_{I,\emptyset}(
t_{S,I}(\iota_S(h')h))$, hence $t_{S,I}(hh')=t_{S,I}(\iota_S(h')h)$. This completes the proof
of the lemma.
\end{proof}

We put $t^+_{I,J}=t_{I,J}$.
We define an $R$-linear map
$$t^-_{I,J}:H_I\to H_J,\
T_v\mapsto\begin{cases}
	T_{vw_Iw_J} & \text{ if } v\in W_J\cdot w_I \\
	0 & \text{ otherwise.}
\end{cases}$$

 We have $t^-_{I,J}(h)=\iota(t^+_{I,J}(\iota(h)))$.

We put $\hat{t}^+_{S,I}=\hat{t}_{S,I}$. We have an isomorphism of $R$-modules
	$$\hat{t}^-_{S,I}:H\iso \Hom_{H_I^\opp}(H,H_I),\ h\mapsto
	(h'\mapsto t^-_{S,I}(hh'))$$
	with
	$$\hat{t}_{S,I}^-(T_x h T_y)=T_x\hat{t}_{S,I}^-(h)T_y
	\text{ for }x\in W_I \text{ and }y\in W.$$

\medskip
Consider $I,J\subset S$ with $I\subset J$ or $J\subset I$.
We define an $(H_I,H_J)$-bimodule $L^\pm(I,J)$ with underlying $R$-module $H$.
We put $a=0$ if $\pm=+$ and $a=1$ if $\pm=-$.

If $I\subset J$, then the right action of $H_J$ is by right multiplication and the left
action of $h\in H_I$ is by left multiplication by $(\iota_J\iota_I)^a(h)$.

If $J\subset I$, then the left action of $H_I$ is by left multiplication and the right
action of $h\in H_J$ is by right multiplication by $(\iota_I\iota_J)^a(h)$.

Note that $L^\pm(I,J)$ is free of finite rank as a left module and as a right module.

\smallskip

There is an isomorphism of $(H,H_I)$-bimodules
$$L^\pm(I,S)^\vee=\Hom_{H^\opp}(L^\pm(I,S),H)\iso L^\pm(S,I),\ \zeta\mapsto \zeta(1).$$

\smallskip
The next result follows immediately from Proposition \ref{pr:trace}.

\begin{cor}
	\label{co:dual}
The map $\hat{t}^\pm_{S,I}$ is an isomorphism of $(H_I,H)$-bimodules
	$$L^\mp(I,S)\iso L^\pm(S,I)^\vee=\Hom_{H_I^\opp}(L^\pm(S,I),H_I).$$
\end{cor}

The results above can be formulated in terms of dual bases.
Note that $\{T_w\}_{w\in W^I}$ is a basis of the free right $H_I$-module $H$, while
$\{T_w\}_{w\in {^IW}}$ is a basis of the free left $H_I$-module $H$.

We have
$$t_{S,I}^+(T_{w_Sw_Iv^{-1}}T_w)=\delta_{v,w} \text{ and }
t_{S,I}^-(T_{v'}T_{w^{\prime -1}w_Iw_S})=\delta_{v',w'} \text{ for }v,w\in W^I
\text{ and }v',w'\in {^IW}.$$
We deduce that the basis $(T_{w_Sw_Iw^{-1}})_{w\in W^I}$ when $\pm=+$ (resp.
$(T_w)_{w\in {^IW}}$ when $\pm=-$)
of the free left $H_I$-module $L^\mp(I,S)$ is dual
to the basis $(T_w)_{w\in W^I}$ when $\pm=+$ (resp. $(T_{w^{-1}w_Iw_S})_{w\in {^IW}}$ when
$\pm=-$)
of the free right $H_I$-module $L^\pm(S,I)$, via the pairing providing the
isomorphism of Corollary \ref{co:dual}.

\smallskip
The counit of the adjoint pair 
$(L^{\mp}(I,S)\otimes_{H}-,L^\pm(S,I)\otimes_{H_I}-)$
is given by the morphism of $(H_I,H_I)$-bimodules
$$L^\mp(I,S)\otimes_{H}L^\pm(S,I)\to H_I,\ a\otimes b\mapsto t^\pm_{S,I}(ab)$$
while the unit is given by the morphism of $(H,H)$-bimodules
$$H\to L^\pm(S,I)\otimes_{H_I}L^\mp(I,S),\ 1\mapsto
\begin{cases}
	\sum_{w\in W^I}T_w\otimes T_{w_Sw_Iw^{-1}} & \text{ if }\pm=+ \\
	\sum_{w\in {^IW}}T_{w^{-1}w_Iw_S}\otimes T_w & \text{ if }\pm=-.
\end{cases}$$

\subsubsection{Nil Hecke algebras}
We define the {\em nil Hecke algebra}\index[ter]{nil Hecke algebra}
$H_\BZ^{\mathrm{nil}}(W)$ of $(W,S)$
as the $\BZ$-algebra $H(W)\otimes_R R/(a_s,b_s)_{s\in S}$. This
is the $\BZ$-algebra generated by $\{T_s\}_{s\in S}$ with relations
$$T_s^2=0,\ \underbrace{T_sT_tT_s\cdots}_{m_{st}\text{ terms}} =
\underbrace{T_tT_sT_t\cdots }_{m_{st}\text{ terms}}
\text{ when }st \text{ has order }m_{st}.$$
This is a $\BZ_{\le 0}$-graded algebra with $T_w$ in degree $-\ell(w)$ for $w\in W$.

The multiplication is given as follows:
\begin{equation}
	\label{eq:multnilHecke}
T_w T_{w'}=\begin{cases}
	T_{ww'} & \text{ if }\ell(ww')=\ell(w)+\ell(w') \\
	0 & \text{ otherwise.}
\end{cases}
\end{equation}

Consider the filtration of the group algebra $\BZ[W]$ where 
$\BZ[W]^{\ge -i}$ is spanned by group elements $w\in W$ with 
$\ell(w)\le i$, for $i\in\BZ_{\ge 0}$.
The associated $\BZ_{\le 0}$-graded algebra is $H_\BZ^{\mathrm{nil}}(W)$ and 
$T_w$ is the image of $w\in W$ in the degree $-\ell(w)$ homogeneous component of
$H_\BZ^{\mathrm{nil}}(W)$.

\subsubsection{Differential}
\label{se:diffHecke}
Let $H^{\mathrm{nil}}(W)=\BF_2\otimes H_\BZ^{\mathrm{nil}}(W)$.
We define a linear map $d:H^{\mathrm{nil}}(W)\to H^{\mathrm{nil}}(W)$
by 
$$d(T_w)=\sum_{w'<w,\ \ell(w')=\ell(w)-1}T_{w'}\indexnot{dT}{d(T_w)}.$$

\begin{prop}
	\label{pr:diffnilHecke}
	The map $d$ defines a structure of differential graded algebra on 
	$H^{\mathrm{nil}}(W)$.
\end{prop}

\begin{proof}
	Let $w\in W$ and $s\in S$ with $ws>w$.
	We have
	$d(T_wT_s)=d(T_{ws})=\sum_{w'<ws,\ \ell(w')=\ell(w)}T_{w'}$. We have
	\cite[Theorem 5.10]{Hu}
	$$\{w'\in W\ | w'<ws,\ \ell(w')=\ell(w)\}=\{w''s\ |\ w''<w,\ w''<w''s,\ \ell(w'')
	=\ell(w)-1\}\sqcup \{w\}.$$
	It follows that $d(T_wT_s)=d(T_w)T_s+T_w=d(T_w)T_s+T_wd(T_s)$.

	\smallskip
	Consider now $v\in W$ and $s\in S$ with $vs<v$.
	We have $d(T_v)=d(T_{vs}T_s)=d(T_{vs})T_s+T_{vs}$ by the result above.
	It follows that $d(T_v)T_s+T_vd(T_s)=T_{vs}T_s+T_v=0=d(T_vT_s)$.

	We deduce that $d(T_wT_{w'})=d(T_w)T_{w'}+T_wd(T_{w'})$ for all
	$w,w'\in W$.

	Since $d^2(T_s)=0$ for $s\in S$, it follows that by induction that $d^2=0$.
\end{proof}

The following corollary shows that the computation of $d(T_w)$ can be done
using the Leibniz rule, given a reduced decomposition of $w$. The terms that do not
vanish are exactly the terms given in the original definition of $d(T_w)$.

\begin{cor}
	\label{co:comparisond}
Let $w=s_{i_1}\cdots s_{i_l}$ be a reduced expression of $w\in W$.
We have
$$d(T_w)=\sum_{r=1}^l T_{i_1}\cdots T_{i_{r-1}}T_{i_{r+1}}T_{i_l}.$$
We have $T_{i_1}\cdots T_{i_{r-1}}T_{i_{r+1}}T_{i_l}\neq
0$ if and only if $s_{i_1}\cdots s_{i_{r-1}}s_{i_{r+1}}\cdots s_{i_l}$ is reduced,
\ie, if and only if $\ell(s_{i_1}\cdots s_{i_{r-1}}s_{i_{r+1}}\cdots s_{i_l})=\ell(w)-1$.

Given $r,r'$ with $s_{i_1}\cdots s_{i_{r-1}}s_{i_{r+1}}\cdots s_{i_l}=
s_{i_1}\cdots s_{i_{r'-1}}s_{i_{r'+1}}\cdots s_{i_l}$ reduced, we have
$r=r'$.
\end{cor}

\begin{proof}
	The first statement follows from Proposition \ref{pr:diffnilHecke}.
	The second statement is a property of the multiplication of $T_w$'s.

	For the third statement, let us assume $r<r'$. We have 
	$s_{i_{r+1}}\cdots s_{i_{r'}}=s_{i_r}\cdots s_{i_{r'-1}}$ reduced, hence
	$s_{i_r}s_{i_{r+1}}\cdots s_{i_{r'}}$ is not reduced, a contradiction.
\end{proof}

\smallskip

\begin{rem}
Note that the algebra $H^{\mathrm{nil}}(W)$ is acyclic if $S\neq\emptyset$.

	Note also that one can introduce a family of commuting differentials $d_s$ for
	$s\in S$ modulo conjugacy by setting $d_s(T_t)=1$ if $t\in S$ is conjugate to
	$s$ and $d_s(T_t)=0$ otherwise.
\end{rem}

The specialization over $\BF_2$ 
at $a_s=b_s=0$ of the bimodules $L^\pm(I,J)$ of \S\ref{se:traces} acquire a
structure of differential graded bimodules, using the differential graded structure 
of $H^{\mathrm{nil}}(W)$.
We keep
the same notation for those differential graded specialized bimodules and for the maps $t$ and
$\hat{t}$.

\begin{prop}
	\label{pr:tracediff}
	If $W$ is finite, then 
	$$t_{S,I}:H^{\mathrm{nil}}(W)\to H^{\mathrm{nil}}(W_I)\langle N-N_I\rangle$$
	is a morphism of differential graded $\BF_2$-modules and
	Corollary \ref{co:dual} provides an isomorphism of differential graded
	$(H^{\mathrm{nil}}(W_I),H^{\mathrm{nil}}(W))$-bimodules
	$$\hat{t}_{S,I}^\pm:L^{\mp}(I,S)\iso L^{\pm}(S,I)^\vee\langle N-N_I\rangle.$$
\end{prop}

\begin{proof}
	Let $v\in W$. There is a unique decomposition $v=v'v''$ where 
	$\ell(v)=\ell(v')+\ell(v'')$, $v''\in W_I$ and
	$v'\in W^I$.

	We have $d(T_v)=d(T_{v'})T_{v''}+T_{v'}d(T_{v''})$. If $u\in W$ and
	$u<v'$, then $u{\not\in}w_S W_I$. It follows that 
$$t_{S,I}(d(T_v))= t_{S,I}(T_{v'}d(T_{v''}))=\delta_{v',w^I}d(T_{v''})=d(t_{S,I}(T_v)).$$
\end{proof}

\subsubsection{Differential graded pointed Hecke monoid}

Let $W^{\nil}$\indexnot{Wnil}{W^{\nil}} be the pointed $\BZ_{\le 0}$-graded monoid with
underlying pointed set 
$\{T_w\}_{w\in W}\coprod\{0\}$ and multiplication given by (\ref{eq:multnilHecke}).
This is the pointed monoid $\gr W$ associated to the filtration on $W$ given by
$W^{\ge -i}=\{w\in W\ |\ \ell(w)\le i\}$ and there is an identification
$\BF_2[W^{\nil}]=H^{\nil}(W)$ making $W^{\nil}$ into a differential graded pointed monoid.


\subsection{Extended affine symmetric groups}
\subsubsection{Finite case}
Fix $n\ge 0$. The symmetric group $\GS_n$ is a Coxeter group with generating set $\{(1,2),\ldots,
(n-1,n)\}$.

Its differential nil Hecke algebra $H_n$\indexnot{Hn}{H_n} is the $k$-algebra generated
by $T_1,\ldots,T_{n-1}$ with relations
\begin{equation}
	\label{eq:relationsHn}
T_i^2=0,\ T_iT_j=T_jT_i \text{ if }|i-j|>1 \text{ and }
T_iT_{i+1}T_i=T_{i+1}T_iT_{i+1}
\end{equation}
and with differential given by $d(T_i)=1$.

The algebra $H_n$ has a basis $(T_w)_{w\in\GS_n}$.

\subsubsection{Definition}
\label{se:defextSn}
Let $n\ge 1$. We denote by $\hat{\GS}_n$\indexnot{Sn}{\hat{\GS}_n}
the extended affine symmetric group:
this is the subgroup of the group of permutations of $\BZ$ with elements those
bijections
$\sigma:\BZ\iso\BZ$ such that $\sigma(n+r)=n+\sigma(r)$ for all $r\in\BZ$.

Given $i,j\in\BZ$ with $i-j{\not\in}n\BZ$, we denote by $s_{ij}$\indexnot{sij}{s_{ij}}
the element
of $\hat{\GS}_n$ defined by
$$s_{ij}(r)=\begin{cases}
	j-i+r & \text{ if }r=i\pmod n \\
	i-j+r & \text{ if }r=j\pmod n\\
	r & \text{otherwise.}
\end{cases}$$
Note that $s_{i+n,j+n}=s_{i,j}$, $s_{ij}=s_{ji}$ and $s_{ij}^2=1$.

The symmetric group $\GS_n$ identifies with the subgroup of $\hat{\GS}_n$ of permutations
$\sigma$ such that $\sigma(\{1,\ldots,n\})=\{1,\ldots,n\}$.
We have a surjective morphism $\hat{\GS}_n\to\GS_n$ sending $\sigma$ to the induced
permutation of $\BZ/n$. We identify its kernel with $\BZ^n$ via the injective morphism
$$\BZ^n\to\hat{\GS}_n,\ (\lambda_1,\ldots,\lambda_n)\mapsto (\{1,\ldots,n\}\ni i\mapsto
i+n\lambda_i).$$
We have $\hat{\GS}_n=\BZ^n\rtimes \GS_n$.

\medskip
Assume $n\ge 2$.
Let $W_n$\indexnot{Wn}{W_n} be the Coxeter group of type $\hat{A}_{n-1}$: it is generated
by $\{s_a\}_{a\in\BZ/n}$ with relations
$$s_a^2=1,\ s_as_b=s_bs_a \text{ if }a\neq b\pm 1$$
$$s_as_{a+1}s_a=s_{a+1}s_as_{a+1}\ (\text{ for }n>2 ).$$

\medskip
Consider the semi-direct product $W_n\rtimes\langle c\rangle$ of $W_n$
by an infinite cyclic group generated by an element $c$, with
relation $cs_ac^{-1}=s_{a+1}$.

\begin{lemma}
	\label{le:isoWnS_n}
There is an isomorphism of groups

$$W_n\rtimes\langle c\rangle\iso \hat{\GS}_n,\
c\mapsto (j\mapsto j+1),\
	s_{i+n\BZ}\mapsto s_{i,i+1} \text{ for }i\in\{1,\ldots,n\}.$$
\end{lemma}

\begin{proof}
	Denote by $f$ the map of the lemma.
	By \cite[\S 3.6]{Lus} (cf also \cite[Proposition 8.3.3]{BjBr}), the
	restriction of $f$ to $W_n$ induces an isomorphism with
	the subgroup of $\hat{\GS}_n$ of elements $\sigma$ such that
	$\sum_{i=1}^n(\sigma(i)-i)=0$. It is immediate to check that
	$f$ extends to a morphism of groups 
	$W_n\rtimes\langle c\rangle\to \hat{\GS}_n$.

	Consider $\sigma\in\hat{\GS}_n$ and let $N=\sum_{i=1}^n(\sigma(i)-i)$.
	Note that $n|N$. Put $\sigma'=\sigma f(c)^{-N/n}$. We have
	$\sigma'\in f(W_n)$, so $f$ is surjective.
	Let $\sigma=f(wc^d)$. We have $\sum_{i=1}^n(\sigma(i)-i)=nd$.
	So, if $\sigma=1$, then $d=0$, hence $w\in\ker(f)\cap W_n=1$.
	This shows that $f$ is injective.
\end{proof}

We will identify $W_n\rtimes\langle c\rangle$ and $\hat{\GS}_n$ via the isomorphism
of Lemma \ref{le:isoWnS_n}.

We put $W_1=1$, so that $\hat{\GS}_1\simeq\langle c\rangle=W_1\rtimes\langle c\rangle$.
We also put $\hat{\GS}_0=1$.

	\subsubsection{Diagrammatic representation}

	The permutations of $\BZ$ can be 
	described as collections of strands in $[-1,1]\times\BR$ going leftwards from integer points on 
	the vertical line $x=1$ to integer points on the vertical line $x=-1$. Thanks to their $n$-periodicity, those permutations
	that are elements of $\hat{\GS}_n$ can also
	be encoded in a collection of strands drawn on a cylinder, going from right to left,
	by passing to the quotient of the vertical strip $[-1,1]\times\BR$ by the vertical
	action by translation of $n\BZ$.

	Here are some elements of $\hat{\GS}_3$:

$$\includegraphics[scale=0.7]{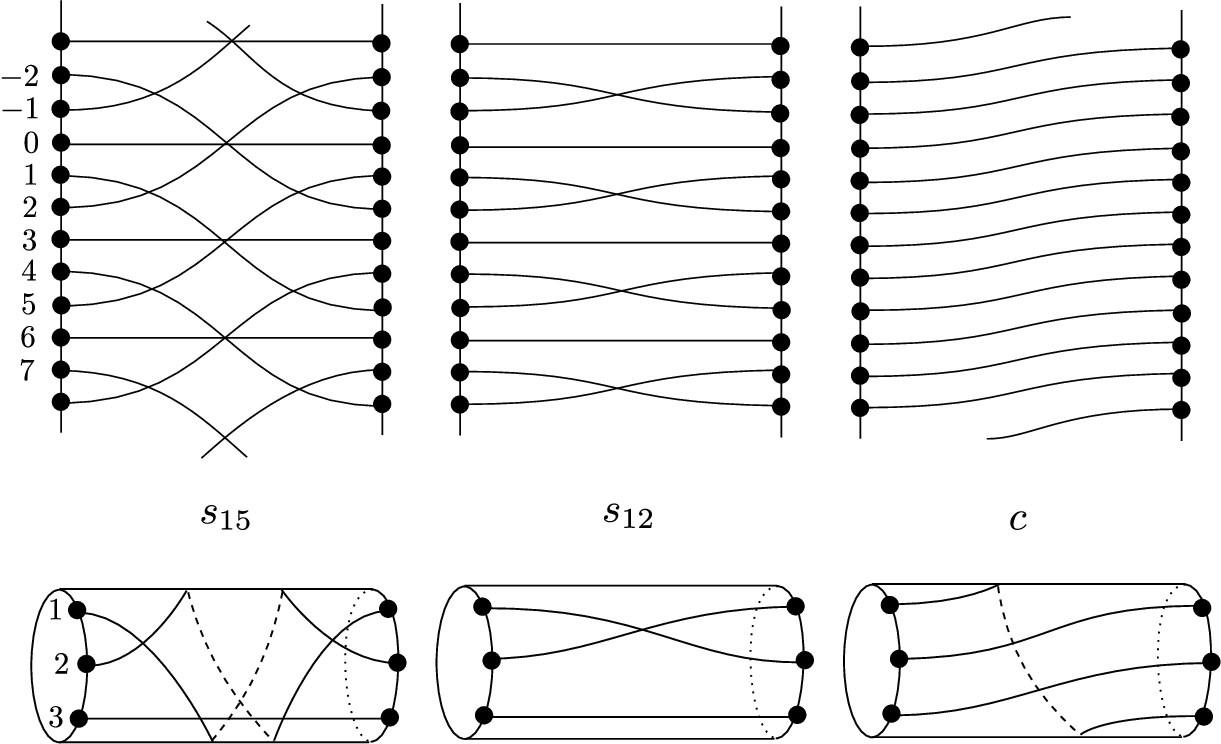}$$

	\smallskip
The multiplication $\sigma\sigma'$ of $\sigma$ and $\sigma'$ in $\hat{\GS}_n$ corresponds to
the concatenation of the diagram of $\sigma$ put to the left of the diagram of $\sigma'$ as 
in the following example:
$$\includegraphics[scale=0.7]{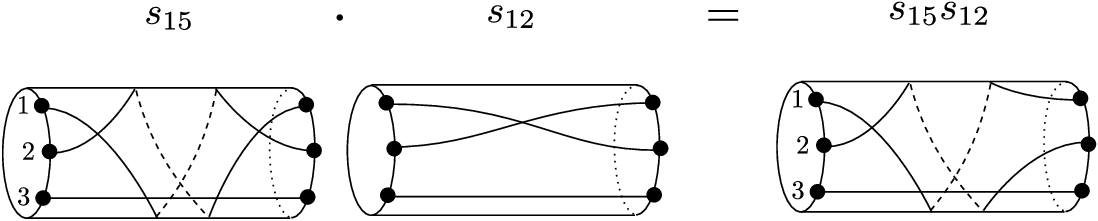}$$

	\smallskip
	The defining relations for $\hat{\GS}_n$ are depicted as follows
$$\includegraphics[scale=0.7]{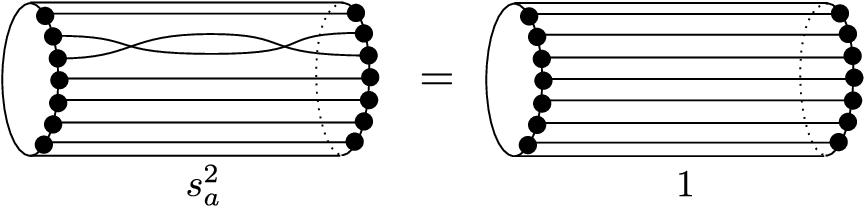}$$
$$\includegraphics[scale=0.7]{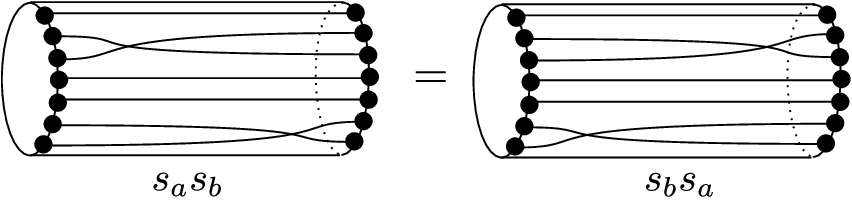}$$
$$\includegraphics[scale=0.7]{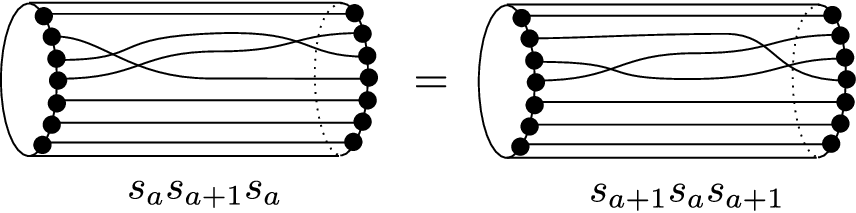}$$
$$\includegraphics[scale=0.7]{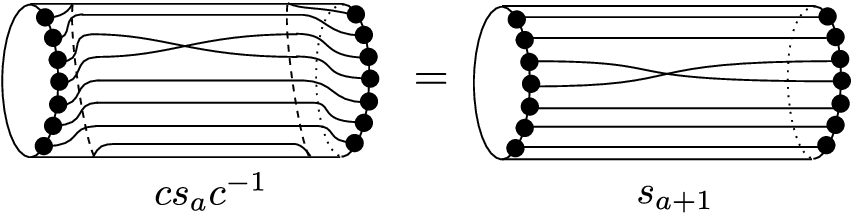}$$

	\smallskip
	The elements of $\GS_n$ correspond to diagrams whose strands do not go in the back
	of the cylinder, hence can be drawn on a rectangle. For example, $s_{12}$ above can
	be represented as follows:
$$\includegraphics[scale=0.6]{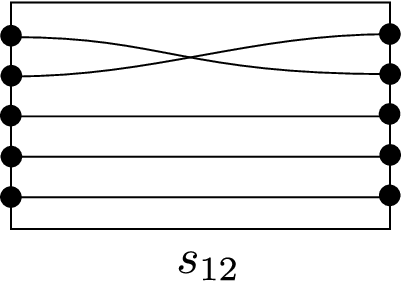}$$

\subsubsection{Length}
Assume now again that $n\ge 1$.
We extend the length function on the Coxeter group $W_n$ to
one on $W_n\rtimes\langle c\rangle$ by setting
$\ell(wc^d)=\ell(w)$ for $w\in W_n$ and $d\in\BZ$. Note that the action of
$c$ on $W_n$ preserves lengths.
Similarly, we extend the Chevalley-Bruhat order on 
$W_n\rtimes\langle c\rangle$ by setting $w'c^{d'}<wc^d$ if
$w'<w$ and $d'=d$ and we consider the corresponding order on $\hat{\GS}_n$.
Note that the action of $c$ on $W_n$ preserves the order, hence
$w'c^{d'}<wc^d$ if and only if $c^{d'}w'<c^dw$.

\begin{lemma}
	\label{le:middleaffineSn}
	Let $\sigma',\sigma''\in \hat{\GS}_n$ and $\sigma=\sigma'\sigma''$.
	Assume 
	$\ell(\sigma)=\ell(\sigma')+\ell(\sigma'')$.
	Let $a\in\BZ/n$ such that $\ell(\sigma s_a)<\ell(\sigma)$
	and $\ell(\sigma'' s_a)>\ell(\sigma'')$.

	Let $\alpha''=\sigma''s_a$ and
	$\alpha'=\sigma' \sigma'' s_a \sigma^{\prime\prime -1}$.
	We have $\sigma=\alpha'\alpha''$ and
	$\ell(\sigma)=\ell(\alpha')+\ell(\alpha'')$.
\end{lemma}

\begin{proof}
	Multiplying if necessary $\sigma'$ and $\sigma''$ by a power of
	$c$, we can assume $\sigma$, $\sigma'$ and $\sigma''$ are in 
	$W_n$.

	Let $\sigma'=s_{a_1}\cdots s_{a_m}$ and $\sigma''=s_{a_{m+1}}\cdots
	s_{a_d}$ be two reduced decompositions. The Exchange Lemma
	\cite[Theorem 5.8]{Hu} shows that there is $i$ such that
	$\sigma s_a=s_{a_1}\cdots s_{a_{i-1}}s_{a_{i+1}}\cdots s_{a_d}$.

	If $i>m$, then $\sigma''s_a=s_{a_{m+1}}\cdots
	s_{a_{i-1}}s_{a_{i+1}}\cdots s_{a_d}$ and this contradicts
	$\ell(\sigma'' s_a)>\ell(\sigma'')$. So, $i\le m$.
	We have
	$\sigma s_a=s_{a_1}\cdots s_{a_{i-1}}s_{a_{i+1}}\cdots s_{a_m}
	\sigma''$. We deduce that $\alpha'=s_{a_1}\cdots s_{a_{i-1}}s_{a_{i+1}}\cdots s_{a_m}$ has length $m-1$ and the lemma follows.
\end{proof}

Given $\sigma\in\hat{\GS}_n$, we put $L(\sigma)=\{(i,j)\in\BZ\times\BZ\ |\
i<j,\ \sigma(i)>\sigma(j)\}$\indexnot{Ls}{L(\sigma)}. This set has a diagonal action of
$n\BZ$ by translation.
We put $\tilde{L}(\sigma)=\{(i,j)\in L(\sigma)\ |\ 1\le i\le n\}$. The canonical
map $\tilde{L}(\sigma)\to L(\sigma)/n\BZ$ is bijective.

\smallskip
The next lemma is a variation on classical results
(cf \cite[Lemma 4.2.2]{Sh}, \cite[Proposition 8.3.6]{BjBr} and \cite[\S 2.2]{BjBr}).

\begin{lemma}
	\label{le:lengthaffine}
	Let $\sigma\in\hat{\GS}_n$. We have $L(\sigma)=L(c^d\sigma)$ for all $d\in\BZ$ and
	$$\ell(\sigma)=|\tilde{L}(\sigma)|=
	\sum_{0\le i<j<n} \bigl|{\lfloor\frac{\sigma(j)-\sigma(i)}{n}\rfloor}\bigr|.$$

	If $(i,j)\in L(\sigma)$, then $\sigma s_{ij}<\sigma$.

\smallskip
	Assume $\sigma=c^dw$ and $w=s_{a_1}\cdots s_{a_l}$ is a reduced decomposition of $w\in W_n$.
	Given $1\le r\le l$, let $i_r\in\{1,\ldots,n\}$ with $i_r+n\BZ=a_r$.

	The set $\{\bigl(
s_{a_l}\cdots s_{a_{r+1}}(i_r),s_{a_l}\cdots s_{a_{r+1}}(i_r+1)\bigr)\}_{1\le r\le l}$ is a subset of $L(\sigma)$. This induces a bijection
$$\{\bigl(
	(s_{a_l}\cdots s_{a_{r+1}}(i_r),s_{a_l}\cdots s_{a_{r+1}}(i_r+1)\bigr)\}_{1\le r\le l}\iso L(\sigma)/n\BZ.$$
\end{lemma}

\begin{proof}
		Consider a pair $(i,j)\in L(\sigma)$ with $1\le i\le n$ and such that
	$(i,j'){\not\in}L(\sigma)$ and $(j',j){\not\in}L(\sigma)$ for $i<j'<j$.
	Given $j'$ with $i<j'<j$, we have $\sigma(i)<\sigma(j')<\sigma(j)$, a contradiction.
	It follows that $j=i+1$. We have
$$L(\sigma)=\bigl(\{(i,i+1)\}+n\BZ\bigr)\coprod (s_{i,i+1},s_{i,i+1})(L(\sigma s_{i,i+1})).$$
	We deduce
	by induction on $|\tilde{L}(\sigma)|$ that $\ell(\sigma)\le |\tilde{L}(\sigma)|$.

	We prove the statements on
	$\{\bigl( s_{a_l}\cdots s_{a_{r+1}}(i_r),s_{a_l}\cdots s_{a_{r+1}}(i_r+1)\bigr)\}_{1\le r\le l}$
	by induction on $\ell(\sigma)$. By induction, the statements hold for
	$\sigma s_{a_l,a_l+1}$. In particular, $\ell(\sigma s_{a_l,a_l+1})=
	|\tilde{L}(\sigma s_{a_l,a_l+1})|$. It follows that
	$\ell(\sigma)=\ell(\sigma s_{a_l,a_l+1})+1>|\tilde{L}(\sigma s_{a_l,a_l+1})|$.
	Assume $(i_l,i_l+1){\not\in}L(\sigma)$. It follows that
$L(\sigma s_{a_l,a_l+1})=s_{a_l,a_l+1}(L(\sigma))\coprod \bigl(\{(i_l,i_l+1)\}+n\BZ\bigr)$,
	hence $|\tilde{L}(\sigma)|<|\tilde{L}(\sigma s_{a_l,a_l+1})|=\ell(\sigma s_{a_l,a_l+1})=
	\ell(\sigma)-1$, a contradiction. It follows that $(i_l,i_l+1)\in L(\sigma)$,
	hence
	$$L(\sigma)=s_{a_l,a_l+1}(L(\sigma s_{i_l,i_l+1}))\coprod 
	\bigl(\{(i_l,i_l+1)\}+n\BZ\bigr).$$
The last statement of the lemma follows now by induction.

\smallskip
	Consider now $(i,j)\in L(\sigma)$. Up to translating $(i,j)$ diagonally by
	$n\BZ$, we can assume there is $r$ such that 
	$i=s_{a_l}\cdots s_{a_{r+1}}(i_r)$ and 
	$j=s_{a_l}\cdots s_{a_{r+1}}(i_r+1)$. So
	$\sigma s_{i,j}=c^d s_{a_1}\cdots s_{a_{r-1}}s_{a_{r+1}}\cdots s_{a_l}$,
	hence $\sigma s_{i,j}<\sigma$. The lemma follows.
\end{proof}

\begin{lemma}
	\label{le:lengthoneaffine}
	Given $\sigma,\sigma'\in\hat{\GS}_n$,
	we have $\sigma'<\sigma$ and
	$\ell(\sigma')=\ell(\sigma)-1$ if and only if
	there is $(j_1,j_2)\in L(\sigma)$ such that $\sigma'=\sigma s_{j_1,j_2}$ and
	\begin{itemize}
		\item $j_2-j_1<n$ or $\sigma(j_1)-\sigma(j_2)<n$ and
		\item given $i\in\BZ$ with $j_1<i<j_2$, we have
			$\sigma(j_1)<\sigma(i)$ or $\sigma(i)<\sigma(j_2)$.
	\end{itemize}
\end{lemma}

\begin{proof}
	Consider $(j_1,j_2)\in L(\sigma)$ and
	let $s=s_{j_1,j_2}$.
	Consider integers $i<j$ with $i-j{\not\in}n\BZ$.
	
	If $s(i)<s(j)$, then
	$(i,j)\in L(\sigma)$ if and only if $s(i,j)=(s(i),s(j))\in L(\sigma s)$.

	Assume now $s(i)>s(j)$.	We have three possibilities:

	$\bullet\ $ $i-j_1\in n\BZ$, $j-j_2{\not\in}n\BZ$: we have
	$(i,j)\in L(\sigma)$ if and only if $(i,j)\in L(\sigma s)$ or
	$\sigma(i)>\sigma(j)>\sigma s(i)$ (and then $(i,j){\not\in}L(\sigma s)$)

	$\bullet\ $ $i-j_1{\not\in} n\BZ$, $j-j_2\in n\BZ$:  we have
	$(i,j)\in L(\sigma)$ if and only if $(i,j)\in L(\sigma s)$ or
	$\sigma s(j)>\sigma(i)>\sigma(j)$ (and then $(i,j){\not\in}L(\sigma s)$).

	$\bullet\ $ $i=j_1+nr$, $j=j_2+nr'$ with $r,r'\in \BZ$: we have
	$(i,j)\in L(\sigma)$ if and only if $(i,j)\in L(\sigma s)$ or
	$\sigma(j_1)-\sigma(j_2)>n(r'-r)>\sigma(j_2)-\sigma(j_1)$
	(and then $(i,j){\not\in}L(\sigma s)$).

 We deduce there is an injective map $a:L(\sigma s)\to L(\sigma)$
	given by 
	$$a((i,j))=\begin{cases}
		(i,j)&\text{ if } s(i)>s(j) \\
		s(i,j) & \text{ otherwise}
	\end{cases}$$
	and 
	\begin{multline*}
	L(\sigma)=a(L(\sigma s))\sqcup \coprod_{|r|<\min\bigl(\frac{j_2-j_1}{n},
	\frac{\sigma(j_1)-\sigma(j_2)}{n}\bigr)}\bigl((j_1+nr,j_2)+n\BZ\bigr)\sqcup\\
	\coprod_{\substack{j_1<i<j_2\\ \sigma(j_1)>\sigma(i)>\sigma(j_2)}}
	\Bigl(\bigl((j_1,i)+n\BZ\bigl)\sqcup\bigl((i,j_2)+n\BZ\bigr)\Bigr).
\end{multline*}
	Note that $a(L(\sigma s))\sqcup ((j_1,j_2)+n\BZ)\subset L(\sigma)$. 
	
	\medskip
	Let us now prove the lemma.
	We have $\sigma=c^dw$ and
	$\sigma'=c^{d'}w'\in\hat{\GS}_n$ for some $w,w'\in W_n$.
	Assume $\sigma'<\sigma$ and $\ell(\sigma')=\ell(\sigma)-1$.
	We have $d=d'$, $w'<w$ and $\ell(w')=\ell(w)-1$. It follows that
	there is a reduced decomposition $w=s_{a_1}\cdots s_{a_l}$ and $r\in\{1,\ldots,l\}$
	such that $w'=s_{a_1}\cdots s_{a_{r-1}}s_{a_{r+1}}\cdots s_{a_l}$.
	Let $j_1=s_{a_l}\cdots s_{a_{r+1}}(i_r)$ and
	$j_2=s_{a_l}\cdots s_{a_{r+1}}(i_r+1)$. We have $(j_1,j_2)\in
	L(\sigma)$ and $\sigma'=\sigma s_{j_1,j_2}$ (Lemma \ref{le:lengthaffine}).

	The discussion above shows that $\{i\in\BZ\ |\ j_1<i<j_2,\
	\sigma(j_1)>\sigma(i)>\sigma(j_2)\}=\emptyset$ and
	$\min\bigl(\frac{j_2-j_1}{n}, \frac{\sigma(j_1)-\sigma(j_2)}{n}\bigr)<1$.
	The lemma follows.
\end{proof}

\begin{example}
	The elements of $\tilde{L}(\sigma)$ are in bijection with intersection points between
	strands of a ``good diagram" representing $\sigma$. Here, we define a strand
	diagram to be good if no more than two strands intersect at a given point and if the
	diagram minimizes the total number of intersection points. Similarly,
	the elements of $L(\sigma)$ correspond to intersections in an unfolded good strand
	diagram. 
	
	These descriptions can be deduced from Lemma \ref{le:intersectionlength} below, that
	shows those
	statements hold for pairs of strands. Now, the intersection point set for a good
	diagram is the disjoint union over intersection sets between pairs of strands, and
	a good diagram minimizes the intersection number among good diagrams if and only of
	each pair of strands minimizes its intersection number.
	
	For example:
	$$\includegraphics[scale=0.9]{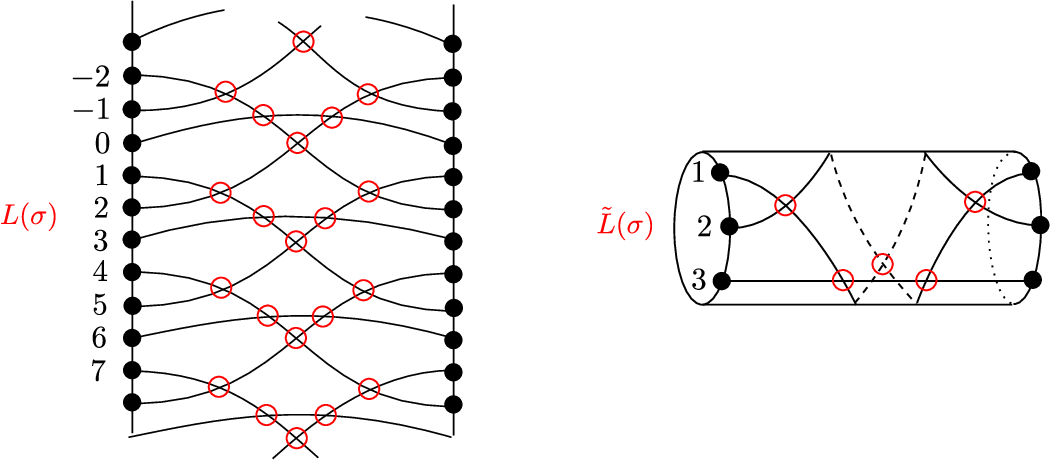}$$
\end{example}

\subsubsection{Extended affine Hecke algebra}
We let $c$ act on the differential graded algebra
$H^{\mathrm{nil}}(W_n)$ by $c(T_a)=T_{a+1}$.
Let $\hat{H}_n=H^{\mathrm{nil}}(W_n)\rtimes\langle c\rangle$\indexnot{Hn}{\hat{H}_n}.
For $n\ge 2$, it
is the differential graded $\BF_2$-algebra generated by
$\{T_a\}_{a\in\BZ/n}$\indexnot{Ta}{T_a} and $c^{\pm 1}$ with relations
$$T_a^2=0,\ cT_a=T_{a+1}c,\ T_aT_b=T_bT_a \text{ if }a\neq b\pm 1$$
$$T_aT_{a+1}T_a=T_{a+1}T_aT_{a+1}\ (\text{ for }n>2\ )$$
and differential $d(T_a)=1$, $d(c)=0$. The element $c$ has degree $0$,
while $T_a$ has degree $-1$. Note that $\hat{H}_1=\BF_2[\hat{\GS}_1]=
\BF_2\langle c\rangle$,
a differential graded algebra in degree $0$ with $d=0$.

\medskip
Let $w\in W_n$, $d\in\BZ$ and $w'=wc^d$. We put $T_{w'}=T_w c^d$. We also put
$T_\sigma=T_w c^d$ for $\sigma=wc^d$.
The set $\{T_{\sigma}\}_{\sigma\in\hat{\GS}_n}$ is a basis of $\hat{H}_n$.

\begin{rem}
Define a filtration on $\BF_2[\hat{\GS}_n]$ with
$(\BF_2[\hat{\GS}_n])^{\ge -i}$ the subspace spanned by group elements
$w\in \hat{\GS}_n$ with $\ell(w)\le i$.
The associated graded algebra is $\hat{H}_n$.
\end{rem}

%

We put $\hat{H}_0=\BF_2$.

\begin{rem}
	The group $\hat{\GS}_n$ is more classically described as a semi-direct product
	$\BZ^n\rtimes\GS_n$ (cf \S\ref{se:defextSn})
	coming from its description as the extended affine Weyl group
	of $\GL_n$. The nil affine Hecke algebra of $\GL_n$ associated with this
	description (cf e.g. \cite[\S 2.2.2]{Rou2})
	is {\em not} isomorphic to $\hat{H}_n$. When considering
	invertible (instead of $0$) parameters, the two algebras are isomorphic.
\end{rem}

\begin{example}
	An element $T_\sigma$ of $\hat{H}_n$ will be representated by a good strand diagram
	for $\sigma$.
	The multiplication of $T_\sigma$ and $T_{\sigma'}$ is obtained by
	concatenating the
	diagrams of $\sigma$ and $\sigma'$ (as in the multiplication of $\sigma$ and
	$\sigma'$). If the corresponding diagram is good, then $T_{\sigma}T_{\sigma'}=
	T_{\sigma''}$, where $\sigma''$ is represented by the concatenated diagram.
	Otherwise, $T_\sigma T_{\sigma'}=0$. For example:
$$\includegraphics[scale=0.9]{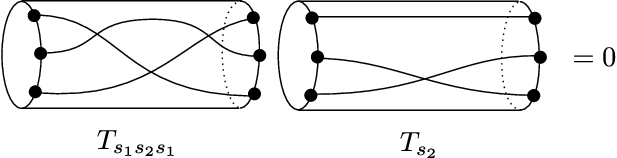}$$
	
\end{example}

\subsubsection{Positive versions}
\label{se:positive}
Let $\hat{\GS}_n^+$\indexnot{Sn+}{\hat{\GS}_n^+} be the submonoid of
$\hat{\GS}_n$ of permutations $\sigma$
such that $\sigma(\BZ_{>0})\subset\BZ_{>0}$.
Note that $\hat{\GS}_n^+$ is stable under left and right multiplication by
$\GS_n$.

There is 
a decomposition $\hat{\GS}_n^+=(\BZ_{\ge 0})^n\rtimes\GS_n$.

We have $s_{r-1}s_{r-2}\cdots s_1cs_{n-1}s_{n-2}\cdots s_r=(\underbrace{0,\ldots,0,1,0,
\ldots,0}_{\mathrm{pos. }r})\in(\BZ_{\ge 0})^n$ for $r\in\{1,\ldots,n\}$, hence
$v$ restricts to an isomorphism
from the submonoid of $W_n\rtimes\langle c\rangle$ generated by $s_1,\ldots,s_{n-1},c$ to
$\hat{\GS}_n^+$.


\medskip
Let $\hat{H}_n^+=\bigoplus_{w\in\hat{\GS}_n^+}\BF_2 T_w$,
\indexnot{Hn+}{\hat{H_n^+}}
an $\BF_2$-subspace of $\hat{H}_n$ containing $H_n$.

\begin{prop}
	\label{le:generatorsrelationsHnhat+}
	$\hat{H}_n^+$ is a differential graded subalgebra of $\hat{H}_n$.

	The algebra $\hat{H}_n^+$ has a presentation with generators $T_1,\ldots,T_{n-1},c$
	and relations
	$$T_i^2=0,\ T_iT_j=T_jT_i \text{ if }|i-j|>1,\
	T_iT_{i+1}T_i=T_{i+1}T_iT_{i+1} (\text{ if }n>2\ )$$
	$$cT_i=T_{i+1}c\text{ for }1\le i<n-1
	\text{ and }c^2T_{n-1}=T_1c^2.$$
\end{prop}

The remainder of \S\ref{se:positive} will be devoted to the proof of 
Proposition \ref{le:generatorsrelationsHnhat+}.

\smallskip
	Let $A_n$ be the $k$-algebra with generators $t_1,\ldots,t_{n-1},b$
	and relations
	$$t_i^2=0,\ t_it_j=t_jt_i \text{ if }|i-j|>1,\
	t_it_{i+1}t_i=t_{i+1}t_it_{i+1}(\text{ if }n>2\ )$$
	$$bt_i=t_{i+1}b\text{ for }1\le i<n-1
	\text{ and }b^2t_{n-1}=t_1b^2.$$

\smallskip
Given $i\in\{1,\ldots,n\}$,  we put $\beta_i=bt_{n-1}\cdots t_i$.
Given $I\subset\{1,\ldots,n\}$ non-empty with elements $1\le i_1<\cdots <i_r\le n$, we put
$\gamma_I=\beta_{i_1+r-1}\beta_{i_2+r-2}\cdots\beta_{i_r}$.
Note that $\gamma_{\{1,\ldots,n\}}=b^n$.

\smallskip
There is a morphism of algebras $H_n\to A_n,\ T_i\mapsto t_i$ and we denote by $t_w$ the 
image of $T_w$ for $w\in\GS_n$.

\begin{example}
	The elements of $\hat{\GS}_n^+$ correspond to strand diagrams where the strands wind
	positively around the cylinder. The relation $c^2T_{n-1}=T_1c^2$ is illustrated below:
	$$\includegraphics[scale=0.8]{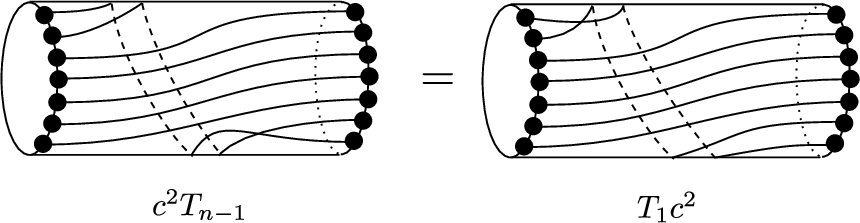}$$

	We describe some elements $w\in\hat{\GS}_7^+$ and the image of $T_w$ in $A_7$:
	$$\includegraphics[scale=0.9]{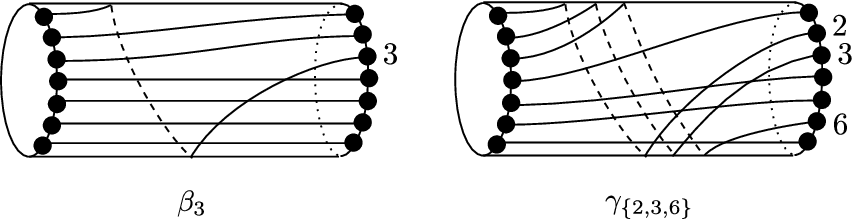}$$

	The element $(0,0,0,0,1,0,0)\in(\BZ_{\ge 0})^7$ corresponds to the following element
	of $\GS_7^+$:
	$$\includegraphics[scale=0.8]{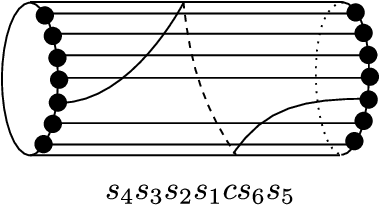}$$

\end{example}

\begin{lemma}
	\label{le:generatorAn}
	The set $\{t_w\gamma_{I_m}\cdots \gamma_{I_1}\}$ with $w\in\GS_n$, $m\ge 0$ and
	$I_1\subset\{1,\ldots,n\}$, $I_r\subset\{1,\ldots,|I_{r-1}|\}$ for $1<r\le m$
	generates $A_n$ as a $k$-vector space.
\end{lemma}

\begin{proof}
	Let $i\in\{1,\ldots,n\}$ and $j\in\{1,\ldots,n-1\}$. We have
	$$\beta_i t_j= \begin{cases}
		t_{j+1}\beta_i & \text{ if }j<i-1\\
		\beta_{i-1} & \text{ if }j=i-1 \\
		0 & \text{ if }j=i \\
		t_j\beta_i & \text{ if }j>i.
	\end{cases}$$

	Consider $I\subset\{1,\ldots,n\}$ non-empty with elements $1\le i_1<\cdots <i_r\le n$.
	We put $i_0=0$ and $i_{r+1}=n+1$.

	Consider $j\in\{1,\ldots,n-1\}$. 
	Fix $k\in\{0,\ldots,r\}$ such that $i_k\le j<i_{k+1}$.
	Let us show that
	\begin{equation}
		\label{eq:gammaT}
	\gamma_I t_j= \begin{cases}
		t_{j+r-k}\gamma_I & \text{ if }i_k<j<i_{k+1}-1 \\
		0 & \text{ if } i_k=j<i_{k+1}-1 \\
		\gamma_{\{i_1<\cdots<i_k<i_{k+1}-1<i_{k+2}<\cdots<i_r\}} & \text{ if }i_k<j=i_{k+1}-1 \\
		t_k\gamma_I & \text{ if }i_k=j=i_{k+1}-1.
	\end{cases}
	\end{equation}

	We have
	$$\gamma_I t_j=
	\beta_{i_1+r-1}\cdots\beta_{i_{k+1}+r-k-1}t_{j+r-k-1}\beta_{i_{k+2}+r-k-2}\cdots
	\beta_{i_r}.$$
	If $j<i_{k+1}-1$, then $\beta_{i_{k+1}+r-k-1}t_{j+r-k-1}=
	t_{j+r-k}\beta_{i_{k+1}+r-k-1}$ and we deduce the first two equalities in
	(\ref{eq:gammaT}). Assume now $j=i_{k+1}-1$. We have
	$\beta_{i_{k+1}+r-k-1}t_{j+r-k-1}=\beta_{i_{k+1}+r-k-2}$ and the third equality in
	(\ref{eq:gammaT}) follows.
	The last equality from the fact that given $i\in\{1,\ldots,n-1\}$, we have
	\begin{align*}
	\beta_{i+1}\beta_i&=b^2t_{n-2}\cdots t_it_{n-1}\cdots t_i=
b^2t_{n-1}\cdots t_it_{n-1}\cdots t_{i+1}=t_1b^2t_{n-2}\cdots t_it_{n-1}\cdots t_{i+1}\\
		&=t_1\beta_{i+1}^2.
	\end{align*}

	We deduce that $\gamma_I t_j=u\gamma_{I'}$ for some $I'\subset\{1,\ldots,n\}$
	with $|I'|=|I|$ and $\max(I')\le\max(I)$ and $u\in\{0,1,t_1,\ldots,t_{n-1}\}$.

	Fix $s\in\{1,\ldots,n\}$ with $s\ge \max(I)$.
	We have
	$$\gamma_I \beta_s=\begin{cases}
		\beta_r\gamma_{\{i_2-1,\ldots,i_r-1,s\}} & \text{ if }
		1\in I \\
		\gamma_{(I-1)\cup\{s\}} &\text{ otherwise.}
	\end{cases}$$

	\smallskip
	Consider $I_1,\ldots,I_m$ as in the lemma. Let $k$ be minimal such that
	$1{\not\in}I_k$. We put $k=m+1$ if there is no such $k$.
	Define $u=\gamma_{\{|I_m|\}}$ if $k=m+1$ and $u=1$ otherwise.
	Put $I_0=\{1,\ldots,n\}$. Recall that $b=\beta_n$.
	We have
	$$\gamma_{I_m}\cdots\gamma_{I_1}b=
	u\gamma_{I'_m}\cdots\gamma_{I'_1}$$
	where
	$I'_r=\{i-1|i\in I_r\setminus\{1\}\}\cup\{|I_{r-1}|\}$ for $1\le r<k$,
	$I'_k=\{i-1|i\in I_k\}\cup\{|I_{k-1}|\}$ and
	$I'_r=I_r$ for $r>k$.

	\smallskip
	We deduce that the set $B=\{t_w\gamma_{I_m}\cdots \gamma_{I_1}\}$ of the lemma
	is stable under right multiplication by $t_j$ for 
	$j\in\{1,\ldots,n-1\}$ and by $b$. Since $B$ contains $1$, it follows that $B$
	is a generating family for $A_n$ as an $\BF_2$-vector space.
\end{proof}

\begin{rem}
	An example of the description of $\gamma_It_j$ in the proof of
	Lemma \ref{le:generatorAn} is
	given below:

$$\includegraphics[scale=0.8]{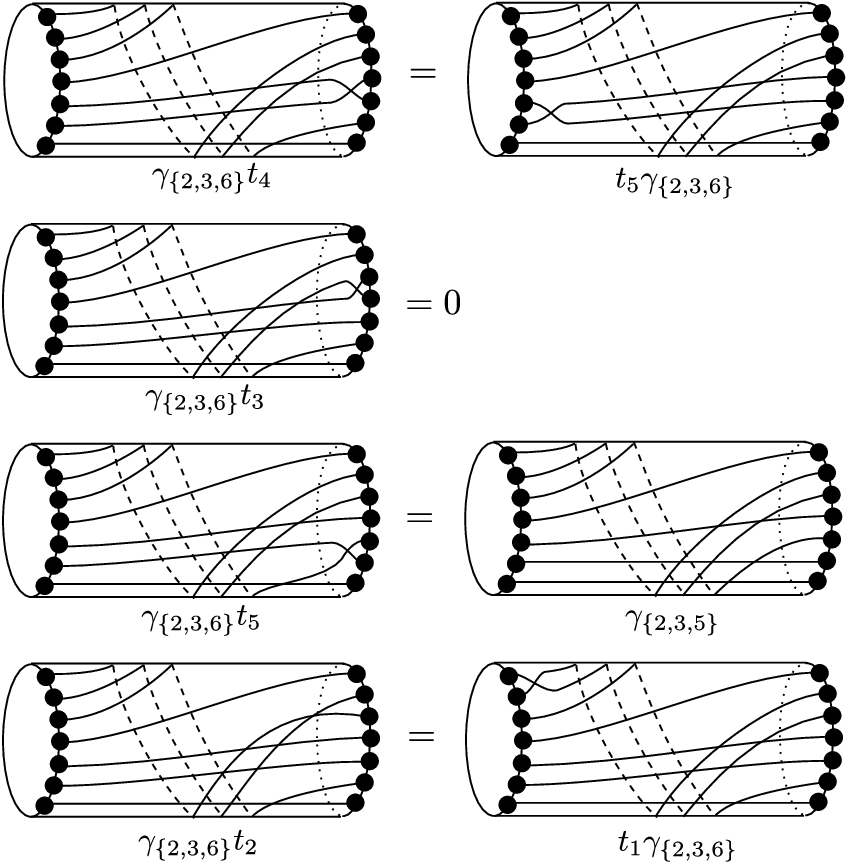}$$
\end{rem}

\begin{proof}[Proof of Proposition \ref{le:generatorsrelationsHnhat+}]
	Let $H$ be the subalgebra of $\hat{H}_n$ generated by $T_1,\ldots,T_{n-1},c$.
	This is a differential graded subalgebra of $\hat{H}_n$.
	Given $w\in\hat{\GS}_n$, let $|w|=\sum_{i=1}^n w(i)$. 
	Let $w\in\hat{\GS}_n^+$, $w\neq 1$. We show by induction on $\ell(w)+|w|$ that
	$T_w\in H$.

	Assume $\ell(ws_i)<\ell(w)$ for some $i\in\{1,\ldots,n-1\}$. We have
	$ws_i\in\hat{\GS}_n^+$ and $|ws_i|=|w|$, hence by induction $T_{ws_i}\in
	H$. We deduce that $T_w=T_{ws_i}T_i\in H$.

	Otherwise, we have $0<w(1)<\cdots<w(n)$, hence $w(n)>n$ since $w\neq 1$. It 
	follows that $wc^{-1}\in\hat{\GS}_n^+$ and $|wc^{-1}|<|w|$, hence
	$T_{wc^{-1}}\in H$ by induction. So $T_w=T_{wc^{-1}}T_c\in H$.

	We have shown that $\hat{H}_n^+\subset H$.
	Since $\hat{H}_n^+$ is stable
	under right multiplication by $T_c$ and by $T_i$ for $i\in\{1,\ldots,n-1\}$, it
	follows that $H=\hat{H}_n^+$.

	\medskip
	There is a surjective morphism of algebras $\rho:A_n\to\hat{H}_n^+,\
	t_i\mapsto T_i,\ b\mapsto c$.
	Given $I=\{i_1<\cdots<i_r\}$ a non-empty subset of $\{1,\ldots,n\}$, we put
	$$c_I=(cs_{n-1}\cdots s_{i_1+r-1})(cs_{n-1}\cdots s_{i_2+r-2})\cdots 
	(cs_{n-1}\cdots s_{i_r})\in \hat{\GS}_n.$$
	We have $c_I(i_l)=n+l$ for $1\le l\le r$ and
	$c_I(j)=j+r-k$ if $i_k<j<i_{k+1}$ (where we put $i_0=0$ and $i_{r+1}=n+1$).

	Let $E$ be the set of families $(I_1,\ldots,I_m)$ where 
	$m\ge 0$, $I_1\subset\{1,\ldots,n\}$ and $I_r\subset\{1,\ldots,|I_{r-1}|\}$
	for $1<r\le m$.

	Given $w\in\GS_n$ and $(I_1,\ldots,I_m)\in E$, we have
	$\rho(t_w\gamma_{I_1}\cdots\gamma_{I_m})=T_wT_{c_{I_1}}\cdots T_{c_{I_m}}$
	and that element is either $T_{wc_{I_1}\cdots c_{I_m}}$ or $0$.

	We define a map $\phi:(\BZ_{\ge 0})^n\to E$. Let $a\in(\BZ_{\ge 0})^n$.
	Let $m=\max\{a(i)\}_{1\le i\le n}$. We put $I_1=a^{-1}(\BZ_{\ge 1})$ and we
	define inductively $I_r$ for $2\le r\le m$ by 
	$I_r=c_{I_{r-1}}\cdots c_{I_1}(a^{-1}(\BZ_{\ge r}))$. 
	We put $\phi(a)=(I_1,\ldots,I_m)$.
	We have 
	$$c_{I_m}\cdots c_{I_1}(i)=na(i)+|a^{-1}(\BZ_{>a(i)})|+
	(\text{position of }i\text{ in }a^{-1}(a(i))).$$

	We define a map $\psi:E\to (\BZ_{\ge 0})^n$. Let $(I_1,\ldots,I_m)\in E$.
	We define $a\in(\BZ_{\ge 0})^n$ by $a(i)=\lfloor\frac{c_{I_m}\cdots c_{I_1}(i)-1}{n}
	\rfloor$ and we put $\psi(I_1,\ldots,I_m)=a$. The maps $\psi$ and $\phi$ are
	inverse bijections. We deduce that
	the map $E\to (\GS_n\setminus \hat{\GS}_n^+)$ sending $(I_1,\ldots,I_m)$
	to the class of $c_{I_m}\cdots c_{I_1}$ is bijective. It follows that
	the map $\GS_n\times E\to\hat{\GS}_n^+,\ (w,(I_1,\ldots,I_m))\mapsto 
	wc_{I_m}\cdots c_{I_1}$ is bijective.

	If $\rho(t_w\gamma_{I_m}\cdots\gamma_{I_1})=
	T_{wc_{I_m}\cdots c_{I_1}}=0$ for some $w\in\GS_n$ and $(I_1,\ldots,I_m)\in E$,
	then the bijectivty of the map above shows that the image of
	$\rho$ is the span of a proper subset of a basis of $\hat{H}_n^+$, contradicting
	the surjectivity of $\rho$.

	This shows that the elements $\rho(t_w\gamma_{I_m}\cdots\gamma_{I_1})$ are
	distinct basis elements of $\hat{H}_n^+$, hence $\rho$ is an isomorphism.
\end{proof}

\begin{rem}
	The same method as the one used in the proof of Proposition 
	\ref{le:generatorsrelationsHnhat+} shows that
	$\hat{\GS}_n^+$ is the free $(\GS_n,\GS_n)$-monoid
on a generator $c$ with relations $c\cdot s_r=s_{r+1}\cdot c$ for $r\in\{1,\ldots,n-1\}$ and
$c^2\cdot s_{n-1}=s_1\cdot c^2$.
\end{rem}

\subsubsection{Pointed versions}
Given $n\ge 0$, we put
$H_n^\bullet=(\GS_n)^{\mathrm{nil}}$\indexnot{Hn}{H_n^\bullet}. This
is the quotient of the free pointed monoid generated by $T_1,\ldots,T_{n-1}$
by the relations (\ref{eq:relationsHn}). 
The differential is given by $d(T_i)=1$.
Note that $k[H_n^\bullet]=H_n$ and
$H_n^\bullet=\{0\}\cup\{T_w\}_{w\in\GS_n}$.

\smallskip
We define $\hat{\GS}_n^{\nil}$\indexnot{Sn}{\hat{\GS}_n^{\nil}}
to be the differential graded pointed monoid with
underlying differential pointed set $\{T_\sigma\}_{\sigma\in\hat{\GS}_n}\coprod\{0\}$
and multiplication, grading and differential that of $\hat{H}_n$.

We define $\hat{\GS}_n^{+,\nil}$\indexnot{Sn}{\hat{\GS}_n^{+,\nil}}
to be its differential graded pointed submonoid with non-zero
elements those that stabilize $\BZ_{>0}$.

\section{$2$-representation theory}
\label{se:2reptheory}
We recall that $k$ is a field of characteristic $2$.

\subsection{Monoidal category}
\subsubsection{Definition}
\label{se:defmoncat}
Let $\CU$\indexnot{U}{\CU} be the differential strict monoidal category generated by an object $e$ and a map $\tau:e^2\to e^2$ \indexnot{\tau}{\tau} subject to the
relations
\begin{equation}
\label{eq:definingtau}
d(\tau)=1,\ \tau^2=0 \text{ and }e\tau\circ \tau e\circ e\tau=
\tau e\circ e\tau \circ \tau e.
\end{equation}

There are isomorphisms of differential monoidal categories
$\opp:\CU\iso\CU^\opp$ and $\mathrm{rev}:\CU\iso\CU^{\mathrm{rev}}$\indexnot{opp}{\opp}
\indexnot{rev}{\mathrm{rev}}given on generators by
$e\mapsto e$ and $\tau\mapsto\tau$.

\medskip
The following result is clear.

\begin{prop}
The objects of the category $\CU$ are the $e^n$, $n\ge 0$. We have
$\Hom(e^n,e^m)=0$ if $n\neq m$ and there is an isomorphism of differential
algebras 
$$H_n\xrightarrow{\sim}\End(e^n),\ T_i\mapsto e^{i-1}\tau e^{n-i-1}.$$
There is a commutative diagram
$$\xymatrix{
	H_m\otimes H_n\ar[r]^-{T_i\otimes T_j\mapsto T_iT_{m+j}}
	\ar[d]_{\can}^\sim & H_{m+n}\ar[d]_\sim^{\can} \\
	\End(E^m)\otimes\End(E^n)\ar[r]_-{\otimes} & \End(E^{m+n})
	}$$
\end{prop}

The isomorphism $\opp:\CU\iso\CU^\opp$ gives rise to the
isomorphism of differential algebras \indexnot{opp}{\opp}
$$\opp:H_n\xrightarrow{\sim}H_n^\opp,\ T_i\mapsto T_i.$$

The isomorphism $\mathrm{rev}:\CU\iso\CU^{\mathrm{rev}}$ gives rise to the
isomorphism of differential algebras
$$\iota_n:H_n\xrightarrow{\sim}H_n,\ T_i\mapsto T_{n-i}\indexnot{iota_n}{\iota_n}.$$

The functor $-\otimes E^n$ induces an injective morphism of differential algebras
$H_r=\End(E^r)\to H_{r+n}=\End(E^{r+n}),\ T_i\mapsto T_i$ and we will identify $H_r$ with a subalgebra
of $H_{r+n}$ via this morphism.

The functor $E^n\otimes -$ induces a morphism of differential algebras
$$f_n:H_r=\End(E^r)\to H_{n+r}=\End(E^{n+r}),\ T_i\mapsto T_{n+i}\indexnot{f_n}{f_n}.$$

Note that $H_n$ commutes with $f_n(H_r)$ and that
$f_n=\iota_{n+r}\circ\iota_r$.

\subsubsection{$2$-representations}
Let $\CV$ be a differential category.

\begin{defi}
A {\em $2$-representation}\index[ter]{$2$-representation on a differential category} on $\CV$ is the data
of a strict monoidal differential functor $\CU\to\End(\CV)$.
\end{defi}

The data of a
$2$-representation on $\CV$ is the same as the data of a
differential endofunctor $E$ of $\CV$ and of $\tau=\tau_E\in\End(E^2)$ 
\indexnot{\tau}{\tau} \indexnot{\tau_E}{\tau_E}
satisfying
(\ref{eq:definingtau}).

Note that a $2$-representation on $\CV$ extends to a $2$-representation
on $\bar{\CV}$ and on $\CV^i$
(uniquely up to an equivalence unique up to isomorphism).

\smallskip
A {\em morphism of $2$-representations}\index[ter]{morphism of $2$-representations}
$(\CV,E,\tau)\to(\CV',E',\tau)$ is the data of a
differential functor $\Phi:\CV\to\CV'$ and of an isomorphism
of functors $\varphi:\Phi E\iso E'\Phi$ (with $d(\varphi)=0$) such that
$\tau'\Phi\circ E'\varphi\circ \varphi E=E'\varphi\circ \varphi E\circ \Phi\tau:
\Phi E^2\to E^{\prime 2}\Phi$.

\begin{example}
Let $\CV=k\mdiff$ and $E=\tau=0$. This is the ``trivial''
$2$-representation.
\end{example}

Let $\CV$ be a $2$-representation. The {\em opposite} $2$-representation\index[ter]{opposite
$2$-representation}
is $(\CV\mdiff,E',\tau')$, where
$E'(\zeta)=\zeta E$ and $\tau'(\zeta)=\zeta\tau\in\End(E^{\prime 2}(\zeta))$
for $\zeta\in\CV\mdiff$. Note that
the canonical functor $\CV\to (\CV\mdiff)\mdiff,\ v\mapsto (\zeta\mapsto
\zeta(v))$ is a fully faithful morphism of $2$-representations.

\smallskip
Assume $E$ has a left adjoint $E^\vee$. We still denote by $\tau$ the endomorphism of
$(E^\vee)^2$ corresponding to $\tau$ (cf \S\ref{se:cat}).  The pair $(E^\vee,\tau)$ defines the
{\em left dual $2$-representation}\index[ter]{left dual $2$-representation} of $(E,\tau)$.
Similarly, if $E$ has a right adjoint ${^\vee E}$, we obtain a 
{\em right dual $2$-representation}\index[ter]{right dual $2$-representation} 
$({^\vee E},\tau)$ of $(E,\tau)$.

\begin{rem}
	One can also consider a {\em lax $2$-representation} on $\CV$: this is the data
	of a lax monoidal differential functor $\CU\to\End(\CV)$.
\end{rem}

\begin{rem}
	The category $\CU$ has a structure of differential graded monoidal category
	with $\tau$ in degree $-1$ and one can consider (lax) $2$-representations on
	differential graded categories.
\end{rem}

\subsubsection{Pointed case}

We denote by $\CU^\bullet$\indexnot{U}{\CU^\bullet} the strict monoidal differential pointed category generated
by an object $e$ and a map $\tau\in\End(e^2)$ subject to the relations 
(\ref{eq:definingtau}). Its objects are the $e^n$, $n\ge 0$,
$\Hom(e^n,e^m)=0$ for $m\neq n$ and $\End(e^n)=H_n^\bullet$.

\medskip
Let $\CV$ be a differential pointed category.

A {\em $2$-representation} on $\CV$\index[ter]{$2$-representation on a differential pointed
category} is the data of a strict monoidal
differential pointed functor $\CU^\bullet\to\End(\CV)$.
This is equivalent to the data of an endofunctor $E$ of the differential pointed category
$\CV$ and $\tau\in\End(E^2)$ such that 
$(E,\tau)$ induce a $2$-representation on $k[\CV]$.

\subsection{Lax cocenter}
\subsubsection{Lax bi-$2$-representations}
\label{se:2birep}
A {\em lax bi-$2$-representation} on $\CV$ \index[ter]{lax bi-$2$-representation} is a lax monoidal differential
functor $E:\CU\otimes\CU\to\End(\CV)$. It corresponds to the data
of
\begin{itemize}
	\item differential endofunctors $E_{i,j}=E(e^i\otimes e^j)$ of $\CV$
	\item morphisms of differential algebras $H_i\otimes H_j\to
		\End(E_{i,j})$
	\item morphisms of differential functors $\mu_{(i,j),(i',j')}:E_{i,j} E_{i',j'}\to E_{i+i',j+j'}$
\end{itemize}
such that
\begin{enumerate}
	\item
$\mu_{(i,j),(i',j')}$ is equivariant for the action of $(H_i\otimes H_j)\otimes (H_{i'}\otimes H_{j'})$, where
the action on $E_{i+i',j+j'}$ is the restriction of the action of $H_{i+i'}\otimes H_{j+j'}$ via
the morphism $(a\otimes b)\otimes (a'\otimes b')\mapsto a f_i(a')\otimes b f_j(b')$
\item $\mu_{(i+i',j+j'),(i'',j'')}\circ(\mu_{(i,j),(i',j')}E_{i'',j''})=
	\mu_{(i,j),(i'+i'',j'+j'')}\circ (E_{i,j}\mu_{(i',j'),(i'',j'')})$.
\end{enumerate}

\medskip
Consider two actions of $\CU$ given by $(F_1,\tau_1)$ and
$(E_2,\tau_2)$ on $\CV$ and a closed morphism of functors
$\lambda:F_1E_2\to E_2F_1$ such that
the following diagrams commute:
\begin{equation}
	\label{eq:diaglambda}
\xymatrix{
	F_1^2E_2\ar[r]^-{F_1\lambda}\ar[d]_{\tau_1E_2} & F_1E_2F_1
	\ar[r]^-{\lambda F_1} &
E_2F_1^2\ar[d]^{E_2\tau_1} \\
F_1^2E_2\ar[r]_-{F_1\lambda} & F_1E_2F_1\ar[r]_-{\lambda F_1} &
E_2F_1^2
}
\ \ \ \ \
\xymatrix{
F_1E_2^2\ar[r]^-{\lambda E_2}\ar[d]_{F_1 \tau_2} & E_2F_1E_2
\ar[r]^-{E_2\lambda} & E_2^2F_1\ar[d]^{\tau_2F_1} \\
F_1E_2^2\ar[r]_-{\lambda E_2} & E_2F_1E_2
\ar[r]_-{E_2\lambda} & E_2^2F_1
}
\end{equation}

\begin{rem}
	The data of $\lambda$ and the required relations
	are described graphically as:
$$\includegraphics[scale=0.9]{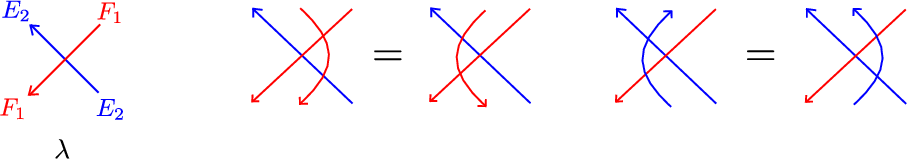}$$
\end{rem}

Define morphisms
$$\lambda_{i,1}=(\lambda F_1^{i-1})\circ\cdots\circ (F_1^{i-2}\lambda F_1)
\circ(F_1^{i-1}\lambda): F_1^iE_2\to E_2F_1^i$$
and
$$\lambda_{i,j}=(E_2^{j-1}\lambda_{i,1})\circ\cdots\circ(E_2\lambda_{i,1}E_2^{j-2})\circ
(\lambda_{i,1}E_2^{j-1}) :F_1^i E_2^j\to E_2^jF_1^i.$$

We define a lax bi-$2$-representation on $\CV$ by
$E_{i,j}=E_2^iF_1^j$. The actions of $H_i$ on $E_2^i$ and $H_j$ on $F_1^j$
provide an action of $H_i\otimes H_j$ on $E_{i,j}$ and 
$\mu_{(i,j),(i',j')}=E_2^i\lambda_{j,i'}F_1^{j'}$:
$$\mu_{(i,j),(i',j')}:
E_2^iF_1^jE_2^{i'}F_1^{j'}\xrightarrow{E_2^i\lambda_{j,i'}F_1^{j'}}
E_2^iE_2^{i'}F_1^jF_1^{j'}=E_2^{i+i'}F_1^{j+j'}.$$

\begin{rem}
	One can also consider the notion of colax $2$-representation.
	A colax $2$-representation on $\CV$ is the same data
	as a lax $2$-representation on $\CV^\opp$.
\end{rem}

\subsubsection{Category}
Let $\CW$ be a differential category endowed with a lax action $(E_{i,j})$
of $\CU^2$.

We define a differential category
$\Delta_E\CW$\indexnot{Delta}{\Delta_E\CW}.

\smallskip
$\bullet\ $ The objects of $\Delta_E\CW$ are pairs $(m,\varsigma)$ where
$m\in \overline{\CW}^i$ and
$\varsigma\in Z\Hom_{\overline{\CW}^i}(E_{0,1}E_{1,0}(m),m)$
such that for all $i\ge 1$, there exists $\varsigma_i\in
Z\Hom_{\overline{\CW}^i}(E_{i,i}(m),m)$ such that
the composition $b_i$
\begin{equation}
	\label{eq:compbeta}
	b_i:(E_{0,1}E_{1,0})^i(m)\xrightarrow{(E_{0,1}E_{1,0})^{i-1}\varsigma}
(E_{0,1}E_{1,0})^{i-1}(m)\xrightarrow{(E_{0,1}E_{1,0})^{i-2}\varsigma}\cdots
\to E_{0,1}E_{1,0}(m)\xrightarrow{\varsigma}m
\end{equation}
is equal to 
$$(E_{0,1}E_{1,0})^i(m)\xrightarrow{\can}E_{i,i}(m)\xrightarrow{\varsigma_i}m$$
and $\varsigma_i\circ (T_r \otimes 1)=\varsigma_i\circ (1\otimes T_r)$ for $1\le r<i$.

\smallskip
$\bullet\ $ $\Hom_{\Delta_E\CW}((m,\varsigma),(m',\varsigma'(m))$ is the
differential submodule of $\Hom_{\overline{\CW}^i}(m,m')$ of
elements
$f$ such that the following diagram commutes
$$\xymatrix{
	E_{0,1}E_{1,0}(m)\ar[r]^-{\varsigma}\ar[d]_{E_{0,1}E_{1,0}f} &m \ar[d]^{f} \\
	E_{0,1}E_{1,0}(m')\ar[r]^-{\varsigma} &m'
}$$

The composition of maps is defined by restricting that of 
$\overline{\CW}^i$.
So, we have a faithful forgetful
functor $\omega:\Delta_E\CW\to\overline{\CW}^i,\ (m,\varsigma)\mapsto m$.
Note that $\Delta_E\CW$ is strongly pretriangulated and idempotent-complete.

\begin{rem}
	\label{re:swapUU}
%
	Note that applying the self-equivalence $(a,b)\mapsto (b,a)$ of $\CU^2$ provides
	another lax action $E'$ of $\CU^2$ on $\CW$. The corresponding differential
	category $\Delta_{E'}\CW$ is not equivalent to $\Delta_E\CW$ in general.
\end{rem}

\subsection{Diagonal action}
\label{se:diagonal}

\subsubsection{Category}
Consider a differential category $\CW$ endowed with
two actions of $\CU$ given by $(E_1,\tau_1)$ and
$(E_2,\tau_2)$ and a closed morphism of functors
$\sigma:E_2E_1\to E_1E_2$ such that the following diagrams commute:
\begin{equation}
	\label{eq:diagsigma}
\xymatrix{
E_2^2E_1\ar[r]^-{E_2\sigma}\ar[d]_{\tau_2E_1} & E_2E_1E_2\ar[r]^-{\sigma E_2} &
E_1E_2^2\ar[d]^{E_1\tau_2} \\
E_2^2E_1\ar[r]_-{E_2\sigma} & E_2E_1E_2\ar[r]_-{\sigma E_2} &
E_1E_2^2
}
\ \ \ \ \
\xymatrix{
E_2E_1^2\ar[r]^-{\sigma E_1}\ar[d]_{E_2 \tau_1} & E_1E_2E_1
\ar[r]^-{E_1\sigma} & E_1^2E_2\ar[d]^{\tau_1E_2} \\
E_2E_1^2\ar[r]_-{\sigma E_1} & E_1E_2E_1
\ar[r]_-{E_1\sigma} & E_1^2E_2
}
\end{equation}

\begin{rem}
	The data of $\sigma$ and the relations can be described graphically as follows:
$$\includegraphics[scale=0.9]{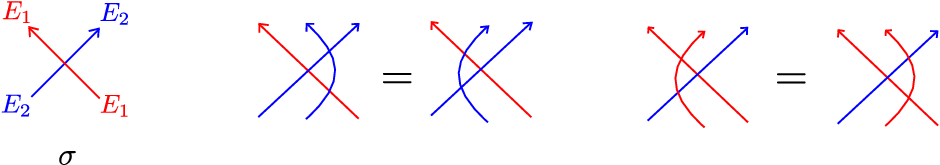}$$
\end{rem}

We define a differential category
$\CV=\Delta_\sigma\CW$\indexnot{Delta}{\Delta_\sigma\CW}.

\smallskip
$\bullet\ $ The objects of $\CV$ are pairs $(m,\pi)$ where
$m\in \overline{\CW}^i$ and
$\pi\in Z\Hom_{\overline{\CW}^i}(E_2(m),E_1(m))$
such that the following diagram commutes
\begin{equation}
	\label{eq:Deltasigmacommutationtau}
\xymatrix{
E_2^2(m)\ar[rr]^-{E_2\pi} \ar[d]_{\tau_2} &&
 E_2E_1(m)\ar[r]^-{\sigma} & E_1E_2(m)
\ar[rr]^-{E_1\pi} && E_1^2(m)\ar[d]^{\tau_1} \\
E_2^2(m)\ar[rr]_-{E_2\pi} && E_2E_1(m)\ar[r]_-{\sigma} & E_1E_2(m)
\ar[rr]_-{E_1\pi} && E_1^2(m)
}
\end{equation}

\smallskip
$\bullet\ $ $\Hom_\CV((m,\pi),(m',\pi'))$ is the
differential submodule of $\Hom_{\overline{\CW}^i}(m,m')$ of
elements
$f$ such that the following diagram commutes
$$\xymatrix{
E_2(m)\ar[r]^-{\pi}\ar[d]_{E_2f} & E_1(m)
\ar[d]^{E_1f} \\
E_2(m')\ar[r]_-{\pi'} & E_1(m')\\
}$$

The composition of maps is defined by restricting that of 
$\overline{\CW}^i$.
So, we have a faithful forgetful
functor $\omega=\omega_\sigma:\CV\to\overline{\CW}^i,\ (m,\pi)\mapsto m$.
Note that $\CV$ is strongly pretriangulated and idempotent-complete.

\begin{rem}
	The structure of objects and maps in $\CV$ can be described graphically as follows:
$$\includegraphics[scale=0.65]{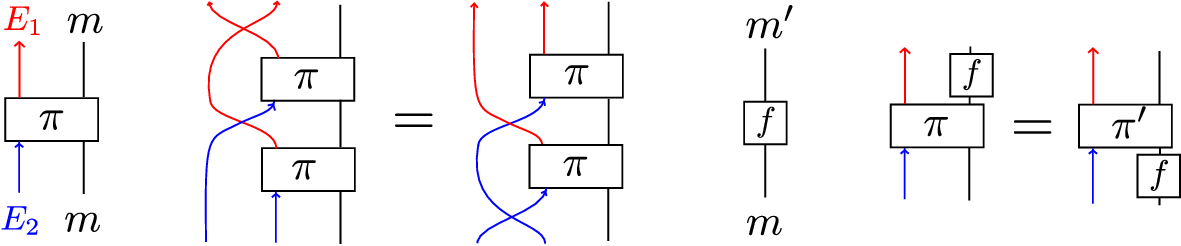}$$
\end{rem}

\begin{rem}
	\label{re:sigmatobirep}
	Assume $E_1$ admits a left adjoint $F_1$. The data of the map $\sigma:E_2E_1\to E_1E_2$ 
	corresponds by adjunction to the data of a map 
$$\lambda:
	F_1 E_2\xrightarrow{\bullet\eta_1}F_1 E_2E_1F_1
\xrightarrow{F_1\sigma F_1}F_1 E_1E_2F_1\xrightarrow{\eps_1\bullet} E_2F_1.$$
	The commutativity of the diagrams (\ref{eq:diagsigma}) is equivalent to the commutativity of the
	diagrams (\ref{eq:diaglambda}).
	Assume the diagrams commute. We obtain a lax bi-$2$-representation $(E_{i,j})$ on
	$\CW$ (cf \S\ref{se:2birep}).
	
	Let $(m,\varsigma)\in \Delta_E\CW$. We have an adjunction isomorphism
	$$\phi:\Hom(E_2(m),E_1(m))\iso \Hom(F_1E_2(m),m).$$
	Let $\pi=\phi^{-1}(\varsigma)\in Z\Hom(E_2(m),E_1(m))$.
	The object $(m,\pi)$ is in $\Delta_\sigma\CW$ and
	$(m,\varsigma)\mapsto(m,\pi)$ defines a fully faithful functor of differential
	categories $\Delta_E\CW\to \Delta_\sigma\CW$.

	\smallskip
	Assume now $\lambda$ is invertible. The canonical map $f_i:(E_{0,1}E_{1,0})^i\to E_{i,i}$ is
	invertible.  Let $\varsigma_i=b_i\circ f_i^{-1}$. Consider $r\in\{1,\ldots,i-1\}$.
	We have
	\begin{align*}
		\varsigma_i\circ (T_r\otimes 1)&=b_{r-1}\circ (E_{01}E_{10})^{r-1}(\varsigma_2\circ
	(T_1\otimes 1)\circ f_2)\circ (E_{01}E_{10})^{r+1}b_{i-r-1}\circ f_i^{-1}\\
		&=b_{r-1}\circ (E_{01}E_{10})^{r-1}(\varsigma_2\circ
	(1\otimes T_1)\circ f_2)\circ (E_{01}E_{10})^{r+1}b_{i-r-1}\circ f_i^{-1}\\
		&=\varsigma_i\circ (1\otimes T_r)
	\end{align*}
	As a consequence, the functor above is an isomorphism of differential categories
	$\Delta_E\CW\iso \Delta_\sigma\CW$.
\end{rem}

\subsubsection{$1$-arrows}
\label{se:1arrows}
We define now a differential functor $E:\CV\to\CV$.

$\bullet\ $ Let $(m,\pi)\in\CV$.
Let $m'={\xy (0,0)*{E_2(m)\oplus E_1(m)},
\ar@/^/^{\pi}(-8,3)*{};(7,3)*{}, \endxy}$. We define
$$\pi'=\left(\begin{matrix}\sigma\circ E_2\pi\circ \tau_2 &
\sigma \\0 & \tau_1\circ E_1\pi\circ\sigma\end{matrix}\right):
E_2(m')\to E_1(m')$$
$$
{\xy 
(-50,-10)*{\pi':},
(0,0)*{E_2^2(m)\oplus E_2E_1(m)},
(0,-20)*{E_1E_2(m)\oplus E_1^2(m)},
\ar@/^/^{E_2\pi}(-8,3)*{};(7,3)*{},
\ar@/_/_{E_1\pi}(-8,-23)*{};(7,-23)*{},
\ar_{\sigma\circ E_2\pi\circ \tau_2}(-12,-3)*{};(-12,-17)*{},
\ar^{\tau_1\circ E_1\pi\circ\sigma}(11,-3)*{};(11,-17)*{},
\ar_{\sigma}(10,-3)*{};(-11,-17)*{},
 \endxy}
$$

\begin{rem}
	The graphical description of $\pi'$ is the following:
$$\includegraphics[scale=0.7]{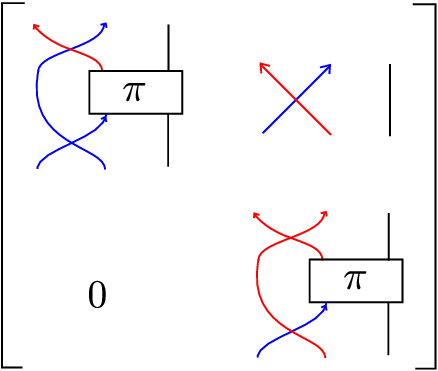}$$
\end{rem}

\begin{lemma}
	\label{le:dualisobject}
$(m',\pi')$ is an object of $\CV$.
\end{lemma}

\begin{proof}
Note that $d(\pi')=0$.

	Let $a=\tau_1\circ E_1\pi'\circ \sigma(m')\circ E_2\pi'$
	and $b=E_1\pi'\circ \sigma(m')\circ E_2\pi'\circ \tau_2$.
We have 
\begin{align*}
a_{11}&=
\tau_1E_2\circ E_1\sigma\circ E_1E_2\pi\circ E_1\tau_2\circ \sigma E_2\circ
	E_2\sigma\circ E_2^2\pi\circ E_2\tau_2\\
	&=\tau_1E_2\circ E_1\sigma\circ E_1E_2\pi\circ\sigma E_2\circ E_2\sigma\circ
	\tau_2 E_1\circ E_2^2\pi\circ E_2\tau_2\\
	&=\tau_1E_2\circ E_1\sigma\circ \sigma E_1\circ E_2E_1\pi\circ E_2\sigma\circ
	E_2^2\pi\circ \tau_2E_2\circ E_2\tau_2\\
	&=E_1\sigma\circ \sigma E_1\circ E_2\tau_1\circ E_2E_1\pi\circ E_2\sigma\circ
	E_2^2\pi\circ \tau_2E_2\circ E_2\tau_2\\
	&=E_1\sigma\circ \sigma E_1\circ E_2E_1\pi\circ E_2\sigma\circ
	E_2^2\pi\circ E_2\tau_2\circ \tau_2 E_2\circ E_2\tau_2\\
	&=E_1\sigma\circ E_1E_2\pi\circ\sigma E_2\circ E_2\sigma\circ
	E_2^2\pi\circ \tau_2 E_2\circ E_2\tau_2\circ \tau_2 E_2\\
	&=E_1\sigma\circ E_1E_2\pi\circ\sigma E_2\circ E_2\sigma\circ
	\tau_2 E_1\circ E_2^2\pi\circ  E_2\tau_2\circ \tau_2 E_2\\
&=E_1\sigma\circ
	E_1E_2\pi\circ E_1\tau_2\circ \sigma E_2\circ E_2\sigma\circ E_2^2\pi\circ
E_2\tau_2\circ \tau_2E_2=b_{11},
\end{align*}

\begin{align*}
a_{12}&=
\tau_1E_2\circ E_1\sigma\circ E_1E_2\pi\circ E_1\tau_2\circ \sigma E_2\circ E_2\sigma+
\tau_1E_2\circ E_1\sigma\circ \sigma E_1\circ E_2\tau_1\circ E_2E_1\pi\circ 
	E_2\sigma\\
	&=\tau_1E_2\circ E_1\sigma\circ E_1E_2\pi\circ \sigma E_2\circ E_2\sigma\circ
	\tau_2 E_1+
\tau_1^2E_2\circ E_1\sigma\circ \sigma E_1\circ E_2E_1\pi\circ E_2\sigma\\
	&=\tau_1E_2\circ E_1\sigma\circ \sigma E_1\circ E_2E_1\pi\circ E_2\sigma\circ
	\tau_2 E_1 \\
&=
E_1\sigma\circ E_1E_2\pi\circ \sigma E_2\circ E_2\sigma\circ\tau_2^2 E_1
	+ E_1\sigma\circ \sigma E_1\circ E_2\tau_1\circ E_2E_1\pi\circ E_2\sigma
	\circ \tau_2 E_1\\
&=
E_1\sigma\circ E_1E_2\pi\circ E_1\tau_2\circ \sigma E_2\circ E_2\sigma\circ\tau_2 E_1
	+ E_1\sigma\circ \sigma E_1\circ E_2\tau_1\circ E_2E_1\pi\circ E_2\sigma
	\circ \tau_2 E_1
=b_{12},
\end{align*}

$$a_{21}=0=b_{21} \text{ and }$$
\begin{align*}
a_{22}&=
\tau_1 E_1\circ E_1\tau_1\circ E_1^2\pi\circ E_1\sigma\circ \sigma E_1\circ E_2\tau_1\circ
E_2E_1\pi\circ E_2\sigma\\
	&=
\tau_1 E_1\circ E_1\tau_1\circ E_1^2\pi\circ E_1\sigma\circ \sigma E_1\circ E_2\tau_1\circ
E_2E_1\pi\circ E_2\sigma\\
	&=\tau_1 E_1\circ E_1\tau_1\circ \tau_1 E_1\circ E_1^2\pi\circ E_1\sigma
\circ E_1E_2\pi\circ \sigma E_2\circ E_2\sigma\\
	&=E_1\tau_1\circ \tau_1 E_1\circ E_1\tau_1\circ  E_1^2\pi\circ E_1\sigma
\circ E_1E_2\pi\circ \sigma E_2\circ E_2\sigma\\
	&=E_1\tau_1\circ \tau_1 E_1\circ E_1^2\pi\circ E_1\sigma
\circ E_1E_2\pi\circ E_1\tau_2\circ \sigma E_2\circ E_2\sigma\\
	&=E_1\tau_1\circ E_1^2\pi\circ \tau_1 E_2\circ E_1\sigma
\circ E_1E_2\pi\circ \sigma E_2\circ E_2\sigma\circ \tau_2 E_1\\
	&=E_1\tau_1\circ E_1^2\pi\circ \tau_1 E_2\circ E_1\sigma
\circ \sigma E_1\circ E_2E_1\pi\circ E_2\sigma\circ \tau_2 E_1\\
	&=E_1\tau_1\circ E_1^2\pi\circ E_1\sigma
\circ \sigma E_1\circ E_2\tau_1\circ E_2E_1\pi\circ E_2\sigma\circ \tau_2 E_1
	=b_{22}.
\end{align*}
The lemma follows.
\end{proof}

\begin{rem}
	The equalities established in the proof of the lemma above have the following
	graphical description:
$$\includegraphics[scale=0.85]{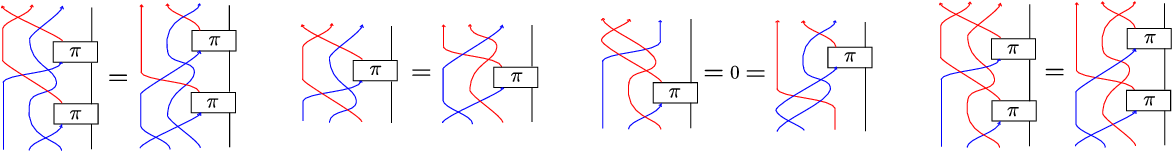}$$
\end{rem}

 We put
$E(m,\pi)=(m',\pi')$.

$\bullet\ $ Given $f\in\Hom_{\CV}((m,\pi),(\tilde{m},\tilde{\pi}))$,
we put $E(f)=\left(\begin{matrix}E_2f&0\\0&E_1f
\end{matrix}\right)$:

$$
{\xy (0,0)*{E_2(m)\oplus E_1(m)},
(0,-20)*{E_2(\tilde{m})\oplus E_1(\tilde{m})},
\ar@/^/^{\pi}(-8,3)*{};(7,3)*{},
\ar@/_/_{\tilde{\pi}}(-8,-23)*{};(7,-23)*{},
\ar_{E_2f}(-12,-3)*{};(-12,-17)*{},
\ar^{E_1f}(11,-3)*{};(11,-17)*{},
 \endxy}
$$

\begin{lemma}
	\label{le:dualEfunctor}
	We have $E(f)\in \Hom_{\CV}(E(m,\pi),E(\tilde{m},\tilde{\pi}))$.
The construction makes $E$ into a differential endofunctor of $\CV$.
\end{lemma}

\begin{proof}
	The lemma follows from the commutativity of the following diagram:

	$${\xy
		(0,0)*{E_2^2(m) \oplus E_2E_1(m)},
		(60,0)*{E_1E_2(m)\oplus E_1^2(m)},
		(0,-20)*{E_2^2(\tilde{m}) \oplus E_2E_1(\tilde{m})},
		(60,-20)*{E_1E_2(\tilde{m})\oplus E_1^2(\tilde{m})},
		\ar@/^2pc/^{\sigma\circ E_2\pi\circ\tau_2}(-8,3)*{};(50,3)*{},
		\ar@/^2pc/^{\tau_1\circ E_1\pi\circ\sigma}(8,3)*{};(68,3)*{},
		\ar@/^1pc/_{\sigma}(8,3)*{};(50,3)*{},
		\ar_{E_2^2f}(-10,-3)*{};(-10,-17)*{},
		\ar^{E_2E_1f}(8,-3)*{};(8,-17)*{},
		\ar_{E_1E_2f}(50,-3)*{};(50,-17)*{},
		\ar^{E_1^2f}(68,-3)*{};(68,-17)*{},
		\ar@/_2pc/_{\sigma\circ E_2\tilde{\pi}\circ\tau_2}(-8,-23)*{};(50,-23)*{},
		\ar@/_2pc/_{\tau_1\circ E_1\tilde{\pi}\circ\sigma}(8,-23)*{};(68,-23)*{},
		\ar@/_1pc/^{\sigma}(8,-23)*{};(50,-23)*{},
		\endxy}$$
\end{proof}

\subsubsection{$2$-arrows}
\label{se:diagonalaction}

We assume in \S\ref{se:diagonalaction} that $\sigma$ is invertible.

We define an endomorphism $\tau$ of $\omega E^2$.
Let $(m,\pi)\in\CV$. We have
$E^2(m,\pi)=(m'',\pi'')$
where $m''=[E_2^2(m)\oplus E_2E_1(m)\oplus
E_1E_2(m)\oplus E_1^2(m),\partial]$
and
$$\partial=\left(\begin{matrix}
0 \\
E_2\pi & 0 \\
\sigma\circ E_2\pi\circ \tau_2&\sigma & 0 \\
0 & \tau_1\circ E_1\pi\circ \sigma &
E_1\pi & 0 \end{matrix}\right).$$

We define an endomorphism $\tau$ of $m''$ by
\begin{equation}
	\label{eq:firstdeftau}
	\tau=\left(\begin{matrix}
\tau_2 &0&0&0\\
	0&0&\sigma^{-1}&0\\
0&0&0&0\\
0&0&0&\tau_1
\end{matrix}\right).
\end{equation}

\begin{thm}
	\label{th:tauisendo}
The endomorphism $\tau$ of $m''$
defines an endomorphism of $E^2$.
The data $(\Delta_\sigma\CW,E,\tau)$ is an idempotent-complete
strongly pretriangulated $2$-representation.
\end{thm}

\begin{proof}
The non-zero coefficients of $\pi''$ are
\begin{align*}
\pi''_{11}&=\sigma E_2\circ E_2\sigma\circ E_2^2\pi\circ E_2\tau_2\circ\tau_2E_2\\
\pi''_{22}&=\sigma E_1\circ E_2\tau_1\circ E_2E_1\pi\circ E_2\sigma\circ \tau_2E_1\\
\pi''_{33}&=\tau_1E_2\circ E_1\sigma\circ E_1E_2\pi\circ E_1\tau_2\circ \sigma E_2\\
\pi''_{44}&=\tau_1E_1\circ E_1\tau_1\circ E_1^2\pi\circ E_1\sigma\circ \sigma E_1
\end{align*}
$$\pi''_{12}=\sigma E_2\circ E_2\sigma\circ \tau_2E_1,\
\pi''_{13}=\sigma E_2,\ \pi''_{24}=\sigma E_1,\
\pi''_{34}= \tau_1E_2\circ E_1\sigma\circ\sigma E_1.$$

Let $a=E_1\tau\circ \pi''$ and $b=\pi''\circ
E_2\tau$. We have
\begin{align*}
a_{11}&=\sigma E_2\circ E_2\sigma\circ E_1\tau_2\circ E_2^2\pi\circ
E_2\tau_2\circ\tau_2E_2=
\sigma E_2\circ E_2\sigma\circ E_2^2\pi\circ\tau_2E_2\circ 
E_2\tau_2\circ\tau_2E_2\\
	&=
\sigma E_2\circ E_2\sigma\circ E_2^2\pi\circ E_2\tau_2\circ\tau_2E_2\circ E_2\tau_2
=b_{11}
\end{align*}
$$a_{12}=E_1\tau_2\circ\sigma E_2\circ E_2\sigma\circ \tau_2E_1=0=b_{12}$$
$$a_{13}=E_1\tau_2\circ \sigma E_2=b_{13}$$
	\begin{align*}
a_{23}&=E_1\sigma^{-1}\circ \tau_1E_2\circ E_1\sigma\circ E_1E_2\pi\circ E_1\tau_2\circ
		\sigma E_2\\
&=E_1\sigma^{-1}\circ \tau_1E_2\circ E_1\sigma\circ E_1E_2\pi\circ 
		\sigma E_2\circ E_2\sigma\circ \tau_2 E_1\circ E_2\sigma^{-1} \\
&=E_1\sigma^{-1}\circ \tau_1E_2\circ E_1\sigma\circ \sigma E_2\circ E_2E_1\pi\circ 
\circ E_2\sigma\circ \tau_2 E_1\circ E_2\sigma^{-1} \\
		&=\sigma E_1\circ E_2\tau_1\circ E_2E_1\pi\circ E_2\sigma\circ \tau_2E_1\circ E_2\sigma^{-1}
	=b_{23}
	\end{align*}
$$a_{24}=E_1\sigma^{-1}\circ\tau_1E_2\circ E_1\sigma\circ\sigma E_1=
	\sigma E_1\circ E_2\tau_1=b_{24}$$
\begin{align*}
a_{44}&=E_1\tau_1\circ\tau_1 E_1\circ E_1\tau_1\circ E_1^2\pi\circ E_1\sigma\circ
\sigma E_1 =
\tau_1 E_1\circ E_1\tau_1\circ \tau_1E_1\circ E_1^2\pi\circ E_1\sigma\circ \sigma E_1\\
&=\tau_1 E_1\circ E_1\tau_1\circ E_1^2\pi\circ E_1\sigma\circ \sigma E_1\circ E_2\tau_1= b_{44}
\end{align*}
All the other coefficients of $a$ and $b$ vanish. We deduce that
$a=b$, hence $\tau$ is an endomorphism of $E^2(m,\pi)$.
It follows easily that $\tau$ defines an endomorphism of $E^2$.

\smallskip
We have $\tau^2=0$ and 
\begin{align*}
d(\tau)&=\left(\begin{matrix}
d(\tau_2) &0&0&0\\
0&0&0&0\\
0&0&0&0\\
0&0&0&d(\tau_1)
\end{matrix}\right)+\tau\circ\partial+\partial\circ\tau \\
&=\left(\begin{matrix}
\id &  & &  \\
2E_2\pi\circ \tau_2 &\id&&  \\
\sigma\circ E_2\pi\circ \tau_2^2&&\id\\
&\tau_1^2\circ E_1\pi&2\tau_1\circ E_1\pi\circ\sigma&\id 
\end{matrix}\right)\\
&=\id.
\end{align*}

\smallskip
We have 
$E^3(m,\pi)=([m''',\delta'],\pi''')$, where
$$m'''=E_2^3(m)\oplus E_2^2E_1(m)\oplus
E_2E_1E_2(m)\oplus E_2E_1^2(m)\oplus
E_1E_2^2(m)\oplus E_1E_2E_1(m)\oplus E_1^2E_2(m)
\oplus E_1^3(m).$$

We have
$$\tau E=
\left(\begin{matrix}
\tau_2 E_2&0&0&0&0&0&0&0\\
0& \tau_2E_1&0&0&0&0&0&0\\
	0&0&0&0&\sigma^{-1}E_2&0&0&0\\
	0&0&0&0&0&\sigma^{-1}E_1&0&0\\
0&0&0&0&0&0&0&0\\
0&0&0&0&0&0&0&0\\
0&0&0&0&0&0&\tau_1E_2 & 0\\
0&0&0&0&0&0&0&\tau_1E_1
\end{matrix}\right)
$$
and
$$E\tau=\left(\begin{matrix}
E_2\tau_2 &0&0&0&0&0&0&0\\
	0&0&E_2\sigma^{-1}&0&0&0&0&0 \\
0&0&0&0&0&0&0&0\\
0&0&0&E_2\tau_1 &0&0&0&0\\
0&0&0&0&E_1\tau_2&0&0&0\\
	0&0&0&0&0&0&E_1\sigma^{-1}&0 \\
0&0&0&0&0&0&0&0\\
0&0&0&0&0&0&0&E_1\tau_1
\end{matrix}\right)$$
Let $a=(E\tau)\circ(\tau E)\circ (E\tau)$ and $b=(\tau E)\circ (E\tau)\circ
(\tau E)$. We have
$$a_{11}=E_2\tau_2\circ\tau_2E_2\circ E_2\tau_2=
\tau_2E_2\circ E_2\tau_2\circ\tau_2 E_2=b_{11}$$
$$a_{44}=E_1\tau_1\circ\tau_1E_1\circ E_1\tau_1=
\tau_1E_1\circ E_1\tau_1\circ \tau_1E_1=b_{44}$$
$$a_{25}=E_2\sigma^{-1} \circ \sigma^{-1}E_2\circ E_1\tau_2=
\tau_2E_1\circ E_2\sigma^{-1}\circ \sigma^{-1}E_2=b_{25}$$
$$a_{47}=E_2\tau_1 \circ \sigma^{-1}E_1\circ E_1\sigma^{-1}=
\sigma^{-1}E_1\circ E_1\sigma^{-1}\circ \tau_1 E_2=
b_{47}$$
and all the other coefficients of $a$ and $b$ vanish. It follows
that $a=b$.
This completes the proof of the theorem.
\end{proof}

\subsubsection{Functoriality}
\label{se:functorialitytensor}
We consider two differential categories $\CW$ and $\CW'$ endowed with
actions $(E_i,\tau_i)$ and $(E'_i,\tau'_i)$ of $\CU$ for $i\in\{1,2\}$
and closed morphisms of functors
$\sigma:E_2E_1\to E_1E_2$ and 
$\sigma':E'_2E'_1\to E'_1E'_2$ 
making (\ref{eq:diagsigma}) and the similar diagram for $\sigma'$ commute.

\smallskip
Let $\Phi:\CW\to\CW'$ be a differential functor and
$\varphi_i:\Phi E_i\iso E'_i\Phi$ be closed isomorphisms of functors making
$(\Phi,\varphi_i)$ into morphisms of $2$-representations for $i\in\{1,2\}$.
Assume 
\begin{equation}
	\label{eq:assumptionmorphism}
	(E'_1\varphi_2)\circ (\varphi_1E_2)\circ (\Phi\sigma)=(\sigma'\Phi)\circ
(E'_2\varphi_1)\circ(\varphi_2E_1):\Phi E_2E_1\to E'_1E'_2\Phi.
\end{equation}

\smallskip

\begin{prop}
	\label{pr:morphism2rep}
	There is a differential functor $\Delta\Phi:
	\Delta_\sigma\CW\to\Delta_{\sigma'}\CW'$ given by
	$(m,\pi)\mapsto (\Phi(m),\varphi_1(m)\circ \Phi(\pi)\circ \varphi_2(m)^{-1})$.

There is a closed isomorphism of functors 
$$\varphi=\left(\begin{matrix}\varphi_2 \\ & \varphi_1\end{matrix}\right): 
	\Delta\Phi E\iso E'\Delta\Phi.$$
	If $\sigma$ and $\sigma'$ are invertible, then $(\Delta\Phi,\varphi)$ defines a morphism
of $2$-representations $\Delta_\sigma\CW\to\Delta_{\sigma'}\CW'$.
\end{prop}

\begin{proof}
	Let $(m,\pi)$ be an object of $\Delta_\sigma\CW$. Let $\pi'=
	\varphi_1(m)\circ \Phi(\pi)\circ \varphi_2^{-1}(m)$, an element of $Z\Hom_{\overline{\CW'}^i}(
	E'_2\Phi(m),E'_1\Phi(m))$.
	
We have
		$$
		(E'_1\pi')\circ (\sigma'\Phi(m))\circ (E'_2\pi')=$$
		\begin{align*}
			&=
	(E'_1\varphi_1(m))\circ	(E'_1\Phi\pi)\circ (E'_1\varphi_2^{-1}(m))
	\circ (\sigma'\Phi(m))\circ (E'_2\varphi_1(m)) \circ (E'_2\Phi\pi)\circ(E'_2\varphi_2^{-1}(m))\\
			&=
	(E'_1\varphi_1(m))\circ	(E'_1\Phi\pi)\circ (\varphi_1(E_2(m)))
	\circ (\Phi\sigma(m))\circ (\varphi_2^{-1}(E_1(m)))\circ
	(E'_2\Phi\pi)\circ(E'_2\varphi_2^{-1}(m))\\
			&=
	(E'_1\varphi_1(m))\circ	(\varphi_1(E_1(m)))\circ \Phi\bigl( (E_1\pi)\circ
	\sigma(m)\circ (E_2\pi)\bigr)\circ (\varphi_2^{-1}(E_2(m)))\circ
	(E'_2\varphi_2^{-1}(m))
		\end{align*}

It follows that
	$$(E'_1\pi')\circ (\sigma'\Phi(m))\circ (E'_2\pi')\circ (\tau'_2\Phi(m))=
	(\tau'_1\Phi(m))\circ(E'_1\pi')\circ (\sigma'\Phi(m))\circ (E'_2\pi'),$$
	hence $(\Phi(m),\pi')$ is an object of $\Delta_{\sigma'}\CW'$. We put 
	$\Delta\Phi(m,\pi)=(\Phi(m),\pi')$.

	\smallskip
	Let $f\in\Hom_{\Delta_\sigma\CW}((m,\pi),(\tilde{m},\tilde{\pi}))$. We have a
	commutative diagram
$$\xymatrix{
	E'_2\Phi(m)\ar[r]^{\varphi_2^{-1}(m)}\ar[d]_{E'_2\Phi(f)} & 
	\Phi E_2(m)\ar[r]^{\Phi\pi}\ar[d]^{\Phi E_2(f)} & 
	\Phi E_1(m)\ar[r]^{\varphi_1(m)}\ar[d]^{\Phi E_1(f)} &
	E'_1\Phi(m)\ar[d]^{E'_1\Phi(f)} \\
	E'_2\Phi(\tilde{m})\ar[r]_{\varphi_2^{-1}(\tilde{m})} & 
	\Phi E_2(\tilde{m})\ar[r]_{\Phi\tilde{\pi}} & 
	\Phi E_1(\tilde{m})\ar[r]_{\varphi_1(\tilde{m})} &
	E'_1\Phi(\tilde{m})
}$$
	and it follows that $\Phi(f)\in\Hom_{\Delta_{\sigma'}\CW'}(\Delta\Phi(m,\pi),
	\Delta\Phi(\tilde{m},\tilde{\pi}))$.
	We put $(\Delta\Phi)(f)=\Phi(f)$. This makes $\Delta\Phi$ into a
	differential functor $\Delta_\sigma\CW\to\Delta_{\sigma'}\CW'$.

	\smallskip
	We have 
	$$(\Delta\Phi) (E(m,\pi))=({\xy (0,0)*{\Phi(E_2(m))\oplus \Phi(E_1(m))},
	\ar@/^/^{\Phi(\pi)}(-7,3)*{};(7,3)*{}, \endxy},\beta),$$
	$$\beta=\left(\begin{matrix}
\varphi_1 E_2\circ \Phi \sigma\circ \Phi E_2\pi\circ\Phi \tau_2\circ \varphi_2^{-1}E_2 &
\varphi_1 E_2\circ \Phi\sigma\circ\varphi_2^{-1}E_1 \\ 0 &
\varphi_1 E_1\circ \Phi \tau_1\circ \Phi E_1\pi\circ\Phi \sigma\circ \varphi_2^{-1}E_1
	\end{matrix}\right)$$
	and
	$$E'((\Delta\Phi)(m,\pi))=({\xy (0,0)*{E'_2(\Phi(m))\oplus E'_1(\Phi(m))},
	\ar@/^/^{\varphi_1(m)\circ\Phi(\pi)\circ\varphi_2(m)^{-1}}(-7,3)*{};(7,3)*{}, \endxy},\beta'),$$
	$$\beta'=\left(\begin{matrix}
	\sigma'\Phi\circ E'_2(\varphi_1\circ\Phi\pi\circ \varphi_2^{-1})\circ\tau'_2\Phi &
		\sigma'\Phi \\
		0 & \tau'_1\Phi\circ E'_1(\varphi_1\circ\Phi\pi\circ\varphi_2^{-1})\circ
		\sigma'\Phi
	\end{matrix}\right))$$

	We have 
$$\beta'\left(\begin{matrix}E'_2\varphi_2 & 0 \\ 0 & E'_2\varphi_1 \end{matrix}\right)=
	\left(\begin{matrix}E'_1\varphi_2 & 0 \\ 0 & E'_1\varphi_1 \end{matrix}\right)
		\beta,$$
	hence
	$\left(\begin{matrix}\varphi_2(m) \\ & \varphi_1(m)\end{matrix}\right)$
	defines
	a closed isomorphism $\Delta\Phi(E(m,\pi))\iso E'(\Delta\Phi(m,\pi))$. The naturality
	of $\varphi_1$ and $\varphi_2$ implies immediately that of $\varphi$.

	\smallskip
We have $\tau'_i\Phi\circ E'_i\varphi_i\circ \varphi_i E_i=E'_i\varphi_i\circ \varphi_iE_i
	\circ\Phi\tau_i$ for $i\in\{1,2\}$. Together with (\ref{eq:assumptionmorphism}),
	it follows that
$\tau'(\Delta\Phi)\circ E'\varphi\circ \varphi E=E'\varphi\circ \varphi E\circ
(\Delta\Phi)\tau$, hence
	$(\Delta\Phi,\varphi)$ defines a morphism of $2$-representations.
\end{proof}

\begin{rem}
	The data of $\varphi_1$ and $\varphi_2$, the relations they are required to
	satisfy, and the map $\pi'$ in the proof of the proposition
	are described graphically as:
$$\includegraphics[scale=0.9]{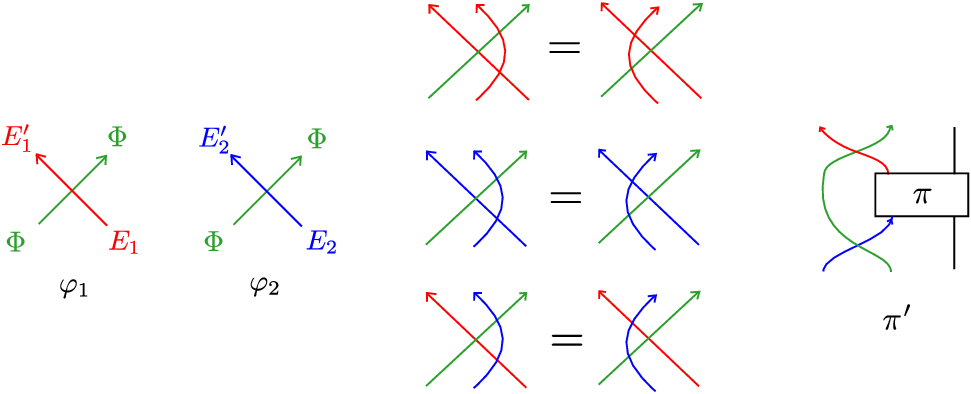}$$
\end{rem}

The following proposition is immediate.

\begin{prop}
	\label{pr:Deltafaithful}
	If $\Phi$ is faithful, then $\Delta\Phi$ is faithful.
\end{prop}

\subsubsection{Associativity}
\label{se:associativity}

We consider a differential category $\CW$ together
with three actions $(E_i,\tau_i)$, $1\le i\le 3$ of $\CU$.

We assume given $\sigma_{ij}:E_iE_j\iso E_jE_i$ for $i\neq j$ such that
\begin{equation}
	\label{eq:braid}
	\sigma_{ij}\sigma_{ji}=\id \text{ for all }i\neq j \text{ and }
E_3\sigma_{12}\circ \sigma_{13}E_2\circ E_1\sigma_{23}=\sigma_{23}E_1\circ E_2\sigma_{13}\circ
\sigma_{12}E_3.
\end{equation}
This ensures that by composing $\sigma$'s, we obtain a transitive system of isomorphisms between
$E_iE_jE_k$'s for $\{i,j,k\}=\{1,2,3\}$.

We also assume the analogs of the diagram (\ref{eq:diagsigma}) for the map $\sigma_{ij}$ commute.

\smallskip
Let $(m,\pi)\in \Delta_{\sigma_{21}}(\CW)$.
We define $\pi'\in Z\Hom(E_2E_3(m),E_1E_3(m))$ as the composition
$E_2E_3(m)\xrightarrow{\sigma_{23}}E_3E_2(m)\xrightarrow{E_3(\pi)}E_3E_1(m)
\xrightarrow{\sigma_{31}}E_1E_3(m)$.

We have a commutative diagram
$$\xymatrix{
	E_2^2E_3(m)\ar[r]^{E_2(\pi')}\ar[d]_{E_2\sigma_{23}} & E_2E_1E_3(m)\ar[r]^{\sigma_{21}E_3}& E_1E_2E_3(m) \ar[d]_{E_1\sigma_{23}}
	\ar[r]^{E_1(\pi')} & E_1^2E_3(m) \\
E_2E_3E_2(m)\ar[r]^{E_2E_3(\pi)}\ar[d]_{\sigma_{23}E_2} & E_2E_3E_1(m)\ar[u]^{E_2\sigma_{31}}\ar[d]_{\sigma_{23}E_1} & E_1E_3E_2(m)\ar[r]^{E_1E_3(\pi)} & E_1E_3E_1(m)\ar[u]^{E_1\sigma_{31}}\\
E_3E_2^2(m) \ar[r]_{E_3E_2(\pi)}& E_3E_2E_1(m) \ar[r]_{E_3\sigma_{21}} & E_3E_1E_2(m)
\ar[u]^{\sigma_{31}E_2}\ar[r]_{E_3E_1(\pi)} & E_3E_1^2(m)\ar[u]^{\sigma_{31}E_1}\\
}$$
It follows that $(E_3(m),\pi')$ defines an object of $\Delta_{\sigma_{21}}(\CW)$ and we
put $\tE_3(m,\pi)=(E_3(m),\pi')$.
Given $f$ a map in $\Delta_{\sigma_{21}}(\CW)$, the map $E_3(f)$ is actually
in $\Delta_{\sigma_{21}}(\CW)$ and this defines $\tE_3(f)$.

We have defined an endofunctor $\tE_3$ of $\Delta_{\sigma_{21}}(\CW)$.

\smallskip
There is a commutative diagram
$$\xymatrix{
	E_2E_3^2(m)\ar[r]^{\sigma_{23}E_3}\ar[d]_{E_2\tau_3} & E_3E_2E_3(m)\ar[r]^{E_3\sigma_{23}} &
	E_3^2E_2(m)\ar[r]^{E_3^2(\pi)}\ar[d]_{\tau_3 E_2} & E_3^2E_1(m)\ar[r]^{E_3\sigma_{31}}
	\ar[d]_{\tau_3 E_1} & E_3E_1E_3(m)\ar[r]^{\sigma_{31}E_3} & E_1 E_3^2(m)\ar[d]^{E_1\tau_3} \\
	E_2E_3^2(m)\ar[r]_{\sigma_{23}E_3} & E_3E_2E_3(m)\ar[r]_{E_3\sigma_{23}} &
	E_3^2E_2(m)\ar[r]_{E_3^2(\pi)} & E_3^2E_1(m)\ar[r]_{E_3\sigma_{31}}
 & E_3E_1E_3(m)\ar[r]_{\sigma_{31}E_3} & E_1 E_3^2(m) \\
}$$
It follows that $\tau_3$ defines an endomorphism of $\tE_3^2$. So, $(\tE_3,\tau_3)$ defines
a $2$-representation on $\Delta_{\sigma_{21}}(\CW)$.

\smallskip
Let $E_{21}$ denote the functor $\tE$ of $\Delta_{\sigma_{21}}(\CW)$.
We have 
$$\tE_3E_{21}(m,\pi)=({\xy (0,0)*{E_3E_2(m)\oplus E_3E_1(m)},
\ar@/^/^{E_3(\pi)}(-7,3)*{};(7,3)*{}, \endxy},\pi'),$$
where
$$\pi'=\left(\begin{matrix}\sigma_{31}E_2\circ E_3(\sigma_{21}\circ E_2\pi\circ \tau_2)\circ \sigma_{23}E_2 &
\sigma_{31}E_2\circ E_3\sigma_{21}\circ \sigma_{23}E_1 \\
0 & \sigma_{31}E_1\circ E_3(\tau_1\circ E_1\pi\circ\sigma_{21})\circ \sigma_{23}E_1\end{matrix}\right)$$
and
$$E_{21}\tE_3(m,\pi)=({\xy (0,0)*{E_2E_3(m)\oplus E_1E_3(m)},
\ar@/^/^{\sigma_{31}\circ E_3(\pi)\circ \sigma_{23}}(-7,3)*{};(7,3)*{}, \endxy},\pi'')$$
where
$$\pi''=
\left(\begin{matrix}\sigma_{21}E_3\circ E_2(\sigma_{31}\circ E_3(\pi)\circ \sigma_{23})\circ \tau_2 E_3 &
\sigma_{21}E_3 \\0 & \tau_1 E_3\circ \sigma_{13}E_1\circ E_1E_3\pi\circ E_1\sigma_{23}\circ
\sigma_{21}E_3\end{matrix}\right)
\bigl)
.$$

We have commutative diagrams
$$\xymatrix{
	E_2^2E_3(m)\ar[r]^{\tau_2E_3}\ar[d]_{E_2\sigma_{23}}& E_2^2E_3(m)\ar[r]^{E_2\sigma_{23}}&
	E_2E_3E_2(m)\ar[r]^{E_2E_3\pi}\ar[d]_{\sigma_{23}E_2}
	& E_2E_3E_1(m)\ar[r]^{E_2\sigma_{31}}\ar[d]_{\sigma_{23}E_1}& E_2E_1E_3(m)\ar[r]^{\sigma_{12}E_3}&E_1E_2E_3(m)
	\ar[d]^{E_1\sigma_{23}}\\
	 E_2E_3E_2(m)\ar[r]_{\sigma_{23}E_2}&E_3E_2^2(m)\ar[r]_{E_3\tau_2}&
	E_3E_2^2(m)\ar[r]_{E_3E_2\pi}& E_3E_2E_1(m)\ar[r]_{E_3\sigma_{21}}& E_3E_1E_2(m)
	\ar[r]_{\sigma_{31}E_2}& E_1E_3E_2(m)
}$$
$$\xymatrix{
	E_2E_1E_3(m)\ar[r]^{\sigma_{21}E_3}\ar[d]_{E_2\sigma_{13}}& E_1E_2E_3(m)
	\ar[r]^{E_1\sigma_{23}}&
	E_1E_3E_2(m)\ar[r]^{E_1E_3\pi}\ar[d]_{\sigma_{13}E_2}
	& E_1E_3E_1(m)\ar[r]^{E_1\sigma_{31}}\ar[d]_{\sigma_{13}E_1}& E_1^2E_3(m)\ar[r]^{\tau_1E_3}&
	E_1^2E_3(m) \ar[d]^{E_1\sigma_{13}}\\
	 E_2E_3E_1(m)\ar[r]_{\sigma_{23}E_1}&E_3E_2E_1(m)\ar[r]_{E_3\sigma_{21}}&
	E_3E_1E_2(m)\ar[r]_{E_3E_1\pi}& E_3E_1^2(m)\ar[r]_{E_3\tau_1}& E_3E_1^2(m)
	\ar[r]_{\sigma_{31}E_1}& E_1E_3E_1(m)
}$$

So, $\left(\begin{matrix}\sigma_{23}&0\\0&\sigma_{13}\end{matrix}\right)$
	defines an isomorphism $E_{21}\tE_3(m,\pi)\iso \tE_3E_{21}(m,\pi)$.
It provides an isomorphism of functors $\sigma_{21,3}:E_{21}\tE_3\iso \tE_3E_{21}$.

\smallskip
Replacing $1$ by $2$, $2$ by $3$ and $3$ by $1$, the construction above provides
a $2$-representation  $(\tE_1,\tau_1)$ of $\Delta_{\sigma_{32}}(\CW)$ and we denote
by $E_{32}$ the endofunctor $\tE$ of $\Delta_{32}(\CW)$. We have an isomorphism
$\sigma_{32,1}:E_{32}\tE_1\iso \tE_1E_{32}$.

\medskip
Let us now define another $2$-representation. The justifications for the constructions
below will be given in the proof of Proposition \ref{pr:associativity}.

We define a differential category $\Delta_{123}(\CW)$. 
Its objects are quadruples $(m,\pi_{21},\pi_{31},\pi_{32})$ where
 $m\in\overline{\CW}^i$, $\pi_{ij}:E_i(m)\to E_j(m)$ satisfy
	$$d(\pi_{21})=d(\pi_{32})=0,\ d(\pi_{31})=\pi_{21}\circ \pi_{32}$$
	and given $i,j,k,l\in\{1,2,3\}$ with $j-l\ge i-k>0$,
	we have an equality between maps $E_iE_j(m)\to E_kE_l(m)$:
$$\sigma_{lk}\circ E_l\pi_{ik}\circ\sigma_{il}\circ E_i\pi_{jl}+
E_k\pi_{jl}\circ \sigma_{jk}\circ E_j\pi_{ik}\circ \sigma_{ij}+\delta_{jk}
E_j\pi_{il}\circ\sigma_{ij}+\delta_{il}\sigma_{lk}\circ E_i\pi_{jk}=0
$$
where we put  $\sigma_{rr}=\tau_r$.

Note that the equality for $(i,j,k,l)=(3,2,2,1)$ is equivalent to the one for $(i,j,k,l)=(2,3,1,2)$.

\smallskip
We define $\Hom_{\Delta_{123}(\CW)}((m,\pi_{21},\pi_{31},\pi_{32}),
(m',\pi'_{21},\pi'_{31},\pi'_{32}))$ to be the differential submodule of
$\Hom_\CW(m,m')$ of maps $f$ such that $E_jf\circ \pi_{ij}=\pi'_{ij}\circ E_if$
for all $i>j$.

\smallskip
We define a differential endofunctor $E$ of $\Delta_{123}(\CW)$ by
$E(m,\pi_{21},\pi_{31},\pi_{32})=(m',\pi'_{21},\pi'_{31},\pi'_{32})$ where
$$m'={\xy (0,0)*{E_3(m)\oplus E_2(m)\oplus E_1(m)},
\ar@/^2pc/^{\pi_{31}}(-17,3)*{};(17,3)*{},
\ar@/^0.5pc/^{\ \pi_{32}}(-16,3)*{};(-1,3)*{},
\ar@/^0.5pc/^{\pi_{21}}(0,3)*{};(16,3)*{},
\endxy}$$

$$\pi'_{31}:{\xy (0,10)*{E_3^2(m)\oplus E_3E_2(m)\oplus E_3E_1(m)},
(0,-10)*{E_1E_3(m)\oplus E_1E_2(m)\oplus E_1^2(m)},
\ar_{\sigma_{31}\circ E_3\pi_{31}\circ\tau_3}(-22,7)*{};(-22,-7)*{},
\ar_{\sigma_{31}}(20,7)*{};(-20,-7)*{},
\ar^{\tau_1\circ E_1\pi_{31}\circ\sigma_{31}}(22,7)*{};(22,-7)*{},
\endxy}$$

$$\pi'_{32}:{\xy (0,10)*{E_3^2(m)\oplus E_3E_2(m)\oplus E_3E_1(m)},
(0,-10)*{E_2E_3(m)\oplus E_2^2(m)\oplus E_2E_1(m)},
\ar_{\sigma_{32}\circ E_3\pi_{32}\circ\tau_3}(-22,7)*{};(-22,-7)*{},
\ar_{\sigma_{32}}(-4,7)*{};(-20,-7)*{},
\ar^{\sigma_{12}\circ E_1\pi_{32}\circ\sigma_{31}}(22,7)*{};(22,-7)*{},
\ar^{\tau_2\circ E_2\pi_{32}\circ\sigma_{32}}(-2,7)*{};(-2,-7)*{},
\endxy}$$

$$\pi'_{21}:{\xy (0,10)*{E_2E_3(m)\oplus E_2^2(m)\oplus E_2E_1(m)},
(0,-10)*{E_1E_3(m)\oplus E_1E_2(m)\oplus E_1^2(m)},
\ar_{\sigma_{31}\circ E_3\pi_{21}\circ\sigma_{23}}(-22,7)*{};(-22,-7)*{},
\ar^{\tau_1\circ E_1\pi_{21}\circ\sigma_{21}}(22,7)*{};(22,-7)*{},
\ar_{\sigma_{21}}(20,7)*{};(0,-7)*{},
\ar_{\sigma_{21}\circ E_2\pi_{21}\circ\tau_2\!}(-2,7)*{};(-2,-7)*{},
\endxy}$$

We define an endomorphism $\tau$ of $E^2$ as follows. We have
$E^2(m,\pi_{21},\pi_{31},\pi_{32})=(m'',\pi''_{21},\pi''_{31},\pi''_{32})$ where
(ignoring differentials)
$$m''=E_3^2(m)\oplus E_3E_2(m)\oplus E_3E_1(m)\oplus E_2E_3(m)\oplus E_2^2(m)
\oplus E_2E_1(m)\oplus E_1E_3(m)\oplus E_1E_2(m)\oplus E_1^2(m).$$
We define the endomorphism $\tau$ of $m''$ by
\begin{equation}
	\label{eq:tau123}
\tau=\left(\begin{matrix}
	\tau_3 \\
	&&&\sigma_{23}\\
	&&&&&&\sigma_{13}\\
	\\
	&&&&\tau_2\\
	&&&&&&&\sigma_{12}\\
	\\
	\\
	&&&&&&&&\tau_1
\end{matrix}\right).
\end{equation}

\begin{prop}
	\label{pr:associativity}
	The construction above defines a $2$-representation on $\Delta_{123}(\CW)$.

	We have isomorphisms of $2$-representations
	$\Delta_{\sigma_{21,3}^{-1}}\Delta_{\sigma_{21}}(\CW)\iso\Delta_{123}(\CW)$
	and $\Delta_{\sigma_{32,1}}\Delta_{\sigma_{32}}(\CW)\iso\Delta_{123}(\CW)$
	whose underlying functors make the following diagram commutative
	$$\xymatrix{
		\Delta_{\sigma_{21,3}^{-1}}\Delta_{\sigma_{21}}(\CW)\ar[r]^-{\sim}
		\ar[dr]_{\omega}& \Delta_{123}(\CW)\ar[d]_\omega &
	\Delta_{\sigma_{32,1}}\Delta_{\sigma_{32}}(\CW) \ar[l]_-\sim\ar[dl]^{\omega}\\
	& \CW	
}$$
\end{prop}

\begin{proof}
	Replacing $\CW$ by $\overline{\CW}^i$, we can assume it is strongly pretriangulated and
	idempotent-complete.

	The category $\Delta_{\sigma_{21,3}^{-1}}\Delta_{\sigma_{21}}(\CW)$ has objects
	$((m,\pi_{21}),\pi_3)$ where $m\in\CW$, $\pi_{21}:E_2(m)\to E_1(m)$ and
	$\pi_3:\tE_3(m,\pi_{21})\to E_{21}(m,\pi_{21})$ satisfy
	$$d(\pi_{21})=d(\pi_3)=0$$
	and the diagram (\ref{eq:Deltasigmacommutationtau}) commutes for $\pi_{21}$ and
	for $\pi_3$.

	For $i\in\{1,2\}$, let $\pi_{3i}$ be the composition of $\pi_3$ with the projection
	onto $E_i(m)$. We have $d(\pi_{32})=0$ and $d(\pi_{31})=\pi_{21}\circ
	\pi_{32}$.

	\smallskip
	The commutativity of (\ref{eq:Deltasigmacommutationtau}) for $\pi_{21}$ is the
	commutativity of

$$\xymatrix{
	E_2^2(m)\ar[rr]^-{E_2\pi_{21}} \ar[d]_{\tau_2} &&
	E_2E_1(m)\ar[r]^-{\sigma_{21}} & E_1E_2(m)
	\ar[rr]^-{E_1\pi_{21}} && E_1^2(m)\ar[d]^{\tau_1} \\
	E_2^2(m)\ar[rr]_-{E_2\pi_{21}} && E_2E_1(m)\ar[r]_-{\sigma_{21}} & E_1E_2(m)
	\ar[rr]_-{E_1\pi_{21}} && E_1^2(m)
}$$

	The maps $\pi_{3i}:E_3(m)\to E_i(m)$ for $i\in\{1,2\}$ give rise to a map
	$$\left(\begin{matrix}\pi_{32}\\ \pi_{31}\end{matrix}\right)\in
		\Hom_{\Delta_{\sigma_{21}}(\CW)}(\tE_3(m,\pi_{21}),E_{21}(m,\pi_{21}))$$
		if and only if the composition
	$$\xymatrix{
		E_2E_3(m) \ar[r]^{\sigma_{23}} & E_3E_2(m)\ar[r]^{E_3\pi_{21}}&
		E_3E_1(m)\ar[r]^{\sigma_{31}} & E_1E_3(m)\ar[r]^{E_1\pi_{32}} & 
		E_1E_2(m)}$$
	is equal to the sum of the following two maps:
	$$\xymatrix{
		E_2E_3(m)\ar[r]^{E_2\pi_{32}} &
		E_2^2(m) \ar[r]^{\tau_2} & E_2^2(m)\ar[r]^{E_2\pi_{21}} & E_2E_1(m)
		\ar[r]^{\sigma_{21}} & E_1E_2(m)
		}$$
	and
		$$\xymatrix{
			E_2E_3(m) \ar[r]^{E_2\pi_{31}} &
		E_2E_1(m) \ar[r]^{\sigma_{21}} & E_1E_2(m)
		}$$
	and the following diagram commutes:
	$$\xymatrix{
		E_2E_3(m) \ar[r]^{\sigma_{23}}\ar[d]_{E_2\pi_{31}} & E_3E_2(m)\ar[r]^{E_3\pi_{21}}&
		E_3E_1(m)\ar[r]^{\sigma_{31}} & E_1E_3(m)\ar[d]^{E_1\pi_{31}} \\
		E_2E_1(m) \ar[r]_{\sigma_{21}} & E_1E_2(m)\ar[r]_{E_1\pi_{21}} & E_1^2(m)
		\ar[r]_{\tau_1} & E_1^2(m)
		}$$

	The commutativity of (\ref{eq:Deltasigmacommutationtau}) for $\pi_3$ is equivalent to the
	commutativity of the following diagrams:
$$\xymatrix{
	E_3^2(m)\ar[rr]^-{E_3\pi_{31}} \ar[d]_{\tau_3} &&
	E_3E_1(m)\ar[r]^-{\sigma_{31}} & E_1E_3(m)
	\ar[rr]^-{E_3\pi_{31}} && E_1^2(m)\ar[d]^{\tau_1} \\
	E_3^2(m)\ar[rr]_-{E_3\pi_{31}} && E_3E_1(m)\ar[r]_-{\sigma_{31}} & E_1E_3(m)
	\ar[rr]_-{E_3\pi_{31}} && E_1^2(m)
}$$

$$\xymatrix{
	E_3^2(m)\ar[rr]^-{E_3\pi_{32}} \ar[d]_{\tau_3} &&
	E_3E_2(m)\ar[r]^-{\sigma_{32}} & E_2E_3(m)
	\ar[rr]^-{E_3\pi_{32}} && E_2^2(m)\ar[d]^{\tau_2} \\
	E_3^2(m)\ar[rr]_-{E_3\pi_{32}} && E_3E_2(m)\ar[r]_-{\sigma_{32}} & E_2E_3(m)
	\ar[rr]_-{E_3\pi_{32}} && E_2^2(m)
}$$

$$\xymatrix{
	E_3^2(m)\ar[r]^{E_3\pi_{31}} \ar[d]_{\tau_3} & E_3E_1(m)\ar[r]^{\sigma_{31}}& 
	E_1E_3(m)\ar[r]^{E_1\pi_{32}} & E_1E_2(m)\ar[d]^{\sigma_{12}} \\
	E_3^2(m)\ar[r]_{E_3\pi_{32}} & E_3E_2(m)\ar[r]_{\sigma_{32}} & E_2E_3(m)
	\ar[r]_{E_2\pi_{31}} & E_2E_1(m)
}$$
and the vanishing of the following composition:
$$E_3^2(m)\xrightarrow{\tau_3} E_3^2(m)\xrightarrow{E_3\pi_{31}}E_3E_1(m)
\xrightarrow{\sigma_{31}}E_1E_3(m)\xrightarrow{E_1\pi_{32}}E_1E_2(m).$$

Note that the vanishing of that composition follows from the commutativity
of the diagram immediately above.

\smallskip
We deduce that the objects of $\Delta_{\sigma_{21,3}^{-1}}\Delta_{\sigma_{21}}(\CW)$
	can be described as quadruples $(m,\pi_{21},\pi_{31},\pi_{32})$ where
 $m\in\CW$, $\pi_{ij}:E_i(m)\to E_j(m)$ satisfy
	$$d(\pi_{21})=d(\pi_{32})=0,\ d(\pi_{31})=\pi_{21}\circ \pi_{32}$$
	and given $i,j,k,l\in\{1,2,3\}$ with $j-l\ge i-k>0$,
	we have an equality between maps $E_iE_j(m)\to E_kE_l(m)$:
$$\sigma_{lk}\circ E_l\pi_{ik}\circ\sigma_{il}\circ E_i\pi_{jl}+
E_k\pi_{jl}\circ \sigma_{jk}\circ E_j\pi_{ik}\circ \sigma_{ij}+\delta_{jk}
E_j\pi_{il}\circ\sigma_{ij}+\delta_{il}\sigma_{lk}\circ E_i\pi_{jk}=0
$$
where we put  $\sigma_{rr}=\tau_r$.

This provides an isomorphism of categories 
$\Delta_{\sigma_{21,3}^{-1}}\Delta_{\sigma_{21}}(\CW)\iso\Delta_{123}(\CW)$.

\medskip
Let us now describe the action of $E$ on
$\Delta_{\sigma_{21,3}^{-1}}\Delta_{\sigma_{21}}(\CW)$.

We have
$E((m,\pi_{21}),\pi_3)=(m',\pi')$ where
$$m'={\xy (0,0)*{\tE_3(m,\pi_{21})\oplus E_{21}(m,\pi_{21})},
\ar@/^/^{\pi_3}(-12,3)*{};(12,3)*{},
\endxy}$$
$$=
({\xy (0,0)*{E_3(m)\oplus E_2(m)\oplus E_1(m)},
\ar@/^2pc/^{\pi_{31}}(-17,3)*{};(17,3)*{},
\ar@/^0.5pc/^{\ \ \pi_{32}}(-16,3)*{};(-1,3)*{},
\ar@/^0.5pc/^{\pi_{21}}(0,3)*{};(16,3)*{},
\endxy},
{\xy (0,10)*{E_2E_3(m)\oplus E_2^2(m)\oplus E_2E_1(m)},
(0,-10)*{E_1E_3(m)\oplus E_1E_2(m)\oplus E_1^2(m)},
\ar_{\sigma_{31}\circ E_3\pi_{21}\circ \sigma_{23}}(-20,7)*{};(-20,-7)*{},
\ar_{\sigma_{21}\circ E_2\pi_{21}\circ\tau_2\!\!}(0,7)*{};(0,-7)*{},
\ar^{\tau_1\circ E_1\pi_{21}\circ\sigma_{21}}(20,7)*{};(20,-7)*{},
\ar^{\sigma_{21}}(18,7)*{};(2,-7)*{},
\endxy})$$

\vskip 0.2cm

$$\pi':
{\xy (0,10)*{\tE_3^2(m,\pi_{21})\oplus \tE_3 E_{21}(m,\pi_{21})},
(0,-10)*{E_{21}\tE_3(m,\pi_{21})\oplus E_{21}^2(m,\pi_{21})},
\ar_{\sigma_{21,3}^{-1}\circ \tE_3\pi_3\circ \tau_3}(-15,7)*{};(-15,-7)*{},
\ar^{\tau_{E_{21}}\circ E_{21}\pi_3\circ \sigma_{21,3}^{-1}}(15,7)*{};(15,-7)*{},
\ar_{\sigma_{21,3}^{-1}}(13,7)*{};(-13,-7)*{},
\endxy}$$
\vskip 0.2cm

$$={\xy (0,15)*{E_3^2(m)\oplus E_3E_2(m)\oplus E_3E_1(m)},
(0,-15)*{E_2E_3(m)\oplus E_1E_3(m)\oplus E_2^2(m)\oplus E_2E_1(m)\oplus E_1E_2(m)
\oplus E_1^2(m)},
\ar@/_1pc/_{\sigma_{32}\circ E_3\pi_{32}\circ \tau_3}(-22,12)*{};(-49,-12)*{},
\ar@/_1pc/^(0.4){\!\!\!\sigma_{31}\circ E_3\pi_{31}\circ\tau_3}(-20,12)*{};(-31,-12)*{},
\ar@/_1.5pc/^(0.7){\!\!\!\!\sigma_{32}}(-6,12)*{};(-47,-12)*{},
\ar^(0.8){\!\!\tau_2\circ E_2\pi_{32}\circ\sigma_{32}}(-4,12)*{};(-10,-12)*{},
\ar^(0.6){\!\!\sigma_{12}\circ E_1\pi_{32}\circ\sigma_{31}}(17,12)*{};(10,-12)*{},
\ar^(0.2){\!\!\sigma_{31}}(15,12)*{};(-25,-12)*{},
\ar^{\tau_1\circ E_1\pi_{21}\circ\sigma_{31}}(19,12)*{};(50,-12)*{},
\endxy}$$

Via the isomorphism of categories above, this corresponds to the functor $E$ on
$\Delta_{123}(\CW)$.

\smallskip

The endomorphism $\tau$ of $E^2((m,\pi_{21}),\pi_3)$ is
$${\xy (0,10)*{\tE_3^2(m,\pi_{21})\oplus \tE_3E_{21}(m,\pi_{21})\oplus
E_{21}\tE_3(m,\pi_{21}) \oplus E_{21}^2(m,\pi_{21})},
(0,-10)*{\tE_3^2(m,\pi_{21})\oplus \tE_3E_{21}(m,\pi_{21})\oplus
E_{21}\tE_3(m,\pi_{21}) \oplus E_{21}^2(m,\pi_{21})},
\ar_{\tau_3}(-42,7)*{};(-42,-7)*{},
\ar^{\tau_{E_{21}}}(43,7)*{};(43,-7)*{},
\ar^{\sigma_{21,3}}(15,7)*{};(-15,-7)*{},
\endxy}$$

$$={\xy (0,10)*{E_3^2(m)\oplus E_3E_2(m)\oplus E_3E_1(m)\oplus
E_2E_3(m)\oplus E_1E_3(m) \oplus E_2^2(m)\oplus E_2E_1(m)\oplus E_1E_2(m)\oplus E_1^2(m)},
(0,-10)*{E_3^2(m)\oplus E_3E_2(m)\oplus E_3E_1(m)\oplus
E_2E_3(m)\oplus E_1E_3(m) \oplus E_2^2(m)\oplus E_2E_1(m)\oplus E_1E_2(m)\oplus E_1^2(m)},
\ar_{\tau_3}(-82,7)*{};(-82,-7)*{},
\ar_{\sigma_{23}}(-22,7)*{};(-62,-7)*{},
\ar^{\sigma_{13}}(-2,7)*{};(-42,-7)*{},
\ar_{\tau_2}(21,7)*{};(21,-7)*{},
\ar^{\tau_1}(79,7)*{};(79,-7)*{},
\ar_{\sigma_{12}}(58,7)*{};(38,-7)*{},
\endxy}$$

This coincides with the endomorphism $\tau$ of the endofunctor $E^2$ of $\Delta_{123}(\CW)$.

\medskip
	The category $\Delta_{\sigma_{32,1}}\Delta_{\sigma_{32}}(\CW)$ has objects
	pairs $((m,\pi_{32}),\pi_1)$ where $m\in\CW$, $\pi_{32}:E_3(m)\to E_2(m)$ and
	$\pi_1:E_{32}(m,\pi_{32})\to \tE_1(m,\pi_{32})$ satisfy
	$$d(\pi_{32})=d(\pi_1)=0$$
	and the diagram (\ref{eq:Deltasigmacommutationtau}) commutes for $\pi_{32}$ and
	for $\pi_1$.

	For $i\in\{2,3\}$, let $\pi_{i1}$ be the composition of the inclusion
	$E_i(m)\to E_{32}(m)$ with $\pi_1$.
	We have $d(\pi_{21})=0$ and $d(\pi_{31})=\pi_{21}\circ
	\pi_{32}$.

	As in the case of the category $\Delta_{\sigma_{21,3}^{-1}}\Delta_{\sigma_{21}}(\CW)$,
	the objects of $\Delta_{\sigma_{32,1}}\Delta_{\sigma_{32}}(\CW)$
	can be described as quadruples $(m,\pi_{21},\pi_{31},\pi_{32})$ where
 $m\in\CW$, $\pi_{ij}:E_i(m)\to E_j(m)$ satisfy
	$$d(\pi_{21})=d(\pi_{32})=0,\ d(\pi_{31})=\pi_{21}\circ \pi_{32},$$
	the composition
	$$E_3E_2(m)\xrightarrow{E_3\pi_{21}}E_3E_1(m)\xrightarrow{\sigma_{31}}E_1E_3(m)
	\xrightarrow{E_1\pi_{32}}E_1E_2(m)\xrightarrow{\sigma_{12}}E_2E_1(m)$$
	is equal to the sum of the following two maps
	$$E_3E_2(m)\xrightarrow{\sigma_{32}}E_2E_3(m)\xrightarrow{E_2\pi_{32}}E_2^2(m)
	\xrightarrow{\tau_2}E_2^2(m)\xrightarrow{E_2\pi_{21}}E_2E_1(m)$$
	and
	$$E_3E_2(m)\xrightarrow{\sigma_{32}}E_2E_3(m)\xrightarrow{E_2\pi_{31}}E_2E_1(m),$$
	the following diagrams commute
	$$\xymatrix{
		E_3^2(m)\ar[r]^{\tau_3}\ar[d]_{E_3\pi_{31}} & E_3^2(m)\ar[r]^{E_3\pi_{32}}&
		E_3E_2(m)\ar[r]^{\sigma_{32}} & E_2E_3(m)\ar[d]^{E_2\pi_{31}} \\
		E_3E_1(m) \ar[r]_{\sigma_{31}} & E_1E_3(m)\ar[r]_{E_1\pi_{32}} &
		E_1E_2(m)\ar[r]_{\sigma_{12}} & E_2E_1(m)}$$

	$$\xymatrix{
		E_3^2(m)\ar[rr]^-{E_3\pi_{32}} \ar[d]_{\tau_3} &&
	E_3E_2(m)\ar[r]^-{\sigma_{32}} & E_2E_3(m)
	\ar[rr]^-{E_3\pi_{32}} && E_2^2(m)\ar[d]^{\tau_2} \\
	E_3^2(m)\ar[rr]_-{E_3\pi_{32}} && E_3E_2(m)\ar[r]_-{\sigma_{32}} & E_2E_3(m)
	\ar[rr]_-{E_3\pi_{32}} && E_2^2(m)
}$$

$$\xymatrix{
	E_3^2(m)\ar[rr]^-{E_3\pi_{31}} \ar[d]_{\tau_3} &&
	E_3E_1(m)\ar[r]^-{\sigma_{31}} & E_1E_3(m)
	\ar[rr]^-{E_3\pi_{31}} && E_1^2(m)\ar[d]^{\tau_1} \\
	E_3^2(m)\ar[rr]_-{E_3\pi_{31}} && E_3E_1(m)\ar[r]_-{\sigma_{31}} & E_1E_3(m)
	\ar[rr]_-{E_3\pi_{31}} && E_1^2(m)
}$$

$$\xymatrix{
	E_2^2(m)\ar[rr]^-{E_2\pi_{21}} \ar[d]_{\tau_2} &&
	E_2E_1(m)\ar[r]^-{\sigma_{21}} & E_1E_2(m)
	\ar[rr]^-{E_1\pi_{21}} && E_1^2(m)\ar[d]^{\tau_1} \\
	E_2^2(m)\ar[rr]_-{E_2\pi_{21}} && E_2E_1(m)\ar[r]_-{\sigma_{21}} & E_1E_2(m)
	\ar[rr]_-{E_1\pi_{21}} && E_1^2(m)
}$$

$$\xymatrix{
	E_2E_3(m) \ar[r]^{E_2\pi_{31}} \ar[d]_{\sigma_{23}} & E_2E_1(m)\ar[r]^{\sigma_{21}}&
	E_1E_2(m)\ar[r]^{E_1\pi_{21}} & E_1^2(m)\ar[d]^{\tau_1}\\
E_3E_2(m) \ar[r]_{E_3\pi_{21}} & E_3E_1(m)\ar[r]_{\sigma_{31}} & E_1E_3(m)
\ar[r]_{E_1\pi_{31}} & E_1^2(m)
}$$
and the following composition vanishes:
$$E_3E_2(m)\xrightarrow{E_3\pi_{21}} E_3E_1(m)\xrightarrow{\sigma_{31}}E_1E_3(m)
\xrightarrow{E_1\pi_{31}}E_1^2(m)\xrightarrow{\tau_1}E_1^2(m).$$

The vanishing of that composition follows from the commutativity of the diagram
immediately above.

\smallskip
This description of objects provides an isomorphism of categories
$\Delta_{\sigma_{32,1}}\Delta_{\sigma_{32}}(\CW)\iso\Delta_{123}(\CW)$.

\medskip
Let us now describe the action of $E$ on 
$\Delta_{\sigma_{32,1}}\Delta_{\sigma_{32}}(\CW)$.

We have
$E((m,\pi_{32}),\pi_1)=(m',\pi')$ where
$$m'={\xy (0,0)*{E_{32}(m,\pi_{32})\oplus \tE_1(m,\pi_{32})},
\ar@/^/^{\pi_1}(-12,3)*{};(12,3)*{},
\endxy}$$
$$=
({\xy (0,0)*{E_3(m)\oplus E_2(m)\oplus E_1(m)},
\ar@/^2pc/^{\pi_{31}}(-17,3)*{};(17,3)*{},
\ar@/^0.5pc/^{\ \ \pi_{32}}(-16,3)*{};(-1,3)*{},
\ar@/^0.5pc/^{\pi_{21}}(0,3)*{};(16,3)*{},
\endxy},
{\xy (0,10)*{E_3^2(m)\oplus E_3E_2(m)\oplus E_3E_1(m)},
(0,-10)*{E_2E_3(m)\oplus E_2^2(m)\oplus E_2E_1(m)},
\ar_{\sigma_{32}\circ E_3\pi_{32}\circ \tau_3}(-20,7)*{};(-20,-7)*{},
\ar^{\!\!\tau_2\circ E_2\pi_{32}\circ\sigma_{32}}(0,7)*{};(0,-7)*{},
\ar^{\sigma_{12}\circ E_1\pi_{32}\circ\sigma_{31}}(20,7)*{};(20,-7)*{},
\ar_{\sigma_{32}}(-2,7)*{};(-18,-7)*{},
\endxy})$$

\vskip 0.2cm
$$\pi':
{\xy (0,10)*{E_{32}^2(m,\pi_{32})\oplus E_{32}\tE_1(m,\pi_{32})},
(0,-10)*{\tE_1E_{32}(m,\pi_{32})\oplus \tE_1^2(m,\pi_{32})},
\ar_{\sigma_{32,1}\circ E_{32}\pi_1\circ \tau_{E_{32}}}(-15,7)*{};(-15,-7)*{},
\ar^{\tau_1\circ \tE_1\pi_1\circ \sigma_{32,1}}(15,7)*{};(15,-7)*{},
\ar_{\sigma_{32,1}}(13,7)*{};(-13,-7)*{},
\endxy}$$
\vskip 0.2cm

$$={\xy (0,15)*{E_3^2(m)\oplus E_3E_2(m)\oplus
E_2E_3(m)\oplus E_2^2(m)\oplus E_3E_1(m)\oplus E_2E_1(m)},
(0,-15)*{E_1E_3(m)\oplus E_1E_2(m) \oplus E_1^2(m)},
\ar@/_1pc/_{\sigma_{31}\circ E_3\pi_{31}\circ \tau_3}(-49,12)*{};(-22,-12)*{},
\ar_(0.7){\sigma_{31}\!\!\!}(24,12)*{};(-16,-12)*{},
\ar@/^0.4pc/_(0.6){\tau_1\circ E_1\pi_{31}\circ\sigma_{31}\!\!}(30,12)*{};(20,-12)*{},
\ar@/_1pc/_(0.3){\sigma_{31}\circ E_3\pi_{21}\circ\sigma_{23}\!\!}(-12,12)*{};(-20,-12)*{},
\ar_(0.2){\sigma_{21}\circ E_2\pi_{21}\circ\tau_2\!\!}(8,12)*{};(0,-12)*{},
\ar@/^1.5pc/_(0.2){\sigma_{21}\!\!}(45,12)*{};(2,-12)*{},
\ar@/^1.5pc/^{\tau_1\circ E_1\pi_{31}\circ\sigma_{21}}(47,12)*{};(22,-12)*{},
\endxy}$$

Via the isomorphism of categories above, this corresponds to the functor $E$ on
$\Delta_{123}(\CW)$.

\smallskip

The endomorphism $\tau$ of $E^2((m,\pi_{32}),\pi_1)$ is
$${\xy (0,10)*{E_{32}^2(m,\pi_{32})\oplus E_{32}\tE_1(m,\pi_{32})\oplus
\tE_1E_{32}(m,\pi_{32}) \oplus \tE_1^2(m,\pi_{32})},
(0,-10)*{E_{32}^2(m,\pi_{32})\oplus E_{32}\tE_1(m,\pi_{32})\oplus
\tE_1E_{32}(m,\pi_{32}) \oplus \tE_1^2(m,\pi_{32})},
\ar_{\tau_{E_{32}}}(-42,7)*{};(-42,-7)*{},
\ar^{\tau_1}(43,7)*{};(43,-7)*{},
\ar^{\sigma_{32,1}^{-1}}(15,7)*{};(-15,-7)*{},
\endxy}$$

$$={\xy (0,10)*{E_3^2(m)\oplus E_3E_2(m)\oplus E_2E_3(m)\oplus
E_2^2(m)\oplus E_3E_1(m) \oplus E_2E_1(m)\oplus E_1E_3(m)\oplus E_1E_2(m)\oplus E_1^2(m)},
(0,-10)*{E_3^2(m)\oplus E_3E_2(m)\oplus E_2E_3(m)\oplus
E_2^2(m)\oplus E_3E_1(m) \oplus E_2E_1(m)\oplus E_1E_3(m)\oplus E_1E_2(m)\oplus E_1^2(m)},
\ar_{\tau_3}(-82,7)*{};(-82,-7)*{},
\ar_{\sigma_{23}}(-42,7)*{};(-62,-7)*{},
\ar_{\sigma_{13}}(38,7)*{};(-2,-7)*{},
\ar_{\tau_2}(-22,7)*{};(-22,-7)*{},
\ar^{\tau_1}(79,7)*{};(79,-7)*{},
\ar_{\sigma_{12}}(58,7)*{};(18,-7)*{},
\endxy}$$

This coincides with the endomorphism $\tau$ of the endofunctor $E^2$ of $\Delta_{123}(\CW)$.
\end{proof}

\subsection{Dual diagonal action}
\label{se:dualdiagonal}
\subsubsection{Category}

Consider two actions of $\CU$ given by $(F_1,\tau_1)$ and
$(E_2,\tau_2)$ on $\CW$ and a closed morphism of functors
$\lambda:F_1E_2\to E_2F_1$ such that
diagrams (\ref{eq:diaglambda}) commute.
As in \S\ref{se:2birep}, we have maps
$\mu_{i,j}=\mu_{(i,i),(j,j)}:E_2^iF_1^iE_2^jF_1^j\to E_2^{i+j}F_1^{i+j}$.


\medskip
We define a differential category 
$\Delta_{\lambda}\CW$\indexnot{Delta}{\Delta_\lambda\CW}.
Its objects are pairs $(m,\varsigma)$ where $m\in \overline{\CW}^i$ and
$\varsigma=(\varsigma_i)_{i\ge 1}$,
$\varsigma_i\in Z\Hom_{\overline{\CW}^i}(E_2^iF_1^i(m),m)$, satisfies that 
\begin{itemize}
	\item for all $i,j\ge 1$, we have 
		$\varsigma_i\circ E_2^iF_1^i\varsigma_j=\varsigma_{i+j}\circ\mu_{i,j}$
\item $\varsigma_i\circ T_rF_1^i=\varsigma_i\circ E_2^iT_r$ for all $1\le r<i$.
\end{itemize}

\smallskip
We define
 $\Hom_{\Delta_\lambda\CW}((m,\varsigma),(m',\varsigma'))$ to be the
differential submodule of $\Hom_{\overline{\CW}^i}(m,m')$ of
elements
$f$ such that for all $i\ge 1$, the following diagram commutes
$$\xymatrix{
	E_2^iF_1^i(m)\ar[r]^-{\varsigma_i}\ar[d]_{E_2^iF_1^if} &m \ar[d]^{f} \\
	E_2^iF_1^i(m')\ar[r]_-{\varsigma'_i} &m'
}$$

The composition of maps is defined to be that of 
$\overline{\CW}^i$.

\begin{rem}
	The structure of objects in $\Delta_\lambda \CW$ can be described graphically as follows:
$$\includegraphics[scale=0.82]{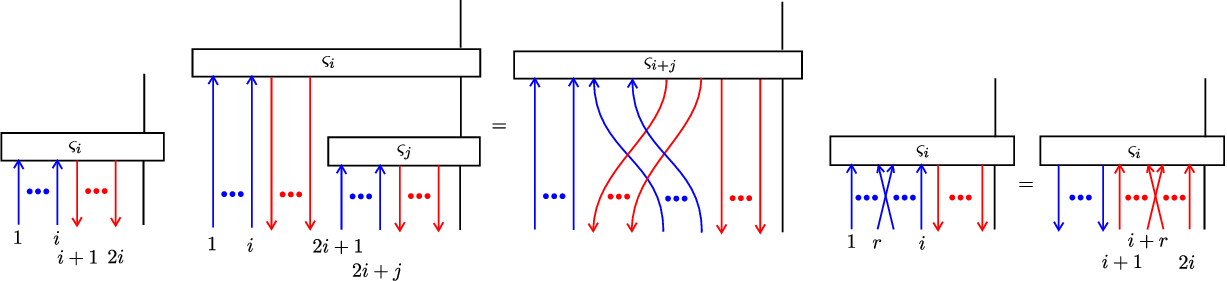}$$
\end{rem}

\begin{rem}
The maps $\mu_{i,j}$
make $A=\bigoplus_{i\ge 0}E_2^iF_1^i$ into a monoid in the monoidal category of
	endofunctors of $\overline{\CW}^i$, when $\overline{\CW}^i$ has enough direct sums. If $\overline{\CW}^i$ has enough
colimits, we have an induced monoid $\bar{A}=\bigoplus_{i\ge 0}(E_2^iF_1^i)
\otimes_{H_i\otimes H_i^\opp}H_i$.
Now, the category $\Delta_\lambda\CW$ is the category of $\bar{A}$-modules in
	$\overline{\CW}^i$.
\end{rem}

\begin{rem}
	\label{re:bi2repswap}
Let us define a lax bi-$2$-representation $E_{i,j}=E_2^jF_1^i$ on $\CW$
as deduced from the one defined
in \S\ref{se:2birep} by applying the swap automorphism of $\CU\times\CU$ (cf
Remark \ref{re:swapUU}).

There is a faithful differential functor $\Delta_\lambda\CW\to\Delta_E\CW,\
	(m,\varsigma)\mapsto (m,\varsigma_1)$.
\end{rem}

\subsubsection{Adjoint}
\label{se:adjointbimod}
We assume $F_1$ has a right adjoint $E_1$ and
denote by $\eps_1$ and $\eta_1$ the counit and unit of the adjunction. We denote by
$\tau_1$ the endomorphism of $E_1^2$ corresponding by adjunction to the endomorphism
$\tau_1$ of $F_1^2$. The pair $(E_1,\tau_1)$ provides an action of $\CU$ on $\CW$.

\begin{rem}
	The maps $\eta_1$, $\eps_1$, the relations they satisfy, and $\lambda$, $\sigma$ and $\rho$
	are described graphically as:
$$\includegraphics[scale=0.90]{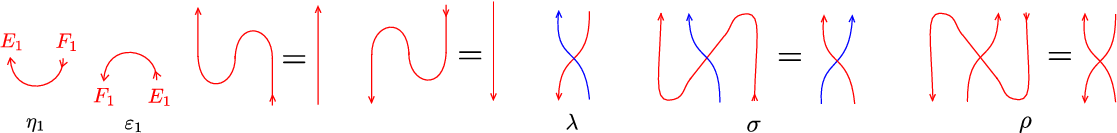}$$
\end{rem}

We denote by $\sigma$ the composition
\begin{equation}
	\label{eq:defsigma}
\sigma:E_2E_1\xrightarrow{\eta_1E_2E_1} E_1F_1E_2E_1 \xrightarrow{E_1\lambda E_2}
E_1E_2F_1E_1\xrightarrow{E_1E_2\eps_1} E_1E_2
\end{equation}
and by $\rho$ the composition
\begin{equation}
	\label{eq:defrho}
\rho:F_1E_1\xrightarrow{F_1E_1\eta_1}F_1E_1^2F_1\xrightarrow{F_1\tau_1 F_1}
F_1E_1^2F_1\xrightarrow{\eps_1E_1F_1}E_1F_1.
\end{equation}
The diagram (\ref{eq:diagsigma}) is commutative.

\begin{lemma}
	\label{le:braidsrhosigmalambda}
	We have
$$E_1\lambda\circ \rho E_2\circ F_1\sigma=\sigma F_1\circ E_2\rho\circ \lambda E_1
	\text{ and }
\rho F_1\circ F_1\rho\circ \tau_1 E_1=E_1\tau_1\circ\rho F_1\circ F_1\rho.$$
\end{lemma}
\begin{proof}
	We have
	$$E_1\lambda\circ \rho E_2\circ F_1\sigma=$$
	\begin{align*}
		&=E_1E_2F_1\eps_1\circ 
	E_1\lambda F_1E_1\circ F_1E_1^2F_1\lambda E_1\circ F_1\tau_1 F_1^2E_2E_1\circ
	F_1E_1\eta_1 F_1E_2E_1\circ F_1\eta_1E_2E_2 \\
		&= E_1E_2F_1\eps_1\circ  E_1\lambda F_1E_1\circ F_1E_1^2F_1\lambda E_1\circ
F_1E_1^2\tau_1 E_2E_1\circ F_1E_1\eta_1 F_1E_2E_1\circ F_1\eta_1E_2E_2\\
		&=E_1E_2F_1\eps_1\circ E_1\lambda F_1E_1\circ E_1F_1\lambda E_1
\circ E_1\tau_1 E_2E_1 \circ \eta_1 F_1E_2E_1\circ \eps_1F_1E_2E_1\circ
	F_1\eta_1E_2E_2\\
		&=E_1E_2F_1\eps_1\circ E_1\lambda F_1E_1\circ E_1F_1\lambda E_1
\circ E_1\tau_1 E_2E_1 \circ \eta_1 F_1E_2E_1\\
		&=E_1E_2F_1\eps_1\circ E_1E_2\tau_1 E_1\circ E_1\lambda F_1E_1\circ
	E_1F_1\lambda E_1 \circ  \eta_1 F_1E_2E_1\\
		&=E_1E_2F_1\eps_1\circ E_1E_2F_1\eps_1E_1F_1\circ E_1E_2F_1^2E_1\eta_1\circ
	E_1E_2\tau_1 E_1\circ E_1\lambda F_1E_1\circ
	E_1F_1\lambda E_1 \circ  \eta_1 F_1E_2E_1\\
		&=E_1E_2F_1\eps_1\circ E_1E_2F_1\eps_1E_1F_1\circ E_1E_2\tau_1E_1^2F_1\circ
	E_1E_2F_1^2E_1\eta_1\circ E_1\lambda F_1E_1\circ
	E_1F_1\lambda E_1 \circ  \eta_1 F_1E_2E_1\\
	&=E_1E_2F_1\eps_1\circ E_1E_2F_1\eps_1E_1F_1\circ E_1E_2F_1^2\tau_1F_1\circ
	E_1E_2F_1^2E_1\eta_1\circ E_1\lambda F_1E_1\circ
	E_1F_1\lambda E_1 \circ  \eta_1 F_1E_2E_1\\
		&=\sigma F_1\circ E_2\rho\circ \lambda E_1.
	\end{align*}

	We have
	\begin{align*}
		\rho F_1\circ F_1\rho\circ \tau_1 E_1&=
\eps_1E_1F_1^2\circ F_1\eps_1E_1^2F_1^2\circ \tau_1E_1^3F_1^2\circ F_1^2E_1\tau_1F_1^2
\circ F_1^2E_1^2\eta_1F_1\circ F_1^2\tau_1F_1\circ F_1^2E_1\eta_1\\
		&=\eps_1E_1F_1^2\circ F_1\eps_1E_1^2F_1^2\circ F_1^2\tau_1E_1F_1^2\circ F_1^2E_1\tau_1F_1^2
\circ F_1^2E_1^2\eta_1F_1\circ F_1^2\tau_1F_1\circ F_1^2E_1\eta_1\\
		&=\eps_1E_1F_1^2\circ F_1\eps_1E_1^2F_1^2\circ F_1^2(\tau_1E_1\circ E_1\tau_1
\circ\tau_1E_1)F_1^2\circ F_1^2E_1^2\eta_1 F_1\circ F_1^2E_1\eta_1\\
		&=\eps_1E_1F_1^2\circ F_1\eps_1E_1^2F_1^2\circ F_1^2(E_1\tau_1
\circ\tau_1E_1\circ E_1\tau_1)F_1^2\circ F_1^2E_1^2\eta_1 F_1\circ F_1^2E_1\eta_1\\
		&=\eps_1E_1F_1^2\circ F_1\eps_1E_1^2F_1^2\circ F_1^2(E_1\tau_1
\circ\tau_1E_1)F_1^2\circ F_1^2E_1^3\tau_1\circ F_1^2E_1^2\eta_1 F_1\circ F_1^2E_1\eta_1\\
		&=E_1\tau_1\circ\rho F_1\circ F_1\rho.
	\end{align*}
\end{proof}

\subsubsection{Relations}

	Let $\CM$ be the strict monoidal pointed category generated by objects $a_l$ for
	$1\le l\le 3$ and maps $\lambda_{lm}:a_la_m\to a_ma_l$ for $l\le m$ with
	relations $\lambda_{ll}^2=0$ and
	$$\lambda_{mn}l\circ m\lambda_{ln}\circ \lambda_{lm}n=
	n\lambda_{lm}\circ \lambda_{ln}m\circ l\lambda_{mn}
	\text{ for }l\le m\le n.$$

\begin{lemma}
	\label{le:monoidalmaps}
	We have a pointed faithful strict monoidal functor
	$$H:\CM\to \CU^\bullet,\ a_l\mapsto e,\
	\lambda_{lm}\mapsto\tau.$$
	Given $l_1,\ldots,l_r,m_1,\ldots,m_r\in\{1,2,3\}$, the non-zero elements of
	$H(\Hom_\CM(a_{l_1}\cdots a_{l_r},a_{m_1}\cdots a_{m_r}))\subset H_r^\bullet$
	are those $T_w$ with $w\in\GS_r$ such that for all $i,j\in\{1,\ldots,r\}$ with $i<j$ and
	$w(i)>w(j)$, we have $l_i\le l_j$.
\end{lemma}

\begin{proof}
	Given the defining relations for $\CU^\bullet$, the construction of the
	lemma does define (uniquely) a monoidal functor $H$.

	Fix $l_1,\ldots,l_n\in\{1,\ldots,3\}$. Given $i\in\{1,\ldots,n-1\}$ such that
	$l_i\le l_{i+1}$, we put $\tilde{T}_i=a_{l_1}\cdots a_{l_{i-1}}
	\lambda_{l_i,l_{i+1}}a_{l_{i+2}}\cdots a_{l_n}$. Note that $\tilde{T}_i\tilde{T}_{i+1}\tilde{T}_i$ is
	well-defined if and only if $l_i\le l_{i+1}\le l_{i+2}$, hence if and only if
	$\tilde{T}_{i+1}\tilde{T}_i\tilde{T}_{i+1}$ is well-defined. As a consequence, given
	$i_1,\ldots,i_r,j_1,\ldots,j_s\in\{1,\ldots,n-1\}$ such that
	$\tilde{T}_{i_1}\cdots \tilde{T}_{i_r}$ and $\tilde{T}_{j_1}\cdots\tilde{T}_{j_s}$ are
	well-defined and $T_{i_1}\cdots T_{i_r}=T_{j_1}\cdots T_{j_s}$, then we have
	$\tilde{T}_{i_1}\cdots \tilde{T}_{i_r}=\tilde{T}_{j_1}\cdots\tilde{T}_{j_s}$. This shows the
	faithfulness of $H$.

	\smallskip
	Consider $i_1,\ldots,i_r$ such that $\tilde{T}_{i_1}\cdots \tilde{T}_{i_r}$ is well-defined and 
	non-zero. Let $w=s_{i_1}\cdots s_{i_r}\in\GS_n$. We show by induction on $r$ that
	given $(i,j)\in\tilde{L}(w)$, we have $l_i\le l_j$.
	
	Let $w'=s_{i_1}\cdots s_{i_{r-1}}$.
	Put $d=i_r$ and $w'=ws_d$.
	Since $T_{i_1}\cdots T_{i_r}\neq 0$, we
	have $r=\ell(w)$. We have $\tilde{L}(w)=\{(d,d+1)\}\coprod s_d(\tilde{L}(w'))$ by
	Lemma \ref{le:lengthaffine}. We have a well-defined map
	$\tilde{T}_{i_1}\cdots \tilde{T}_{i_{r-1}}$ from $a_{l_1}\cdots a_{l_{d-1}}a_{l_{d+1}}
	a_{l_d}a_{l_{d+2}}\cdots a_{l_n}$. It follows by induction that given $(i,j)\in \tilde{L}(w')$, we
	have $l_{s_d(i)}\le l_{s_d(j)}$. Since
	$\tilde{L}(w)=\{(d,d+1)\}\coprod s_d(\tilde{L}(w'))$ (Lemma \ref{le:lengthaffine}), we deduce that
	$l_i\le l_j$ for all $(i,j)\in \tilde{L}(w)$.

	\smallskip
	Consider now $w\in\GS_n$ such that given $(i,j)\in\tilde{L}(w)$, we have $l_i\le l_j$.
	Let $w=s_{i_1}\cdots s_{i_r}$ be a reduced decomposition of $w$. 
	We show by induction on $r$ that $\tilde{T}_{i_1}\cdots\tilde{T}_{i_r}$ is well-defined.
	As before, we define $d$ and $w'$. By induction on $r$, the element
	$\tilde{T}_{i_1}\cdots\tilde{T}_{i_{r-1}}$ gives a well-defined map from
	$a_{l_1}\cdots a_{l_{d-1}}a_{l_{d+1}} a_{l_d}a_{l_{d+2}}\cdots a_{l_n}$. Since
	$(d,d+1)\in \tilde{L}(w)$, it follows that $l_d\le l_{d+1}$, hence
	$\tilde{T}_d$ is a well-defined map from $a_{l_1}\cdots a_{l_n}$. We deduce that 
	$\tilde{T}_{i_1}\cdots\tilde{T}_{i_r}$. This shows that $T_w$ is in the image of $H$.
\end{proof}

Given $l_1,\ldots,l_r,m_1,\ldots,m_r\in\{1,2,3\}$ and $w\in\GS_r$ satisfying the
assumptions of Lemma \ref{le:monoidalmaps},
we put $\lambda_w=H^{-1}(T_w)$.

\smallskip
We denote by $\CM'$ the strict monoidal $k$-linear category obtained from $k[\CM]$ by
adding maps $\eps:a_1a_3\to 1$ and $\eta:1\to a_3a_1$ and relations
$$a_3\eps\circ \eta a_3=\id,\ \eps a_1\circ a_1\eta=\id$$
$$\lambda_{23}=a_3a_2\eps\circ a_3\lambda_{12}a_3\circ\eta a_2a_3,\
\lambda_{13}=\eps a_3a_1\circ a_1\lambda_{33}a_1\circ a_1a_3\eta$$
$$\lambda_{11}=\eps a_1^2\circ a_1\eps a_3a_1^2\circ a_1^2\lambda_{33}a_1^2\circ a_1^2a_3\eta a_1
\circ a_1^2\eta.$$

There is a monoidal duality, i.e. a monoidal equivalence $\CM^{\prime\opp}\iso\CM'$ given by
$$a_1\mapsto a_3,\ a_2\mapsto a_2,\ a_3\mapsto a_1,\ \lambda_{12}\mapsto\lambda_{23},\ 
\lambda_{23}\mapsto\lambda_{12},\ \lambda_{13}\mapsto\lambda_{13}$$
$$\lambda_{11}\mapsto\lambda_{33},\ \lambda_{22}\mapsto\lambda_{22},\
\lambda_{33}\mapsto\lambda_{11},\ \eps\mapsto\eta,\ \eta\mapsto\eps.$$

\begin{lemma}
	\label{le:etaepsrhosigmalambda0}
	Let $G_1,\ldots,G_n\in\{a_1,a_2,a_3\}$. We have
	$$\lambda_{(1\cdots n+1)}\circ G_1\cdots G_n \eta=
	\lambda_{(n+2\cdots 2)}\circ \eta G_1\cdots G_n:
	G_1\cdots G_n\to a_3 G_1\cdots G_n a_1$$
	and
	$$\eps G_1\cdots G_n\circ \lambda_{(2\cdots n+2)}=G_1\cdots G_n\eps\circ
	\lambda_{(n+1\cdots 1)}:a_1G_1\cdots G_n a_3\to G_1\cdots G_n.$$
\end{lemma}
\begin{proof}
	We have
	\begin{align*}
		a_3\lambda_{13}\circ \eta a_3&=
	a_3\eps a_3a_1\circ a_3a_1\lambda_{33}a_1\circ a_3a_1a_3\eta\circ \eta a_3\\
		&=a_3\eps a_3a_1\circ \eta a_3^2a_1\circ \lambda_{33}a_1\circ a_3\eta\\
		&=\lambda_{33}a_1\circ a_3\eta
	\end{align*}

	\begin{align*}
		\lambda_{13} a_1\circ a_1\eta&=
		\eps a_3a_1^2\circ a_1\lambda_{33}a_1^2\circ a_1a_3\eta a_1\circ a_1\eta\\
		&=\eps a_3a_1^2\circ a_1a_3^2\lambda_{11}\circ a_1a_3\eta a_1\circ a_1\eta\\
		&=a_3\lambda_{11}\circ\eps a_3a_1^2\circ a_1a_3\eta a_1\circ a_1\eta\\
		&=a_3\lambda_{11}\circ\eta a_1\circ \eps a_1\circ a_1\eta\\
		&=a_3\lambda_{11}\circ\eta a_1
	\end{align*}

	\begin{align*}
		\lambda_{23} a_1\circ a_2\eta&=
		a_3a_2\eps a_1\circ a_3\lambda_{12} a_3a_1\circ \eta a_2a_3a_1\circ a_2\eta\\
		&=a_3a_2\eps a_1\circ a_3a_2a_1\eta\circ a_3\lambda_{12}\circ\eta a_2\\
		&=a_3\lambda_{12}\circ\eta a_2
	\end{align*}

	It follows that the first statement of the lemma holds when $n=1$. Consider now $n\ge 2$. We prove
	the first statement of the lemma by induction on $n$.
	We have
	\begin{align*}
		\lambda_{(n+2\cdots 2)}\circ \eta G_1\cdots G_n&=\lambda_{(n+2\cdots 3)}\circ
	(\lambda_{(23)}\circ \eta G_1)G_2\cdots G_n\\
		&=\lambda_{(n+2\cdots 3)}\circ (\lambda_{(12)}\circ G_1\eta)G_2\cdots G_n\\
		&=\lambda_{(12)}\circ G_1(\lambda_{(n+1\cdots 2)}\circ \eta G_2\cdots G_n)\\
		&=\lambda_{(12)}\circ G_1(\lambda_{(1\cdots n)}\circ G_2\cdots G_n\eta)\\
		&=\lambda_{(1\cdots n+1)}\circ G_1\cdots G_n\eta
	\end{align*}

	The second statement of the lemma follows by applying the duality of $\CM'$.
\end{proof}

\smallskip
Lemmas \ref{le:braidsrhosigmalambda} and \ref{le:monoidalmaps} show that there is
a $k$-linear monoidal functor $R:\CM'\to\CW$
$$a_1\mapsto F_1,\ a_2\mapsto E_2,\ a_3\mapsto E_1,\ 
\lambda_{12}\mapsto\lambda, \lambda_{23}\mapsto\sigma,\ \lambda_{13}\mapsto\rho,\ 
\lambda_{11}\mapsto\tau_1,\ \lambda_{22}\mapsto\tau_2,\ \lambda_{33}\mapsto\tau_1$$
$$\eta\mapsto\eta_1,\ \eps\mapsto\eps_1.$$

Given $l_1,\ldots,l_r,m_1,\ldots,m_r\in\{1,2,3\}$ and $w\in\GS_r$ satisfying the
assumptions of Lemma \ref{le:monoidalmaps},
we still denote by $\lambda_w$ the element $R(\lambda_w)$.

Lemma \ref{le:etaepsrhosigmalambda0} has the following consequence.
\begin{lemma}
	\label{le:etaepsrhosigmalambda}
	Let $G_1,\ldots,G_n\in\{E_1,E_2,F_1\}$. We have
	$$\lambda_{(1\cdots n+1)}\circ G_1\cdots G_n \eta_1=
	\lambda_{(n+2\cdots 2)}\circ \eta_1 G_1\cdots G_n:
	G_1\cdots G_n\to E_1 G_1\cdots G_n F_1$$
	and
	$$\eps_1 G_1\cdots G_n\circ \lambda_{(2\cdots n+2)}=G_1\cdots G_n\eps_1\circ
	\lambda_{(n+1\cdots 1)}:F_1G_1\cdots G_n E_1\to G_1\cdots G_n.$$
\end{lemma}

\subsubsection{$1$-arrows}
Let $(m,\varsigma)\in\Delta_\lambda \CW$.
Let $\pi=\pi(\varsigma)$ be the composition
$$\pi:E_2(m)\xrightarrow{E_2\eta_1}E_2E_1F_1(m)\xrightarrow{\sigma F_1}
E_1E_2F_1(m)\xrightarrow{E_1\varsigma_1}E_1(m).$$
Note that $\pi$ is also equal to the composition
$$\pi:E_2(m)\xrightarrow{\eta_1E_2}E_1F_1E_2(m)\xrightarrow{E_1\lambda}
E_1E_2F_1(m)\xrightarrow{E_1\varsigma_1}E_1(m)$$
since $E_1\lambda E_1F_1\circ \eta_1 E_2E_1F_1\circ E_2\eta_1=
E_1E_2F_1\eta_1 \circ E_1\lambda\circ \eta_1 E_2$ and
$E_1E_2F_1\eta_1\circ E_1E_2F_1\eps_1=\id_{E_1E_2F_1}$.

\smallskip
The pair $(m,\pi)$ defines an object of $\Delta_\sigma\CW$.
We obtain a faithful differential functor
$\Gamma:\Delta_\lambda\CW\to\Delta_\sigma\CW,\ (m,\varsigma)\mapsto (m,\pi)$.

\begin{rem}
	The construction of $\pi$ from $\varsigma_1$ is illustrated below.
$$\includegraphics[scale=1.3]{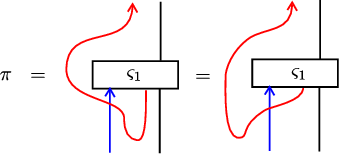}$$
\end{rem}

\smallskip
We define now a differential functor $E:\Delta_\lambda\CW\to\Delta_\lambda\CW$.

Let $(m,\varsigma)\in\Delta_\lambda\CW$.
Let $m'={\xy (0,0)*{E_2(m)\oplus E_1(m)},
\ar@/^/^{\pi}(-8,3)*{};(7,3)*{}, \endxy}$ where $\pi=\pi(\varsigma_1)$. Given
$i\ge 1$, we define
	
	$$\varsigma'_i=\left(\begin{matrix}
		E_2\varsigma_i\circ \lambda_{(1\cdots 2i+1)} & 
		\sum_{r=1}^{i} E_2\varsigma_{i-1}\circ
		E_2^iF_1^{i-1}\eps_1\circ \lambda_{(1\cdots r)(2i\cdots i+r)}\\
		0 & E_1\varsigma_i\circ \lambda_{(1\cdots 2i+1)}
	\end{matrix}\right)
		:E_2^iF_1^i(m')\to m'$$

\begin{lemma}
$(m',\varsigma')$ is an object of $\Delta_\lambda\CW$.
\end{lemma}

\begin{proof}
	We have
	\begin{align*}
		d((\varsigma'_i)_{11})&=E_2\varsigma_i\circ 
	d(\tau_2 E_2^{i-1}\circ\cdots\circ E_2^{i-1}\tau_2)F_1^i\circ \lambda_{(i+1\cdots 2i+1)}\\
		&=\sum_{r=1}^i E_2\varsigma_i\circ \lambda_{(1\cdots r)(r+1\cdots i+1)}
\circ \lambda_{(i+1\cdots 2i+1)}\\
		&=\sum_{r=1}^i E_2\varsigma_i\circ \lambda_{(1\cdots r)(r+1\cdots i+1)}
\circ \lambda_{(i+1\cdots 2i+1)}\\
		&=\sum_{r=1}^i E_2\varsigma_i\circ \lambda_{(1\cdots r)(2i+1\cdots i+r+1)}
\circ \lambda_{(i+1\cdots 2i+1)}\\
		&=E_2\varsigma_i\circ \lambda_{(1\cdots i)}
	\end{align*}

	$$(\varsigma'_i)_{12}\circ E_2^iF_1^i\pi=$$
	\begin{align*}
		&=\sum_{r=1}^i
	E_2\varsigma_{i-1}\circ E_2^iF_1^{i-1}\varsigma_1\circ \lambda_{(2i,2i+1)}\circ
	E_2^iF_1^{i-1}\eps_1F_1E_2\circ E_2^iF_1^i\eta_1E_2\circ \lambda_{(1\cdots r)
	(2i\cdots i+r)}\\
		&=\sum_{r=1}^iE_2\varsigma_i\circ\lambda_{(i+1\cdots 2i)}\circ\lambda_{(2i,2i+1)}\circ
\lambda_{(1\cdots r) (2i\cdots i+r)}\\
		&=E_2\varsigma_i\circ \lambda_{(1\cdots i)}\\
		&=d((\varsigma'_i)_{11}).
	\end{align*}

	\begin{align*}
		d((\varsigma'_i)_{22})&=E_1\varsigma_i\circ \lambda_{(1\cdots i+1)}\circ
	E_2^id(\rho F_1^{i-1}\circ F_1\rho F_1^{i-2}\circ\cdots\circ F_1^{i-1}\rho)\\
		&=\sum_{r=1}^i E_1\varsigma_i\circ
	\lambda_{(1\cdots i+r)}\circ E_2^iF_1^{r-1}\eta_1F_1^{i-r}\circ
	E_2^i F_1^{r-1}\eps_1F_1^{i-r}\circ \lambda_{(i+r+1\cdots 2i+1)}\\
		&=\sum_{r=1}^i E_1\varsigma_i\circ
	\lambda_{(i+r+1\cdots 2)}\circ \eta_1E_2^iF_1^{i-1}\circ
	E_2^i F_1^{i-1}\eps_1\circ \lambda_{(2i\cdots i+r)}\\
		&=\sum_{r=1}^i E_1\varsigma_i\circ
	\lambda_{(i+r+1\cdots i+2)}\circ\lambda_{(i+2\cdots 2)}\circ
	\eta_1E_2^iF_1^{i-1}\circ
	E_2^i F_1^{i-1}\eps_1\circ \lambda_{(2i\cdots i+r)}\\
		&=\sum_{r=1}^i E_1\varsigma_i\circ
	\lambda_{(2\cdots r+1)}\circ\lambda_{(i+2\cdots 2)}\circ
	\eta_1E_2^iF_1^{i-1}\circ
	E_2^i F_1^{i-1}\eps_1\circ \lambda_{(2i\cdots i+r)}\\
		&= E_1\varsigma_i\circ \lambda_{(i+2\cdots 2)}\circ
	\eta_1E_2^iF_1^{i-1}\circ
	E_2^i F_1^{i-1}\eps_1\circ \lambda_{(2i\cdots i+1)}
	\end{align*}
	
	\begin{align*}
		\pi\circ (\varsigma'_i)_{12}&=\sum_{r=1}^i E_1\varsigma_1\circ
	\lambda_{(12)}\circ E_2\eta_1\circ E_2\varsigma_{i-1}\circ 
	E_2^iF_1^{i-1}\eps_1\circ \lambda_{(1\cdots r)(2i\cdots i+r)}\\
		&=\sum_{r=1}^i E_1\varsigma_1\circ
	\lambda_{(23)}\circ \eta_1E_2\circ E_2\varsigma_{i-1}\circ 
	E_2^iF_1^{i-1}\eps_1\circ \lambda_{(1\cdots r)(2i\cdots i+r)}\\
		&=\sum_{r=1}^i E_1\varsigma_1\circ E_1E_2F_1\varsigma_{i-1}\circ \lambda_{23}\circ
	\eta_1 E_2^iF_1^{i-1}\circ
	E_2^iF_1^{i-1}\eps_1\circ \lambda_{(1\cdots r)(2i\cdots i+r)}\\
		&=\sum_{r=1}^i E_1\varsigma_i\circ \lambda_{(i+2\cdots 3)}\circ \lambda_{23}\circ
	\lambda_{(3\cdots r+2)}\circ \eta_1 E_2^iF_1^{i-1}\circ
	E_2^iF_1^{i-1}\eps_1\circ \lambda_{(2i\cdots i+r)}\\
		&=E_1\varsigma_i\circ \lambda_{(i+2\cdots 2)}
	\circ \eta_1 E_2^iF_1^{i-1}\circ
	E_2^iF_1^{i-1}\eps_1\circ \lambda_{(2i\cdots i+r)}=d((\varsigma'_i)_{22}).
	\end{align*}

	We have
	$$d((\varsigma'_i)_{12})=A+B$$
	where 
	\begin{align*}
		A&=\sum_{1\le s<r\le i}E_2\varsigma_{i-1}\circ
	E_2^iF_1^{i-1}\eps_1\circ \lambda_{(1\cdots s)(s+1\cdots r)(2i\cdots i+r)}\\
		&=\sum_{1\le s<r\le i}E_2\varsigma_{i-1}\circ\lambda_{(s+1\cdots r)}\circ
	E_2^iF_1^{i-1}\eps_1\circ \lambda_{(1\cdots s)(2i\cdots i+r)}\\
		&=\sum_{1\le s<r\le i}E_2\varsigma_{i-1}\circ\lambda_{(i+r-1\cdots s+i)}\circ
	E_2^iF_1^{i-1}\eps_1\circ \lambda_{(1\cdots s)(2i\cdots i+r)}\\
		&=\sum_{1\le s<r\le i}E_2\varsigma_{i-1}\circ
	E_2^iF_1^{i-1}\eps_1\circ \lambda_{(1\cdots s)(2i\cdots i+r)(i+r-1\cdots i+s)}
	\end{align*}
	and
	$$B=
	\sum_{\substack{1\le r'\le i\\ 1\le s'\le i-r'}} E_2\varsigma_{i-1}\circ
	E_2^iF_1^{i-1}\eps_1\circ \lambda_{(1\cdots r')(2i\cdots i+r'+s')(i+r'+s'-1\cdots
	i+r')}$$
	So $A=B$ and $d((\varsigma'_i)_{12})=0$.

	We have shown that $d(\varsigma'_i)=0$,

	\medskip
	Fix $r\in\{1,\ldots,i\}$. We put
	$b_r=E_2\varsigma_{i-1}\circ E_2^iF_{i-1}\eps_1\circ\lambda_{(1\cdots r)(2i\cdots i+r)}:
	E_2^iF_1^iE_1(m)\to E_2(m)$.

	Consider $s\in\{1,\ldots,i-1\}$.

	If $s>r$, we have
	\begin{align*}
		b_r(T_s\otimes 1)&=
	E_2\varsigma_{i-1}\circ\lambda_{(s,s+1)}
	\circ E_2^iF_{i-1}\circ\lambda_{(1\cdots r)(2i\cdots i+r)}\\
		&=E_2\varsigma_{i-1}\circ\lambda_{(i+s-1,i+s)}
	\circ E_2^iF_{i-1}\circ\lambda_{(1\cdots r)(2i\cdots i+r)}\\
		&=E_2\varsigma_{i-1}
	\circ E_2^iF_{i-1}\circ\lambda_{(i+s-1,i+s)}\lambda_{(1\cdots r)(2i\cdots i+r)}\\
		&=E_2\varsigma_{i-1}
	\circ E_2^iF_{i-1}\circ\lambda_{(1\cdots r)(2i\cdots i+r)}\lambda_{(i+s,i+s+1)}\\
		&=b_r(1\otimes T_s).
	\end{align*}

	If $s<r-1$, we have
	\begin{align*}
		b_r(1\otimes T_s)&=
	E_2\varsigma_{i-1}\circ\lambda_{(i+s,i+s+1)}
	\circ E_2^iF_{i-1}\circ\lambda_{(1\cdots r)(2i\cdots i+r)}\\
		&=E_2\varsigma_{i-1}\circ\lambda_{(s+1,s+2)}
	\circ E_2^iF_{i-1}\circ\lambda_{(1\cdots r)(2i\cdots i+r)}\\
		&=E_2\varsigma_{i-1}
	\circ E_2^iF_{i-1}\circ\lambda_{(s+1,s+2)}\circ\lambda_{(1\cdots r)(2i\cdots i+r)}\\
		&=E_2\varsigma_{i-1}
	\circ E_2^iF_{i-1}\circ\lambda_{(1\cdots r)(2i\cdots i+r)}\circ
	\lambda_{(s,s+1)}\\
		&=b_r(T_s\otimes 1).
	\end{align*}

	We have
	$$b_r(T_{r-1}\otimes 1)=E_2\varsigma_{i-1}
	\circ E_2^iF_{i-1}\circ\lambda_{(1\cdots r)(2i\cdots i+r)}
	\circ\lambda_{(r-1,r)}=0$$

$$b_r(1\otimes T_r)=
	E_2\varsigma_{i-1}
	\circ E_2^iF_{i-1}\circ\lambda_{(1\cdots r)(2i\cdots i+r)}\circ
	\lambda_{(i+r,i+r+1)}=0$$

$$b_r(1\otimes T_{r-1})=E_2\varsigma_{i-1}
	\circ E_2^iF_{i-1}\circ\lambda_{(1\cdots r)(2i\cdots i+r-1)}
	=b_{r-1}(T_{r-1}\otimes 1).$$

	\smallskip
We have shown that $(\varsigma_i)_{12}(1\otimes T_s)=(\varsigma_i)_{12}(T_s\otimes 1)$.

\smallskip
We have
\begin{align*}
	(\varsigma'_i)_{11}(T_s\otimes 1)&=E_2\varsigma_i\circ\lambda_{(s+1,s+2)}\circ
\lambda_{(1\cdots 2i+1)}\\
	&= E_2\varsigma_i\circ\lambda_{(i+s+1,i+s+2)}\circ \lambda_{(1\cdots 2i+1)}\\
	&= E_2\varsigma_i\circ\lambda_{(1\cdots 2i+1)}\lambda_{(i+s,i+s+1)}\\
	&=(\varsigma'_i)_{11}(1\otimes T_s).
\end{align*}

Similarly,
$$(\varsigma'_i)_{22}(T_s\otimes 1)=(\varsigma'_i)_{11}(1\otimes T_s).$$

	So $\varsigma_i (1\otimes T_s)=\varsigma_i (T_s\otimes 1)$.

\medskip	
		Let $l\in\{1,2\}$. We have
$$(\varsigma'_{i+j})_{ll}\circ\mu_{ij}=E_l\varsigma_{i+j}\circ
\lambda_w$$
where $w(r)=r$ and $w(i+r)=i+r+j+1$ for $1\le r\le i$,
$w(2i+r)=i+r$ and $w(2i+j+r)=2i+j+r+1$ for $1\le r\le j$ and
$w(2i+2j+1)=i+j+1$.

We have
\begin{align*}
	(\varsigma'_i)_{ll}\circ(\varsigma'_j)_{ll}&=
E_l\varsigma_i\circ E_lE_2^iF_1^i\varsigma_j\circ\lambda_{(1\cdot 2i+1)}\circ
\lambda_{(2i+1\cdots 2i+2j+1)}\\
	&=E_l\varsigma_{i+j}\circ \lambda_{w'}\circ\lambda_{(1\cdot 2i+1)}\circ
\lambda_{(2i+1\cdots 2i+2j+1)}
\end{align*}
where $w'(r)=r$ for $1\le r\le i+1$, $w'(i+1+r)=i+j+1+r$ for $1\le r\le i$,
$w'(1+2i+r)=1+i+r$ and $w'(1+2i+j+r)=1+2i+j+r$ for $1\le r\le j$.

It follows that $(\varsigma'_{i+j})_{ll}\circ\mu_{ij}=(\varsigma'_i)_{ll}\circ(\varsigma'_j)_{ll}$.

\smallskip
	Given $l\le l'\le 1$, we put
	$b_{l',l}=E_2\varsigma_{l'-1}\circ E_2^{l'}F_{l'-1}\eps_1\circ
	\lambda_{(1\cdots l)(2l'\cdots l'+l)}: E_2^{l'}F_1^{l'}E_1(m)\to E_2(m)$.
 We denote by $w_{l_1,l_2}$ the permutation of $\GS_{l_1+l_2}$ given by
 $s\mapsto s+l_2$ for $1\le s\le l_1$ and $s\mapsto s-l_1$ for $l_1+1\le s\le l_1+l_2$.

Consider $r\in\{1,\ldots,i\}$. We have
\begin{align*}
	b_{i,r}\circ(\varsigma'_j)_{22}&=E_2\varsigma_{i-1}\circ E_2^iF_1^{i-1}\varsigma_j\circ
E_2^iF_1^{i-1}\eps_1E_2^jF_1^j\circ \lambda_{(1\cdots r)(2i\cdots i+r)}\circ
\lambda_{(2i+1\cdots 2i+2j+1)}\\
	&=E_2\varsigma_{i+j-1}\circ \lambda_{w_{i-1,j}}\circ
E_2^iF_1^{i-1}\eps_1E_2^jF_1^j\circ \lambda_{(2i+1\cdots 2i+2j+1)}\circ
\lambda_{(1\cdots r)(2i\cdots i+r)}\\
	&=E_2\varsigma_{i+j-1}\circ E_2^i\lambda_{w_{i-1,j}}F_1^j\circ
E_2^iF_1^{i-1}E_2^jF_1^{j-1}\eps_1\circ \lambda_{(2i+2j\cdots 2i)}\circ
\lambda_{(1\cdots r)(2i\cdots i+r)}\\
	&=E_2\varsigma_{i+j-1}\circ E_2^i\lambda_{w_{i-1,j}}F_1^j\circ
E_2^iF_1^{i-1}E_2^jF_1^{j-1}\eps_1\circ \lambda_{(2i+2j\cdots i+r)}\\
	&=E_2\varsigma_{i+j-1}\circ E_2^{i+j}F_1^{i+j-1}\eps_1\circ E_2^i\lambda_{w_{i-1,j}}
F_1^{j+1}E_1\circ
\lambda_{(2i+2j\cdots i+r)}\\
	&=E_2\varsigma_{i+j-1}\circ E_2^{i+j}F_1^{i+j-1}\eps_1\circ 
\lambda_{(2i+2j\cdots i+j+r)}\circ E_2^i\lambda_{w_{i,j}}F_1^jE_1\\
	&=b_{i+j,r}\circ\mu_{i,j}.
\end{align*}

Consider $r\in\{1,\ldots,j\}$. We have
\begin{align*}
	(\varsigma'_i)_{11}\circ b_{j,r}&=
E_2\varsigma_i\circ E_2^{i+1}F_1^i\varsigma_{j-1}\circ\lambda_{(1\cdots 2i+1)}\circ 
E_2^iF_1^iE_2^jF_1^{j-1}\eps_1\circ\lambda_{(2i+1\cdots 2i+r)(2i+2j\cdots 2i+j+r)}\\
	&=E_2\varsigma_{i+j-1}\circ E_2^{i+1}\lambda_{w_{i,j-1}}F_1^{j-1}\circ\lambda_{(1\cdots 2i+1)}\circ 
E_2^iF_1^iE_2^jF_1^{j-1}\eps_1\circ\lambda_{(2i+1\cdots 2i+r)(2i+2j\cdots 2i+j+r)}\\
	&=E_2\varsigma_{i+j-1}\circ E_2^{i+j}F_1^{i+j-1}\eps_1\circ
E_2^{i+1}\lambda_{w_{i,j-1}}F_1^jE_1\circ\lambda_{(1\cdots 2i+1)}\circ 
\lambda_{(2i+1\cdots 2i+r)(2i+2j\cdots 2i+j+r)}\\
	&=E_2\varsigma_{i+j-1}\circ E_2^{i+j}F_1^{i+j-1}\eps_1\circ
\lambda_{(1\cdots i+r)(2i+2j\cdots 2i+j+r)}\circ E_2^i\lambda_{w_{i,j}}F_1^jE_1\\
	&=b_{i+j,j+r}\circ\mu_{i,j}.
\end{align*}

\smallskip
It follows that 
for all $i,j\ge 1$, we have 
		$\varsigma_i\circ E_2^iF_1^i\varsigma_j=\varsigma_{i+j}\circ\mu_{i,j}$.
\end{proof}

\begin{rem}
	The graphical description of $\varsigma'$ is the following:
$$\includegraphics[scale=1.2]{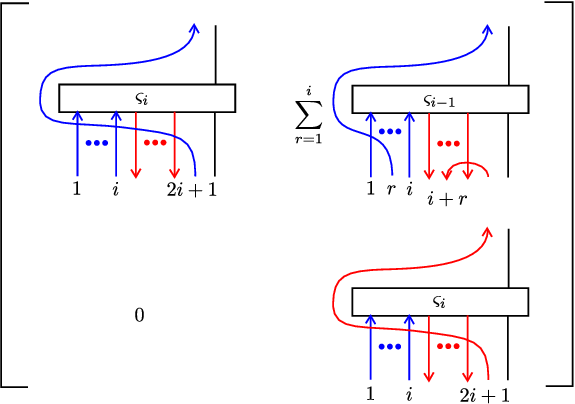}$$
\end{rem}

 Given $f\in\Hom_{\Delta_\lambda\CW}((m,\varsigma),(\tilde{m},\tilde{\varsigma}))$,
we put $E(f)=\left(\begin{matrix}E_2(f)&0\\0&E_1(f)
\end{matrix}\right)$.

\begin{lemma}
	We have $E(f)\in \Hom_{\Delta_\lambda\CW}(E(m,\varsigma),E(\tilde{m},\tilde{\varsigma}))$.
The construction makes $E$ into a differential endofunctor of $\Delta_\lambda\CW$.
\end{lemma}

\begin{proof}
	The lemma follows from the commutativity of the following diagram:

	$${\xy
		(0,0)*{E_2^iF_1^iE_2(m) \oplus E_2^iF_1^iE_1(m)},
		(120,0)*{E_2(m)\oplus E_1(m)},
		(0,-20)*{E_2^iF_1^iE_2(\tilde{m}) \oplus E_2^iF_1^iE_1(\tilde{m})},
		(120,-20)*{E_2(\tilde{m})\oplus E_1(\tilde{m})},
		\ar@/^2pc/^(0.25){E_2\varsigma_i\circ \lambda_{(1\cdots 2i+1)}}(-15,3)*{};(110,3)*{},
		\ar@/^2pc/^(0.75){E_1\varsigma_i\circ \lambda_{(1\cdots 2i+1)}}(15,3)*{};(128,3)*{},
		\ar@/^1pc/_{\sum_{r=1}^i E_2\varsigma_{i-1}\circ
		E_2^iF_1^{i-1}\eps_1\circ \lambda_{(1\cdots r)(2i\cdots i+r)}}(15,3)*{};(110,3)*{},
		\ar_{E_2^iF_1^iE_2f}(-10,-3)*{};(-10,-17)*{},
		\ar^{E_2^iF_1^iE_1f}(8,-3)*{};(8,-17)*{},
		\ar_{E_2f}(110,-3)*{};(110,-17)*{},
		\ar^{E_1f}(128,-3)*{};(128,-17)*{},
		\ar@/_2pc/_(0.25){E_2\tilde{\varsigma}_i\circ \lambda_{(1\cdots 2i+1)}}(-15,-23)*{};(110,-23)*{},
		\ar@/_2pc/_(0.75){E_1\tilde{\varsigma}_i\circ \lambda_{(1\cdots 2i+1)}}(15,-23)*{};(128,-23)*{},
		\ar@/_1pc/^{\sum_{r=1}^i E_2\tilde{\varsigma}_{i-1}\circ
		E_2^iF_1^{i-1}\eps_1\circ \lambda_{(1\cdots r)(2i\cdots i+r)}}(15,-23)*{};(110,-23)*{},
		\endxy}$$
\end{proof}

\begin{lemma}
	\label{le:GammaE}
	We have $E\circ\Gamma=\Gamma\circ E$.
\end{lemma}

\begin{proof}
	Let $(m,\varsigma)\in\Delta_\lambda\CW$. We have $E(m,\pi)=(m',\pi')$ where $m'=\cone(\pi)$ and $\pi'$ is given in
	\S\ref{se:1arrows}. We have $\Gamma\circ E(m,\varsigma)=(m',\pi'')$ where
	
	$$\pi''_{12}=E_1E_2\eps_1\circ E_1\lambda E_1\circ \eta_1 E_2E_1=\sigma,\ \pi''_{21}=0$$

	\begin{align*}
		\pi''_{11}&=E_1E_2\varsigma_1\circ E_1E_2\lambda\circ E_1\lambda E_2\circ E_1F_1\tau_2\circ\eta_1 E_2^2\\
		&=E_1E_2\varsigma_1\circ E_1E_2\lambda\circ E_1\lambda E_2\circ \eta_1 E_2^2\circ\tau_2\\
		&=E_1E_2\varsigma_1\circ E_1E_2\lambda\circ \sigma F_1E_2\circ E_2\eta_1 E_2\circ\tau_2\\
		&=\sigma\circ E_2E_1\varsigma_1\circ E_2E_1\lambda\circ E_2\eta_1 E_2\circ\tau_2\\
		&=\pi'_{11}
	\end{align*}

	\begin{align*}
		\pi''_{22}&=E_1^2\varsigma_1\circ E_1^2\lambda\circ E_1\rho E_2\circ E_1F_1\sigma\circ \eta_1E_2E_1 \\
		&=E_1^2\varsigma_1\circ E_1^2\lambda\circ E_1\rho E_2\circ \eta_1E_1E_2\circ\sigma \\
		&=E_1^2\varsigma_1\circ E_1^2\lambda\circ \tau_1 F_1E_2\circ E_1\eta_1E_2\circ\sigma \\
		&=\tau_1\circ E_1^2\varsigma_1\circ E_1^2\lambda\circ E_1\eta_1E_2\circ\sigma \\
		&=\pi'_{22}
	\end{align*}
	It follows that $\pi''=\pi'$.
\end{proof}

\subsubsection{$2$-arrows}
\label{se:dualdiagonalaction}

We assume in \S\ref{se:dualdiagonalaction} that $\sigma$ is invertible.

Given $(m,\varsigma)\in\Delta_\lambda\CW$, write $E^2(m,\varsigma)=(m'',\varsigma'')$.
The formula (\ref{eq:firstdeftau}) defines an endomorphism $\tau$
of $m''$.

\begin{lemma}
	\label{le:dualtau}
Given $i\ge 1$, we have
	$\tau\circ\varsigma''_i=\varsigma''_i\circ E_2^iF_1^i\tau$.
\end{lemma}

\begin{proof}
	Let $A=\tau\circ\varsigma''_i$ and $B=\varsigma''_i\circ E_2^iF_1^i\tau$.

	We have
	$$a_{21}=a_{22}=a_{31}=a_{32}=a_{33}=a_{34}=a_{41}=a_{42}=a_{43}=0$$

	\begin{align*}
		a_{11}&=\lambda_{(12)}\circ E_2^2\varsigma_i\circ \lambda_{(2\cdots 2i+2)}\circ
	\lambda_{(1\cdots 2i+1)}\\
		&=E_2^2\varsigma_i\circ \lambda_{(1\cdots 2i+2)}\circ
	\lambda_{(1\cdots 2i+1)}
	\end{align*}
	
	\begin{align*}
		a_{12}&=\sum_{r=1}^i \lambda_{(12)}\circ E_2^2\varsigma_{i-1}\circ E_2^{i+1}F_1^{i-1}\eps_1\circ
	\lambda_{(2\cdots r+1)}\circ\lambda_{(2i+1\cdots i+r+1)}\circ \lambda_{(1\cdots 2i+1)}\\
		&=\sum_{r=1}^i E_2^2\varsigma_{i-1}\circ E_2^{i+1}F_1^{i-1}\eps_1\circ\lambda_{(12)}\circ
	\lambda_{(1\cdots i+1)}\circ\lambda_{(1\cdots r)}\circ\lambda_{(2i+1\cdots i+r+1)}\circ \lambda_{(i+1\cdots 2i+1)}\\
		&=0
	\end{align*}

	\begin{align*}
		a_{13}&=\sum_{r=1}^i \lambda_{(12)}\circ E_2^2\varsigma_{i-1}\circ\lambda_{(2\cdots 2i)}\circ 
	E_2^iF_1^{i-1}\eps_1 E_2 \circ \lambda_{(1\cdots r)(2i\cdots i+r)}\\
		&=\sum_{r=1}^i E_2^2\varsigma_{i-1}\circ\lambda_{(1\cdots 2i)}\circ 
	E_2^iF_1^{i-1}\eps_1 E_2 \circ \lambda_{(1\cdots r)(2i\cdots i+r)}\circ
	\lambda_{(2i+1,2i+2)}\circ E_2^iF_1^{i-1}\sigma^{-1}\\
		&=\sum_{r=1}^i E_2^2\varsigma_{i-1}\circ\lambda_{(1\cdots 2i)}\circ 
	E_2^iF_1^{i-1}E_2\eps_1 \circ \lambda_{(2i,2i+1)}\circ\lambda_{(2i\cdots i+r)}\circ
	\lambda_{(1\cdots r)}\circ E_2^iF_1^{i-1}\sigma^{-1}\\
		&=\sum_{r=1}^i E_2^2\varsigma_{i-1}\circ
	E_2^{i+1}F_1^{i-1}\eps_1 \circ \lambda_{(1\cdots 2i+1)}\circ\lambda_{(2i\cdots i+r)}\circ
	\lambda_{(1\cdots r)}\circ E_2^iF_1^{i-1}\sigma^{-1}
	\end{align*}

	\begin{align*}
		a_{14}&=\sum_{\substack{1\le r\le i\\ 1\le s<i}} \lambda_{(12)}\circ E_2^2\varsigma_{i-2}\circ
	E_2^iF_1^{i-2}\eps_1\circ \lambda_{(2\cdots s+1)(2i-1\cdots i+s)}\circ E_2^iF_1^{i-1}\eps_1E_1\circ
	\lambda_{(1\cdots r)(2i\cdots i+r)}\\
		&=\sum_{\substack{1\le r\le i\\ 1\le s<i}} E_2^2\varsigma_{i-2}\circ
	E_2^iF_1^{i-2}\eps_1\circ E_2^iF_1^{i-1}\eps_1E_1\circ\lambda_{(1\cdots s+1)}\circ
	\lambda_{(1\cdots r)}\circ \lambda_{(2i-1\cdots i+s)}\circ \lambda_{(2i\cdots i+r)}\\
		&=\sum_{1\le r\le s<i} E_2^2\varsigma_{i-2}\circ
	E_2^iF_1^{i-2}\eps_1\circ E_2^iF_1^{i-1}\eps_1E_1\circ\lambda_{(2\cdots r+1)}\circ
	\lambda_{(1\cdots s+1)}\circ \lambda_{(2i\cdots i+r)}\circ \lambda_{(2i\cdots i+s+1)}
	\end{align*}

	\begin{align*}
		a_{23}&=\sigma^{-1}\circ  E_1E_2\varsigma_i\circ \lambda_{(2\cdots 2i+2)}\circ
	\lambda_{(1\cdots 2i+1)}\\
		&=\sigma^{-1}\circ  E_1E_2\varsigma_i\circ \lambda_{(2\cdots 2i+2)}\circ
	\lambda_{(1\cdots 2i+1)}\circ\lambda_{(2i+1,2i+2)}\circ E_2^iF_1^i\sigma^{-1}\\
		&=\sigma^{-1}\circ  E_1E_2\varsigma_i\circ \lambda_{(12)}\circ
	\lambda_{(2\cdots 2i+2)}\circ\lambda_{(1\cdots 2i+1)}\circ E_2^iF_1^i\sigma^{-1}\\
		&=E_1E_2\varsigma_i\circ
	\lambda_{(2\cdots 2i+2)}\circ\lambda_{(1\cdots 2i+1)}\circ E_2^iF_1^i\sigma^{-1}
	\end{align*}

	\begin{align*}
		a_{24}&=\sum_{r=1}^i \sigma^{-1}\circ E_1E_2\varsigma_{i-1}\circ E_1E_2^iF_1^{i-1}\eps_1\circ
	\lambda_{(2\cdots r+1)}\circ\lambda_{(2i+1\cdots i+r+1)}\circ \lambda_{(1\cdots 2i+1)}\\
		&=\sum_{r=1}^i E_2E_1\varsigma_{i-1}\circ E_2E_1E_2^{i-1}F_1^{i-1}\eps_1\circ
	\sigma^{-1}E_2^{i-1}F_1^iE_1\circ
	\lambda_{(1\cdots 2i+1)}\circ\lambda_{(1\cdots r)}\circ \lambda_{(2i\cdots i+r)}\\
		&=\sum_{r=1}^i E_2E_1\varsigma_{i-1}\circ E_2E_1E_2^{i-1}F_1^{i-1}\eps_1\circ
	\lambda_{(2\cdots 2i+1)}\circ\lambda_{(1\cdots r)}\circ \lambda_{(2i\cdots i+r)}
	\end{align*}

	\begin{align*}
		a_{44}&=\lambda_{(12)}\circ E_1^2\varsigma_i\circ \lambda_{(2\cdots 2i+2)}\circ
	\lambda_{(1\cdots 2i+1)}\\
		&=E_1^2\varsigma_i\circ \lambda_{(1\cdots 2i+2)}\circ
	\lambda_{(1\cdots 2i+1)}
	\end{align*}

	\smallskip
	We have
	$$b_{21}=b_{12}=b_{22}=b_{31}=b_{32}=b_{33}=b_{41}=b_{42}=b_{43}=0$$

	\begin{align*}
		b_{11}&=E_2^2\varsigma_i\circ \lambda_{(2\cdots 2i+2)}\circ
	\lambda_{(1\cdots 2i+1)}\circ \lambda_{(2i+1,2i+2)}\\
		&=E_2^2\varsigma_i\circ \lambda_{(1\cdots 2i+2)}\circ
	\lambda_{(1\cdots 2i+1)}
	\end{align*}

	\begin{align*}
		b_{13}&=\sum_{r=1}^i E_2^2\varsigma_{i-1}\circ E_2^{i+1}F_1^{i-1}\eps_1\circ
	\lambda_{(2\cdots r+1)}\circ\lambda_{(2i+1\cdots i+r+1)}\circ \lambda_{(1\cdots 2i+1)}\circ
	E_2^iF_1^i\sigma^{-1}\\
		&=\sum_{r=1}^i E_2^2\varsigma_{i-1}\circ E_2^{i+1}F_1^{i-1}\eps_1\circ
	\lambda_{(1\cdots 2i+1)}\circ\lambda_{(2i\cdots i+r)}\circ \lambda_{(1\cdots r)}\circ
	E_2^iF_1^i\sigma^{-1}
	\end{align*}

	\begin{align*}
		b_{14}&=\sum_{\substack{1\le r\le i\\ 1\le s<i}} E_2^2\varsigma_{i-2}\circ
	E_2^iF_1^{i-2}\eps_1\circ \lambda_{(2\cdots s+1)(2i-1\cdots i+s)}\circ E_2^iF_1^{i-1}\eps_1E_1\circ
	\lambda_{(1\cdots r)(2i\cdots i+r)}\circ \lambda_{(2i+1,2i+2)}\\
		&=\sum_{\substack{1\le r\le i\\ 1\le s<i}} E_2^2\varsigma_{i-2}\circ
	E_2^iF_1^{i-2}\eps_1\circ E_2^iF_1^{i-1}\eps_1E_1\circ \lambda_{(2i+1,2i+2)}\circ
	\lambda_{(2\cdots s+1)(2i-1\cdots i+s)}\circ \lambda_{(1\cdots r)(2i\cdots i+r)}\\
		&=\sum_{\substack{1\le r\le i\\ 1\le s<i}} E_2^2\varsigma_{i-2}\circ
	E_2^iF_1^{i-2}\eps_1\circ E_2^iF_1^{i-1}\eps_1E_1\circ \lambda_{(2i-1,2i)}\circ
	\lambda_{(2\cdots s+1)(2i-1\cdots i+s)}\circ \lambda_{(1\cdots r)(2i\cdots i+r)}\\
		&=\sum_{1\le s<r\le i} E_2^2\varsigma_{i-2}\circ
	E_2^iF_1^{i-2}\eps_1\circ E_2^iF_1^{i-1}\eps_1E_1\circ \lambda_{(2\cdots s+1)}\circ
	\lambda_{(1\cdots r)}\circ \lambda_{(2i\cdots i+s)}\circ \lambda_{(2i\cdots i+r)}\\
		&=\sum_{1\le r'\le s'\le i} E_2^2\varsigma_{i-2}\circ
	E_2^iF_1^{i-2}\eps_1\circ E_2^iF_1^{i-1}\eps_1E_1\circ \lambda_{(2\cdots r'+1)}\circ
	\lambda_{(1\cdots s'+1)}\circ \lambda_{(2i\cdots i+r')}\circ \lambda_{(2i\cdots i+s'+1)}
	\end{align*}

	$$b_{23}=E_2E_1\varsigma_i\circ
	\lambda_{(2\cdots 2i+2)}\circ \lambda_{(1\cdots 2i+1)}\circ E_2^iF_1^i\sigma^{-1}$$

	\begin{align*}
		b_{24}&=\sum_{r=1}^iE_2E_1\varsigma_{i-1}\circ\lambda_{(2\cdots 2i)}\circ E_2^iF_1^{i-1}\eps_1E_1
	\circ \lambda_{(1\cdots r)(2i\cdots i+r)}\circ\lambda_{(2i+1,2i+2)}\\
		&=\sum_{r=1}^iE_2E_1\varsigma_{i-1}\circ\lambda_{(2\cdots 2i)}\circ E_2^iF_1^{i-1}E_1\eps_1
	\circ\lambda_{(2i,2i+1)}\circ \lambda_{(1\cdots r)(2i\cdots i+r)}\\
		&=\sum_{r=1}^iE_2E_1\varsigma_{i-1}\circ E_2^iF_1^{i-1}E_1\eps_1
	\circ\lambda_{(2\cdots 2i+1)}\circ \lambda_{(1\cdots r)(2i\cdots i+r)}
	\end{align*}

	\begin{align*}
		b_{34}&=\sum_{r=1}^i E_1E_2\varsigma_{i-1}\circ E_1E_2^iF_1^{i-1}\eps_1\circ
	\lambda_{(2\cdots r+1)}\circ\lambda_{(2i+1\cdots i+r+1)}\circ \lambda_{(1\cdots 2i+1)}
	\circ \lambda_{(2i+1,2i+2)}\\
		&=\sum_{r=1}^i E_1E_2\varsigma_{i-1}\circ E_1E_2^iF_1^{i-1}\eps_1\circ
	\lambda_{(1\cdots 2i+2)}\circ\lambda_{(1\cdots r)}\circ \lambda_{(2i\cdots i+r)}\\
		&=\sum_{r=1}^i E_1E_2\varsigma_{i-1}\circ \lambda_{(1\cdots 2i)}\circ E_2^iF_1^{i-1}E_1\eps_1\circ
	\lambda_{(2i,2i+1)}\circ\lambda_{(2i+1,2i+2)}\circ\lambda_{(1\cdots r)}\circ \lambda_{(2i\cdots i+r)}\\
		&=\sum_{r=1}^i E_1E_2\varsigma_{i-1}\circ \lambda_{(1\cdots 2i)}\circ E_2^iF_1^{i-1}\eps_1E_1\circ
	\lambda_{(2i+1,2i+2)}\circ\lambda_{(2i+1,2i+2)}\circ\lambda_{(1\cdots r)}\circ \lambda_{(2i\cdots i+r)}\\
		&=0
	\end{align*}

	\begin{align*}
		b_{44}&=E_1^2\varsigma_i\circ \lambda_{(2\cdots 2i+2)}\circ
	\lambda_{(1\cdots 2i+1)}\circ \lambda_{(2i+1,2i+2)}\\
		&=E_1^2\varsigma_i\circ \lambda_{(1\cdots 2i+2)}\circ
	\lambda_{(1\cdots 2i+1)}
	\end{align*}

	We deduce that $A=B$ and the lemma follows.
\end{proof}

Lemma \ref{le:dualtau} shows that $\tau$ defines an endomorphism of $E^2(m,\varsigma)$ for
all $(m,\varsigma)\in\Delta_\lambda\CW$.
The functor $\Gamma$ is faithful, $\Gamma E^2=E^2\Gamma$ (Lemma \ref{le:GammaE}) and $\tau$
commutes with $\Gamma$. It follows that $\tau$ is functorial.

\smallskip
Theorem \ref{th:tauisendo} has the following consequence.

\begin{thm}
	\label{th:maindual}
The data $(\Delta_\lambda\CW,E,\tau)$ is an idempotent-complete
strongly pretriangulated $2$-representation.
\end{thm}

The following proposition is a consequence of Lemma \ref{le:GammaE} and the construction of $\tau$.

\begin{prop}
	\label{pr:oldtodual}
The functor $\Gamma:\Delta_\lambda\CW\to\Delta_\sigma\CW$ induces a morphism of
$2$-representations.
\end{prop}

\subsection{Tensor product and internal $\Hom$}

Let us give two applications of the construction of \S\ref{se:diagonal}.
Let $(\CV_1,E_1,\tau_1)$ and $(\CV_2,E_2,\tau_2)$ be 
idempotent-complete strongly pretriangulated $2$-representations.

\smallskip
We view $\CV_1\otimes\CV_2$ as endowed with two strictly commuting 
actions of $\CU$ given by $(E_1\otimes 1,\tau_1\otimes 1)$ and
$(1\otimes E_2,1\otimes\tau_2)$: the isomorphism $\sigma:
(1\otimes E_2)\circ(E_1\otimes 1)\iso
(E_1\otimes 1)\circ (1\otimes E_2)$ is the identity.

We define the {\em tensor product $2$-representation}\index[ter]{tensor product
$2$-representation}
$$\CV_1\dotimes\CV_2=\Delta_\sigma(\CV_1\otimes\CV_2)\indexnot{\dotimes}{\dotimes}.$$

Given $(\Phi_i,\varphi_i):\CV_i\to\CV'_i$ a morphism of $2$-representations for $i\in\{1,2\}$,
Proposition \ref{pr:morphism2rep} provides
a morphism of $2$-representations $\CV_1\dotimes\CV_2\to \CV'_1\dotimes\CV'_2$.

Given $\CV_1$, $\CV_2$ and $\CV_3$ $2$-representations, Proposition \ref{pr:associativity} provides
an isomorphism 
$$(\CV_1\dotimes\CV_2)\dotimes\CV_3\iso \CV_1\dotimes(\CV_2\dotimes\CV_3)$$
that commutes with forgetful functors $\omega$.

\smallskip
Since the forgetful functors $\omega$ are faithful, we deduce that
idempotent-complete strongly pretriangulated $2$-representations form a monoidal $2$-category.


\medskip
Consider now $\Hom(\CV_1,\CV_2)$. It is endowed with two strictly commuting
structures of $2$-representations: the first one is given by
$((\Phi\mapsto \Phi\circ E_1),\Phi\tau_1)$ and the second one by
$((\Phi\mapsto E_2\circ\Phi),\tau_2\Phi)$. The isomorphism $\sigma$ is the identity.

We define the {\em internal $\Hom$ $2$-representation}\index[ter]{internal $\Hom$ $2$-representation}
$$\CHH{om}(\CV_1,\CV_2)=\Delta\Hom(\CV_1,\CV_2)\indexnot{Hom}{\CHH{om}}.$$

\smallskip

The category $\CHH{om}(\CV_1,\CV_2)$ has objects pairs
$(\Phi,\pi)$ where $\Phi:\CV_1\to\overline{\CV_2}^i$ is a differential
functor and $\pi:E_2\Phi\to \Phi E_1$ is a closed
natural transformation of functors such that  
$$\tau_1\Phi\circ \pi E_1\circ E_2\pi=
\pi E_1\circ E_2\pi\circ\tau_2\Phi:E_2^2\Phi\to \Phi E_1^2.$$

Note that $\Hom_\CU(\CV_1,\CV_2)$ is the full subcategory of
$\CHH{om}(\CV_1,\CV_2)$ with objects pairs $(\Phi,\pi)$ where
$\Phi$ takes values in $\CV_2$ and $\pi$ is invertible.

\smallskip
Given $(\Phi_1,\varphi_1):\CV'_1\to\CV_1$ and
$(\Phi_2,\varphi_2):\CV_2\to\CV'_2$ two morphisms of $2$-representations,
Proposition \ref{pr:morphism2rep} provides
a morphism of $2$-representations $\CHH{om}(\CV_1,\CV_2)\to\CHH{om}(\CV'_1,\CV'_2)$.

%
%
%

\section{Bimodule $2$-representations}
\label{se:bimodule2rep}
\subsection{Differential algebras}
\label{se:differentialalgebras}

\subsubsection{$2$-representations}
\label{se:2repalgebra}
Let $A$ be a differential algebra.

\begin{defi}
	\label{de:2repalg}
	A $2$-representation on $A$\index[ter]{$2$-representation on a differential algebra}
	is the data of a differential $(A,A)$-bimodule
	$E$ and of an endomorphism $\tau$ of the $(A,A)$-bimodule $E\otimes_A E$ such that
	$$\tau^2=0,\ d(\tau)=\id \text{ and }
(E\otimes\tau)\circ (\tau\otimes E)\circ (E\otimes\tau)=
(\tau\otimes E)\circ (E\otimes\tau)\circ (\tau\otimes E).$$

	We say that the $2$-representation is {\em right finite} if
 $E$ is finitely generated and projective as a (non-differential) $A^\opp$-module.
\end{defi}

Consider a $2$-representation on $A$.
Note that $E\otimes_A-$ is a differential endofunctor of
$A\mdiff$, and $\tau$ defines an endomorphism of $(E\otimes_A -)^2$. This gives
a structure of $2$-representation on $A\mdiff$. It
restricts to a $2$-representation on $(\bar{A})^i$ if
$E$ is strictly perfect as a differential $A$-module.

\smallskip
Note that there is a morphism of differential algebras
$$H_n\to\End_{A\otimes A^\opp}(E^n),\ T_i\mapsto
E^{n-i-1}\otimes \tau\otimes E^{i-1}.$$

\smallskip
Let $A'$ be another differential algebra with a $2$-representation $(E',\tau')$.
We define a {\em morphism of $2$-representations}\index[ter]{morphism of $2$-representations}
from $(A,E,\tau)$ to $(A',E',\tau')$ to be
an $(A',A)$-bimodule $P$ together with a closed isomorphism of $(A',A)$-bimodules $\varphi:
P\otimes_A E\iso E'\otimes_{A'}P$ such that
\begin{equation}
	\label{eq:taumorphism}
\tau'P\circ E'\varphi\circ \varphi E=E'\varphi\circ \varphi E\circ P\tau:PE^2\to E^{\prime
2}P.
\end{equation}
Note that such a pair $(P,\varphi)$ gives rise to a morphism of $2$-representations
$(P\otimes_A-,\varphi):(A\mdiff,E\otimes_A-,\tau)\to (A'\mdiff,E'\otimes_{A'}-,\tau')$.

We obtain a differential $2$-category of $2$-representations on differential algebras.

\smallskip
The {\em opposite $2$-representation}\index[ter]{opposite $2$-representation}
is the data $(A',E',\tau')$
where $A'=A^\opp$, $E'=E$ and $\tau'=\tau$.
Note that $(A,E,\tau)$ coincides with its double dual.

\smallskip
Assume now the $2$-representation is right finite.
We have two morphisms of $(A,A)$-bimodules $\eta:A\to E\otimes_A
E^\vee$ and $\eps:E^\vee\otimes_A E\to A$ (unit and counit of adjunction).
 We have a morphism of $(A,A)$-bimodules 
$\rho:E^\vee E\to EE^\vee$ defined as the composition
$$\rho:E^\vee E\xrightarrow{\bullet\eta}E^\vee EEE^\vee
\xrightarrow{E^\vee\tau E^\vee}E^\vee EEE^\vee\xrightarrow{\eps\bullet}
EE^\vee.$$

There is a canonical isomorphism of differential algebras
$\End(E^2)^\opp\iso \End((E^\vee)^2)$ and we still denote by $\tau$
the endomorphism of $(E^\vee)^2$ corresponding to $\tau$.

\smallskip
We define the {\em left dual $2$-representation}
\index[ter]{left dual $2$-representation} on $A$ with the bimodule $E^\vee$ and
the endomorphism $\tau$.

\subsection{Lax cocenter}
\label{se:laxcocenterbimod}
Let $B$ be a differential algebra. A {\em lax bi-$2$-representation} on $B$\index[ter]{lax bi-$2$-representation}
is the data of
\begin{itemize}
	\item differential $(B,B)$-bimodules $E_{i,j}$ for $i,j\ge 0$
	\item morphisms of differential algebras $H_i\otimes H_j\to
		\End(E_{i,j})$
	\item morphisms $\mu_{(i,j),(i',j')}:E_{i,j} E_{i',j'}\to E_{i+i',j+j'}$
		satisfying properties (1) and (2) of \S\ref{se:2birep}.
\end{itemize}

Consider a lax bi-$2$-representation $E$.
Note that the functors $(E_{i,j}\otimes_B-)$ provide a structure of lax bi-$2$-representation on 
$B\mdiff$.

\smallskip
We define the differential algebra
$A=\Delta_E(B)$\indexnot{Delta}{\Delta_E(B)} as the
quotient of the tensor algebra $T_B(E_{0,1}E_{1,0})$ by
the two-sided ideal generated by $\bigoplus_{i\ge 0}K_i$, where $K_i$ is the kernel of the composition
$$(E_{0,1}E_{1,0})^i\xrightarrow{\can}E_{i,i}\xrightarrow{\can}
E_{i,i}/((T_r\otimes 1)x-(1\otimes T_r)x)_{x\in E_{i,i},\ 1\le r<i}.$$

We have $A^0=B$ and $A$ is generated by $A^0$ and $A^1=(E_{0,1}E_{1,0})/K_1$ as an algebra.

\medskip
Let $(M,\varsigma)$ be an object of $\Delta_{E\otimes_B-}(B\mdiff)$. The action of
$T_B(E_{0,1}E_{1,0})$ on $M$ vanishes on $K_i$ for all $i$, hence defines an action of $A$ on $M$. This gives a fully faithful differential functor
$\Delta_{E\otimes_B-}(B\mdiff)\to (\Delta_E(B))\mdiff$. If the canonical injective morphism
of differential $(B,B)$-bimodules
\begin{equation}
	\label{eq:diagonalpower}
	(E_{0,1}E_{1,0})^i/K_i\to  E_{i,i}/((T_r\otimes 1)x-(1\otimes T_r)x)_{x\in E_{i,i},\ 
1\le r<i}
\end{equation}
is a split injection for all $i\ge 1$, then the functor above is an isomorphism
$$\Delta_{E\otimes_B-}(B\mdiff)\iso (\Delta_E(B))\mdiff.$$

\subsection{Diagonal action}
\label{se:tensoralgebras}
\subsubsection{Algebra}
\label{se:tensoralgebra1}
Let $B$ be a differential algebra endowed with two $2$-representations
$(F_1,\tau_1)$ and $(E_2,\tau_2)$ together with a closed morphism
$\lambda:F_1E_2\to E_2F_1$ such that
the diagrams (\ref{eq:diaglambda}) commute.

\smallskip
We define the algebra
$A=\Delta'_\lambda(B)$\indexnot{Delta'}{\Delta'_{\lambda}(B)} as the quotient of the tensor
algebra $T_B(F_1E_2)$
by the two-sided ideal generated by the image of the composition
$$F_1^2E_2^2\xrightarrow{\tau_1 E_2^2-F_1^2\tau_2}
F_1^2E_2^2\xrightarrow{F_1\lambda E_2}
(F_1 E_2)^2.$$

We have $A^0=B$ and $A^1=F_1E_2$.

%


\medskip
Let $B'$ be a differential algebra endowed with two $2$-representations
$(F'_1,\tau'_1)$ and $(E'_2,\tau'_2)$ together with a closed morphism
$\lambda':F'_1E'_2\to E'_2F'_1$ such that the analogs of 
the diagrams (\ref{eq:diaglambda}) commute. Let $A'=\Delta'_{\lambda'}(B')$.
Let $P$ be a $(B',B)$-bimodule and
$\varphi_1:PF_1\iso F'_1P$ and $\varphi_2:PE_2\iso E'_2P$ be two closed isomorphisms
of bimodules such that $(P,\varphi_1)$ and $(P,\varphi_2)$ are morphisms of 
$2$-representations and such that
$$\lambda' P\circ F'_1\varphi_2\circ \varphi_1 E_2=E'_2\varphi_1\circ\varphi_2F_1\circ 
P\lambda:PF_1E_2\to E'_2F'_1P.$$

The isomorphism $F'_1\varphi_2\circ \varphi_1 E_2:PF_1E_2\iso F'_1E_2'P$ induces an
isomorphism of $(B',B)$-bimodules $f:P\otimes_B T_B(F_1 E_2)\iso
T_{B'}(F'_1 E'_2)\otimes_{B'}P$. This isomorphism $f$ endows the right $T_B(F_1E_2)$-module
$P\otimes_B T_B(F_1 E_2)$ with a commuting left action of $T_{B'}(F'_1E'_2)$.
The isomorphism $f$ induces an isomorphism 
$$P\otimes_B T_B(F_1 E_2)\otimes_{T_B(F_1E_2)}A\iso
A'\otimes_{T_{B'}(F'_1 E'_2)}T_{B'}(F'_1 E'_2)\otimes_{B'}P.$$
So, we obtain a structure of $(A',A)$-bimodule on $P\otimes_B A$.

\begin{rem}
	The data of $\varphi_1$ and $\varphi_2$ and the relations they are required to
	satisfy are described graphically as:
$$\includegraphics[scale=0.9]{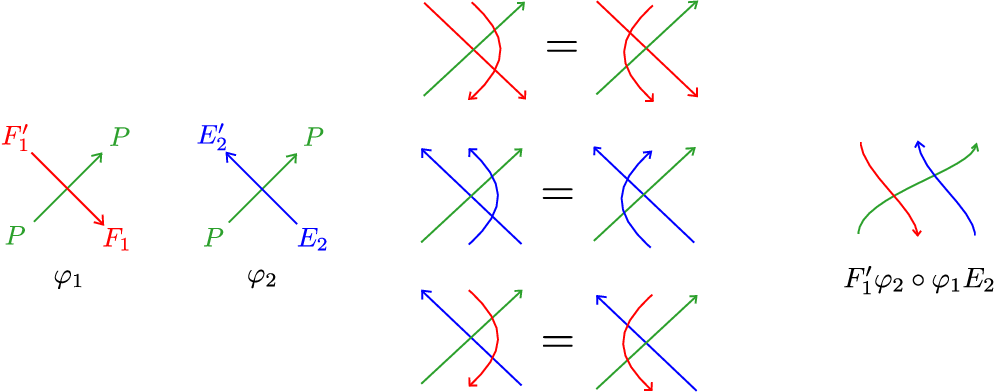}$$
\end{rem}

\subsubsection{Left dual}
\label{se:leftdual}
Let $B$ be a differential algebra endowed with two $2$-representations
$(E_1,\tau_1)$ and $(E_2,\tau_2)$, the first of which is right finite.

We consider the data
of $\sigma\in Z\Hom(E_2E_1,E_1E_2)$ such that the diagrams (\ref{eq:diagsigma})
commute.

We define 
\begin{equation}
	\label{eq:gamma}
\lambda:E_1^\vee E_2\xrightarrow{\bullet\eta_1}E_1^\vee E_2E_1E_1^\vee
\xrightarrow{E_1^\vee\sigma E_1^\vee}E_1^\vee E_1E_2E_1^\vee\xrightarrow{\eps_1\bullet}
E_2E_1^\vee.
\end{equation}

Let $A=\Delta_\sigma(B)=\Delta'_{\lambda}(B)$\indexnot{Deltas}{\Delta_{\sigma}(B)}. This is the graded quotient
of the tensor algebra $T_B(E_1^\vee E_2)$ by
the ideal generated by the image of the composition
$$(E_1^\vee)^2E_2^2\xrightarrow{\tau_1 E_2^2-(E_1^\vee)^2\tau_2}
(E_1^\vee)^2E_2^2\xrightarrow{E_1^\vee\lambda E_2}
(E_1^\vee E_2)^2.$$

The algebra $A$ is generated by $A^0=B$ and $A^1=E_1^\vee E_2$.

\medskip
Let $L$ be a differential $B$-module. The data
of a structure of $A$-module on $L$ extending the action of $B$
is the same as the data of a morphism of
$B$-modules
$\varsigma:E_1^\vee E_2\otimes_BL\to L$
such that $d(\varsigma)=0$ and the following diagram commutes
\begin{equation}
\label{eq:diagbeta}
\xymatrix{
&(E_1^\vee)^2E_2^2 L \ar[rr]^-{E_1^\vee\lambda E_2 }&&
(E_1^\vee E_2)^2 L
\ar[rr]^-{E_1^\vee E_2 \varsigma} && E_1^\vee E_2L
\ar[dr]^{\varsigma} \\
(E_1^\vee)^2E_2^2 L\ar[ur]^{\tau_1 E_2^2} \ar[dr]_{(E_1^\vee)^2\tau_2 } 
&&&&&& L \\
& (E_1^\vee)^2E_2^2 L\ar[rr]_-{E_1^\vee\lambda E_2 }&&
(E_1^\vee E_2)^2 L
\ar[rr]_-{E_1^\vee E_2 \varsigma} && E_1^\vee E_2L
\ar[ur]_{\varsigma}
}
\end{equation}
This gives us an identification (isomorphism of categories)
between differential $A$-modules and pairs consisting of a
differential $B$-module $L$ and a map $\varsigma$ as above.

\smallskip
Consider the adjunction isomorphism
$$\phi:\Hom_B(E_2L,E_1L)\iso
\Hom_B(E_1^\vee E_2L, L)$$
Let $\pi\in Z\Hom_B(E_2L,E_1L)$
and let $\varsigma=\phi(\pi)\in
Z\Hom_B(E_1^\vee E_2L,L)$.
The commutativity of the diagram (\ref{eq:diagbeta}) is equivalent to
the commutativity of the diagram
\begin{equation}
\label{eq:diagalpha}
\xymatrix{
E_2^2L \ar[r]^-{E_2\pi}\ar[d]_{\tau_2} & E_2E_1L\ar[r]^-{\sigma} &
E_1E_2L\ar[r]^-{E_1\pi}& E_1^2L 
\ar[d]^{\tau_1} \\
E_2^2L \ar[r]_-{E_2\pi} & E_2E_1L\ar[r]_-{\sigma}
& E_1E_2L\ar[r]_-{E_1\pi} &E_1^2L 
}
\end{equation}
This gives us an identification (isomorphism of categories)
between differential $A$-modules and pairs $[L,\pi]$ where
$L$ is a differential $B$-module,
$\pi\in Z\Hom_B(E_2L,E_1L)$ 
and the diagram (\ref{eq:diagalpha}) commutes.
We have obtained the following lemma.

\begin{lemma}
	\label{le:diagonalbimodules}
The construction $(m,\pi)\mapsto [m,\pi]$ defines an isomorphism of
differential categories $\Phi:\Delta_\sigma(B\mdiff)\to (\Delta_\sigma B)\mdiff$.
\end{lemma}

We will show that the structure of $2$-representation on
$\Delta_\sigma(B\mdiff)$ comes from a structure of $2$-representation
on $\Delta_\sigma B$, when $\sigma$ is invertible.

\begin{rem}
	The map $\varsigma$, the relations it is required to satisfy, and the relation
	$\varsigma=\phi(\pi)$ 
	are described graphically as:
$$\includegraphics[scale=0.70]{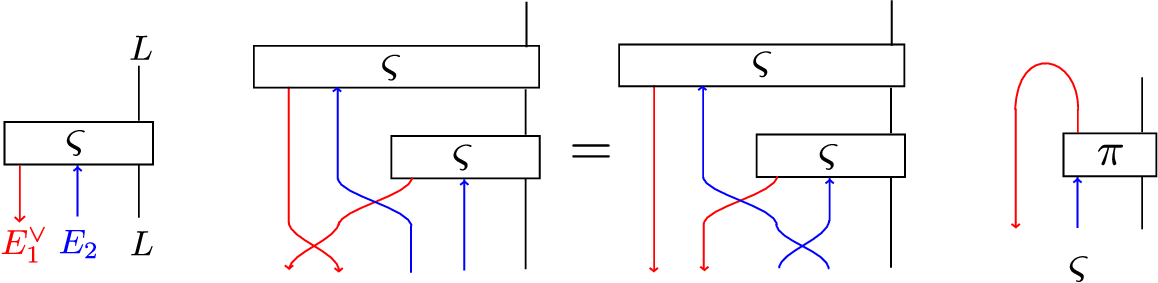}$$
\end{rem}

\subsubsection{Action}
\label{se:action}

We define the closed morphism of $(B,A)$-bimodules
$u:E_2\otimes_{B}A\to E_1\otimes_{B}A$ as the adjoint to
the multiplication map $E_1^\vee E_2\otimes_B A\to A$.
We define $E$ as the cone of $u$. 

We define a morphism of $(B,A)$-bimodules
$v:E_2\otimes_B E\to E_1\otimes_B E$ by
$$v_{11}:E_2^2\otimes_BA\xrightarrow{\tau_2\otimes 1}
E_2^2\otimes_BA\xrightarrow{E_2\eta_1\bullet}
E_2E_1E_1^\vee E_2\otimes_BA\xrightarrow{\sigma\bullet}
E_1E_2E_1^\vee E_2\otimes_BA\xrightarrow{E_1E_2\mathrm{mult.}}
E_1E_2\otimes_BA$$
$$v_{12}:E_2E_1\otimes_BA\xrightarrow{\sigma\otimes 1}
E_1E_2\otimes_BA$$
$$v_{21}=0$$
$$v_{22}:E_2E_1\otimes_BA\xrightarrow{\sigma\otimes 1}
E_1E_2\otimes_BA\xrightarrow{E_1\eta_1\bullet}
E_1^2E_1^\vee E_2\otimes_BA\xrightarrow{\tau_1\bullet}
E_1^2E_1^\vee E_2\otimes_BA\xrightarrow{E_1^2\mathrm{mult.}}
E_1^2\otimes_BA
$$

Note that the morphism $v$ corresponds, by adjunction, to the morphism
$w:E_1^\vee E_2\otimes_B E\to E$ defined as follows
$$w_{11}:E_1^\vee E_2^2\otimes_B A\xrightarrow{E_1^\vee\tau_2\otimes 1}
E_1^\vee E_2^2\otimes_B A\xrightarrow{\lambda E_2\otimes 1}
E_2E_1^\vee E_2\otimes_B A\xrightarrow{E_2\mathrm{mult.}}
E_2\otimes_B A$$
$$w_{12}:E_1^\vee E_2E_1\otimes_BA\xrightarrow{E_1^\vee \sigma\bullet}
E_1^\vee E_1E_2\otimes_BA\xrightarrow{\eps_1\bullet} E_2\otimes_B A,\
w_{21}=0$$
$$w_{22}:E_1^\vee E_2E_1\otimes_BA\xrightarrow{E_1^\vee\sigma\otimes 1}
E_1^\vee E_1E_2\otimes_BA
\xrightarrow{\rho_1\bullet}
E_1E_1^\vee E_2\otimes_BA
\xrightarrow{E_1\mathrm{mult.}}
E_1\otimes_B A.$$

\begin{lemma}
The pair $[E,v]$ gives $E$ a structure of differential $(A,A)$-bimodule via Lemma
\ref{le:diagonalbimodules}. Furthermore,  there is an isomorphism of functors
$\Phi \CE\iso (E\otimes_A -)\Phi:\Delta_\sigma(B\mdiff)\to A\mdiff$.
\end{lemma}

\begin{proof}
The vanishing of $d(v)_{11}$ and $d(v)_{22}$ follows from 
$d(\tau_1)=\id$ and $d(\tau_2)=\id$. The vanishing of $d(v)_{12}$ is
clear. Finally, the vanishing of $d(v)_{21}$ follows from the 
commutativity of the diagram (\ref{eq:diagalpha}).
Since $d(v)=0$, we have obtained a structure of differential
$(T_B(E_1^\vee E_2),A)$-bimodule on $E$.

\smallskip
The object of $\Delta_\sigma(B\mdiff)$ corresponding to $A$ via Lemma \ref{le:diagonalbimodules} is
$(A,u)$.
We have $\CE(A,u)=(E,v)$, where $\CE$ is the endofunctor defining the
$2$-representation on $\Delta_\sigma(B\mdiff)$.
	Since $(E,v)$ is an object
of $\Delta_\sigma(B\mdiff)$, it follows that the action of
$T_B(E_1^\vee E_2)$ on $E$ factors
through an action of $A$. So, $E$ has a structure of differential
$(A,A)$-bimodule and we have an isomorphism of functors
$\Phi \CE\iso (E\otimes_A -)\Phi:\Delta_\sigma(B\mdiff)\to A\mdiff$.
\end{proof}

\begin{rem}
	The maps $v$ and $w$ 
	are described graphically as:
$$\includegraphics[scale=0.85]{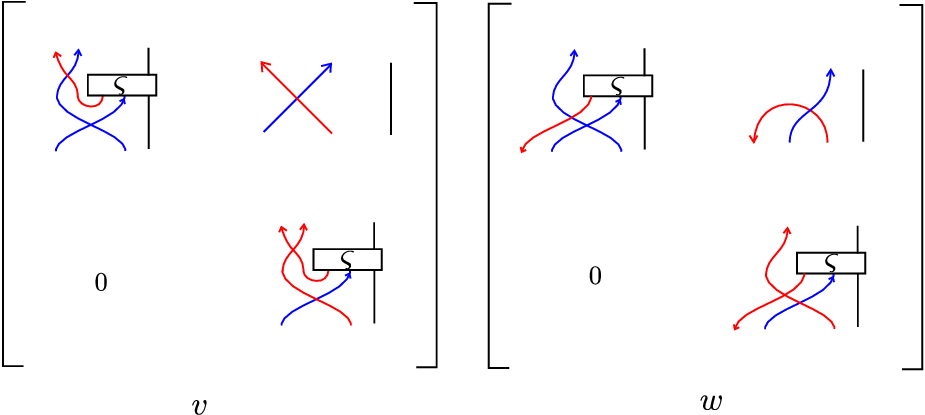}$$
\end{rem}

\smallskip
We assume now that $\sigma$ is invertible.
We define $\tau$ an endomorphism of $(B,A)$-bimodules of 
$E_2^2\otimes_{B}A\oplus E_2E_1\otimes_BA\oplus
E_1E_2\otimes_BA\oplus E_1^2\otimes_BA$ by
\begin{equation}
\label{eq:deftau}
\tau=\left(\begin{matrix}
\tau_2\otimes 1 &0&0&0\\
	0&0&\sigma^{-1}\otimes 1&0\\
0&0&0&0\\
0&0&0&\tau_1\otimes 1
\end{matrix}\right).
\end{equation}

\begin{prop}
The pair $(E,\tau)$ defines
a $2$-representation on $A$ and $\Phi$ induces a isomorphism of
$2$-representations $\Delta_\sigma(B\mdiff)\iso (\Delta_\sigma B)\mdiff$.
	If $E_2$ is right finite, then $E$ is right finite.
\end{prop}

\begin{proof}
The fact that $\tau$ defines an endomorphism of $(A,A)$-bimodules of $E^2$ satisfying
the appropriate relations follows from the fact that it agrees with
the endomorphism of $\CE^2$ defining the $2$-representation on
$\Delta_\sigma(B\mdiff)$. We deduce that
$(E,\tau)$ is a $2$-representation on $A$ and $\Phi$ is a morphism
of $2$-representations.

\smallskip
Note that $E$ is finitely generated and projective as a (non-differential)
	$A^\opp$-module if $E_1$ and $E_2$ are finitely generated and projective
	$B^\opp$-modules.
\end{proof}

\begin{rem}
	Consider three $2$-representations $(E_i,\tau_i)_{1\le i\le 3}$ on a differential algebra $B$
	together with closed morphisms $\sigma_{ij}:E_iE_j\iso E_jE_i$ for $i\neq j$ satisfying
	(\ref{eq:braid}).
	Assume $E_1$ and $E_2$ are right finite.
	We will construct a triple tensor product $2$-representation.

	Define 
	$$\lambda_{ij}:E_i^\vee E_j\xrightarrow{E_i^\vee E_j \eta_i}E_i^\vee E_j E_i E_i^\vee
	\xrightarrow{E_i^\vee \sigma_{ji}E_i^\vee}E_i^\vee E_i E_j E_i^\vee
	\xrightarrow{\eps_i E_jE_i^\vee} E_jE_i^\vee$$
	and denote by $\sigma_{ij}^\vee:E_i^\vee E_j^\vee\to E_j^\vee E_i^\vee$ the map adjoint
	to $\sigma_{ij}$.

	Let $A'=T_B(E_1^\vee E_2\oplus E_2^\vee E_3\oplus E_1^\vee E_3)$. There is a derivation
	$\partial$ of $A'$ whose restriction to $B\oplus E_1^\vee E_2\oplus E_2^\vee E_3$ is $0$
	and whose restriction to $E_1^\vee E_3$ is 
	$$\partial:E_1^\vee E_3\xrightarrow{E_1^\vee \eta_2 E_3}(E_1^\vee E_2)(E_2^\vee E_3).$$

	Define $A''$ to be the
	differential algebra with underlying algebra $A'$ and with differential $\partial+d_{A'}$.

	Let $E$ be the set of quadruples $(i,j,k,l)$ with
	$i,j,k,l\in\{1,2,3\}$, $j-l\ge i-k>0$ and $(i,j,k,l)\neq (2,3,1,2)$.
	Given such a quadruple, we define
	$$f_{ijkl}:E_l^\vee E_k^\vee E_i E_j\xrightarrow{\sigma_{lk}^\vee E_iE_j}E_k^\vee E_l^\vee E_iE_j
	\xrightarrow{E_k\lambda_{li} E_j}(E_k^\vee E_i)(E_l^\vee E_j)$$
	$$g_{ijkl}:E_l^\vee E_k^\vee E_i E_j\xrightarrow{E_l^\vee E_k^\vee \sigma_{ij}}
	E_l^\vee E_k^\vee E_jE_i
	\xrightarrow{E_l\lambda_{kj} E_i}(E_l^\vee E_j)(E_k^\vee E_i)$$
	$$h_{3221}:E_1^\vee E_2^\vee E_3 E_2\xrightarrow{E_1^\vee\lambda_{23}E_2}E_1^\vee
	E_3E_2^\vee E_2\xrightarrow{E_1^\vee E_3\eps_2}E_1^\vee E_3$$
	where we put $\sigma_{rr}=\tau_r$ and $\sigma_{rr}^\vee=\tau_r$.
	We define $I''$ to be the two-sided ideal generated by the images of
	$f_{ijkl}+g_{ijkl}+\delta_{jk}h_{3221}$ for $(i,j,k,l)\in E$.
	We put $A=A''/I''$.

	As in \S \ref{se:leftdual}, we have an isomorphism of differential categories
	$\Delta_{123}(B\mdiff)\iso A\mdiff$ (cf \S \ref{se:associativity}).

	\smallskip
	We obtain a bimodule $2$-representation on $A$ as in \S \ref{se:associativity}.
	We define the differential $(B,A)$-bimodule
	$$E={\xy (0,0)*{E_3 \otimes_B A\oplus E_2\otimes_B A\oplus E_1\otimes_B A},
\ar@/^2pc/^{\pi_{31}}(-25,3)*{};(25,3)*{},
\ar@/^0.5pc/^{\ \pi_{32}}(-24,3)*{};(2,3)*{},
\ar@/^0.5pc/^{\pi_{21}\ \ }(3,3)*{};(24,3)*{},
\endxy}$$
where
	$$\pi_{ij}:E_i\otimes_B A\xrightarrow{\eta_j \id} E_jE_j^\vee E_i\otimes_B A
	\xrightarrow{E_j\mathrm{ mult}}E_j\otimes_B A.$$
We extend the left action of $B$ to an action of $A$ by letting the action maps
$$E_i^\vee E_j\otimes_B E
	\to E$$
	for $i<j$ be given by
$${\xy (0,10)*{E_1^\vee E_3^2\otimes_B A\oplus E_1^\vee E_3E_2\otimes_B A\oplus
	E_1^\vee E_3E_1\otimes_B A},
	(0,-10)*{E_3 \otimes_B A\oplus E_2\otimes_B A\oplus E_1\otimes_B A},
	\ar_{E_3\text{ mult}\circ \lambda_{13}E_3\circ E_1^\vee\tau_3}(-27,7)*{};(-22,-7)*{},
	\ar^{E_1\mathrm{ mult}\circ \rho_1 E_3\circ E_1^\vee\sigma_{31}}(29,7)*{};(20,-7)*{},
	\ar_{E_3\eps_1\circ \lambda_{13}E_1\ \ }(27,7)*{};(-20,-7)*{},
	\endxy}$$
$${\xy (0,10)*{E_2^\vee E_3^2\otimes_B A\oplus E_2^\vee E_3E_2\otimes_B A\oplus
	E_2^\vee E_3E_1\otimes_B A},
	(0,-10)*{E_3 \otimes_B A\oplus E_2\otimes_B A\oplus E_1\otimes_B A},
	\ar_{E_3\text{ mult}\circ \lambda_{23}E_3\circ E_2^\vee\tau_3}(-27,7)*{};(-22,-7)*{},
	\ar@/_1pc/^{\!\!E_2\mathrm{ mult}\circ \rho_2 E_3\circ E_2^\vee\sigma_{32}}(1,7)*{};(0,-7)*{},
	\ar_(0.4){E_3\eps_2\circ \lambda_{23}E_2\!}(-1,7)*{};(-20,-7)*{},
	\ar@/^1pc/^{E_1\mathrm{ mult}\circ \lambda_{21}E_3\circ E_2^\vee\sigma_{31}}(25,7)*{};(20,-7)*{},
	\endxy}$$
	$${\xy (0,10)*{E_1^\vee E_2E_3\otimes_B A\oplus E_1^\vee E_2^2\otimes_B A\oplus
	E_1^\vee E_2E_1\otimes_B A},
	(0,-10)*{E_3 \otimes_B A\oplus E_2\otimes_B A\oplus E_1\otimes_B A},
	\ar@/_1pc/_{E_3\mathrm{ mult}\circ \lambda_{13}E_2\circ E_1^\vee\sigma_{23}}(-30,7)*{};(-22,-7)*{},
	\ar_(0.3){E_2\text{ mult}\circ \lambda_{12}E_2\circ E_1^\vee\tau_2\!\!}(1,7)*{};(0,-7)*{},
	\ar@/^0.5pc/_(0.3){E_2\eps_1\circ \lambda_{12}E_1\!}(27,7)*{};(2,-7)*{},
	\ar@/^1pc/^{E_1\mathrm{ mult}\circ \rho_1 E_2\circ E_1^\vee\sigma_{21}}(29,7)*{};(20,-7)*{},
	\endxy}$$
	Finally, we define the endomorphism $\tau$ of $E^2$ as in (\ref{eq:tau123}).
\end{rem}

\subsubsection{Tensor product case}
\label{se:tensoralg}

Let $A_1$ and $A_2$ be two differential algebras equipped with
structures of $2$-representations $(E_i,\tau_i)$, $i=1,2$. 

Let $B=A_1\otimes A_2$. It is endowed with commuting
$2$-representations $(E_1\otimes A_2, \tau_1\otimes 1)$ and
$(A_1\otimes E_2,1\otimes\tau_2)$: the isomorphism $\sigma$ is
induced by the swap map $E_2\otimes E_1\xrightarrow{\sim}E_1\otimes E_2,\
a_2\otimes a_1\mapsto a_1\otimes a_2$.
The tensor product identifies $(A_1\mdiff)\otimes (A_2\mdiff)$ with a full
subcategory of $B\mdiff$.

\smallskip
Assume $E_1$ is right finite.  The map $\lambda$ is an isomorphism.
We put $A_1\dotimes A_2=\Delta'_\lambda(B)$. It is
the quotient
of the tensor algebra $T_{A_1\otimes A_2}(E_1^\vee\otimes E_2)$ by
the ideal generated by $p\tau_2(q)-\tau_1(p)q$ for $p\in (E_1^\vee)^{\otimes 2}$
and $q\in (E_2)^{\otimes 2}$.
The underlying differential module is
$$A=\bigoplus_{i\ge 0}(E_1^\vee)^i\otimes_{H_i}E_2^i.$$
The multiplication is defined by
$$\bigl((E_1^i)^\vee\otimes_{H_i}E_2^i\bigr)\otimes
\bigl((E_1^j)^\vee\otimes_{H_j}E_2^j\bigr)\to
(E_1^{i+j})^\vee\otimes_{H_{i+j}}E_2^{i+j},\
(a_1\otimes a_2)\otimes (b_1\otimes b_2)\mapsto (a_1b_1)\otimes (a_2b_2).$$

We have 
$$E={\xy (0,0)*{\bigl(\bigoplus_{i\ge 0}(E_1^\vee)^i\otimes_{H_i}E_2E_2^i \bigr)\oplus
\bigl(\bigoplus_{i\ge 0}E_1(E_1^\vee)^i\otimes_{H_i}E_2^i\bigr)},
	\ar@/^/^{\eta_1\otimes 1}(-28,4)*{};(24,4)*{}, \endxy}.$$
The right action of $A$ on $E$ is given by right multiplication, while the left action of
$E_1^\vee\otimes E_2$ on $A_1\otimes E_2\oplus E_1\otimes A_2\subset E$ is given by
$$(E_1^\vee\otimes E_2)\otimes_{A_1\otimes A_2}(A_1\otimes E_2)\xrightarrow[\sim]{\can}
E_1^\vee\otimes E_2^2\xrightarrow{1\otimes\tau_2} E_1^\vee\otimes E_2^2$$
$$(E_1^\vee\otimes E_2)\otimes_{A_1\otimes A_2}(E_1\otimes A_2)\xrightarrow[\sim]{\can}
E_1^\vee E_1\otimes E_2\xrightarrow{(\eps_1,\rho_1)}A_1\otimes E_2\oplus E_1E_1^\vee\otimes E_2.$$

We have 
$$E^2={\xy (0,0)*{\bigl(\bigoplus(E_1^\vee)^i\otimes_{H_i}E_2^2E_2^i\bigr) \oplus
\bigl(\bigoplus E_1(E_1^\vee)^i\otimes_{H_i}E_2E_2^i\bigr) \oplus
\bigl(\bigoplus E_1(E_1^\vee)^i\otimes_{H_i}E_2E_2^i\bigr) \oplus
\bigl(\bigoplus E_1^2(E_1^\vee)^i\otimes_{H_i}E_2^i\bigr)},
	\ar@/^/^(0.7){\eta_1(E_1^\vee)^i\otimes E_2^{2+i}}(-64,4)*{};(-24,4)*{},
	\ar@/^2pc/^{\eta_1(E_1^\vee)^i\otimes \tau_2E_2^i}(-67,4)*{};(20,4)*{},
	\ar@/^/^(0.3){E_1\eta_1(E_1^\vee)^i\otimes E_2^{1+i}}(24,3)*{};(64,3)*{},
	\ar@/^2pc/^{(E_1\rho_1\circ\eta_1E_1)(E_1^\vee)^i\otimes E_2^{1+i}}(-14,4)*{};(64,4)*{},
	\ar@/^/^{1}(-14,3)*{};(20,3)*{},
	\endxy}.$$

The endomorphism $\tau$ of $E^2$ is given on 
$$\bigl((E_1^\vee)^i\otimes_{H_i}E_2^2E_2^i\bigr) \oplus
\bigl(E_1(E_1^\vee)^i\otimes_{H_i}E_2E_2^i\bigr) \oplus
 \bigl(E_1(E_1^\vee)^i\otimes_{H_i}E_2E_2^i\bigr) \oplus
 \bigl(E_1^2(E_1^\vee)^i\otimes_{H_i}E_2^i\bigr)$$
 by
$$\tau=\left(\begin{matrix}
1\otimes \tau_2 E_2^i &0&0&0\\
	0&0& 1&0\\
0&0&0&0\\
	0&0&0&\tau_1 (E_1^\vee)^i\otimes 1
\end{matrix}\right).$$

\medskip
This construction provides the differential $2$-category of right finite $2$-representations
on differential algebras with a monoidal structure.

%

\subsection{Dual diagonal action}
\label{se:dualbimodule}

\subsubsection{Algebra}
Let $B$ be a differential algebra endowed with two $2$-representations
$(F_1,\tau_1)$ and $(E_2,\tau_2)$ together with a closed morphism
$\lambda:F_1E_2\to E_2F_1$ such that
the diagrams (\ref{eq:diaglambda}) commute.

\smallskip

We define the differential algebra

$$A=\Delta_{\lambda}(B)=\bigoplus_{i\ge 0}(E_2^iF_1^i)/
((T_r\otimes 1)x-(1\otimes T_r)x)_{x\in E_2^iF_1^i,\ 
1\le r<i}\indexnot{Delta}{\Delta_{\lambda}(B)}.$$
Its multiplication is given by the maps
$\mu_{i,j}=
\mu_{(i,i),(j,j)}E_2^iF_1^iE_2^jF_1^j\to E_2^{i+j}F_1^{i+j}$ defined in \S\ref{se:2birep}.

\smallskip
Given $M$ a differential $A$-module and given $i\ge 1$,
we have differential $B$-module maps $\varsigma_i:E_2^iF_1^i\otimes_BM\to M$. These
make $(M,(\varsigma_i)_i)$ into an object of $\Delta_\lambda(B\mdiff)$ and
provides an isomorphism of differential categories
$\Delta_{\lambda}(B)\mdiff\iso \Delta_\lambda(B\mdiff)$.

\begin{rem}
	\label{re:lambdatobirep}
As in Remark \ref{re:bi2repswap}, we obtain a lax bi-$2$-representation on $B$ by setting
$E_{i,j}=E_2^jF_1^i$.
We have an injective morphism of differential algebras
	$\Delta_E(B)\to \Delta_\lambda(B)$.

	\smallskip
	Assume the morphisms (\ref{eq:diagonalpower}) are isomorphisms for all $i$ (this
	holds for example if $\lambda$ is an isomorphism).
	Then we have a canonical isomorphism $\Delta_E(B)\iso \Delta_\lambda(B)$.
	The algebra $\Delta_{\lambda}(B)$ is generated by $B$ and $E_2F_1$.
\end{rem}

The map $\lambda$ extends (uniquely)
to a morphism of algebras $\Delta'_\lambda(B)\to\Delta_\lambda(B)$ that is the identity
on $B$.
If $\lambda$ is an isomorphism, then this map is an isomorphism
$\Delta'_\lambda(B)\iso\Delta_\lambda(B)$.  

\subsubsection{Left dual}
\label{se:leftdualbimod}
We assume now that $F_1$ is left finite and we put $E_1={^\vee F}_1$. Consider
$\sigma\in Z\Hom(E_2E_1,E_1E_2)$ defined as in (\ref{eq:defsigma}).

Let $\pi:E_2\otimes_BA\to E_1\otimes_B A$ be the closed morphism of
$(B,A)$-bimodules given as a composition
$$\pi:E_2\otimes_BA\xrightarrow{E_2\eta_1}E_2E_1F_1\otimes_B A
\xrightarrow{\sigma F_1} E_1E_2E_1^\vee\otimes_B A\xrightarrow{E_1
\mathrm{mult}}E_1\otimes_B A.$$
We put $E=\cone(\pi)$. Given $i\ge 1$, we define a morphism of
$(B,B)$-bimodules
$\varsigma_i:E_2^iF_1^iE\to E$

	$$\varsigma_i=\left(\begin{matrix}
		E_2\mathrm{mult}\circ \lambda_{(1\cdots 2i+1)} & 
		\sum_{r=1}^{i} E_2\mathrm{mult}\circ
		E_2^iF_1^{i-1}\eps_1\circ \lambda_{(1\cdots r)(2i\cdots i+r)}\\
		0 & E_1\mathrm{mult}\circ \lambda_{(1\cdots 2i+1)}
	\end{matrix}\right)$$

The following lemma is a consequence of Lemmas \ref{le:dualisobject} and
\ref{le:dualEfunctor} applied to $m=A$.

	\begin{lemma}
		The $\varsigma_i$'s define a left action of $A$ on $E$, giving $E$
		a structure of differential $(A,A)$-bimodule.
	\end{lemma}

	Note that the isomorphism of differential categories
	$\Delta_\lambda(B)\mdiff\iso\Delta_\lambda(B\mdiff)$ commutes with $E$.
	
	\medskip
	Assume now $\sigma$ is an isomorphism. We define $\tau$ a $(B,A)$-bimodule
	endomorphism of $E^2$ as in (\ref{eq:deftau}).

	Theorem \ref{th:maindual} has the following consequence.

	\begin{thm}
		The data $(E,\tau)$ defines a $2$-representation on $\Delta_\lambda(B)$.
\end{thm}

	Note that we have an isomorphism of $2$-representations 
	$\Delta_\lambda(B)\mdiff\iso\Delta_\lambda(B\mdiff)$.

	\smallskip
	Consider the $(\Delta_\sigma(B),\Delta_\lambda(B))$-bimodule $\Delta_\lambda(B)$, where the right
	action is given by multiplication and the left action by multiplication preceded by the
morphism of algebras $\Delta_\sigma(B)=\Delta'_\lambda(B)\to\Delta_\lambda(B)$. 
It follows from Proposition \ref{pr:oldtodual} that this bimodule  induces a morphism of
$2$-representations from $\Delta_\lambda(B)$ to $\Delta_\sigma(B)$.

\subsection{Differential categories}

\subsubsection{Bimodule $2$-representations}
\label{se:bimod2rep}
All the definitions and constructions of \S\ref{se:differentialalgebras}--\ref{se:dualbimodule} extend from
the setting of differential algebras to that of differential categories. We will
describe this explicitly.

\medskip
We view the monoidal category $\CU$ as a $2$-category with one object $\ast$.

\begin{defi}
	A {\em bimodule $2$-representation}\index[ter]{bimodule $2$-representation}
	is the data of a $2$-functor
$\Upsilon:\CU\to\mathrm{Bimod}$.
	
It is {\em right finite} if $\Upsilon(e)$ is right finite.
\end{defi}

We say that $\Upsilon$ is a bimodule $2$-representation on $\Upsilon(\ast)$.

\smallskip
Bimodule $2$-representations form a differential $2$-category.

\medskip
Let $\CC$ be a differential category.
There are equivalences of differential $2$-categories between

\begin{itemize}
	\item the $2$-category of bimodule $2$-representations $\Upsilon$ on $\CC$
	\item the $2$-category with objects differential functors $M:\CC\times\CC^\opp\times
		\CU\to k\mdiff$ together with 
		\begin{itemize}
			\item isomorphisms $\mu_{m,n}:M(c,-,e^m)\otimes_{\CC}M(-,c',e^n)\iso M(c,c',e^{n+m})$
		functorial in $c$ and $c'$, compatible with the canonical morphism
		$\End(e^m)\otimes\End(e^n)\to\End(e^{n+m})$ and satisfying
	$\mu_{l,n+m}\circ(\id\otimes \mu_{m,n})=\mu_{m+l,n}\circ(\mu_{l,m}\otimes\id)$
\item an isomorphism $\mu_0:M(-,-,e^0)\iso\Id$ such that $\mu_{m,0}=\mathrm{mult}\circ
	(M(c,-,e^m)\otimes\mu_0)$ and $\mu_{0,m}=\mathrm{mult}\circ (\mu_0\otimes M(-,c,e^m))$
		\end{itemize}
	\item the $2$-category of pairs $(E,\tau)$ where $E$ is a $(\CC,\CC)$-bimodule and 
		$\tau\in\End(E^2)$ satisfies (\ref{eq:definingtau}).
\end{itemize}

The category $\CH{om}((\CC,E,\tau),(\CC',E',\tau'))$ of $1$-arrows in the third $2$-category
above has objects pairs $(P,\varphi)$ where $P$ is a $(\CC',\CC)$-bimodule and
$\varphi:P\otimes_\CC E\iso E'\otimes_{\CC'}P'$ is a closed isomorphism of
$(\CC',\CC)$-bimodules satisfying (\ref{eq:taumorphism}).
We leave it to the reader to describe $1$-arrows in the second $2$-category above.
In these $2$-categories, the $2$-arrows are morphisms of (non-differential) bimodules or
functors compatible with the additional structure.

\medskip
The equivalences are given by
$$\Upsilon\mapsto (M:(c_1,c_2,e^n)\mapsto \Upsilon(e^n)(c_1,c_2)),\
M\mapsto (E=M(-,-,e),\tau=M(-,-,\tau))$$
$$E\mapsto (\Upsilon:e^n\mapsto E^n).$$
We will use the terminology ``bimodule $2$-representation" for either one of those three
equivalent structures.

\smallskip
Note that a $2$-representation $\Upsilon:\CU\to\End(\CC)$ gives rise to a
bimodule $2$-representation $M$ on $\CC$ given by $M(c_1,c_2,e^n)=\Hom_\CC(c_2,
\Upsilon\circ\mathrm{rev}(e^n)(c_1))$ (cf \S\ref{se:bimodfunctors}).
Note also that a bimodule $2$-representation $M$ on a differential category $\CC$
gives rise to a $2$-representation $\Upsilon:\CU\to\End(\CC\mdiff)$ given by
$\Upsilon(e^n)=M(-,-,e^n)\otimes_\CC -$.

\subsubsection{Diagonal action}
A {\em bimodule lax bi-$2$-representation}\index[ter]{bimodule lax bi-$2$-representation} is a lax differential $2$-functor
$\Upsilon:\CU\otimes\CU\to\mathrm{Bimod}$. We say it is a bimodule lax bi-$2$-representation on $\Upsilon(
\ast\otimes\ast)$.

A bimodule lax bi-$2$-representation on $\CC$ is the same as the data of
\begin{itemize}
	\item $(\CC,\CC)$-bimodules $E_{i,j}$ for $i,j\ge 0$
	\item morphisms of differential algebras $H_i\otimes H_j\to
		\End(E_{i,j})$
	\item morphisms $\mu_{(i,j),(i',j')}:E_{i,j} E_{i',j'}\to E_{i+i',j+j'}$
		satisfying properties (1) and (2) of \S\ref{se:2birep}.
\end{itemize}

We define 
the differential category 
$\Delta_E(\CC)$\indexnot{Delta}{\Delta_E(\CC)}
as the additive category quotient of $T_{\CC}(E_{0,1}E_{1,0})$ by the ideal of maps
generated by the kernels of the compositions
$$(E_{0,1}E_{1,0})^i(c_1,c_2)\xrightarrow{\can}E_{i,i}(c_1,c_2)\xrightarrow{\can} E_{i,i}(c_1,c_2)
/((T_r\otimes 1)x-(1\otimes T_r)x)_{x\in E_{i,i},\ 1\le r<i}.$$

\medskip

Assume now $\CC$ is a differential category endowed with two structures
$(F_1,\tau_1)$ and $(E_2,\tau_2)$ of bimodule $2$-representations
together with a closed morphism
$\lambda:F_1E_2\to E_2F_1$ such that
the diagrams (\ref{eq:diaglambda}) commute.

We define the differential category 
$\Delta'_\lambda(\CC)$\indexnot{Delta'}{\Delta'_\lambda(\CC)}
as the additive category quotient of $T_\CC(F_1E_2)$ by the ideal of maps
generated by the image of the composition
$$F_1^2E_2^2(c_1,c_2)\xrightarrow{\tau_1 E_2^2-F_1^2\tau_2}
F_1^2E_2^2(c_1,c_2)\xrightarrow{F_1\lambda E_2}
(F_1 E_2)^2(c_1,c_2).$$

\smallskip
We have a differential category  $\CC'=\bigoplus_{i\ge 0}E_2^iF_1^i$. Its objects are those of $\CC$ and
$\Hom_{\CC'}(c_1,c_2)=\bigoplus_{i\ge 0}E_2^iF_1^i(c_1,c_2)$. The multiplication is induced by
the maps $\mu_{i,j}$.
We define the differential category 
$\Delta_\lambda(\CC)$\indexnot{Delta}{\Delta_\lambda(\CC)}
as the additive category quotient of $\bigoplus_{i\ge 0}E_2^iF_1^i$ by the ideal of maps
generated by the images of $T_r\otimes 1-1\otimes T_r:E_2^iF_1^i\to E_2^iF_1^i$ for $1\le r<i$.

\medskip
Assume now $\CC$ is a differential category endowed with two structures
$(E_1,\tau_1)$ and $(E_2,\tau_2)$ of bimodule $2$-representations, the first of which
is right finite. Consider $\sigma:E_2E_1\to E_1E_2$ closed such that the diagrams
(\ref{eq:diagsigma}) commute.
We define $\lambda:E_1^\vee E_2\to E_2E_1^\vee$ as in (\ref{eq:gamma}).

\smallskip
$\bullet\ $We put $\Delta_\sigma(\CC)=\Delta'_\lambda(\CC)$.
As in \S \ref{se:action}, we define a $(\Delta_\sigma\CC,\CC)$-bimodule $E$ and extend it to a
$(\Delta_\sigma\CC,\Delta_\sigma\CC)$-bimodule.
Assume finally that $\sigma$ is invertible. We construct in addition
an endomorphism $\tau$ of $E^2$. We obtain
a bimodule $2$-representation on $\Delta_\sigma\CC$ and an isomorphism of $2$-representations
$\Delta_\sigma(\CC\mdiff)\iso \Delta_\sigma(\CC)\mdiff$. The $2$-representation is right finite
if $E_2$ is right finite.

\smallskip
As in \S\ref{se:tensoralg}, we have a monoidal structure on the differential $2$-category
of right finite bimodule $2$-representations.

\medskip
$\bullet\ $We drop now the assumption that $\sigma$ is invertible. We define as in \S \ref{se:leftdualbimod} a
$(\Delta_\lambda\CC,\Delta_\lambda\CC)$-bimodule $E$. Assume $\sigma$ is invertible. We obtain
an endomorphism $\tau$ of $E^2$ and a bimodule $2$-representation on $\Delta_\lambda\CC$.

\subsection{Pointed categories}
\label{se:2reppointed}

Let $\CV$ be a differential pointed category.
A {\em bimodule $2$-represesentation} on $\CV$\index[ter]{bimodule $2$-representation} is the data of a strict monoidal
differential pointed functor from the $2$-category with one object given by $\CU^\bullet$
to $\mathrm{Bimod}^\bullet$.
Note that a bimodule $2$-representation on $\CV$ gives rise to a bimodule $2$-representation
on $k[\CV]$.

\medskip
A {\em bimodule lax bi-$2$-representation}\index[ter]{bimodule lax bi-$2$-representation} is a lax differential pointed $2$-functor
$\Upsilon:\CU^\bullet\wedge\CU^\bullet\to\mathrm{Bimod}^\bullet$. We say it is a bimodule lax bi-$2$-representation on $\Upsilon(
\ast\wedge\ast)$.

A bimodule lax bi-$2$-representation on $\CV$ is the same as the data of
\begin{itemize}
	\item $(\CV,\CV)$-bimodules $E_{i,j}$ for $i,j\ge 0$
	\item morphisms of differential pointed algebras $H_i\wedge H_j\to
		\End(E_{i,j})$
	\item morphisms $\mu_{(i,j),(i',j')}:E_{i,j} E_{i',j'}\to E_{i+i',j+j'}$
		satisfying properties (1) and (2) of \S\ref{se:2birep}.
\end{itemize}

We define 
the differential pointed category 
$\Delta_E(\CV)$\indexnot{Delta}{\Delta_E(\CV)}
as the quotient of $T_{\CV}(E_{0,1}E_{1,0})$ by the equivalence relation generated by 
$f\sim f'$ if $(f,f')$ is in the equalizer of a
composition
$$(E_{0,1}E_{1,0})^i(c_1,c_2)\xrightarrow{\can}E_{i,i}(c_1,c_2)\xrightarrow{\can} E_{i,i}(c_1,c_2)
/((T_r\wedge 1)x\sim (1\wedge T_r)x)_{x\in E_{i,i},\ 1\le r<i}.$$

\medskip
Consider a differential pointed category $\CV$ endowed with two bimodule $2$-representations
$(F_1,\tau_1)$ and $(E_2,\tau_2)$ and a closed morphism $\lambda:F_1E_2\to E_2F_1$
such that the diagrams (\ref{eq:diaglambda}) commute.

We define 
the differential pointed category 
$\Delta'_\lambda(\CV)$\indexnot{Delta'}{\Delta'_\lambda(\CV)}
as the quotient of $T_{\CV}(F_1E_2)$ by the equivalence relation
generated by 
$$(F_1\lambda E_2)\circ (\tau_1 E_2^2)(f)\sim
(F_1\lambda E_2)\circ (F_1^2\tau_2)(f)
\text{ for }f\in F_1^2E_2^2(c_1,c_2) \text{ and }c_1,c_2\in\CV.$$

\smallskip
We define the differential pointed category 
$\Delta_\lambda(\CV)$\indexnot{Delta}{\Delta_\lambda(\CV)}. We consider first the differential pointed
category with same objects as $\CV$ and pointed set of maps $v_1\to v_2$ given by
$\bigvee_{i\ge 0}E_2^iF_1^i(v_1,v_2)$. The category $\Delta_\lambda(\CV)$ is the quotient of that
category by the equivalence relation generated by $(T_r\wedge 1)(f)\sim (1\wedge T_r)(f)$ for
$f\in E_2^iF_1^i$ and $1\le r<i$.

\smallskip
Note that there is a canonical isomorphism of differential categories for $?\in\{\emptyset,\prime\}$
$$k[\Delta^?_\lambda(\CV)]\iso \Delta^?_\lambda(k[\CV])$$

%

\subsection{Douglas-Manolescu's algebra-modules}
\label{se:algebramodules}

Let us recall some aspects of Douglas-Manolescu's theory \cite{DouMa}.

Note
that Douglas and Manolescu work in the differential graded setting, and we
translate their constructions to the differential setting.

\smallskip

Their nil-Coxeter $2$-algebra \cite[\S 2.2]{DouMa}
can be viewed as the same data as our monoidal
category $\CU$ (cf \cite[Remark 2.4]{DouMa}).
A bottom-algebra module \cite[\S 2.4]{DouMa} for the nil-Coxeter $2$-algebra
is the same data as a lax bimodule $2$-representation on a differential algebra $A$, where
a lax bimodule $2$-representation on $A$ is defined to be a lax $2$-functor 
$\Upsilon:\CU\to\mathrm{Bimod}$ with $\Upsilon(1)$ the differential category with
one object whose endomorphism ring is $A$. They also consider top-algebra modules,
where $\CU$ above is replaced by $\CU^\opp$. Using the isomorphism
$\CU\iso\CU^\opp$ (\S \ref{se:defmoncat}), a top-algebra module can be viewed as a bottom-algebra
module, hence as a lax bimodule $2$-representation.

Douglas and Manolescu define a tensor product of a top algebra-module
and a bottom algebra-module \cite[Definition 2.11]{DouMa}. This corresponds to
our construction of
a differential algebra $A$ as a tensor product $\dotimes$. Note that they do not
endow this tensor product with any algebra-module structure.

\section{Hecke $2$-representations}
\label{se:Hecke2rep}
\subsection{Regular $2$-representations}
\label{se:regular2rep}
\subsubsection{Bimodules}
\label{se:Lrn}
Fix $r,n\ge 0$. We define some bimodules $L^\pm(r,n)$ and $R^\pm(r,n)$ with
underlying differential graded module $H_{r+n}$, following \S \ref{se:traces} and
Proposition \ref{pr:tracediff}.

\smallskip
We endow $L^+(r,n)$ (resp. $L^-(r,n)$) with a structure of differential graded
$(H_r\otimes H_n,H_{r+n})$-bimodule where
\begin{itemize}
	\item $H_{r+n}$ acts by right multiplication
	\item $h\in H_r$ acts by left multiplication by $h$ (resp. by $f_n(h)$)
	\item $h\in H_n$ acts by left multiplication by $f_r\circ\iota_n(h)$ (resp. by $h$).
\end{itemize}

\smallskip

We endow $R^+(r,n)$ (resp. $R^-(r,n)$) with a structure of differential graded
$(H_{r+n},H_r\otimes H_n)$-bimodule where
\begin{itemize}
	\item $H_{r+n}$ acts by left multiplication
	\item $h\in H_r$ acts by right multiplication by $h$ (resp. by $f_n(h)$)
	\item $h\in H_n$ acts by right multiplication by $f_r\circ\iota_n(h)$ (resp. by $h$).
\end{itemize}

\begin{example}
	Elements of $L^\pm(r,n)$ and $R^\pm(r,n)$ can be represented by good strand 
	diagrams in a rectangle, as in the examples below.
	$$\includegraphics[scale=0.55]{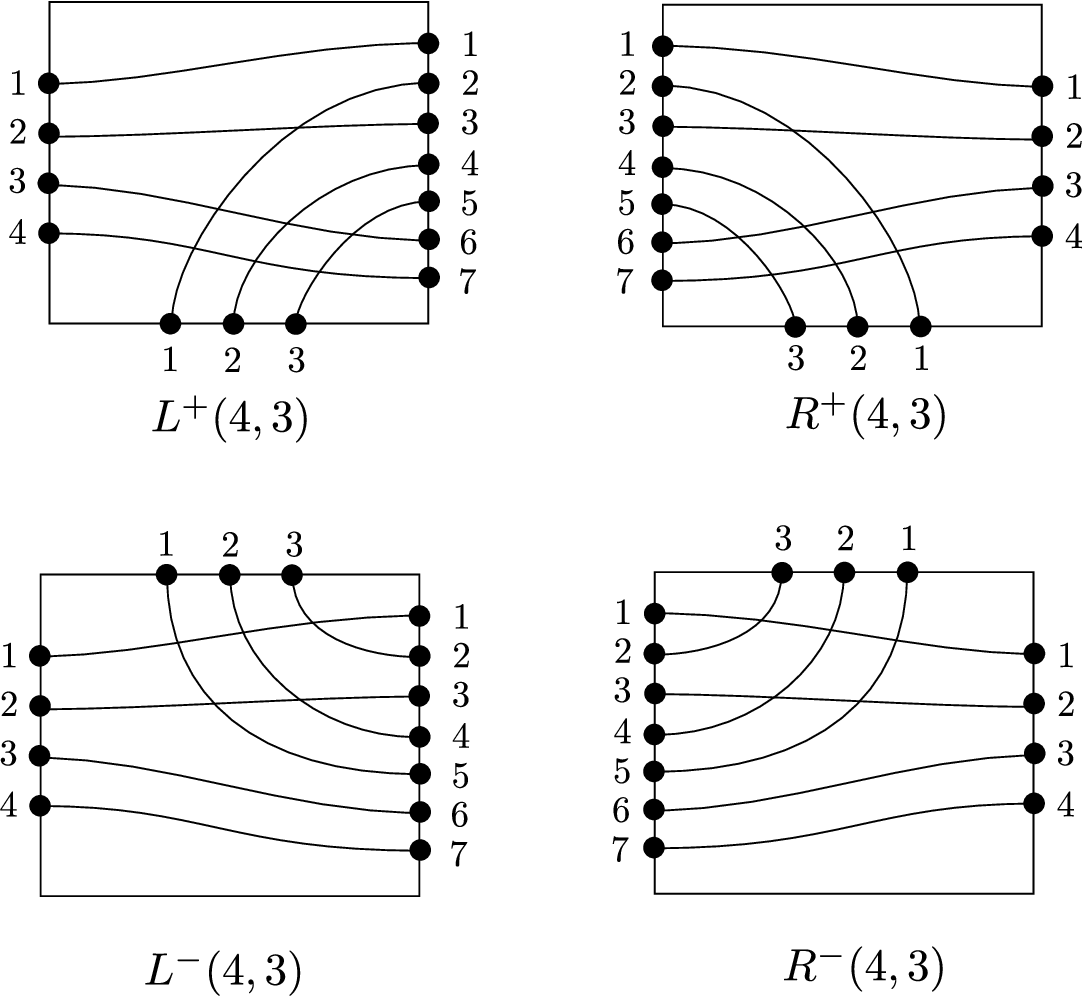}$$
	The actions are obtained by concatenation of diagrams (note that a diagram that is
	not good represents $0$), as in the example below, where we first apply the reflection
	of the rectangle swapping the top and the bottom, then rotate 
	$90$ degrees anticlockwise the diagram of $h'$:
$$\includegraphics[scale=0.75]{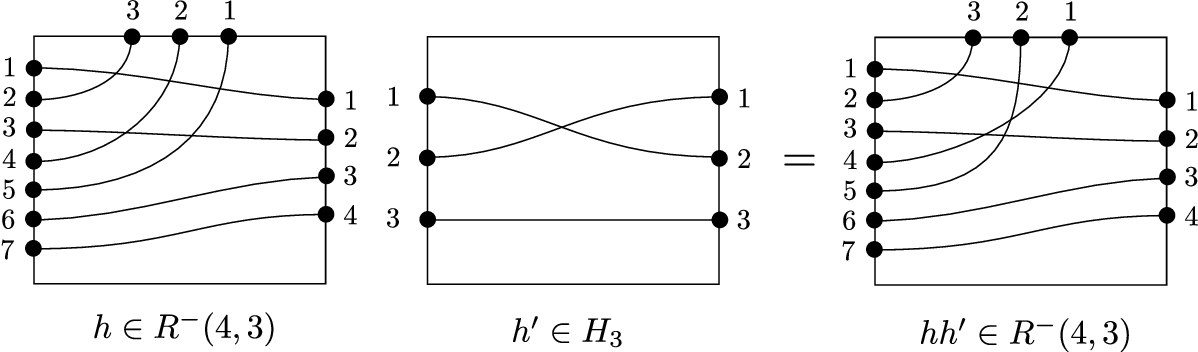}$$

\end{example}

\smallskip
These bimodules coincide with (the nil version of) the bimodules introduced in \S\ref{se:traces}, after restricting
the action of $H_r\otimes H_n$ to $H_r$:
$$L^\pm(r,n)=L^\pm(I,S) \text{ and } R^\pm(r,n)=L^\pm(S,I)
\text{ where }S=\{s_1,\ldots,s_{r+n-1}\} \text{ and }I=\{s_1,\ldots,s_{r-1}\}.$$

\medskip

Given $m\ge 0$, we denote by $w_m\in\GS_m$ the longest element, \ie,
$w_m(i)=m-i+1$.
We have two morphisms of differential graded $\BF_2$-modules (cf Proposition \ref{pr:tracediff})
$$t_{r+n,r}^\pm=t_{S,I}^\pm:H_{r+n}\to H_r\langle \frac{1}{2}n(2r+n-1)\rangle$$ given by

$$t_{r+n,r}^+(T_w)= \begin{cases}
	T_{w_rw_{r+n}w}  & \text{ if }w\in w_{r+n}\GS_r \\
	0 & \text{ otherwise}
\end{cases}
\text{ and }
t_{r+n,r}^-(T_w)= \begin{cases}
	T_{ww_{r+n}w_r}  & \text{ if }w\in \GS_r w_{r+n} \\
	0 & \text{ otherwise}
\end{cases}
$$

\begin{example}
	Let us describe some examples of $t_{7,4}^\pm(T_w)$:
$$\includegraphics[scale=0.6]{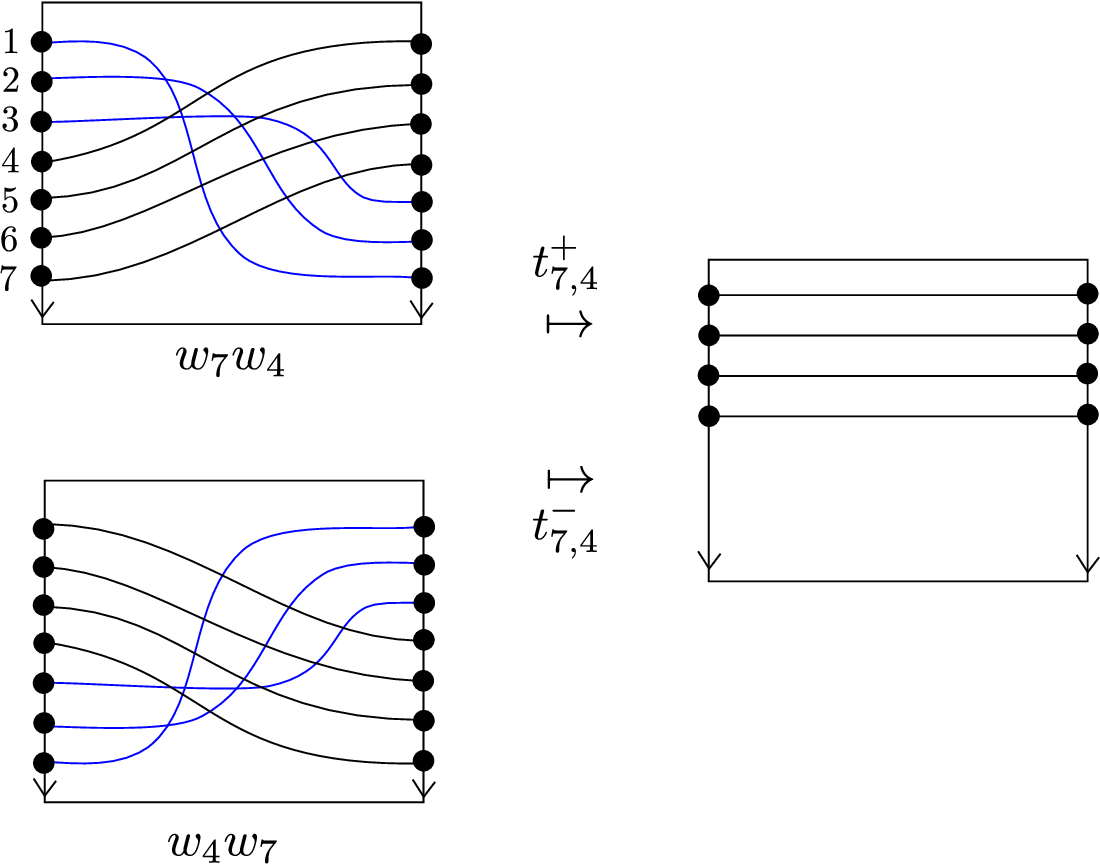}$$
\end{example}

It is immediate that there is an isomorphism of differential graded
$(H_{r+n},H_r\otimes H_n)$-modules
$$\Hom_{H_{r+n}^\opp}(L^\pm(r,n),H_{r+n})\iso R^\pm(r,n),f\mapsto f(1)$$
and it follows from Proposition \ref{pr:tracediff} that there is
an isomorphism of differential graded
$(H_r\otimes H_n,H_{r+n})$-modules
$$L^\mp(r,n)\iso\Hom_{H_r^\opp}(R^\pm(r,n),H_r)\langle \frac{1}{2}n(2r+n-1)\rangle,\ h\mapsto (h'\mapsto t_{n+r,r}^\pm(hh')).$$

\subsubsection{Twisted description}
We describe now $L^+(r,n)$ as a twisted free $(H_r\otimes H_n)$-module.

Consider $E\subset\{1,\ldots,r+n\}$ with $|E|=r$. Let $w_E\in\GS_{r+n}$ be the permutation
such that $w_E(E)=\{1,\ldots,r\}$ and
the restrictions of $w_E$ to $E$ and to $\{1,\ldots,r+n\}\setminus E$ are increasing.
If $E=\{i_1<\cdots<i_r\}$, then we have a reduced decomposition
$$w_E=(s_r\cdots s_{i_r-1})(s_{r-1}\cdots s_{i_{r-1}-1})\cdots (s_2\cdots s_{i_2-1})
(s_1\cdots s_{i_1-1})$$
and 
$$\tilde{L}(w_E)=
\coprod_{b=1}^r \bigl((\{1,\ldots,i_b-1\}\setminus\{i_1,\ldots,i_{b-1}\})
\times\{i_b\}\bigr).$$
There is a bijection
$$\beta:\GS_r\times\GS_n\times\{E\subset\{1,\ldots,r+n\}\ |\ |E|=r\}\iso \GS_{r+n},
(v,v',E)\mapsto vf_r(v')w_E$$
where $f_r(v')\in\GS_{n+r}$ is given by $f_r(v')(i)=i$ for $i\le r$ and
$f_r(v')(r+i)=r+v'(i)$ for $1\le i\le n$.
We have $\ell(\beta(v,v',E))=\ell(v)+\ell(v')+\ell(w_E)$.

\smallskip
Given $(a,i_b)\in \tilde{L}(w_E)$, we
define $v(E,a,b)\in\GS_r$ and $v'(E,a,b)\in\GS_n$ as follows.
Let $b'\in\{1,\ldots,r\}$ be minimal such that $a<i_{b'}$. We define
$v(E,a,b)$ to be the cycle $(b,b-1,\ldots,b')$ and
$v'(E,a,b)$ to be the cycle $(a-b'+1,a-b'+2,\ldots,i_b-b)$.
We have
$$w_Es_{a,i_b}=v(E,a,b)f_r(v'(E,a,b))w_{(E\cup\{a\})\setminus\{i_b\}}$$
and $\ell(w_E)-\ell(w_{(E\cup\{a\})\setminus\{i_b\}})=i_b-a$.

\medskip
Given $m\ge 1$, we define a free differential $(H_r\otimes H_n)$-module
$$V_m=\bigoplus_{E\subset\{1,\ldots,r+n\},\ |E|=r,\ \ell(w_E)=m-1}(H_r\otimes H_n)b_E.$$

Given $m'<m$,
we define $f_{m',m}:V_m\to V_{m'}$
as the morphism of $(H_r\otimes H_n)$-modules given by
$$b_E\mapsto \sum_{\substack{i\in E,\ j\in \{1,\ldots,r+n\}\setminus E\\ i-j=m-m'}}
(T_{v(E,j,i)}\otimes T_{v'(E,j,i)}) b_{(E\cup\{j\})\setminus\{i\}}.$$

We will show below (Lemma \ref{le:decompositionL+}) that
$d(f_{m',m})=\sum_{m>m''>m'}f_{m'm''}\circ f_{m''m}$.
We denote by $V$ the differential $(H_r\otimes H_n)$-module obtained as the corresponding
twisted object $[\bigoplus V_m, (f_{m'm})]$ (cf \S\ref{se:objects}).
We have $V=\bigoplus_m V_m$ as a $(H_r\otimes H_n)$-module and
$d_V=\sum_m d_{V_m}+\sum_{m,m'}f_{m',m}$.

\begin{lemma}
	\label{le:decompositionL+}
	The maps $(f_{m'm})$ define a twisted object $V=[\bigoplus V_m, (f_{m'm})]$.
There is an isomorphism of differential $(H_r\otimes H_n)$-modules
	$$V\iso L^+(r,n),\ (h\otimes h')b_E\mapsto hf_r(\iota_n(h'))T_{w_E}
	\text{ for }h\in H_r \text{ and }h'\in H_n.$$
\end{lemma}

\begin{proof}
The length property of the bijection $\beta$ above shows that the map of the lemma
is an isomorphism of $(H_r\otimes H_n)$-modules.
	Since 
$$d(T_{w_E})=\sum_{i\in E,\ j\in \{1,\ldots,r+n\}\setminus E,\ j<i}
	T_{w_Es_{i,j}},$$
	it follows that the map of the lemma intertwines $d_V$ and the differential
	of $L^+(r,n)$. The lemma follows.
\end{proof}

There is a dual version of Lemma \ref{le:decompositionL+}. In particular, there is
a decomposition of right $(H_r\otimes H_n)$-modules
$$R^+(r,n)=
\bigoplus_{E\subset\{1,\ldots,r+n\},\ |E|=r}T_{w_E^{-1}}
(H_r\otimes f_r(H_n))$$

\subsubsection{Actions}
There is a ``left" $2$-representation on $\CU$
$$\Upsilon^-:\CU\to\End(\CU),\ e^n\mapsto e^n\otimes-$$
and a ``right" $2$-representation on $\CU$
$$\Upsilon^+:\CU\xrightarrow[\sim]{\mathrm{rev}}\CU^{\mathrm{rev}}
\xrightarrow{e^n\mapsto -\otimes e^n}\End(\CU).$$

The bimodule $2$-representation $L^\pm$ associated to $\Upsilon^\pm$ is given by 
$$L^\pm(e^r,e^s,e^n)=\delta_{s,r+n}L^\pm(r,n)$$
and it is left and right finite.
Its left dual is isomorphic to
the bimodule $2$-representation $R^\pm$ given by
$$R^\pm(e^s,e^r,e^n)=\delta_{s,r+n}R^\pm(r,n)$$
while its right dual is isomorphic to $R^\mp\langle-\frac{1}{2}n(2r+n-1)\rangle$
(note that the action of $\CU$ on the duals is obtained
from the natural action of $\CU^{\mathrm{rev}\opp}$ by
applying the isomorphism $\mathrm{rev}\circ\opp$).

\subsubsection{Gluing}
\label{se:glueingregular}
The $2$-representations $\Upsilon^+$ and $\Upsilon^-$ commute
strictly. Let us describe this in terms of bimodules.

We consider the bimodule $2$-representations
$E_1=\bigoplus_{s\ge 0}L^-(s,1)$ and $E_2=\bigoplus_{s\ge 0}L^+(s,1)$ as above.
There is a canonical isomorphism $E_1^\vee\iso \bigoplus_{s\ge 0}
R^-(s,1)$ and we identify those bimodules.

Define $\sigma:E_2E_1\iso E_1E_2$ as the isomorphism such that for $s\ge 1$, the
following diagram of morphism of
$(H_{s-1},H_{s+1})$-bimodules is commutative:
$$\xymatrix{
	L^+(s-1,1)\otimes_{H_s}L^-(s,1)\ar[rr]^-\sigma_-\sim \ar[dr]_{a\otimes b
	\mapsto ab}^-\sim&&
L^-(s-1,1)\otimes_{H_s}L^+(s,1)\ar[dl]^{a\otimes b\mapsto ab}_-\sim \\
 &H_{s+1} 
 }$$
Note that the left action of $a\in H_{s-1}$ on $H_{s+1}$ is given by
left multiplication by $f_1(a)$. It is immediate to check that the diagrams
(\ref{eq:diagsigma}) commute.

\medskip
As in (\ref{eq:gamma}), the morphism $\sigma$ gives a
morphism of functors 
$$\lambda:R^-(-_1,-,e)\otimes L^+(-,-_2,e) \to L^+(-_1,-,e)\otimes R^-(-,-_2,e)$$
where 
$\lambda(e^s,e^s)$ is
given by the following morphism of differential graded $(H_s,H_s)$-bimodules
$$\xymatrix{
	R^-(e^s,-,e)\otimes L^+(-,e^s,e)=R^-(s-1,1)\otimes_{H_{s-1}}L^+(s-1,1)\ar[d] &
a\otimes b\ar@{|->}[d] \\
L^+(e^s,-,e)\otimes R^-(-,e^s,e) = L^+(s,1)\otimes_{H_{s+1}}R^-(s,1) &
a\otimes f_1(b)
}$$

\begin{rem}
	An example of a diagrammatic description of $\lambda$ is given below:
$$\includegraphics[scale=0.45]{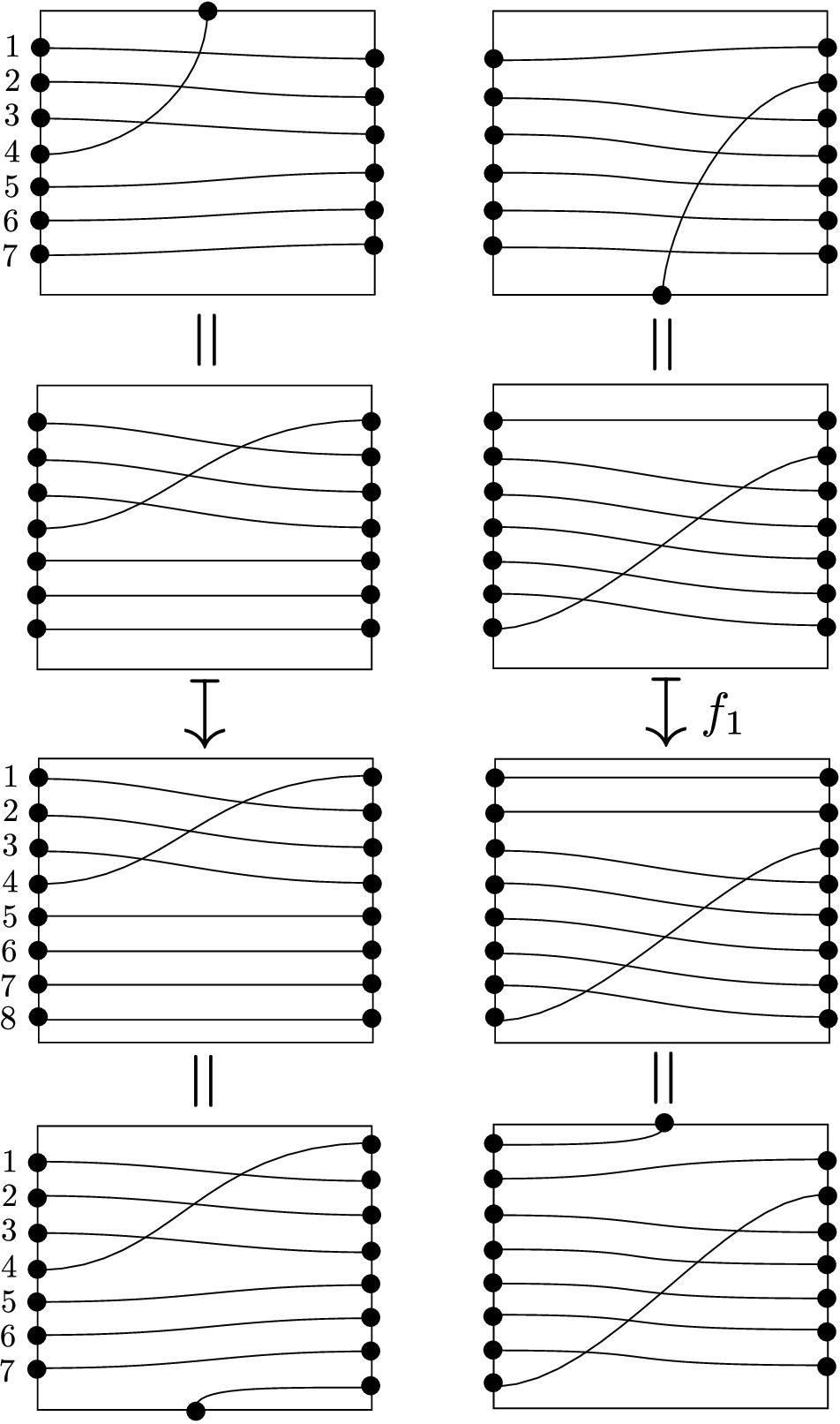}$$
\end{rem}

The morphism 
\begin{multline*}
(R^-(-,-,e)\lambda L^+(-,-,e))\circ(R^-(-,-,e)^2\tau-\tau L^+(-,-,e)^2):\\
R^-(-,-,e)^2L^+(-,-,e)^2\to (R^-(-,-,e)L^+(-,-,e))^2\langle -1\rangle
\end{multline*}
is on $(e^s,e^s)$ the morphism of differential graded $(H_s,H_s)$-bimodules
\begin{multline*}
	R^-(s-1,1)\otimes_{H_{s-1}}R^-(s-2,1)\otimes_{H_{s-2}}L^+(s-2,1)\otimes_{H_{s-1}}
L^+(s-1,1) \to \\
R^-(s-1,1)\otimes_{H_{s-1}}L^+(s-1,1)\otimes_{H_s}
R^-(s-1,1)\otimes_{H_{s-1}}L^+(s-1,1)\langle -1\rangle
\end{multline*}
given by
$$ 1\otimes 1\otimes 1\otimes 1\mapsto
 T_1\otimes 1\otimes 1\otimes 1-
1\otimes 1\otimes 1\otimes T_{s-1}.$$

\medskip
Given $s\ge 1$, 
let $M_s=R^-(s-1,1)\otimes_{H_{s-1}}L^+(s-1,1)$, a differential graded
$(H_s,H_s)$-bimodule. When $s\ge 2$,
we define $\kappa=1\otimes 1\otimes 1\otimes T_{s-1}-T_1\otimes 1\otimes 1\otimes 1\in 
M_s\otimes_{H_s}M_s$. We put $\kappa=0$ when $s=1$.
We put $M_0=0$.

\begin{lemma}
	\label{le:quotientgluing}
	There is a morphism of differential graded $(H_s,H_s)$-bimodules
	$M_s\to \hat{H}_s^+$ given by  $a\otimes b\mapsto acb$ for $a,b\in H_s$. It induces an
	isomorphism of differential graded algebras and of differential graded
	$(H_s,H_s)$-bimodules
	$T_{H_s}(M_s)/(\kappa)\iso \hat{H}_s^+$.
\end{lemma}

\begin{proof}
	We have $acT_ib=aT_{i+1}cb$ for $i\in\{1,\ldots,s-2\}$. This shows the first
	statement of the lemma.

	We have now a morphism of differential graded algebras and of $(H_s,H_s)$-bimodules
	$f':T_{H_s}(M_s)\to\hat{H}_s^+$ induced by the morphism $M_s\to \hat{H}_s^+$. 
	We have 
	$$f'((1\otimes 1)\otimes (1\otimes T_{s-1}))=c^2T_{s-1}=T_1c^2=f'((T_1\otimes 1)\otimes
	(1\otimes 1)),$$
	hence
	$f'(\kappa)=0$. So, $f$ induces a morphism of algebras
	$f:T_{H_s}(M_s)/(\kappa)\to\hat{H}_s^+$.
	
	On the other hand, $\hat{H}_s^+$ is
	the free algebra generated by $H_s$ and $c$ with the relations 
	$cT_i=T_{i+1}c$ for $i\in\{1,\ldots,s-2\}$ and $c^2T_{s-1}=T_1c^2$
	(Proposition \ref{le:generatorsrelationsHnhat+}). Since
	$T_{i+1}\otimes 1=1\otimes T_i$ in $M_s$ for $i\in\{1,\ldots,s-2\}$ and
	$(1\otimes 1)\otimes (1\otimes T_{s-1})=(T_1\otimes 1)\otimes (1\otimes 1)$ in
	$M_s\otimes M_s$, we deduce that there
	is a morphism of algebras $g:\hat{H}_s^+\to T_{H_s}(M_s)/(\kappa),\
	T_i\mapsto T_i,\ c\mapsto 1\otimes 1$. The morphisms $f$ and $g$ are inverse and
	we are done.
\end{proof}

Let $\CH^+$\indexnot{H+}{\CH^+}
be the differential graded pointed category with set of objects $\BZ_{\ge 0}$ and 
$\Hom_{\CH^+}(m,n)=\delta_{mn}\hat{\GS}_n^{+,\mathrm{nil}}$.
Lemma \ref{le:quotientgluing} has the following consequence.

\begin{thm}
	\label{th:tensorU}
The construction of Lemma \ref{le:quotientgluing} induces an isomorphism of differential
	graded pointed categories $\Theta:\Delta'_\lambda(\CU^\bullet)\iso \CH^+$\indexnot{theta}{\Theta}.
\end{thm}

Since $\sigma$ is an isomorphism, we have a diagonal 
bimodule $2$-representation on $\Delta'_\lambda(\CU^\bullet)$ 
(cf \S\ref{se:action}). 
Via the isomorphism of Theorem \ref{th:tensorU}, this corresponds to
the bimodule $2$-representation on $\CH^+$ defined as follows.
Define a differential graded right $\hat{H}_n^+$-module 
$$E_n=
{\xy (0,0)*{\hat{H}_n^+[1]\oplus \hat{H}_n^+}
\ar@/^/^{h\mapsto ch}(-5,4)*{};(5,4)*{}, \endxy}$$

We define a left action of $\hat{H}_{n-1}^+$ on $E_n$ as follows:
$$T_i \text{ acts by }\left(\begin{matrix}T_i&0\\0&T_{i+1}\end{matrix}\right)
	\text{ for }1\le i\le n-2 \text{ and }
c \text{ acts by } \left(\begin{matrix}cT_{n-1}&1\\0&T_1c\end{matrix}\right).$$
	This defines a structure of differential graded
	$(\hat{H}_{n-1}^+,\hat{H}_n^+)$-bimodule on $E_n$.

	\smallskip
	Note that setting $E=0$ corresponds to inverting $c$: this turns
	$\CH^+$ into the differential graded pointed category with same
	objects and with 
	$\Hom_{\CH}(m,n)=\delta_{mn}\hat{\GS}_n^{\mathrm{nil}}$.

\subsection{Nil Hecke category}
\label{se:nilHeckecat}
\subsubsection{Definition}
\label{se:Sncat}
We now define a groupoid of $n$-periodic bijections.

Given $I$ a subset of $\BZ/n$ we denote by $\tilde{I}$ its inverse image in $\BZ$.

Let $\CS_n$\indexnot{Sn}{\CS_n} be the category with objects the subsets of $\BZ/n$ and where
$\Hom_{\CS_n}(I,J)$ is the set of $n$-periodic
bijections $\sigma:\tilde{I}\iso\tilde{J}$. The group $n\BZ$ acts by
translation on $\Hom$-sets.
Note that $\hat{\GS}_n=\End_{\CS_n}(\BZ/n)$. 

Given $i,j\in\tilde{I}$ with $i-j{\not\in}n\BZ$, the element
$s_{ij}\in\hat{\GS}_n$ restricts to an $n$-periodic bijection $\tilde{I}\iso\tilde{I}$,
which we also denote by $s_{ij}$\indexnot{sij}{s_{ij}}.

\medskip
Let $I$ be a subset of $\BZ/n$.
There is a unique increasing bijection
$\beta_I:\{1,\ldots,|I|\}\iso \tilde{I}\cap\{1,\ldots,n\}$. We extend it to an
increasing bijection $\BZ\iso\tilde{I}$ by $\beta_I(r+d|I|)=\beta_I(r)+dn$ for
$r\in\{1,\ldots,|I|\}$ and $d\in\BZ$.
There is an isomorphism of groups
$$F_I:
\hat{\GS}_{|I|}\iso\End_{\CS_n}(I),\ \sigma\mapsto \beta_I\circ\sigma\circ\beta_I^{-1}
\indexnot{FI}{F_I}.$$


\smallskip

\subsubsection{Length}
Consider $\sigma\in\Hom_{\CS_n}(I,J)$. We define
$$L(\sigma)=\{(i,i')\in\tilde{I}^2\ |\ i<i',\ \sigma(i)>\sigma(i')\}$$
\indexnot{Ls}{L(\sigma)}
and $\tilde{L}(\sigma)=\{(i,i')\in L(\sigma)\ |\ 1\le i\le n\}$.
The canonical map $\tilde{L}(\sigma)\to L(\sigma)/n\BZ$ is bijective.
We define
$\ell(\sigma)=|\tilde{L}(\sigma)|\indexnot{ls}{\ell(\sigma)}$.

\begin{lemma}
	\label{le:lengthincrease}
	We have $\ell(\sigma'\circ\sigma)\le \ell(\sigma')+\ell(\sigma)$ for all
$\sigma\in\Hom_{\CS_n}(I,J)$ and $\sigma'\in\Hom_{\CS_n}(J,K)$.
\end{lemma}

\begin{proof}
	We have
		\begin{multline*}
		L(\sigma'\circ\sigma)=
	\{(i_1,i_2)\in\tilde{I}^2\ |\ i_1<i_2,\ \sigma(i_1)>\sigma(i_2),\
	\sigma'\circ\sigma(i_1)>\sigma'\circ\sigma(i_2)\}\sqcup \\
	\{(i_1,i_2)\in\tilde{I}^2\ |\ i_1<i_2,\ \sigma(i_1)<\sigma(i_2),\
	\sigma'\circ\sigma(i_1)>\sigma'\circ\sigma(i_2)\} \\
	=\{(i_1,i_2)\in L(\sigma)\ |\ \sigma'\circ\sigma(i_1)>\sigma'\circ\sigma(i_2)\}
		\sqcup(\sigma^{-1}\times\sigma^{-1})\bigl(
		\{(j_1,j_2)\in L(\sigma')\ |\ \sigma^{-1}(j_1)<\sigma^{-1}(j_2)\}\bigr).
		\end{multline*}

		It follows that
		$$\ell(\sigma')+\ell(\sigma)-\ell(\sigma'\circ\sigma)=
	2|\{(i_1,i_2)\in\tilde{I}^2\ |\ i_1<i_2,\ \sigma(i_1)>\sigma(i_2),\
	\sigma'\circ\sigma(i_1)<\sigma'\circ\sigma(i_2)\}/n\BZ|\ge 0.$$
\end{proof}

\smallskip
Let $\sigma\in\Hom_{\CS_n}(I,J)$. We have $\ell(\sigma)=0$ if and only if
$\sigma$ is an increasing bijection.

Given $\tau\in\Hom_{\CS_n}(J,I)$ with $\ell(\tau)=0$, we have
$L(\tau\circ\sigma)=L(\sigma)=(\tau\times\tau)\bigl(L(\sigma\circ\tau)\bigr)$,
hence $\ell(\tau\circ\sigma)=\ell(\sigma\circ\tau)=\ell(\sigma)$. 

Since $L(\tau\circ\sigma)=(\beta_I\times\beta_I)(L(F_I^{-1}(\tau\circ\sigma)))$,
we have $\ell(\sigma)=\ell(F_I^{-1}(\tau\circ\sigma))$.
As a consequence,
we deduce the following result from Lemma \ref{le:lengthaffine}.

\begin{lemma}
	\label{le:formulalength}
	Let $\sigma\in\Hom_{\CS_n}(I,J)$. We have
	$$\ell(\sigma)=
	\sum_{\substack{0\le i_1<i_2<n\\ i_1,i_2\in\tilde{I}}}
	\bigl|{\lfloor\frac{\sigma(i_2)-\sigma(i_1)}{n}\rfloor}\bigr|.$$
\end{lemma}

The next lemma relates length and number of intersections of paths on a cylinder.

\begin{lemma}
	\label{le:intersectionlength}
	Let $\sigma\in\Hom_{\CS_n}(I,J)$ where $I=\{i_1+n\BZ,i_2+n\BZ\}$ and
	$J=\{j_1+n\BZ,j_2+n\BZ\}$ with $1\le i_1\neq i_2\le n$,
	$1\le j_1\neq j_2\le n$ and $\sigma(i_r)=j_r\pmod n$ for $r\in\{1,2\}$.
	Fix $\beta:\{i_1,i_2,j_1,j_2\}\to\BR$ increasing with
	$|\beta(u)-\beta(v)|<1$ for all $u,v$.

	Consider $\gamma_r:[0,1]\to\BR$ continuous with
	$\gamma_r(0)=\beta(i_r)$ and $\gamma_r(1)=\beta(j_r)+\frac{\sigma(i_r)-j_r}{n}$
	for $r\in\{1,2\}$. We have
	$$\ell(\sigma)\le
	|\{t\in[0,1]\ |\ e^{2i\pi \gamma_1(t)}=e^{2i\pi\gamma_2(t)}\}|$$
	with equality if, for all $r\in\{1,2\}$, the map
	$\gamma_r$ is affine.
\end{lemma}

\begin{proof}
Without loss of generality, we can assume $i_1<i_2$. The lemma follows
by applying the intermediate value theorem to $\gamma_2(t) - \gamma_1(t)$ and
using Lemma \ref{le:formulalength}, considering four cases according to the signs of
$j_2-j_1$ and $\sigma(i_2)-\sigma(i_1)$.
\end{proof}

\subsubsection{Filtration}
Given $I,J\subset\BZ/n$, we define
$\Hom_{\CS_n^{\ge -r}}(I,J)=\{\sigma\in \Hom_{\CS_n}(I,J)\ |\ 
l(\sigma)\le r\}$ for $r\in\BZ_{\ge 0}$. It follows from Lemma \ref{le:lengthincrease} that
this defines a structure of $\BZ_{\le 0}$-filtered category on $\CS_n$.
We put $\CH_n=\mathrm{gr}\CS_n^\bullet$\indexnot{Hn}{\CH_n}, a pointed $\BZ_{\le 0}$-graded
category.

\smallskip
Note that a map $\sigma$ of length $0$ is invertible in $\CH_n$. Note also
that $F_I$ induces an isomorphism of graded pointed monoids
$\hat{\GS}_{|I|}^{\nil}\iso\End_{\CH_n}(I)$.

\subsubsection{Non-commutative degree}

\smallskip

Let us consider the free abelian groups
$R_n=\bigoplus_{a\in\BZ/n}\BZ \alpha_a$\indexnot{Rn}{R_n} and
$L_n=\bigoplus_{a\in\BZ/n} \BZ\eps_a$\indexnot{Ln}{L_n}. We define a linear map
$\rho:R_n\to L_n$ by $\rho(\alpha_a)=\eps_{a+1}-\eps_a$ and
a representation of the group $R_n$ on $L_n$ given by
$$\alpha_a\cdot\eps_b=(\delta_{a,b}+\delta_{a+1,b})\eps_b.$$
Note that $\delta=\sum_{a\in\BZ/n}\alpha_a\in\ker\rho$ and
$\delta\cdot\eps_b=2\eps_b$ for all $b$.

\smallskip
We define a bilinear map
$$\langle -,-\rangle:R_n\times R_n\to L_n,\
\langle \alpha,\alpha'\rangle=\alpha\cdot \rho(\alpha')\indexnot{<}{\protect\langle\alpha,\beta\protect\rangle}.$$

Let $\Gamma_n'=L_n\times R_n$. We define a group structure on
$\Gamma_n'$ by
$$(l,\alpha)\cdot(l',\alpha')=(l+l'+\langle \alpha,\alpha'\rangle,
\alpha+\alpha').$$

\medskip
Given $I\subset\BZ/n$, we put 
$\eps_I=\sum_{a\in I}\eps_a\in L_n$.
Given $i,j\in\BZ$, we put
$$\alpha_{i,j}=\sum_{i\le r<j}\alpha_{r+n\BZ}-
\sum_{j\le r<i}\alpha_{r+n\BZ}.$$
Note that $\alpha_{i,i+1}=\alpha_{i+n\BZ}$,
$\alpha_{i+n,j+n}=\alpha_{i,j}$ and 
$\alpha_{i,j}+\alpha_{j,k}=\alpha_{i,k}$ for all $i,j,k\in\BZ$.
Note also that $\delta=\alpha_{i,i+n}$ for all $i\in\BZ$. Note finally that
$\rho(\alpha_{i,j})=\eps_{j+n\BZ}-\eps_{i+n\BZ}$.


\medskip
Consider $\sigma\in\Hom_{\CS_n}(I,J)$. We put
$$\llbracket\sigma\rrbracket=\sum_{i\in \tilde{I}\cap [1,n]}
\alpha_{i,\sigma(i)}\in R_n.$$
Note that
$\rho(\llbracket\sigma\rrbracket)=\eps_J-\eps_I$ and
$\llbracket\sigma'\circ\sigma\rrbracket=\llbracket\sigma'\rrbracket+
\llbracket\sigma\rrbracket$ for any 
$\sigma'\in\Hom_{\CS_n}(J,K)$.

\medskip
We define
$$
m(\sigma)=\llbracket\sigma\rrbracket\cdot\eps_I\in L_n\indexnot{ms}{m(\sigma)} \text{ and }
\dm(\sigma)=(-m(\sigma),-\llbracket\sigma\rrbracket)\in
\Gamma'_n \indexnot{deg}{\dm(\sigma)}.$$

\smallskip
%

\begin{lemma}
	\label{le:classesWeyl}
	Let $w\in W_{|I|}$, $m\in\BZ$ and let $\sigma=F_I(wc^m)$
	be the element of $\End_{\CS_n}(I)$
corresponding to $wc^m$.
	We have $\ell(\sigma)=\ell(w)$,
	$\llbracket \sigma\rrbracket=m\cdot\delta$
	and $m(\sigma)=2m\eps_I$.
\end{lemma}

\begin{proof}
	The first statement follows from the fact that
	$F_I$ preserves lengths (cf the discussion before Lemma \ref{le:formulalength}).

	Note that $\llbracket s_{i,j}\rrbracket=0$ for
	$i,j\in\tilde{I}$ with $i-j{\not\in}n\BZ$, while
	$\llbracket F_I(c)\rrbracket=\delta$. We deduce that
	$\llbracket \sigma\rrbracket=m\cdot\delta$.

	The last statement of the lemma is immediate.
\end{proof}

\begin{lemma}
	\label{le:degreelength}
	Consider $\sigma\in\Hom_{\CS_n}(I,J)$ and $\sigma'\in\Hom_{\CS_n}(J,K)$.
	We have $\dm(\sigma'\circ\sigma)=\dm(\sigma')\cdot\dm(\sigma)$.
\end{lemma}

\begin{proof}
We have
	$$m(\sigma'\circ\sigma)=\llbracket\sigma'\rrbracket\cdot\eps_I+
	\llbracket\sigma\rrbracket\eps_I=
	m(\sigma')+m(\sigma)+\llbracket\sigma'\rrbracket\cdot(\eps_I-\eps_J),$$
	hence
	$$m(\sigma')+m(\sigma)-m(\sigma'\circ\sigma)=
	\llbracket\sigma'\rrbracket\cdot\rho(\llbracket\sigma\rrbracket).$$
	The lemma follows.
\end{proof}

We put $\Gamma_n=\frac{1}{2}\BZ\times \Gamma'_n$\indexnot{Ga}{\Gamma_n}. We endow
$\Gamma_n$ with a structure of $\BZ$-monoid by using the canonical embedding
$\BZ\hookrightarrow \frac{1}{2}\BZ\hookrightarrow\Gamma_n$.

Given 
$\sigma\in\Hom_{\CS_n}(I,J)$, we put $\deg(\sigma)=(-\ell(\sigma),\dm(\sigma))\in
\Gamma_n\indexnot{deg}{\deg(\sigma)}$.

\medskip
Let $D$ be a subset of $\{1,\ldots,n\}\times\{\pm 1\}$ that embeds in its
projection on $\{1,\ldots,n\}$.
We denote by $\Gamma_D$\indexnot{GaD}{\Gamma_D} the quotient of
$\Gamma_n$ by the subgroup generated by $(0,\eps_{i+n\BZ})+(\frac{1}{2}\nu_i,0)$,
where $(i,\nu_i)\in D$. We identify $\frac{1}{2}\BZ$ with the image of $\frac{1}{2}\BZ\times 0$ in
$\Gamma_D$.
We define a partial order on $\Gamma_D$ by $h\ge g$ if $hg^{-1}$ is in $\frac{1}{2}\BZ_{\ge 0}$.
We denote by $\deg_D(\sigma)$ the image of $\deg(\sigma)$\indexnot{degD}{\deg_D(\sigma)}
in $\Gamma_D$.

Given $E$ a subset of $\{1,\ldots,n\}$, we put $E^+=\{(i,1)\ |\ i\in E\}$.

\smallskip
By Lemmas \ref{le:lengthincrease} and \ref{le:degreelength},
we obtain a $\Gamma_D$-filtration on $\CS_n$ by defining
$$\Hom_{\CS_n^{\ge g}}(I,J)=\{\sigma\in\Hom_{\CS_n}(I,J)\ |\ \deg_D(\sigma)\ge g\}.$$

It follows from Lemma \ref{le:degreelength} that the pointed category $\CH_n$
is isomorphic to the graded pointed category associated to the $\Gamma_D$-filtration of $\CS_n$
(after forgetting the $\Gamma_D$-grading to $\BZ$).

\smallskip
Note that if $D=\emptyset$, then $\Gamma_D=\Gamma_n$, $\deg_D=\deg$ and the
$\BZ_{\le 0}$ grading on $\CH_n$ given by the length can be recovered from
the $\Gamma_n$-grading by using the quotient map
$\Gamma_n\to\Gamma_n/\Gamma'_n=\frac{1}{2}\BZ$.

This quotient map provides a $\BZ$-grading on
the $\Gamma_n$-graded pointed category associated to the $\Gamma_n$-filtration of
$\CS_n$. This $\BZ$-graded pointed category is isomorphic to $\CH_n$.

\begin{rem}
	\label{re:quotientform}
	The bilinear form 
	$\langle\langle -,-\rangle\rangle:R_n\times R_n\to \frac{1}{2}\BZ$
	obtained from $\langle-,-\rangle$ by composing with the
morphism $L_n\to\frac{1}{2}\BZ,\ \eps_i\mapsto -\frac{1}{2}$ is given by
	$\langle\langle\alpha_a,\alpha_b\rangle\rangle=\frac{1}{2}(
	\delta_{b,a+1}-\delta_{b+1,a})$.
It is antisymmetric.
\end{rem}

	\label{le:gradinggroup}
Assume $D=[1,n]^+$. Composing with the quotient map
	$\Gamma_n\twoheadrightarrow\Gamma_D$, 
the embedding $\frac{1}{2}\BZ\hookrightarrow
\Gamma_n,\ r\mapsto (r,0)$ and the quotient map
$\Gamma_n\twoheadrightarrow R_n,\ (r,(l,\alpha))\mapsto \alpha$ induce
	an embedding of $\frac{1}{2}\BZ$ as a central
	subgroup of $\Gamma_D$ with quotient map 
$\Gamma_D\twoheadrightarrow R_n$. So, $\Gamma_{[1,n]^+}$ identifies with the set
$\frac{1}{2}\BZ\times R_n$, with multiplication given by
$(r,\alpha)\cdot (r',\alpha')=(r+r'+\langle\langle \alpha,\alpha'\rangle
\rangle,\alpha+\alpha')$.
When $n\ge 3$, the group $\Gamma_{[1,n]^+}$ has a presentation with generators 
$z=(\frac{1}{2},0)$, $g_a=(0,\alpha_a)$, $a\in\BZ/n$ and relations
$$zg_a=g_az,\
g_ag_bg_a^{-1}g_b^{-1}=\begin{cases}
	z & \text{ if }b=a+1\\
	z^{-1} & \text{ if }b=a-1\\
	1 & \text{ otherwise.}
\end{cases}
$$

\medskip
We define a morphism of groups
$$\epsilon:\Gamma_{[1,n]^+}\to\BZ/2,\
z\mapsto 1,\ g_a\mapsto 1.$$

\begin{lemma}
	\label{le:degepsilon}
	Given $(r,\sum_a v_a\alpha_a)\in \Gamma_{[1,n]^+}$, we have
$\epsilon(r,\sum_a v_a\alpha_a)=2r+\frac{1}{2}\bigl|\{a\in\BZ/n\ |\ v_a+v_{a+1}
\text{ odd}\}\bigr|$.

Given $\sigma\in\Hom_{\CS_n}(I,J)$, we have $\epsilon(\deg_{[1,n]^+}(\sigma))=0$.
\end{lemma}
	
	\begin{proof}
		Denote by $\tilde{\epsilon}$ the map defined by the right
		hand side of the equality of the lemma. 
	
		Let $N$ (resp. $N'$) be the cardinality of the set 
		of $a\in\BZ/n$ such that $v_a+v_{a+1}$
		(resp.  $v'_a+v'_{a+1}$) is odd, where
		$v'_a=v_a+\delta_{ab}$. The integers $N$ and $N'$ are even.
	We have
		\begin{align*}
		\tilde{\epsilon}(r,\sum_a v_a\alpha_a)+
			\tilde{\epsilon}((r,\sum_a v_a\alpha_a)(s,\alpha_b))&=
\tilde{\epsilon}(r,\sum_a v_a\alpha_a)+
		\tilde{\epsilon}(r+s+\frac{1}{2}(v_{b-1}-v_{b+1}),
		\alpha_b+\sum_a v_a\alpha_a)\\
			&=2s+v_{b+1}+v_{b-1}+\frac{1}{2}(N+N').
		\end{align*}
				We have 
		$$N'=
		\begin{cases}
			N+2	& \text{ if }v_{b-1},\ v_b \text{ and }
			v_{b+1} \text{ have the same parity} \\
			N-2	& \text{ if }v_{b-1},\ v_b+1 \text{ and }
			v_{b+1} \text{ have the same parity} \\
			N & \text{otherwise}.
		\end{cases}$$
		It follows that
$$\tilde{\epsilon}(r,\sum_a v_a\alpha_a)+
		\tilde{\epsilon}((r,\sum_a v_a\alpha_a)(s,\alpha_b))=
		2s+1.$$
		We deduce by
		induction on $\sum_{a\in\BZ/n}|v_a|$ 
that $\tilde{\epsilon}(r,\sum_a v_a\alpha_a)=\epsilon(r,\sum_a v_a\alpha_a)$.

\medskip
Given $a,b\in\BZ/n$ and $i\in I$, we have
		$$\alpha_{a,b}\cdot\eps_i=
		\delta_{\substack{i\in\{a,b\}\\ a\neq b}}\ \eps_i
		\mod 2L_n.$$
It follows that $m(\sigma)\equiv\sum_{i\in I\setminus(I\cap J)}\eps_i
	\mod{2L_n}$. Write $\llbracket\sigma\rrbracket=\sum_a v_a\alpha_a$.
	Given $a\in \BZ/n$, the integer $v_a+v_{a+1}$ is odd
	if and only if $a\in I\Delta J$, hence
	$\epsilon(0,\llbracket\sigma\rrbracket)=\frac{1}{2}|I\Delta J|=
		|I\setminus(I\cap J)|$. It follows that 
		$\epsilon(\deg_{[1,n]^+}(\sigma))=0$.
	\end{proof}

	It follows from Lemma \ref{le:degepsilon} that the $\Gamma_n$-grading on
	$\CH_n$ comes from a grading by the kernel of the 
	composition $\Gamma_n\xrightarrow{\can}\Gamma_{[1,n]^+}
	\xrightarrow{\epsilon}\BZ/2$.

\subsubsection{Differential}
\label{se:Heckedifferential}
Given $\sigma\in\Hom_{\CS_n}(I,J)$, let
$D(\sigma)$\indexnot{Ds}{D(\sigma)} be the set of pairs $(i_1,i_2)\in L(\sigma)$ such that
	\begin{itemize}
		\item $i_2-i_1<n$ or $\sigma(i_1)-\sigma(i_2)<n$ and
		\item given $i\in\tilde{I}$
with $i_1<i<i_2$, we have $\sigma(i_1)<\sigma(i)$ or $\sigma(i)<\sigma(i_2)$.
	\end{itemize}

We put $\tilde{D}(\sigma)=D(\sigma)\cap \tilde{L}(\sigma)$. The diagonal action
of $n\BZ$ on $L(\sigma)$ preserves $D(\sigma)$ and we have a canonical bijection
$\tilde{D}(\sigma)\iso D(\sigma)/n\BZ$.

Given $(i_1,i_2)\in L(\sigma)$, we put $\sigma^{i_1,i_2}:=
\sigma\circ s_{i_1,i_2}$\indexnot{si}{\sigma^{i_1,i_2}}.

\smallskip
We define a partial order\indexnot{<}{\sigma'<\sigma} on 
$\Hom_{\CS_n}(I,J)$ as the transitive closure of $\sigma'<\sigma$ if
$\sigma'=\sigma^{i_1,i_2}$ for some $(i_1,i_2)\in D(\sigma)$.

When $I=J=\BZ/n$, 
this coincides with the extended Chevalley-Bruhat order on $\hat{\GS}_n$ 
by Lemma \ref{le:lengthoneaffine} and given $(i_1,i_2)\in L(\sigma)$, we have
$\sigma^{i_1,i_2}<\sigma$ (Lemma \ref{le:lengthaffine}).
The next lemma shows that this holds for general maps in $\CS_n$.

\begin{lemma}
	\label{le:tauBruhat}
	Let $\sigma,\sigma'\in\Hom_{\CS_n}(I,J)$.
	Given $\tau\in\Hom_{\CS_n}(J,I)$ with $\ell(\tau)=0$, we have
	$\sigma'<\sigma$ if and only if $\tau\circ\sigma'<\tau\circ\sigma$
	if and only if $\sigma'\circ\tau<\sigma\circ\tau$.
\end{lemma}

\begin{proof}
	Note that $\tau$ is an increasing bijection since $\ell(\tau)=0$.
	We have $D(\tau\circ\sigma)=D(\sigma)$ and
	given $(i_1,i_2)\in D(\sigma)$,
	we have $(\tau\circ\sigma)^{i_1,i_2}=\tau\circ \sigma^{i_1,i_2}$. This shows
	the first equivalence. The second equivalence follows from the fact that
	$D(\sigma\circ\tau)=(\tau^{-1}\times\tau^{-1})(D(\sigma))$ and
	given $(i_1,i_2)\in D(\sigma)$,
	we have $(\sigma\circ\tau)^{\tau^{-1}(i_1),\tau^{-1}(i_2)}=
	\sigma^{i_1,i_2}\circ\tau$.
\end{proof}

\begin{lemma}
	\label{le:setDforSn}
	Given $\sigma\in\Hom_{\CS_n}(I,J)$, there is a bijection
	$$\tilde{D}(\sigma) \iso \{\sigma'\in\Hom_{\CS_n}(I,J)\ |\ \sigma'<\sigma,\
\ell(\sigma')=\ell(\sigma)-1\},\ 
		(i_1,i_2)\mapsto \sigma^{i_1,i_2}.$$
	Note that
	$$\{\sigma'\in\Hom_{\CS_n}(I,J)\ |\ \sigma'<\sigma,\
\ell(\sigma')=\ell(\sigma)-1\}=
		\{\sigma'\in\Hom_{\CS_n}(I,J)\ |\ \sigma'<\sigma,\
		\deg(\sigma')=\deg(\sigma)+1\}.$$

	Given $(i_1,i_2)\in L(\sigma)$, we have $(i_1,i_2)\in D(\sigma)$ if and only
	if $\deg_D(\sigma)=\deg_D(\sigma^{i_1,i_2})-1$ for some subset (equivalently, for 
	any subset) $D$ of $\{1,\ldots,n\}\times\{\pm 1\}$ that embeds in its
projection on $\{1,\ldots,n\}$.
\end{lemma}

\begin{proof}
	Let $\tau\in\Hom_{\CS_n}(J,I)$ be an increasing bijection.
	We have $D(\tau\circ\sigma)=D(\sigma)$ and
$$\{\sigma''\in\End_{\CS_n}(I)\ |\ \sigma''<\tau\circ\sigma,\ \ell(\sigma'')=
	\ell(\tau\circ\sigma)-1\}=
	\{\tau\circ\sigma'\ |\ \sigma'\in\Hom_{\CS_n}(I,J),\ \sigma'<\sigma,\ \ell(\sigma')=\ell(\sigma)-1\}$$
by Lemma \ref{le:tauBruhat}.
	Since the first statement of the
	lemma holds for $\tau\circ\sigma$ by Lemma \ref{le:lengthoneaffine},
	it holds for $\sigma$.

	The other statements follow from Lemmas \ref{le:classesWeyl} and
	\ref{le:degreelength}.
\end{proof}

\begin{lemma}
	\label{le:middleHeckecategory}
	Consider $\sigma''\in\Hom_{\CS_n}(I,J)$ and
	$\sigma'\in\Hom_{\CS_n}(J,K)$ and let $\sigma=\sigma'\sigma''$.
	Assume $\ell(\sigma)=\ell(\sigma')+\ell(\sigma'')$.

	Let $(i_1,i_2)\in D(\sigma)\setminus (D(\sigma)\cap D(\sigma''))$.
	Let $\alpha''=\sigma'' s_{i_1,i_2}$ and
	$\alpha''=(\sigma')^{\sigma''(i_1),\sigma''(i_2)}$.
	We have $\sigma=\alpha'\alpha''$ and
	$\ell(\sigma)=\ell(\alpha')+\ell(\alpha'')$.
\end{lemma}

\begin{proof}
	Assume first $I=J=K$.
	The lemma follows in that case from Lemmas \ref{le:lengthoneaffine} and \ref{le:middleaffineSn}.

	\smallskip
	Consider now the general case.
	There are increasing bijections $\tau:J\to I$ and $\tau':K\to J$.
	We have $D(\sigma)=\tau^{-1}(D(\tau'\sigma\tau))$ and 
	$D(\sigma'')=\tau^{-1}(D(\sigma''\tau))$ (proof of Lemma 5.4.7).
	The lemma follows now from the previous case
	applied to the decomposition
	$\tau'\sigma\tau=(\tau'\sigma')(\sigma''\tau)$.
\end{proof}

Consider $\sigma\in\Hom_{\CH_n}(I,J)$ non-zero. 
We put
$$d(\sigma)=\sum_{(i_1,i_2)\in \tilde{D}(\sigma)}\sigma^{i_1,i_2}\in
\Hom_{\BF_2[\CH_n]}(I,J)\indexnot{ds}{d(\sigma)}.$$

\begin{prop}
	\label{pr:Heckedg}
	The maps $d$ equip
the $\BF_2$-linear $\Gamma_n$-graded category $\BF_2[\CH_n]$ with a differential 
	$\Gamma_n$-graded structure,
hence equip $\CH_n$ with a differential $\Gamma_n$-graded pointed structure.

	Given $I\subset\BZ/n$, the morphism $F_I$ induces
an isomorphism of differential $\BZ$-graded pointed monoids
	$$\hat{\GS}_{|I|}^{\nil}\iso \End_{\CH_n}(I).$$
\end{prop}

\begin{proof}
	Note that Lemma \ref{le:setDforSn} shows that $d$ is homogeneous of degree $1$.
	The compatibility of $d$ with $F_I$ follows from Lemma \ref{le:lengthoneaffine}.
	
Consider now $\sigma\in\Hom_{\CH_n}(I,J)$ non-zero. There exists
	$\tau\in\Hom_{\CH_n}(J,I)$ with $\ell(\tau)=0$. We have
	$d(\tau\circ\sigma)=\tau\circ d(\sigma)$, hence
	$d^2(\tau\circ\sigma)=\tau\circ d^2(\sigma)$. The compatibility of $F_I$ with $d$
	shows that $d^2(\tau\circ \sigma)=0$. Since $\tau$ is invertible, we deduce
	that $d^2(\sigma)=0$.

	Consider finally $\sigma'\in\Hom_{\CH_n}(J,K)$ and fix 
	$\tau'\in\Hom_{\CH_n}(K,J)$ with $\ell(\tau')=0$.
	We have $d(\tau'\circ\sigma'\circ\sigma\circ\tau)=
	\tau'\circ d(\sigma'\circ\sigma)\circ\tau$ and
	it follows from the compatibility of $F_J$ with $d$ that
	\begin{align*}
		d(\tau'\circ\sigma'\circ\sigma\circ\tau)&=
	F_J\bigl(d(F_J^{-1}(\tau'\circ\sigma'\circ\sigma\circ\tau))\bigr)=
	F_J\bigl(d(F_J^{-1}(\tau'\circ\sigma')\circ F_J^{-1}(\sigma\circ\tau))\bigr)\\
		&=F_J\bigl(d(F_J^{-1}(\tau'\circ\sigma'))\circ
		F_J^{-1}(\sigma\circ\tau)\bigr)+
		F_J\bigl(F_J^{-1}(\tau'\circ\sigma'))\circ d(F_J^{-1}(\sigma\circ\tau)\bigr)\\
		&=d(\tau'\circ\sigma')\circ\sigma\circ\tau+\tau'\circ\sigma'\circ
		d(\sigma\circ\tau).
	\end{align*}
\end{proof}

\begin{example}
	Elements of $\tilde{L}(\sigma)$ correspond to intersections in a representing
	diagram. Given $(i_1,i_2)\in \tilde{L}(\sigma)$, the
	element $\sigma^{i_1,i_2}$ correspond to the diagram obtained by smoothing the
	intersection point corresponding to $(i_1,i_2)$. If $(i_1,i_2){\not\in}
	\tilde{D}(\sigma)$, the element associated to the diagram will vanish in
	$\CH_n$.
$$\includegraphics[scale=0.5]{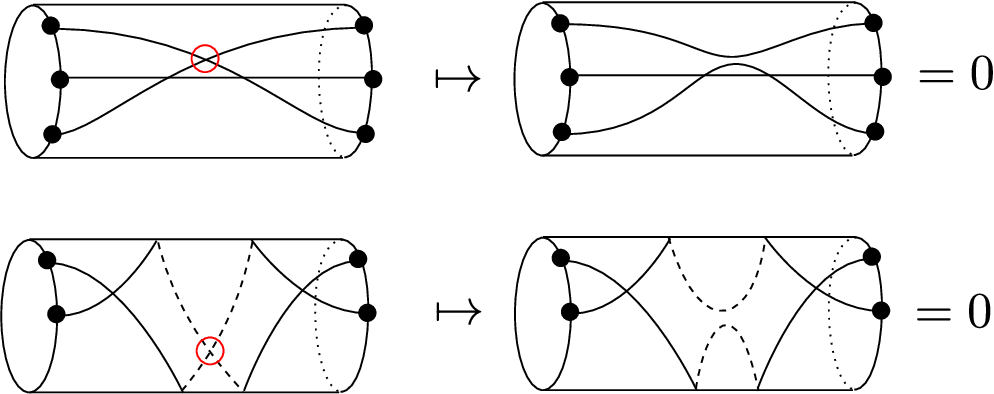}$$
\end{example}

\subsubsection{Change of $n$}

Fix a positive integer $n'\le n$ and an increasing injection
$\alpha:\{1,\ldots,n'\}\hookrightarrow\{1,\ldots,n\}$.
We extend $\alpha$ to an increasing injection $\BZ\to\BZ$ by
$\alpha(r+dn')=\alpha(r)+dn$ for $r\in\{1,\ldots,n'\}$ and $d\in\BZ$.

\medskip
Consider $\alpha:\{1,\ldots,n'\}\hookrightarrow\{1,\ldots,n\}$ an increasing injection as
in \S\ref{se:Sncat}. We define two injective morphisms of groups
$$R_\alpha:R_n\to R_{n'},\ \alpha_{i+n'\BZ}\mapsto\alpha_{\alpha(i),\alpha(i+1)}
\text{ and }
L_\alpha:L_n\to L_{n'},\ \eps_{i+n'\BZ}\mapsto\eps_{i+n\BZ}$$
for $1\le i\le n'$.
We have commutative diagrams
$$\xymatrix{
	R_{n'}\ar[r]^-{\rho}\ar[d]_{R_\alpha} & L_{n'}\ar[d]^{L_\alpha} &
	R_{n'}\times L_{n'}\ar[r]^-{\cdot} \ar[d]_{R_\alpha\times L_\alpha} & 
	  L_{n'}\ar[d]^{L_\alpha} &
	R_{n'}\times R_{n'}\ar[r]^-{\langle -,-\rangle} \ar[d]_{R_\alpha\times R_\alpha} & 
	  L_{n'}\ar[d]^{L_\alpha} 
	  \\
	R_n\ar[r]_-{\rho} & L_n & 
	R_n\times L_n\ar[r]_-{\cdot} & L_n &
	R_n\times R_n\ar[r]_-{\langle -,-\rangle} & L_n &
}$$

As a consequence, we have two injective morphisms of groups
$$\Gamma'_\alpha=L_\alpha\times R_\alpha:\Gamma'_{n'}\to\Gamma'_n\text{ and }
\Gamma_\alpha=\id\times L_\alpha\times R_\alpha:\Gamma_{n'}\to\Gamma_n,$$
the last of which induces an injective morphism of groups
$\Gamma_D\to\Gamma_{(\alpha\times\id)(D)}$, for 
$D$ a subset of $\{1,\ldots,n'\}\times\{\pm 1\}$ that embeds in its
projection on $\{1,\ldots,n'\}$.

\medskip
We define now a fully faithful functor $F=F_\alpha:\CS_{n'}\to\CS_n$\indexnot{Fa}{F_\alpha}.
Given $I$ a subset of $\BZ/n'$, we define $F(I)$ to be the image of
$\alpha(\tilde{I}\cap[1,n'])$ in $\BZ/n$.
Given $\sigma\in\Hom_{\CS_{n'}}(I,J)$,
we put $F(\sigma)=\alpha\circ \sigma\circ\alpha^{-1}$.

Note that the isomorphism of groups $\hat{\GS}_{n'}=\End_{\CS_{n'}}(\BZ/n')\iso
\End_{\CS_n}(F(\BZ/n'))$ induced by $F$ coincides with $F_{F(\BZ/n')}$ defined
in \S\ref{se:Sncat}.

\smallskip

As a consequence, $F_\alpha$ induces a fully faithful graded functor
$\CH_{n'}\to\CH_n$.

\begin{lemma}
Given $n'\le n$ and $\alpha:\{1,\ldots,n'\}\hookrightarrow\{1,\ldots,n\}$ an
	increasing injection, the functor $F_\alpha$ induces a differential
	$\Gamma_n$-graded pointed functor $\CH_{n'}\to \CH_n$.
\end{lemma}
	
\begin{proof}
		Let $\sigma\in\Hom_{\CS_n}(I,J)$.
	We have $L(F_\alpha(\sigma))=(\alpha\times\alpha)(L(\sigma))$, hence
	$\ell(F_\alpha(\sigma))=\ell(\sigma)$. We have $R_\alpha(\llbracket\sigma\rrbracket)=
	\llbracket F_\alpha(\sigma)\rrbracket$, hence
	$L_\alpha(m(\sigma))=m(F_\alpha(\sigma))$. We deduce that
	$\Gamma_\sigma(\deg(\sigma))=\deg(F_\alpha(\sigma))$.

	We have $D(F_\alpha(\sigma))=(\alpha\times\alpha)(D(\sigma))$ and
	$F_\alpha(s_{i_1,i_2})=s_{\alpha(i_1),\alpha(i_2)}$ for $i_1,i_2\in\tilde{I}$
	with $i_1-i_2{\not\in}n\BZ$, hence
	$F_\alpha$ is compatible with $d$.
\end{proof}

\subsection{Positive and finite variants}
\subsubsection{Constructions}
We define now positive and finite variants of the categories.

\smallskip
We define  $\hat{\GS}_n^{++}$\indexnot{Sn+}{\hat{\GS}_n^{++}}
to be the submonoid of $\hat{\GS}_n$ of elements
$\sigma$ such that $\sigma(r)\ge r$ for all $r\in\BZ$.

\smallskip
Let $?\in\{+,++\}$.
We define
$\CS_n^?$\indexnot{Sn+}{\CS_n^+,\ \CS_n^{++}} to be the $\Gamma_n$-filtered
subcategory of $\CS_n$ with same objects as $\CS_n$ and with
maps those $\sigma\in\Hom_{\CS_n}(I,J)$ such that $\sigma(r)>0$ if $?=+$
(resp. $\sigma(r)\ge r$ if $?=++$) for all $r\in \tilde{I}\cap\BZ_{>0}$.
We define $\CH_n^?$\indexnot{Hn+}{\CH_n^+,\ \CH_n^{++}}
as the $\Gamma_n$-graded pointed subcategory of
$\CH_n$ with same objects as $\CH_n$ and non-zero maps those of $\CS_n^?$. Note
that there is a canonical isomorphism of $\Gamma_n$-graded pointed categories
$\gr\CS_n^?\iso \CH_n^?$.

\smallskip
Note that the usual symmetric group
$\GS_n$ identifies with the subgroup of $\hat{\GS}_n$ of
elements $\sigma$ such that $\sigma(\{1,\ldots,n\})=\{1,\ldots,n\}$.
The subalgebra of $\hat{H}_n$ generated by $T_1,\ldots,T_{n-1}$ is isomorphic to $H_n$.

We denote by $\CS_n^f$\indexnot{Snf}{\CS_n^f} the $\Gamma_n$-filtered subcategory of 
$\CS_n$ with same objects as $\CS_n$ and with
maps those $\sigma\in\Hom_{\CS_n}(I,J)$ such that $\sigma(r)\in\{1,\ldots,n\}$
for all $r\in \tilde{I}\cap\{1,\ldots,n\}$. We denote by $\CH_n^f$\indexnot{Hnf}{\CH_n^f}
the corresponding
$\Gamma_n$-graded pointed subcategory of $\CH_n$. There is a canonical isomorphism of
$\Gamma_n$-graded pointed categories $\gr\CS_n^f\iso \CH_n^f$.

\smallskip
We have also subcategories $\CS_n^{f++}=\CS_n^f\cap \CS_n^{++}$\indexnot{Snf+}{\CS_n^{f++}}
of $\CS_n$ and $\CH_n^{f++}=\CH_n^f\cap\CH_n^{++}$\indexnot{Hnf+}{\CH_n^{f++}} of $\CH_n$.

\begin{lemma}
	$\CH_n^f$, $\CH_n^+$, $\CH_n^{++}$ and $\CH_n^{f++}$ are differential
$\Gamma_n$-graded pointed subcategories of $\CH_n$.
\end{lemma}

\begin{proof}
	Let $\sigma\in\Hom_{\CH_n^f}(I,J)$. There is $\tau\in\Hom_{\CH_n^f}(J,I)$
	with $\ell(\tau)=0$. We have $d(\tau\circ\sigma)=\tau\circ d(\sigma)$.
	The isomorphism $\hat{H}_n\iso\End_{\BF_2[\CH_n]}(\BZ/n)$ 
	given by Proposition \ref{pr:Heckedg} restricts to an isomorphism of
	differential graded algebras $H_n\iso\End_{\BF_2[\CH_n^f]}(\BZ/n)$. It
	follows that
	$d(\tau\circ\sigma)\in \BF_2[\CH_n^f]$, hence $d(\sigma)\in \BF_2[\CH_n^f]$. So,
	$\BF_2[\CH_n^f]$ is a differential subcategory of $\BF_2[\CH_n]$.

	One shows similarly that $\BF_2[\CH_n^+]$ is a differential subcategory
	of $\BF_2[\CH_n]$.

	\smallskip
	Let $\sigma\in\Hom_{\CH_n^{++}}(I,J)$.
	Let $(i_1,i_2)\in D(\sigma)$ and let $\sigma'=\sigma^{i_1,i_2}$.
	Given $i\in\tilde{I}$, we have $\sigma'(i)=\sigma(i)$ if
	$i\not\in (i_1+n\BZ)\cup (i_2+n\BZ)$, while
	$$\sigma'(i_1)=\sigma(i_2)\ge i_2>i_1 \text{ and }
	\sigma'(i_2)=\sigma(i_1)>\sigma(i_2)\ge i_2.$$
	It follows that $\sigma'\in\Hom_{\CH_n^{++}}(I,J)$, hence
	$d(\sigma)\in \BF_2[\CH_n^{++}]$.
\end{proof}

We extend all previous constructions to the case $n=0$ by setting
$\hat{\GS}_0=\hat{\GS}_0^{++}=\GS_0=1$, $H_0=\hat{H}_0=\BF_2$,
$\CS_0=\CS_0^{++}=\CS_0^f$ is the category with one object $\emptyset$ and one map
and $\CH_0=\CH_0^f=\CH_0^{++}$ is its associated pointed category.

\medskip
	Let
$R^f_n=\bigoplus_{a\in\BZ/n,\ a\neq -1}\BZ\alpha_a$\indexnot{Rf}{R^f_n}.

	Let $\Gamma^f_n=\{(r,(l,\alpha))\ |\ \alpha\in R^f_n\}$,
\indexnot{Gf}{\Gamma_n^f} a subgroup
of $\Gamma_n$. Given $D$ as above, we denote by 
$\Gamma_D^f$\indexnot{GDf}{\Gamma^f_n} the image
of $\Gamma^f_n$ in $\Gamma_D$.

Given $\sigma$ a map in $\CS_n^f$, we have $\deg(\sigma)\in\Gamma^f_n$.
This shows that the $\Gamma_n$-gradings on 
$\CH_n^f$ and $\CH_n^{f++}$ come from $\Gamma_n^f$-gradings.

\subsubsection{Lipshitz-Ozsv\'ath-Thurston's strands algebras}
\label{se:strandsalgebras}

Fix $n\ge 1$.
The differential algebra\indexnot{A(}{\CA(n)}
$$\CA(n)=\End_{\add(\BF_2[\CH_n^{f++}])}(\bigoplus_{I\subset\BZ/n}I)$$
is the opposite of the strands algebra $\CA_{LOT}(n)$ with $n$ places of
\cite[Definition 3.2]{LiOzTh1}.

There is a grading on $\CA_{LOT}(n)$ by a group $G'(n)$
\cite[\S 3.3.1]{LiOzTh1}. This gives rise to a grading by
$G'(n)^{\opp}$ on $\CA(n)$.

The group $G'(n)^\opp$ identifies with the index $2$ subgroup
$\ker\epsilon\cap \Gamma_{[1,n]^+}^f$
of $\Gamma_{[1,n]^+}^f$ via $(r,\alpha)\mapsto (-r,-\alpha)$
(cf Remark \ref{re:quotientform} and the identification of
$\Gamma_{[1,n]^+}$ with the set
$\frac{1}{2}\BZ\times R_n$ before Lemma \ref{le:degepsilon}). Via
this isomorphism, the $G'(n)^{\opp}$-grading on $\CA(n)$ comes from 
our $\Gamma_{[1,n]^+}^f$-grading on $\CH_n^{f++}$.

\section{Strand algebras}
\label{se:strandalgebras}
\subsection{$1$-dimensional spaces}
\label{se:1dimensionalspaces}

\subsubsection{Definitions}
\label{se:1dimspaces}

A manifold is defined to be a topological manifold with boundary
with finitely many connected components, all of which have the same dimension. 
A $1$-dimensional manifold is
a finite disjoint union of copies of $S^1$, $\BR$, $\BR_{\ge 0}$ and $[0,1]$.

\smallskip
Given a point $x$ of a topological space $X$, we put
$C(x)=C_X(x)=\lim_U\pi_0(U-\{x\})$\indexnot{Cx}{C(x)}, where $U$ runs over the
set of open neighbourhoods of $x$.
If $X'$ is a subspace of $X$ containing an open neighbourhood of $x$,
then we have a canonical bijection $C_{X'}(x)\iso C_X(x)$
and we identify those two sets.

We put $T(X)=\coprod_{x\in X}C(x)$\indexnot{Tx}{T(X)} and we denote by $\pt:T(X)\to X$
\indexnot{pt}{\pt} the canonical map.

\begin{defi}
	We define a {\em $1$-dimensional space}\index[ter]{$1$-dimensional space}
	to be a topological space
that is homeomorphic to the complement
of a finite set of points in a $1$-dimensional finite CW-complex, and that has no
connected component that is a point.
\end{defi}

\medskip
Given $E$ a finite subset of $S^1=\{z\in\BC\ |\ ||z||=1\}$, we put
$\St(E)=\bigcup_{e\in E}\BR_{\ge 0}e$\indexnot{St}{\St} and
$\St^\circ(E)=\St(E)-\{0\}$.
These are $1$-dimensional spaces.
Given $n\ge 1$, we put $\St(n)=\St(\{e^{2i\pi r/n}\}_{0\le r<n})$.

\medskip
Let $X$ be a $1$-dimensional space. There is a finite subset $E$ of $X$ such that
$X-E$ is homeomorphic to a finite disjoint union of copies of $\BR$.

Let $x\in X$.
If $U$ is a small enough connected open neighbourhood of $x$, then there is a
homeomorphism $U\iso \St(n_x),\ x\mapsto 0$ for some $n_x=n_{x,X}\ge 1$\indexnot{nx}{n_x}.
In addition,
we have a canonical bijection $C(x)\iso \pi_0(U-\{x\})$ and we identify those two sets
of cardinality $n_x$.

\smallskip

We define the boundary $\partial X=\{x\in X\ |\ n_x=1\}$.
We put $X_{exc}=\{x\in X | n_x\ge 3\}$\indexnot{Xexc}{X_{exc}}.

\begin{defi}
	We say $X$ is {\em non-singular}\index[ter]{non-singular $1$-dimensional space}
	if $X_{exc}=\emptyset$. Note that
$X-X_{exc}$ is a non-singular $1$-dimensional space.
\end{defi}

A $1$-dimensional space is non-singular if and only if it is a $1$-dimensional
manifold.

\begin{defi}
	We say that an open neighbourhood $U$ of $x\in X$ is {\em small}
	\index[ter]{small open neighbourhood} if 
it is homeomorphic to $\St(n_x)$, if
$|\overline{U}-U|=n_x$ and if $n_{x'}=2$ for all $x'\in\overline{U}-\{x\}$.
\end{defi}

Note that every point of a $1$-dimensional space admits a small open neighbourhood.

\subsubsection{Morphisms}
Let $X'$ be a $1$-dimensional space and let $f:X\to X'$ be a continuous map.
Let $X'_f$\indexnot{Xf}{X_f}
be the set of points $x'\in X'$ such that there is no open neighbourhood
$U$ of $x'$ with the property that $f_{|f^{-1}(U)}:f^{-1}(U)\to U$ is a 
homeomorphism. Let $X_f=f^{-1}(X'_f)$.

\begin{lemma}
	\label{le:equivmorphism1space}
	The following conditions are equivalent:
	\begin{enumerate}
		\item there is a finite subset $E_1$ of $X$ such that $f(X-E_1)$ is
			open in $X'$ and $f_{|X-E_1}:X-E_1\to f(X-E_1)$ is a 
			homeomorphism
		\item $X_f$ is finite
		\item there is a finite subset $E_2$ of $X$ such that
			$f_{|X-E_2}:X-E_2\to f(X-E_2)$ is a homeomorphism
		\item given $x\in X$, there is a finite subset $E_x$ of $X-\{x\}$
			such that $f_{|X-E_x}$ is injective
		\item there is a finite subset $E_3$ of $X$ such that $f_{|X-E_3}$
			is injective.
	\end{enumerate}
\end{lemma}

\begin{proof}
	The implication $(1)\Rightarrow(2)$ follows from the fact that
	$X_f\subset f^{-1}(f(E_1))$. For the implication 
	$(2)\Rightarrow(3)$, take $E_2=X_f$. For $(3)\Rightarrow(4)$, take
	$E_x=(X-\{x\})\cap (f^{-1}(f(x))\cup E_2)$.
	The implication $(4)\Rightarrow(5)$ is immediate.

	\smallskip
	Let us show that $(5)\Rightarrow(1)$. 
	Note first that an injective continuous map $\BR\to\BR$ is open
	and a homeomorphism onto its image. It follows that the implication holds
	when $X$ and $X'$ are homeomorphic to $\BR$ and $E_3=\emptyset$.
	
	Consider now the general case. There is a finite subset $E_1$ of $X$ 
	containing $E_3$ such that $X-E_1$ and $X'-f(E_1)$ are homeomorphic to
	a finite disjoint union of copies of $\BR$. By the discussion above, the
	restriction of $f$ to a connected component of $X-E_1$ is open and
	a homeomorphism onto its image, so the same holds for $f_{|X-E_1}$.

\end{proof}

\begin{defi}
We say that $f$ is a 
	{\em morphism of $1$-dimensional spaces}\index[ter]{morphism of $1$-dimensional
	spaces} if it satisfies any of the equivalent
conditions of Lemma \ref{le:equivmorphism1space}.
\end{defi}

Note that 
\begin{itemize}
	\item a composition of morphisms of $1$-dimensional
spaces is a morphism of $1$-dimensional spaces
\item a morphism of $1$-dimensional spaces is invertible if and only if it is
	a homeomorphism.
\end{itemize}

\begin{defi}
We define a {\em $1$-dimensional subspace}\index[ter]{$1$-dimensional subspace} of $X$
	to be a subspace $Y$ with only
finitely many connected components, none of which are points.
\end{defi}

\smallskip
Let us record some basic facts on subspaces.

\begin{lemma}
	\label{le:subspaces}
	\begin{enumerate}
		\item
The image of a morphism of $1$-dimensional spaces is a $1$-dimensional subspace.
		\item
If $Y$ is a $1$-dimensional subspace of $X$, then $Y$ is a $1$-dimensional space and
the inclusion map $Y\hookrightarrow X$ is a morphism of $1$-dimensional spaces.
		\item
Let $f:X\to X'$ be a morphism of $1$-dimensional spaces and
$Y'$ be a $1$-dimensional subspace of $X'$. Let $F$ be the set of connected components
of $f^{-1}(Y')$ that are points. Then $F$ is finite, $Y=f^{-1}(Y')-F$ is
a $1$-dimensional subspace of $X$ and $f_{|Y}:Y\to Y'$
is a morphism of $1$-dimensional spaces.
	\end{enumerate}
\end{lemma}

We now provide a description of the local structure of morphisms of
$1$-dimensional spaces.

\begin{lemma}
	\label{le:localstructuremaps}
	Let $f:X\to X'$ be a morphism of $1$-dimensional spaces and let $x'\in X'$.
	Let $r=|f^{-1}(x')|$.
	There exists
	\begin{itemize}
		\item a small open neighbourhood $U$ of $x'$ and a homeomorphism 
			$a:\St(n_{x'})\iso U$ with $a(0)=x'$,
	\item a family of disjoint subsets $I_0,I_1,\ldots,I_r$
		of $\{e^{2i\pi d/n_{x'}}\}_{0\le d<n_{x'}}$ with $I_l\neq\emptyset$ for
			$1\le l\le r$
	and a homeomorphism
	$b:\St^\circ(I_0)\sqcup
			\St(I_1)\sqcup\cdots\sqcup\St(I_r)\iso f^{-1}(U)$
	\end{itemize}
	such that
	$f_{|f^{-1}(U)}=a\circ g\circ b^{-1}$ where 
	$g:\St^\circ(I_0)\sqcup\St(I_1)\sqcup\cdots\sqcup\St(I_r)\to
	\St(n_{x'})$ is the map whose restriction to $\St^\circ(I_0)$ and $\St(I_l)$ 
	is the inclusion map.

	In particular, the canonical map, still denoted by
	$f:T(X)\to T(X')$ is injective and $f(X_{exc})\subset X'_{exc}$.
\end{lemma}

\begin{proof}
	Let $E$ be a finite subset of $X$ such that $f^{-1}(f(E))=E$,
	$f(X-E)$ is open in $X'$ and $f_{|X-E}:X-E\to f(X-E)$ is a
	homeomorphism.
	Let $U$ be a small open neighbourhood of $x'$ such that
	$U-\{x'\}\subset X'-f(E)$. Note that
	$f(X)\cap (U-\{x'\})$ is open in $X'$ and
	$f_{|f^{-1}(U-\{x'\})}:f^{-1}(U-\{x'\})\to f(X)\cap (U-\{x'\})$ is a
	homeomorphism.

	\smallskip
	Let $L$ be a connected component of $U-\{x'\}$.
	Note that $f(f^{-1}(L))$ is an open $1$-dimensional subspace of $L$ and
	$L$ is homeomorphic to $\BR$.
	By shrinking $U$, we can assume
	that $f^{-1}(L)=\emptyset$ or $f(f^{-1}(L))=L$. So, we can assume that
	given $L$ a connected component of $U-\{x'\}$ with $f^{-1}(L)\neq\emptyset$, 
	the map $f_{|f^{-1}(L)}:f^{-1}(L)\to L$ is a homeomorphism.

	\smallskip
	Since $U$ is small, there is a homeomorphism 
	$a:\St(n_{x'})\iso U,\ 0\mapsto x'$. Let $\{x_1,\ldots,x_r\}=f^{-1}(x')$ and define
	$$I_l=\{e^{2i\pi d/n_{x'}} | 0\le d<n_{x'},\
	x_l\in \overline{f^{-1}(a(\BR_{>0}e^{2i\pi d/n_{x'}}))}\}$$
	for $l\in\{1,\ldots,r\}$.
	Define 
	$$I_0=\{e^{2i\pi d/n_{x'}} | 0\le d<n_{x'},\ 
	f^{-1}(a(\BR_{>0}e^{2i\pi d/n_{x'}}))\neq \emptyset,\ 
	f^{-1}(x')\cap\overline{f^{-1}(a(\BR_{>0}e^{2i\pi d/n_{x'}}))}=\emptyset\}.$$
	Note that $a$ restricts to a homeomorphism
	$\St(\bigcup_{0\le l\le r}I_r)\iso f(f^{-1}(U))$.

	The composition $a\circ g$ takes values in $f(f^{-1}(U))$.
	Its restriction to $\St^\circ(I_0)$ defines a homeomorphism
	$\St^\circ(I_0)\iso a(\St^\circ(I_0))$. Since 
	$f_{|f^{-1}(a(\St^\circ(I_0)))}:f^{-1}(a(\St^\circ(I_0)))\to
	a(\St^\circ(I_0))$ is a homeomorphism, we have a
	homeomorphism $b_0=(f_{|f^{-1}(a(\St^\circ(I_0)))})^{-1}\circ 
	(a\circ g)_{|\St^\circ(I_0)}:\St^\circ(I_0)\iso f^{-1}(a(\St^\circ(I_0)))$.

	Consider now $l\in\{1,\ldots,r\}$. We construct as above a
	homeomorphism $b'_l:\St^\circ(I_l)\iso f^{-1}(a(\St^\circ(I_l)))$ such that
	$(a\circ g)_{|\St^\circ(I_l)}=f\circ b'_l$. The homeomorphism $b'_l$
	extends uniquely to a homeomorphism $b_l:\St(I_l)\to f^{-1}(a(\St(I_l)))$.
	We define $b=b_0\sqcup b_1\sqcup\cdots\sqcup b_r$. We have
	$f_{|f^{-1}(U)}=a\circ g\circ b^{-1}$.
\end{proof}

\begin{example}
	Here is an example of map $g$ as in Lemma \ref{le:localstructuremaps}:
$$\includegraphics[scale=0.92]{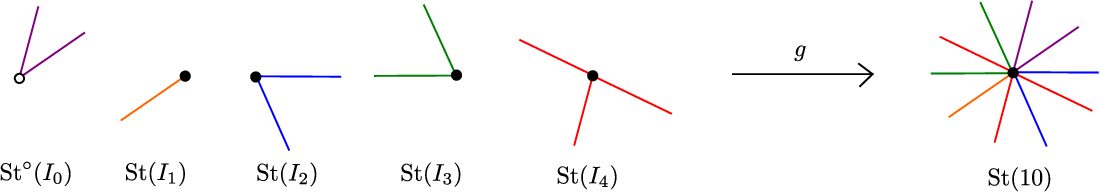}$$
\end{example}

The next two results follow immediately from Lemma \ref{le:localstructuremaps}.

\begin{lemma}
	Let $Y$ be a $1$-dimensional subspace of $X$ and let $y\in Y$.
	Let $I=\{e^{2i\pi d/n_{y,X}}\}_{0\le d<n_{y,Y}}$. There is an open
	neighbourhood $U$ of $y$ in $X$ and a homeomorphism
	$\St(n_{y,X})\iso U,\ 0\mapsto y$ whose restriction to $\St(I)$ is a
	homeomorphism $\St(I)\iso U\cap Y$. We have a commutative diagram
	$$\xymatrix{
		\St(n_{y,X})\ar[r]^-\sim & U \\
		\St(I)\ar@{^{(}->}[u] \ar[r]_-\sim & U\cap Y \ar@{^{(}->}[u]
		}$$
\end{lemma}

\begin{lemma}
Let $f:X\to X'$ be a surjective morphism of $1$-dimensional spaces.
	It induces a bijection $T(X)\iso T(X')$.
\end{lemma}



\subsubsection{Quotients}
\label{se:quotientsspaces}

Let $\tilde{X}$ be a $1$-dimensional space and $\sim$ be an equivalence relation on
$\tilde{X}$.

\begin{defi}
	We say that $\sim$ is a {\em finite relation}\index[ter]{finite relation} if
the set of points that are not alone in their equivalence class is finite.
\end{defi}

Assume $\sim$ is a finite relation.
Let $q:\tilde{X}\to X=\tilde{X}/\!\!\sim$ be the quotient map. Note that $X$ is a
$1$-dimensional 
space with 
$$X_{exc}=q(\tilde{X}_{exc})\cup \{x\in X|\ |q^{-1}(x)|>2\}\cup
\{x\in X\ |\ |q^{-1}(x)|=2,\ q^{-1}(x){\not\subset}\partial\tilde{X}\}$$
and
$q$ is a morphism of $1$-dimensional spaces.

Given $x\in X$, the quotient map induces a bijection
$q:\coprod_{\tilde{x}\in q^{-1}(x)}C(\tilde{x})\iso C(x)$.

\medskip

Quotients have a universal property. In particular, we have the following result.

\begin{lemma}
	\label{le:quotientsfactor}
Let $f:X\to X'$ be a morphism of $1$-dimensional spaces. Define an equivalence
	relation on $X$ by $x_1\sim x_2$ if $f(x_1)=f(x_2)$. This defines a finite
	relation on $X$ and $f$ factors uniquely as a composition 
	$f=\bar{f}\circ q$ where $\bar{f}:X/\!\!\sim\ \to X'$ is a morphism of
	$1$-dimensional spaces and $q:X\to X/\!\!\sim$ is the quotient map.
\end{lemma}

The next lemma shows that $1$-dimensional spaces $X$ can be viewed (non-uniquely)
as $1$-dimensional manifolds with a finite relation.

\begin{lemma}
	\label{le:1manifoldsquotients}
	Given $X$ a $1$-dimensional space, there is a $1$-dimensional
	manifold $\hat{X}$ with a finite relation $\sim$
	and an isomorphism $f:\hat{X}/\!\!\sim\ \iso X$
	such that $f(\hat{X}_f)=X_{exc}$.
\end{lemma}

\begin{proof}
Fix, for every $x\in X_{exc}$, a small
open neighbourhood $U_x$ of $x$ and a homeomorphism
$f_x:U_x\iso \mathrm{St}(E_x)$, where $E_x$ is a finite subset of $S^1$.
We choose now an equivalence relation on $E_x$ whose classes have cardinality
	at most $2$. Note that $f_x$ induces a bijection between $C(x)$ and
	$E_x$, hence the equivalence relation can be viewed on $C(x)$.

	Define $\hat{U}_x =\coprod_{E'\in E_x/\!\sim}\St(E')$.
	The map $f_x$ provides an open embedding 
	$$U_x-\{x\}\iso \mathrm{St}^\circ(E_x)
	\iso\coprod_{E'\in E_x/\!\sim}\St^\circ(E')\hookrightarrow \hat{U}_x.$$
We put 
$$\hat{X}=(X-X_{exc})\coprod_{(\coprod_{x\in X_{exc}}(U_x-\{x\}))}
\bigl(\coprod_{x\in X_{exc}}
\hat{U}_x\bigr).$$
Note that $\hat{X}$ is a $1$-dimensional manifold.
Let $q:\hat{X}\to X$ be the canonical map: it
identifies $X$ with the quotient
of $\hat{X}$ by the equivalence relation given by $\hat{x}_1\sim\hat{x}_2$
if $q(\hat{x}_1)=q(\hat{x}_2)$.
	Up to isomorphism, $\hat{X}$ depends only on the choice of 
	an equivalence relation on $C(x)$ for $x\in X_{exc}$.
\end{proof}

\subsubsection{Paths}
\label{se:paths}

\begin{lemma}
	\label{le:nullhomotopic}
	Let $E$ be a finite subset of $X$ and $\gamma$ be a path in $X$ such that
	for all connected components $I$ of $[0,1]\setminus\gamma^{-1}(E)$, the restriction
	of $\gamma$ to $\bar{I}$ is nullhomotopic. Then $\gamma$ is nullhomotopic.
\end{lemma}

\begin{proof}
Given $e\in E$, let $U_e$ be a connected and simply connected open neighborhood of $e$.
Choose $U_e$ small enough so that $U_e\cap U_{e'}=\emptyset$ for $e\neq e'$.
Let $U=\bigcup_{e\in E}U_e$. Let $V$ be an open subset of $X\setminus E$ containing
$X\setminus U$.

Let $C$ be the set of connected components $I$ of $[0,1]\setminus\gamma^{-1}(E)$ such that
		$\bar{I}$ is not contained in $\gamma^{-1}(U)$ nor in $\gamma^{-1}(V)$.
	By Lebesgue's
number Lemma, that set is finite. Since the restriction of
	$\gamma$ to $\bar{I}$ is nullhomotopic for
$I\in C$, it follows that $\gamma$ is homotopic to a path $\gamma'$ that is
constant on $\bar{I}$ for $I\in C$ and that coincides with $\gamma$ on
$[0,1]-\bigcup_{I\in C}I$. Let $I'$ be a connected component of 
	$[0,1]\setminus\gamma^{-1}(E)$ with $I'{\not\in}C$. We have 
	$\bar{I}\cap\gamma^{-1}(E)\neq\emptyset$, hence $\bar{I}\subset\gamma^{-1}(U)$.
	We deduce that $\gamma'([0,1])\subset U$, hence $\gamma'$ is nullhomotopic.
\end{proof}

\begin{lemma}
	\label{le:componentspath}
Let $E$ be a finite subset of $X$ and $\gamma$ a path in $X$. 
Let $B$ be the set of connected components $I$ of
$[0,1]\setminus\gamma^{-1}(E)$ such that $\gamma_{|\bar{I}}$ is not nullhomotopic.
Then $B$ is finite and
there are paths $\gamma'$ and $\gamma''$ homotopic to $\gamma$ such that
	\begin{itemize}
		\item $\gamma$ and $\gamma'$ coincide on $\bigcup_{I\in B}\bar{I}$ and
			$\gamma'([0,1]\setminus \bigcup_{I\in B}\bar{I})\subset E$
		\item $\gamma^{\prime\prime -1}(E)$ is finite.
	\end{itemize}
\end{lemma}

\begin{proof}
	Let $\CU$ be an open covering of $X$ by connected and simply connected subsets,
	each of which contain at most one element of $E$.
	By Lebesgue's number Lemma, there are only finitely many $I\in
	\pi_0([0,1]\setminus\gamma^{-1}(E))$ such that $\bar{I}$ is not contained in
	an element of $\gamma^{-1}(\CU)$. So, $B$ is finite.

We can write $\gamma$ as a finite composition of its restrictions to $\bar{I}$ for 
$I \in B$ interlaced with finitely many paths that satisfy the assumptions of
Lemma \ref{le:nullhomotopic}. Thanks to that lemma, we obtain a path $\gamma'$
satisfying the requirements of the lemma. By shrinking the intervals on which
$\gamma'$ is constant to points, we obtain a path $\gamma''$ as desired.
\end{proof}

\begin{defi}
	We say that a path $\gamma$ in a $1$-dimensional space $X$ is {\em minimal}\index[ter]{minimal path} if
there is a finite covering of $[0,1]$ by open subsets such that the restriction
of $\gamma$ to any of those open subsets is injective.
\end{defi}

\smallskip
Given a continuous map $f:X\to X'$ and a path $\gamma:[0,1]\to X$, we will usually
denote by $f(\gamma)$ the path $f\circ\gamma$.

We denote by $[\gamma]$ the homotopy class of a path $\gamma$. 
Note that we always consider homotopies relative to the endpoints. 
We denote by $\Pi(X)$\indexnot{PiX}{\Pi(X)} the fundamental groupoid of $X$.

Given $x_0,x_1\in X$ such that there is a unique homotopy class of paths from $x_0$ to
$x_1$ in $X$, we denote by $[x_0\to x_1]$\indexnot{[}{[x_0\to x_1]} that homotopy class.

\smallskip
The following lemma is classical for $1$-dimensional finite CW-complexes.

\begin{lemma}
	\label{le:minimalpaths}
Let $X$ be a $1$-dimensional space. A homotopy class
	of paths in $X$ contains a minimal path if and only if it is not an identity.

Given $\gamma,\gamma'$ two homotopic minimal paths in $X$, there
	is a homeomorphism $\phi:[0,1]\iso [0,1]$ with $\phi(0)=0$ and $\phi(1)=1$
	such that $\gamma'=\gamma\circ\phi$.
\end{lemma}

\begin{proof}
		Let $\gamma_1$, $\gamma_2$ be two minimal paths in $X$ with
$\gamma_1(1)=\gamma_2(0)$. 
The path $\gamma_2\circ\gamma_1$ is minimal if and only if there are
$t_1,t_2\in (0,1)$ such that $\gamma_1((t_1,1))\cap\gamma_2((0,t_2))=\emptyset$.
If $\gamma_2\circ\gamma_1$ is not minimal, then there are unique elements
$t_1\in [0,1)$ and $t_2\in (0,1]$ such that
$(\gamma_2)_{|[0,t_2]}\circ(\gamma_1)_{|[t_1,1]}$ is 
homotopic to a constant path and $(\gamma_2)_{|[t_2,1]}\circ(\gamma_1)_{|[0,t_1]}$
is minimal (if $t_2\neq 1$ or $t_1\neq 0$).

We deduce by induction that a composition of minimal paths is homotopic to a minimal
path or to a constant path.

	\smallskip
	Let $\gamma$ be a path in $X$.
If $X$ is homeomorphic to an interval of $\BR$, then $\gamma$ is homotopic to
a minimal path or a constant path.
In general there is a finite subset $E$ of $X$ such that given $U$ a connected component
	of $X\setminus E$, the space $\bar{U}$ is homeomorphic to
	an interval of $\BR$.
By Lemma \ref{le:componentspath}
	there is a path $\gamma'$ homotopic to $\gamma$ and such that
	$\gamma^{\prime -1}(E)$ is finite. So, $\gamma'$ is a composition of
	paths contained in subspaces of $X$ that are homeomorphic to
	intervals of $\BR$. Consequently, $\gamma'$ is a composition of
	minimal paths. It follows that $\gamma'$, hence $\gamma$, is homotopic to a minimal
	or constant path.

\smallskip

	Let $\gamma$ be a path homotopic to a constant path. The image $\bar{\gamma}$
	of $\gamma$
	in $\bar{X}=X/(X_{exc}\cup\{\gamma(0),\gamma(1)\})$ is homotopic to a constant
	path. Since $\bar{X}$ is homotopy equivalent to a wedge of circles, its fundamental
	group is free and $\bar{\gamma}$ cannot be a minimal path. It follows that
	$\gamma$ is not minimal.

\smallskip

	Let $\gamma$ be a minimal path.
		Let $\{0=t_0<t_1<\ldots<t_n=1\}=\{0,1\}\cup\gamma^{-1}(X_{exc})$.
		Note that $\gamma((t_i,t_{i+1}))$ is contained in a
		connected component $U_i$ of $X\setminus X_{exc}$ and it is a connected
		component if $\gamma(t_i),\ \gamma(t_{i+1})\in X_{exc}$.
		If $\bar{U}_i$ is homeomorphic to an interval of $\BR$, then
		$U_i\neq U_{i+1}$ and 
		$U_i\neq U_{i-1}$. Otherwise, $\bar{U}_i$ is 
		homeomorphic to $S^1$ and if $U_i=U_{i+1}$, then the paths
		$\gamma_{|U_i}$ and $\gamma_{|U_{i+1}}$ have the same
		orientation.

	Let $\gamma'$ be a minimal path homotopic to $\gamma$. We will
	show the existence of $\phi$ as in the lemma by induction on $n$.
	Since $\gamma\circ \gamma^{\prime -1}$ is not minimal, there is
	$\eps>0$ such that $\gamma'([0,\eps])\subset\bar{U}_1$. Consider $\eps$
	maximal with this property. 

	Assume $\gamma'(\eps){\not\in} X_{exc}$. We have $\eps=1$. Let
	$\eps'\in (t_0,t_1]$ such that $\gamma(\eps')=\gamma'(\eps)$. The path
	$\gamma_{|[\eps',1]}$ is homotopic to the identity, hence $n=1$,
	$\eps'=1$ and $\gamma(t_1)=\gamma'(\eps)$.
	
	If $\gamma'(\eps)\in X_{exc}$, then $\gamma'(\eps)=\gamma(t_1)$ as well.
	In both cases, the paths
	$\gamma_{|[0,t_1]}$ and $\gamma'_{|[0,\eps]}$ are injective and have
	the same image. So, there is a homeomorphism
	$\psi:[0,\eps]\iso [0,t_1]$ such that
	$\gamma'(t)=\gamma(\psi(t))$ for $t\in [0,\eps]$ and the
	existence of $\psi$ follows by induction.

\end{proof}

\begin{defi}
Let $\zeta$ be a non-identity homotopy class of paths in a $1$-dimensional space
$X$. We define
	the {\em support}\index[ter]{support of a path}
	$\supp(\zeta)$\indexnot{supp}{\supp(\zeta)}
of $\zeta$ to be the subspace $\gamma([0,1])$ of $X$, where
$\gamma$ is a minimal path in $\zeta$.
\end{defi}

Lemma \ref{le:minimalpaths} ensures that the support
is well defined. Note that $\supp(\zeta)=\bigcap_\gamma \gamma([0,1])$, where
$\gamma$ runs over paths with $[\gamma]=\zeta$.

Since a minimal path $[0,1]\to X$ is a morphism of $1$-dimensional
manifolds, it follows that the support of $\zeta$ is a compact connected
$1$-dimensional subspace of $X$.

We define the support of the identity homotopy class $\id_x$ at a point $x$
to be $\{x\}$.

\begin{lemma}
	\label{le:mappaths}
	Let $f:X\to X'$ be a morphism of $1$-dimensional spaces and let
	$\gamma$, $\gamma'$ be two paths in $X$.
	
	\begin{itemize}
		\item $\gamma$ is minimal if and only if $f(\gamma)$ is minimal.
			In particular, 
			$\supp([f(\gamma)])=f(\supp([\gamma)])$.
		\item
	If $f(\gamma)=f(\gamma')$, then
	$\gamma=\gamma'$ or $\gamma$ and $\gamma'$ are constant paths at two distinct points
	of $X$ having the same image under $f$.
		\item
	If $[f(\gamma)]=[f(\gamma')]$, then $[\gamma]=[\gamma']$ or
	$[\gamma]=\id_{x_1}$ and
	$[\gamma']=\id_{x_2}$ for some $x_1\neq x_2\in X$ with $f(x_1)=f(x_2)$.
	\end{itemize}
\end{lemma}

\begin{proof}
	A minimal path is a locally injective path. Since every point of $X$ has
	an open neighbourhood on which $f$ is injective (cf Lemma 
	\ref{le:localstructuremaps}), the image by $f$ of
	a minimal path is a minimal path.

	\smallskip
	Consider the set $\Omega=\{t\in [0,1]\ |\ \gamma(t)\neq\gamma'(t)\}$, an open
	subset of $[0,1]$. Let $I$ be a connected component of $\Omega$.
	If $I=[0,1]$,
	then $\gamma$ and $\gamma'$ are constant paths at distinct points of
	$X$ with the same image under $f$. Otherwise,
	let $s\in \overline{I}-I$. There is an open neighbourhood $U$ of 
	$\gamma(s)=\gamma'(s)$ such that $f_{|U}$ is injective.
	There is $t\in I$ such that $\gamma(t)$ and $\gamma'(t)$ are in $U$,
	hence $\gamma(t)=\gamma'(t)$, a contradiction.
	This shows the second assertion of the lemma.
	
	\smallskip
	Assume $\gamma$ and $\gamma'$ are minimal. Since $f(\gamma)$
	and $f(\gamma')$ are minimal and homotopic, it follows from
	Lemma \ref{le:minimalpaths} that there is
	$\phi:[0,1]\iso [0,1]$ with $\phi(0)=0$ and
	$\phi(1)=1$ such that $f(\gamma')=f(\gamma)\circ \phi=f(\gamma\circ\phi)$.
	It follows from the previous assertion of the lemma that
	$\gamma'=\gamma\circ\phi$. 

	Assume now $\gamma$ is minimal. Since $f(\gamma)$ is minimal, it
	follows that $[f(\gamma')]$ is not the identity, hence $[\gamma']$ is
	not the identity. We deduce that the third assertion of the lemma holds
	when $[\gamma]$ and $[\gamma']$ are not both identities. The case where
	they are both identities is clear.
\end{proof}

\subsubsection{Tangential multiplicity}
\label{se:tangential}
Let $X$ be a $1$-dimensional space.
Let $x\in X$ and $U$ be a small open neighborhood of $x$.

Let $c\in C(x)$ and let $U_c$ be the connected component of $U-\{x\}$
corresponding to $c$.
Given $\gamma$ a path in $X$, let 
$I_c^+(\gamma)$ (resp. $I_c^-(\gamma)$)\indexnot{Ic}{I_c^{\pm}(\gamma)}
be the set of elements $t\in [0,1]$ such that
$\gamma(t)=x$ and there is 
$\eps>0$ with $t+\eps<1$ and $\gamma((t,t+\eps))\subset U_c$ (resp. $t-\eps>0$
and $\gamma((t-\eps,t))\subset U_c$).

\smallskip
When $\gamma$ is minimal, the set $I_c^\pm(\gamma)$ is finite and it follows from Lemma
\ref{le:minimalpaths} that its cardinality depends only
on the homotopy class $[\gamma]$. We put $m_c^\pm([\gamma])=|I_c^\pm(\gamma)|
\in\BZ_{\ge 0}$\indexnot{mc}{m_c^\pm}
for $\gamma$ minimal and $m_c([\gamma])=m_c^+([\gamma])-m_c^-([\gamma])$\indexnot{mc}{m_c}.
Similarly, whether or not $0\in I_c^+$ depends only on the homotopy
class $[\gamma]$ (for $\gamma$ minimal).


\begin{lemma}
	\label{le:multiplicityset}
Let $\gamma$ be a path in $X$ such that
$\gamma^{-1}(x)$ has finitely many connected components, none of which contain
$0$ or $1$ in the closure of their interior.

We have
$\partial(\gamma^{-1}(x))=
\bigcup_{c\in C(x)}(I_c^+(\gamma)\cup I_c^-(\gamma))$ and
$|I_c^+(\gamma)|-|I_c^-(\gamma)|=m_c([\gamma])$ for all
$c\in C(x)$.
\end{lemma}

\begin{proof}
	The first statement is clear. Let us now prove the second statement.
	That statement is clear if $\gamma((0,1))\cap (X_{exc}\cup\{x\})=\emptyset$.

	The left side of the equality is additive under compositions of paths, and so
	is the right side by Lemma \ref{le:additivitymc} below. 
	
	Assume now $\gamma^{-1}(X_{exc}\cup\{x\})$ is finite.
	The path $\gamma$ is a (finite) composition of paths mapping $(0,1)$ into the
	complement of $X_{exc}\cup\{x\}$, hence the statement holds for $\gamma$.

	Consider now the general case.
The proof of Lemma \ref{le:componentspath} for $E=X_{exc}\cup\{x\}$ produces a path $\gamma'$
	homotopic to $\gamma$ such that $\gamma^{\prime -1}(E)$ is finite and such that 
$|I_c^+(\gamma)|-|I_c^-(\gamma)|=|I_c^+(\gamma')|-|I_c^-(\gamma')|$. Since
the statement holds for $\gamma'$, it follows that it holds for $\gamma$.
\end{proof}

\medskip
Let $\zeta$ be the homotopy class of a minimal path $\gamma$.
Let $x=\zeta(0)$. There is a unique $c\in C(x)$ such that 
$0\in I_c(\gamma)^+$ and we define $\zeta(0+)=\{c\}$\indexnot{z0}{\zeta(0+)}.
Similarly, we define $\zeta(1-)=\{c'\}$\indexnot{z1}{\zeta(1-)}, where
$c'\in C(\zeta(1))$ is unique such that $1\in I_c(\gamma)^-$.

\smallskip
When $\zeta$ is the homotopy class of a constant path we put $\zeta(0+)=C(\zeta(0))$,
$\zeta(1-)=C(\zeta(1))$ and $m_c^\pm(\zeta)=m_c(\zeta)=0$.

\bigskip
%


Given a category $\CC$,
we denote by $H_0(\CC)$ the abelian group generated
by maps in $\CC$ modulo the relation $f+g=f\circ g$ for any two composable
maps $f$ and $g$. We denote by $\llbracket f\rrbracket$\indexnot{[}{\protect\llbracket f\protect\rrbracket} the class in $H_0(\CC)$ of
a map $f$ of $\CC$.
Note that if $f$ is an identity map, then $\llbracket f\rrbracket=0$.

Note that $H_0$ is left adjoint to the functor sending an abelian group to the
category with one object with endomorphism monoid that abelian group.

\bigskip
Let $R(X)=H_0(\Pi(X))$\indexnot{R}{R(X)}.
Note that $R(X)$ is generated by the set $I$ of homotopy classes
of paths $\gamma$ such that $\gamma$ is injective.
It follows from the description of the composition of two minimal
paths in \S\ref{se:paths} that $R(X)$ has a presentation with generating set the
non-identity homotopy classes of paths and relations 
$[\gamma\circ\gamma']=[\gamma]+[\gamma']$ if $\gamma$, $\gamma'$ and $\gamma\circ\gamma'$
are minimal and $[\gamma]+[\gamma^{-1}]=0$ for $\gamma$ minimal.
Note finally that every element of $R(X)$ is a linear combination of non-identity homotopy
classes of paths
such that the intersection between the supports of two distinct homotopy classes is finite.

\begin{lemma}
	\label{le:additivitymc}
	Given $c\in T(X)$, the map $m_c$ 
	induces a morphism of groups $R(X)\to\BZ$.
\end{lemma}

\begin{proof}
	Consider $\gamma$ and $\gamma'$ two injective composable paths such that
	$\gamma\circ\gamma'$ is injective.
	We have $m_c^\pm([\gamma\gamma'])=m_c^\pm([\gamma])+m_c^\pm([\gamma'])$.

	Consider now $\gamma$ a minimal path. We have
	$m_c^\pm([\gamma])=m_c^\mp([\gamma^{-1}])$, hence
	$m_c([\gamma])+m_c([\gamma^{-1}])=0=m_c([\gamma^{-1}\circ\gamma])$. The lemma
	follows.
\end{proof}

The next lemma shows how to realize $R(X)$ as a subgroup of the group of maps
$U\to\BZ$, where $U$ is a dense subset of $X$.

\begin{lemma}
	\label{le:embeddingR}
	Let $U$ be a dense subset of $X-(\partial X\cup X_{exc})$.
Given $x\in U$, fix a group morphism $l_x:\BZ^{C(x)}\to\BZ$
that does not factor through the sum map.

	The morphism $(l_x\circ (m_c)_{c\in C(x)})_{x\in U}:
	R(X)\to \BZ^U$ is injective.
\end{lemma}

\begin{proof}
	Let $L$ be a non-empty
	finite subset of $I$ such that $\supp(\zeta)\cap\supp(\zeta')$ is 
	finite for any two distinct elements $\zeta$ and $\zeta'$ in $L$.
	Let $r=\sum_{\zeta\in L}a_\zeta \llbracket\zeta\rrbracket$ where 
	$a_\zeta\in\BZ-\{0\}$ for $\zeta\in L$.
	 Let $\zeta_0\in L$. There is
	$x\in \supp(\zeta_0)\cap U$ with $x{\not\in}\{\zeta_0(0),\zeta_0(1)\}$ and
	$x{\not\in}\bigcup_{\zeta\in
	L-\{\zeta_0\}}\supp(\zeta)$.
	Let $c\in C(x)$ and $\iota(c)$ be the other element of $C(x)$.
	We have $m_c(\zeta_0)=-m_{\iota(c)}(\zeta_0)=\pm 1$, while
	$m_c(\zeta')=m_{\iota(c)}(\zeta')=0$ for $\zeta'\in L-\{\zeta_0\}$. It follows
	that $m_c(r)=-m_{\iota(c)}(r)=\pm a_\gamma$. Consequently,
	$\bigl(l_x\circ (m_c,m_{\iota(c)})\bigr)(r)=\pm l_x(a_\gamma,-a_\gamma)\neq 0$.
	Since every non-zero element of $R(X)$ is of the form $r$ as above, the
	lemma follows.
\end{proof}

Let $f:X\to X'$ be a morphism of $1$-dimensional spaces.
	The next lemma follows from
	the injectivity statement of Lemma \ref{le:localstructuremaps}.

\begin{lemma}
	\label{le:functorialitym}
	Given $x\in X$, $c\in C(X)$ and $\zeta$ a homotopy
	class of paths in $X$, we have
$m_{f(c)}^\pm(f(\zeta))=m_c^\pm(\zeta)$ and
	$m_{f(c)}(f(\zeta))=m_c(\zeta)$.
\end{lemma}

Note that $f$ induces a morphism of groups $f:R(X)\to R(X')$.

\begin{lemma}
	\label{le:finjR}
	Let $H$ be the subgroup of $R(X')$ generated by classes $[\gamma]$ with
	$\supp(\gamma)\subset \overline{X'-f(X)}$.

	The composition
	$R(X)\xrightarrow{f}R(X')\xrightarrow{\can}R(X')/H$ is injective.
\end{lemma}

\begin{proof}
	Let $U'=X'-(X'_f\cup X'_{exc}\cup \partial X')$, a dense subset of
	$X'$. Note that $U=f^{-1}(U')$ is a dense subset of
	$X-(X_{exc}\cup \partial X)$.
	Given $x'\in U'$, fix a morphism $l_{x'}:\BZ^{C(x')}\to\BZ$ that does
	not factor through the sum map. Given $x\in U$, let
	$l_x=l_{x'}\circ f:\BZ^{C(x)}\to\BZ$.
	Lemma \ref{le:embeddingR} shows that
	$(l_x\circ (m_c)_{c\in C(x)})_{x\in U}:R(X)\to \BZ^U$
	is injective. This map
	is equal to the composition
	$$R(X)\xrightarrow{f}R(X')
	\xrightarrow{(l_{x'}\circ (m_{c'})_{c'\in C(x')})_{x'\in U'}}
	\BZ^{U'}\xrightarrow{f^*}\BZ^U$$	
	since $m_{f(c)}^\pm(f(\zeta))=m_c^\pm(\zeta)$ and
	$m_{f(c)}(f(\zeta))=m_c(\zeta)$
	for all $x\in X$, $c\in C(X)$ and all homotopy classes of paths $\zeta$
	in $X$ (Lemma \ref{le:functorialitym}).
	Since $H$ is contained in the kernel of the composition
$$R(X') \xrightarrow{(l_{x'}\circ (m_{c'})_{c'\in C(x')})_{x'\in U'}}
	\BZ^{U'}\xrightarrow{f^*}\BZ^U,$$
	it follows that the composite map of the lemma is injective.
\end{proof}


Given $M$ a subset of $X$, we denote by $R_M(X)$\indexnot{RM}{R_M(X)}
the subgroup of $R(X)$ generated
by classes of paths $\gamma$ with endpoints in $M$.

\subsection{Curves}
\label{se:curves}
\subsubsection{Definitions}
We consider now partially oriented $1$-dimensional spaces. We build the theory
so that the unoriented part is a manifold, and morphisms are injective on the
unoriented part.

\begin{defi}
	We define a {\em curve}\index[ter]{curve} to be a $1$-dimensional space $Z$
endowed with
\begin{itemize}
	\item an open subset $Z_o$\indexnot{Z}{Z_o} containing $Z_{exc}$
	\item an orientation of $Z_o-Z_{exc}$ and 
	\item a fixed-point free involution $\iota$ of 
		$C_Z(z)$ for every $z\in Z_{exc}$ 
\end{itemize}
satisfying the following conditions:
\begin{itemize}
	\item $\partial Z=\emptyset$
	\item $Z-Z_o$ has finitely many connected components,
		none of which are points
	\item given $z\in Z_{exc}$,
		given $U$ a small open neighbourhood of $z$ in $Z_o$,
		and given $L\in\pi_0(U-\{z\})$,
		then
		$L\cup \iota(L)\cup\{z\}$ has an orientation extending
		the given orientations on $L$ and $\iota(L)$.
\end{itemize}
\end{defi}

We put $Z_u=Z-Z_o$\indexnot{Z}{Z_u}. Note that $\partial Z_u=Z_u\cap\overline{Z_o}$.
Given $z\in Z-Z_{exc}$, we have
$|C(z)|=2$ and we define $\iota$ as the unique non-trivial automorphism of
$C(z)$.

\smallskip
We denote by $Z^\opp$\indexnot{Z}{Z^\opp} the {\em opposite} curve
\index[ter]{opposite curve} to $Z$ all of whose data
coincides with that of $Z$, except for $Z_o-Z_{exc}$, whose orientation is reversed.

\smallskip
Fix $n\ge 1$. The $1$-dimensional space $Z=\St(2n)$\indexnot{St}{\St} (cf \S\ref{se:1dimspaces})
can be endowed with a structure of curve by giving $\BR e^{i\pi r/n}$ the
orientation of $\BR$ for $0\le r<n$ and setting $Z_o=Z$. The involution $\iota$ is
defined by $\iota(\BR_{>0}e^{i\pi r/n})=\BR_{<0}e^{i\pi r/n}$.

\subsubsection{Morphisms and subcurves}

\begin{defi}
	A {\em morphism of curves}\index[ter]{morphism of curves} $f:Z\to Z'$ is a morphism of $1$-dimensional spaces such that
\begin{itemize}
	\item $f(Z_u)\subset Z'_u$
	\item $f_{|f^{-1}(Z'_o-Z'_{exc})}$ is orientation-preserving
	\item given $z\in f^{-1}(Z'_{exc})$, the canonical map
		$C(f):C_Z(z)\to C_{Z'}(f(z))$ is $\iota$-equivariant.
\end{itemize}
\end{defi}

Note that a composition of morphisms of curves is a morphism of curves.
Let $f:Z\to Z'$ be a morphism of curves. We have the following statements.
\begin{properties}
		\label{it:morphismscurves}
	\mbox{}
\begin{itemize}
	\item $f$ 
is invertible if and only if it is a homeomorphism and $f(Z_o)\subset Z'_o$.
\item $f(Z_{exc})\subset Z'_{exc}$ and
		$C(f):C_Z(z)\to C_{Z'}(f(z))$ is $\iota$-equivariant for
		all $z\in Z$.
\item $f$ restricts to a homeomorphism from $f^{-1}(Z'-Z'_{exc})$ to the open
	subset $f(Z)\cap
	(Z'-Z'_{exc})=f(Z-Z_{exc})\cap (Z'-Z'_{exc})$ of $Z'$, since 
		$Z_f\subset f^{-1}(Z'_{exc})$.
	In particular,	
	the restriction of $f$ to $Z_u$ is a homeomorphism $Z_u\iso f(Z_u)$.
\item If $Z'$ is non-singular, then $f$ is an open embedding.
\end{itemize}
\end{properties}

We say that $f$ is {\em strict}\index[ter]{strict morphism of curves} if 
$f(Z_u)$ is closed in $Z'_u$ and $f(Z_o)\subset Z'_o$. Note that this implies
that $f(Z_u)$ is also open in $Z'_u$.

\medskip
Let $Z$ be a  curve.

\begin{defi}
	A {\em subcurve}\index[ter]{subcurve}
	of $Z$ is a $1$-dimensional subspace $X$ of $Z$ such that
given $z\in X$, the image of $C_X(z)$ in $C_Z(z)$ is $\iota$-stable.
\end{defi}

If $X$ is a subcurve of $Z$, then $X$ is a  curve with
$X_o=X\cap Z_o$, $X_{exc}\subset Z_{exc}$ and $\iota$ is defined on 
$C_X(z)$ as the restriction of $\iota$ on $C_Z(z)$, for $z\in X_{exc}$. Note that
$X_u$ is open in $Z_u$.

\smallskip
Equivalently, a subspace $X$ of $Z$ is a subcurve if it is a curve,
$X_o=X\cap Z_o$ and
the inclusion map $X\to Z$ is a morphism of curves.


\medskip
We define an equivalence relation on connected components of $Z-Z_{exc}$: it is the 
relation generated by $T\sim T'$ if there is $z\in Z_{exc}\cap\overline{T}\cap
\overline{T'}$, $U$ a small open neighbourhood of $z$ and
	$L\in\pi_0(U-\{z\})$
such that $L\subset T$ and $\iota(L)\subset T'$.

Let $\CE$ be the set of equivalence classes of connected components of $Z-Z_{exc}$.
Given $E\in\CE$, let $Z_E=\bigcup_{T\in E}\overline{T}$. The subspaces
$Z_E$ of $Z$ are called the {\em components}\index[ter]{components of a curve} of $Z$.

\smallskip
A curve has only finitely many components, each of which is a closed subcurve.

\smallskip
If $Z$ is non-singular, then its components are its connected components.

\medskip
The local structure of a curve is described as follows.
Let $z\in Z$. 
There is an open neighbourhood $U$ of $z$ that is a subcurve of $Z$
and an isomorphism of curves $U\iso X,\ z\mapsto 0$, where $X\subset\BC$ is
one of the following:
\begin{itemize}
	\item $\BR$ viewed as an unoriented manifold, if
		$z\in Z_u-\partial Z_u$
	\item $\BR$ where $\BR_{\ge 0}$ is unoriented and $\BR_{<0}$ has either
		of its two orientations, if $z\in \partial Z_u$
	\item $\BR$ viewed as an oriented manifold, if $z\in Z_o-Z_{exc}$
	\item $\mathrm{St}(n_z)$ if $z\in Z_{exc}$.
\end{itemize}

\begin{rem}
	\label{re:smooth}
	Let $Z$ be a closed subspace of $\BR^N$ for some $N>0$.
	Assume
	there is a finite subset $E$ of $Z$ such that $Z-E$ is a
	$1$-dimensional submanifold of $\BR^N$ with no boundary and such that
	given $e\in E$, there is $n'_e>1$ and
	a finite family $\{j_{e,i}\}_{1\le i\le n'_e}$
	of smooth embeddings 
	$j_{e,i}:(-1,1)\to\BR^N$ such that
	\begin{itemize}
		\item $j_{e,i}(0)=e$, 
		\item $j_{e,i}((-1,0)\cup (0,1))\subset Z-\{e\}$,
		\item $j_{e,i}((-1,1))\cap j_{e,i'}((-1,1))=\{e\}$ for $i\neq i'$
		\item $\BR\frac{dj_{e,i}}{dt}(0)\neq \BR\frac{dj_{e,i'}}{dt}(0)$
			for $i\neq i'$ and
	\item $\bigcup_i j_{e,i}(-1,1)$ is an open neighborhood of $e$ in $Z$.
	\end{itemize}

	Let us choose in addition an open subset $Z_o$ of $Z$ containing
	$E$ and an orientation of the $1$-dimensional manifold $Z_o-E$.
	We assume that $Z-Z_o$ has finitely many connected components, none of
	which are points. We assume furthermore that
	given $e\in E$ and $i\in\{1,\ldots,n'_e\}$, the orientation of
	$j_{e,i}^{-1}(Z_o-\{e\})$ extends to an orientation of
	$j_{e,i}^{-1}(Z_o)$.

	\smallskip
	Given $e\in E$, we denote by $\iota$ the involution
	of $C(e)$ that swaps $j_{e,i}((-1,0))$ and $j_{e,i}((0,1))$ for 
	$1\le i\le n'_e$. Note that $Z_{exc}=E$ and $n_e=2n'_e$ for $e\in E$.
	This defines a structure of curve on $Z$ that does not depend on the choice
	of the maps $j_{e,i}$.

	We leave it to the reader
	to check that any curve is isomorphic to a curve obtained by such a construction.
\end{rem}

\subsubsection{Quotients}
Let $(\tilde{Z},\tilde{Z}_o,\tilde{\iota})$ be a curve. 

\begin{defi}
	A {\em finite relation}\index[ter]{finite relation}
	on $\tilde{Z}$ is an equivalence relation $\sim$ such that
the set of points that are not alone in their equivalence class is finite and
contained in $\tilde{Z}_o$.
\end{defi}

Consider a finite relation $\sim$ on $\tilde{Z}$.
We define a curve structure on the $1$-dimensional space $Z=\tilde{Z}/\!\!\sim$.

\smallskip

Let $q:\tilde{Z}\to Z$ be the quotient map. We have
$Z_{exc}=q(\tilde{Z}_{exc})\cup \{z\in Z|\ |q^{-1}(z)|>1\}$
(cf \S\ref{se:quotientsspaces}).
Let $Z_o=q(\tilde{Z}_o)$.
The map $q_{|\tilde{Z}_o-q^{-1}(Z_{exc})}:\tilde{Z}_o-q^{-1}(Z_{exc})\to
Z_o-Z_{exc}$ is a homeomorphism and we provide $Z_o-Z_{exc}$ with the orientation
coming from $\tilde{Z}_o-q^{-1}(Z_{exc})$. Let $z\in Z_{exc}$. We define
$\iota$ on $C(z)$ to make the canonical bijection 
$\coprod_{\tilde{z}\in q^{-1}(z)}C(\tilde{z})\iso C(z)$ $\iota$-equivariant.
This makes $q$ into a strict morphism of curves.

\begin{lemma}
	\label{le:quotientsfactorcurves}
Let $f:Z\to Z'$ be a morphism of curves.

Define an equivalence
	relation on $Z$ by $z_1\sim z_2$ if $f(z_1)=f(z_2)$. This is a finite
	relation on $Z$ and $f$ factors as a composition of morphisms
	of curves $Z\xrightarrow{f_1} Z/\!\!\sim\xrightarrow{f_2}Z'$ where
	$f_1$ is the quotient map and $f_2$ is injective.
\end{lemma}

\begin{proof}
	We have $Z_f\subset f^{-1}(Z'_o)\subset Z_o$.
	It follows that $\sim$ is a finite relation on $Z$ and the lemma follows
	from Lemma \ref{le:quotientsfactor}.
\end{proof}

We define the category of non-singular curves with a finite relation as the category
with objects pairs $(Z,\sim)$ where $Z$ is a non-singular curve and $\sim$ is a
finite relation on $Z$, and where
$\Hom((Z,\sim),(Z',\sim'))$ is the set of morphisms of curves $f:Z\to Z'$ such that
if $z_1\sim z_2$, then $f(z_1)\sim' f(z_2)$.

\smallskip
The next proposition shows that curves can be viewed as non-singular curves with
a finite relation.

\begin{prop}
	\label{pr:quotients}
The quotient construction defines an equivalence from the category of
non-singular curves with a finite relation to the category of curves.
\end{prop}

\begin{proof}
Let $(\tilde{Z},\sim)$ and $(\tilde{Z}',\sim')$ be two non-singular
curves with finite relations
and let $q:\tilde{Z}\to Z=\tilde{Z}/\!\!\sim$ and $q':\tilde{Z}'\to
Z'=\tilde{Z}'/\!\!\sim'$ be the quotient maps.

A morphism of curves
$f:\tilde{Z}\to \tilde{Z}'$ such that $z_1\sim z_2$ implies $f(z_1)\sim' f(z_2)$ induces
a morphism of curves $Z\to Z'$. So, the quotient
	functor induces indeed a functor as claimed. Consider
$f':\tilde{Z}\to \tilde{Z}'$ such that $z_1\sim z_2$ implies $f'(z_1)\sim' f'(z_2)$.
If $q'\circ f=q'\circ f'$, then $f$ and $f'$ coincide outside a finite set of points,
hence $f=f'$. So, the quotient functor is faithful.

\smallskip
	Consider now a morphism of curves $g:Z\to Z'$. Let $E'$
be the finite subset of $\tilde{Z}'$ of points that are not alone
	in their equivalence class and $E=q^{-1}(g^{-1}(q'(E')))$.
	Consider the composition of continuous maps
	$$f:\tilde{Z}-E\xrightarrow{q}Z-q(E)\xrightarrow{g}Z'-q'(E')
	\xrightarrow{(q'_{|\tilde{Z}'-E'})^{-1}} \tilde{Z}'-E'.$$

		Given $z\in E$, the
	$\iota$-equivariance of $C(g):C_{Z}(q(z))\to C_{Z'}(g(q(z)))$
	ensures that $f$ extends to a continuous map at $z$.
	So, $f$ extends (uniquely) to a continuous map $\tilde{Z}\to
	\tilde{Z}'$, and that map is a morphism of $1$-dimensional spaces.

	We have $\tilde{Z}_u\subset \tilde{Z}-E$ and
	$f(\tilde{Z}_u)\subset \tilde{Z}'_u$. Since
	$g_{|g^{-1}(\tilde{Z}'_o)-E}$ is
	orientation-preserving, it follows that
	$f_{|f^{-1}(\tilde{Z}'_o-E')}$ is orientation-preserving.
 So, $f:\tilde{Z}\to \tilde{Z}'$ is a morphism of curves and it is
	compatible with the relations.
	This shows that the quotient functor is fully faithful.

	\smallskip
	Let now $Z$ be a curve.
	Let $z\in Z_{exc}$ and $U_z\subset Z_o$ be a small open neighbourhood of $z$.
	Fix an isomorphism of
	curves $f_z:U_z\iso \mathrm{St}(n_z),\ z\mapsto 0$.
	The equivalence relation on $\pi_0(U_z-\{z\})$ whose equivalence classes are the
	orbits of $\iota$ defines via $f_z$ the equivalence relation on
	$\{e^{i\pi r/2n_z}\}_{0\le r<2n_z}$ given by $\zeta\sim \zeta'$ if and only
	if $\zeta'=\zeta^{\pm 1}$.

	The proof of Lemma \ref{le:1manifoldsquotients} provides us
	a non-singular curve $\hat{Z}$ with a finite relation.
	Indeed, with the notations of the proof of Lemma \ref{le:1manifoldsquotients},
	we have $\hat{U}_z=\coprod_{0\le r<n_z}\BR e^{i\pi r/n_z}$.
	Note that $\hat{Z}_o$ is the subspace of $\hat{Z}$ obtained by adding to
	$Z_o-Z_{exc}$ the point
	$0$ of $\BR e^{i\pi r/n_z}$ for each $r\in\{0,\ldots,n_z-1\}$ and
	each $z\in Z_{exc}$.

	This gives $\hat{Z}$ a structure of non-singular curve.
	As in the proof of Lemma \ref{le:1manifoldsquotients}, we obtain
	a finite relation on $\hat{Z}$ and an isomorphism of curves
	$Z\iso \hat{Z}/\!\sim$. This shows that the quotient functor is essentially
	surjective.
\end{proof}

\begin{defi}
	Given $Z$ a curve, the {\em non-singular cover}\index[ter]{non-singular cover} 
	of $Z$ is a non-singular curve
	$\hat{Z}$\indexnot{Z}{\hat{Z}}, together with a finite relation $\sim$ and an isomorphism
$\hat{Z}/\!\!\sim\ \iso Z$.
\end{defi}

Note that $Z_{exc}=Z_q$ where $q:\hat{Z}\to Z$ is the canonical map.
	Proposition \ref{pr:quotients} shows that non-singular covers exist and
are unique up to a unique isomorphism. The following proposition makes this more precise.

\begin{prop}
	\label{pr:coveradjoint}
	The functor sending a curve $Z$ to its non-singular cover is right adjoint to
	the embedding of the category of non-singular curves in the category of curves.
\end{prop}

\begin{proof}
	Let $Z'$ be a non-singular curve.
	We have a map $h:\Hom(Z',\hat{Z})\to\Hom(Z',Z),\ g\mapsto q\circ g$. Since $Z_q$
	is finite, it follows that $h$ is injective.

	Consider now a morphism of curves $f:Z'\to Z$. We factor $f$ as
	$Z'\xrightarrow{f_1}Z'/\!\sim\ \xrightarrow{f_2}Z$ as in
	Lemma \ref{le:quotientsfactorcurves}. By Proposition
	\ref{pr:quotients}, there is a morphism $\hat{f}:Z'\to \hat{Z}$ such that
	$q\circ \hat{f}=f$, hence $h(\hat{f})=f$. So $h$ is surjective.
\end{proof}

\begin{example}
	\label{ex:cover}
	Let us provide some examples of curves and non-singular covers.
	The dotted lines link the points in the same equivalence class.
	The grey part corresponds to $Z_u$.

$$\includegraphics[scale=0.85]{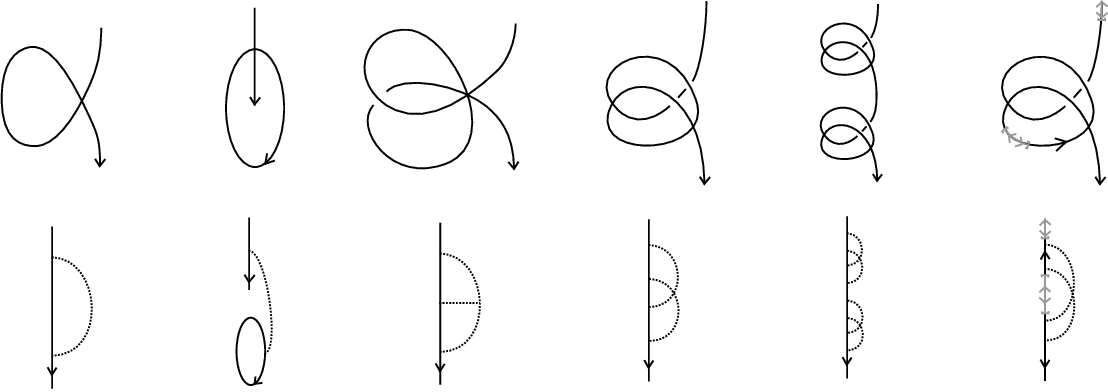}$$
\end{example}
	
	\subsubsection{Chord diagrams as singular curves}
	\label{se:arcdiagrams}

	We describe here the relation between singular curves and
	chord (or arc) diagrams.

	\smallskip
	We define a {\em chord diagram}\index[ter]{chord diagram} to be 
	to be a triple $(\CZ,\Ba)$ where
	\begin{itemize}
		\item $\CZ$ is a closed oriented $1$-dimensional manifold
			(i.e., a finite disjoint union of copies of
			$S^1$ and $[0,1]$)
		\item $\Ba$ is a finite set of pairs of points of $\mathring{\CZ}$, all
			of which are distinct.
	\end{itemize}

	A chord diagram gives rise to a smooth oriented curve $\tilde{Z}=
	\mathring{\CZ}$ with the following relation: given $z\neq z'$, we have
	$z\sim z'$ if $\{z,z'\}\in\Ba$.
	We obtain an oriented curve $Z=\tilde{Z}/\!\sim$ and a map
	$\mu:\bigcup_{\{z,z'\}\in\Ba}\{z,z'\}\to Z_{exc}$ inducing a bijection
	$\Ba \iso Z_{exc}$.

	\smallskip
	Up to suitable isomorphism, this defines a bijection from chord diagrams
	to oriented singular curves with $n_z\in\{2,4\}$ for all $z$.

	\medskip

\begin{convention}
	\label{con:reversal}
We will use the above bijection composed with the reversal of all orientations
when identifying chord diagrams with certain singular curves. This orientation
reversal is related to the usual direction reversal between arrows in a quiver
and morphisms in the corresponding path category, and to the time-reversal of
graphs mentioned in Example \ref{ex:examplespaths} below.
\end{convention}

	\medskip

	When $\CZ$ is a union of intervals, we
	recover the notion of (possibly degenerate) arc diagram due
	to Zarev \cite[Definition 2.1]{Za} (compare Example
	\ref{ex:cover} and \cite[Figures 3 and 4]{Za}).

	The chord diagrams such that the singular curve $Z$ is connected
	and $k>0$ correspond to the chord diagrams of \cite{AnChePeReiSu}.

	Zarev's definition generalizes that of pointed matched circles due to
	Lipshitz, Ozsv\'ath and Thurston \cite[\S 3.2]{LiOzTh1}: they correspond
	to the case where $\CZ$ is a single interval
	($\mathring{\CZ}$ is obtained from
	the circle considered in \cite{LiOzTh1} by removing its
	basepoint).

	\subsubsection{Sutured surfaces and topological field theories}
\label{se:TQFT}

	\medskip
	We define a {\em sutured surface}\index[ter]{sutured surface}
	to be a quadruple
	$(F,\Lambda,S^+,S^-)$ where $F$ is a compact oriented surface,
	$\Lambda$ is a finite subset of $\partial{F}$, and
	$S^+$ and $S^-$ are unions of components of $\partial{F} - \Lambda$
such that $\Lambda=\overline{S^+} \cap \overline{S^-}$ and $\partial{F}-\Lambda=S^+\cup S^-$
(this is \cite[Definition 1.2]{Za} without the topological restrictions). Note that
$(F,\Lambda,S^+,S^-)$ is determined by the data of $(F,S^+)$: we have $\Lambda=\overline{S}^+
-S^+$ and $S^-=\partial F-\overline{S}^+$.
A sutured surface is representable by a chord diagram (as we define it) if and
only if each component of $F$ (not $\partial{F}$) intersects $S^+$ and $S^-$
nontrivially.








\smallskip
Let $(\CZ,\Ba)$ be a chord diagram. We define a sutured surface
$(F,\Lambda,S^+,S^-)$:
\begin{itemize}
	\item the oriented surface $F$ is obtained
from $\CZ\times [0,1]$ by adding $1$-handles at
		$\{(z,0),(z',0)\}$ for all pairs $\{z,z'\}$ in $\Ba$
\item $S^+=\bigl(\CZ \times \{1\}\bigr) \cup \bigl(\partial\CZ \times
	(\frac{1}{2},1]\bigr)$
		\end{itemize}

		When $\CZ$ is a union of intervals, this is
		Zarev's construction \cite[\S 2.1]{Za}.

		\smallskip
		Let $Z$ be a singular curve giving rise to $(\CZ,\Ba)$. The
		oriented
		surface $F$ can be identified with $Z\times [0,1]$ and paths in $Z$ give rise
		to paths in $F$.

		\medskip
		The sutured surface $F$ also comes with an arc decomposition: for each
		$\{z,z'\}$ in $\Ba$, 
		we have an arc $\omega_{\{z,z'\}}$ with set of end points 
		$\{(z,1),(z',1)\}$ in $S^+$ corresponding to the $1$-handle 
		added at $\{(z,0),(z',0)\}$. 

	\begin{example}
		\label{ex:sutured}
In the table below, the first row depicts some chord diagrams. The second and
third rows show the corresponding sutured surfaces with the $S^+$ part of the boundary in
		green and with the arcs $\omega_z$ in red; the second row applies the
above construction directly, and the third row gives an alternate perspective.
The fourth row shows the sutured surfaces as open-closed cobordisms (with
empty source and with target colored in green);
			this interpretation is
discussed in \S\ref{se:TQFT}.

Under the strands algebra construction of \S \ref{se:leftaction},
		the first and second
columns give rise to simple 2-representations of $\mathcal{U}$, categorifying
the vector representation and its dual.

		Tensor powers of the algebra of the first column give algebras very
		similar to the one considered by Tian \cite{Ti};
in fact, Tian's algebras were an important early clue in the development of
the present work. Tensor powers of the algebra of the second column are
	studied from the Heegaard Floer perspective by the first-named author in \cite{Man}.

The algebra of the third column is the $n=3$ case of a family of algebras
considered in \cite{ManMarWi,LePo}.
For general $n$, these are isomorphic to the algebras 
$\CB(n) = \oplus_{k=0}^n \CB(n,k)$ used by Ozsv\'ath and Szab\'o in their theory
		of bordered knot Floer homology \cite{OsSz4,OsSz5,OsSz6}
(their notation is slightly different). The middle summand of the algebra of the
fourth column is the undeformed version of a curved $A_{\infty}$-algebra used
		by Lipshitz-Ozsv\'ath-Thurston \cite{LiOzTh2,LiOzTh3} to define bordered
		$HF^-$ for $3$-manifolds with torus boundary.
The middle summand of the algebra of the fifth column is the
well-known ``torus algebra" from bordered Floer homology. The fifth and sixth
columns together illustrate our perspective on cornered Floer homology;
following Zarev's ideas, we view the cornered Floer gluing theorem as
recovering the algebra of two matched intervals glued end-to-end, rather than
as the invariants of two matched intervals with distinguished endpoints being
glued to form a pointed matched circle.

The first, fifth, and sixth columns give algebras that are among Zarev's
strands algebras $\mathcal{A}(\mathcal{Z})$, although the first diagram is
degenerate (equivalently, its sutured surface has closed circles in $S^-$).
The second, third, and fourth columns do not satisfy the restrictions that
Zarev imposes. As far as we are aware, our strands categories below give the
first detailed description of strands algebras associated to general chord
diagrams with circles as well as intervals; less formal descriptions have
appeared previously, cf. \cite[Proposition 11]{Au2}. As indicated by
Lipshitz-Ozsv\'ath-Thurston's work \cite{LiOzTh2,LiOzTh3}, curved
$A_{\infty}$-deformations of the algebras appear necessary in the general
setting when defining modules and bimodules for 3-manifolds with boundary,
although in special cases like Ozsv\'ath-Szab\'o's bordered knot Floer homology
(third column) this complication should be avoidable.
$$\includegraphics[scale=0.85]{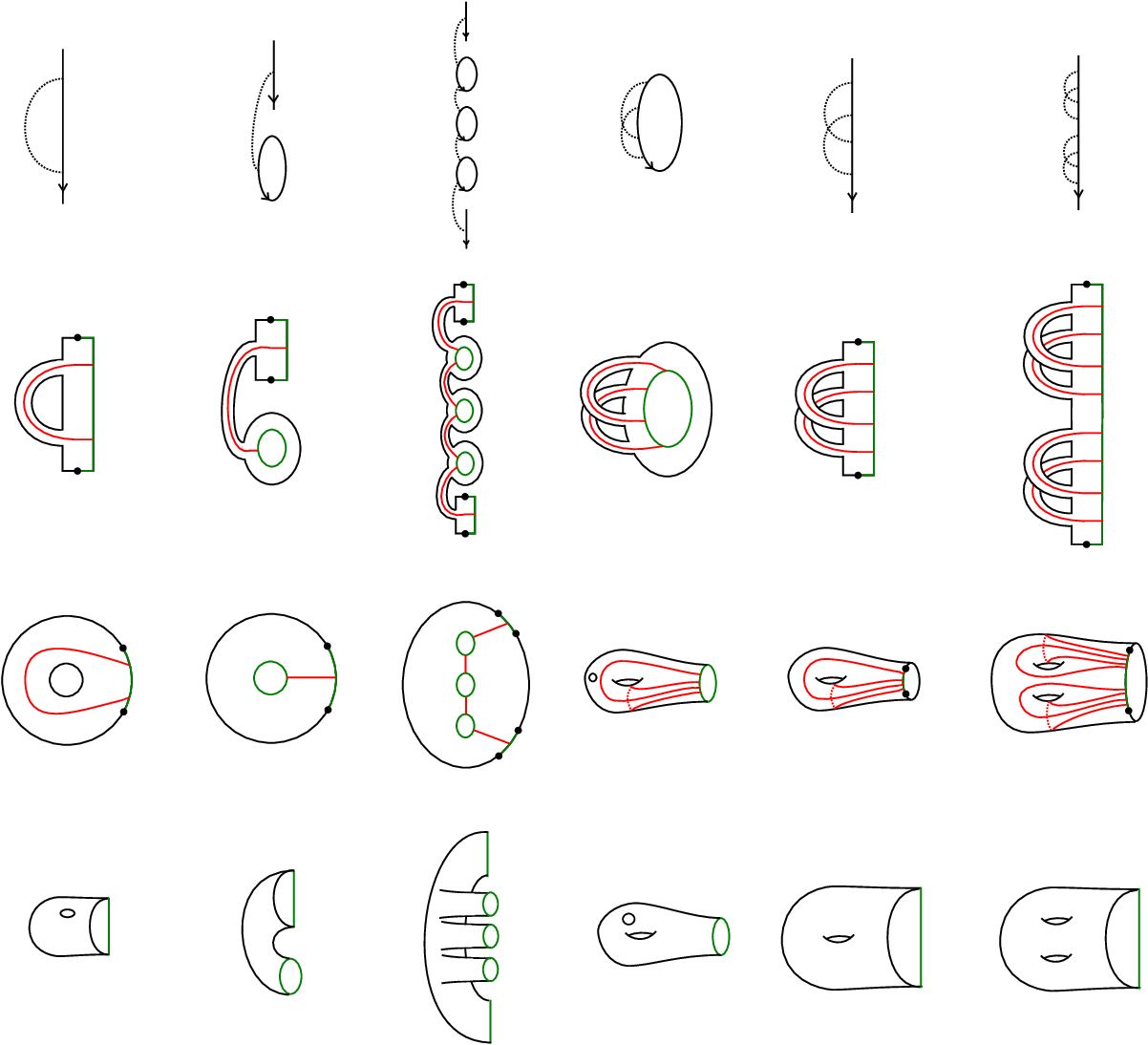}$$
	\end{example}

\newcommand{\Otimes}{ 
  \mathbin{
    \mathchoice
      {\buildcircleotimes{\displaystyle}}
      {\buildcircleotimes{\textstyle}}
      {\buildcircleotimes{\scriptstyle}}
      {\buildcircleotimes{\scriptscriptstyle}}
  } 
}

\newcommand\buildcircleotimes[1]{%
  \begin{tikzpicture}[baseline=(X.base), inner sep=0, outer sep=0]
    \node[draw,circle] (X)  {$#1\otimes$};
  \end{tikzpicture}%
}

A sutured surface can be viewed as a morphism in the 2d open-closed cobordism category with empty source; if $(F,\Lambda,S^+,S^-)$ is a sutured surface, the corresponding open-closed cobordism has target given by $S^+$ and non-gluing boundary given by $S^-$. See the bottom row of the figure in Example~\ref{ex:sutured}; the targets of these open-closed cobordisms are shown in green and the non-gluing boundary is shown in black.

Let us consider how the end-to-end gluings of chord diagrams covered by our results in \S\ref{se:2repstrand}
can be viewed in terms of open-closed cobordisms. When gluing two distinct intervals of a chord diagram end-to-end, the corresponding sutured surface gets glued as in the top-left picture below: the two intervals marked in blue are glued together to form the top-middle picture. However, we can also consider the top-middle picture as arising from the top-right picture; in this latter case the gluing is an instance of composition (with an open pair of pants) in the open-closed cobordism category. Similarly, when self-gluing the two endpoints of an interval of a chord diagram, the sutured surface gets glued as in the bottom-left picture below, producing the bottom-middle picture; we can also think of the bottom-middle picture as arising from the bottom-right picture, which is another instance of composition in the open-closed category. 

$$\includegraphics[scale=0.6]{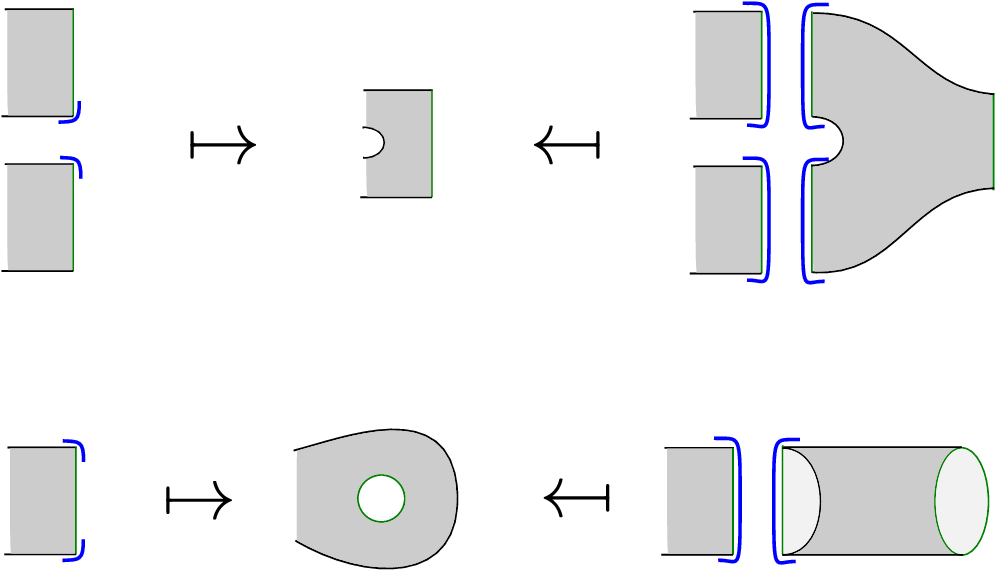}$$

One could try to view our constructions as giving part of the structure of an open-closed 2d TQFT valued in a category whose objects are dg 2-categories and whose morphisms are certain dg 2-functors. In particular, this hypothetical open-closed TQFT would assign a dg 2-category of 2-representations of $\CU$ to an interval. To an open-closed cobordism with empty source, the open-closed TQFT would assign an object of the dg 2-category of the target, encoding the data of a lax multi-2-action of $\CU$ for the interval components of the target. 
Our approach doesn't quite realize that. We associate 2-representations of $\CU$ to chord diagrams or singular curves rather than directly to surfaces. 

One can also consider the extent to which such a theory would extend to a point.
Things are considerably simpler for the decategorified version
of the theory, where one sees many relationships with other work on 3d TQFTs; this will be addressed in more detail in a follow-up paper \cite{ArMa}.


\subsection{Paths}
\label{se:curvepaths}

\subsubsection{Admissible paths}
\label{se:admissiblepaths}
Let $Z$ be a curve. 

\begin{defi}
	An {\em oriented path}\index[ter]{oriented path} $\gamma$ in $Z$ is defined to be a
	path whose
restriction to $\gamma^{-1}(Z_o-Z_{exc})$ is compatible (non strictly) with the
orientation.
\end{defi}

Let us note some basics facts about oriented paths.

\begin{properties}
		\label{it:orientedpaths}
	\mbox{}
	Let $\gamma$ be a non-constant oriented path in $Z$.
\begin{itemize}
	\item[(1)] We have
		$\gamma([0,1])\cap Z_o=\supp([\gamma])\cap Z_o$ and 
		$\gamma([0,1])\cap Z_u$ is contained in the union of the connected
	components of $Z_u$ that have a non-empty intersection with 
	$\supp([\gamma])$.
\item[(2)]
If $\gamma$ is homotopic to a constant path,
		then it is contained in $Z_u$ (as $\gamma([0,1])$ is contractible).
	\item[(3)] There are unique real numbers $0=t_0<t_1<\cdots<t_r=1$
	such that
		\begin{itemize}
			\item given $0\le i<r$,
		there are $\{j,k\}=\{i,i+1\}$ with the property that
				$\gamma([t_j,t_{j+1}])\subseteq Z_u$ (if
				$j<r$) and 
				$\gamma([t_k,t_{k+1}])\subseteq \bar{Z}_o$ (if
				$k<r$)
			(cf Lemma \ref{le:componentspath} for $E=Z_u\cap\bar{Z}_o$).
		\item given $0<i<r$ and $\eps>0$ such that $\gamma([t_i,t_i+\eps])\subset
			Z_u\cap\bar{Z}_o$, we have $\gamma([t_i,t_{i+1}]){\not\subseteq}
				\bar{Z}_o$
		\item given $0<i<r$ and $\eps>0$ such that $\gamma([t_i-\eps,t_i])\subset
			Z_u\cap\bar{Z}_o$, we have $\gamma([t_{i-1},t_i]){\not\subseteq}
				\bar{Z}_o$.
		\end{itemize}
		The sequence $[\gamma_{|[t_0,t_1]}],\ldots,[\gamma_{|[t_{r-1},t_r]}]$
		depends only on $[\gamma]$.
	
\item[(4)] Consider homotopy classes of oriented paths
	$\zeta_1$, $\zeta_2$ and $\zeta_3$
	with $[\gamma]=\zeta_3\circ\zeta_2\circ\zeta_1$. If $\supp(\zeta_2)$ is
		contained in $\overline{Z_o}$ but not in $Z_u$, then there are
		$0\le t_1\le t_2\le  1$ such that $[\gamma_{|[0,t_1]}]=\zeta_1$,
		$[\gamma_{|[t_1,t_2]}]=\zeta_2$ and
		$[\gamma_{|[t_2,1]}]=\zeta_3$.
\end{itemize}
\end{properties}

\begin{lemma}
\label{le:equivalencesmooth}
Let $\gamma$ be a path in $Z$. The following conditions are equivalent:
	\begin{itemize}
		\item[(i)] $\gamma$ lifts to a path in the non-singular cover of $Z$
		\item[(ii)] given $z\in Z_{exc}$, given a small open neighbourhood 
			$U$ of $z$ in $Z_o$ and given $K$ a connected component
			of $\gamma^{-1}(z)$, the set of
			$L\in\pi_0(U-\{z\})$ with 
			$K\cap \overline{\gamma^{-1}(L)}\neq \emptyset$ is contained
			in an orbit of $\iota$.
	\end{itemize}
\end{lemma}

\begin{proof}
	Let $\hat{Z}$ be the non-singular cover of $Z$ and
	$q:\hat{Z}\to Z$ be the quotient map.

	Assume (i).
	Consider $z$, $U$, $K$ as in the lemma and let $\hat{\gamma}$ be a lift of
	$\gamma$. Consider $L_i\in\pi_0(U-\{z\})$ with
	$K\cap\overline{\gamma^{-1}(L_i)}\neq\emptyset$ for $i\in\{1,2\}$. We
	have $\hat{\gamma}(K)\subset \overline{q^{-1}(L_i)}$.
	Consequently, we have 
	$\overline{q^{-1}(L_1)}\cap\overline{q^{-1}(L_2)}\neq\emptyset$.
	If $L_1$ and $L_2$ are not in the same $\iota$-orbit, then
	$\overline{q^{-1}(L_1)}$ and $\overline{q^{-1}(L_2)}$
	are in distinct connected components of $\overline{q^{-1}(U)}$, a contradiction.
	So, (ii) holds.

	Assume (ii).
	Since lifts of non-identity paths are unique if they exist 
	(Lemma \ref{le:mappaths}), it is
	enough to show the existence of lifts locally on $Z$. This is clear for a small
	open neighbourhood of a point of $Z-Z_{exc}$. Consider now
$z\in Z_{exc}$ and a small open neighbourhood 
$U$ of $z$ in $Z_o$.
	Let $K$ be a connected component of $\gamma^{-1}(z)$ and let $W$ be
	the connected component of $\gamma^{-1}(U)$ containing $K$. There is
	$L\in\pi_0(U-\{z\})$ such that $\gamma(W)\subset L\cup\{z\}\cup\iota(L)$.
	Since $q$ splits over $L\cup\{z\}\cup\iota(L)$, it
	follows that the restriction of $\gamma$ to $W$ lifts to $\hat{Z}$.
\end{proof}

\begin{defi}
	We say that a path $\gamma$ in $Z$ is {\em smooth}\index[ter]{smooth path}
	if it satisfies the
equivalent conditions of Lemma \ref{le:equivalencesmooth}.

	We say that a path $\gamma$ in $Z$ is {\em admissible}\index[ter]{admissible
	path} if it is oriented and smooth.

We say that 
a homotopy class of paths is smooth (resp. admissible, resp. oriented) if it contains a
smooth (resp. an admissible, resp. an oriented) path.
\end{defi}

\smallskip
Let us note some basic properties of smooth and admissible paths and classes.

\begin{properties}
	\label{it:smoothadmissible}
	\mbox{}
\begin{itemize}
	\item[(1)] A path is smooth if and only if its inverse is smooth.
	\item[(2)]
A smooth path is contained in a component of $Z$.
\item[(3)]  Every admissible path $\gamma$ is homotopic to a minimal admissible path
via a homotopy involving only admissible paths contained in the support of
$\gamma$ (cf Lemma \ref{le:minimalpaths}).
		\item[(4)] A minimal path in a smooth (resp. admissible)
			homotopy class is smooth (resp. admissible).
		\item[(5)] An oriented path is admissible if and only if
			its homotopy class is admissible (Lemma
			\ref{le:componentspath} provides a minimal oriented
			path $\gamma_{min}$ homotopic to a given oriented path $\gamma$
			with the property that $\gamma$ is admissible if $\gamma_{min}$ is 
			admissible, hence we obtain the desired equivalence by (4) above).
		\item[(6)] Given two oriented homotopy classes of paths
$\alpha$ and $\beta$ with $\alpha\circ\beta$ admissible, then
		$\alpha$ and $\beta$ are admissible (cf (5) above).
\end{itemize}
\end{properties}

\begin{defi}
	\label{de:opposite}
Given two smooth non-identity homotopy classes of paths $\zeta_1$ and $\zeta_2$
	contained in the same component of $Z$, there is a unique
	$\eps\in\{\pm 1\}$ such that there is a minimal smooth path
		$\gamma$ in $Z$ with the property that $\zeta_1$
		and $\zeta_2^\eps$ are equal to the classes of restrictions of
		$\gamma$. We say that $\zeta_1$ and $\zeta_2$ have the
		{\em same orientation}\index[ter]{same orientation} (resp. {\em opposite
		orientation})\index[ter]{opposite orientation} if $\eps=1$
		(resp. $\eps=-1$).
\end{defi}

Note that $Z^\opp$ and $Z$ have the same smooth paths.
Note also that the notion of ``opposite orientation" does not depend on the orientation
of $Z$.

\begin{rem}
	Assume $X\subset\BR^N$ is obtained by the construction of 
	Remark \ref{re:smooth}.
A homotopy class of paths in $X$ is smooth if and only if it contains a path
	$\gamma$ such that the composition $[0,1]\xrightarrow{\gamma}X\hookrightarrow
	\BR^N$ is a smooth immersion.
\end{rem}

\begin{example}
	\label{ex:examplespaths}
	We give below some examples of paths. The top and bottom
	paths are  admissible, while the middle one is not. The left and
	middle columns describe the path in the singular curve, while
	the right column describes the lifted path (if it exists) in
	the non-singular cover.

	In the middle and right columns, and throughout the paper, we depict paths $\gamma$ using their time-reversed graphs, so that $\gamma(0)$ is on the right and $\gamma(1)$ is on the left.

$$\includegraphics[scale=0.85]{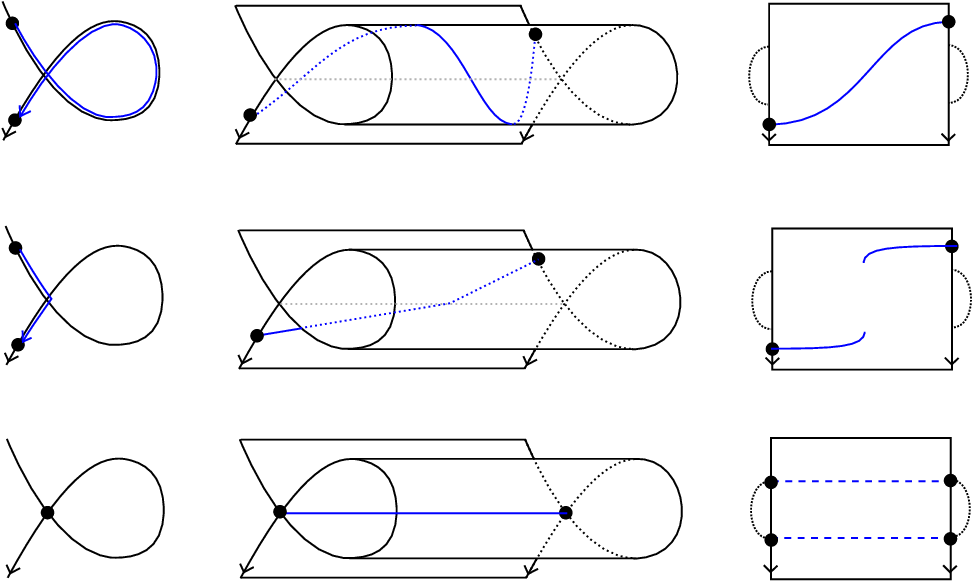}$$
\end{example}

\subsubsection{Pointed category of admissible paths}

We now define a category associated with admissible paths.

\begin{defi}
	We define $\CS^\bullet(Z,1)$\indexnot{S(}{\CS^\bullet(Z,n)}
	to be the pointed category with object set $Z$, with
$$\Hom_{\CS^\bullet(Z,1)}(x,y)=\{0\}\sqcup \{\text{admissible homotopy classes of
paths }x\to y\}$$
 and
$$\alpha \beta=\begin{cases}
\alpha\circ\beta & \text{ if }\alpha\circ\beta \text{ is admissible} \\
0 &\text{ otherwise.}
\end{cases}$$
\end{defi}

\begin{rem}
	Consider $\Pi_o(Z)$ the category with objects the points of $Z$ and arrows the
	oriented homotopy classes of
	paths, a subcategory of $\Pi(Z)$. We define a $\BZ_{\ge 0}$-filtration
	on $\Pi_o(Z)$ by defining a class $\zeta$
	to have degree $\le d$ if it is the product
	of $d+1$ admissible homotopy classes of paths.
	The category $\CS^\bullet(Z,1)$ is isomorphic to the degree $0$ part of $\gr\Pi_o(Z)$.

	Note finally that if $Z$ is non-singular, then $\CS^\bullet(Z,1)$ is the pointed category
	associated to $\Pi_o(Z)$.
\end{rem}


We put $\CS(Z,1)=\BF_2[\CS^\bullet(Z,1)]$.

\begin{example}
	We describe below some examples of products in $\CS^\bullet(Z,1)$.
	Here $Z$ is the third singular curve of example \ref{ex:cover} and
	the paths are drawn in the smooth cover.
$$\includegraphics[scale=0.70]{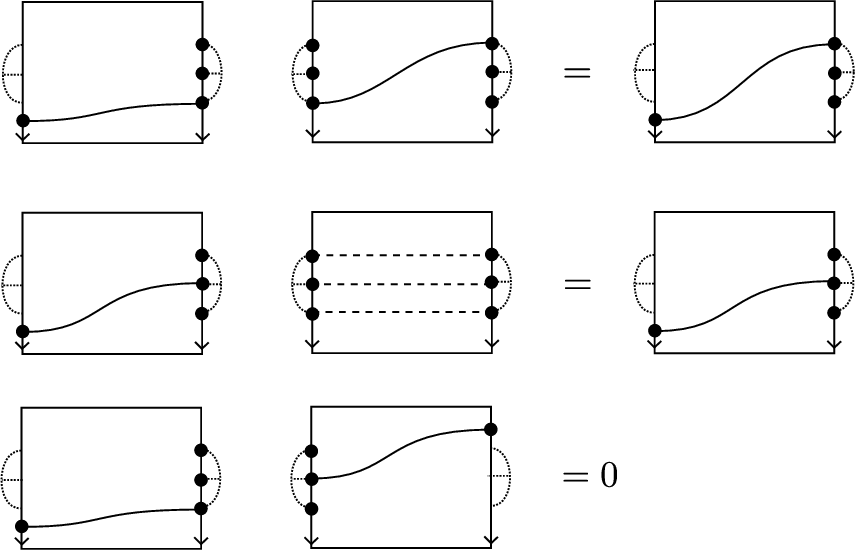}$$
\end{example}

\subsubsection{Central extension}
\label{se:central}

\smallskip
Let $L(Z)=(\bigoplus_{c\in T(Z)}\BZ e_c)/(\bigoplus_{c\in T(Z)}\BZ
(e_c+e_{\iota(c)}))$\indexnot{LZ}{L(Z)}\indexnot{ec}{e_c}.
We define a bilinear map
$\langle-,-\rangle:R(Z)\times R(Z)\to L(Z)$\indexnot{<}{\protect\langle
\alpha,\beta\protect\rangle} by
$$\langle \alpha,\beta\rangle=\frac{1}{2}\sum_{c\in T(Z)}
(m_{\iota(c)}-m_c)(\alpha)\cdot(m_c+m_{\iota(c)})(\beta)e_c.$$

Note that $(m_c+m_{\iota(c)})(\beta)=0$ for all but
finitely many $c$'s, hence the sum above is finite.
More precisely,
let $\zeta$ be a non-identity homotopy class of paths in $Z$.
We have $\zeta(0+)\neq\zeta(1-)$ and
\begin{equation}
	\label{eq:scalar}
	(m_c+m_{\iota(c)})(\zeta)=\begin{cases}
		1 & \text{if }c\in\{\zeta(0+)\cup\iota(\zeta(0+))\} \text{ and }
		c{\not\in}\{\zeta(1-)\cup\iota(\zeta(1-))\}\\
		-1 & \text{if } c\in\{\zeta(1-)\cup\iota(\zeta(1-))\}
		\text{ and } c{\not\in}\{\zeta(0+)\cup\iota(\zeta(0+))\}\\
		0 & \text{otherwise.}
\end{cases}
\end{equation}

If $\zeta$ is admissible and non-identity,
then $\zeta(1-)\neq\zeta(0+)$, hence
$$\langle\alpha,\llbracket\zeta\rrbracket\rangle=
(m_{\iota(\zeta(0+)}-m_{\zeta(0+)})(\alpha)e_{\zeta(0+)}-
(m_{\iota(\zeta(1-)}-m_{\zeta(1-)})(\alpha)e_{\zeta(1-)}.$$

\medskip
We define a group $\Gamma'(Z)$, a central extension of $R(Z)$ by $L(Z)$.
The set of elements of 
$\Gamma'(Z)$ is $L(Z)\times R(Z)$ and the multiplication is given by
$$(m,\alpha)(n,\beta)=(m+n+\langle\alpha,\beta\rangle,\alpha+\beta).$$

We put $\Gamma(Z)=(\bigoplus_{\Omega\in\pi_0(Z)}\frac{1}{2}\BZ e_\Omega)\times\Gamma'(Z)$\indexnot{Ga}{\Gamma(Z)}\indexnot{eo}{e_\Omega}.

\smallskip
Note that $L(Z)$, $\Gamma'(Z)$ and $\Gamma(Z)$ depend only on the $1$-dimensional
space underlying $Z$ and on $\iota$.

\smallskip
Let $D$ be a subset of $T(Z)$ such that $D\cap\iota(D)=\emptyset$.
We denote by $\Gamma(Z,D)$\indexnot{Ga}{\Gamma(Z,D)} the quotient of $\Gamma(Z)$ by
the central subgroup generated by $\{e_c+\frac{1}{2}e_\Omega\}$, where
$c\in D$ and $\Omega$ is the connected component of $Z$ containing $c$.
The canonical map $\bigoplus_{\Omega\in\pi_0(Z)}\frac{1}{2}\BZ e_\Omega\to
\Gamma(Z,D)$ is
injective and we identify $(\frac{1}{2}\BZ)^{\pi_0(Z)}$ with its image.

We put a partial order on $\Gamma(Z,D)$ by setting $g_1\ge g_2$ if
$g_1g_2^{-1}\in (\frac{1}{2}\BZ_{\ge 0})^{\pi_0(Z)}$.

\smallskip
We define $\bar{\Gamma}(Z,D)$\indexnot{Ga}{\bar{\Gamma}(Z,D)}
to be the quotient of $\Gamma(Z,D)$ by the central
subgroup generated by $\frac{1}{2}e_\Omega-\frac{1}{2}e_{\Omega'}$ for
$\Omega,\Omega'\in\pi_0(Z)$. 

The image of $(\frac{1}{2}\BZ_{\ge 0})^{\pi_0(Z)}$ in $\bar{\Gamma}(Z,D)$
is $\frac{1}{2}\BZ$ (where $\frac{1}{2}e_\Omega\mapsto\frac{1}{2}$).
Let $r\in\frac{1}{2}\BZ$. We still denote by $r$ the image of $r$ in 
$\bar{\Gamma}(Z,D)$. Given $x\in \bar{\Gamma}(Z,D)$, we put
$x+r=x\cdot r=r\cdot x$.

We define a partial order on $\bar{\Gamma}(Z,D)$ by setting $g_1\ge g_2$\indexnot{<}{g_1<g_2}
if $g_1g_2^{-1}\in \frac{1}{2}\BZ_{\ge 0}$.

\medskip
Given $z\in Z_o$, we denote by $C(z)^+$\indexnot{Cz}{C(z)^+} the set of $c\in C(z)$  
such that there is an oriented path $\gamma$ in $Z$ with $m_c^+(\gamma)=1$.
Note that $C(z)=C(z)^+\coprod \iota(C(z)^+)$. Note also that given
$\zeta$ an oriented homotopy class of paths in $Z$, we have
\begin{equation}
	\label{eq:mzeta+}
m_c(\zeta)e_c+m_{\iota(c)}(\zeta)e_{\iota(c)}=
(m_c^+(\zeta)+m_{\iota(c)}^-(\zeta))e_c \text{ for }
z\in Z_o \text{ and }c\in C(z)^+.
\end{equation}

Given $E$ a subset of $Z_o$, we put $E^+=\coprod_{z\in E}C(z)^+$\indexnot{E}{E^+}.

\begin{rem}
	Fix an orientation of each component of $Z$ (forgetting about the
	already given orientation of $Z_o)$ and define
$Z^+\subset T(Z)$ to be the set of pairs $(z,c)$ such that 
there is an oriented path $\gamma$ in $Z$ (for the given new orientation)
	with $m_c^+(\gamma)=1$.

There is a quotient map $L(Z)\to \BZ^{\pi_0(Z)}$ given by 
$e_c\mapsto e_\Omega$ for all $s\in\Omega$ and $(z,c)\in Z^+$.
Let us show that the bilinear form $R(Z)\times R(Z)\to \BZ^{\pi_0(Z)}$
obtained by composing $\langle-,-\rangle$ with this quotient map is
antisymmetric. Let $\gamma$ and $\gamma'$ be two injective oriented
	paths in $Z$ (for the given new orientation).
	If the supports of $\gamma$ and $\gamma'$ are disjoint, then
	$\langle\llbracket\gamma\rrbracket,\llbracket\gamma'\rrbracket\rangle=0$. We have
	$\langle\llbracket\gamma\rrbracket,\llbracket\gamma\rrbracket\rangle=
	-e_{\gamma(0+)}-e_{\gamma(1-)}$. If
	$\gamma([0,1])\cap\gamma'([0,1])=\{\gamma(1)\}$, then
$$\langle\llbracket\gamma\rrbracket,\llbracket\gamma'\rrbracket\rangle=
	-e_{\gamma'(0+)}
	\text{ and }
\langle\llbracket\gamma\rrbracket,\llbracket\gamma'\rrbracket\rangle=
	-e_{\gamma(1-)}.$$
We deduce the antisymmetry statement.
\end{rem}

Let $M$ be a subset of $Z$. We denote by $L_M(Z)$\indexnot{LM}{L_M(Z)}
the subgroup of $L(Z)$ generated
by elements $e_c$ with $\mathrm{pt}(c)\in M$.
The restriction of the pairing $\langle-,-\rangle$ to $R_M(Z)\times R_M(Z)$
takes values in $L_M(Z)$ and we denote by $\Gamma'_M(Z)$
\indexnot{Gamma'M}{\Gamma'_M(Z)} the subgroup
of $\Gamma'_M(Z)$ with elements $(m,\alpha)$ where $m\in L_M(Z)$ and
$\alpha\in R_M(Z)$. Finally, we define $\Gamma_M(Z)$
\indexnot{GammaM}{\Gamma_M(Z)} as the subgroup
$(\bigoplus_{\Omega\in\pi_0(Z),\ M\cap\Omega\neq\emptyset}\frac{1}{2}\BZ e_\Omega)\times\Gamma'_M(Z)$ of
$\Gamma(Z)$.

We denote by $\Gamma_M(Z,D)$\indexnot{GM}{\Gamma_M(Z,D)} (resp.
$\bar{\Gamma}_M(Z,D)$\indexnot{GM}{\bar{\Gamma}_M(Z,D)}) the image 
of $\Gamma_M(Z)$ in $\Gamma(Z,D)$ (resp. $\bar{\Gamma}(Z,D)$).

\smallskip
If $M\subset Z_o$, we put 
$\Gamma_{M^+}(Z)=\Gamma_M(Z,M^+)$\indexnot{GM}{\Gamma_{M^+}(Z)}.

\subsubsection{Functoriality}
\label{se:functorialitypaths}

Let $f:Z\to Z'$ be a morphism of curves.

\begin{lemma}
	\label{le:smoothf}
Let $\zeta$ be a
homotopy class of paths in $Z$. The class $f(\zeta)$ is smooth if and only
	if $\zeta$ is smooth. If $\zeta$ is admissible, then $f(\zeta)$ is admissible.
\end{lemma}

\begin{proof}
	Given $\gamma$ an oriented path in $Z$, the path $f(\gamma)$ is oriented. It
is smooth if and only $f(\gamma)$ is smooth. This
shows that if $\zeta$ is a smooth (resp. admissible) homotopy class of paths in $Z$,
	then $f(\zeta)$ is smooth (resp. admissible).

Consider now $\zeta$ a homotopy class of paths in $Z$ such that
	$f(\zeta)$ is smooth. Given
	$\gamma$ a minimal path in $\zeta$, then $f(\gamma)$ is minimal
	(Lemma \ref{le:mappaths}). Since $f(\zeta)$ is smooth, it
	follows that $f(\gamma)$ is smooth (Properties \ref{it:smoothadmissible}(4)),
	hence $\gamma$ is
	smooth and finally $\zeta$ is smooth.
\end{proof}

It follows from the previous lemma that
the morphism $f$ induces a functor
$f:\CS^\bullet(Z,1)\to\CS^\bullet(Z',1)$. We have constructed a functor $\CS^\bullet(-,1)$ from the category
of curves to the category of pointed categories.

\smallskip
Let us state a version of Lemma \ref{le:mappaths} for morphisms of curves.

\begin{lemma}
	\label{le:uniqueliftpath}
Let $\gamma,\gamma'$ be two admissible paths in $Z$.
	If $[f(\gamma)]=[f(\gamma')]\neq\id$, then $[\gamma]=[\gamma']$. 
	The functor $f:\CS^\bullet(Z,1)\to\CS^\bullet(Z',1)$ is faithful.
\end{lemma}

Note that $f$ induces an injective morphism of groups $f:L(Z)\to L(Z')$ and a map
$f:\pi_0(Z)\to\pi_0(Z')$.

The next lemma is an immediate consequence of Lemma \ref{le:functorialitym}.
\begin{lemma}
	\label{le:bilinearfunctor}
Given $\alpha,\beta\in R(Z)$, we have
	$\langle f(\alpha),f(\beta)\rangle=f(\langle\alpha,\beta\rangle)$.
\end{lemma}

It follows from Lemmas \ref{le:bilinearfunctor} and \ref{le:finjR}
that we have a morphism
of groups $f:\Gamma(Z)\to\Gamma(Z'),\ (r,(m,\alpha))\mapsto \bigl(f(r),(f(m),f(\alpha))\bigr)$
which restricts
to an injective morphism of groups $\Gamma'(Z)\to\Gamma'(Z')$.

Let $D$ be a subset of $T(Z)$  such that given $z\in \pt(D)$, the composition
$D\cap\pt^{-1}(z)\to C(z)\to C(z)/\iota$ is bijective. The morphism
$f:\Gamma(Z)\to\Gamma(Z')$ induces a morphism $f:\Gamma(Z,D)\to\Gamma(Z',f(D))$.
Let $g,h\in \Gamma(Z,D)$. If $g<h$, then $f(g)<f(h)$. If $f:\pi_0(Z)\to\pi_0(Z')$
is injective and $f(g)<f(h)$, then $g<h$.

Finally, the morphism $f:\Gamma(Z,D)\to\Gamma(Z',f(D))$ induces a morphism
$f:\bar{\Gamma}(Z,D)\to\bar{\Gamma}(Z',f(D))$. Given 
$g,h\in \bar{\Gamma}(Z,D)$, we have $g<h$ if and only if $f(g)<f(h)$.

\medskip
Let $Z_1,\ldots,Z_r$ be the connected components of $Z$. There are
isomorphisms of groups
$R(Z_1)\times\cdots\times R(Z_r)\iso R(Z)$ and
$L(Z_1)\times\cdots\times L(Z_r)\iso L(Z)$ 
given by the inclusions $Z_i\hookrightarrow Z$. They
induce an isomorphism of groups 
\begin{equation}
	\label{eq:Gammacomponents}
	\Gamma(Z_1)\times\cdots\times \Gamma(Z_r)\iso\Gamma(Z).
\end{equation}

The inclusions $Z_i\hookrightarrow Z$ induce pointed functors $\CS^\bullet(Z_i,1)\to\CS^\bullet(Z,1)$ and
give rise to an isomorphism of pointed categories
\begin{equation}
	\label{eq:Acomponents}
\CS^\bullet(Z_1,1)\vee\cdots\vee\CS^\bullet(Z_r,1)\iso\CS^\bullet(Z,1).
\end{equation}

\subsubsection{Pullback}
Let $f:Z\to Z'$ be a morphism of curves.
We define a non-multiplicative ``functor''
$f^\#:\add(\CS(Z',1))\to \add(\CS(Z,1))$\indexnot{f}{f^\#}. It commutes with coproducts
but is not a functor, i.e., it is not compatible with composition for a general $f$.
We put $f^\#(z')=\coprod_{z\in f^{-1}(z')}z$.
Given $\zeta'\in\Hom_{\CS^\bullet(Z',1)}(z'_1,z'_2)$ non-zero, we define 
$f^\#(\zeta')$ to be
\begin{itemize}
	\item $\id$ if $\zeta'=\id$ \\
	\item $0$ if $\zeta'$ does not lift to an admissible class of 
		paths in $Z$ \\
	\item the composition
	$$\coprod_{z\in f^{-1}(z'_1)}z\xrightarrow{\text{projection}}z_1
	\xrightarrow{\zeta}z_2 \xrightarrow{\text{inclusion}} 
	\coprod_{z\in f^{-1}(z'_2)}z$$
		where $\zeta:z_1\to z_2$ is the unique lift of $\zeta'$, otherwise
		(cf Lemma \ref{le:uniqueliftpath}).
\end{itemize}

\smallskip
We denote by $f^{-1}(\zeta')$ the set of admissible lifts of $\zeta'$. We have
$f^\#(\zeta')=\sum_{\zeta\in f^{-1}(\zeta')}\zeta$.

\smallskip
Given $\zeta'_1\in\Hom_{\CS^\bullet(Z',1)}(z'_1,z'_2)$ and
$\zeta'_2\in\Hom_{\CS^\bullet(Z',1)}(z'_2,z'_3)$ such that $f^\#(\zeta'_1)\neq 0$ and
$f^\#(\zeta'_2)\neq 0$, we have
$f^\#(\zeta'_2)f^\#(\zeta'_1)=f^\#(\zeta'_2\zeta'_1)$ (cf Lemma \ref{le:smoothf}).

\medskip
Given $f':Z'\to Z''$ a morphism of curves, we have $(f'f)^\#=f^\#f^{\prime\#}$.

%
%

\begin{lemma}
	\label{le:compositionliftpath}
Let $\gamma'$ be a smooth path in $Z'$. 
Consider the following assertions:
	\begin{enumerate}
		\item $\gamma'$ lifts to a smooth path in $Z$
		\item $\gamma'([0,1])\subset f(Z)$.
		\item $[\gamma']$ lifts to a smooth homotopy class in $Z$
		\item $\supp([\gamma'])\subset f(Z)$.
	\end{enumerate}
	We have $(1)\Leftrightarrow (2)\Rightarrow (3)\Leftrightarrow (4)$.

	Assume $f$ is strict. Then $(3)\Rightarrow (2)$. Furthermore, if
	$\gamma'$ is admissible and it lifts to a smooth path in $Z$, then that path
	is admissible.
\end{lemma}

\begin{proof}
	The implications  $(1)\Rightarrow(2)$, $(1)\Rightarrow(3)\Rightarrow(4)$
	are clear. We can assume that $\gamma'$ is not constant, for otherwise
	the other implications are trivial.

	\smallskip
	Assume $(2)$.
		Let $\hat{f}:\hat{Z}\to\hat{Z}'$ be the map between non-singular
	covers corresponding to $f$.
	Since $\gamma'$ is smooth, it lifts uniquely to a path
	$\hat{\gamma}'$ on $\hat{Z}'$ and $\hat{\gamma}'([0,1])\subset
	\hat{f}(\hat{Z})$. Since $\hat{f}$ is an open embedding, it follows that
	$\hat{\gamma}'$ is the image of a path of $\hat{Z}$. Its image in
	$Z$ is a smooth path that lifts $\gamma'$, hence $(1)$ holds.

	\smallskip
	Assume $(4)$.
	Let $\gamma'_0$ be a minimal smooth path homotopic to $\gamma'$
	(cf Properties \ref{it:smoothadmissible}(4)). We
	have $\gamma'_0([0,1])=\supp([\gamma'])\subset f(Z)$, hence
	$\gamma'_0$ lifts to a smooth path in $Z$. So $(3)$ holds.

	\smallskip
	Assume $(3)$ and $f$ is strict. 
	Note that $\gamma'([0,1])\cap Z'_o=\supp([\gamma'])\cap Z'_o$ and 
	$\gamma'([0,1])\cap Z'_u$ is contained in the union of the connected
	components of $Z'_u$ that have a non-empty intersection with 
	$\supp([\gamma'])$ (Properties \ref{it:orientedpaths}(1)).
	Since $f(Z_u)$ is open and closed in $Z'_u$, it follows that
	$\gamma'([0,1])\subset f(Z)$, so $(2)$ holds.

	Assume $\gamma'$ is admissible and lifts to $Z$. Since $f$ is strict, it follows
	that the lift is oriented.
\end{proof}

Since quotient maps are strict, we have the following consequence of 
Lemma \ref{le:compositionliftpath}.

\begin{lemma}
	\label{le:liftquotientpaths}
	Assume $f$ is the quotient map of $Z$ by a finite relation.
	Every non-constant admissible path in $Z'$ lifts uniquely to a path in $Z$
	and that lift is admissible.
\end{lemma}

\begin{prop}
	\label{pr:functorfsharp}
	If $f$ is strict, then
$f^\#:\add(\CS(Z',1))\to \add(\CS(Z,1))$ is a functor.
\end{prop}

\begin{proof}
	We need to check that $f^\#$ is compatible with composition. 
	This is clear if $Z$ and $Z'$ are non-singular. In general, consider
	two maps $\zeta'_1$ and $\zeta'_2$ in $\CS(Z',1)$ such that
	$f^\#(\zeta'_2\circ\zeta'_1)\neq 0$.
	Let $\hat{f}:\hat{Z}\to\hat{Z}'$ be the map corresponding to $f$
	between non-singular covers $q:\hat{Z}\to Z$ and $q':\hat{Z}'\to Z'$.

	We have 
	\begin{align*}
		q^\#f^\#(\zeta'_2\circ\zeta'_1)&=\hat{f}^\#q^{\prime
	\#}(\zeta'_2\circ\zeta'_1)=\hat{f}^\#\bigl(q^{\prime
	\#}(\zeta'_2)\circ q^{\prime\#}(\zeta'_1)\bigr)=
	\bigl(\hat{f}^\#q^{\prime
	\#}(\zeta'_2)\bigr)\circ \bigl(\hat{f}^\#q^{\prime\#}(\zeta'_1)\bigr)\\
		&= q^\#f^\#(\zeta'_2)\circ q^\#f^\#(\zeta'_1),
	\end{align*}
	hence $f^\#(\zeta'_1)\neq 0$ and $f^\#(\zeta'_2)\neq 0$ since
	$q^\#f^\#(\zeta'_2\circ\zeta'_1)\neq 0$ by Lemma \ref{le:liftquotientpaths}.
	It follows that $f^\#(\zeta'_2\circ\zeta'_1)=f^\#(\zeta'_2)\circ f^\#(\zeta'_1)$.
\end{proof}

The construction $Z\mapsto \add(\CS(Z,1))$ and
$f\mapsto f^\#$ defines a contravariant
functor from the category
of curves with strict morphisms to the category of $\BF_2$-linear categories.

\medskip

Lemma \ref{le:liftquotientpaths} and
Proposition \ref{pr:functorfsharp} have the following consequence.

\begin{prop}
	\label{pr:qsharpfaithful}
	Let $Z$ be a curve with an admissible relation $\sim$ and
let $q:Z\to Z/\!\!\sim$ be the quotient map.
	The functor $q^\#:\add(\CS(Z/\!\!\sim,1))\to\add(\CS(Z,1))$
is faithful.
\end{prop}

Note that Proposition \ref{pr:qsharpfaithful} provides an identification of
$\CS(Z/\!\!\sim,1)$ with a (non-full) subcategory of $\add(\CS(Z,1))$.

\begin{example}
	We describe the image by the map $f^\#$ of two paths, the first
	of which is the constant path at
	the singular point of $Z_{exc}$
	(we draw the lifts in the non-singular cover).
$$\includegraphics[scale=0.80]{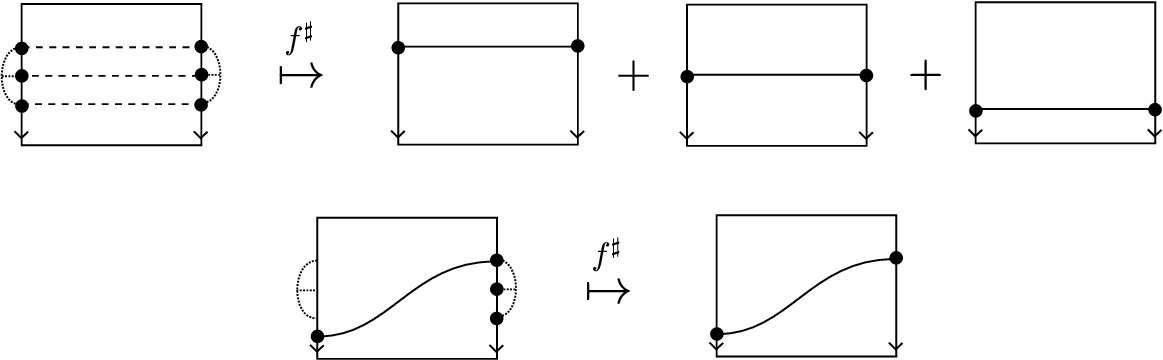}$$
\end{example}

\subsubsection{One strand bordered algebras}

Consider a chord diagram $(\CZ,\Ba)$ as in
\S\ref{se:arcdiagrams} with associated singular curve $Z$.
Define
$$\CA(Z,1)=\End_{\add(\CS(Z,1))}(\bigoplus_{
	\substack{z\in Z_{exc}}}z).$$
Proposition \ref{pr:qsharpfaithful} shows that the algebra $\CA(Z,1)$ is
the opposite of
Zarev's algebra $\CA_{Za}(\CZ,1)$ \cite[Definition 2.6]{Za} (this will be
explained for the more general algebras $\CA(Z)$ in \S\ref{se:HF}).

\medskip
$\bullet\ $Consider the chord diagram $(\BR,\{\{1,3\},\{2,4\}\})$.

The associated singular curve $Z$ is the quotient of oriented $\BR$ by
the relation whose non-trivial equivalence classes are
$1=\{1,3\}$ and $2=\{2,4\}$.

The full pointed subcategory of $\CS(Z,1)$ with object set $\{1,2\}$  is 
generated by $\alpha,\alpha':1\to 2$ and $\beta:2\to 1$ with relations
$\beta\alpha=\alpha'\beta=0$. This corresponds to the well-known
``torus algebra" in bordered Floer homology.

$$\includegraphics[scale=0.8]{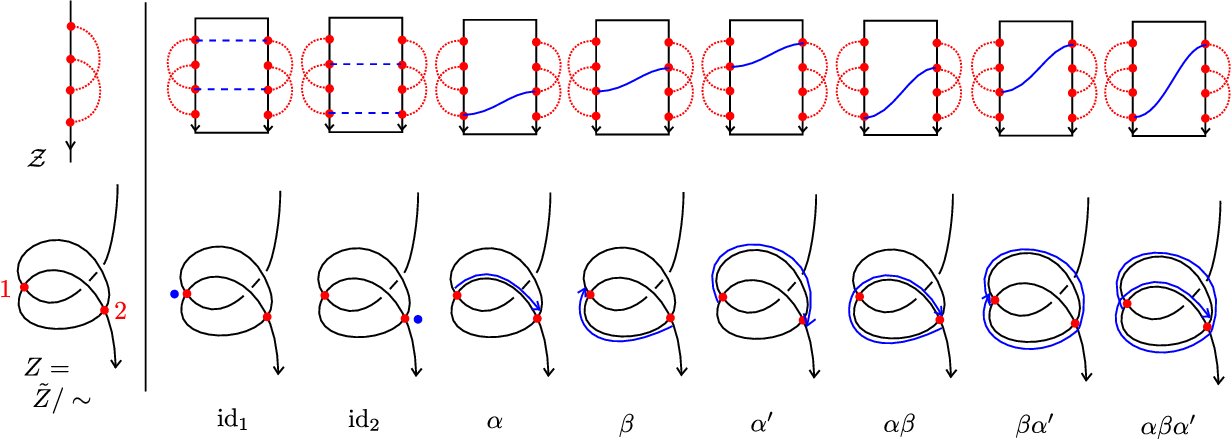}$$

\medskip
$\bullet\ $Consider the chord diagram $(S^1,\{\{\pm 1\},\{\pm i\}\})$.

The associated singular curve $Z$ is the quotient of oriented $S^1$ by
the relation whose non trivial equivalence classes are
$1=\{\pm 1\}$ and $2=\{\pm i\}$.

The full pointed subcategory of $\CS(Z,1)$ with object set $\{1,2\}$  is 
generated by $\alpha,\alpha':1\to 2$ and $\beta,\beta':2\to 1$ with relations
$\beta\alpha=\alpha'\beta=\alpha\beta'=\beta'\alpha'=0$.
A curved $A_{\infty}$-deformation of this subcategory appears in \cite{LiOzTh3}.

\smallskip
We have 
$$\End_{\CS(Z,1)}(1)=\{\id\}\sqcup\{(\beta'\alpha\beta\alpha')^n\}_{n\ge 1}
\sqcup\{(\beta'\alpha\beta\alpha')^n\beta'\alpha\}_{n\ge 0}
\sqcup\{\beta\alpha'(\beta'\alpha\beta\alpha')^n)\}_{n\ge 0}
\sqcup\{(\beta\alpha'\beta'\alpha)^n)\}_{n\ge 1}$$
$$\End_{\CS(Z,1)}(2)=\{\id\}\sqcup\{(\alpha'\beta'\alpha\beta)^n\}_{n\ge 1}
\sqcup\{(\alpha'\beta'\alpha\beta)^n\alpha'\beta'\}_{n\ge 0}
\sqcup\{\alpha\beta(\alpha'\beta'\alpha\beta)^n)\}_{n\ge 0}
\sqcup\{(\alpha\beta\alpha'\beta')^n)\}_{n\ge 1}$$
$$\Hom_{\CS(Z,1)}(1,2)=\{\alpha'(\beta'\alpha\beta\alpha')^n\}_{n\ge 0}
\sqcup\{\alpha\beta\alpha'(\beta'\alpha\beta\alpha')^n\}_{n\ge 0}
\sqcup\{\alpha'\beta'\alpha(\beta\alpha'\beta'\alpha)^n)\}_{n\ge 0}
\sqcup\{\alpha(\beta\alpha'\beta'\alpha)^n)\}_{n\ge 0}$$
$$\Hom_{\CS(Z,1)}(2,1)=\{(\beta'\alpha\beta\alpha')^n\beta'\}_{n\ge 0}
\sqcup\{(\beta'\alpha\beta\alpha')^n\beta'\alpha\beta'\}_{n\ge 0}
\sqcup\{(\beta\alpha'\beta'\alpha)^n)\beta\}_{n\ge 0}
\sqcup\{(\beta\alpha'\beta'\alpha)^n)\beta\alpha'\beta'\}_{n\ge 0}$$

$$\includegraphics[scale=0.9]{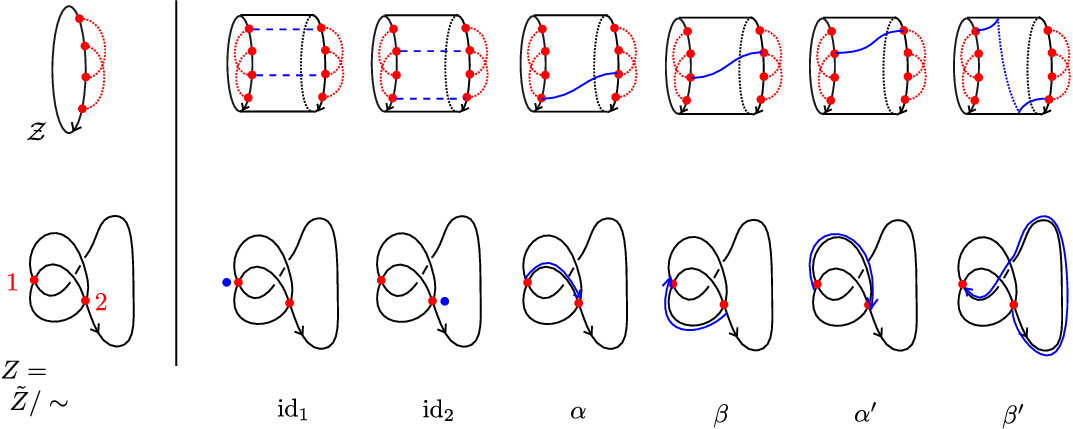}$$

\subsubsection{Intersection multiplicity}
Let $\gamma$ and $\gamma'$ be two paths in $Z$. We consider the number
of intersection points between the graphs of $\gamma$ and $\gamma'$
$$i(\gamma,\gamma')=
|\{t\in [0,1]\ |\ \gamma(t)=\gamma'(t)\}|\in\BZ_{\ge 0}\cup\{\infty\}\indexnot{iz}{i(\zeta,\zeta')}.$$
Note that $i(\gamma_1\circ\gamma_2,\gamma'_1\circ\gamma'_2)=i(\gamma_1,\gamma'_1)
+i(\gamma_2,\gamma'_2)-\delta_{\gamma_1(0)=\gamma'_1(0)}$.

\smallskip
Given $\zeta$ and $\zeta'$ two admissible homotopy classes of paths in $Z$, we put
$$i(\zeta,\zeta')=\min_{\gamma,\gamma'}i(\gamma,\gamma'),$$
where $\gamma$ (resp. $\gamma'$) runs over admissible paths in $[\zeta]$
(resp. in $[\zeta']$).
Note that $i(\zeta_1\zeta_2,\zeta'_1\zeta'_2)\le i(\zeta_1,\zeta'_1)+
i(\zeta_2,\zeta'_2)-\delta_{\zeta_1(0)=\zeta'_1(0)}$.

\medskip
The next lemma relates the intersection multiplicity with a constant path and
tangential multiplicities.

\begin{lemma}
	\label{le:m=i}
Let $\gamma_0$ be a minimal admissible path in $Z$ and let $z\in Z$. We have
	$$i([\gamma_0],\id_z)= \min_{\substack{\gamma \text{ admiss.}\\
	[\gamma]=[\gamma_0]}}i(\gamma,\id_z)=
		i(\gamma_0,\id_z) 
		=\frac{1}{2}\bigl(\sum_{c\in C(z)}(m_c^+([\gamma_0])+m_c^-([\gamma_0]))+
	\delta_{\gamma_0(0)=z}+\delta_{\gamma_0(1)=z}\bigr).$$
	If $z\in Z_o$, then we have
	$$i([\gamma_0],\id_z)=\frac{1}{2}\bigl(\sum_{c\in C(z)^+}(m_c([\gamma_0])-m_{\iota(c)}([\gamma_0]))+
	\delta_{\gamma_0(0)=z}+\delta_{\gamma_0(1)=z}\bigr).$$

\end{lemma}

\begin{proof}
Note that
$$i([\gamma_0],\id_z)\le \min_{\substack{\gamma \text{ admiss.}\\
	[\gamma]=[\gamma_0]}}i(\gamma,\id_z)\le i(\gamma_0,\id_z).$$
The third equality of the lemma follows from Lemma \ref{le:multiplicityset}.

\smallskip
	 When $Z=S^1$ unoriented, the lemma follows from Lemma \ref{le:intersectionlength}.

	\smallskip
	When $Z$ is a connected non-singular curve, there is
	an injective morphism of curves $f:Z\to S^1$. We have
	$i([\gamma_0],\id_z)\ge i(f([\gamma_0]),\id_{f(z)})=
	i(f(\gamma_0),\id_{f(z)})=i(\gamma_0,\id_z)$,
	hence the first two equalities of the lemma hold for $Z$.
	It follows that they hold for any non-singular curve.

	\smallskip
	Consider now a general $Z$ and let $q:\hat{Z}\to Z$ be the non-singular
	cover. Let $\hat{\gamma}_0$ be the lift of $\gamma_0$ to $\hat{Z}$. We have
	$$i(\gamma_0,\id_z)=\sum_{\hat{z}\in f^{-1}(z)}i(\hat{\gamma}_0,\id_{\hat{z}})=
	\sum_{\hat{z}\in f^{-1}(z)}([\hat{\gamma}_0],\id_{\hat{z}})\le
	i([\gamma_0],\id_z).$$
	We deduce that the first two equalities of the lemma hold.

	\smallskip
	The last equality of the lemma follows from (\ref{eq:mzeta+}).
\end{proof}

\smallskip
Let us now state some basic properties of intersection counts.

\begin{lemma}
	\label{le:intersections}
	Let $\zeta_1$ and $\zeta_2$ be two admissible homotopy classes of paths in $Z$.
	Assume $\zeta_1(t)\neq \zeta_2(t)$ for $t\in\{0,1\}$.

	\begin{enumerate}
				\item We have $i(\zeta_1,\zeta_2)<\infty$.
				\item There are minimal or identity admissible
	paths $\gamma_1$ in $\zeta_1$ and $\gamma_2$ in $\zeta_2$ such that
	$i(\zeta_1,\zeta_2)=i(\gamma_1,\gamma_2)$.
\item Given $f:Z'\to Z$ a morphism of curves such that
	$\zeta_1$ and $\zeta_2$ are images of admissible homotopy
			classes of paths in $Z'$,
	we have	$i(\zeta_1,\zeta_2)=\sum_{\zeta'_i\in f^{-1}(\zeta_i)}
	i(\zeta'_1,\zeta'_2)$.
	\end{enumerate}

\end{lemma}

\begin{proof}
		$\bullet\ $  Assume $\zeta_1$ or $\zeta_2$ is an identity.
		In that case, (1) and (2) follow from Lemma \ref{le:m=i} and
		(3) follows from Lemmas \ref{le:functorialitym}
		and \ref{le:m=i}.

		\medskip
		From now on, we assume that neither $\zeta_1$ nor $\zeta_2$ is
		an identity.

		\medskip
		$\bullet\ $  Let $f:Z\to Z'$ be an injective morphism of curves and 
		assume $f(\zeta_1)$ and $f(\zeta_2)$ satisfy (1) and (2).
	We have $i(f(\zeta_1),f(\zeta_2))\le i(\zeta_1,\zeta_2)$.
	There are minimal admissible paths $\gamma'_i$ in $f(\zeta_i)$ for $i\in\{1,2\}$
	such that
	$i(f(\zeta_1),f(\zeta_2))=i(\gamma'_1,\gamma'_2)$. There are admissible paths
	$\gamma_i$ of $Z$ such that $\gamma'_i=f(\gamma_i)$ for $i\in\{1,2\}$.
	It follows that
	$i(\zeta_1,\zeta_2)\ge i(\gamma'_1,\gamma'_2)=i(\gamma_1,\gamma_2)$,
	hence $i(f(\zeta_1),f(\zeta_2))=i(\zeta_1,\zeta_2)$.
	We deduce also that (1) and (2) hold for $\zeta_1$ and $\zeta_2$.

	\smallskip
	$\bullet\ $ Assume $Z=S^1$ unoriented. The assertions (1) and (2) follow from
	Lemma \ref{le:intersectionlength}.

	\smallskip
	$\bullet\ $ Assume $Z$ is non-singular and connected. There is an
	injective map $f:Z\to S^1$. It follows that $Z$ satisfies (1) and (2).
	This shows that (1) and (2) hold for a general non-singular curve.

	\smallskip
	$\bullet\ $ Let $Z'$ be an arbitrary curve and let $f:Z\to Z'$
	be the non-singular cover of $Z'$.
	Assume $f(\zeta_1(t))\neq f(\zeta_2(t))$ for $t\in\{0,1\}$.
	Since all admissible paths in $Z'$ lift to $Z$, it follows
	that $i(\zeta_1,\zeta_2)\le i(f(\zeta_1),f(\zeta_2))$.

	Consider two minimal admissible
	paths $\gamma_1$ and $\gamma_2$ in $\zeta_1$ and $\zeta_2$
	such that $i(\gamma_1,\gamma_2)=i(\zeta_1,\zeta_2)$. We assume that given
	$\rho_1,\rho_2:[0,1]\iso [0,1]$ any two homeomorphisms fixing $0$ and $1$
	and such that $i(\gamma_1\circ\rho_1,\gamma_2\circ\rho_2)=i(\zeta_1,\zeta_2)$,
	we have $i(f(\gamma_1),f(\gamma_2))\le
	i(f(\gamma_1\circ \rho_1),f(\gamma_2\circ\rho_2))$.
	Let $t_0\in (0,1)$ such that $\gamma_1(t_0)\neq\gamma_2(t_0)$ but
	$f(\gamma_1(t_0))=f(\gamma_2(t_0))$.
	There is a small open neighbourhood $U$ of $z'=f(\gamma_1(t_0))$
	homeomorphic to $\St(n_{z'})$ and
	with $U\cap f(Z_f)=\{z'\}$ and there are $0\le t_1<t_0<t_2\le 1$
	such that 
	$f(\gamma_1)([t_1,t_2])\subset U$ and $f(\gamma_2)([t_1,t_2])\subset U$.
	The paths $(\gamma_1)_{|[t_1,t_2]}$ and $(\gamma_2)_{|[t_1,t_2]}$
	are contained in disjoint connected components of $f^{-1}(U)$, 
	hence $f(\gamma_1)([t_1,t_2])\cap f(\gamma_2)([t_1,t_2])=\{z'\}$. So, by
	reparametrizing
	$f(\gamma_1)$ and $f(\gamma_2)$ in the interval $[t_1,t_2]$, we can assume they
	do not have a common value in that interval. This contradicts the minimality
	of $i(f(\gamma_1),f(\gamma_2))$. It follows that 
	$$i(\zeta_1,\zeta_2)=i(\gamma_1,\gamma_2)=i(f(\gamma_1),f(\gamma_2))\ge
	i(f(\zeta_1),f(\zeta_2)),$$
	hence $i(\zeta_1,\zeta_2)=i(f(\zeta_1),f(\zeta_2)).$
	This shows that (1) and
	(2) hold for $f(\gamma_1)$ and $f(\gamma_2)$. We deduce that
	(1) and (2) hold in full generality. It follows also that (3) holds
	when $f$ is injective.

	\smallskip
	$\bullet\ $ Consider now a morphism of curves $f:Z\to Z'$.
	Consider the map $\hat{f}:\hat{Z}\to\hat{Z}'$ between non-singular
	covers corresponding to $f$. 
	Let $\hat{\zeta}_i$ be the lift of $\zeta_i$ to $\hat{Z}$.
	Since $\hat{f}$ is injective, it follows that
	$i(\hat{f}(\hat{\zeta}_1),\hat{f}(\hat{\zeta}_2))=
	i(\hat{\zeta}_1,\hat{\zeta}_2)$. The study above shows that
	$i(\hat{f}(\hat{\zeta}_1),\hat{f}(\hat{\zeta}_2))=
	i(f(\zeta_1),f(\zeta_2))$ and 
	$i(\hat{\zeta}_1,\hat{\zeta}_2)=i(\zeta_1,\zeta_2)$. It follows
	that $i(f(\zeta_1),f(\zeta_2))=i(\zeta_1,\zeta_2)$. This completes the proof of the
	lemma.
\end{proof}

We provide now an upper bound for intersections involving a composition of
paths.

\begin{lemma}
\label{le:boundintercompo}
	Consider $\zeta$, $\zeta_1$ and $\zeta_2$ three homotopy classes
	of admissible paths in $Z$. Assume $\zeta$ is not an identity,
	$\zeta_2(1)=\zeta_1(0)$, $\zeta(0)\neq\zeta_2(0)$ and
	$\zeta(1)\neq\zeta_1(1)$.
	We have
	$$i(\zeta,\zeta_1\circ\zeta_2)\le
	\mathrm{min}(m_{\zeta(0+)}^+(\zeta_2)+i(\zeta,\zeta_1),
	m_{\zeta(1-)}^-(\zeta_1)+i(\zeta,\zeta_2)).$$
\end{lemma}

\begin{proof}
	Let $\zeta'$ and $\zeta''$ be homotopy classes of admissible 
	paths such that $\zeta=\zeta'\circ\zeta''$. We have
	$$i(\zeta,\zeta_1\circ\zeta_2)\le i(\zeta',\zeta_1)+i(
	\zeta'',\zeta_2).$$
	Let $\gamma$ be a minimal path in $\zeta$ and let $t\in (0,1)$.
	We have $m_{\zeta(0)^+}(\zeta_2)=i(\gamma_{|[0,t]},
	\zeta_2)$ for $t$ small enough. Since
	$i([\gamma_{[t,1]},\zeta_1)\le i(\zeta,\zeta_1)$, it follows that
	$$i(\zeta,\zeta_1\circ\zeta_2)\le i(\zeta,\zeta_1)+
	m_{\zeta(0)^+}(\zeta_2).$$

	The second inequality follows from the first one by replacing
	$Z$ by $Z^{\mathrm{opp}}$.
\end{proof}

\medskip
Recall that we denote by $\Pi(Z)$ the fundamental groupoid of $Z$.
Consider $\zeta_1,\zeta_2$ two admissible homotopy classes of paths in $Z$
with $\zeta_1(t)\neq\zeta_2(t)$ for $t\in\{0,1\}$.

Let $I(\zeta_1,\zeta_2)\indexnot{Iz}{I(\zeta_1,\zeta_2)}$ be the set of  non-identity
classes $\zeta\in\Hom_{\Pi(Z)}(\zeta_1(0),\zeta_2(0))$ such that
\begin{itemize}
	\item[(i)]
$\zeta$, $\zeta_2\circ\zeta$ and $\zeta\circ\zeta_1^{-1}$ are smooth
\item[(ii)] $\zeta$ and $\bar{\zeta}:=\zeta_2\circ\zeta\circ\zeta_1^{-1}$\indexnot{zeta}{\bar{\zeta}} have opposite orientations (cf Definition \ref{de:opposite}).
\end{itemize}

Note that there are bijections
$$\mathrm{inv}:I(\zeta_1,\zeta_2)\iso 
I(\zeta_2,\zeta_1),\ \zeta\mapsto\zeta^{-1} \text{ and }
I(\zeta_1,\zeta_2)\iso I(\zeta_1^{-1},\zeta_2^{-1}),\
\zeta\mapsto \bar{\zeta}\indexnot{inv}{\mathrm{inv}}.$$

If $Z$ is non-singular, then the condition (i) in the definition of $I(\zeta_1,\zeta_2)$
is automatically satisfied.

\medskip
Let  $f:Z\to Z'$ be a morphism of curves. 
If $f(\zeta_1(t))\neq f(\zeta_2(t))$ for $t\in\{0,1\}$, then
the map $f$ induces an injection $I(\zeta_1,\zeta_2)\hookrightarrow
	I(f(\zeta_1),f(\zeta_2))$ with image
	$f\bigl(\Hom_{\Pi(Z)}(\zeta_1(0),\zeta_2(0))\bigr)\cap I(f(\zeta_1),f(\zeta_2))$.

	\smallskip
The next lemma is immediate.

\begin{lemma}
	\label{le:coverI}
	Let $q:\hat{Z}\to Z$ be the non-singular cover of $Z$.
	The map $q$ induces a bijection
	$$\coprod_{\hat{\zeta}_i\in q^{-1}(\zeta_i)}I(\hat{\zeta}_1,\hat{\zeta_2})
	\iso I(\zeta_1,\zeta_2).$$
	\end{lemma}

\begin{lemma}
	\label{le:supportI}
If $\zeta\in I(\zeta_1,\zeta_2)$, then
$\supp(\zeta)\subset \supp(\zeta_1)\cup\supp(\zeta_2)$.
\end{lemma}

\begin{proof}
	Consider three non-identity homotopy classes of paths $\zeta$, $\zeta_1$ and
	$\zeta_2$ in $\BR$ with $\zeta(0)=\zeta_1(0)\neq\zeta(1)=\zeta_2(0)$.
	If $\zeta$ and $\zeta_2\circ\zeta\circ\zeta_1^{-1}$ have opposite
	orientations, then $\supp(\zeta)\subset \supp(\zeta_1)\cup\supp(\zeta_2)$.
	We deduce that the lemma holds for $Z=S^1$ by using the universal cover of $Z$.
	As a consequence, the lemma holds when $Z$ is connected and smooth by embedding
	it in $S^1$, hence it holds for $Z$ smooth. Lemma \ref{le:coverI}
	shows that the lemma holds for any $Z$, since it holds for
	the non-singular cover of $Z$.
\end{proof}

\begin{example}
	In the two examples below, we describe 
	the set $I(\zeta_1,\zeta_2)$. In the second example, $\zeta_2$ is the
	identity at the singular point.
$$\includegraphics[scale=1]{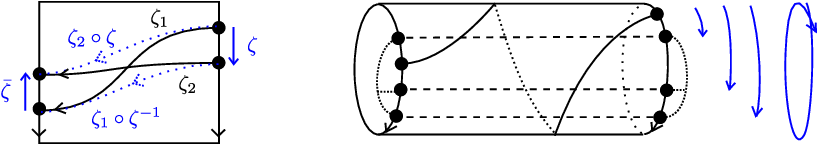}$$
\end{example}

\subsection{Strands}
\label{se:strands}

\subsubsection{Braids}
\label{se:braids}
Let $Z$ be a curve.
Let $I$ and $J$ be two finite subsets of $Z$. 

\begin{defi}
	A {\em parametrized braid}\index[ter]{parametrized braid} $I\to J$ is a
family $\vartheta=(\vartheta_s)_{s\in I}$ where $\vartheta_s$ is an
admissible path in $Z$ with $\vartheta_s(0)=s$ and such that
	$s\mapsto \vartheta_s(1)$ defines a bijection $\bij{\vartheta}:I\iso J$
	\indexnot{ch}{\bij{\theta}}.
A {\em braid}\index[ter]{braid} $I\to J$ is a homotopy class of parametrized braids, \ie, a family
of admissible homotopy classes of paths.
\end{defi}

\begin{defi}
	We define the pre-strand category $\CP^\bullet(Z)=S(\CS^\bullet(Z,1))$\indexnot{P}{\CP^\bullet(Z)} (cf \S\ref{se:symmetricpowers}).
\end{defi}

The objects of this pointed
category are the finite subsets of $Z$ and
$\Hom_{\CP^\bullet(Z)}(I,J)$ is the set of braids $I\to J$, together with a 
$0$-element.  Given $\theta:I\to J$ and $\theta':J\to K$ two braids,
we have $\theta'\circ\theta=(\theta'_{\theta_s(1)}\circ\theta_s)_{s\in I}$ if
$\theta'_{\theta_s(1)}\circ\theta_s$ is admissible for all $s\in I$, and
we have $\theta'\circ\theta=0$ otherwise. If $\theta'\circ\theta\neq 0$, we have
$\bij{\theta'\circ\theta}=\bij{\theta'}\circ\bij{\theta}$.

\smallskip
We put $\CP(Z)=\BF_2[\CP^\bullet(Z)]$.\indexnot{P}{\CP(Z)}

\smallskip
Note that there is a decomposition
$\CP^\bullet(Z)=\bigvee_{n\ge 0}\CP^\bullet(Z,n)$, where
$\CP^\bullet(Z,n)$ is the full subcategory of $\CP^\bullet(Z)$ with objects subsets with $n$ elements. We have $\CP^\bullet(Z,1)=\CS^\bullet(Z,1)$.

\smallskip
Given $M$ a subset of $Z$, we denote by $\CP^\bullet_M(Z)$\indexnot{P}{\CP^\bullet_M(Z)} the full subcategory of
$\CP^\bullet(Z)$ with objects the finite subsets of $M$.

\smallskip
Given $\theta:I\to J$ a braid and $I'$ a subset of $I$, we denote by 
$\theta_{|I'}$ the braid $(\theta_s)_{s\in I'}$.


\medskip
Let $f:Z\to Z'$ be a morphism of curves. We denote by
$\CP^\bullet_f(Z)$\indexnot{P}{\CP^\bullet_f(Z)} the full subcategory of $\CP^\bullet(Z)$ with objects
those finite subsets $I$ of $Z$ such that $|f(I)|=|I|$.


\smallskip
The next proposition follows immediately from Lemma \ref{le:uniqueliftpath}
and \S\ref{se:symmetricpowers}.

\begin{prop}
	\label{pr:finjective}
	The functor $f:\CS^\bullet(Z,1)\to\CS^\bullet(Z',1)$ defines a faithful pointed functor
	$$f:\CP^\bullet_f(Z)\to\CP^\bullet(Z'),\ I\mapsto f(I),\ \theta\mapsto (f(\theta_{s}))_{f(s)}.$$
In particular if $f:Z\to Z'$ is injective then we have
a faithful pointed functor $f:\CP^\bullet(Z)\to\CP^\bullet(Z')$.
\end{prop}

\medskip
We define a non-multiplicative 
$f^\#:\add(\CP(Z'))\to\add(\CP(Z))$\indexnot{f}{f^\#}
that commutes with coproduct.
Given $I'$ a finite subset of $Z'$, we put
$$f^\#(I')=\coprod_{p:I'\to Z,\ fp=\id_{I'}}p(I').$$
Consider now $\theta'\in\Hom_{\CP^\bullet(Z')}(I',J')$ non-zero.
Given $s'\in I'$, we have a decomposition
$f^\#(\theta'_{s'})=\sum_{s\in f^{-1}(s')}
f^\#(\theta'_{s'})_s$ along the decomposition $f^\#(s')=\bigoplus_{s\in f^{-1}(s')}s$
(cf \S \ref{se:functorialitypaths}).
Given $p:I'\to Z$ with $fp=\id_{I'}$, we put $f^\#_p(\theta')=
\bigl(f^\#(\theta'_{f(s)})_s\bigr)_{s\in p(I')}$, a map in $\CP(Z)$ with source $p(I')$.

We define
$$f^\#(\theta')=\sum_{p:I'\to Z,\ fp=\id_{I'}}f^\#_p(\theta').$$

Note that $f^\#(\theta')=\sum_{\theta\in f^{-1}(\theta')}\theta$, where
$f^{-1}(\theta')$ is the set of braids in $Z$ lifting $\theta$.

\smallskip
Given $f':Z'\to Z''$ a morphism of curves, we have $(f'f)^\#=f^\#f^{\prime\#}$.

\medskip
The next two propositions are immediate consequences of
Propositions \ref{pr:functorfsharp} and \ref{pr:qsharpfaithful}
(cf \S\ref{se:symmetricpowers}).

\begin{prop}
	\label{pr:fsharpsprestrands}
	If $f$ is strict, then
	$f^\#$ defines a functor $\add(\CP(Z'))\to\add(\CP_f(Z))$ commuting with
	coproducts.
\end{prop}

\begin{prop}
	\label{pr:qprestrands}
Let $Z$ be a curve with a finite admissible relation $\sim$ and
let $q:Z\to Z/\!\!\sim$ be the quotient map. 
	The functor $q^\#:\add(\CP(Z/\!\!\sim))\to\add(\CP_q(Z))$ is 
	faithful
	and every map in $\CP^\bullet(Z/\!\!\sim)$ is in the image of
	the functor $q:\CP_q^\bullet(Z)\to\CP^\bullet(Z/\!\!\sim)$.
\end{prop}


\medskip
Note that the construction $Z\mapsto \add(\CP(Z))$ and
$f\mapsto f^\#$ defines a contravariant
functor from the category
of curves with strict morphisms to the category of $\BF_2$-linear categories.

\medskip
Let $Z_1,\ldots,Z_r$ be the connected components of $Z$. 
The isomorphism (\ref{eq:Acomponents}) induces an isomorphism of pointed categories
\begin{equation}
	\label{eq:Pcomponents}
\CP^\bullet(Z_1)\wedge\cdots\wedge\CP^\bullet(Z_r)\iso \CP^\bullet(Z).
\end{equation}

Note that the inverse functor sends a braid $\theta:I\to J$ in $Z$
to $(\theta_1,\ldots,\theta_r)$, where $\theta_i$ is the restriction of
$\theta$ to $I\cap Z_i$.

\begin{example}
We describe below an example of product in $\CP^\bullet(Z)$.
$$\includegraphics[scale=0.85]{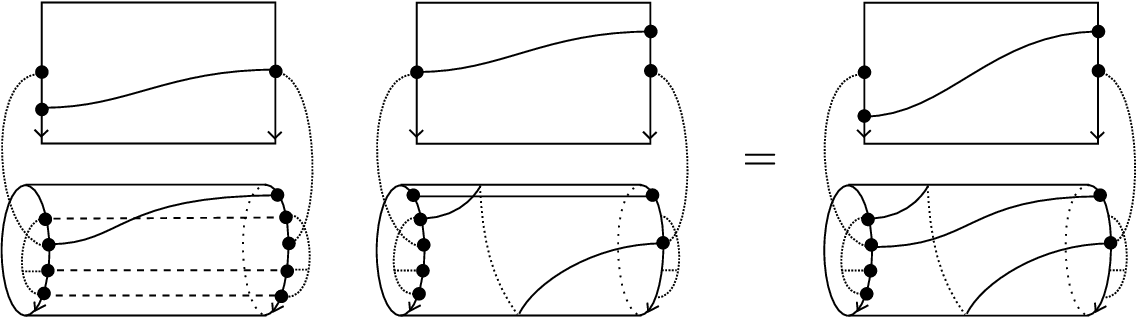}$$
\end{example}
	
\subsubsection{Degree}

Consider $\theta:I\to J$ a braid.
We put
$$i(\theta)=\frac{1}{2}\sum_{\Omega\in\pi_0(Z)}
\sum_{\substack{s\neq s'\in I\cap\Omega}}
i(\theta_s,\theta_{s'})e_\Omega\in(\BZ_{\ge 0})^{\pi_0(Z)}\indexnot{i}{i(\theta)}$$

\medskip
We define
$\llbracket\theta\rrbracket=\sum_{s\in I}\llbracket\theta_s\rrbracket\in R(Z)\indexnot{[}{\protect\llbracket\theta\protect\rrbracket}$ and
$$m(\theta)=\sum_{s\in I}
\sum_{c\in \theta_s(0+)\cup\iota(\theta_s(0+))} m_c(\llbracket\theta\rrbracket)e_c\in L(Z)\indexnot{m}{m(\theta)}.$$

Finally, we define $\deg'(\theta)\in\Gamma(Z)$ by
$$\deg'(\theta)=(-i(\theta),(-m(\theta),-\llbracket\theta\rrbracket))\indexnot{d}{\deg'(\theta)}.$$

Given $D\subset T(Z)$ with $D\cap\iota(D)=\emptyset$,
we denote by $\deg_D(\theta)$\indexnot{d}{\deg_D(\theta)} the image of $\deg'(\theta)$ in
$\Gamma(Z,D)$. Note that if $D'\subset D$, then 
$\deg_D(\theta)$ is the image of $\deg_{D'}(\theta)\in \Gamma(Z,D')$
in $\Gamma(Z,D)$.

\smallskip
We put $\deg(\theta)=\deg_{Z_{exc}^+}(\theta)$\indexnot{d}{\deg(\theta)} and we denote by
$\overline{\deg}(\theta)$\indexnot{d}{\overline{\deg}(\theta)}
(resp.  $\overline{\deg}_D(\theta)$\indexnot{d}{\overline{\deg}_D(\theta)})
the image of $\deg(\theta)$ (resp. $\deg_D(\theta)$) in 
$\bar{\Gamma}(Z,Z_{exc}^+)$
(resp. $\bar{\Gamma}(Z,D)$).


\begin{lemma}
	\label{le:removeid}
	Let $\theta:I\to J$ be a braid in $Z$. 
	Let $E$ be a subset of $\{s\in I\cap Z_o\ |\ \theta_s=\id_s\}$ and let
$\bar{\theta}= (\theta_s)_{s\in I-E}$. 
	We have $\deg_{E^+}(\theta)=\deg_{E^+}(\bar{\theta})$.
\end{lemma}

\begin{proof}
	Note that $\llbracket\theta\rrbracket= \llbracket\bar{\theta}\rrbracket$.
	Let $s\in E$. We have
	$$\sum_{c\in C(s)}m_c(\llbracket\theta\rrbracket)e_c=
	\sum_{c\in C(s)^+}\sum_{s'\in I,\ s'\neq s}
	(m_c-m_{\iota(c)})(\llbracket\theta_{s'}\rrbracket)e_c\xrightarrow{e_c\to 1}
	2\sum_{s'\in I,\ s'\neq s}i(\id_s,\theta_{s'})$$
	by Lemma \ref{le:m=i}.
	The lemma follows.
\end{proof}


\begin{rem}
	\label{re:decomposition I}
	Note that
	$i(\theta)=\sum_{I'\subset I,\ |I'|=2}i(\theta_{|I'})$.
\end{rem}

The next lemma shows that the failure of multiplicativity of $\deg$ and $i$
coincide up to terms involving points in $Z_{exc}$.

\begin{lemma}
	\label{le:degproduct}
Let $\theta:I\to J$ and $\theta':I'\to I$ be two braids such that $\theta\circ\theta'$ is a
braid. The element
		$\deg(\theta)\cdot\deg(\theta')\cdot\deg(\theta\circ\theta')^{-1}$ of
		$\Gamma(Z,Z_{exc}^+)$ is in $\bigoplus_\Omega \frac{1}{2}\BZ e_\Omega$
		and it is equal to
	\begin{multline*}
	i(\theta\circ\theta')-i(\theta)-i(\theta')\\
		+\frac{1}{2}\sum_{\substack{\Omega,\ s'\in I'\cap Z_{exc}\cap\Omega\\
		\theta'_{s'}=\id,\ \theta_{s'}\neq\id
		\\ c'\in C(s')^+\setminus\theta_{s'}(0+)}}
		(m_{c'}-m_{\iota(c')})(
		\llbracket\theta'\rrbracket)e_\Omega+
		\frac{1}{2}\sum_{\substack{\Omega,\ s'\in I' \cap\Omega\\
		\theta'_{s'}\neq\id,\ \theta_{\theta'_{s'}(1)}=\id
		\\ c\in C(\theta'_{s'}(1))^+\setminus\iota(\theta'_{s'}(1-))\\ \theta'_{s'}(1)\in
		Z_{exc}}}
		(m_{c}-m_{\iota(c)})(\llbracket\theta\rrbracket)e_\Omega.
	\end{multline*}
	and is also equal to
	\begin{multline*}\frac{1}{2}
		\sum_{\substack{\Omega, (s'_1,s'_2)\in (I'\cap \Omega)^2\\
		(s'_1,s'_2){\not\in}E\cup E'}}\bigl(i(\theta_{s_1}\circ\theta'_{s'_1},
	\theta_{s_2}\circ\theta'_{s'_2})-i(\theta_{s_1},\theta_{s_2})-
	i(\theta'_{s'_1},\theta'_{s'_2})\bigr)e_\Omega+\\
		+\sum_{\Omega,\ (s'_1,s'_2)\in E\cap\Omega}
		\bigl(i(\theta_{s_1},\theta_{s_2}\circ\theta'_{s'_2})-
		i(\theta_{s_1},\theta_{s_2})-
		m_{\theta_{s_1}(0+)}^+(\theta'_{s'_2})\bigr)e_\Omega+\\
		+\sum_{\Omega,\ (s'_1,s'_2)\in E'\cap\Omega}
		\bigl(i(\theta'_{s'_1},\theta_{s_2}\circ\theta'_{s'_2})
		-i(\theta'_{s'_1},\theta'_{s'_2})-
		m_{\theta'_{s'_1}(1-)}^-(\theta_{s_2})\bigr)e_\Omega
	\end{multline*}
	where
	\begin{itemize}
		\item given $(s'_1,s'_2)\in I^{\prime 2}$, we put
			$s_i=\theta'_{s'_i}(1)$
		\item 
	$E$ is the set of pairs $(s'_1,s'_2)\in I'\times I'$ with
	$s'_1\in\ Z_{exc}$, $\theta'_{s'_1}=\id$, $\theta_{s'_1}\neq\id$,
	$\theta'_{s'_2}\neq\id$
\item
	$E'$ is the set of pairs $(s'_1,s'_2)\in I'\times I'$ with
	$s_1\in\ Z_{exc}$, $\theta'_{s'_1}\neq\id$, 
	$\theta_{s_1}=\id$ and $\theta_{s_2}\neq\id$.
	\end{itemize}
\end{lemma}

\begin{proof}
		Given $s'\in I'$ and $s=\theta'_{s'}(1)$, the class $\theta_s\circ\theta'_{s'}$
	is admissible, hence
	$\theta_s(0+)\cup\iota(\theta_s(0+))=\theta'_{s'}(1-)\cup
	\iota(\theta'_{s'}(1-))$ unless 
	$s\in Z_{exc}$ and one of $\theta_s$ and $\theta_{s'}$ is the identity, but not
	the other.

	Given $c\in T(Z)$, we put 
	$$v_c=(m_c-m_{\iota(c)})(\llbracket\theta
	\rrbracket)e_c=m_c(\llbracket\theta\rrbracket)e_c+
	m_{\iota(c)}(\llbracket\theta\rrbracket)e_{\iota(c)}=v_{\iota(c)}.$$
	Let
	$$a=\sum_{\substack{s'\in I'\cap Z_{exc}\\ \theta'_{s'}=\id,\ \theta_{s'}
	\neq\id\\ c'\in C(s')\setminus\bigl((\theta_{s'}(0+)\cup
	\iota(\theta_{s'}(0+))\bigr)}}m_{c'}(\llbracket\theta'\rrbracket)e_{c'}=
\sum_{\substack{s'\in I'\cap Z_{exc}\\ \theta'_{s'}=\id,\ \theta_{s'}
	\neq\id\\ c'\in C(s')^+\setminus\theta_{s'}(0+)
	}}(m_{c'}-m_{\iota(c')})(\llbracket\theta'\rrbracket)e_{c'}
	.$$
	We have
	$$m(\theta\circ\theta')-m(\theta)-m(\theta')=$$
	$$=
	\sum_{\substack{s'\in I'\\ c'\in (\theta\circ\theta')_{s'}(0+)\cup
	\iota((\theta\circ\theta')_{s'}(0+))}}
	m_{c'}(\llbracket\theta\rrbracket)e_{c'}-
	\sum_{\substack{s\in I\\ c\in \theta_{s}(0+)\cup\iota(\theta_{s}(0+))}}
	m_{c}(\llbracket\theta\rrbracket)e_c-a$$
	$$=\sum_{\substack{s'\in I'\\ \theta'_{s'}\neq\id}}
	v_{\theta'_{s'}(0+)}
	- \sum_{\substack{s'\in I'\\ \theta'_{s'}\neq\id\\ \theta_{\theta'_{s'}(1)}\neq\id}}
	v_{\theta'_{s'}(1-)}
	- \frac{1}{2}\sum_{\substack{s'\in I'\\ \theta'_{s'}\neq\id\\ \theta_{\theta'_{s'}(1)}=\id\\ c\in C(\theta'_{s'}(1))}}
	v_c -a$$

	Using (\ref{eq:scalar}), we find
	\begin{align*}
		\langle\llbracket\theta\rrbracket,\llbracket\theta'\rrbracket\rangle&=
	-\frac{1}{2}\sum_{\substack{s'\in I'\\ c'\in \theta'_{s'}(0+)\cup\iota(\theta'_{s'}(0+))}}
	v_{c'}+\frac{1}{2}\sum_{\substack{s'\in I'\\ c\in \theta'_{s'}(1-)\cup\iota(\theta'_{s'}(1-))}} v_{c}\\
		&=
	-\sum_{\substack{s'\in I'\\\theta'_{s'}\neq\id}}
	v_{\theta'_{s'}(0+)}+
	\sum_{\substack{s'\in I'\\ \theta'_{s'}\neq\id}}
	v_{\theta'_{s'}(1-)}.
	\end{align*}
We deduce that
$$\langle\llbracket\theta\rrbracket,\llbracket\theta'\rrbracket\rangle+
m(\theta\circ\theta')-m(\theta)-m(\theta')=
		-\sum_{\substack{s'\in I'\\ \theta'_{s'}\neq\id\\ \theta_{\theta'_{s'}(1)}=\id
		\\ c\in C(\theta'_{s'}(1))^+\setminus\iota(\theta'_{s'}(1-))\\ \theta'_{s'}(1)\in
		Z_{exc}}}
		(m_{c}-m_{\iota(c)})(\llbracket\theta\rrbracket)e_{c}-a$$
and the first equality of the lemma follows.

\smallskip
	Consider $s'_1\neq s'_2$ in $I'$.

If $s'_1\in Z_{exc}$,
$\theta'_{s'_1}=\id_{s'_1}$ and $\theta_{s_1}\neq\id_{s_1}$, it follows from
Lemma \ref{le:m=i} that
$$\sum_{\substack{s'_2\in I'\\ \theta'_{s'_2}\neq\id}} i(\id_{s'_1},
\theta'_{s'_2})=
\frac{1}{2}\sum_{\substack{s'_2\in I'\\ \theta'_{s'_2}\neq\id_{s'_2}\\ 
c'\in C(s'_1)^+}}(m_{c'}-m_{\iota(c')})(\theta'_{s'_2})=
\frac{1}{2}\sum_{c'\in C(s'_1)^+}(m_{c'}-m_{\iota(c')})(\llbracket\theta'\rrbracket).$$
Similarly, if $s_1\in Z_{exc}$, $\theta'_{s'_1}\neq\id$ and $\theta_{s_1}=\id$,
we have
$$\sum_{\substack{s'_2\in I'\\ \theta_{s_2}\neq\id}} i(\id_{s_1},
\theta_{s_2})=
\frac{1}{2}\sum_{c\in C(s_1)^+}(m_{c}-m_{\iota(c)})(\llbracket\theta\rrbracket).$$

The second equality of the lemma follows.
\end{proof}

\begin{example}
	The left (respectively second) side of the diagram below shows a
typical instance where the left (respectively right) sum of
Lemma \ref{le:degproduct} is nonzero.
	$$\includegraphics[scale=1.0]{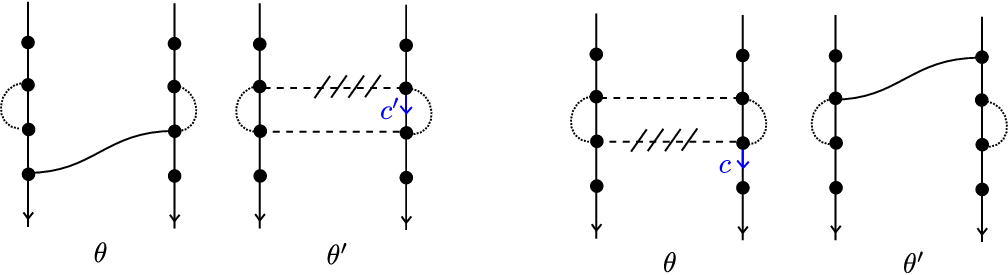}$$
\end{example}

\begin{rem}
	\label{re:productnonexc}
	Let $\theta:I\to J$ and $\theta':I'\to I$ be two braids such that
	$\theta\circ\theta'$ is a braid. 
	By Lemma \ref{le:boundintercompo}, the terms
	$i(\theta_{s_1},\theta_{s_2})+i(\theta'_{s'_1},\theta'_{s'_2})-i(\theta_{s_1}\circ\theta'_{s'_1},
	\theta_{s_2}\circ\theta'_{s'_2})$,
		$i(\theta_{s_1},\theta_{s_2})+
		m_{\theta_{s_1}(0+)}^+(\theta'_{s'_2})-
		i(\theta_{s_1},\theta_{s_2}\circ\theta'_{s'_2})$ and
		$i(\theta'_{s'_1},\theta'_{s'_2})+
		m_{\theta'_{s'_1}(1-)}^-(\theta_{s_2})-
		i(\theta'_{s'_1},\theta_{s_2}\circ\theta'_{s'_2})$ in 
		Lemma \ref{le:degproduct} are all non-negative.

	We deduce that the following
	assertions are equivalent:
	\begin{itemize}
		\item 
	$\deg(\theta)\cdot\deg(\theta')=\deg(\theta\circ\theta')$
\item $\deg(\theta_{|E})\cdot\deg(\theta'_{|E'})=\deg(\theta_{|E}\circ\theta'_{|E'})$ for
	any two-element subset $E'\subset I'$, where $E=\bij{\theta'}(E')$.
	\end{itemize}

	If  given $s\in I'$ with
	$\theta'_s=\id$ or $\theta_{\bij{\theta'}(s)}=\id$, we have $s{\not\in}Z_{exc}$, then
	$\deg(\theta)\cdot\deg(\theta')=\deg(\theta\circ\theta')$ if and only
	if 
$i(\theta_{\theta'_s(1)},\theta_{\theta'_{s'}(1)})+ i(\theta'_s,\theta'_{s'})=
	i((\theta\circ\theta')_s,(\theta\circ\theta')_{s'})$
	for all $s\neq s'$ in $I'$.
\end{rem}

\begin{lemma}
	\label{le:fdeg}
Let $f:Z\to Z'$ be a morphism of curves. 
	Let $I$ and $J$ be two finite subsets of $Z$ such that $|f(I)|=|f(J)|=|I|=|J|$.
	Let $\theta:I\to J$ be a braid in $Z$.
	Let $E=\{s\in I\cap Z_f\ |\ \theta_s=\id_s\}$.
	
	We have
	$f(\deg_{f^{-1}(f(E))^+}(\theta))=\deg_{f(E)^+}(f(\theta))$.
\end{lemma}

\begin{proof}
	Assume first $E=\emptyset$. Given $s\in I$ with $\theta_s=\id_s$, we have
	a bijection $C(s)\iso C(f(s))$. It follows that
	\begin{align*}
		f(m(\theta))&=\sum_{\substack{s\in I\\ \theta_s\neq\id_s}}
	\sum_{c\in \theta_s(0+)\cup\iota(\theta_s(0+))} m_c(\llbracket\theta\rrbracket)
	f(e_c)+
	\sum_{\substack{s\in I\\ \theta_s=\id_s}}
	\sum_{c\in C(s)} m_c(\llbracket\theta\rrbracket) f(e_c) \\
		&=\sum_{\substack{s'\in f(I)\\ f(\theta)_{s'}\neq\id_{s'}}}
		\sum_{c'\in f(\theta)_{s'}(0+)\cup\iota(f(\theta)_{s'}(0+))}
		m_{c'}(\llbracket f(\theta)\rrbracket) e_{c'}+
		\sum_{\substack{s'\in f(I)\\ f(\theta)_{s'}=\id_{s'}}}
		\sum_{c'\in C(s')} m_{c'}(\llbracket f(\theta)\rrbracket) e_{c'} \\
		&=m(f(\theta))
	\end{align*}
	by Lemma \ref{le:functorialitym}.

	Given $s'\in f(I)$ such that $f(\theta)_{s'}=\id_{s'}$, we have
	$s'{\not\in}Z'_f$. We deduce that $i(\theta_s,\theta_t)=i(f(\theta)_{f(s)},
	f(\theta)_{f(t)})$ for all $s\neq t\in I$ by Lemma \ref{le:intersections}.
	So $f(i(\theta))=i(f(\theta))$. We deduce that the lemma holds for $\theta$.

	\smallskip
	Consider now the case where $E\neq\emptyset$.
	Let $\bar{\theta}= (\theta_s)_{s\in I-E}$. We have 
	$\deg_{E^+}(\theta)=\deg_{E^+}(\bar{\theta})$ by Lemma \ref{le:removeid}; taking
	quotients, we obtain
	$\deg_{f^{-1}(f(E))^+}(\theta)=\deg_{f^{-1}(f(E))^+}(\bar{\theta})$.
	Since $f(\bar{\theta})=(f(\theta)_t)_{t\in f(I)-f(E)}$, it follows again from
	Lemma \ref{le:removeid} that $\deg_{f(E)^+}(f(\theta))=\deg_{f(E)^+}
	(f(\bar{\theta}))$. Since the lemma holds for $\bar{\theta}$, we deduce that
	the lemma holds for $\theta$.
\end{proof}

As a consequence of Lemma \ref{le:fdeg}, we have the following result.

\begin{prop}
	\label{pr:degfsharp}
	Let $f:Z\to Z'$ be a morphism of curves and let $\theta'$ be a non-zero
	map in $\CP^\bullet(Z')$. Then $f^\#(\theta')$ is a sum of maps $\theta$ such that
	$f(\deg_{Z_f^+}(\theta))=\deg_{f(Z_f)^+}(\theta')$.
\end{prop}

Let $Z_1,\ldots,Z_r$ be the connected components of $Z$. The isomorphism
(\ref{eq:Pcomponents}) is compatible with the degree function in the following sense.
Given $\theta:I\to J$ a braid in $Z$, let $\theta_i$ be the restriction of
$\theta$ to $I\cap Z_i$. The image of
$(\deg(\theta_1),\ldots,\deg(\theta_r))$ in $\Gamma(Z)$ by the map
of (\ref{eq:Gammacomponents}) is $\deg(\theta)$.

%
%
%
%

\medskip

Let $I$ and $J$ be two finite subsets of $Z$ and let 
$\theta:I\to J$ be a braid in $Z$. We define
$$L(\theta)=\coprod_{i_1\neq i_2\in I}I(\theta_{i_1},\theta_{i_2})\indexnot{L}{L(\theta)}.$$

Note that $\zeta\mapsto\zeta^{-1}$ 
induces a fixed-point free involution $\mathrm{inv}$ on $L(\theta)$.

\smallskip
Let $\zeta\in L(\theta)$.  Put $i_1=\zeta(0)$ and $i_2=\zeta(1)$.
We define $\theta^\zeta$\indexnot{t}{\theta^\zeta} by $(\theta^\zeta)_i=\theta_i$ if $i\in I-\{i_1,i_2\}$,
$(\theta^\zeta)_{i_1}=\theta_{i_2}\circ\zeta=\bar{\zeta}\circ\theta_{i_1}$ and
$(\theta^\zeta)_{i_2}= \theta_{i_1}\circ\zeta^{-1}=\bar{\zeta}^{-1}\circ\theta_{i_2}$.
Note that $\theta^{\zeta^{-1}}=\theta^\zeta$.

\medskip
Let $D(\theta)$\indexnot{D}{D(\theta)} be the set of classes $\zeta$ in $L(\theta)$
such that
\begin{itemize}
	\item[(a)] given a class of smooth paths $\zeta':\zeta(0)\to\zeta(1)$
		such that
		$\zeta\circ\zeta^{\prime -1}$ and $\zeta^{\prime -1}\circ
		\zeta$ are smooth and have the same
		orientation as $\zeta$ and $\zeta'$, and given
		a class of smooth paths $\zeta'':\bar{\zeta}(0)\to
		\bar{\zeta}(1)$ such that
		$\bar{\zeta}\circ\zeta^{\prime\prime -1}$ and
		$\zeta^{\prime\prime -1}\circ\bar{\zeta}$ are smooth and have
		the same
		orientation as $\bar{\zeta}$ and $\zeta^{\prime\prime}$,
		then $\zeta'=\zeta$ or $\zeta''=\bar{\zeta}$.
	\item[(b)] given $\zeta'$ and $\zeta''$ in $L(\theta)$
with $\zeta=\zeta'\circ\zeta''$, then $\zeta'$ and $\zeta''$ have opposite orientations.
\end{itemize}

\begin{rem}
Condition (a) above is automatically satisfied if the component of
the support of $\zeta$ is not isomorphic to $S^1$.

The subset $D(\theta)$ of $L(\theta)$ is stable under the involution $\mathrm{inv}$.
\end{rem}

The next lemma restricts the cases where condition (b) above needs to be
checked.

\begin{lemma}
	\label{le:conditionbDtheta}
	Let $\zeta,\zeta',\zeta''\in L(\theta)$ such that
	$\zeta=\zeta'\circ\zeta''$.
	If $\zeta'(0)\in Z_o$ and $\theta_{\zeta'(0)}=\id$, then
	$\zeta'$ and $\zeta''$ have opposite orientations.
\end{lemma}

\begin{proof}
	Let $z=\zeta'(0)=\zeta''(1)$.  We have $\zeta'\in
	I(\id_z,\theta_{\zeta(1)})$. Since $\bar{\zeta}'=\theta_{\zeta(1)}\circ\zeta'$
	is smooth and has opposite orientation to $\zeta'$, it follows that
	$\zeta'(0+)\in \iota(C(z)^+)$.
	Similarly, $\zeta''(1-)\in \iota(C(z)^+)$. We deduce that
	$\zeta'$ and $\zeta''$ have opposite orientations.
\end{proof}

\begin{lemma}
	\label{le:restrictionDtheta}
	Let $I'$ be a subset of $I$ such that $I-I'\subset Z_o$ and
	$\theta_i=\id$ for $i\in I-I'$. 

	We have $D(\theta_{|I'})\subset D(\theta)$.
\end{lemma}

\begin{proof}
	We have $L(\theta_{|I'})\subset L(\theta)$ and Lemma
	\ref{le:conditionbDtheta} shows that 
	$D(\theta_{|I'})\subset D(\theta)$.
\end{proof}

\begin{example}
	In the picture below, the left side shows a valid braid $\theta$,
for which the conclusion of Lemma \ref{le:conditionbDtheta} holds. For contrast,
the right side shows a braid $\theta$ that is disallowed since $\theta_{i_2}$
is not oriented, and the conclusion of Lemma \ref{le:conditionbDtheta}
fails.
$$\includegraphics[scale=1.2]{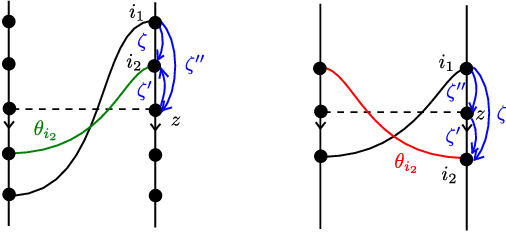}$$
\end{example}

\subsubsection{Strands on $S^1$}
\label{se:StrandsS1}

Let $Z=S^1$, viewed as an unoriented manifold.
Fix a family $M=\{a_1,\ldots,a_n\}$ of cyclically ordered points on $S^1$, \ie,
$a_j=e^{i e_j}$ for some real numbers $e_1<\cdots<e_n$ with $e_n-e_1<2\pi$.

\smallskip
Fix $r',r\in\{1,\ldots,n\}$.
There is a bijection
$$F_{r',r}:r'-r+n\BZ\iso\Hom_{\Pi(S^1)}(a_r,a_{r'}):$$
it sends $l$ to the homotopy class of paths going in the positive direction
and winding $\lfloor \frac{l}{n}\rfloor$ times around $S^1$, if $l\ge 0$,
and to the homotopy class of paths going in the negative direction
and winding $\lfloor \frac{-l}{n}\rfloor$ times around $S^1$, otherwise.

We put
$$F_r=\sum_{r'}F_{r',r}:\BZ\iso\coprod_{r'}
\Hom_{\Pi(S^1)}(a_r,a_{r'}).$$
Given $r,r'\in\{1,\ldots,n\}$ and $l,l'\in\BZ$ with $r'-r=l\pmod n$, we have
$F_r(l+l')=F_{r'}(l')\circ F_r(l)$.
Note also that given $j\in\{1,\ldots n\}$ and $j'\in\BZ$, we have
$$\supp(F_j(j'-j))=\begin{cases}
	S^1 & \text{ if }|j'-j|\ge n \\
	\{e^{iu}\ |\ e_j\le u\le e_{j''}+2\pi\delta_{j'>n} & \text{ if }
	j'-j\in\{0,\ldots,n-1\} \\
	\{e^{iu}\ |\ e_{j''}-2\pi\delta_{j'\le 0}\le u\le e_j & \text{ if }
	j-j'\in\{0,\ldots,n-1\}
\end{cases}$$
where $j''\in\{1,\ldots,n\}$ and $j''-j'\in n\BZ$.

\smallskip
We denote by $\vec{S}^1$ the oriented curve $S^1$.
	Fix  $z=e^{ix}\in S^1$ with $x<e_1$ and $e_n-x<2\pi$ and $\Omega$ a connected
	open neighbourhood of $z$ in $S^1$ containing no $a_i$. Let
	$I=S^1-\{z\}$ unoriented and $\vec{I}=S^1-\{z\}$ oriented. We define
	$\dot{S}^1$ to be the curve $S^1$ with $(\dot{S}^1)_o=\Omega$ with its
	standard orientation.

\begin{prop}
	\label{pr:prestrandsnonsingular}
There is an isomorphism of pointed categories
	$F:(\CS_n)_+\iso \CP^\bullet_M(S^1)$ given by
	$F(J)=\{a_j\}_{j\in \tilde{J}\cap[1,n]}$ and
	$F(\sigma)_{a_j}=F_j(\sigma(j)-j)$ for $\sigma$ a map of $\CS_n$.

	It restricts to isomorphisms of
	pointed categories
	$$(\CS_n^+)_+\iso \CP^\bullet_M(\dot{S}^1),\ (\CS_n^{++})_+\iso\CP^\bullet_M(\vec{S}^1),\
	(\CS_n^f)_+\iso\CP^\bullet_M(I) \text{ and }
	(\CS_n^{f++})_+\iso\CP^\bullet_M(\vec{I}).$$
\end{prop}

\begin{proof}
Consider $J,J'\subset \BZ/n$. We have an injective map
	$f:\Hom_{\CS_n}(J,J')\to \BZ^J,\ \sigma\mapsto (\sigma(j)-j))_b$,
	where $j\in\{1,\ldots,n\}$ and $b=j+n\BZ$.
	The image of that map is the set of 
	those $c\in\BZ^J$ such that $\{c_b+b\}_b=J'$ and we obtain a bijection
\begin{align*}
	\Hom_{\CS_n}(J,J')&\iso \Hom_{\CP^\bullet_M(\vec{S}^1)}(\{a_j\}_{j\in \tilde{J}\cap[1,n]},
	\{a_{j'}\}_{j'\in \tilde{J}'\cap[1,n]})\\
	\sigma&\mapsto \bigl(F_{j',j}(f(\sigma)_{j+n\BZ})\bigr)_{
		j\in \tilde{J}\cap[1,n],\ j'\in \tilde{J}'\cap[1,n],\ \sigma(j)-j'\in n\BZ}.
\end{align*}
	We deduce
	that $F$ induces a bijection on pointed $\Hom$-sets.
	Consider now $\sigma:J\to J'$ and $\sigma':J'\to J''$ two maps
	in $\CS_n$. Given $j\in\tilde{J}\cap[1,n]$, we have
$$F(\sigma'\sigma)_{a_j}=F_j(\sigma'\sigma(j)-j)=
	F_j(\sigma'(\sigma(j))-\sigma(j)+\sigma(j)-j)=
	F_{\sigma(j)}(\sigma')_{a_{\sigma(j)}}\circ F_j(\sigma)_{a_j}.$$
We deduce that $F$ is a functor and the first statement of the proposition follows.

	\smallskip
	Consider now $\sigma\in\Hom_{\CS_n}(J,J')$.
	
	The map
	$F(\sigma)$ is in $\CP^\bullet_M(\dot{S}^1)$ if and only if
	$\sigma(j)\ge 0$ for all $j\in [1,n]\cap \tilde{J}$, hence if and only if
	$\sigma$ is in $\CS_n^+$.

	The map
	$F(\sigma)$ is in $\CP^\bullet_M(\vec{S}^1)$ if and only if
	$\sigma(j)-j\ge 0$ for all $j\in \tilde{J}$, hence if and only if
	$\sigma$ is in $\CS_n^{++}$.

	The map
	$F(\sigma)$ is in $\CP^\bullet_M(I)$ if and only if
	$\sigma(j)\in [1,n]$ for all $j\in \tilde{J}\cap [1,n]$, hence if and only if
	$\sigma$ is in $\CS_n^f$.

	The proposition follows.
\end{proof}

There are morphisms of groups $F_R:R_n\to R(S^1),\ \alpha_{j+n\BZ}\mapsto 
\llbracket F_j(1)\rrbracket$ and $F_L:L_n\to L(S^1),\
\eps_{j+n\BZ}\mapsto e_{c_j}$,
where $j\in\{1,\ldots,n\}$,
$c_j=(a_j,a_je^{iu})\in C(a_j)$ and $u\in\BR_{>0}$ is small enough.

\begin{lemma}
	\label{le:S1comparison}
	Given $\alpha,\beta\in R_n$, we have
	$F_L(\langle\alpha,\beta\rangle)=\langle F_R(\alpha),F_R(\beta)\rangle$
	and there is an injective morphism of groups
	$F_\Gamma:\Gamma_n\to\Gamma(S^1),\ (r,(l,\alpha))\mapsto (r,(F_L(l),F_R(\alpha)))$.

	Let $D$ be a subset of $\{1,\ldots,n\}\times\{\pm 1\}$ that embeds in its
	projection on $\{1,\ldots,n\}$. Define $\partial:D\to T(S^1)$ by
	$\partial((i,\nu_i))=c_i$ if $\nu_i=1$ and
	$\partial((i,\nu_i))=\iota(c_i)$ otherwise. The morphism $F_\Gamma$
	induces an isomorphism of groups
	$F_D:\Gamma_D\to\Gamma_M(S^1,\partial(D))$.
	We have $u<u'$ if and only if $F_D(u)<F_D(u')$.

	Let $\sigma$ be a map in $\CS_n$. We have 
$F_R(\llbracket \sigma\rrbracket)=\llbracket F(\sigma)\rrbracket$,
	$m(F(\sigma))=F_L(m(\sigma))$,
	$i(F(\sigma))=\ell(\sigma)$ and
	$\deg(F(\sigma))=F_\Gamma(\deg(\sigma))$.
\end{lemma}

\begin{proof}
	Let $r,j\in\{1,\ldots,n\}$ and let $j'\in\BZ$.
	We have
$$m_{c_r}(\llbracket F_j(j'-j)\rrbracket)=
		|\{i\in r+n\BZ\ |\ j\le i<j'\}|
		-|\{i\in r+n\BZ\ |\ j>i\ge j'\}|$$
and
	$$m_{\iota(c_r)}(\llbracket F_j(j'-j)\rrbracket)=
		-|\{i\in r+n\BZ\ |\ j< i\le j'\}|
		+|\{i\in r+n\BZ\ |\ j\ge i> j'\}|$$

	In particular,
	$m_{c_r}(\llbracket F_j(1)\rrbracket)=
	\delta_{r,j}$ and 
	$m_{\iota(c_r)}(\llbracket F_j(1)\rrbracket)=-\delta_{r,j+1}$.
	This shows that $F_R$ is injective. This shows also that given 
	$i\in\{1,\ldots,n\}$, we have
	$$\langle \llbracket F_i(1)\rrbracket,\llbracket F_j(1)\rrbracket\rangle=
	(\delta_{i,j+1}+\delta_{i,j})F_L(\eps_{j+1+n\BZ})-(\delta_{i,j}+\delta_{i+1,j})F_L(\eps_{j+n\BZ})=F_L(\langle\alpha_{i+n\BZ},\alpha_{j+n\BZ}\rangle).$$
	This shows the first equality and this shows that $F_R$ induces an injective
	morphism of groups $F_\Gamma$.
	
	\smallskip
	Taking quotients, we obtain an injective morphism of groups
	$F_D:\Gamma_D\to\Gamma(S^1,\partial(D))$ compatible with
	the order and with image $\Gamma_M(S^1,\partial(D))$.

	\smallskip
	Consider $\sigma\in\Hom_{\CS_n}(I,J)$.
	Given $d\in\BZ$, we have
	$\llbracket F_r(d)\rrbracket=F_R(\alpha_{r,r+d})$,
	hence $F_R(\llbracket \sigma\rrbracket)=\llbracket F(\sigma)\rrbracket$.

	\smallskip
We have
$$m(F(\sigma)) = \sum_{r,j \in \tilde{I} \cap [1.n]} (m_{c_r} - m_{\iota(c_r)})(\llbracket F_j(\sigma(j) - j)\rrbracket) e_{c_r}$$ and
	$$m(\sigma) = \sum_{r,j \in \tilde{I} \cap [1,n]} \alpha_{j,\sigma(j)} \cdot \varepsilon_{r+n\BZ}.$$
	Since
$$\alpha_{j,\sigma(j)} \cdot \varepsilon_{r+nZ} =
	(m_{c_r}-m_{\iota(c_r)})(F(\sigma)) \varepsilon_{r+nZ},$$
	it follows that $m(F(\sigma))=F_L(m(\sigma))$. 

	\smallskip

	Consider $i_1,i_2\in \tilde{I}$ with $0\le i_1<i_2<n$. We have
	$i(F(\sigma)_{a_{i_1}},F(\sigma)_{a_{i_2}})=
	i(\gamma_1,\gamma_2)$ for some minimal paths $\gamma_l$ in
	$F(\sigma)_{a_{i_l}}$ by Lemma \ref{le:intersections}.
	Lemma \ref{le:intersectionlength} shows that
	$i(F(\sigma)_{a_{i_1}},F(\sigma)_{a_{i_2}})=
	\bigl|{\lfloor\frac{\sigma(i_2)-\sigma(i_1)}{n}\rfloor}\bigr|$.
	Lemma \ref{le:formulalength} shows now that $i(F(\sigma))=\ell(\sigma)$.
\end{proof}

	Given $i_1,i_2\in\BZ$ with
	$i_2{\not\in} i_1+n\BZ$, we put $\lambda(i_1,i_2)=F_{i'_1}(i_2-i_1)$,
	where $i'_1\in[1,n]\cap (i_1+n\BZ)$.

\begin{lemma}
	\label{le:inversionpaths}
	Let $\sigma$ be a map in $\CS_n$. Given $(i_1,i_2)$ in $L(\sigma)$ (resp.
	$D(\sigma)$), the class
	$\lambda(i_1,i_2)$ is in $L(F(\sigma))$ (resp. $D(F(\sigma))$)
	and
	$F(\sigma)^{\lambda(i_1,i_2)}=F(\sigma^{i_1,i_2})$. 
	Furthermore, $\lambda$ induces 
	bijections
	$$L(\sigma)/n\BZ\iso L(F(\sigma))/\mathrm{inv}\text{ and }
	D(\sigma)/n\BZ\iso D(F(\sigma))/\mathrm{inv}.$$

	
\end{lemma}

\begin{proof}
	Note first that, given $i'_1$ and $i'_2$ two distinct elements of
	$\{1,\ldots,n\}$, then $\lambda$ induces a bijection
	$$\bigl((i'_1+n\BZ)\times (i'_2+n\BZ)\bigr)/n\BZ\iso \Hom_{\Pi(S^1)}(
	a_{i'_1},a_{i'_2}).$$

	Consider $\sigma\in\Hom_{\CS_n}(I,J)$ and $i_1,i_2\in\tilde{I}$
	with $i_2{\not\in} i_1+n\BZ$.
	Note that $\lambda(i_1,i_2)=\lambda(i_2,i_1)^{-1}$ for any
	$i_1,i_2$.

	Let $\zeta_r=F(\sigma)_{a_{i_r}}$ and
	$\zeta=\lambda(i_1,i_2)$.
	We have $\bar{\zeta}=\lambda(\sigma(i_1),\sigma(i_2))$. So, $\zeta\in L(F(\sigma))$
	if and only if $i_1-i_2$ and $\sigma(i_1)-\sigma(i_2)$ have opposite signs.
	On the other hand, $(i_1,i_2)\in L(\sigma)$ if and only if
	$i_1<i_2$ and $\sigma(i_2)<\sigma(i_1)$. This shows that $\lambda(L(\sigma))\subset
	L(F(\sigma))$ and $\lambda$  induces 
	a bijection $L(\sigma)/n\BZ\iso L(F(\sigma))/\mathrm{inv}$.

	\smallskip
	Consider $(i_1,i_2)\in L(\sigma)$.
	Let $r=\lfloor \frac{i_2-i_1}{n}\rfloor$ and
	$s=\lfloor \frac{\sigma(i_1)-\sigma(i_2)}{n}\rfloor$.
	We have $r>0$ if and only if $\supp(\lambda(i_1,i_2-rn))\subsetneq
	\supp(\lambda(i_1,i_2))$ and
	$s>0$ if and only if $\supp(\lambda(i_1,i_2+sn))\subsetneq
	\supp(\lambda(i_1,i_2))$.
	There is $i$ such that $(i_1,i)$ and
	$(i,i_2)$ are in $L(\sigma)$ if and only if
	there are $\zeta'$ and $\zeta''$ with the same orientations
	in $L(F(\sigma))$ such that
	$\lambda(i_1,i_2)=\zeta''\circ\zeta'$.
	We have $i_2-i_1>n$ if and only if there is $\zeta'$ such that
	$\zeta$, $\zeta'$ and $\zeta\circ\zeta^{\prime -1}$ have the
	same orientation. We have $\sigma(i_1)-\sigma(i_2)>n$ 
	if and only if there is $\zeta''$ such that
	$\bar{\zeta}$, $\zeta^{\prime\prime}$ and
	$\bar{\zeta}\circ\zeta^{\prime\prime -1}$
have the same orientation.
	
	We deduce that $(i_1,i_2)\in D(\sigma)$ 
	if and only if $\lambda(i_1,i_2)\in D(F(\sigma))$.

	\smallskip
	Assume now $(i_1,i_2)\in L(\sigma)$. Let $i'_r\in [1,n]\cap(i_r+n\BZ)$ for
	$r\in\{1,2\}$.
	We have 
	$$(F(\sigma)^{\lambda(i_1,i_2)})_{a_{i'_1}}=F_{i'_2}(\sigma(i_2)-i_2)\circ
	F_{i'_1}(i_2-i_1)=F_{i'_1}(\sigma(i_2)-i_1)=(F(\sigma^{i_1,i_2}))_{a_{i_1}}.$$
	Similarly, 
	$(F(\sigma)^{\lambda(i_1,i_2)})_{a_{i'_2}}=(F(\sigma^{i_1,i_2}))_{a_{i_2}}$.
	It follows that $F(\sigma)^{\lambda(i_1,i_2)}=F(\sigma^{i_1,i_2})$.
This completes the proof of the lemma.
%
\end{proof}

\subsubsection{Strand category}
\label{se:definitionstrandcategory}
Let $Z$ be a curve.

\smallskip
We have a positivity result in the setting of Lemma \ref{le:degproduct}.
\begin{lemma}
	\label{le:degsemi}
Let $\theta$ and $\theta'$ be two braids such that $\theta\circ\theta'$ is a braid.
	We have $\deg(\theta)\cdot\deg(\theta')\le \deg(\theta\circ\theta')$.

	Given $D$ a subset of $T(Z)$ containing $Z_{exc}^+$ and
	such that $D\cap\iota(D)=\emptyset$, the following
	assertions are equivalent:
	\begin{itemize}
		\item $\deg(\theta)\cdot\deg(\theta')=\deg(\theta\circ\theta')$
		\item $\deg(\theta)\cdot\deg(\theta')$ and $\deg(\theta\circ\theta')$
			have the same image in $\Gamma(Z,D)$
		\item $\deg(\theta)\cdot\deg(\theta')$ and $\deg(\theta\circ\theta')$
			have the same image in $\bar{\Gamma}(Z,D)$.
	\end{itemize}
\end{lemma}

\begin{proof}
	Assume $Z=S^1$ unoriented. Let $M$ be a family 
as in \S\ref{se:StrandsS1}. Assume $M$ 	
	contains $\theta_s(r)$ and $\theta'_s(r)$ for $r\in\{0,1\}$ and all $s$.
	Proposition \ref{pr:prestrandsnonsingular} and Lemma \ref{le:S1comparison}
	show that the inequality follows from the corresponding inequality for
	maps in $\CS_n$, which is given by Lemmas \ref{le:lengthincrease} and
	\ref{le:degreelength}.

	Given $Z$ a non-singular connected curve, there is an injective morphism
	of curves $Z\to S^1$, and the lemma follows from
	Proposition \ref{pr:finjective} and Lemma \ref{le:fdeg}. We deduce that the 
	inequality holds for any non-singular curve $Z$.

	Consider now a general curve $Z$ and let $q:\hat{Z}\to Z$ be the
	non-singular cover.  Since the functor
	$q^\#:\add(\CP(Z))\to\add(\CP(\hat{Z}))$ is compatible with degrees
	(Proposition \ref{pr:degfsharp}),
	it follows that the inequality holds for $Z$.

	\smallskip
	The equivalence of the three assertions follows from the fact that an element
	of $(\frac{1}{2}\BZ_{\ge 0})^{\pi_0(Z)}\subset\Gamma(Z,Z_{exc}^+)$ is zero
	if and only if its image in $\frac{1}{2}\BZ_{\ge 0}\subset
	\bar{\Gamma}(Z,D)$
	is zero.
\end{proof}

\smallskip
By Lemma \ref{le:degsemi}, the degree function gives a $\Gamma(Z,Z_{exc}^+)$-filtration
on the category $\CP^\bullet(Z)$.

\begin{defi}
We define the {\em strand category}
	$\CS^\bullet(Z)$\indexnot{S(}{\CS^\bullet(Z)} as the
	$\Gamma(Z,Z_{exc}^+)$-graded pointed category
associated with the filtered pointed category $\CP^\bullet(Z)$ (cf \S\ref{se:pointedcat}).
\end{defi}

The category $\CS^\bullet(Z)$ has the same objects and the same maps as the category $\CP^\bullet(Z)$.
It is a pointed
category with objects the finite subsets of $Z$ and with
$\Hom_{\CS^\bullet(Z)}(I,J)$ the set of braids $I\to J$, together with a 
$0$-element.

The product of two braids
$\theta:I\to J$ with $\theta':J\to K$ is defined as follows:
$$\theta'\cdot\theta=
\begin{cases}
	\theta'\circ\theta &\text{ if }\deg(\theta'\circ\theta)=\deg(\theta')\cdot\deg(\theta)\\
	0 & \text{ otherwise.}
\end{cases}$$

\smallskip
Note that the strand category decomposes as a disjoint union
$\CS^\bullet(Z)=\coprod_{n\ge 0}\CS^\bullet(Z,n)$, where $\CS^\bullet(Z,n)$
\indexnot{S(}{\CS^\bullet(Z,n)}is the full subcategory with
objects subsets with $n$ elements.

\smallskip
It follows from Lemma \ref{le:degsemi} that 
	given $D$ a subset of $T(Z)$ containing $Z_{exc}^+$ and
	such that $D\cap\iota(D)=\emptyset$, the structure of
	$\Gamma(Z,D)$-graded (resp. 
	$\bar{\Gamma}(Z,D)$-graded) category on $\CS(Z)$ obtained from the quotient
	morphism $f:\Gamma(Z,Z_{exc}^+)\to\Gamma(Z,D)$
	(resp. $f:\Gamma(Z,Z_{exc}^+)\to\bar{\Gamma}(Z,D)$) is the same as the
	graded category obtained from the structure of $\Gamma(Z,D)$-filtered
	(resp. $\bar{\Gamma}(Z,D)$-filtered) category on $\CP^\bullet(Z)$ that is deduced
	from the structure of $\Gamma(Z,Z_{exc}^+)$-filtered category via $f$.

\begin{rem}
	We leave to the reader to check the following alternate definition of
	the product in the strand category.
	
	We have $\theta'\cdot\theta\neq 0$ if and only 
if there are parametrized braids $\vartheta,\vartheta'$ with $\theta=[\vartheta]$ and $\theta'=[\vartheta']$
and there are $\alpha:I'\to I$ and $\alpha':K\to K'$ two parametrized
braids with $I',K'\subset Z\setminus Z_{exc}$
	such that $i(\alpha)=i(\alpha')=0$, $\alpha'_{\vartheta'\circ\vartheta\circ\alpha_s(1)}\circ
	\vartheta'_{\vartheta\circ\alpha_s(1)}
	\circ\vartheta_{\alpha_s(1)}\circ\alpha_s$ 
is admissible for all $s\in I$ and
$i(\alpha'\circ\theta'\circ\theta\circ\alpha)=
i(\alpha'\circ\theta')+i(\theta\circ\alpha)$.
\end{rem}

Let $\CS(Z)=\BF_2[\CS^\bullet(Z)]$\indexnot{S(}{\CS(Z)},
a $\Gamma(Z,Z_{exc}^+)$-graded $\BF_2$-linear category.

\smallskip
Let $f:Z\to Z'$ be a morphism of curves.

Let $\CS^\bullet_f(Z)$\indexnot{S(}{\CS^\bullet_f(Z)} be the full subcategory of $\CS^\bullet(Z)$ with objects
those finite subsets $I$ of $Z$ such that $|f(I)|=|I|$.
We deduce from Proposition \ref{pr:finjective} and Lemma \ref{le:fdeg} a faithful
$\Gamma(Z',Z_{exc}^{\prime +})$-graded pointed functor $f:\CS^\bullet_f(Z)\to
\CS^\bullet(Z')$. Here, the $\Gamma(Z',Z_{exc}^{\prime +})$-grading on 
$\CS^\bullet_f(Z)$ comes from the $\Gamma(Z,f^{-1}(Z_{exc}')^+)$-grading via the morphism
$\Gamma(f)$.

\smallskip
Assume $f$ is strict. Propositions
\ref{pr:fsharpsprestrands} and \ref{pr:degfsharp} provide an additive $\BF_2$-linear 
$\Gamma(Z',Z_{exc}^{\prime +})$-graded functor
$f^\#:\add(\CS(Z'))\to\add(\CS_f(Z))$, where the $\Gamma(Z',Z_{exc}^{\prime +})$-grading on
$\CS_f(Z)$ is deduced from the $\Gamma(Z,f^{-1}(Z_{exc}')^+)$-grading via the 
morphism $\Gamma(f)$.

If $f$ is a quotient morphism,
it follows from Proposition \ref{pr:qprestrands} that $f^\#$ is faithful.

\smallskip
Given $M$ a subset of $Z$, we denote by $\CS^\bullet_M(Z)$\indexnot{S(}{\CS^\bullet_M(Z)} the full
subcategory of $\CS^\bullet(Z)$ whose objects are the finite subsets of $M$.
The $\Gamma(Z,Z_{exc}^+)$-grading on $\CS^\bullet_M(Z)$ comes from a 
$\Gamma_M(Z,Z_{exc}^+)$-grading.

We
denote by $\CS_{M,f}^\bullet(Z)$\indexnot{S(}{\CS_{M,f}^\bullet(Z)} the full subcategory of
$\CS_f^\bullet(Z)$ with objects subsets contained in $M$.

\smallskip
We put $\CA^\bullet(Z)=\CS_{Z_{exc}}^\bullet(Z)$\indexnot{A(}{\CA^\bullet(Z)} and
$\CA^\bullet(Z,n)=\CS_{Z_{exc}}^\bullet(Z,n)$\indexnot{A(}{\CA^\bullet(Z,n)}. Let
$\CA(Z)=\BF_2[\CA^\bullet(Z)]$\indexnot{A(}{\CA(Z)} and
$\CA(Z,n)=\BF_2[\CA^\bullet(Z,n)]$\indexnot{A(}{\CA(Z,n)}.

\medskip
Let $Z_1,\ldots,Z_r$ be the connected components of $Z$.
The isomorphism (\ref{eq:Pcomponents}) induces an isomorphism of 
$\Gamma(Z,Z_{exc}^+)$-graded
pointed categories

\begin{equation}
	\label{eq:Scomponents}
\CS^\bullet(Z_1)\wedge\cdots\wedge\CS^\bullet(Z_r)\iso \CS^\bullet(Z)
\end{equation}
where the grading on the left hand term is deduced from the
the $\bigl(\prod_{i=1}^r\Gamma(Z_i,(Z_i)_{exc}^+)\bigr)$-grading via (\ref{eq:Gammacomponents})
and an isomorphism of $\BF_2$-linear categories
\begin{equation}
	\label{eq:Strandscomponents}
\CS(Z_1)\otimes\cdots\otimes\CS(Z_r)\iso \CS(Z).
\end{equation}

\begin{example}
In the example below the first
	row is the product in $\CP^\bullet(Z)$, while the second
	row is the product in $\CS^\bullet(Z)$.
$$\includegraphics[scale=0.75]{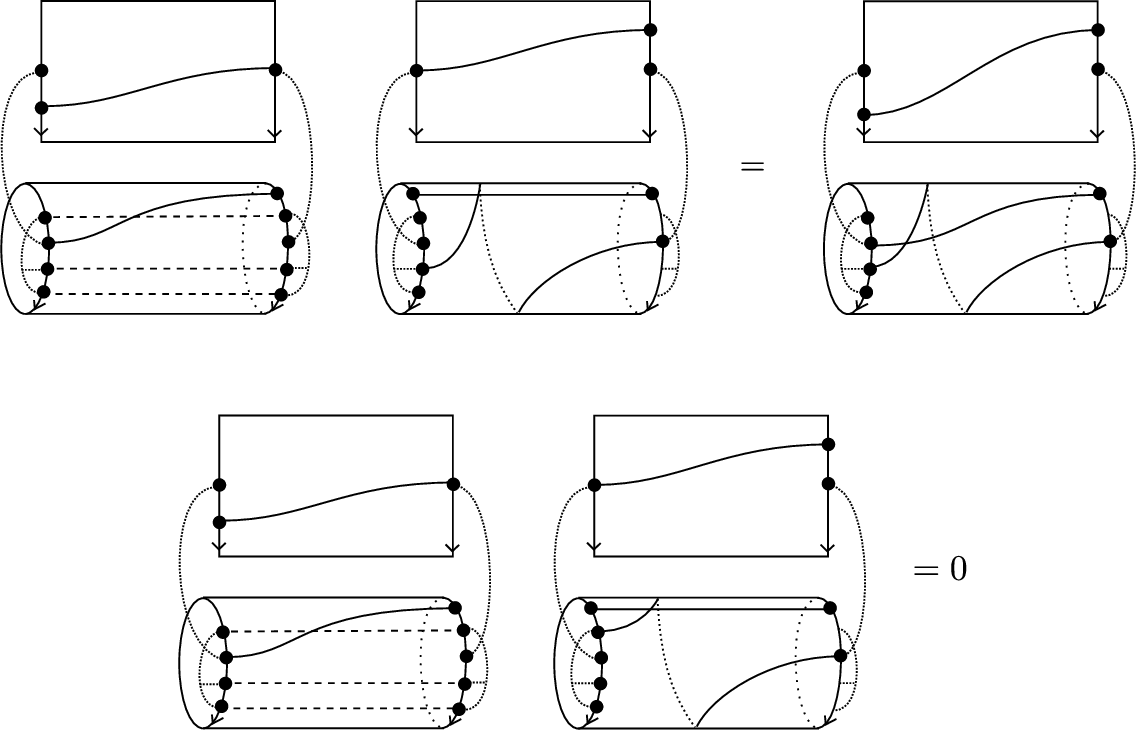}$$
\end{example}

\subsubsection{Generation}
\label{se:generation}
We equip $Z$ with a metric. Given $\xi$ a path in $Z$, we denote
by $|\xi|$ its length. Given $\zeta$ a homotopy class of paths in $Z$,
we put $|\zeta|=|\xi|$, where $\xi$ is a minimal path in $\zeta$.
Given $\theta:I\to J$ a braid in $Z$, we put
$|\theta|=\sum_{s\in I}|\theta_s|$.

\smallskip
Let $M$ be a finite subset of $Z$.

\begin{lemma}
	\label{le:producttensorid}
	Let $\theta\in\Hom_{\CS^\bullet_M(Z)}(I,J)$ and 
	$\theta'\in\Hom_{\CS^\bullet_M(Z)}(I',I)$ such that $\theta\circ\theta'$
	is a braid. Let $I_0$ be
	a finite subset of $M\setminus (I\cup I'\cup J)$.

	If $|\theta|+|\theta'|=|\theta\circ\theta'|$, then
	$(\theta\boxtimes\id_{I_0})\cdot (\theta'\boxtimes\id_{I_0})=
	(\theta\cdot\theta')\boxtimes\id_{I_0}$.
\end{lemma}

\begin{proof}
	Let $s'\in I'$.
	Since $|\theta_{\theta'(s')}|+|\theta'_{s'}|=|\theta_{\theta'(s')}
	\circ\theta'_{s'}|$, it follows that
	$i(\theta_{\theta'(s')},\id_i)+i(\theta'_{s'},\id_i)=
	i(\theta_{\theta'(s')} \circ\theta'_{s'},\id_i)$ for all $i\in I_0$.
	As a consequence,
	$$i((\theta\circ\theta')\boxtimes\id_{I_0})-i(\theta\boxtimes\id_{I_0})
	-i(\theta'\boxtimes\id_{I_0})=
	i(\theta\circ\theta')-i(\theta)-i(\theta').$$
	The lemma follows now from Lemma \ref{le:degproduct}.
\end{proof}

Let $\theta\in\Hom_{\CS^\bullet_M(Z)}(I,J)$ be a non-zero braid.

Let $I_0=\{i\in I\ |\ \theta_i=\id_i\}$.
Let $i\in I\setminus I_0$. There is a (unique) decomposition
$\theta_i=\beta^i\cdot\alpha^i$  in $\CS^\bullet(Z,1)$ with
\begin{itemize}
	\item $\alpha^i(1)\in M\setminus I_0$
	\item $|\theta_i|=|\alpha^i|+|\beta^i|$ 
	\item given a minimal path
$\xi$ in $\alpha^i$, we have $\xi((0,1))\cap M\subset I_0$.
\end{itemize}

We define a quiver $\Gamma(\theta)$ with vertex set $I\setminus I_0$.
There is an arrow $i\to i'$ if $\alpha^i(1)=i'$.

Note that there is at most one arrow with a given source (that arrow can
be a loop).

\begin{lemma}
	\label{le:rightdivisor}
	Let $I'$ be a non-empty finite subset of $I\setminus I_0$ such that
	\begin{itemize}
		\item if there is an arrow $i\to i'$ in $\Gamma(\theta)$ with
			$i\in I'$, then $i'\in I'$
		\item given $i\neq i'\in I'$, we have $\alpha^i(1)\neq
			\alpha^{i'}(1)$.
	\end{itemize}
	There is a (unique) decomposition $\theta=\theta^u\cdot u$ in
	$\CS_M^\bullet(Z)$, where $|\theta|=|\theta^u|+|u|$ and
	$$u_i=\begin{cases}
		\alpha^i & \text{ if }i\in I'\\
		\id_i & \text{ otherwise.}
	\end{cases}$$
\end{lemma}

\begin{proof}
	Note that the second assumption on $I'$ shows that the
	full subquiver of $\Gamma(\theta)$
	with vertex set $I'$ is a disjoint union of oriented lines and
	oriented circles.

	\smallskip
	Let $i_1\neq i_2\in I$. If $i_1,i_2\in I\setminus I'$, then
	$u(i_1)\neq u(i_2)$. Assume now $i_1\in I'$ and $i_2\in I\setminus I'$.
	Since $i_1\to i_2$ is not an arrow of the quiver, we have
	$u(i_1)\neq i_2$, hence $u(i_1)\neq u(i_2)$. Finally if
	$i_1,i_2\in I'$, then $u(i_1)\neq u(i_2)$.
	We have shown that $u$ is a braid.

	\smallskip
	Note that there is a (unique) decomposition $\theta=\theta^u\circ u$ with
	$|\theta|=|\theta^u|+|u|$. In order to show that $\theta^u\cdot u\neq 0$, we can replace $\theta$ by $\theta_{|I\setminus I_0}$ and $M$ by 
	$M\setminus I_0$, thanks to Lemma \ref{le:producttensorid}.
	So, we assume now that $I_0=\emptyset$.

	Let $q:\tilde{Z}\to Z$ be a non-singular cover of $Z$. Let
	$\tilde{M}=q^{-1}(M)$. Let
	$\tilde{\theta}:\tilde{I}\to\tilde{J}$ be the unique lift of
	$\theta$ to $\tilde{Z}$. We have a decomposition
	$\tilde{\theta}_i=\tilde{\beta}^i\cdot\tilde{\alpha}^i$ for
	$i\in\tilde{I}$ and $q(\tilde{\alpha}^i)=\alpha^{q(i)}$.
	
	Let $\tilde{I}'=q^{-1}(I')\cap\tilde{I}$. 
	Note that
	$q$ induces a morphism of quivers $\Gamma(\tilde{\theta})\to
	\Gamma(\theta)$, hence
	$\tilde{I}'$ satisfies the assumptions of the lemma and we have
	a decomposition $\tilde{\theta}=\tilde{\theta}^{\tilde{u}}\circ
	\tilde{u}$. Since $q(\tilde{u})=u$, it follows that if the
	lemma holds for $\tilde{\theta}$, then it holds for $\theta$.

	\smallskip
	We assume now that $Z$ is non-singular. If the lemma
	holds for connected components of $Z$, it will hold for $Z$, hence
	it is enough to prove the lemma for $Z$ connected.
	Assume now $Z$ is connected. There is an injective morphism of
	curves $f:Z\to S^1$, where $S^1$ is unoriented. It the lemma
	holds for $S^1$, it holds for $Z$.

	\smallskip
	We assume finally that $Z=S^1$ unoriented.
	Let $i_1\neq i_2\in I'$ such that $i(u_{i_1},u_{i_2})\neq 0$.
	Note that $u_{i_1}$ and $u_{i_2}$ have opposite directions
	and $i(u_{i_1},u_{i_2})=1$. Furthermore,
	$\theta_{i_r}$ has the same direction as $u_{i_r}$, hence
	$i(\theta_{i_1},\theta_{i_2})=i((\theta^u)_{i_1},(\theta^u)_{i_2})+
	1$. Given $i_1\neq i_2\in I$ with $i_1{\not\in} I'$,
	we have $i(u_{i_1},u_{i_2})=0$.
	It follows from Remark \ref{re:productnonexc} that $\theta^u\cdot u\neq 0$.
	This completes the proof of the lemma.
\end{proof}


Note that the length of a map in $\CS^\bullet_M(Z)$ takes value in
a finitely generated submonoid of $\BR_{\ge 0}$. So, a repeated application of the previous
lemma provides a decomposition of any map $\theta$ of
$\CS_M^\bullet(Z)$ as a product $\theta=u_n\cdots u_1$, where $u_i$
is a map $u$ as in the lemma.

\subsubsection{Decomposition at a point}
\label{se:decompositionz0}
Let $z_0\in Z_o$ with $z_0{\not\in}M$.

Given $\zeta$ a homotopy class of admissible paths in $Z$ with
$\zeta\neq\id_{z_0}$, we put
$\mu(\zeta)=i(\zeta,\id_{z_0})$\indexnot{muzeta}{\mu(\zeta)}.

Assume $\mu(\zeta)\ge 1$. There is a unique decomposition
$\zeta=\zeta^{r-}\cdot \zeta^r$ in $\CS^\bullet(Z,1)$ such that
$\zeta^r(1)=z_0$ and $\mu(\zeta^r)=1$.

\smallskip

Given $\theta\in\Hom_{\CS_M^\bullet(Z)}(I,J)$,
we put $\mu(\theta)=\sum_{s\in I}\mu(\theta_s)$.  Given
$\theta'\in\Hom_{\CS_M^\bullet(Z)}(I',I)$ with
$\theta\cdot\theta'\neq 0$, we have $\mu(\theta\cdot\theta')=
\mu(\theta)+\mu(\theta')$.

\begin{lemma}
	\label{le:divisorlength1}
	Let $\theta\in\Hom_{\CS_M^\bullet(Z)}(I,J)$ with
	$\mu(\theta)\ge 2$. 

	There exists a decomposition $\theta=r'(\theta)\cdot r(\theta)$ in
	$\CS_M(Z)$ with $\mu(r(\theta))=1$ and with the following property.

	Let $s\in I$ such that $\mu(r(\theta)_s)=1$. Given $s'\in I$
	such that $\mu(\theta_{s'})\ge 1$ and $\mathrm{supp}(\theta_{s'}^r)\subset
	\mathrm{supp}(\theta_s^r)$, then $s'=s$.
%
%
\end{lemma}

\begin{proof}
We prove the lemma by induction on $|\theta|$. Assume there is
	a set $I'$ satisfying the assumptions of Lemma \ref{le:rightdivisor}
	and such that $\mu(u)=0$. By induction, there is
	a decomposition $\theta^u=r'(\theta^u)\cdot r(\theta^u)$ as in the lemma.
	Now $r(\theta)=r(\theta^u)\cdot u$ and $r'(\theta)=r'(\theta^u)$
	satisfy the requirements of the lemma.

	\smallskip
	Assume now that given any set $I'$ satisfying the assumptions of
	Lemma \ref{le:rightdivisor}, we have $\mu(u)\ge 1$.

	Let $s\in I$ with $\mu(\theta_s)\ge 1$ such that given $s'\in I$
	with $\mu(\theta_{s'})\ge 1$, we have 
	$\mathrm{supp}(\theta_s^r)\subset \mathrm{supp}(\theta_{s'}^r)$.
	Given $s'\in I\setminus\{s\}$, we have $\mu(\alpha^{s'})=0$
	(notations of \S\ref{se:generation}).

	Let $I'$ be the set of $s'\in I$ such that there is a sequence
	$s_0=s,s_1,\ldots,s_r=s'$ of elements of $I$ such that
	$s_i\to s_{i+1}$ is an arrow of $\Gamma(\theta)$ for
	$0\le i<r$. Assume there exist $s_1,\ldots,s_d$ in $I'\setminus\{s_0\}$
	such that $s_d=s_1$ and $s_i\to s_{i+1}$ is an arrow of $\Gamma(\theta)$ for
	$1\le i<d$. Then $I''=\{s_1,\ldots,s_d\}$ satisfies the assumptions of
	Lemma \ref{le:rightdivisor}. On the other hand, we have
	$\mu(\alpha^{s'})=0$ for $s'\in I''$, hence we get a contradiction.
	It follows that $I'$ is a cycle or a line and it satisfies the
	assumptions of Lemma \ref{le:rightdivisor}.
	The braids $r'(\theta)=\theta^u$ and $r(\theta)=u$ 
	of Lemma \ref{le:rightdivisor} satisfy the requirements of the
	lemma.
%
\end{proof}

\subsubsection{Differential}
Let us start with a description of $i(\theta)$ in terms of $L(\theta)$, using
our previous analysis of $S^1$.

\medskip
	Let $f:Z\to Z'$ be a morphism of curves. Given
	$\theta\in\Hom_{\CP^\bullet_f(Z)}(I,J)$,
	the map $f$ induces an injection $f:L(\theta)\hookrightarrow L(f(\theta))$ by the
	discussion above Lemma \ref{le:coverI}.
	
\begin{lemma}
	\label{le:functorialityL}
	Given $\theta'\in f(\Hom_{\CP^\bullet_f(Z)}(I,J))$, the
	map $f$ induces a bijection $\bigcup_{\theta\in f^{-1}(\theta')}
	L(\theta)\iso L(\theta')$. It restricts to a bijection
	$\bigcup_{\theta\in f^{-1}(\theta')}
	D(\theta)\iso D(\theta')$.
\end{lemma}

\begin{proof}
	Assume first $f$ is a  non-singular cover of $Z'$.

	Let $\zeta'\in L(\theta')$. There are $i'_1\neq i'_2\in f(I)$ such that
	$\zeta'\in I(\theta'_{i'_i}, \theta'_{i'_2})$. By Lemma \ref{le:coverI},
	there are elements $\zeta_r\in f^{-1}(\theta'_{i_r})$ and
	$\zeta\in I(\zeta_1,\zeta_2)$
	such $\zeta'=f(\zeta)$.  We define $\theta\in f^{-1}(\theta')$ by setting
	$\theta_{\zeta_r(0)}=\zeta_r$ and by setting $\theta_i$ to be
	any lift of $\theta'_{f(i)}$ for all $f(i) \notin \{i'_1,i'_2\}$. This shows the
	surjectivity part of the first statement of the lemma.
	
	Consider now $\theta$ and $\hat{\theta}$ maps in $\CP_f^\bullet(Z)$ such that
	$f(\theta) = f(\hat{\theta}) = \theta'$. Let
	$\zeta\in L(\theta)$ and $\hat{\zeta}\in L(\hat{\theta})$ such that
	$f(\zeta)=f(\hat{\zeta})=\zeta'$.
	There are $i'_1\neq i'_2\in f(I)$ such that $\zeta' \in 
	I(\theta'_{i'_1}, \theta'_{i'_2})$. We have $\hat{\theta}_{\hat{\zeta}(t)},
	\theta_{\zeta(t)} \in 
	f^{-1}(\theta'_{i'_r})$ for $t \in \{0,1\}$. It follows from
	Lemma \ref{le:coverI} that $\zeta=\hat{\zeta}$. So, the first statement of the lemma
	holds.

	\smallskip
	Assume now $f$ is an open embedding. The injectivity of the first map of the lemma
	is clear, while the surjectivity follows from Lemma \ref{le:supportI}.

	We deduce the first part of the lemma for $Z$ and $Z'$ non-singular and the
	general case follows now by taking non-singular covers
	of $Z$ and $Z'$ and the lift of $f$.

	\smallskip
	Let us prove now the second statement of the lemma about $D(\theta)$.
	
	Consider
	$\theta\in f^{-1}(\theta')$ and
	$\zeta\in L(\theta)$. 
	It is clear that if 
$f(\zeta)\in D(\theta')$, then $\zeta\in D(\theta)$.
	
	Assume now $\zeta\in D(\theta)$. 
	Fix $i_1\neq i_2\in I$ so that $\zeta\in I(\theta_{i_1},\theta_{i_2})$.

	Let $\zeta',\zeta''\in L(\theta')$ such that $f(\zeta)=\zeta'\circ\zeta''$. Let $z=\zeta'(0)=\zeta''(1)$. If $|f^{-1}(\theta'_z)|>1$, then
	$z\in Z_o$ and $\theta'_z=\id$, hence $\zeta'$ and $\zeta''$ have
	opposite orientations by Lemma \ref{le:conditionbDtheta}.
	Assume now $\theta'_z$ has a unique lift.
	Let $\hat{\zeta}'$ and $\hat{\zeta}''$ be the unique lifts of
	$\zeta'$ and $\zeta''$ (first part of the lemma). By unicity of lifts, we have
	$\zeta=\hat{\zeta}'\circ\hat{\zeta}''$. 
	We have $\hat{\zeta}',\hat{\zeta}''\in L(\theta)$, hence
	$\hat{\zeta}'$ and $\hat{\zeta}''$ have opposite orientations.
	It follows that $\zeta'$ and $\zeta''$ have opposite orientations as well.

	Consider now $\zeta':f(i_1)\to f(i_2)$ a smooth homotopy class of paths such that $f(\zeta)\circ\zeta^{\prime -1}$ and
	$\zeta^{\prime -1}\circ f(\zeta)$ are smooth and have the same
	orientation as $f(\zeta)$ and $\zeta'$.
Let $\hat{\zeta}'$ be the unique lift
	of $\zeta'$. Since
	$f(\zeta)\circ\zeta^{\prime -1}$ is smooth, it follows that
	$\hat{\zeta}'(0)=i_1$ and
	$\zeta\circ\hat{\zeta}^{\prime -1}$ is smooth and has the same
	orientation as $\zeta$. Similarly,
	$\hat{\zeta}'(1)=i_2$ and
	$\hat{\zeta}^{\prime -1}\circ \zeta$ is smooth and has the same
orientation as $\zeta$.
	 A similar statement holds for $\zeta$ replaced by $\bar{\zeta}$.
	We deduce that $f(\zeta)\in D(\theta')$.
\end{proof}

\begin{rem}
The picture below shows what would go wrong in Lemma \ref{le:functorialityL} 
if we allowed unoriented points in $Z_{exc}$. In the proof, we need
$\hat{\zeta}'$ and
$\hat{\zeta}''$ to be in $L(\theta)$, which would not be true if this example were valid.  
	$$\includegraphics[scale=1.2]{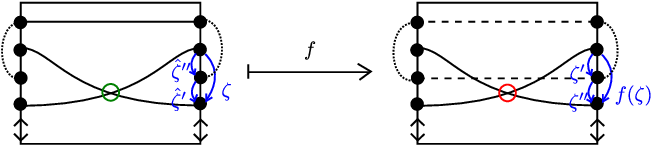}$$
\end{rem}

\begin{prop}
	\label{pr:i=L}
	Let $\theta\in\Hom_{\CP^\bullet(Z)}(I,J)$. We have $i(\theta)=\sum_{\Omega\in\pi_0(Z)}
	|(L(\theta)\cap\Omega)/\mathrm{inv}|e_\Omega$.
	In particular, $L(\theta)$ is finite.
\end{prop}

\begin{proof}
	The statement is true for $Z=S^1$ unoriented by Lemmas \ref{le:lengthaffine},
	\ref{le:S1comparison} and \ref{le:inversionpaths}.
	It follows from Lemmas \ref{le:functorialityL} and \ref{le:intersections}
	that it holds for any connected
	non-singular $Z$, by embedding it in $S^1$. So, the lemma holds for any
	non-singular $Z$.
	By realizing an arbitrary $Z$ as a quotient of its non-singular cover, we
	deduce from Lemmas \ref{le:functorialityL} and \ref{le:intersections} that
	the lemma holds for any $Z$.
\end{proof}

	Given $f:Z\to Z'$ a morphism of curves, given
	$\theta\in\Hom_{\CP^\bullet_f(Z)}(I,J)$ and
	given $\zeta\in L(\theta)$, we have $f(\theta^\zeta)=f(\theta)^{f(\zeta)}$.

\begin{lemma}
	\label{le:thetazeta}
Given $\theta\in\Hom_{\CP^\bullet(Z)}(I,J)$ and $\zeta\in L(\theta)$, we have
	$\theta^\zeta\in \Hom_{\CP^\bullet(Z)}(I,J)$.
	We have $\zeta\in D(\theta)$ if and only if
	$\overline{\deg}_D(\theta^\zeta)=\overline{\deg}_D(\theta)+1$ for some (or
	equivalently, any) finite subset $D$ of $T(Z)$ such that $D\cap\iota(D)=\emptyset$.
\end{lemma}

\begin{proof}
	Let us show the first statement. We can assume
	$\theta_{\zeta(0)}^\zeta\neq\id$.

	Assume $\theta_{\zeta(0)}^\zeta$
	has the same orientation as $\zeta^{-1}$. We have
	$\theta_{\zeta(1)}=\theta_{\zeta(0)}^\zeta\circ\zeta^{-1}$.
If $\gamma$ and $\gamma'$ are minimal paths in $\theta_{\zeta(0)}^\zeta$ and
	$\zeta^{-1}$, then $\gamma\circ\gamma'$ is a minimal path in
	$\theta_{\zeta(1)}$. Since $\gamma\circ\gamma'$ is admissible,
	it follows that $\gamma$ is admissible, hence
	$\theta_{\zeta(0)}^\zeta$ is admissible.

	Otherwise,
	$\theta_{\zeta(0)}^\zeta$ has the same orientation as
	$\bar{\zeta}^{-1}$ and
	$\theta_{\zeta(0)}=\bar{\zeta}^{-1}\circ\theta_{\zeta(0)}^\zeta$,
	hence we deduce as above that $\theta_{\zeta(0)}^\zeta$ is admissible.

	Similarly, $\theta_{\zeta(1)}^\zeta$ is admissible and we deduce that
        $\theta^\zeta$ is a braid.

	\smallskip
	Let us prove the second part of the lemma.
	When $Z=S^1$ unoriented, this holds by Lemmas \ref{le:inversionpaths},
	\ref{le:setDforSn} and
	\ref{le:S1comparison} and Proposition \ref{pr:prestrandsnonsingular}.
	We deduce that the lemma holds when $Z$ is a connected non-singular curve,
	by embedding $Z$ in $S^1$. So, it holds
	when $Z$ is a non-singular curve (since $\supp(\zeta)$ is contained
	in a connected component of $Z$).

	\smallskip
	Consider now a general $Z$ and the non-singular cover $q:\hat{Z}\to Z$.
	There is a braid $\hat{\theta}$ in $\hat{Z}$ with $q(\hat{\theta})=\theta$
	(Lemma \ref{pr:qprestrands}) and there is
	$\hat{\zeta}\in L(\hat{\theta})$ such that
	$\zeta=q(\hat{\zeta})$ (Lemma \ref{le:functorialityL}).
	The considerations above show that $\hat{\theta}^{\hat{\zeta}}$ is a
	braid in $\hat{Z}$, hence $\theta^\zeta=q(\hat{\theta}^{\hat{\zeta}})$
	is a braid in $Z$. The statement on degrees follows from Lemmas
	\ref{le:functorialityL} and \ref{le:fdeg}.
\end{proof}

\medskip
Given $\theta\in\Hom_{\CS^\bullet(Z)}(I,J)$, we put
$$d(\theta)=\sum_{\zeta\in D(\theta)/\mathrm{inv}}\theta^\zeta\in
\Hom_{\CS(Z)}(I,J)\indexnot{d}{d(\theta)}.$$
Note that the set $D(\theta)$ is finite by Proposition \ref{pr:i=L}.

\begin{thm}
	\label{th:strandsdifferential}
	The map $d$ equips $\CS(Z)$ with a structure of
	differential $\bar{\Gamma}(Z,Z_{exc}^+)$-graded
	$\BF_2$-linear category and $\CS^\bullet(Z)$ with a structure of
	differential $\bar{\Gamma}(Z,Z_{exc}^+)$-graded pointed category.

	\smallskip
	Let $f:Z\to Z'$ be a morphism of curves.

	$\bullet\ $The functor $f:\CS_f^\bullet(Z)\to\CS^\bullet(Z')$ is a faithful
	pointed functor and its restriction to $\CS_{\{z\in Z\ |\
	|f^{-1}f(z)|=1\}}^\bullet(Z)$
	is a differential $\bar{\Gamma}(Z',Z_{exc}^{\prime +})$-graded pointed functor.

	$\bullet\ $If $f$ is strict, then
	$f^\#:\add(\CS(Z'))\to\add(\CS_f(Z))$ is a
	differential $\bar{\Gamma}(Z',Z_{exc}^{\prime +})$-graded 
	functor commuting with coproducts.

	$\bullet\ $If $f$ is a quotient morphism, then $f^\#$ is faithful and every map in
	$\CS^\bullet(Z')$ is in the image by $f$ of a map of $\CS^\bullet_f(Z)$.
\end{thm}

\begin{proof}
	Lemma \ref{le:functorialityL} shows that
	$d(f^\#(\theta'))=f^\#(d(\theta'))$ for any $\theta'$ and that
	$d(f(\theta))=f(d(\theta))$ if $|f^{-1}f(\theta)|=1$.

	Assume $Z=S^1$ (unoriented) and consider a finite
	subset $M$ of $Z$ as in \S\ref{se:StrandsS1}. We use the notations
	of that section.
	It follows from Lemma \ref{le:S1comparison} that the isomorphism $F$
	of Proposition \ref{pr:prestrandsnonsingular} induces an isomorphism of
	$\BF_2$-linear categories $F:\BF_2[\CH_n]\iso\CS_M(Z)$.
	It follows now from Lemma \ref{le:inversionpaths} that this isomorphism
	commutes with $d$.
	In particular, $d$ is a differential on $\CS_M(Z)$.
	Since this holds for any finite subset $M$ of $Z$, we deduce that
	$d$ is a differential on $\CS(Z)$.

	Consider now a non-singular connected $Z$ and an injective morphism
	of curves $f:Z\hookrightarrow S^1$. Since $f$ induces a
	faithful $\BF_2$-linear functor $\CS(Z)\to\CS(S^1)$ commuting with
	$d$, we deduce that $d$ is a differential on $\CS(Z)$.

	The decomposition (\ref{eq:Strandscomponents}) is compatible with $d$, hence $d$ is a 
	differential on $\CS(Z)$ for any non-singular $Z$.

	Consider now a general $Z$ and $q:\hat{Z}\to Z$ its non-singular cover.
	Since the additive $\BF_2$-linear functor $q^\#$ commutes with $d$,
	it follows that $d$ is a differential on $\CS(Z)$.

	The last statement of the theorem follows from Lemma \ref{le:liftquotientpaths}.
\end{proof}

There is an isomorphism of differential pointed categories
\begin{equation}
	\label{eq:oppositeStrands}
\CS^\bullet(Z^\opp)\iso\CS^\bullet(Z)^\opp,\ I\mapsto I,\ \theta\mapsto (\theta_s^{-1})_s.
\end{equation}


\medskip
Note that the construction $Z\mapsto \add(\CS(Z))$ and
$f\mapsto f^\#$ defines a contravariant
functor from the category
of curves with strict morphisms to the category of
differential categories.

\subsubsection{Strands on non-singular curves}
\label{se:strandsnonsingular}

We consider as in \S\ref{se:StrandsS1} a family 
$M=\{a_1,\ldots,a_n\}$ of points on $S^1$ and $z\in S^1-M$
such that $a_1,\ldots,a_n,z$ is cyclically ordered.

\smallskip
The next proposition follows immediately from Proposition \ref{pr:prestrandsnonsingular} and Lemmas \ref{le:S1comparison} and \ref{le:inversionpaths}.

\begin{prop}
	\label{pr:strandsnonsingular}
The functor $F$ induces an isomorphism of differential pointed categories
$\CH_n\iso \CS^\bullet_M(S^1)$. It restricts to isomorphisms of
	differential pointed categories
	$$\CH_n^+\iso\CS^\bullet_M(\dot{S}^1),\
	\CH_n^{++}\iso\CS^\bullet_M(\vec{S}^1),\
	\CH_n^f\iso\CS^\bullet_M(I) \text{ and }
	\CH_n^{f++}\iso\CS^\bullet_M(\vec{I}).$$
\end{prop}

The isomorphism 
$\Gamma_{[1,n]^+}\iso \Gamma_{M}(\vec{S^1})$
 of Lemma \ref{le:S1comparison}
restricts to an isomorphism of groups 
$\Gamma^f_{[1,n]^+}\iso \Gamma_{M}(\vec{I})$
and
the isomorphism
	$\CH_n^{f++}\to\CS^\bullet_M(\vec{I})$ of Proposition 
\ref{pr:strandsnonsingular} is compatible with the grading by those groups.

\medskip

Consider $Z=\BR_{>0}$ as an unoriented curve.
We denote by $\CS_{\otimes}^\bullet(\BR_{>0})$ the full subcategory of $\CS^\bullet(\BR_{>0})$
with objects the subsets of the form $\{1,\ldots,n\}$ for some $n\in\BZ_{\ge 0}$.
We define a monoidal structure on 
the differential pointed category $\CS_{\otimes}^\bullet(\BR_{>0})$ by
$\{1,\ldots,n\}\otimes \{1,\ldots,m\}=\{1,\ldots,n+m\}$
and $\theta''=\theta\otimes\theta'$ is defined by
$\theta''_i=\theta_i$ if $i\le n$ and
$\theta''_i=\theta'_{i-n}$ otherwise.

\smallskip
The next theorem follows immediately from Proposition \ref{pr:strandsnonsingular}.

\begin{thm}
	\label{th:Ufromstrands}
There is an isomorphism of
differential pointed monoidal categories
	$\CU^\bullet\iso \CS_{\otimes}^\bullet(\BR_{>0})$ defined by $e\mapsto \{1\}$ and
	$\tau$ maps to the non-zero and non-identity element of $\End_{\CS^\bullet(\BR_{>0})}(\{1,2\})$.
\end{thm}

\subsubsection{Products and divisibility}

\begin{lemma}
	\label{le:middlebraids}
	Consider braids $\theta'':I\to J$ and $\theta':J\to K$
	and assume $\theta=\theta'\cdot\theta''$ is non-zero.
	Let $\zeta\in D(\theta)\setminus (D(\theta)\cap D(\theta''))$.
	Assume $\zeta$ and $\zeta^{-1}$ are oriented.

	Define $\alpha'':I\to J$ by
	$$\alpha''_s=\begin{cases}
		\theta''_{\zeta(1)}\circ\zeta & \text{ if }s=\zeta(0)\\
		\theta''_{\zeta(0)}\circ\zeta^{-1} & \text{ if }s=\zeta(1)\\
		\theta''_s & \text{ otherwise.}
	\end{cases}$$
	Let $\zeta'=\theta''_{\zeta(1)}\circ\zeta\circ(\theta''_{\zeta(0)})^{-1}$
	and $\alpha'=(\theta')^{\zeta'}$.
	Then $\alpha''$ and $\alpha'$ are braids and
	$\theta=\alpha'\cdot\alpha''$.
\end{lemma}

\begin{proof}
	Since $\zeta$ and $\zeta^{-1}$ are oriented, it follows that
	$\alpha''_s$ is oriented for all $s$.
	Also, it follows from Lemma \ref{le:thetazeta} that $\alpha'$ is a braid.

	\smallskip
	Consider first the case where $Z=S^1$ unoriented. In that case,
	the lemma follows from Proposition \ref{pr:strandsnonsingular} and Lemmas \ref{le:inversionpaths} and
	\ref{le:middleHeckecategory}.

	\smallskip
	Assume now $Z$ is smooth and connected.
	There is an injective
	morphism of curves $f:Z\to S^1$, where $S^1$ is unoriented.
	Since the lemma holds for $S^1$, we deduce that 
	it holds for $Z$.

	When $Z$ is only assumed to be smooth, the lemma follows
	from the case of the connected component containing $\zeta$.

	\smallskip
	Consider now the general case.
	Let $f:\tilde{Z}\to Z$ be a smooth cover. Let $\tilde{\theta}$
	be a braid lifting $\theta$. There are unique braids 
	$\tilde{\theta}'$ and $\tilde{\theta}''$ in $\tilde{Z}$
	with $\tilde{\theta}=\tilde{\theta}'\cdot\tilde{\theta}''$ and
	$f(\tilde{\theta}')=\theta'$, 
	$f(\tilde{\theta}'')=\theta''$.
	There is a unique $\tilde{\zeta}\in D(\tilde{\theta})$ with $f(\tilde{\zeta})=
	\zeta$ (Lemma \ref{le:functorialityL}). We have $\tilde{\zeta}{\not\in}
	D(\tilde{\theta}'')$ (Lemma \ref{le:functorialityL}). Since the lemma holds
	for $\tilde{Z}$, we deduce it holds for $Z$.
\end{proof}

\subsubsection{Subcurves}
\label{se:tensorstrands}
Let $Z$ be a curve.

Let $S$ and $T$ be two finite subsets of $Z$.
 Let $S_1$ be a subset of $S$ and $S_2=S\setminus S_1$.
 Let $T_1$ be a subset of $T$ and $T_2=T\setminus T_1$.
Let $\Phi_i\in\Hom_{\CS^\bullet(Z)}(S_i,T_i)$.
We define $\Phi=\Phi_1\boxtimes\Phi_2\in \Hom_{\CS^\bullet(Z)}(S,T)$ by
	$\Phi_s=(\Phi_i)_s$ when $s\in S_i$.
This gives an injective map of pointed sets
$$\Hom_{\CS^\bullet(Z)}(S_1,T_1)\wedge\Hom_{\CS^\bullet(Z)}(S_2,T_2)
\hookrightarrow \Hom_{\CS^\bullet(Z)}(S,T).$$
Note that this is not compatible with composition in general.
We obtain an isomorphism of pointed sets
$$\bigvee_{\substack{T'_1\subset T\\ |T'_1|=|S_1|}}
\bigl(\Hom_{\CS^\bullet(Z)}(S_1,T'_1)\wedge\Hom_{\CS^\bullet(Z)}(S_2,T\setminus
T'_1)\bigr)
\iso \Hom_{\CS^\bullet(Z)}(S,T).$$
We have corresponding morphisms of $\BF_2$-modules between $\Hom$-spaces in
	$\CS(Z)$. Note these are not compatible with the differential.

\smallskip

Assume $S_2=T_2$. The map $\Phi_1\mapsto \Phi_1\boxtimes \id_{S_2}$ defines
a canonical embedding of pointed sets (not compatible
with the differential nor the multiplication in general)
$$\Hom_{\CS^\bullet(Z)}(S_1,T_1)\hookrightarrow
\Hom_{\CS^\bullet(Z)}(S_1\sqcup S_2,T_1\sqcup S_2).$$

\medskip
Given $Z_1$ and $Z_2$ two disjoint closed subcurves of $Z$, we obtain
a faithful differential pointed functor
$$\CS^\bullet(Z_1)\wedge\CS^\bullet(Z_2)\to
\CS^\bullet(Z),\
(S_1,S_2)\mapsto S_1\sqcup S_2.$$

\smallskip
Let $Z_1,\ldots,Z_r$ be the connected components of $Z$. The construction
above induces
an isomorphism of differential pointed categories (cf (\ref{eq:Scomponents}))
\begin{equation}
	\label{eq:decompositioncomponents}
\CS^\bullet(Z_1)\wedge\cdots\wedge\CS^\bullet(Z_r)\iso
\CS^\bullet(Z),\
(S_1,\ldots,S_r)\mapsto S_1\sqcup\cdots\sqcup S_r.
\end{equation}

\smallskip

	Let us record a case where the tensor product construction $\boxtimes$
	is compatible
	with composition and the differential in the following immediate lemma.

	\begin{lemma}
		\label{le:commutetensor}
	Let $M$ be a subset of $Z$ and let $Z'$ be a subcurve of $Z$.
Assume that given an admissible homotopy class of paths $\zeta$ in $Z$
with endpoints in $M$, there is an admissible path $\gamma$ in $\zeta$ contained
in $Z-Z'$. 
There is a faithful functor of differential pointed categories
		\begin{align*}
			\CS^\bullet_M(Z)\wedge\CS^\bullet(Z')&\to \CS^\bullet_{M\cup Z'}(Z)\\
			(S,T)&\mapsto S\sqcup T\\
			(\alpha,\beta)&\mapsto
\alpha\boxtimes\beta=(\alpha\boxtimes \id)\cdot (\id\boxtimes\beta)=
		(\id\boxtimes\beta)\cdot(\alpha\boxtimes \id).
		\end{align*}
\end{lemma}


\subsubsection{Bordered Heegaard Floer algebras} \label{se:HF}

We consider a chord diagram $(\CZ,\Ba)$ as in \S \ref{se:arcdiagrams}.
Let $Z_1,\ldots, Z_l$ be the connected components of $\CZ$. Let $\tilde{\Ba}=
\bigcup_{\{z,z'\}\in\Ba}\{z,z'\}$, 
$n_i=|\tilde{\Ba}\cap Z_i|$ and
let $q:\tilde{Z}\to Z$ be the quotient map.

The isomorphism (\ref{eq:decompositioncomponents}) associated with
the decomposition $\tilde{Z}=\mathring{Z}_1\coprod\cdots\coprod \mathring{Z}_l$ together
with the strands algebra description of \S \ref{se:strandsalgebras}
and the isomorphism of Proposition \ref{pr:strandsnonsingular}
induce an isomorphism of differential algebras

$$\CA(n_1)\otimes\cdots\otimes\CA(n_l)\iso
\End_{\add(\CS(\tilde{Z}))}(\bigoplus_{I\subset\tilde{\Ba}}I).$$
It is compatible with the gradings, via the embedding
$G'(n_1)\times\cdots\times G'(n_l)\hookrightarrow \Gamma_{\Ba}(\tilde{Z})$
given by \S \ref{se:strandsalgebras} and \S \ref{se:strandsnonsingular}.

\smallskip
The differential algebra $\CA(\CZ)$ associated to $\CZ$ is a differential
$(G'(n_1)\times\cdots\times G'(n_l))$-graded non-unital 
subalgebra of $\CA(n_1)\otimes\cdots\otimes\CA(n_l)$ (cf 
\cite[Definition 2.6]{Za} and \cite[Definition 3.23]{LiOzTh1} for the original
setting where $l=1$).
There is a unique isomorphism of differential algebras
	$$\CA(\CZ)\iso \End_{\add(\CS(Z))}\bigl(\bigoplus_{S\subset \Ba}S\bigr)$$
making the following diagram commutative
$$\xymatrix{\CA(\CZ)\ar[r]^-\sim \ar@{^{(}->}[d] &
	\End_{\add(\CS(Z))}\bigl(\bigoplus_{S\subset \Ba}S\bigr)\ar@{^{(}->}[d]^{q^\#} \\
	\CA(n_1)\otimes\cdots\otimes\CA(n_l) \ar[r]_-\sim & \End_{\add(\CS(\tilde{Z}))}(\bigoplus_{I\subset\tilde{\Ba}}I)}$$

\subsubsection{Fukaya categories from strand algebras}
\label{se:Fukayastrands}
Consider an oriented singular curve $Z$ with $n_z\in\{2,4\}$ for all $z$ and its
corresponding chord diagram $(\CZ,\Ba)$ (cf \S \ref{se:arcdiagrams}).
Let $(F,\Lambda,S^+,S^-)$ be the associated sutured surface. We assume
that every component of $\partial F$ intersects
$S^+$ non-trivially (cf \S\ref{se:TQFT}).
Choose for each component $E$ of $S^-$ a point $e_E\in E$ and let
$S=\{e_E\}_E$. We have obtained a pair $(F,S)$ where $S$ is a finite subset of $\partial F$.

Consider the arcs $\omega_z$ for $z\in Z_{exc}$ (cf \S \ref{se:arcdiagrams}).
Note that $F\setminus\bigl(\bigcup_{z\in Z_{exc}}\omega_z\bigr)$ is a union of discs, each of which
contains one point of $S$.

\smallskip
Auroux \cite[Definition 8]{Au2} considers a partially wrapped Fukaya
category $\CF(\mathrm{Sym}^nF,S)$
of the symmetric power
$\mathrm{Sym}^n(F)$ of $F$ with set of stops $S\times\mathrm{Sym}^{n-1}(F)$. This
is an (ungraded) $A_\infty$-category over $k$.

Let $s,t\in Z_{exc}$. A path in $Z$ gives rise to a path in $F$ and this defines 
a bijection $f$ from the set of admissible homotopy classes of paths $s\to t$ in $Z$ to 
the set $\bar{\chi}_t^s$ of \cite[Proposition 11]{Au2} (recall the orientation reversal,
cf Convention \ref{con:reversal}). When $s=t$, the trivial path is
sent to the element $\bf{1}_i$ of Auroux.

Auroux \cite[Proposition 11]{Au2} relates the $A_\infty$-category
$\CF(\mathrm{Sym}^nF,S)$ to
the strand algebra associated with
 $Z$.

\begin{thm}[Auroux]
	\label{th:Auroux}
	There is a fully faithful $A_\infty$-functor 
	$$\Phi:\CA(Z,n)\to \CF(\mathrm{Sym}^nF,S),\
	I\mapsto \prod_{i\in I}\omega_i,\ \theta\mapsto (\chi(\theta),(f(\theta_s))_s)$$
	inducing an equivalence of derived categories.
\end{thm}

\section{$2$-representations on strand algebras}
\label{se:2repstrand}
\subsection{Action on ends of curves}
\label{se:leftaction}

\subsubsection{Definition}
\label{se:defleftaction}

Let $\xi:\BR_{>0}\to Z$ be an injective morphism of curves, where $\BR_{>0}$ is
viewed as an unoriented curve. Let $M$ be a subset of $Z\setminus\xi(\BR_{\ge 1})$.

\smallskip
We say that $\xi$ is {\em terminal}\index[ter]{terminal morphism} for $(Z,M)$ if the
following two conditions hold:
\begin{itemize}
	\item given an admissible homotopy class of paths $\zeta$ in $Z$
with endpoints in $M$, there is an admissible path $\gamma$ in $\zeta$ contained
in $Z\setminus\xi(\BR_{\ge 1})$
		\item there is no admissible path in $Z\setminus\{\xi(1)\}$
from a point of $M$ to $\xi(2)$.
\end{itemize}

Note that $\xi$ is terminal for $(Z,M)$ if and only if
$\xi$ is terminal for $(Z(\xi),Z(\xi)\cap M)$, where $Z(\xi)$\indexnot{Zx}{Z(\xi)} is the component of $Z$
containing $\xi(\BR_{>0})$.

\smallskip
We say that $\xi$ is {\em outgoing}\index[ter]{outgoing morphism}
for $Z$ if $\xi(\BR_{\ge 1})$ is closed in $Z$. Note
that if $\xi$ is outgoing for $Z$ then it is terminal for $(Z,M)$ for any
$M\subset Z\setminus\xi(\BR_{\ge 1})$.

\begin{rem}
	\label{re:reductionoutgoing}
Assume $\xi$ is not outgoing for $Z$ and let $z_0\in Z$ such that 
$\overline{\xi(\BR_{\ge 1})}\setminus\xi(\BR_{\ge 1})=\{z_0\}$. Note that
$\xi$ is outgoing for $Z\setminus\{z_0\}$. The map
$\xi$ is terminal for $(Z,M)$ if and only if $z_0{\not\in M}$ and
the inclusion induces an isomorphism
	$\Hom_{\CA^\bullet(Z\setminus\{z_0\},1)}(m,z)\iso 
	\Hom_{\CA^\bullet(Z,1)}(m,z)$ for all $m\in M$ and $z\in M\cup\{\xi(1)\}$.
\end{rem}

\medskip
We assume now that $\xi$ is terminal for $(Z,M)$.
Thanks to Lemma \ref{le:commutetensor}, we have
a differential pointed functor

$$L^\bullet=L_\xi^\bullet:\CS^\bullet_M(Z)\times\CS^\bullet_M(Z)^\opp\times\CU^\bullet\to \diff$$
$$L^\bullet(T,S,e^n)=\Hom_{\CS^\bullet(Z)}(S,T\sqcup\{\xi(1),\ldots,\xi(n)\})$$
$$L^\bullet(\beta,\alpha,\sigma)(f)=(\beta\boxtimes\xi(\sigma))\cdot f\cdot\alpha \in L^\bullet(T',S',n)$$
for $\alpha\in\Hom_{\CS^\bullet(Z)}(S',S)$, $\beta\in\Hom_{\CS^\bullet(Z)}(T,T')$,
$\sigma\in\End_{\CU^\bullet}(e^n)$, and $f\in L^\bullet(T,S,n)$.
We have used the strands realization of $\CU^\bullet$ given by Theorem \ref{th:Ufromstrands}.


\smallskip
We put $L^\bullet(T,S)=L^\bullet(T,S,e)$.
As usual, we put $L_\xi=\BF_2[L_\xi^\bullet]$.

\smallskip
The naturality in the next lemma is immediate as in Lemma \ref{le:commutetensor}.

\begin{lemma}
	\label{le:decompositionnodiff}
Given $S\subset M$ and $n\ge 0$, there is an isomorphism of functors
	$\CS^\bullet_M(Z)\to \Sets^\bullet$ (forgetting the differential)
\begin{align*}
\bigvee_{\substack{S'\subset S\\ |S'|=n}}
	\Hom_{\CS^\bullet(Z)}(S',\{\xi(1),\ldots,\xi(n)\})\wedge \Hom_{\CS^\bullet(Z)}(S\setminus S',-) &\iso L^\bullet(-,S,e^n)\\
	(\alpha,\beta)&\mapsto \alpha\boxtimes \beta.
\end{align*}
\end{lemma}

Lemma \ref{le:decompositionnodiff} shows that there is an isomorphism
of functors, functorial in $S$ and $T$
\begin{align*}
	L^\bullet(T,-,e^n)\wedge L^\bullet(-,S,e^m)&\iso L^\bullet(T,S,e^{n+m})\\
	(\alpha,\beta) &\mapsto
	(\alpha\boxtimes \xi([r\to n+r]_{1\le r\le m}))\cdot\beta.
\end{align*}

\smallskip
The functor $E=E_\xi=L^\bullet(-,-)$ gives a 
bimodule $2$-representation on $\CS^\bullet_M(Z)$.
The endomorphism $\tau$ of $L^\bullet(-,-,e^2)$ is given by
the non-identity non-zero braid $\{1,2\}\to\{1,2\}$.

\medskip
We have obtained the following proposition.

\begin{prop}
	\label{pr:left2rep}
The bimodule $E$ and the endomorphism $\tau$ define a
	bimodule $2$-representation on $\CS^\bullet_M(Z)$ and on $\CS_M(Z)$.
\end{prop}

Lemma \ref{le:decompositionnodiff} shows that $L_\xi(-,-)$ is left finite. 

\begin{rem}
	\label{re:leftasbottomalgebramodule}
	Proposition \ref{pr:left2rep} generalizes and make more precise a result
	of Douglas and Manolescu \cite[\S 5.2]{DouMa}.

	Let $(\CZ,\Ba)$ be a chord diagram where $\CZ=[0,1]$. Let
	$\tilde{Z}'=(0,\infty)$, viewed as a curve with
	$\tilde{Z}'_o=\tilde{Z}=(0,1)$ (with its usual orientation).
	We extend the equivalence relation from $\tilde{Z}$ to $\tilde{Z}'$
	by having all points of $[1,\infty)$ alone in their class.
	Let $Z'=\tilde{Z}'/\!\sim$. We have $Z'_o=Z_o$.
	Let $M=Z'_{exc}$ be the image of $\Ba$ in $Z'$. Let $\xi:\BR_{>0}\to
	Z',\ x\mapsto x+1$. Note that $\xi$ is outgoing for $Z'$. 

	The lax $2$-representation underlying the $2$-representation on
	$\CS_M(Z)=\CS_M(Z')$
	provided by Proposition \ref{pr:left2rep} is the ``bottom algebra
	module" constructed by Douglas and Manolescu, via the identification
	of \S \ref{se:algebramodules}.
\end{rem}

\begin{example}
	\label{ex:leftaction}
	The left picture below gives an example where $\xi$ is terminal
	for $(Z,M)$ but not outgoing for $Z$. The right picture is an
	example where $\xi$ is outgoing for $Z$.
$$\includegraphics[scale=1.60]{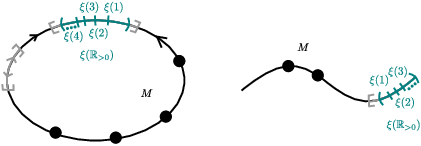}$$

The picture below considers the case of a curve quotient of
the disjoint union of an interval and a circle, with an outgoing $\xi$
	at an end of the interval. The middle picture describes an
	element of $L^\bullet_\xi(-,-,e^2)$. The rightmost picture
	provides a different graphical representation of that element: the
	interval $\xi(\BR_{\ge 1})$ has been moved to the bottom horizontal
	line.
$$\includegraphics[scale=0.90]{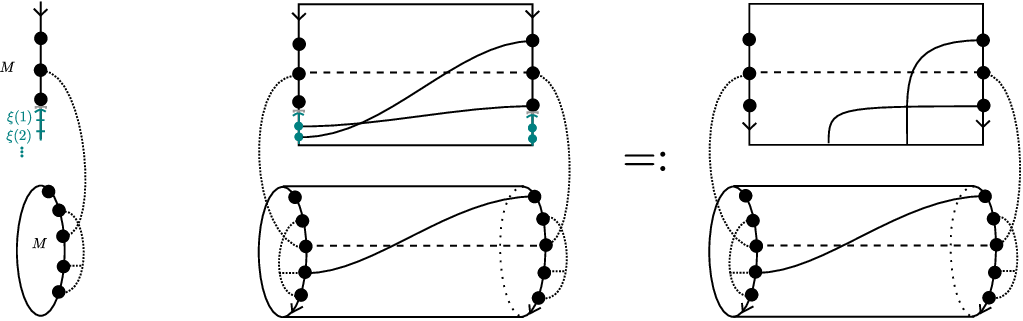}$$
\end{example}

The next remark discusses the dependence of $L_\xi^\bullet$ on $\xi$.

\begin{rem}
\label{se:indep}
	Assume $\xi$ is terminal for $(Z,M)$.
Consider $f:Z\iso Z$ an isomorphism of curves fixing $M$. Note that
$f\circ\xi$ is terminal for $(Z,M)$ and the map $f$ induces an
isomorphism $L_{\xi}^\bullet\iso L_{f\circ\xi}^\bullet$.

	Consider now another injective morphism of curves $\xi':\BR_{>0}\to Z$
	such that $\xi'$ is terminal for $(Z,M)$. Assume 
	there is a connected open subset $U$ of $Z_u$ containing
	$\overline{\xi(\BR_{>0})}$ and $\overline{\xi'(\BR_{>0})}$ and assume
	the canonical
	orientations on $\xi(\BR_{>0})$ and $\xi'(\BR_{>0})$ extend
	to an orientation of $U$. There is an isomorphism of curves
	$f:Z\iso Z$ fixing $Z\setminus U$ such that $\xi'=f\circ\xi$.
	It induces an isomorphism $L_{\xi}^\bullet\iso L_{\xi'}^\bullet$,
	and that isomorphism does not depend on the choice of $f$.
\end{rem}

\subsubsection{Approximation}
Assume $\xi^{-1}(M)$ has no maximum. Fix
an increasing sequence $m_0,m_1,\ldots$ of points of
$(0,1)$ with $\xi(m_i)\in M$ for all $i$ and with $\lim_i m_i>t$ for all $t\in \xi^{-1}(M)$.

Fix $n\ge 0$ and 
define the braid $\beta_r:\{m_r,\ldots,m_{r+n-1}\}\to \{1,\ldots,n\}$ of $\BR_{>0}$ by
$(\beta_r)_{m_{r+i}}=[m_{r+i}\to i+1]$.

\smallskip
Let $S$ and $T$ be two finite subsets of $M$.
Consider $r$ such that $m_r>\xi^{-1}(t)$ for all $t\in T\cap\xi(\BR_{>0})$. There is an isomorphism
$$\Hom_{\CS^\bullet(Z)}
(S,T\sqcup \xi(\{m_r,\ldots,m_{r+n-1}\}))\iso L^\bullet(T,S,e^n),\
\alpha\mapsto (\id_T\boxtimes\xi(\beta_r))\cdot\alpha.$$

It follows that there are isomorphisms functorial in $S$ and $T$
\begin{equation}
	\label{eq:Lascolim}
	\colim_{r\to\infty}\Hom_{\CS^\bullet_M(Z)}
(S,T\sqcup \xi(\{m_r,\ldots,m_{r+n-1}\}))\iso L^\bullet(T,S,e^n).
\end{equation}
Here, the colimit is taken over the invertible maps
$\xi(\theta_r)$, where $\theta_r:\{m_r,\ldots,m_{r+n-1}\}\to\{m_{r+1},\ldots,m_{r+n}\}$ is
the braid in $\BR_{>0}$ given by
$(\theta_r)_{m_s}=[m_s\to m_{s+1}]$.

We deduce that $T\mapsto
(S\mapsto L^\bullet(T,S,e^n))$ is isomorphic to the functor
$$\CS^\bullet_M(Z)\to\CS^\bullet_M(Z)\mdiff,\ 
T\mapsto \colim_{r\to\infty}T\sqcup \xi(\{m_r,\ldots,m_{r+n-1}\}).$$

\subsubsection{$2$-representations and morphisms of curves}
\label{se:2repmorphismscurves}
Let $f:Z\to Z'$ be a morphism of curves.
Assume $\xi$ is terminal for
$(Z,M)$ and $f\circ\xi$ is terminal for $(Z',f(M))$.

\smallskip
Assume that $|f^{-1}(f(z))|=1$ for all $z\in M$.
Let $M_f$  be the $(\CS^\bullet_M(Z),\CS^\bullet_{f(M)}(Z'))$-bimodule corresponding to $f$,
\ie\ given by $M_f(S,S')=\Hom_{\CS^\bullet(Z')}(S',f(S))$. There is a morphism
of functors $E_\xi\wedge_{\CS^\bullet_{M}(Z)}M_f\to M_f\wedge_{\CS^\bullet_{f(M)}(Z')}E_{f\circ\xi}
$ defined as making the following diagram commutative
$$\xymatrix{
	\Hom_{\CS^\bullet(Z)}(-,T\sqcup\{\xi(1)\})\wedge \Hom_{\CS^\bullet(Z')}(S',f(-))\ar[d]
	\ar[rr]^-{\beta\wedge\alpha'\mapsto f(\beta)\cdot\alpha'} &&
 \Hom_{\CS^\bullet(Z')}(S',f(T)\sqcup\{f\circ\xi(1)\}) \\
	\Hom_{\CS^\bullet(Z')}(-,f(T))\wedge\Hom_{\CS^\bullet(Z')}(S',-\sqcup\{f\circ\xi(1)\})
	\ar[urr]^\sim_-{\ \ \ \ \ \ \ \ \  \beta'\wedge\alpha' \mapsto (\beta'\boxtimes\id_{f\circ\xi(1)})\cdot\alpha'}
}$$

\smallskip
The following lemma is a consequence of  (\ref{eq:Lascolim}).

\begin{lemma}
	\label{le:2repcurvesf}
	If $\xi^{-1}(M)$ has no maximum,
	then the construction above gives an isomorphism
	$$E_\xi\wedge_{\CS^\bullet_{M}(Z)}M_f\iso M_f\wedge_{\CS^\bullet_{f(M)}(Z')}E_{f\circ\xi},$$
and $f$ provides a morphism of bimodule
	$2$-representations $L^\bullet_{f\circ\xi}\to L^\bullet_{\xi}$.
\end{lemma}

\medskip
We consider now an arbitrary $M$ but we assume that $f$ is strict.
Let $M_{f^\#}$  be the $(\CS_{f(M)}(Z'),\CS_M(Z))$-bimodule corresponding to $f^\#$,
\ie\ given by 
$$M_{f^\#}(S',S)=\bigoplus_{p:S'\to Z,\ f\circ p=\id_{S'}} \Hom_{\CS_{M}(Z)}(S,p(S')).$$

 There is a morphism
of functors $E_{f\circ\xi}\otimes_{\CS_{f(M)}(Z')}M_{f^\#}\to  M_{f^\#}
\otimes_{\CS_M(Z)}E_{\xi}$
defined as making the following diagram commutative
$$\xymatrix{
	{\displaystyle\bigoplus_{\substack{p:-\to Z\\ f\circ p=\id}}}
	\Hom_{\CS(Z')}(-,T'\sqcup\{f\circ\xi(1)\})\otimes\Hom_{\CS(Z)}(S,p(-))\ar[d]
	\ar[rr]^-{\beta'\wedge\alpha\mapsto f^\#(\beta')\cdot\alpha} &
& {\displaystyle\bigoplus_{\substack{p:T'\to Z\\ f\circ p=\id_{T'}}}}
\Hom_{\CS(Z)}(S,p(T')\sqcup\{\xi(1)\}) \\
	{\displaystyle\bigoplus_{\substack{p:T'\to Z\\ f\circ p=\id_{T'}}}}
	\Hom_{\CS(Z)}(-,p(T'))\otimes\Hom_{\CS(Z)}(S,-\sqcup\{\xi(1)\})\ar[urr]^\sim_-{\ \ \ \ \ \ \ \beta\wedge\alpha
\mapsto (\beta\boxtimes\id_{\{\xi(1)\}})\cdot\alpha}
}$$

\smallskip
The following lemma is a consequence of  (\ref{eq:Lascolim}).

\begin{lemma}
	If $\xi^{-1}(M)$ has no maximum, then the construction above gives an isomorphism
$$E_{f\circ\xi}\otimes_{\CS_{f(M)}(Z')}M_{f^\#}\iso  M_{f^\#}
\otimes_{\CS_M(Z)}E_{\xi},$$
and $f^\#$ provides a morphism of bimodule
	$2$-representations $L_{\xi}\to L_{f\circ\xi}$.
\end{lemma}

\subsubsection{Twisted object description}
\label{se:twistedL}
We explain how to obtain a version of Lemma \ref{le:decompositionnodiff} with a
differential.

We say that a homotopy class of path in $Z(\xi)$ is positive if it has the same
orientation as $\xi([1\to 2])$.

\smallskip
Fix a finite subset $S$ of $M$ and $n\ge 0$.

\smallskip
Let $S''$ be a subset of $S$ with $n$ elements, let $s'\in S\setminus S''$ and $s''\in S''$.
Let $\zeta:s''\to s'$ be a positive smooth homotopy class of paths in $Z$.
We put $S'=(S''\setminus\{s''\})\sqcup\{s'\}$.

We define a map
$$g_{S'',\zeta}:\Hom_{\CS^\bullet(Z)}(S'',\{\xi(1),\ldots,\xi(n)\})\to 
\Hom_{\CS^\bullet(Z)}(S',\{\xi(1),\ldots\xi(n)\})\wedge
\Hom_{\CS^\bullet(Z)}(S\setminus S',S\setminus S'').$$
We put
$$g_{S'',\zeta}(\alpha)=
(\alpha_{|S'\setminus\{s'\}}\boxtimes (\alpha_{s''}\circ\zeta^{-1}))\wedge
(\id_{S\setminus (S''\sqcup\{s'\})}\boxtimes\zeta)
$$
if
\begin{itemize}
	\label{le:2repcurvesfsharp}
	\item $\alpha_{s''}\circ\zeta^{-1}$ is smooth
	\item and
given $s\in S''\setminus\{s''\}$ and $\zeta':s\to s'$ and $\zeta'':s''\to s$ smooth positive
with $\zeta=\zeta'\circ\zeta''$ and with
		$\alpha_{s''}\circ\zeta^{\prime\prime -1}$ smooth,
		then $\alpha_{s''}\circ\zeta^{\prime\prime -1}\circ\alpha_{s}^{-1}$ is negative.
\end{itemize}

We put $g_{S'',\zeta}(\alpha)=0$ otherwise.

\begin{rem}
Note that if $\alpha_{s''}\circ\zeta^{-1}$ is smooth, then the support of
$\zeta$ is contained in $Z(\xi)$.

Given $\alpha$ non-zero, if $\zeta$ is positive,
then both $\alpha_{s''}\circ\zeta^{-1}$ and $\zeta$ are oriented, since $\alpha_{s''}$
is oriented.
\end{rem}

\medskip
We obtain a map
$f_{S'',\zeta}:\alpha\wedge\beta\mapsto (\id\wedge\beta)\circ g_{S'',\zeta}(\alpha)$
$$\Hom_{\CS^\bullet(Z)}(S'',\{\xi(1),\ldots,\xi(n)\})\wedge
\Hom(S\setminus S'',-)\to 
\Hom_{\CS^\bullet(Z)}(S',\{\xi(1),\ldots\xi(n)\})\wedge
\Hom_{\CS^\bullet(Z)}(S\setminus S',-).$$

Let $r(S'')$ be the number of pairs $(s'',s)\in S''\times (S\setminus S'')$ such that
there exists a positive path $s''\to s$.

\smallskip
We define now
$$V_r=\bigoplus_{\substack{S'\subset S,\ |S'|=n\\ r(S')=r}}
\Hom(S',\{\xi(1),\ldots,\xi(n)\})\otimes
\Hom_{\CS(Z)}(S\setminus S',-)\in\CS_M(Z)\mdiff.$$

Given $r'<r''$, define 
$f_{r',r''}=\sum_{S'',\zeta}f_{S'',\zeta}$, where 
\begin{itemize}
	\item $S''$ is a subset of $S$ with $|S''|=n$ and $r(S'')=r''$
	\item $\zeta$ is a positive admissible homotopy class of paths in
$Z$ with $\zeta(0)\in S''$ and $\zeta(1)\in S\setminus S''$
\end{itemize}
such that
$\supp(\zeta)\cap S''=\{s''\}$ and $r((S''\setminus\{\zeta(0)\})\sqcup\{\zeta(1)\})=r'$.

\smallskip
Let $V=V_n(S)=\bigoplus_r V_r$ and let $d_V=\sum_r d_{V_r}+\sum_{r',r''}f_{r',r''}$.
We will show below (Proposition \ref{pr:descriptiononeside})
that $d_V^2=0$,\ \ie\ $V$ is the object
of $\CS_M(Z)\mdiff$ corresponding to the twisted object
$[\bigoplus V_r,(f_{r',r''})]$.

\begin{prop}
	\label{pr:descriptiononeside}
	Given $S\subset M$ and $n\ge 0$, then $d_{V_n(S)}^2=0$ and the
	map of Lemma \ref{le:decompositionnodiff}
	defines an isomorphism of functors $\CS_M(Z)\to k\mdiff$
	$$V_n(S)\iso L(-,S,e^n).$$
\end{prop}

\begin{proof}
	By Remark \ref{re:reductionoutgoing}, we can assume $\xi$ is outgoing for $Z$.
We will show that
	\begin{equation}
		\label{eq:eq1}
		\text{the isomorphism of Lemma \ref{le:decompositionnodiff} is compatible
		with the differentials.}
	\end{equation}
	The proposition will follow immediately from (\ref{eq:eq1}).

	Let $S''$ be a subset of $S$ with $n$ elements, and let $T$ be
	a finite subset of $M$.
	Let $a:\Hom_{\CS^\bullet(Z)}(S'',\{\xi(1),\ldots\xi(n)\})\wedge
	\Hom_{\CS^\bullet(Z)}(S\setminus S'',T)\to L^\bullet(T,S,e^n)$ be the map of
	Lemma \ref{le:decompositionnodiff}.
	Let $\alpha\in\Hom_{\CS^\bullet(Z)}(S'',\{\xi(1),\ldots\xi(n)\})$ and
	$\beta\in\Hom_{\CS^\bullet(Z)}(S\setminus S'',T)$. Let $\theta=\alpha\boxtimes\beta=a(
	\alpha\wedge\beta)$.
	The statement (\ref{eq:eq1}) will follow from the following property:
	\begin{equation}
		\label{eq:eq2}
	a(d(\alpha\otimes\beta))=d(\theta).
	\end{equation}

	\smallskip
	We have 
	$$a(d(\alpha\otimes\beta))=d(\alpha)\boxtimes\beta+\alpha\boxtimes d(\beta)+
	\sum_\zeta a((\id\otimes\beta)\cdot g_{S'',\zeta}(\alpha))$$
	where $\zeta$ runs over positive admissible homotopy classes of paths
	starting in $S''$ and ending in $S\setminus S''$.

	We have 
	$$D(\theta)/\mathrm{inv}=\bigl(D(\alpha)/\mathrm{inv}\bigr)\sqcup
	\bigl(D(\beta)/\mathrm{inv}\bigr)\sqcup 
\coprod_{(s_1,s_2)\in S''\times(S\setminus S'')} I(\alpha_{s_1},\beta_{s_2})\cap D(\theta).$$

	Fix $(s_1,s_2)\in S''\times(S\setminus S'')$. 
	Let $S'=(S''\setminus\{s_1\})\sqcup\{s_2\}$.
	Let $\zeta$ be a smooth path $s_1\to s_2$.
	Let $u'=\id_{S\setminus(S'\sqcup\{s_1\})}\boxtimes\zeta$.
	Write $g_{S'',\zeta}(\alpha)=v\wedge u$ with $u:S\setminus S'\to S\setminus S''$
	and $v:S'\to\{\xi(1),\ldots,\xi(n)\}$. We take $u=0$ and $v=0$ if
	$g_{S'',\zeta}(\alpha)=0$. If $g_{S'',\zeta}(\alpha)\neq 0$, then
	$u=u'$. 
	
	\smallskip
	Assume $(\beta\cdot u)\neq 0$. Then $\alpha_{s_1}\circ\zeta^{-1}$ and
	$\beta_{s_2}\circ\zeta$ are smooth, and $\zeta$ and
	$\bar{\zeta}$ have opposite orientations, since $\bar{\zeta}$ is negative (it
	starts in $\xi(\BZ_{\ge 1})$ and ends in $M$).
	It follows that $\zeta\in L(\theta)$.

	Assume $\zeta\in L(\theta)$. Since $\bar{\zeta}$ is negative, it follows that
	$\zeta$ is positive, then 
$(\theta^\zeta)_s=\theta_s$ for $s{\not\in}\{s_1,s_2\}$, while
	$(\theta^\zeta)_{s_1}=\beta_{s_2}\circ\zeta$ and
	$(\theta^\zeta)_{s_2}=\alpha_{s_1}\circ\zeta^{-1}$.
	We deduce that $(\beta\cdot u)\boxtimes v=\theta^\zeta$ if
	$\beta\cdot u\neq 0$. So, the assertion (\ref{eq:eq2}) is a consequence of the
	following:
	\begin{equation}
		\label{eq:reductionsmooth}
		\text{given }\zeta\in L(\theta)\text{ positive, we have }\beta\cdot u\neq 0
		\text{ if and only if } \zeta\in D(\theta).
	\end{equation}

	\medskip
We will prove that statement by reduction to the non-singular case.
	Let $f:\hat{Z}\to Z$ be a non-singular cover.
	The morphism $\xi:\BR_{>0}\to Z$ lifts uniquely to a morphism of
	curves $\hat{\xi}:\BR_{>0}\to\hat{Z}$. Let $\hat{M}=f^{-1}(M)$ and
	let $\hat{\alpha}:\hat{S}''\to\{\hat{\xi}(1),\ldots\hat{\xi}(n)\}$
	and $\hat{\zeta}$ be the unique lifts of $\alpha$ and
	$\zeta$ to $\hat{Z}$. There exist subsets $\hat{S},\hat{T}$ of
	$\hat{M}$ and a lift
	$\hat{\beta}:\hat{S}\setminus\hat{S}''\to\hat{T}$ of $\beta$ such that
	$\hat{\zeta}\in L(\hat{\theta})$, where
	$\hat{\theta}=\hat{\alpha}\boxtimes\hat{\beta}$ (Lemma \ref{le:functorialityL}).
	We have $\zeta\in D(\theta)$ if and only if 
	$\hat{\zeta}\in D(\hat{\theta})$ (Lemma \ref{le:functorialityL}).

	Write $g_{\hat{S}'',\hat{\zeta}}(\hat{\alpha})=\hat{v}\wedge\hat{u}$
	as above. We have $f(\hat{u})=u$ and $f(\hat{v})=v$.
	We have $g_{S'',\zeta}(\alpha)\neq 0$ if and only if
	$g_{\hat{S}'',\hat{\zeta}}(\hat{\alpha})\neq 0$.  Finally,
	$\beta\cdot u\neq 0$ if and only if $\hat{\beta}\cdot\hat{u}\neq 0$.
	This completes the reduction of (\ref{eq:reductionsmooth}) to the case of $\hat{Z}$.

	\medskip
	So, we now prove (\ref{eq:reductionsmooth}) assuming $Z$ is smooth.
	Note that $Z(\xi)$ is isomorphic (as a $1$-dimensional space) to an interval of $\BR$.
	We consider $\zeta:s_1\to s_2$ in $L(\theta)$ positive with 
	$s_1\in S''$ and $s_2\in S\setminus S''$.

	\smallskip
	Remark \ref{re:productnonexc} shows that $\beta\cdot u'\neq 0$ if and only if
	$i(\beta_s,\beta_{s_2}\circ\zeta)=i(\beta_s,\beta_{s_2})+i(\id_s,\zeta)$ for
	all $s\in S\setminus (S'\sqcup\{s_1\})$. That equality is always satisfied
	unless there are $\zeta'':s_1\to s$ and $\zeta':s\to s_2$ positive.
	In that case, $\bar{\zeta}''$ is negative and
	the equality is satisfied if and only if $\bar{\zeta}'$ is positive.

	\smallskip
	We have $u\neq 0$ if and only if given $\zeta'':s_1\to s$ and $\zeta':s\to s_2$
	positive with $s\in S'\setminus\{s_2\}$, then $\bar{\zeta}''=
	\alpha_{s_2}\circ\zeta''\circ \alpha_{s_1}^{-1}$ is positive.

	We deduce that $\zeta\in D(\theta)$ if and only if $\beta\cdot u\neq 0$.
The proposition follows.
\end{proof}


\begin{example}
	The picture below gives two examples of description of
	the map $g_{S'',\zeta}$.
	$$\includegraphics[scale=0.85]{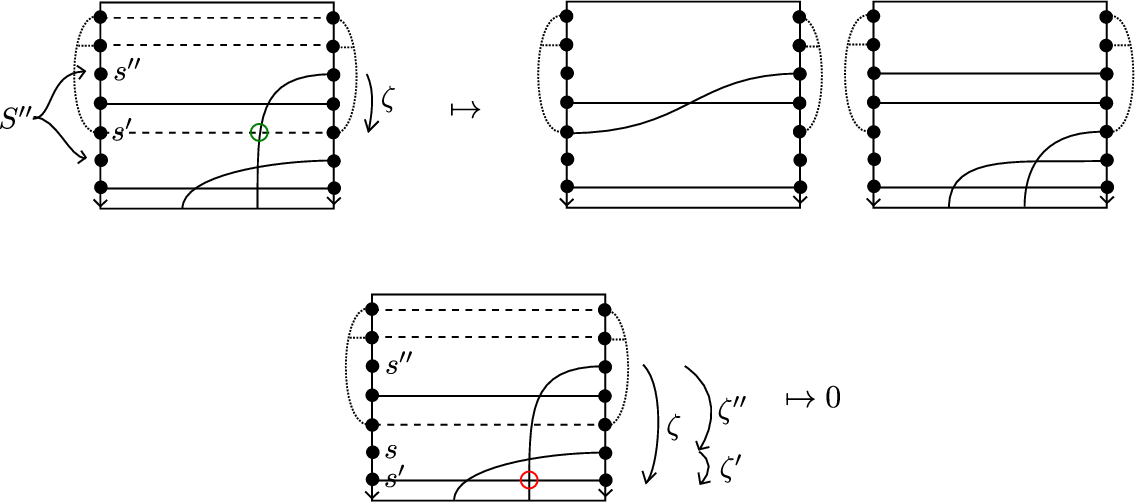}$$
\end{example}

\subsubsection{Right action}
\label{se:rightaction}
Consider now $\xi':\BR_{<0}\to Z$ an injective morphism of curves, where $\BR_{<0}$
is unoriented. Identifying $(\BR_{<0})^\opp$ with $\BR_{>0}$ by $x\mapsto -x$,
we obtain a morphism of curves $\xi:\BR_{>0}\to Z^\opp$. Let
$M$ be a subset of $Z\setminus\xi'(\BR_{\le -1})$.

We say that $\xi'$ is {\em initial}\index[ter]{initial morphism} for $(Z,M)$ if
$\xi$ is terminal for $(Z^\opp,M)$ and that $\xi'$ is {\em incoming}
\index[ter]{incoming morphism} for $Z$ if
$\xi'(\BR_{\le -1})$ is closed in $Z$.

\smallskip
Assume $\xi'$ is initial for $(Z,M)$.
As in the left action case, we define a differential functor

\begin{align*}
	R^\bullet=R_{\xi'}^\bullet:\CC\times\CC^\opp\times\CU&\to k\mdiff\\
	R^\bullet(S,T,e^n)&=\Hom(T\sqcup\{\xi'(-1),\ldots,\xi'(-n)\},S)\\
	R^\bullet(\beta,\alpha,\sigma)(f)&=\beta\cdot f\cdot
	(\alpha\boxtimes\xi'(\sigma^{\mathrm{rev}\opp}))\in R^\bullet(S',T',n)
\end{align*}
for $\alpha\in\Hom_{\CS^\bullet(Z)}(T',T)$, $\beta\in\Hom_{\CS^\bullet(Z)}(S,S')$ and
$\sigma\in\End_{\CU^\bullet}(e^n)$, and $f\in R^\bullet(S,T,n)$.

\smallskip
We put $R_\xi^\bullet(S,T)=R_\xi^\bullet(S,T,e)$ and $R_\xi=\BF_2[R_\xi^\bullet]$.

\smallskip
Recall that the isomorphism (\ref{eq:oppositeStrands}) of differential categories 
$\CS^\bullet_M(Z)\iso \CS^\bullet_M(Z^\opp)^\opp$. This isomorphism provides
an isomorphism $R^\bullet_{\xi'}(S,T,e^n)\iso L^\bullet_{\xi}(T,S,e^n)$ functorial in $S$, $T$ and $e^n$.

In particular, $R^\bullet$ provides a ``right" $2$-representation on $\CS^\bullet_M(Z)$ and
all results of \S\ref{se:defleftaction}--\ref{se:twistedL} have counterparts for $R^\bullet$.

\smallskip
Given $S\subset M$ and $n\ge 0$, there is an isomorphism of functors
$$\bigvee_{\substack{S'\subset S\\ |S'|=n}}
\Hom_{\CS^\bullet(Z)}(\{\xi'(-1),\ldots,\xi'(-n)\},S')\wedge \Hom_{\CS^\bullet(Z)}(-,S\setminus S')
\iso R^\bullet(S,-,e^n).$$

\smallskip
There is an isomorphism of functors, functorial in $S$ and $T$
\begin{align*}
	R^\bullet(T,-,e^n)\wedge R^\bullet(-,S,e^m)&\iso R^\bullet(T,S,e^{n+m})\\
	(\alpha,\beta) &\mapsto
	\alpha\cdot (\beta\boxtimes \xi'([-m-r\to -r]_{1\le r\le n})).
\end{align*}

\medskip
Assume there is a decreasing sequence $m_0,m_{-1},\ldots$ of points of
$\xi^{\prime -1}(M)$ with $\lim_i m_i<t$ for all $t\in\xi^{\prime -1}(M)$.

We obtain as in (\ref{eq:Lascolim}) isomorphisms functorial in $S$ and $T$
\begin{equation}
	\label{eq:Rascolim}
	\colim_{r\to\infty}\Hom_{\CS^\bullet_M(Z)}
	(T\sqcup \xi'(\{m_{-r},\ldots,m_{-r-n+1}\}),S)\iso R^\bullet(S,T,e^n).
\end{equation}

\smallskip
Let us finally consider functoriality as in 
\S\ref{se:2repmorphismscurves}. Let $f:Z\to Z'$ be a morphism of curves and
assume $f\circ \xi'$ is initial for $(Z',f(M))$.

The functor $f:\CS_{f,M}^\bullet(Z)\to\CS^\bullet_{f(M)}(Z')$ induces a morphism of bimodule
$2$-representations $R_{f\circ\xi'}^\bullet\to R_{\xi'}^\bullet$, when
$|f^{-1}(f(z))|=1$ for all $z\in M$.

If $f$ is strict, then 
the functor $f^\#:\add(\CS_{f(M)}(Z'))\to\add(\CS_M(Z))$ induces a morphism of bimodule
$2$-representations $R_{\xi'}\to R_{f\circ\xi'}$.

\begin{rem}
	As in Remark \ref{re:leftasbottomalgebramodule}, we recover
	the construction of ``top algebra module" of Douglas and
	Manolescu by taking the underlying lax $2$-representation of
	$R_{\xi'}$.
\end{rem}

\begin{example}
	As in Example \ref{ex:leftaction}, we use an
	alternative graphical description for $R_{\xi'}^\bullet$. This is
	illustrated in the example of $R_{\xi'}^\bullet(-,-,e^2)$ below.
$$\includegraphics[scale=0.95]{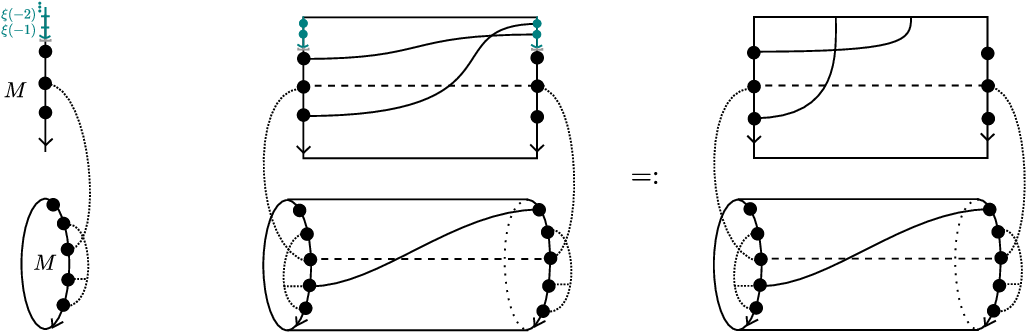}$$
\end{example}

\subsubsection{Duality}
	\label{se:duality}
	Let $Z'=\BR$ be the smooth curve with $Z'_o=(-\frac{1}{2},\frac{1}{2})$, with its
	standard orientation.
Consider a morphism of curves $\tilde{\xi}:Z'\to Z$ such that
$\tilde{\xi}(Z')$ is a component of $Z$.

	Fix an increasing homeomorphism $\alpha:\BR_{>0}\iso\BR_{>\frac{1}{2}}$ fixing
	the positive integers and define $\alpha':\BR_{<0}\iso\BR_{<-\frac{1}{2}}$ by
	$\alpha'(t)=-\alpha(-t)$.
	Let $\xi^+=\tilde{\xi}\circ\alpha:\BR_{>0}\to Z$ and
	$\xi^-=\tilde{\xi}\circ\alpha':\BR_{<0}\to Z$. These are injective
morphisms of curves, $\xi^+$ is outgoing for $Z$ and $\xi^-$ is incoming for $Z$.

\smallskip
Given $n\ge 0$,
we denote by $\theta(n)\in\Hom_{\CS^\bullet(Z')}(\{-n,\ldots,-1\},\{1,\ldots,n\})$
the braid given by $\theta(n)_{-i}=[-i\to i]$.

\smallskip
Let $T$ and $T'$ two finite subsets of $Z$ and $I\subset\BZ_{\ge 1}$ finite.
Assume that $\tilde{\xi}(-I)\subset T$ and that
given $x\in\BR$ with $x<i$ for all $i\in -I$, we have
$\tilde{\xi}(x){\not\in}T$.
Assume also that $\tilde{\xi}(I)\subset T'$ and that
given $x\in\BR$ with $x>i$ for all $i\in I$, we have
$\tilde{\xi}(x){\not\in}T'$.

We consider the pointed map

\begin{align*}
	\kappa_I:\Hom_{\CS^\bullet(Z)}(T,T')&\to 
	\Hom_{\CS^\bullet(Z)}(T\setminus (T\cap\tilde{\xi}(-I)),
	T'\setminus (T'\cap \tilde{\xi}(I))) \\
	\theta&\mapsto\begin{cases}
		(\theta_t)_{t\in T\setminus \tilde{\xi}(-I)} &
		\text{ if }
\bij{\theta}(\tilde{\xi}(-i))=\tilde{\xi}(i) \text{ for }i\in I\\
	0 & \text{ otherwise.}
\end{cases}
\end{align*}
We put $\kappa_n=\kappa_{\{1,\ldots,n\}}$.
Note that $\kappa_n=\kappa_{\{n\}}\circ\cdots\circ\kappa_{\{2\}}\circ\kappa_{\{1\}}$.

	\smallskip
Let  $f:Z\to \bar{Z}$ be a morphism of curves such that
	$f\circ\tilde{\xi}$ is a homeomorphism from $Z'$ to
	a component of $\bar{Z}$.
		Put $\tilde{\bar{\xi}}=f\circ\tilde{\xi}$.
	 	Denote by $\bar{\kappa}_n$ the map defined
	as above with $Z$ replaced by $\bar{Z}$.

		Let $T$ and $T'$ be two finite subsets of
	$Z$ such that $|f(T)|=|T|$ and $|f(T')|=|T'|$. Put
	$\mathring{T}=T\setminus (T\cap\tilde{\xi}(\{-n,\ldots,-1\})$ and
	$\mathring{T}'=T'\setminus (T'\cap\tilde{\xi}(\{-n,\ldots,-1\})$.
	There is a commutative diagram
	\begin{equation}
		\label{eq:fkappa}
		\xymatrix{\Hom_{\CS^\bullet(Z)}(T,T')\ar[r]^-{\kappa_n} \ar[d]_f &
		\Hom_{\CS^\bullet(Z)}(\mathring{T},\mathring{T}') \ar[d]^f \\
	\Hom_{\CS^\bullet(\bar{Z})}(f(T),f(T'))\ar[r]_-{\kappa_n} &
		\Hom_{\CS^\bullet(\bar{Z})}(f(\mathring{T}),f(\mathring{T}'))
	}
	\end{equation}

	Similarly, if $f$ is strict and $U$ and $U'$ are two finite subsets of $\bar{Z}$,
	there is a commutative diagram

	\begin{equation}
		\label{eq:fsharpkappa}
		\xymatrix{\Hom_{\CS(\bar{Z})}(U,U')\ar[r]^-{\kappa_n} \ar[d]_{f^\#} &
	\Hom_{\CS(\bar{Z})}(U\setminus (U\cap\tilde{\bar{\xi}}(\{-n,\ldots,-1\})),U'\setminus
	(U'\cap\tilde{\bar{\xi}}(\{1,\ldots,n\})) \ar[d]^{f^\#} \\
	\bigoplus_{T,T'}\Hom_{\CS(Z)}(T,T')\ar[r]_-{\kappa_n} &\bigoplus_{T,T'}
		\Hom_{\CS(Z)}(\mathring{T},\mathring{T}')
	}
	\end{equation}

	where $T$ (resp. $T'$) runs over finite subsets of $Z$ such that $f(T)=U$
	(resp. $f(U')=T'$).

\begin{lemma}
	\label{le:kappadiff}
	The map $\kappa_n$ commutes with differentials.
\end{lemma}

\begin{proof}
	Assume first $\tilde{\xi}$ is a homeomorphism and $Z_o=\emptyset$.
	Let $T$ and $T'$ be two finite subsets of $\BR$ with same cardinality $m$. Let
	$a:\{1,\ldots,m\}\iso T$ and $a':\{1,\ldots,m\}\iso T'$ be the increasing
	bijections. There
	is an isomorphism of differential modules
	(Proposition \ref{pr:strandsnonsingular})
	$\phi:\Hom_{\CS(Z)}(T,T')\iso H_m$:
	given $\theta\in\Hom_{\CS^\bullet(Z)}(T,T')$ non-zero
	and given $i\in\{1,\ldots,m\}$,
	we put $\phi(\theta)(i)=a^{\prime -1}(\theta_{a(i)}(1))$.

	Assume in addition that $\{-n,\ldots,-1\}\subset T$ and
	$T\setminus\{-n,\ldots,-1\}\subset (-1,\infty)$
	and $\{1,\ldots,n\}\subset T'$ and $T'\setminus\{1,\ldots,n\}
	\subset (-\infty,1)$.
There is a commutative diagram
	$$\xymatrix{\Hom_{\CS(Z)}(T,T')\ar[r]^-{\kappa_n} \ar[d]_{\phi}^\sim &
	\Hom_{\CS(Z)}(T\setminus\{-n,\ldots,-1\},T'\setminus\{1,\ldots,n\}) \ar[d]^\phi_\sim \\
	H_m\ar[r]_-{t^-_{m,m-n}} & H_{m-n}
	}$$
	The lemma follows now from \S\ref{se:Lrn}.

	\smallskip
	Assume now $Z$ is smooth.
	If $Z(\xi^+)$ is unoriented, then the lemma holds by the discussion above, using
	\S\ref{se:tensorstrands}. In general, we consider the morphism
	of curves $f:Z\to \bar{Z}$ that is an isomorphism outside $Z(\xi^+)$ and
	the identity on $Z(\xi^+)$, with $f(Z(\xi^+))_o=\emptyset$. The vertical maps
	of the commutative diagram (\ref{eq:fkappa}) are injective, hence the lemma
	holds for $Z$ since it holds for $\bar{Z}$.

	\smallskip
	Consider now a general $Z$. Let $f:\hat{Z}\to Z$ be a non-singular cover. 
	The vertical maps
	of the commutative diagram (\ref{eq:fsharpkappa}) are injective, hence the lemma
	holds for $Z$ since it holds for $\hat{Z}$.
\end{proof}

Let $M$ be a subset of $Z\setminus
\tilde{\xi}\bigl((-\infty,-1]\cup[1,\infty)\bigr)$.

\smallskip
Given $S$ a finite subset of $M$, the pointed map
$$L_{\xi^+}^\bullet(T,S,e^n)\wedge R_{\xi^-}^\bullet(S,T',e^n)\to \Hom_{\CS^\bullet(Z)}(T',T),\
(\theta',\theta)\mapsto \kappa_n(\theta'\cdot\theta)$$
induces an $\BF_2$-linear map
\begin{align*}
	\hat{\kappa}(T,S):L_{\xi^+}(T,S,e^n)&\to \Hom_{\CS(Z)^\opp\mdiff}(R_{\xi^-}(S,-,e^n),
	\Hom(-,T))\\
	\theta'&\mapsto \bigl( (\theta\in R_{\xi^-}^\bullet(S,T',e^n))\mapsto \kappa_n(\theta'\cdot\theta)\bigr).
\end{align*}

\begin{prop}
	\label{pr:Ldual}
	The map $\hat{\kappa}$  induces an isomorphism of differential pointed
	bimodules $L_{\xi^+}(-_2,-_1,e^n)\iso R_{\xi^-}(-_1,-_2,e^n)^\vee$.
\end{prop}

\begin{proof}
	Lemma \ref{le:kappadiff} shows that $\hat{\kappa}$ commutes with differentials.

	      Let $S$ be a finite subset of $M$ of cardinality $n$. 

	      Assume
	      $\tilde{\xi}$ is a homeomorphism and $Z_o=\emptyset$. There is
	      a commutative diagram (see the proof of Lemma \ref{le:kappadiff} with
	      $(T,T')=(S,\xi^+(\{1,\ldots,n\})$ and $(T,T')=(\xi^-(\{-n,\ldots,-1\}),S)$)
	      $$\xymatrix{
		      \Hom_{\CS(Z)}(S,\xi^+(\{1,\ldots,n\})\ar[r]^-{\hat{\kappa}}\ar[d]_\phi^\sim &
		      \Hom_{\CS(Z)}(\xi^-(\{-n,\ldots,-1\},S)^*\ar[d]_\sim^{(\phi^*)^{-1}} \\
		      H_n\ar[r]_{\hat{t}^-_{S,\emptyset}} & H_n^*
		      }$$
The bottom horizontal map is bijective by Corollary \ref{co:dual}, hence
	$\hat{\kappa}(\emptyset,S)$ is bijective.

	Assume now $Z(\xi^+)$ is smooth and unoriented. The map
	$\hat{\kappa}(\emptyset,S)$ is the same for $Z$ and for $Z(\xi^+)$, so $\hat{\kappa}(\emptyset,S)$
	is still bijective.

	Assume $Z(\xi^+)$ is smooth. There is a morphism of curves $f:Z\to \bar{Z}$ that
	is an isomorphism outside $Z(\xi^+)$ and the identity on $Z(\xi^+)$ with 
	$f(Z(\xi^+))_o=\emptyset$. The map
	$\hat{\kappa}(\emptyset,S)$ is the same for $Z$ and for $\bar{Z}$, so $\hat{\kappa}(\emptyset,S)$
	is still bijective.

	Consider now a general $Z$ and let $f:\hat{Z}\to Z$ be a non-singular cover. Let
	$\tilde{\hat{\xi}}:Z'\to \hat{Z}$ be the morphism of curves such that 
	$\tilde{\xi}=f\circ\tilde{\hat{\xi}}$.
	The functors $f$ and $f^\#$ are inverse bijections between
	$\Hom_{\CS(Z)}(S,\xi^+(\{1,\ldots,n\})$ and
	$\bigoplus_{S'}\Hom_{\CS(\hat{Z})}(S',\tilde{\hat{\xi}}(\{1,\ldots,n\})$
	(resp. $\Hom_{\CS(Z)}(\xi^-(\{-n,\ldots,-1\},S)$ and 
	$\bigoplus_{S'}\Hom_{\CS(\hat{Z})}(\tilde{\hat{\xi}}(\{-n,\ldots,-1\},S')$),
	where $S'$ runs over $n$-elements subsets of $\hat{Z}$ such that $f(S')=S$.
	Furthermore,
	$\hat{\kappa}(\emptyset,S)$ is compatible with these bijections (see the proof of
	Lemma \ref{le:kappadiff}). It follows that $\hat{\kappa}(\emptyset,S)$ is bijective.

	\smallskip
	We consider now two arbitrary subsets $S$ and $T$ of $M$.	      
	      The canonical isomorphisms of Lemma \ref{le:decompositionnodiff} and of \S
        \ref{se:rightaction} fit in a commutative diagram of $\BF_2$-modules 
        $$\xymatrix{
		\bigoplus_{S'} L_{\xi^+}(\emptyset,S',e^n)\otimes
		\Hom_{\CS(Z)}(S\setminus S',T) \ar[r]^-\sim 
		\ar[d]_{\sum_{S'}\hat{\kappa}(\emptyset,S')\otimes
		\id}&
		 L_{\xi^+}(T,S,e^n) \ar[d]^{\hat{\kappa}(T,S)}\\
		\bigoplus_{S'}R_{\xi^-}(S',\emptyset,e^n)^*\otimes\Hom_{\CS(Z)}(S\setminus S',T) \ar[r]_-\sim  &
                \Hom(R_{\xi^-}(S,-,e^n),\Hom(-,T))
                }$$
		where $S'$ runs over $n$ elements subsets of $S$.
The discussion above shows that the left vertical arrow is an isomorphism, hence
	$\hat{\kappa}(T,S)$ is an isomorphism.
\end{proof}

Given $x_1,x_2\in [-1,1]$, the homotopy class
$\tilde{\xi}([x_1\to x_2])$ is admissible if
$x_1\le x_2$ or $x_1\le -\frac{1}{2}$ or $x_2\ge \frac{1}{2}$.
Given $x\in [-1,1]$ and $\zeta$ an admissible class of paths in $Z$ with $\zeta(1)=
\tilde{\xi}(x)$ and $\tilde{\xi}([x\to 1])\cdot\zeta\neq 0$, there is a unique $y\in [-1,1]$ such
that $\zeta=\tilde{\xi}([y\to x])$.

\smallskip
Let us describe now the unit of the adjunction when $n=1$.
\begin{lemma}
	\label{le:unitLR}
	The unit of the adjunction $(L_{\xi^+}(-,-)\otimes-,
	R_{\xi^-}(-,-)\otimes-)$ is given by the morphism of bimodules
	whose evaluation at $(T,S)$ is
	\begin{align*}
		\Hom_{\CS(Z)}(S,T)&\to R_{\xi^-}(T,-)\otimes L_{\xi^+}(-,S)\\
		\gamma &\mapsto \sum_{x\in \tilde{\xi}^{-1}(S)}
		(\gamma_{|S\setminus\{\tilde{\xi}(x)\}}\boxtimes \tilde{\xi}([-1\to x])) \otimes
		(\id_{S\setminus\{\tilde{\xi}(x)\}}\boxtimes \tilde{\xi}([x\to 1])).
	\end{align*}
\end{lemma}

\begin{proof}
	The counit of the adjunction is $\eps=\kappa_1\circ\mathrm{ mult}$.
	Let $\gamma\in R_{\xi^-}^\bullet(T,S)$. Let $\eta$ be the map defined in the lemma.
	We have
$$\eta(\id_T)=\sum_{x\in\tilde{\xi}^{-1}(T)}
	(\tilde{\xi}([-1\to x])\boxtimes\id_{T\setminus\{\tilde{\xi}(x)\}})
	\otimes (\tilde{\xi}([x\to 1])\boxtimes\id_{T\setminus\{\tilde{\xi}(x)\}}),$$
	hence
		$$(\id\otimes\eps)\circ(\eta\otimes\id)(\gamma)=
\sum_{x\in\tilde{\xi}^{-1}(T)}
	(\tilde{\xi}([-1\to x])\boxtimes\id_{T\setminus\{\tilde{\xi}(x)\}})\cdot
	\kappa_1\bigl((\tilde{\xi}([x\to 1])\boxtimes
	\id_{T\setminus\{\tilde{\xi}(x)\}})\cdot\gamma\bigr)$$
Let $x$ be the unique element of 
	$\tilde{\xi}^{-1}(\bij{\gamma}(\xi^-(-1)))$. We have
	$\gamma_{\xi^-(-1)}= \tilde{\xi}([-1\to x])$ and
	$\kappa_1\bigl((\tilde{\xi}([x\to 1])\boxtimes
	\id_{T\setminus\{\tilde{\xi}(x)\}})\cdot\gamma\bigr)=\gamma_{|S}$, hence
		$$(\id\otimes\eps)\circ(\eta\otimes\id)(\gamma)=
(\gamma_{\xi^-(-1)}\boxtimes\id_{T\setminus\{\bij{\gamma}(\xi^-(-1))\}})\otimes
		\gamma_{|S}$$
We deduce that
		$$\mathrm{mult}\circ(\id\otimes\eps)\circ(\eta\otimes\id)(\gamma)=\gamma$$
	and the lemma follows.
\end{proof}

\begin{rem}
There is a bifunctorial injective map
	$$R_{\xi^-}^\bullet(T,-)\wedge L_{\xi^+}^\bullet(-,S)\to \Hom(S\sqcup\{\xi^-(-1)\},
	T\sqcup\{\xi^+(1)\}),\ \beta\wedge\alpha\mapsto (\beta\boxtimes\id_{\xi^+(1)})\cdot
	(\alpha\boxtimes\id_{\xi^-(-1)}).$$
	The composition of the unit given by Lemma \ref{le:unitLR} with this map is the
	following map
	$$\Hom(S,T)\to \Hom(S\sqcup\{\xi^-(-1)\}, T\sqcup\{\xi^+(1)\}),\ 
	\gamma\mapsto (\gamma\otimes\id_{\xi^+(1)})\cdot d(\tilde{\xi}([-1\to 1])\boxtimes\id_S).$$
\end{rem}

\begin{example}
	The first picture below provides an example of description of
	the unit of the adjunction as in Lemma \ref{le:unitLR}.
$$\includegraphics[scale=0.85]{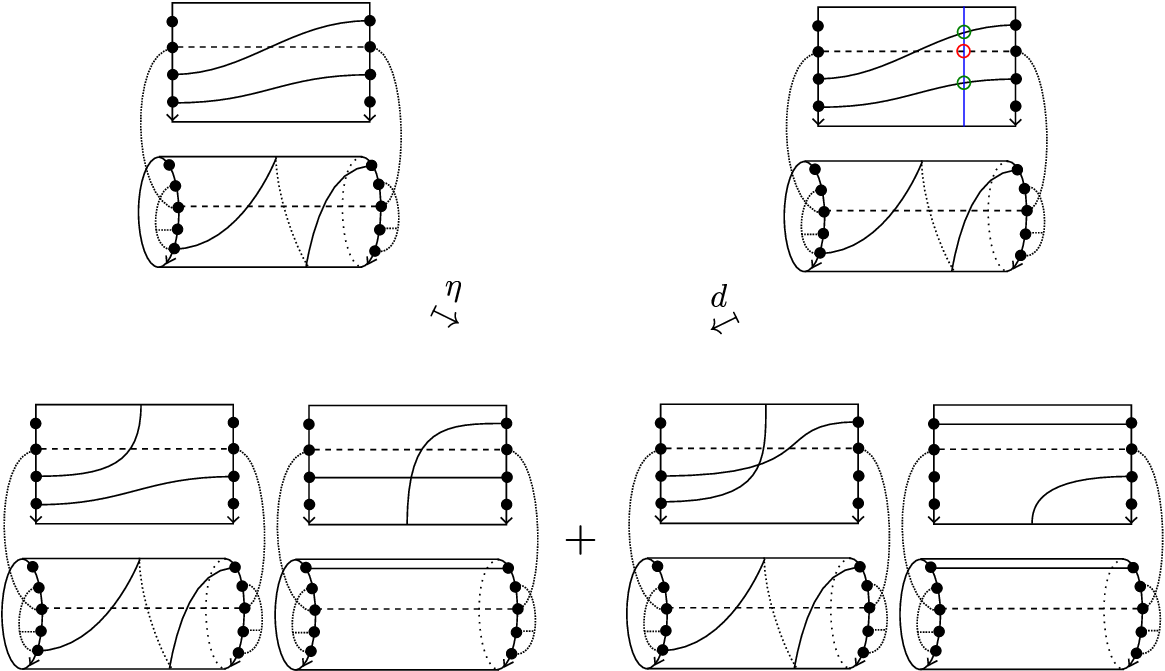}$$
	The second picture
	describes a calculation of an image by the counit.
$$\includegraphics[scale=0.85]{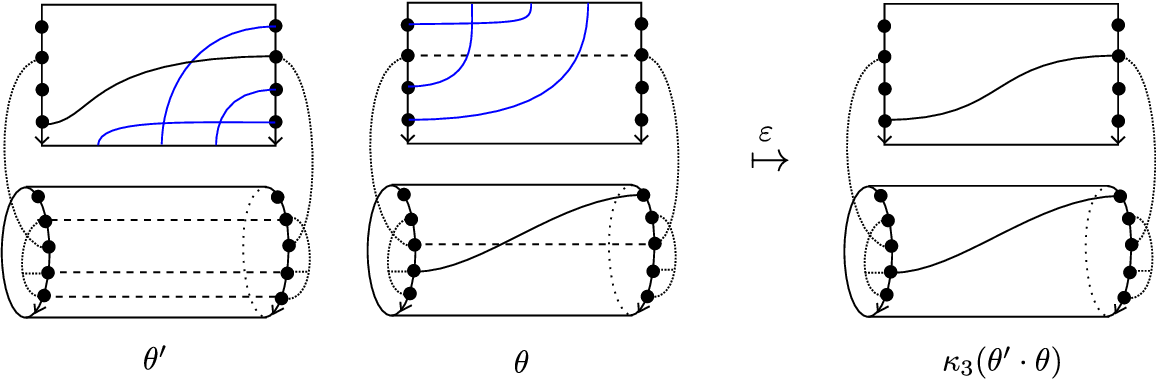}$$
\end{example}

\subsubsection{Actions for the line}
\label{se:actionsline}
We consider the unoriented curve $\BR$. Let $M=\{\pm(1-\frac{1}{n})\}_{n\in\BZ_{>0}}$.

Consider $S,T$ two finite subsets of $\BR$ with $|S|=|T|=n$. Let $f_S:S\iso\{1,\ldots,n\}$ and
$f_T:T\iso\{1,\ldots,n\}$ be the unique increasing bijections.
We define 
$$\phi(S,T):\Hom_{\CS^\bullet(\BR)}(S,T)\iso H_n^\bullet=\End_{\CU^\bullet}(e^n),\
\theta\mapsto T_{f_T\circ\bij{\theta}\circ f_S^{-1}}.$$

\smallskip
We define a functor $\Phi:\CS^\bullet_M(\BR)\to\CU^\bullet$. We put $\Phi(S)=e^{|S|}$ and
$\Phi(f)=\phi(S,T)(f)$ for $f\in\Hom_{\CS^\bullet(\BR)}(S,T)$.

\smallskip
The next proposition follows from Proposition \ref{pr:strandsnonsingular}.

\begin{prop}
The functor $\Phi:\CS^\bullet_M(\BR)\to \CU^\bullet$ is an equivalence of differential pointed categories.
\end{prop}

 Consider  $\xi_+:\BR_{>0}\to\BR$ and $\xi_-:\BR_{<0}\to\BR$ the inclusion maps.

We define $\varphi_\pm:L_{\xi_\pm}(-,-,e^n)\iso L^\pm(-,-,e^n)\circ(\Phi\wedge \Phi)$ by 
$$\varphi_\pm(T,S)=\phi(S,T\sqcup\xi_\pm(\{\pm1,\ldots,\pm n\})):
\Hom_{\CS^\bullet(\BR)}(S,T\sqcup\xi_\pm(\{\pm 1,\ldots,\pm n\}))\iso
L^\pm(e^{|T|},e^{|S|},n).$$

Similarly, 
we define $\varphi'_\pm:R_{\xi_\pm}(-,-,e^n)\iso R^\pm(-,-,e^n)\circ(\Phi\wedge \Phi)$ by 
$$\varphi'_\pm(T,S)=\phi(T\sqcup\xi_\pm(\{\pm1,\ldots,\pm n\}),S):
\Hom_{\CS^\bullet(\BR)}(T\sqcup\xi_\pm(\{\pm 1,\ldots,\pm n\}),S)\iso
R^\pm(e^{|S|},e^{|T|},n).$$

\begin{prop}
	\label{pr:2repline}
	Together with $\varphi_\pm$ (resp. $\varphi_\pm'$),
the functor $\Phi$ induces equivalences of bimodule
	$2$-representations between $L_{\xi_\pm}$ and $L^\pm$ (resp.
	$R_{\xi_\pm}$ and $R^\pm$).
\end{prop}

\subsubsection{Action as functors}
\label{se:actionfunctors}
We explain here how the $2$-representation constructed in \S\ref{se:defleftaction} can
be described using functors between strand categories of different curves.

\smallskip
Let $Z$ be a singular curve and\
$\xi:\BR_{>0}\to Z$ an injective morphism of curves with
$\xi(\BR_{\ge 1})$ closed and contained in $Z\setminus Z_{exc}$.

Let $A=\bigoplus_{I,J\subset Z_{exc}}\Hom_{\CA(Z)^{\opp}}(J,I)$. We denote by
$e_I\in A$ the idempotent corresponding to the projection on $I$, so that
$e_IAe_J=\Hom_{\CA(Z)}(I,J)$.

The 
equivalence $A\mdiff\iso \CA(Z)^{\opp}\mdiff$ restricts to an equivalence
$(\bar{A})^i\iso \bar{\CA}^i(Z)$ (cf \S \ref{se:Algebras}).

\smallskip
We consider a new
singular curve $\hat{Z}=Z\sqcup_{\xi(1)}(-1,1)$ obtained as the quotient of the disjoint union of $Z$ 
and the oriented interval $(-1,1)$ identifying $\xi(1)$ with $0$.
Note that $\hat{Z}_{exc}=Z_{exc}\cup\{\xi(1)\}$.

We put  $\hat{A}=\bigoplus_{I,J\subset \hat{Z}_{exc}}\Hom_{\CA(\hat{Z})^{\opp}}(J,I)$.
As before, we have idempotents $e_I\in\hat{A}$ for $I\in\hat{Z}_{exc}$. We put
$e=\sum_{I\subset Z_{exc}}e_{I\sqcup\{\xi(1)\}}$.

\smallskip
The inclusion $i:Z\hookrightarrow \hat{Z}$ provides a fully faithful functor 
$\Xi:\bar{\CA}^i(Z)\to \bar{\CA}^i(\hat{Z})$. This gives rise to an isomorphism
of algebras $h:A\iso (1-e)\hat{A}(1-e)$ and we have a commutative diagram
$$\xymatrix{
	\bar{\CA}^i(Z)\ar[rr]^{\Xi} && \bar{\CA}^i(\hat{Z}) \\
(\bar{A})^i \ar[u]_{\sim}^{\can}\ar[rr]_{\hat{A}(1-e)\otimes_A-} && (\bar{\hat{A}})^i 
\ar[u]^{\sim}_{\can}
}$$
where $A$ acts on the right on $\hat{A}(1-e)$ by right multiplication preceded by $h$.

\smallskip
The inclusion $(1-e)\hat{A}(1-e)\hookrightarrow\hat{A}$ induces a surjective
morphism of algebras $g:(1-e)\hat{A}(1-e)\twoheadrightarrow \hat{A}/\hat{A}e\hat{A}$.
We have $e\hat{A}(1-e)=0$, hence $\hat{A}e\hat{A}\cap (1-e)\hat{A}(1-e)=0$.
It follows that $g$ is an isomorphism.

\smallskip
The right adjoint to $\hat{A}(1-e)\otimes_A-$ is 
$\Hom_{\hat{A}}(\hat{A}(1-e),-)$, which is canonically isomorphic to
$(1-e)\hat{A}\otimes_{\hat{A}}-$ and  we have a commutative diagram
$$\xymatrix{
	\bar{\CA}^i(\hat{Z})\ar[rr]^{\Gamma} && \bar{\CA}^i(Z) \\
(\bar{\hat{A}})^i \ar[u]_{\sim}^{\can}\ar[rr]_{(1-e)\hat{A}\otimes_{\hat{A}}-} && (\bar{A})^i 
\ar[u]^{\sim}_{\can}
}$$
where $\Gamma:\bar{\CA}^i(\hat{Z})\to\bar{\CA}^i(Z)$ is
the right adjoint of $\Xi$.

\begin{rem}
	There is a sequence of four adjoint functors between $A$-modules and
	$\hat{A}$-modules:
	$$\bigl(A\otimes_{\hat{A}}-,\hat{A}(1-e)\otimes_A -,(1-e)\hat{A}\otimes_{\hat{A}}-,
	\Hom_A((1-e)\hat{A},-)\bigr).$$
The first and fourth functors are not exact in general.
	Here, 
\begin{itemize}
	\item $\hat{A}$ acts on the right on $A$ by right multiplication
	preceded by the composition
	$$\hat{A}\xrightarrow{\can}\hat{A}/\hat{A}e\hat{A}\xrightarrow[\sim]{g^{-1}}
		(1-e)\hat{A}(1-e)\xrightarrow[\sim]{h^{-1}}A$$
	\item $\bigl(\hat{A}(1-e)\otimes_A -\bigr)=\bigl((1-e)\hat{A}(1-e)\otimes_A-
		\bigr)		\xrightarrow[\sim]{h^{-1}}\bigl(A\otimes_A-\bigr)=\Hom_A(A,-)$
	\item $\bigl((1-e)\hat{A}\otimes_{\hat{A}}-\bigr)\xrightarrow[\sim]{\can}
		\bigl(\Hom_{\hat{A}}(\hat{A}(1-e),\hat{A})\otimes_{\hat{A}}-\bigr)
		\xrightarrow[\sim]{\can}\Hom_{\hat{A}}(\hat{A}(1-e),-)$.
	\end{itemize}
\end{rem}

There is also a fully faithful functor 
$$\Upsilon:\bar{\CA}^i(Z)\to \bar{\CA}^i(\hat{Z}),\ T\mapsto T\sqcup\{\xi(1)\}$$
sending a braid $(\theta_t)_{t\in T}$ to $(\theta_t)_{t\in T}\sqcup (\mathrm{id}_{\xi(1)\}})$. It gives rise to an isomorphism of
algebras $u:A\iso e\hat{A}e$ and there is  a commutative diagram
$$\xymatrix{
	\bar{\CA}^i(Z)\ar[rr]^{\Upsilon} && \bar{\CA}^i(\hat{Z}) \\
(\bar{A})^i \ar[u]_{\sim}^{\can}\ar[rr]_{\hat{A}e\otimes_A-} && (\bar{\hat{A}})^i 
\ar[u]^{\sim}_{\can}
}$$
where the right action of $A$ on $\hat{A}e$ is by right multiplication preceded by $u$.

\smallskip
We have 
$$L(T,S)=\Hom_{\CA(\hat{Z})}(\Xi(S),\Upsilon(T))\iso
\mathrm{Hom}_{\CA(Z)}(S,\Gamma\Upsilon(T)).$$ 

 Denote by $E=L\otimes_{\bar{\CA}^i(Z)} -$ the
endofunctor of $\bar{\CA}^i(Z)$ induced by the bimodule $L$. The 
isomorphism above gives rise to an isomorphism
of functors $E\xrightarrow{\sim} \Gamma\Upsilon$.

\medskip
We put $\hat{Z}_0=Z$ and we
define inductively $\hat{Z}_r=\hat{Z}_{r-1}\sqcup_{\xi(r)}(-1,1)$ for $r\ge 1$,
where $\xi(r)$ is identified with $0$.

We denote by $\Xi_r:\bar{\CA}^i(\hat{Z}_{r-1})\to \bar{\CA}^i(\hat{Z}_r),\ I\mapsto I$ the functor associated
with the inclusion $\hat{Z}_{r-1}\hookrightarrow\hat{Z}_r$, defined as $\Xi$ above.

We denote by $\Upsilon_r:\bar{\CA}^i(\hat{Z}_r)\to \bar{\CA}^i(\hat{Z}_r),\ T\mapsto T\sqcup\{\xi(r)\}$
the functor $\Upsilon$ for $Z$ replaced by $\hat{Z}_{r-1}$.

Composition with $\id_T\sqcup\{[\xi(1)\to\xi(r)]\}$ gives an isomorphism
$$L(T,S)\iso \Hom_{\CA(\hat{Z}_r)}(\Xi_r\cdots\Xi_1(S),\Upsilon_r\Xi_{r-1}\cdots
\Xi_1(T))=\Hom_{\CA(\hat{Z}_r)}(S,T\sqcup\{\xi(r)\})$$
for $r\ge 1$. Similarly, we have an isomorphism
$$f=(\id_T\sqcup\{[\xi(1)\to\xi(2)],[\xi(2)\to\xi(3)]\})\circ -$$
$$L(T,S,2) \iso \Hom_{\CA(\hat{Z}_3)}(\Xi_3\Xi_2\Xi_1(S),\Upsilon_3\Upsilon_2\Xi_1(T)
)=\Hom_{\CA(\hat{Z}_3)}(S,T\sqcup\{\xi(2),\xi(3)\}).$$
We consider the morphism
$$g=(\id_{T\sqcup\{\xi(2)\}}\sqcup\{[\xi(1)\to\xi(3)]\})\circ -$$
$$L(T,S,2) \to \Hom_{\CA(\hat{Z}_3)}(\Xi_3\Xi_2\Xi_1(S),\Upsilon_3\Upsilon_2\Xi_1(T)
)=\Hom_{\CA(\hat{Z}_3)}(S,T\sqcup\{\xi(2),\xi(3)\}).$$

The composition $f^{-1}\circ g$ is the endomorphism $\tau$ of $L(T,S,2)$.

\begin{center}
	\includegraphics[scale=1.0]{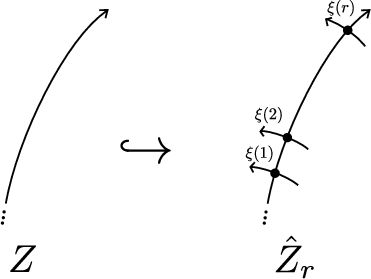}
\end{center}
\subsubsection{Action on Fukaya categories}
\label{se:actionFukaya}

Assume now $Z$ is as in \S \ref{se:Fukayastrands}, so that we have an associated pair
$(F,S)$. We sketch a construction of the $2$-representation on Fukaya
categories of symmetric powers of $F$ via Auroux's equivalences (\S \ref{se:Fukayastrands}).
A rigorous construction would require
a general theory of partially wrapped Fukaya categories and Lagrangian correspondences.

\smallskip
The surface associated with the singular curve $\hat{Z}_r$ of \S\ref{se:actionfunctors}
can be 
 identified with $F$, with set of stops $\hat{S}_r$ obtained from $S$ by adding points
 $z_1,\ldots,z_r$. We have $(\hat{Z}_r)_{exc}=Z_{exc}\sqcup\{z_1,\ldots,z_r\}$ and
 we put $\omega_i=\omega_{z_i}$. We denote by $z_0$ the point of $Z_{exc}\cap\partial F$
 such that the interval $(z,z_1)$ of $\partial F$ contains no point of $Z_{exc}$.

\begin{center}
	\includegraphics[scale=1.0]{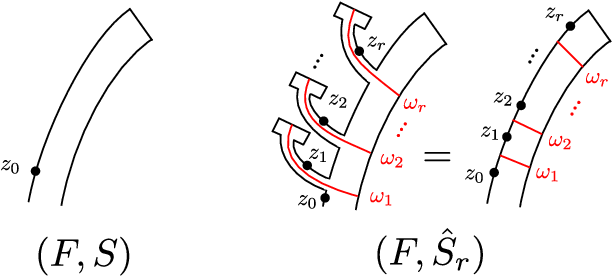}
\end{center}


Consider a positive integer $r$.
We have a fully faithful functor
$\Xi'_r:\CF(\mathrm{Sym}^nF,\hat{S}_{r-1})\to
\CF(\mathrm{Sym}^nF,\hat{S}_r)$ obtained by moving
endpoints of Lagrangians so that they are not on the interval $[z_{r-1},z_r]$
of $\partial F$.
There is a commutative diagram where
 the vertical functors are Auroux's functors:
 $$\xymatrix{
	 \CA(\hat{Z}_{r-1},n)\ar[r]^-{\Xi_r}\ar[d]_\Phi &
	 \CA(\hat{Z}_r,n)\ar[d]^\Phi \\
\CF(\mathrm{Sym}^nF,\hat{S}_{r-1})\ar[r]_-{\Xi'_r} &\CF(\mathrm{Sym}^{n}F,\hat{S}_r)
}$$

\smallskip
The Lagrangian correspondence
$$\bigl\{(\{x_1,\ldots,x_n\},\{x_1,\ldots,x_n,y\})\ |\ x_1,\ldots,x_n\in F,\ 
y\in\omega_r\bigr\}
\subset -\mathrm{Sym}^nF\times \mathrm{Sym}^{n+1}F$$
induces a functor 
$$\Upsilon'_r:\bar{\CF}^i(\mathrm{Sym}^nF,\hat{S}_{r-1})\to\bar{\CF}^i(\mathrm{Sym}^{n+1}F,\hat{S}_r),\ L\mapsto L\sqcup \omega_r$$
and there is a commutative diagram
 $$\xymatrix{
	 \CA(\hat{Z}_{r-1},n)\ar[r]^-{\Upsilon_r}\ar[d]_\Phi &
	 \CA(\hat{Z}_r,n+1)\ar[d]^\Phi \\
\bar{\CF}^i(\mathrm{Sym}^nF,\hat{S}_{r-1})\ar[r]_-{\Upsilon'_r} &\bar{\CF}^i(\mathrm{Sym}^{n+1}F,\hat{S}_r)
}$$

\medskip
We define a bimodule 
\begin{align*}
	L'_r=L'_{r,n}:\CF(\mathrm{Sym}^nF,S)\otimes \CF(\mathrm{Sym}^{n+r}F,S)^\opp&\to
k\mdiff\\
\lambda_1\otimes \lambda_2&\mapsto \Hom(\Xi'_r\cdots\Xi'_1(\lambda_2),
	\Upsilon'_r\cdots\Upsilon'_1(\lambda_1)).
\end{align*}
We put $L'_r=\bigoplus_{n\ge 0}L'_{r,n}$, a $(\CF(\mathrm{Sym}^* F,S),
\CF(\mathrm{Sym}^* F,S))$-bimodule.
We have an isomorphism of bimodules
$L(-,-,r)\iso L'_r\circ (\Phi\otimes\Phi)$.

\medskip
Consider $t\in\Hom_{\CF(F,\hat{S}_2)}(\omega_1,\omega_2)$ corresponding, via Auroux's
equivalence, to $[\xi(1)\to\xi(2)]$. Similarly,
we consider the two maps $u,v\in\Hom_{\CF(\mathrm{Sym}^2(F),\hat{S}_3)}(\omega_1\sqcup
\omega_2,\omega_2\sqcup\omega_3)$ corresponding, via Auroux's equivalences,
to $\{[\xi(1)\to\xi(2)],[\xi(2)\to\xi(3)]\}$ and
to $\{\id_{\xi(2)},[\xi(1)\to\xi(3)]\}$ respectively.

\begin{center}
	\includegraphics[scale=1.0]{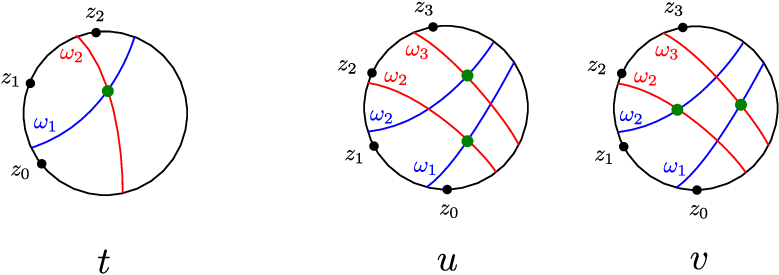}
\end{center}

Composition with $t$ induces an isomorphism
$$f_t:L'_1(\lambda_1,\lambda_2)\iso
\Hom(\Xi'_2\Xi'_1(\lambda_2),\Upsilon'_2\Xi'_1(\lambda_1)).$$

Composition with $u$ and $v$ induce morphisms 
$$f_u,f_v:L'_2(\lambda_1,\lambda_2)\to 
\Hom(\Xi'_3\Xi'_2\Xi'_1(\lambda_2),\Upsilon'_3\Upsilon'_2\Xi'_1(\lambda_1)).$$
The map $f_u$ is invertible and we put $\tau=f_u^{-1}\circ f_v$.

The composition map
$$L'_1(\lambda_1,-)\otimes_{\CF(\mathrm{Sym}^nF,S)}L'_1(-,\lambda_2)\to
L'_2(\lambda_1,\lambda_2),\ x\otimes y\mapsto \Upsilon'_2(x)\circ f_t(y)$$
is an isomorphism. Via this isomorphism, $\tau$ defines an endomorphism of
$(L'_1)^2$.

%
%
%
%
%
%

%

The relations (\ref{eq:definingtau}) are satisfied because $\tau$ arises from a map coming from strand algebras.

\begin{rem}
	Our construction is similar to the sketch provided by Douglas and
	Manolescu in \cite[\S 2.3]{DouMa}.
\end{rem}

\subsection{Gluing}
\label{se:gluing}
\subsubsection{Construction}
Consider two injective morphisms of curves $\xi_1^+:\BR_{>0}\to Z$ and
$\xi_2^-:\BR_{<0}\to Z$ where $\BR_{<0}$ and $\BR_{>0}$ are unoriented.
We assume that $\xi_1^+$ is outgoing for $Z$, that $\xi_2^-$ is incoming for $Z$ and that
$\xi_1^+(\BR_{>0})\cap\xi_2^-(\BR_{<0})=\emptyset$.
We write $r$ instead of $\xi_1^+(r)$ and $-r$ instead of $\xi_2^-(-r)$,
for $r\in\BZ_{>0}$.

\smallskip

Let $M$ be a subset of $Z\setminus(\xi_1^+(\BR_{\ge 1})\sqcup\xi_2^-(\BR_{\le -1}))$.

\medskip
Fix an oriented diffeomorphism $\BR_{>0}\iso\BR_{<-1}$ and let 
$i_+:\BR_{>0}\to \BR$ be its composition
with the inclusion map. Similarly, fix an oriented diffeomorphism $\BR_{<0}
\iso\BR_{>1}$ and let $i_-:\BR_{<0}
\to \BR$ be its composition with the inclusion map.

\smallskip
Consider $m,n\ge 0$.
Let $E_{m,n}$ be the $(\CS_M^\bullet(Z),\CS_M^\bullet(Z))$-bimodule given by
$$E_{m,n}(T,S)=\Hom_{\CS^\bullet(Z)}(S\sqcup (-n,-1),T\sqcup (1,m)).$$

Note that $E_{0,1} = R^{\bullet}_{\xi_2^-}$ and $E_{1,0} = L^{\bullet}_{\xi_1^+}$, but
$E_{m,n}$ is not isomorphic to $(R^{\bullet}_{\xi_2^-})^n (L^{\bullet}_{\xi_1^+})^m$
in general.

\smallskip
There is an action of $H_m^\bullet\wedge H_n^\bullet$ on
$E_{m,n}$ given by
$$
(T_a\wedge T_b)\cdot\sigma=(\id_T\boxtimes ([i\mapsto a(i)]_{1\le i\le m})\cdot
\sigma \cdot (\id_S\boxtimes (-i\mapsto b^{-1}(n+1-i)-n-1)_{1\le i\le n})
$$
for $\sigma\in\Hom_{\CS^\bullet(Z)}(S\sqcup (-n,-1),T\sqcup (1,m))$, $a\in\GS_m$
and $b\in \GS_n$.

\smallskip
There is a map $\ast:E_{m,n}E_{m',n'}\to E_{m+m',n+n'}$ given by
$$\alpha\wedge\beta\mapsto \alpha\ast\beta=
(\alpha\boxtimes ([i\to i+m])_{1\le i\le m'})\cdot
(\beta\boxtimes ([-n'-i\to -i])_{1\le i\le n}).$$

This map is compatible with the action of
$(H_m^\bullet\wedge H_n^\bullet)\wedge (H_{m'}^\bullet\wedge H_{n'}^\bullet)$ via
the canonical embeddings $H_m^\bullet H_{m'}^\bullet\to H_{m+m'}^\bullet$ and
$H_n^\bullet H_{n'}^\bullet\to H_{n+n'}^\bullet$.
We have $(\alpha\ast\beta)\ast\gamma=\alpha\ast(\beta\ast\gamma)$.

\smallskip
So, we have defined a bimodule lax bi-$2$-representation on $\CS_M^\bullet(Z)$.

\medskip
Let $Z_\xi=Z\sqcup_{\BR_{>0}\sqcup\BR_{<0}}\BR$,\indexnot{Zx}{Z_\xi} where
the gluing is done along the maps $\xi_1^+\sqcup\xi_2^-:\BR_{>0}\sqcup
\BR_{<0}\to Z$
and $i_+\sqcup i_-:\BR_{>0}\sqcup\BR_{<0}\to\BR$.  
Note that $Z_\xi$ is a $1$-dimensional space and it comes with an injective open
morphism of $1$-dimensional spaces $\xi:\BR\to Z_\xi$. We endow $\BR$ with a curve
structure by setting $\BR_u=\BR_{\le -1}\sqcup\BR_{\ge 1}$ and by endowing $(-1,1)$
with its usual orientation. We extend the curve structure on $Z$
by endowing $\xi(\BR)$ with the curve structure of $\BR$.
Note that $(Z_\xi)_u=\overline{Z_u}$.

\smallskip
Given $\eps,\eps'\in\{+,-\}$ and $a\in\BR_{\eps}$, $b\in\BR_{\eps'}$,
we put $[a\to b]= \xi([i_\eps(a),i_{\eps'}(b)])$.

\medskip
We consider the differential pointed category
$T_{\CS^\bullet_M(Z)}(R_{\xi_2^-}^\bullet L_{\xi_1^+}^\bullet)$ with objects those of
$\CS^\bullet_M(Z)$ and with
$$\Hom_{\tilde{\CS}^\bullet_M(Z)}(S,T)=\bigvee_{i\ge 0}
R_{\xi_2^-}^\bullet(T,-_i)\wedge L_{\xi_1^+}^\bullet(-_i,-_{i-1})\wedge\cdots\wedge
R_{\xi_2^-}^\bullet(-_2,-_1)\wedge L_{\xi_1^+}^\bullet(-_1,S).$$

We define a differential pointed functor 
$\tilde{\Xi}:T_{\CS^\bullet_M(Z)}(R_{\xi_2^-}^\bullet L_{\xi_1^+}^\bullet)\to \CS^\bullet_M(Z_\xi)$. It is
the identity on objects and defined on maps by
$$\beta_i\wedge\alpha_i\wedge\cdots\wedge\beta_1\wedge\alpha_1\mapsto 
(\beta_i\cdot(\id\boxtimes [1\to -1])\cdot\alpha_i)\cdot\cdots\cdot
(\beta_1\cdot(\id\boxtimes [1\to -1])\cdot\alpha_1):$$
$$S\xrightarrow{\alpha_1}U_1\sqcup\{\xi_1^+(1)\}\xrightarrow{\id_{U_1}\boxtimes [1\to -1]}
U_1\sqcup\{\xi_2^-(-1)\}\xrightarrow{\beta_1}V_1\xrightarrow{\alpha_2}\cdots\to T.$$

\begin{thm}
	\label{th:gluing}
The functor $\tilde{\Xi}$ factors through $\Delta_E\CS^\bullet_{M}(Z)$
and induces an isomorphism of differential pointed categories
	$\Xi:\Delta_E\CS^\bullet_{M}(Z)\iso \CS^\bullet_M(Z_\xi)$.\indexnot{Xi}{\Xi}
\end{thm}

The sections \S \ref{se:Gbimodules}-\ref{se:equivrelation} below are devoted to the proof of Theorem \ref{th:gluing}.

\begin{example}
	We give below an illustration of the gluing data.
$$\includegraphics[scale=0.85]{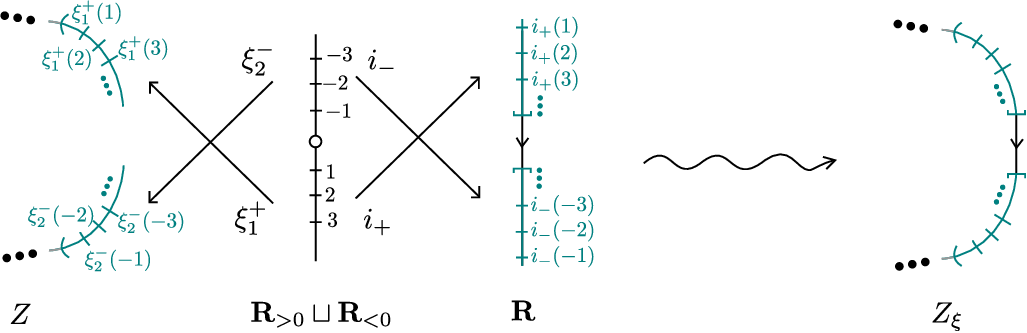}$$
\end{example}

\begin{example}
	The pictures below give two examples of description of
	$\tilde{\Xi}$. The first picture corresponds to the gluing of two
	intervals to form an interval. The second picture corresponds to
	the self-gluing of an interval to form a circle.
$$\includegraphics[scale=0.85]{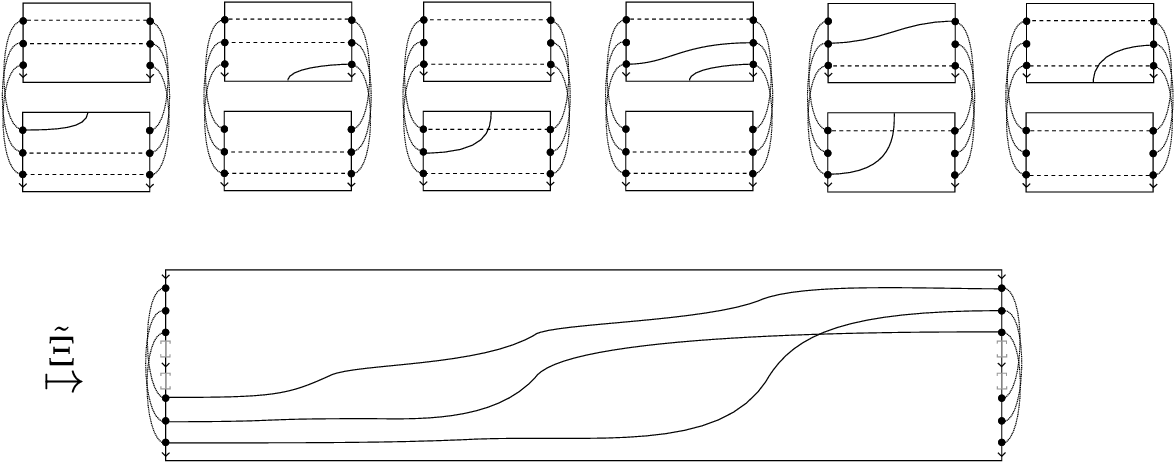}$$
$$\includegraphics[scale=0.85]{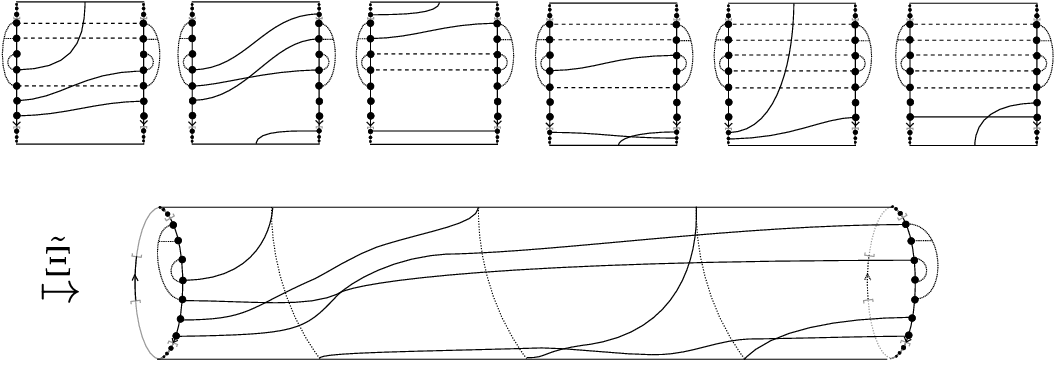}$$
\end{example}

\subsubsection{Bimodules}
\label{se:Gbimodules}

If $\Hom_{\CS(Z_\xi)^\bullet}(\{-1\},\{1\})\neq 0$, then there is
$\kappa'\in\Hom_{\CS(Z_\xi)^\bullet}(\{-1\},\{1\})$ such that
$\Hom_{\CS(Z_\xi)^\bullet}(\{-1\},\{1\})=\{\kappa^n\cdot\kappa'\}_{n\ge 0}$,
where $\kappa=\kappa'\cdot [1\to -1]$. 

When $\Hom_{\CS(Z_\xi)^\bullet}(\{-1\},\{1\})=0$, we put $\kappa=\id_1$.

\smallskip
We define a partial order on the component $Z'$ of $Z$ containing
$1$. We define $s<s'$ if there exists an admissible path
$\zeta:s'\to 1$ in $Z'$ whose support does not contain $s$.

\smallskip
We consider the map $\mu$ of \S\ref{se:decompositionz0} for the curve
$Z_\xi$ and its point $z_0=0$.

\medskip
Given $n\ge 0$, we put $G_n=E_{n,n}$.

\begin{lemma}
	\label{le:GnTi}
	Let $\alpha\in G_n-\{0\}$.

	Given $i\in (1,n-1)$, the following assertions are equivalent
	\begin{enumerate}
		\item $\alpha(i-n-1)>\alpha(i-n)$
\item $L(\alpha_{|\{i-n-1,i-n\}})\neq\emptyset$
\item $[i-n-1\to i-n]\in D(\alpha)$
\item $\alpha\in G_n T_i$
\item $\alpha T_i=0$.
	\end{enumerate}

	There exists $i\in (1,n-1)$ such that $\alpha\in G_n T_i$ if
	and only if $L(\alpha_{|(-n,-1)})\neq\emptyset$.
\end{lemma}

\begin{proof}
	The equivalence between (1) and (2) follows from Lemma \ref{le:inversionpaths}.

	Assume (2). We deduce that $[i-n-1\to i-n]\in L(\alpha)$, hence
	$[i-n-1\to i-n]\in D(\alpha)$. So (3) holds.

	Assume (3). Writing $\alpha=\alpha\cdot 1$, we deduce from 
	Lemma \ref{le:middlebraids} that (4) holds.

	The implication (4)$\Rightarrow$(5) is immediate.

	Asssume (5). We have $\alpha_{|\{i-n-1,i-n\}}\cdot (
	[i-n-1\to i-n]\boxtimes [i-n\to i-n-1])=0$
	by Remark \ref{re:productnonexc}. Lemma \ref{le:degproduct} shows that
	$i(\alpha_{|\{i-n-1,i-n\}})\neq 0$, hence (2) holds.

	\smallskip
	Assume now $L(\alpha_{|(-n,-1)})\neq\emptyset$. It follows
	from Lemma \ref{le:inversionpaths} that there is $i\in (1,n-1)$ with 
	$\alpha(i-n-1)>\alpha(i-n)$, hence $\alpha\in G_n T_i$. This shows
	the last statement of the lemma.
\end{proof}

\medskip
There is a map
$\nu_n:R_{\xi_2^-}^\bullet(-,-,e^n)L^{\bullet}_{\xi_1^+}(-,-,e^n)\to G_n$ given by
$$\Hom(-\sqcup (-n,-1),T)\wedge\Hom(S,-\sqcup(1,n))\to
\Hom(S\sqcup(-n,-1),T\sqcup(1,n))$$
$$\beta\wedge\alpha\mapsto (\beta\boxtimes\id_{(1,n)})\cdot
(\alpha\boxtimes\id_{(-n,-1)})$$
We have
$$\nu_n\bigl((\beta\cdot T_b)\wedge (T_a\cdot\alpha)\bigr)=
T_a\cdot\nu_n(\beta\wedge \alpha)\cdot\iota_n(T_b)$$
for $a,b\in\GS_n$.

\smallskip
The multiplication map on $E$ defines
a map $\mu_n:(R_{\xi_2^-}^\bullet L^{\bullet}_{\xi_1^+})^n=(E_{0,1}E_{1,0})^n\to E_{n,n}=G_n$, hence gives
a morphism $T^*(R_{\xi_2^-}^\bullet L^{\bullet}_{\xi_1^+})\to G=\bigvee_{n\ge 0}G_n$
compatible with multiplication.

\medskip
We define $(\CS_M^\bullet(Z),\CS_M^\bullet(Z))$-subbimodules 
$A_n$, $B_n$, $C_n$, $D_n$, $E_n$ and $F_n$ of $G_n$. Let
$\sigma\in G_n(T,S)$.

We have
\begin{itemize}
\item $\sigma\in A_n$ if $\sigma(-i)\in T\sqcup (1,n-i)$ for $1\le i\le n$
	\item $\sigma\in B_n$ if there exists $1\le j\le i\le n$ with
$\sigma(-i)=n-j+1$
	\item $\sigma\in C_n$ if it is in the image of $\mu_n$
	\item $\sigma\in D_n$ if $\sigma(-i)\in T$ for $1\le i\le n$
	\item $\sigma\in E_n$ if $\sigma\in A_n$ and
		$L(\sigma_{|(-n,-1)})=\emptyset$.
	\item $\sigma\in F_n$ if $\sigma\in A_n$ and
		$L(\sigma_{|\sigma^{-1}(1,n)})=\emptyset$.
\end{itemize}

\medskip

We put $A=\bigvee_{n\ge 0}A_n$, $B=\bigvee_{n\ge 0}B_n$, etc.

Note that $G_n=A_n\vee B_n$.

We have $C_n,D_n,E_n,F_n\subset A_n$.

\begin{lemma}
	\label{le:nu}
	We have an isomorphism $\nu_n:R_{\xi_2^-}^\bullet (-,-,e^n)L^{\bullet}_{\xi_1^+}(-,-,e^n)\iso D_n$.

	In particular, we have an isomorphism $\nu_1:R_{\xi_2^-}^\bullet L^{\bullet}_{\xi_1^+} \iso D_1=A_1=C_1$
	and $C_n=C_1^{\ast n}=A_1^{\ast n}$.
\end{lemma}

\begin{proof}
	Let $\beta\wedge\alpha\in \Hom(-\sqcup(-n,-1),T)\wedge
	\Hom(S,-\sqcup(1,n))$. We have
	$\beta\wedge\alpha=\beta'\wedge\alpha'$ where
	$\beta'=\beta_{|(-n,-1)}\boxtimes\id$ and
	$\alpha'=(\beta_{|-}\boxtimes\id_{(1,n)})\cdot\alpha$.
	If $\nu_n(\beta'\wedge\alpha')=0$, then 
	$\beta'=\alpha'=0$ (cf the beginning of \S \ref{se:tensorstrands}).
	Now $\nu_n$ has an inverse given by
	$\sigma\mapsto (\id\boxtimes\sigma_{|(-n,-1)})\wedge\sigma_{|S}$.
\end{proof}

\begin{rem}
	\label{re:oppositebimodules}
	Consider $\bar{\xi}_1^+:\BR_{>0}\to Z^{\opp},\ x\mapsto \xi_2^-(-x)$
	and 
	 $\bar{\xi}_2^-:\BR_{<0}\to Z^{\opp},\ x\mapsto \xi_1^+(-x)$.
	 There is an isomorphism $(Z^\opp)_{\bar{\xi}}\iso (Z_\xi)^\opp$
	 that is the identity on $Z$ and $x\mapsto -x$ on $\BR$. This
	 provides an isomorphism $(\CS^\bullet(Z_\xi))^\opp\iso
	 \CS^\bullet(Z^{\opp}_{\bar{\xi}})$.
	It induces isomorphisms
	$$\Hom_{\CS^\bullet(Z)}(S\sqcup (-n,-1),T\sqcup(1,n))\iso
	\Hom_{\CS^\bullet(Z^\opp)}(T\sqcup(-n,-1),S\sqcup (1,n)).$$
	This restricts to isomorphisms between $A_n$ (resp. $B_n$, $D_n$, $E_n$,
	$F_n$) for $Z$ and 
	$A_n$ (resp. $B_n$, $D_n$, $F_n$, $E_n$) for $Z^\opp$.
\end{rem}

\begin{lemma}
	\label{le:propertiesAB}
	\begin{itemize}
		\item
	$B_n$ and $D_n$ are stable under the action of
	$H_n^\bullet\wedge (H_n^\bullet)^\opp$.
\item $E_n$ is stable under the action of $H_n^\bullet$ and
	$F_n$ is stable under the action of $(H_n^\bullet)^\opp$.
\item $A$ and $C$ are stable under multiplication
\item Given $\alpha\in B$ and $\beta\in G$, we have $\alpha\ast\beta\in B$
	and $\beta\ast\alpha\in B$.
	\end{itemize}
\end{lemma}

\begin{proof}
	Let $\sigma\in B_n$ and $r\in\{1,\ldots,n-1\}$.
	Assume $\sigma T_r\neq 0$.

	If there is $1\le j\le i\le n$ with $\sigma(-i)=n-j+1$
	and $i\neq n+1-r$, then $\sigma T_r\in B_n$.

	Assume now $\sigma(-i)\in T\sqcup (1,n-i)$ for all
	$i\neq n+1-r$. We deduce that 
	$L(\sigma_{|\{-(n+1-r),-(n-r)\}})\neq\emptyset$, hence
	$\sigma T_r=0$ (cf Lemma \ref{le:GnTi}), a contradiction.

	Using Remark \ref{re:oppositebimodules}, we deduce that
	$T_r\sigma\in B_n$.

	The other assertions of the lemma are immediate.
\end{proof}	

\subsubsection{Gluing map}
We define a morphism of $(\CS_M^\bullet(Z),\CS_M^\bullet(Z))$-bimodules
$q:G\to \Id_{\CS_M^\bullet(Z_\xi)}$:
$$\Hom_{\CS^\bullet(Z)}(S\sqcup (-n,-1),T\sqcup (1,n))\to
\Hom_{\CS^\bullet(Z_\xi)}(S,T).$$

Let $\alpha\in A_n$. We put $T_1=\alpha(S)\cap T$ and
$I_1=\alpha(S)\cap (1,n)$. We define inductively $T_m\subset T$ and 
$I_m\subset (1,n-m+1)$ for $1<m\le n+1$ by
$T_m=T_{m-1}\sqcup (\alpha(-n+I_{m-1}-1)\cap T)$ and
$I_m=\alpha(-n+I_{m-1}-1)\cap (1,n)$.

Note that $-n+I_{m-1}-1\subset (-n,-m+1)$, hence 
$I_m\subset (1,n-m+1)$ since $\alpha\in A_n$.

Note that $T_{n+1}=T$ and $I_{n+1}=\emptyset$.

Define 
$$\beta^m=\id_{T_m}\boxtimes\bigl(\bigboxtimes_{r\in I_m}(\alpha_{-n+r-1}\cdot
[r\to -n+r-1])\bigr):T_m\sqcup I_m\to T_{m+1}\sqcup I_{m+1}$$
for $1\le m\le n$.
We define $q(\alpha)=\beta^n\cdot\beta^{n-1}\cdots\beta^1\cdot\alpha_{|S}$
$$q(\alpha):S\xrightarrow{\alpha_{|S}} T_1\sqcup I_1\xrightarrow{\beta^1}
T_2\sqcup I_2\to\cdots\to T_n\sqcup I_n\xrightarrow{\beta^n} T.$$

We put $q(\alpha)=0$ if $\alpha\in B_n$.

\medskip
Assume now $q(\alpha)\neq 0$, hence $\alpha\in A_n$.
Let $S'=S\cap\alpha^{-1}(T)$ and $T'=T\cap\alpha(S)$.
Let $S''=S-S'$ and $T''=T-T'$. 

Given $s\in S'$, we have $q(\alpha)_s=\alpha_s$.

Note in particular that $S'=\{s\in S\ |\ \mu(q(\alpha)_s)=0\}$.

\smallskip
Let $s\in S''$, $t=q(\alpha)(s)$ and $i=\alpha^{-1}(t)$. Put $d_s=\mu(q(\alpha)_s)-1\ge 0$. We have
$$q(\alpha)_s= \alpha_i\cdot
[1\to i]\cdot \kappa^{d_s}\cdot [\alpha(s)\to 1] \cdot \alpha_s.$$

Given a decomposition $q(\alpha)_s= \xi\cdot [1\to -1]\cdot\kappa^{d_s}
\cdot\xi'$ with $\xi'\in\Hom_{\CS^\bullet(Z)}(\{s\},\{1\})$ and
$\xi\in\Hom_{\CS^\bullet(Z)}(\{-1\},\{t\})$, we have
$\alpha_i=\xi\cdot [i\to -1]$ and $\alpha_s=[1\to\alpha(s)] \cdot \xi'$.

\smallskip
The next lemma is immediate.
\begin{lemma}
	\label{le:propq}
The map $q$ defines a morphism of 
$(\CS_M^\bullet(Z),\CS_M^\bullet(Z))$-bimodules
$G\to \Id_{\CS_M^\bullet(Z_\xi)}$ and
	$q(\alpha\ast\alpha')=q(\alpha)\cdot q(\alpha')$.

	Given $h\in H^\bullet_n$ and $\alpha\in G_n$,
	we have $q(h\alpha)=q(\alpha h)$.
\end{lemma}

\begin{lemma}
	\label{le:qinjEF}
	The restrictions of $q$ to $E$ and to $F$ are injective.
\end{lemma}

\begin{proof}
	Let $\alpha:S\sqcup(-n,-1)\to T\sqcup (1,n)$ be a non-zero element of
	$F_n$. 
Let $s\in S''$. Given $1\le m\le n$, we put
	$i_m(s)=\beta^{m-1}\circ\cdots\circ\beta^1\circ\alpha(s)$. We put
	$d_s=\min\{m| i_{m+1}(s)\in T_{m+1}\}$.

	Let $s,s'$ be two distinct elements of $S$ and
	let $\theta=\beta^n_{|\beta^{n-1}\circ\cdots\circ\beta^1\circ\alpha(\{s,s'\})}
	\cdots\beta^1_{|\alpha(\{s,s'\})}
	\cdot\alpha_{|\{s,s'\}}$.

	\smallskip
	$\bullet\ $If $s,s'\in S'$, then 
	$\theta=\alpha_{|\{s,s'\}}\neq 0$.

	$\bullet\ $Assume $s\in S'$ and $s'\in S''$.
	We have
	$\theta=\theta^1\cdot\alpha_{|\{s,s'\}}$ where
	$$\theta^1=(\id_{\alpha(s)}\boxtimes
	(\alpha_{-n+i_{d_{s'}}(s')-1}\cdot [1\to -n+i_{d_{s'}}(s')-1]
	\cdot\kappa^{d_{s'}-1} \cdot [i_{1}(s')\to 1])).$$
	We have 
	$$i(\theta^1\circ\alpha_{|\{s,s'\}})=i(\alpha_s,\alpha_{s'})+i(\alpha_s,
	\alpha_{-n+i_{d_{s'}}(s')-1})+d_{s'}-1=
			i(\alpha_{|\{s,s'\}})+i(\theta^1).$$
	Ir follows that $\theta\neq 0$.

	\smallskip
	$\bullet\ $Assume finally $s,s'\in S''$ and $d_{s'}\ge d_s$.
We have $\theta=\theta^1\cdot\theta^2\cdot\theta^3\cdot
	\alpha_{|\{s,s'\}}$ where
$$\theta^1=(\alpha_{-n+i_{d_{s'}}(s')-1}\cdot [1\to -n+i_{d_{s'}}(s')-1]
	\cdot\kappa^{d_{s'}-d_s}
	\cdot [i_{d_{s}+1}(s')\to 1])\boxtimes\id_{i_n(s)}$$
$$\theta^2=[-n+i_{d_s}(s')+1\to i_{d_s+1}(s')]\boxtimes(\alpha_{n-i_{d_s}(s)+1}\cdot[-n\to n-i_{d_s}(s)+1])$$
$$\theta^3=(([1\to -n+i_{d_s}(s')+1]\cdot\kappa^{d_s-1}\cdot [i_{1}(s')\to 1])\boxtimes
	([1\to -n]\cdot\kappa^{d_s-1}\cdot [i_{1}(s)\to 1])).$$
We have
	$$i(\theta^1\circ\theta^2\circ\theta^3\circ \alpha_{|\{s,s'\}})=
	i(\alpha_{i_{d_s}(s)},\alpha_{i_{d_{s'}}(s')})+
	d_{s'}-d_s+i(\alpha_s,\alpha_{s'})=
	i(\theta^1)+i(\theta^2)+i(\theta^3)+i(\alpha_{|\{s,s'\}}).
	$$
	It follows that $\theta\neq 0$.

	\smallskip
	It follows from Remark \ref{re:productnonexc} that $q(\alpha)\neq 0$.	
	
	\medskip
	Define $S'$ and $S''$ as above.
	Let $r=|S''|$. We have $\alpha(S'')=(n-r+1,n)$ and
	$\alpha^{-1}(i)<\alpha^{-1}(i')$ for $i<i'$ in
	$(n-r+1,n)$.

	Given $i<i'$ in $(-n,-1)$ with $\alpha(i),\alpha(i')\in (1,n)$, we have
	$\alpha(i)<\alpha(i')$.

	\smallskip
	Consider now $\tilde{\alpha}:S\sqcup(-n,-1)\to T\sqcup (1,n)$ another
	non-zero element of $F_n$ and assume $q(\alpha)=q(\tilde{\alpha})\neq0 $.
	We have $S\cap\tilde{\alpha}^{-1}(T)=S'$ and
	$T\cap \tilde{\alpha}(S)=T'$. The discussion above shows that
	$\alpha(s)=\tilde{\alpha}(s)$ for $s\in S''$.
	Note also that $\alpha_s=\tilde{\alpha}_s$ for $s\in S'$.
	As a consequence, $\alpha=\tilde{\alpha}$ if $\mu(q(\alpha))=0$.

	\smallskip

	Let $s\in S''$, $t=q(\alpha)(s)$, $\tilde{t}=q(\tilde{\alpha})(s)$, 
	$i=\alpha^{-1}(t)$ and $\tilde{i}=\tilde{\alpha}^{-1}(t)$.
	Since $q(\alpha)_s=q(\tilde{\alpha})_s$, it follows that
	$\tilde{t}=t$, $\mu(q(\alpha)_s)=\mu(q(\tilde{\alpha})_s)$,
	$[\alpha(s)\to 1]\cdot \alpha_s=[\tilde{\alpha}(s)\to 1]\cdot 
	\tilde{\alpha}_s$, and
	$\alpha_i\cdot [-1\to i]=\tilde{\alpha}_{\tilde{i}}\cdot
	[-1\to \tilde{i}]$. We deduce that $\tilde{\alpha}_s=\alpha_s$ for
	$s\in S''$.

	\medskip
	We proceed now by induction on $\mu(q(\alpha))$ to show
	that $q(\alpha)$ determines $\alpha$, for $\alpha\in F$.

	\smallskip
	Assume there is $s\in S''$ such that $\mu(q(\alpha)_s)=1$.
	Let $j=\alpha(s)\in (1,n)$ and $i=-n+j-1$. We have
	$t=\alpha(i)=q(\alpha)(s)\in T$. Define $\alpha':S\setminus\{s\}\sqcup
	(-n+1,-1)\to T\setminus\{t\}\sqcup (1,n-1)$ an element of
	$F_{n-1}$ as follows.
	Given $s'\in S\setminus\{s\}$, we put
	$\alpha'_{s'}=\alpha_{s'}$ if $\alpha(s')<j$,
	$\alpha'_{s'}=[\alpha(s')\to\alpha(s')-1]\cdot
	\alpha_{s'}$ if $\alpha(s')>j$. Given $i'\in (-i+1,-1)$, we
	put $\alpha'_{i'}=\alpha_{i'}$. Given $i'\in (-n+1,-i)$, we put
	$\alpha'_{i'}=\alpha_{i'-1}$. This defines an element of
	$F_{n-1}$. Furthermore, $q(\alpha')=q(\alpha)_{|S\setminus\{s\}}$.
 
	We define similarly $\tilde{i}$, $\tilde{j}$, $\tilde{t}$
	and $\tilde{\alpha}'$ starting with $\tilde{\alpha}$ and $s$.
	We have $\tilde{j}=j$ and $\tilde{t}=t$, hence also $\tilde{i}=i$.
	We have
	$q(\alpha')=q(\tilde{\alpha}')$, hence $\alpha'=\tilde{\alpha}'$ by
	induction. Since $\alpha_s=\tilde{\alpha}_s$ and
	$\alpha_i=\tilde{\alpha}_i$, it follows that $\alpha=\tilde{\alpha}$.

	\smallskip
	Assume $\mu(q(\alpha)_s)\ge 2$ for all $s\in S''$.
We have $\alpha^{-1}((1,n-r))=\{i_1<\cdots<i_{n-r}\}\subset (-n,-1)$.
	Note that $\alpha_{-i_d}=[-i_d\to d]$ for $1\le d\le n-r$.
Let $\varphi:(-r,-1)\to (-n,-1)\setminus\alpha^{-1}((1,n-r))$ be the unique
increasing bijection.
	We define $\alpha':S\sqcup (-r,-1)\to T\sqcup (1,r)$
	and an element of $F_r$ as follows.
	We put $\alpha'_{s}=\alpha_s$ for $s\in S'$, 
	$\alpha'_s=[\alpha(s)\to\alpha(s)-n+r]\cdot\alpha_s$ for
	$s\in S''$ and
	$\alpha_i=\alpha_{\varphi(i)}\cdot [i\to\varphi(i)]$ for $i\in (-r,-1)$.

Let $s\in S''$, $t=q(\alpha)(s)$ and $i=\alpha^{-1}(t)$. We have
$$q(\alpha')_s= \alpha_i\cdot
[1\to i]\cdot [\alpha(s)\to 1] \cdot \alpha_s.$$
Define $\tilde{\alpha}'$ similarly, starting with $\tilde{\alpha}$ instead
of $\alpha$.
We have $q(\alpha')=q(\tilde{\alpha}')$. By induction, we deduce
that $\alpha'=\tilde{\alpha}'$, hence $\alpha=\tilde{\alpha}$.

This completes the proof that the restriction of $q$ to $F$ is injective.

We deduce that the restriction of $q$ to $E$ is injective using
Remark \ref{re:oppositebimodules}
\end{proof}

\begin{lemma}
	\label{le:qsurj}
	The restrictions of $q$ to $E\cap C$ and to $F\cap C$ are surjective.
\end{lemma}

\begin{proof}
	Let $\theta\in\Hom_{\CS^\bullet_M(Z_\xi)}(I,J)$. Let $n=\mu(\theta)$.
	We show by induction on $n$ that there exists
	$\alpha\in F_n\cap C_n$ such that $q(\alpha)=\theta$.

	Assume $n=1$. Let $s\in I$ such that $\mu(\theta_s)=1$.
	There is a decomposition $\theta_s=\theta_s^{r-}\cdot
	\theta_s^r$ as in \S\ref{se:decompositionz0}.
	We define $\alpha\in\Hom_{\CS^\bullet_M(Z)}(I\sqcup\{-1\},J\sqcup\{1\})$
	by
	$\alpha_{s'}=\theta_{s'}$ for $s'\neq s$,
	$\alpha_s=[0\to 1]\cdot\theta_s^r$ and
	$\alpha_{-1}=\theta_s^{r-}\cdot [-1\to 0]$.
	We have $\alpha\in A_1=F_1\cap C_1$ and $q(\alpha)=\theta$.

	\smallskip
	Assume now $n>1$. Consider a decomposition
	$\theta=r'(\theta)\cdot r(\theta)$ as in 
	Lemma \ref{le:divisorlength1}. There exists
	$\alpha\in A_1$ and $\beta\in F_{n-1}\cap C_{n-1}$ such that
	$q(\alpha)=r(\theta)$ and $q(\beta)=r'(\theta)$.
	Let $\gamma=\beta\ast\alpha\in C_n$. We have $q(\gamma)=\theta$.

	Let $s=\gamma^{-1}(n)=\alpha^{-1}(1)$. We have $\mu(r(\theta)_s)=1$.
	Let $i\in (1,n-1)$ and $s'=\gamma^{-1}(i)$. If $s'\in (-n,-1)$,
	then $I(\gamma_{|\{s',s\}})=\emptyset$. Assume $s'{\not\in}(-n,-1)$.
	We have $\theta_{s'}^r=[i\to 0]\cdot\gamma_{s'}$. Since
	$\mathrm{supp}(\theta_s^r)\subset\mathrm{supp}(\theta_{s'}^r)$,
	it follows that $I(\gamma_{|\{s',s\}})=\emptyset$.
Since $\beta\in F_{n-1}$, we deduce that $\gamma\in F_n$.

	\smallskip
	The case of $E\cap C$ follows from that of $F\cap C$ applied to
	$Z^\opp$, cf Remark \ref{re:oppositebimodules}.
\end{proof}

\subsubsection{Equivalence relation}
\label{se:equivrelation}
We define an equivalence relation $\sim$ on $G$ as the transitive, symmetric
and reflexive closure of the relation
$T_i\sigma\sim \sigma T_i$ for $\sigma\in G_n$ and $1\le i<n$ and
$\sigma\sim0$ if $\sigma\in B_n$.

\begin{lemma}
	\label{le:simtoEF}
	Let $\alpha\in G_n$. There exists $\sigma\in E_n$ and $\sigma'\in F_n$
	such that $\alpha\sim\sigma\sim\sigma'$.
\end{lemma}

\begin{proof}
	If $\alpha\in B_n$, then $\alpha\sim 0$ and we are done.
	Assume now $\alpha\in A_n$. We proceed by induction
	on $M(\alpha)=\frac{1}{2}|L(\alpha_{|(-n,-1)})|$ and then on
	$N(\alpha)=n-\max\{i\ |\ [-n+i-1\to -n+i]\in L(\alpha)\}$ if
	$M(\alpha)\neq 0$ to show that there exists
	$\sigma\in E_n$ with $\alpha\sim\sigma$.
	
	If $M(\alpha)=0$, then $\alpha\in E_n$ and we are
	done. Assume now $M(\alpha)>0$. By Lemma \ref{le:GnTi}, there are
	$i\in (1,n-1)$ and $\beta\in G_n$ such that $\alpha=\beta T_i$, and
	we choose $i$ maximal with this property, so that
	$N(\alpha)=n-i$.
	We have $\alpha\sim T_i\beta$. If $T_i\beta\in B_n$ then we are done.
	We assume now $T_i\beta{\not\in}B_n$.
	We have $L(\beta_{|(-n,-1)})=L(\alpha_{|(-n,-1)})
	\setminus\{[-n+i-1\to -n+i],[-n+i\to -n+i-1]\}$.

	If $\beta^{-1}(\{i,i+1\}){\not\subset}(-n,-1)$, then
	$L(T_i\beta_{|(-n,-1)})=L(\beta_{|(-n,-1)})$, hence
	$M(T_i\beta)<M(\alpha)$. By induction, there is $\sigma\in E_n$
	with $T_i\beta\sim\sigma$, hence $\alpha\sim\sigma$.

	Assume now there are $j,k\in (1,n)$ with
	$\beta(-n+j-1)=i$ and $\beta(-n+k-1)=i+1$.
	Since $T_i\beta\neq 0$, we have
	$j<k$. Since $\beta\in A_n$, we have
	$j>i$ and $k>i+1$.
	We have $M(T_i\beta)\le M(\beta)+1=M(\alpha)$. On the other
	hand, $[j\to k]\in L(T_i\beta)$ (cf Lemma \ref{le:inversionpaths}), hence
	$N(T_i\beta)<N(\alpha)$. We conclude by induction.

	The case of $F_n$ follows by applying Remark \ref{re:oppositebimodules}.
\end{proof}

\begin{lemma}
	\label{le:sim=q}
	Let $\alpha,\beta\in G_n$. We have $q(\alpha)=q(\beta)$ if
	and only if $\alpha\sim\beta$.
\end{lemma}

\begin{proof}
	Lemma \ref{le:propq} shows
	that if $\alpha\sim\beta$, then $q(\alpha)=q(\beta)$.
	Assume now $q(\alpha)=q(\beta)$. There are $\alpha',\beta'\in E_n$
	with $\alpha'\sim\alpha$ and $\beta'\sim\beta$ (Lemma
	\ref{le:simtoEF}) and we have $q(\alpha')=q(\alpha)=q(\beta)=
	q(\beta')$. It follows now from Lemma \ref{le:qinjEF} that
	$\alpha'=\beta'$, hence $\alpha\sim\beta$.
\end{proof}

\begin{proof}[Proof of Theorem \ref{th:gluing}]
Lemma \ref{le:sim=q} shows that $q$ factors through an isomorphism
$G/\!\sim\ \iso \Id_{\CS_M^\bullet(Z_\xi)}$. Since the restriction of $q$ to
	$C$ is surjective (Lemma \ref{le:qsurj}), it follows that $q$
	induces an isomorphism $C/\!\sim\ \iso \Id_{\CS_M^\bullet(Z_\xi)}$.

	Recall that $\mu_i:(R_{\xi_2^-}^\bullet L^{\bullet}_{\xi_1^+})^i\to G_i$
	has image $C_i$, hence $\mu_i$ induces an isomorphism
	$(R_{\xi_2^-}^\bullet L^{\bullet}_{\xi_1^+})^i/K_i\iso C_i/\!\sim$.
	As a consequence,
the canonical surjective map $T^*(R_{\xi_2^-}^\bullet L^{\bullet}_{\xi_1^+})\to
\Id_{\Delta_E(\CS^\bullet_{M}(Z))}$ factors through a surjective map
$C/\!\sim\ \to\Id_{\Delta_E(\CS^\bullet_{M}(Z))}$. Since the restriction
of $q$ to $C$ factors through $\Id_{\Delta_E(\CS^\bullet_{M}(Z))}$, we deduce that
we have an isomorphism $\Id_{\Delta_E(\CS^\bullet_{M}(Z))}\iso
\Id_{\CS_M^\bullet(Z_\xi)}$.
\end{proof}

\subsubsection{Complement}
We provide here a more direct description of the equivalence relation $\sim$ on $C$.

\begin{cor}
	\label{co:EFinC}
	We have $E\subset C$ and $F\subset C$.
\end{cor}

We define an  equivalence relation $\sim'$ on $C$ as the relation
generated by $\alpha'\ast(T_1 \alpha)\ast\alpha''\sim'
\alpha'\ast(\alpha T_1)\ast\alpha''$ for $\alpha',\alpha''\in C$ and
$\alpha\in D_2$.

\begin{lemma}
	\label{le:divisiblemiddle}
	Let $\sigma\in G_n$ and  $i\in\{1,\ldots,n-1\}$

	If $\sigma T_i\in C_n\setminus\{0\}$, then
$T_i\sigma\in C_n$ and $\sigma T_i\sim' T_i\sigma$.
	
	If $T_i\sigma\in C_n\setminus\{0\}$, then
	$\sigma T_i\in C_n$ and $\sigma T_i\sim' T_i\sigma$.
\end{lemma}

\begin{proof}
	Put $\sigma'=\sigma T_i$ and assume $\sigma'\in C_n\setminus\{0\}$.
	There are
	$\gamma\in C_{n-i-1}$, $\beta\in C_2$ and $\alpha\in C_{i-1}$
	such that $\sigma'=\alpha\ast\beta\ast\gamma$.

	Lemma \ref{le:GnTi} shows that $[-n+i-1\to -n+i]\in
	D(\sigma')$.
	We have $\sigma'_{|\{-n+i-1,-n+i\}}=(\alpha\ast\beta)_{|(-i-1,-i)}
	\circ ([-n+i-1\to -i-1]\boxtimes [-n+i\to -i])$.
	It follows from Lemma \ref{le:GnTi} that
	$[-i-1\to -i]\in D(\alpha\ast\beta)$.
	Since $[-i-1\to -i]\in L((\alpha\ast\beta)_{|(-i-1,-i)})$,
	it follows that $\beta(-1)\neq 1$, hence $\beta \in D_2$.

	\smallskip
	$\bullet\ $Assume $[-1\to -2]\in D(\beta)$. We have $\beta=\beta' T_1$ for some
	$\beta'\in G_2$ by Lemma \ref{le:GnTi}. 
	Since $\beta \in D_2$, we have
	$\beta'\in D_2\subset A_2$.  We deduce that $\beta'\in E_2$,
	hence $T_1\beta'\in E_2\subset C_2$ (Corollary \ref{co:EFinC}).
	So, $\sigma T_i=\alpha\ast(\beta' T_1)\ast\gamma\sim'
	\alpha\ast(T_1\beta')\ast\gamma=T_i\sigma$.

	\smallskip
	$\bullet\ $Assume now $[-1\to -2]{\not\in} D(\beta)$, i.e., $\beta\in E_2$. 
We have $T_1\beta,\beta T_1\in D_2\subset A_2$ and $T_1\beta\subset E_2\subset
C_2$ (Corollary \ref{co:EFinC}).

	\smallskip
	$\ \ \diamond\ $Assume $T_1\beta=0$. There is $\beta''\in G_2$ such that $\beta=T_1
	\beta''$ (Lemma \ref{le:GnTi}).
	Since $\beta\in E_2\cap D_2$, we have $\beta''\in E_2\cap D_2
	\subset C_2$, hence also $\beta''\in F_2$. As a consequence,
	$\beta''T_1\in F_2\subset C_2$. We deduce that
	$\alpha\ast\beta\sim' \alpha\ast (\beta'' T_1)$.
	We have $L(\alpha_{|\beta((-2,-1))})\neq\emptyset$ and
	$L((\beta'' T_1)_{|(-2,-1)})\neq\emptyset$, hence
	$(\alpha\ast (\beta'' T_1))_{|(-2,-1)}=0$ and
	$\alpha\ast (\beta'' T_1)=0$. We have 
	$\sigma T_i=\alpha\ast (T_1\beta'')\ast\gamma\sim'
	\alpha\ast(\beta'' T_1)\ast\gamma=0$.
	Since $T_i\sigma T_i=0$ and $\sigma T_i\neq 0$,
	it follows that $L((\sigma T_i)_{|(\sigma T_i)^{-1}(\{i,i+1\})})\neq
	\emptyset$, by
	applying Lemma \ref{le:GnTi} to $Z^{\opp}$.
	Since $\sigma T_i\in A_n$, we deduce that 
	$L(\sigma_{|\sigma^{-1}(\{i,i+1\})})\neq\emptyset$,
	hence $T_i\sigma=0\sim' \sigma T_i$ (using Lemma \ref{le:GnTi} for $Z^{\opp}$
	again).

	\smallskip
	$\ \ \diamond\ $Assume now $T_1\beta\neq 0$. It follows that $\beta\in F_2$,
	hence $\beta T_1\in F_2\subset C_2$.

	There are $\alpha^1,\ldots,\alpha^{i-1}\in C_1$
	with $\alpha=\alpha^{i-1}\ast\cdots\ast\alpha^1$.
	Let $s_i=\beta(-i)$ for $i\in\{1,2\}$.
	Consider $j\ge 1$ minimal such that
	$L((\alpha^j\ast\cdots\ast\alpha^1)_{|\{s_1,s_2\}})\neq\emptyset$.

	Define
	$u'=\alpha^j\boxtimes ([l\to l+1])_{1\le l\le j+1}$ and
	$u''=(\alpha^{j-1}\ast\cdots\alpha^1\ast\beta)
\boxtimes [-j-2\to -1]$.

	Let $\zeta=u''_{-2}\circ[-1\to -2]\circ (u''_{-1})^{-1}$. Define
	$I$ and $J$ to be the domain and codomain of $u''$, intersected 
	with $M$.
	Note that $\zeta(0),\zeta(1)\in M$. Let	
	$v'=(u')^\zeta=(\alpha^j)^{\zeta}\boxtimes ([l\to l+1])_{1\le l\le j+1}$ and define $v'':I\sqcup (-j-2,-1)\to
	J\sqcup (1,2)\sqcup \{-1\}\sqcup (1,j+1)$ by
	$$v''_s=\begin{cases}
		u''_{-2}\circ [-1\to -2] & \text { if } s=-1 \\
		u''_{-1}\circ [-2\to -1] & \text { if } s=-2 \\
		u''_s & \text{ otherwise.}
	\end{cases}$$
Lemma \ref{le:middlebraids} shows that $v'$ and $v''$ are braids and
$\alpha^j\ast\cdots\ast\alpha^1\ast\beta=u'\cdot u''=v'\cdot v''$. We
 have $v''=(\alpha^{j-1}\ast\cdots\alpha^1\ast(\beta T_1))
\boxtimes [-j-2\mapsto -1]$ and we deduce that
$\alpha\ast\beta=\alpha'\ast (\beta T_1)$, where
$\alpha'=\alpha^{i-1}\ast\cdots\ast\alpha^{j+1}\ast
(\alpha^j)^{\zeta}\ast\alpha^{j-1}\cdots\ast\alpha^1\in C_{i-1}$.
We have $\sigma T_i=\alpha'\ast (\beta T_1)\ast\gamma\sim'
\alpha'\ast (T_1\beta)\ast\gamma=T_i\sigma$. This completes the proof
of the first statement of the lemma.

The second statement of the lemma follows from the first one applied to
$Z^\opp$ thanks to Remark \ref{re:oppositebimodules}.
\end{proof}

\begin{prop}
	\label{le:sim'=sim}
	Let $\alpha,\beta\in C_n$. We have $\alpha\sim'\beta$ if and
	only if $\alpha\sim\beta$.
\end{prop}

\begin{proof}
	It is clear that $\alpha\sim'\beta$ implies $\alpha\sim\beta$.
The converse follows from Lemma \ref{le:divisiblemiddle}.
	\end{proof}

\begin{cor}
	\label{co:quotientC}
	We have $C/\!\sim'\ =G/\!\sim$.
\end{cor}

\begin{proof}
	The surjectivity of $C/\!\sim'\ \to G/\!\sim\ $ is given by Lemma \ref{le:qsurj}.
	The injectivity follows from Lemmas \ref{le:sim=q} and
	\ref{le:sim'=sim}.
\end{proof}

\subsubsection{Large enough $M$}
\label{se:largeM}
We assume in \S \ref{se:largeM}
that $(\xi_1^+)^{-1}(M)$ has no maximum and $(\xi_2^-)^{-1}(M)$ has no minimum.
Fix an increasing sequence $(m^+_0,m^+_1,\ldots)$ of points of $(\xi_1^+)^{-1}(M)$ and 
a decreasing sequence $(m^-_0,m^-_1,\ldots)$ of points of $(\xi_2^-)^{-1}(M)$ such that
$\lim_i m^+_i>t$ for all $t\in(\xi_1^+)^{-1}(M)$ and
$\lim_i m^-_i<t$ for all $t\in(\xi_2^-)^{-1}(M)$.


\smallskip
\begin{lemma}
	We have a canonical isomorphism $LR_{\xi_2^-}^\bullet\iso G_1$.
\end{lemma}

\begin{proof}
Using (\ref{eq:Lascolim}) and (\ref{eq:Rascolim}), we have isomorphisms
\begin{multline*}
	L_{\xi_1^+}^\bullet(T,-)\wedge R_{\xi_2^-}^\bullet(-,S) \iso \\
	\colim_{r,s\to\infty}\Hom_{\CS^\bullet(Z)}(-,T\sqcup\{\xi_1^+(m^+_r)\})\wedge\Hom_{\CS^\bullet(Z)}(S\sqcup\{\xi_2^-(m^-_s)\},-)\\
	\iso
	\colim_{r,s\to\infty}\Hom_{\CS^\bullet(Z)}(S\sqcup\{\xi_2^-(m^-_s)\},T\sqcup\{\xi_1^+(m^+_r)\})\\
	\iso
	\Hom_{\CS^\bullet(Z)}(S\sqcup\{\xi_2^-(-1)\},T\sqcup\{\xi_1^+(1)\}).
\end{multline*}
and the lemma follows.
\end{proof}

	Let us define $\lambda:R_{\xi_2^-}^\bullet L\to LR_{\xi_2^-}^\bullet$ as the composition of the injective
	map $\mu_1:R_{\xi_2^-}^\bullet L\to G_1$ (cf Lemma \ref{le:nu})
	with the inverse of the isomorphism of the lemma above.

	\smallskip
	Under the assumptions above, we have a simpler version of Theorem 
	\ref{th:gluing}.

	\begin{thm}
	\label{th:gluing2}
	The functor $\tilde{\Xi}$ factors through $\Delta'_{\lambda}
	\CS^\bullet_{M}(Z)$
and induces an isomorphism of differential pointed categories
	$\Delta'_{\lambda}\CS^\bullet_{M}(Z)\iso \CS^\bullet_M(Z_\xi)$.\indexnot{Xi}{\Xi}
\end{thm}

	\begin{proof}
		Every element of $R_{\xi_2}^-(S,T)$ is of the form
		$(\id_T\boxtimes\zeta)\cdot(\alpha\boxtimes\id_{-1})$ for some
		$\zeta$ admissible class of paths starting at $-1$ and
		$\alpha$ a braid starting at $S$.
		
		Every element of $L_{\xi_1}^+(S,T)$ is of the form
		$([m_i^+\to 1]\boxtimes \id_T)\cdot \alpha$ for some
		braid $\alpha$ starting at $S$.

		It follows that every element of $(R_{\xi_2}^-L_{\xi_1}^+)^n$
		is of the form 
		$$(\id\boxtimes\zeta_1)\wedge ([m_i^+\to 1]\boxtimes\id)
		\wedge\cdots\wedge
		(\id\boxtimes\zeta_{n-1})\wedge ([m_{i+n-2}^+\to 1]\boxtimes\id)\wedge
		(\id\boxtimes\zeta_n)\wedge \alpha$$
		for some $i\ge 0$ and $\zeta_r$ an admissible class of paths
		starting at $-1$ for $1\le r\le n$.
		The image by $\mu_n$ of such an element is
		$$\bigl(([m_{i+r-1}^+\to r])_{1\le r\le n-1}\boxtimes\id\bigr)\circ
		\alpha)\boxtimes
		\zeta_1\boxtimes (\zeta_2\circ [-2\to -1])
		\boxtimes\cdots\boxtimes (\zeta_n\circ [-n\to -1]).$$
		It follows that $\mu_n$ is injective, hence
		it induces an isomorphism $(R_{\xi_2}^-L_{\xi_1}^+)^n\iso C_n$.

		\smallskip
		Let $L$ be the image of $\lambda\circ (T_1\otimes 1-1\otimes T_1)$.
		We have $\nu_2=\mu_2\circ\lambda$. It follows that
		$\mu_2(L)=(T_1\otimes 1-1\otimes T_1)(D_2)$,
		since $D_2$ is the image of $\nu_2$ (Lemma \ref{le:nu}).
		The theorem follows now from Corollary \ref{co:quotientC} and
		Theorem \ref{th:gluing}.
	\end{proof}

\begin{rem}
	Consider $Z$ the singular curve quotient of oriented $\BR$ by the identification of
two points. Take $M$ to be the single exceptional point of $Z$. The construction above 
	applied to $\CS_M^\bullet(Z)$ gives a category where going twice around the circle,
	avoiding the loop, is non-zero (cf picture below),
	while it is not represented by a smooth path in
	$Z_\xi$. Theorem \ref{th:gluing2} does not hold because $M$ is too small.
$$\includegraphics[scale=1.2]{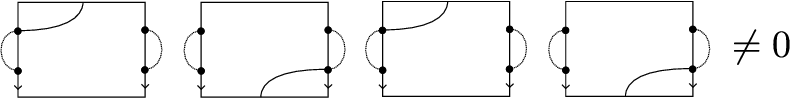}$$
\end{rem}

\subsubsection{Functoriality}
\label{se:functorialityDelta}
Consider $f:Z\to Z'$ a morphism of curves and
assume $f\circ\xi_1^+$ is outgoing for $Z'$ and $f\circ\xi_2^-$ is incoming for $Z'$.
 
The morphism $f$ extends uniquely to a morphism of curves $f:Z_\xi\to Z'_{f\circ\xi}$.

\smallskip
The functor $f:\CS^\bullet_{f,M}(Z)\to \CS^\bullet_{f(M)}(Z')$ can be equipped with a structure of morphism
of $2$-representations $L_{\xi_1^+}^\bullet\to L_{f\circ\xi_1^+}^\bullet$ and
of morphism of $2$-representations $R_{\xi_2^-}^\bullet\to R_{f\circ\xi_2^-}^\bullet$
(Lemma \ref{le:2repcurvesf} and \S\ref{se:rightaction}),
and it induces a differential pointed functor
(cf \S\ref{se:functorialitytensor})
$$\Delta f:\Delta_{R_{\xi_2^-},L_{\xi_1^+},\lambda}\CS^\bullet_{f,M}(Z)\to
\Delta_{R_{f\circ \xi_2^-},L_{f\circ \xi_1^+},\lambda'}\CS^\bullet_{f(M)}(Z'),$$
where $\lambda'$ is the analog of the map $\lambda$ for $Z'$.

We obtain a commutative diagram of differential pointed functors

\begin{equation}
	\label{eq:fDelat}
\xymatrix{
	\Delta_{R_{\xi_2^-},L_{\xi_1^+},\lambda}\CS^\bullet_{f,M}(Z)\ar[r]^-{\Xi}\ar[d]_{\Delta f} &
	\CS^\bullet_{f,M}(Z_\xi)
	\ar[d]^f\\
\Delta_{R_{f\circ \xi_2^-},L_{f\circ \xi_1^+},\lambda'}\CS^\bullet_{f(M)}(Z')\ar[r]_-{\Xi} &
\CS^\bullet_{f(M)}(Z'_{f\circ\xi}) \\
}
\end{equation}

Assume $f_{|Z}$ is strict. It follows that $f$ is strict.
The functor $f^\#:\add(\CS_{f(M)}(Z'))\to \add(\CS_{M}(Z))$ can be equipped with a structure of morphism
of $2$-representations $L_{f\circ\xi_1^+}\to L_{\xi_1^+}$ and
of morphism of $2$-representations $R_{f\circ\xi_2^-}\to R_{\xi_2^-}$
(Lemma \ref{le:2repcurvesfsharp} and \S\ref{se:rightaction}),
and it induces a differential functor
(cf \S\ref{se:functorialitytensor})
$$\Delta f^\#:\Delta_{R_{f\circ \xi_2^-},L_{f\circ \xi_1^+},\lambda'}\add(\CS_{f(M)}(Z'))\to
\Delta_{R_{\xi_2^-},L_{\xi_1^+},\lambda}\add(\CS_{M}(Z)).$$

We obtain a commutative diagram of differential functors
commuting with coproducts
\begin{equation}
	\label{eq:fsharpDelat}
\xymatrix{
	\Delta_{R_{\xi_2^-},L_{\xi_1^+},\lambda}\add(\CS_{M}(Z))\ar[r]^-\Xi & \add(\CS_{M}(Z_\xi)) \\
\Delta_{R_{f\circ \xi_2^-},L_{f\circ \xi_1^+},\lambda'}\add(\CS_{f(M)}(Z'))\ar[r]_-\Xi
\ar[u]^{\Delta f^\#} &
\add(\CS_{f(M)}(Z'_{f\circ\xi})))\ar[u]_{f^\#} \\
}
\end{equation}

\subsection{Diagonal action}
\label{se:actiontensor}
\subsubsection{Isomorphism Theorem}

Let $Z'=\BR$ be the smooth curve
with $Z'_o=(-\frac{1}{2},\frac{1}{2})$ with its standard orientation.
	Fix an increasing homeomorphism $\alpha:\BR_{>0}\iso\BR_{>\frac{1}{2}}$ fixing
	the positive integers and define $\alpha':\BR_{<0}\iso\BR_{<-\frac{1}{2}}$ by
	$\alpha'(t)=-\alpha(-t)$.

Assume $Z(\xi_1^+)\neq Z(\xi_2^-)$ and assume
there is a morphism $\tilde{\xi}_1:Z'\to Z$ with image $Z(\xi_1^+)$
and such that $\xi_1^+=\tilde{\xi}_1\circ\alpha$. Put
	$\xi_1^-=\tilde{\xi}_1\circ\alpha':\BR_{<0}\to Z$ and
	denote by $\xi^-$ the composition $\BR_{<0}\xrightarrow{\xi_1^-}Z\hookrightarrow
	Z_{\xi_1}$.

\smallskip
	Proposition \ref{pr:Ldual} gives an isomorphism of differential pointed bimodules
	$\hat{\kappa}_1:L_{\xi_1^+}(-_2,-_1)\iso R_{\xi_1^-}(-_1,-_2)^\vee$.

	\smallskip
	Since there is no admissible path from $\xi_2^-(-1)$ to $\xi_1^+(1)$
in $Z$, we have $A_n=D_n=G_n$ (with the notations of \S\ref{se:Gbimodules}),
hence we have an isomorphism (Lemma \ref{le:nu})
$$\nu_n:R_{\xi_2^-}^\bullet(T,-,e^n)\wedge L_{\xi_1^+}^\bullet(-,S,e^n) \iso
\Hom_{\CS^\bullet(Z)}(S\sqcup\{\xi_2^-(-n),\ldots,\xi_2^-(-1)\},T\sqcup\{\xi_1^+(1),
\ldots,\xi_1^+(n)\}).$$

Consider 
\begin{align*}
	\lambda:L_{\xi_1^+}^\bullet(T,-) R_{\xi_2^-}^\bullet(-,S)&\to
	R_{\xi_2^-}^\bullet(T,-) L_{\xi_1^+}^\bullet(-,S)\\
	\alpha\wedge\beta&\mapsto \nu_1^{-1}(\alpha\cdot\beta)=((\alpha\cdot\beta)_{\xi_2^-(-1)}\boxtimes
	\id_{T\setminus\{\chi(\alpha\circ\beta)(\xi_2^-(-1))\}})\wedge(\alpha\cdot\beta)_{|S}.
\end{align*}

Since $\nu_n$ is an isomorphism, the morphisms (\ref{eq:diagonalpower}) are
isomorphisms (cf proof of Theorem \ref{th:gluing}) and we obtain from 
Remark \ref{re:lambdatobirep}
an isomorphism of differential pointed categories
$$\Delta_E\CS_M^\bullet(Z)\iso\Delta_{\lambda}\CS_M^\bullet(Z).$$
Composing its inverse with $\Xi$, we deduce from Theorem \ref{th:gluing}
an isomorphism of differential pointed categories
$$\Xi':\Delta_\lambda\CS_M^\bullet(Z)\iso\CS_M^\bullet(Z_\xi).$$

	\begin{thm}
		\label{th:iso2rep}
		The isomorphism $\Xi'$ provides an isomorphism of $2$-representations,
		where $\Delta_\lambda\CS_M^\bullet(Z)$
		is equipped with the diagonal action and $\CS_M^\bullet(Z_\xi)$ with
		the action of $R_{\xi^-}$.
	\end{thm}

	The remainder of \S\ref{se:actiontensor} is devoted to the proof of Theorem \ref{th:iso2rep}.

	\subsubsection{Setting}
Let $\sigma:R_{\xi_2^-}(T,-)\otimes R_{\xi_1^-}(-,S)\to
R_{\xi_1^-}(T,-)\otimes R_{\xi_2^-}(-,S)$ be defined as in (\ref{eq:defsigma}).

		\begin{lemma}
			The morphism $\sigma$ is invertible. Given
$\alpha\in R_{\xi_2^-}^\bullet(T,U)$ and $\beta\in R_{\xi_1^-}^\bullet(U,S)$, we have
			$$\sigma(\alpha\otimes\beta)=
			\delta_{\alpha_{|U}\cdot\beta\neq 0}
\bigl(\id\boxtimes
	(\alpha_{\bij{\beta}(\xi_1^-(-1))}\cdot\beta_{\xi_1^-(-1)})\bigr)\otimes
	\bigl(\alpha_{\xi_2^-(-1)}\boxtimes(\alpha_{|U\setminus\{\bij{\beta}(\xi_1^-(-1))\}}
	\cdot\beta_{|S})\bigr).$$
			Given $\alpha'\in R_{\xi_1^-}^\bullet(T,U')$ and $\beta'\in
		R_{\xi_2^-}^\bullet(U',S)$, we have
			$$\sigma^{-1}(\alpha'\otimes\beta')=
			\delta_{\alpha'_{|U'}\cdot\beta'\neq 0}
			\bigl(\id\boxtimes(\alpha'_{\bij{\beta'}(\xi_2^-(-1))}\cdot\beta'_{\xi_2^-(-1)})
			\bigr)\otimes\bigl(\alpha'_{\xi_1^-(-1)}\boxtimes(
			\alpha'_{|U'\setminus\bij{\beta'}(\xi_2^-(-1))}\circ\beta'_{|S})\bigr).
			$$
		\end{lemma}

\begin{proof}
%
	We have
	$$\sigma=(R_{\xi_1^-}\circ \textrm{mult})\circ
	(R_{\xi_1^-}\otimes R_{\xi_2^-}\otimes \eps_{L_{\xi_1^+},R_{\xi_1^-}})\circ(R_{\xi_1^-}\otimes
	\lambda\otimes R_{\xi_1^-})\circ (\eta_{L_{\xi_1^+},R_{\xi_1^-}}\otimes\id).$$
	We have 
	$\alpha=(\id\boxtimes\alpha_{\xi_2^-(-1)})\cdot(\alpha_{|U}\boxtimes\id)$, hence
	$\alpha\otimes\beta=(\id\boxtimes\alpha_{\xi_2^-(-1)})
	\otimes(\alpha_{|U}\cdot\beta)$. As a consequence, it is enough to prove the
	first statement of the lemma assuming that $\alpha_{|U}=\id_U$.
In that case, the composition above is given by
\begin{align*}
	\alpha\otimes\beta&\mapsto \sum_{x\in\tilde{\xi}_1^{-1}(T)}
	(\id_{T\setminus\{\tilde{\xi}_1(x)\}}\boxtimes\tilde{\xi}_1([-1\to x]))\otimes
	(\id_{T\setminus\{\tilde{\xi}_1(x)\}}\boxtimes\tilde{\xi}_1([x\to 1]))\otimes
\alpha\otimes\beta \\
	&\mapsto\sum_{x\in\tilde{\xi}_1^{-1}(T)}
	(\id_{T\setminus\{\tilde{\xi}_1(x)\}}\boxtimes\tilde{\xi}_1([-1\to x]))\otimes
	(\alpha_{\xi_2^-(-1)}\boxtimes \id)\otimes
	(\id\boxtimes\tilde{\xi}_1([x\to 1]))\otimes\beta\\
	&\mapsto(\id\boxtimes
	\beta_{\xi_1^-(-1)})\otimes
	(\alpha_{\xi_2^-(-1)}\boxtimes\id)\otimes \beta_{|S}\\
	&\mapsto(\id\boxtimes \beta_{\xi_1^-(-1)})\otimes
	(\alpha_{\xi_2^-(-1)}\boxtimes\beta_{|S}).
\end{align*}
	It is immediate to check that the formula for $\sigma^{-1}$ does produce an inverse.
\end{proof}

Consider the map $\rho:L_{\xi_1^+}(T,-)\otimes R_{\xi_1^-}(-,S)\to R_{\xi_1^-}(T,-)\otimes
L_{\xi_1}^-(-,S)$ defined in \S\ref{eq:defrho}.

\begin{lemma}
	\label{le:calculationrho}
	Given $\alpha\in L_{\xi_1^+}^\bullet(T,U)$ and $\beta\in R_{\xi_1^-}^\bullet(U,S)$,
we have
	$$\rho(\alpha\otimes\beta)=\delta_1 \bigr(\alpha_{|U\setminus\chi(\alpha)^{-1}(\xi_1^+(1))}
	\cdot (\beta_{\xi_1^-(-1)}\boxtimes\id)\bigl)\otimes \bigl((\alpha_{\chi(\alpha)^{-1}(\xi_1^+(1))}\boxtimes\id)\cdot
	\beta_{|S}\bigr)$$
	where
	$\delta_1=1$ if $\chi(\alpha\circ\beta)(\xi_1^-(-1))\neq \xi_1^+(1)$ and
	$(\id_{\chi(\beta)(\xi_1^-(-1))}\boxtimes \alpha_{\chi(\alpha)^{-1}(\xi_1^+(1))})\cdot
	(\beta_{\xi_1^-(-1)}\boxtimes\id_{\chi(\alpha)^{-1}(\xi_1^+(1))})\neq 0$
	and $\delta_1=0$ otherwise.
\end{lemma}

\begin{proof}
	Assume first $\alpha_{|U\setminus\{\chi(\alpha)^{-1}(\xi_1^+(1))\}}=\id$ and
	$\beta_{|S}=\id$.
We have
	\begin{align*}
		\rho(\alpha\otimes\beta)&=\eps_1 R_{\xi_1^-}L_{\xi_1^+}\circ L_{\xi_1^+}\tau L_{\xi_1^+}
	(\alpha\otimes\beta\otimes\eta_1(\id_S))\\
		&=\eps_1 R_{\xi_1^-}L_{\xi_1^+}\biggl(\sum_{x\in\tilde{\xi}_1^{-1}(S)}\alpha\otimes\tau\Bigl((\beta_{\xi_1^-(-1)}\boxtimes\id)
	\otimes \bigl(\id\boxtimes
	\tilde{\xi}_1([-1\to x])\bigr)\Bigr)\otimes \bigl(\tilde{\xi}_1([x\to 1])\boxtimes\id\bigr)\biggr)\\
		&=\eps_1 R_{\xi_1^-}L_{\xi_1^+}\biggl(\sum_{x\in I} \alpha\otimes
	\bigl(\id\boxtimes \tilde{\xi}_1([-1\to x])\bigr)\otimes (\beta_{\xi_1^-(-1)}\boxtimes\id)
	\otimes \bigl(\tilde{\xi}_1([x\to 1])\boxtimes\id\bigr)\biggr)\\
		&=\delta_1 (\beta_{\xi_1^-(-1)}\boxtimes\id)
	\otimes(\alpha_{\chi(\alpha)^{-1}(\xi_1^+(1))}\boxtimes\id)
	\end{align*}
	where $I=\{x\in\tilde{\xi}_1^{-1}(S)\ |\ 
	(\tilde{\xi}_1([-1\to x])\boxtimes(\beta_{\xi_1^-(-1)}\circ [\xi_1^-(-2)\to\xi_1^-(-1)])\cdot
	\tau \neq 0\}$.

	Since $\rho$ is a morphism of $(\CS_M(Z),\CS_M(Z))$-bimodules, the general
	result follows using the decompositions
	$\alpha=(\alpha_{|U\setminus\{\chi(\alpha)^{-1}(\xi_1^+(1))\}}\boxtimes
	\id_{\xi_1^+(1)})\cdot(\id\boxtimes\alpha_{\chi(\alpha)^{-1}(\xi_1^+(1))})$ and
	$\beta=(\id\boxtimes\beta_{\xi_1^-(-1)})\cdot (\beta_{|S}\boxtimes\id_{\xi_1^-(-1)})$.
\end{proof}

\subsubsection{Diagonal bimodule}
		Recall that we have a $(\Delta_\lambda\CS_M(Z),\Delta_\lambda\CS_M(Z))$-bimodule
		$E$. Its restriction to a $(\CS_M(Z),\Delta_\lambda\CS_M(Z))$-bimodule is the
		cone of $\pi:R_{\xi_2^-}\otimes_{\CS_M(Z)}\Id_{\Delta_\lambda\CS_M(Z)}\to
R_{\xi_1^-}\otimes_{\CS_M(Z)}\Id_{\Delta_\lambda\CS_M(Z)}$.

		The $(\CS_M(Z),\CS_M(Z_\xi))$-bimodule $E'=E\circ (1\otimes\Xi^{\prime -1})$ is the
		cone of the map $u$ defined as follows.

%

		Given $\alpha\in R_{\xi_2^-}^\bullet(T,U)$ with $\alpha_{|U}=\id$ and given $\beta\in
		\Hom_{\CS^\bullet(Z_\xi)}(S,U)$, we have
\begin{multline*}
u(\alpha\otimes\beta)=
		\sum_{x\in\tilde{\xi}_1^{-1}(T)}\bigl(\id
		\boxtimes \tilde{\xi}_1([-1\to x])\bigr)\otimes
		\Bigl(\bigl(\alpha_{\xi_2^-(-1)}\cdot [\xi_1^+(1)\to\xi_2^-(-1)]\cdot\tilde{\xi}_1([x\to 1])
		\bigr) \boxtimes\id\Bigr)\cdot\beta.
\end{multline*}

		\medskip
		We construct now an isomorphism between $E'$ and the restriction of
		$R_{\xi^-}$ to a $(\CS_M(Z),\CS_M(Z_\xi))$-bimodule.

		We define two morphisms of pointed sets

		\begin{align*}
			f_1:R_{\xi_1^-}^\bullet(T,-)\wedge\Hom_{\CS^\bullet(Z_\xi)}(S,-)&\to
		R_{\xi^-}^\bullet(T,S)\\
			(\alpha:U\sqcup\{\xi_1^-(-1)\}\to T)\wedge(\beta:S\to U)&\mapsto
			\alpha\cdot(\beta\boxtimes\id_{\xi_1^-(-1)})=
			(\alpha_{|U}\cdot\beta)\boxtimes\alpha_{\xi_1^-(-1)}
		\end{align*}
and
		\begin{align*}
			f_2:R_{\xi_2^-}^\bullet(T,-)\wedge\Hom_{\CS^\bullet(Z_\xi)}(S,-)&\to
		R_{\xi^-}^\bullet(T,S)\\
			(\alpha:U\sqcup\{\xi_2^-(-1)\}\to T)\wedge(\beta:S\to U)&\mapsto
			\alpha\cdot(\beta\boxtimes[\xi_1^-(-1)\to \xi_2^-(-1)])\\
			&\ \ \ \ =
			(\alpha_{|U}\cdot\beta)\boxtimes
			(\alpha_{\xi_2^-(-1)}\cdot[\xi_1^-(-1)\to \xi_2^-(-1)]).
		\end{align*}
		Note that we have an isomorphism of pointed sets
		$$f_2\vee f_1:\bigl(R_{\xi_2^-}^\bullet(T,-)\wedge
	\Hom_{\CS^\bullet(Z_\xi)}(S,-)\bigr)
	\vee \bigl( (R_{\xi_1^-}^\bullet(T,-)\wedge\Hom_{\CS^\bullet(Z_\xi)}(S,-)\bigr)\iso
R_{\xi^-}^\bullet(T,S).$$

		\begin{lemma}
			We have $d(f_1)=0$ and 
			$d(f_2)=f_1\circ u$.
	There is an isomorphism of differential modules
	$$(f_2,f_1):E'(T,S)\iso R_{\xi^-}(T,S)$$
	functorial in $T\in\CS_M(Z)$ and $S\in\CS_M(Z_\xi)$.
		\end{lemma}

		\begin{proof}
		It is immediate that $d(f_1)=0$. For the second equality, 
consider $\alpha\in R^\bullet_{\xi_2^-}(T,U)$ and $\beta\in\Hom_{\CS^\bullet(Z_\xi)}(S,U)$.
			Since
			$\alpha\otimes\beta=(\id\boxtimes\alpha_{\xi_2^-(-1)})\otimes
			(\alpha_{|U}\cdot\beta)$, we can assume that $\alpha_{|U}=\id$. We have
		$$d(f_2)(\alpha\otimes\beta)=
		\alpha\cdot \bigl(d(\beta\boxtimes[\xi_1^-(-1)\to \xi_2^-(-1)])+d(\beta)\boxtimes[\xi_1^-(-1)\to \xi_2^-(-1)]
		\bigr).$$
		We have
			\begin{align*}
				d(\beta\boxtimes[\xi_1^-(-1)\to \xi_2^-(-1)])&=
				d\bigl((\id\boxtimes[\xi_1^-(-1)\to \xi_2^-(-1)])\cdot
		(\beta\boxtimes\id_{\xi_1^-(-1)})\bigr)\\
				&=(\id\boxtimes[\xi_1^+(1)\to \xi_2^-(-1)])\cdot
				d(\id_U\boxtimes[\xi_1^-(-1)\to \xi_1^+(1)])\cdot(\beta\boxtimes\id_{\xi_1^-(-1)})\\
				&\ \ \ \ +
		(\id\boxtimes[\xi_1^-(-1)\to \xi_2^-(-1)])\cdot(d(\beta)\boxtimes\id_{\xi_1^-(-1)})\\
				&=(\id_U\boxtimes[\xi_1^+(1)\to \xi_2^-(-1)])
				\sum_{x\in\tilde{\xi}_1^{-1}(U)}
				\bigl(\id\boxtimes\tilde{\xi}_1([x\to 1])\boxtimes
				\tilde{\xi}_1([-1\to x])\bigr)\\
				& \ \ \ \ \cdot(\beta\boxtimes\id_{\xi_1^-(-1)})+
			d(\beta)\boxtimes [\xi_1^-(-1)\to \xi_2^-(-1)]
			\end{align*}
hence
$$				d(f_2)(\alpha\otimes\beta)=
				\alpha\cdot \sum_{x\in\tilde{\xi}_1^{-1}(U)}
				\Bigl(\id\boxtimes\bigl([\xi_1^+(1)\to \xi_2^-(-1)]\cdot\tilde{\xi}_1([x\to 1])\bigr)\boxtimes
				\tilde{\xi}_1([-1\to x])\Bigr)\cdot (\beta\boxtimes\id_{\xi_1^-(-1)})$$
			\begin{align*}
				&=\sum_{x\in\tilde{\xi}_1^{-1}(U)}\Bigl(
				\id\boxtimes \tilde{\xi}_1([-1\to x])
				\boxtimes\bigl(\alpha_{\xi_2^-(-1)}\cdot[\xi_1^+(1)\to \xi_2^-(-1)]\cdot\tilde{\xi}_1([x\to 1])\bigr)\Bigr)
				\cdot (\beta\boxtimes\id_{\xi_1^-(-1)})\\
				&=f_1\circ u(\alpha\otimes\beta).
			\end{align*}
	The lemma follows.
	\end{proof}

	\subsubsection{Matching of extended action}
	Recall that $E'$ is the restriction of the $(\CS_M(Z_\xi),\CS_M(Z_\xi))$-bimodule
	$E\circ (\Xi^{\prime -1}\otimes\Xi^{\prime -1})$.

We show here that the previous isomorphism is functorial in
	$T\in\CS_M(Z_\xi)$. Consider the diagram
\begin{equation}
	\label{eq:commdiag}
\xymatrix{
	R_{\xi_2^-}(T,-)\otimes L_{\xi_1^+}(-,U)\otimes E(U,S) \ar[r]^-w \ar[d]_-{\Xi\otimes (f_2,f_1)} &
	E(T,S)\ar[d]^{(f_2,f_1)} \\
	\Hom_{\CS(Z_\xi)}(T,U)\otimes R_{\xi^-}(U,S)\ar[r]_-{\mathrm{action}} &
	R_{\xi^-}(T,S)
	}
\end{equation}
where $w=\left(\begin{matrix}w_{11}&w_{12}\\0&w_{22}\end{matrix}\right)$
	(cf \S\ref{se:leftdualbimod}) with
$$w_{11}=
	(R_{\xi_2^-}(\mathrm{mult}\circ \Xi\Hom))\circ(\tau L_{\xi_1^+} \Hom)\circ(R_{\xi_2^-}\lambda\Hom)$$
$$w_{12}=R_{\xi_2^-}\eps \Hom$$
$$w_{22}=(R_{\xi_1^-}(\mathrm{mult}\circ \Xi\Hom))\circ 
	(\sigma L_{\xi_1^+}\Hom)\circ(R_{\xi_2^-}\rho\Hom).$$

\begin{lemma}
	The diagram (\ref{eq:commdiag}) is commutative.
\end{lemma}

\begin{proof}
	Note first that all the maps of the diagram are functorial with respect to
	$S\in\CS_M(Z_\xi)$. 

	\smallskip
	$\bullet\ $Let $\gamma\in R_{\xi_2^-}(U,S)$, $\beta\in L_{\xi_1^+}(V,U)$ and $\alpha\in R_{\xi_2^-}(T,V)$.
We will show that	
	
	\begin{equation}
		\label{eq:comm1}
	\mathrm{action}\circ (\Xi\otimes f_2)(\alpha\otimes\beta\otimes\gamma)=
	f_2\circ R_{\xi_2^-}(\mathrm{mult}\circ\Xi)\circ\tau L_{\xi_1^+}\circ R_{\xi_2^-}\lambda
	(\alpha\otimes\beta\otimes\gamma).
	\end{equation}

	Since $\gamma=(\id\boxtimes\gamma_{\xi_2^-(-1)})\cdot \gamma_{|S}$ and since
	$\mathrm{action}\circ (\Xi\otimes f_2)$ and $f_2\circ R_{\xi_2^-}(\mathrm{mult}\circ\Xi)\circ\tau L_{\xi_1^+}\circ R_{\xi_2^-}\lambda$ are morphisms of $\CS_M(Z)^\opp$-modules, we can assume $\gamma_{|S}=\id$.
	We have $\alpha\otimes\beta=(\id\boxtimes\alpha_{\xi_2^-(-1)})\otimes (\alpha_{|V}
	\boxtimes\id_{\xi_1^+(1)}\cdot \beta)$, hence we can assume $\alpha_{|V}=\id$.
	We can also assume that $\beta\otimes\gamma\neq 0$.

	\smallskip
	We have
	$$\mathrm{action}\circ (\Xi\otimes f_2)(\alpha\otimes\beta\otimes\gamma)=
	\bigl(\id_V\boxtimes (\alpha_{\xi_2^-(-1)}\cdot [\xi_1^+(1)\to\xi_2^-(-1)])\bigr)\cdot\beta$$
	$$\cdot \bigl(\id_S\boxtimes (\gamma_{\xi_2^-(-1)}\cdot [\xi_1^-(-1)\to\xi_2^-(-1)])\bigr)$$
	$$=\delta_1\beta_{|S\setminus\chi(\beta)^{-1}(\xi_1^+(1))}\boxtimes 
	(\alpha_{\xi_2^-(-1)}\cdot [\xi_1^+(1)\to\xi_2^-(-1)]\cdot\beta_{\chi(\beta)^{-1}(\xi_1^+(1))})$$
	$$\boxtimes\bigl((\beta\circ\gamma)_{\xi_2^-(-1)}\cdot [\xi_1^-(-1)\to\xi_2^-(-1)]\bigr)$$
where
	$\delta_1=\delta_{i(\alpha_{\xi_2^-(-1)},(\beta\circ\gamma)_{\xi_2^-(-1)})=0}$. On the other
	hand, we have
	$$f_2\circ R_{\xi_2^-}(\mathrm{mult}\circ\Xi)\circ\tau L_{\xi_1^+}\circ R_{\xi_2^-}\lambda
	(\alpha\otimes\beta\otimes\gamma)=$$
	\begin{align*}
		&=f_2\circ R_{\xi_2^-}(\mathrm{mult}\circ\Xi)\circ\tau L_{\xi_1^+}\Bigl(
	\alpha\otimes \bigl((\beta_{\chi(\gamma)(\xi_2^-(-1))}\cdot\gamma_{\xi_2^-(-1)})\boxtimes\id\bigr)
	\otimes\beta_{|S}\Bigr)\\
		&=\delta_1 f_2\circ R_{\xi_2^-}(\mathrm{mult}\circ\Xi)\Bigl(
	\bigl(\id\boxtimes (\beta_{\chi(\gamma)(\xi_2^-(-1))}\cdot\gamma_{\xi_2^-(-1)})\bigr)\otimes
	(\id\boxtimes\alpha_{\xi_2^-(-1)})\otimes\beta_{|S}\Bigr)\\
		&=\delta_1 f_2\biggl(\bigl(\id\boxtimes (\beta_{\chi(\gamma)(\xi_2^-(-1))}\cdot\gamma_{\xi_2^-(-1)})\bigr)\otimes
	\Bigl(\bigl(\id\boxtimes (\alpha_{\xi_2^-(-1)}\cdot [\xi_1^+(1)\to\xi_2^-(-1)])\bigr)\cdot\beta_{|S}
	\Bigr)\biggr)\\
		&=\delta_1 \bigl((\beta_{\chi(\gamma)(\xi_2^-(-1))}\cdot\gamma_{\xi_2^-(-1)})\cdot 
	[\xi_1^-(-1)\to\xi_2^-(-1)]\bigr)\boxtimes \Bigl(\bigl(\id\boxtimes (\alpha_{\xi_2^-(-1)}\cdot [\xi_1^+(1)\to\xi_2^-(-1)])\bigr)\cdot\beta_{|S}\Bigr)\\
		&=\mathrm{action}\circ (\Xi\otimes f_2)(\alpha\otimes\beta\otimes\gamma).
	\end{align*}

	We deduce that (\ref{eq:comm1}) holds.

\medskip
	$\bullet\ $Let $\gamma\in R_{\xi_1^-}(U,S)$, $\beta\in L_{\xi_1^+}(V,U)$ and $\alpha\in R_{\xi_2^-}(T,V)$.
We will show that	
	\begin{equation}
		\label{eq:comm2}
		\mathrm{action}\circ (\Xi\otimes f_1)(\alpha\otimes\beta\otimes\gamma)=
	\bigl(f_1\circ R_{\xi_1^-}(\mathrm{mult}\circ\Xi)\circ\sigma L_{\xi_1^+}\circ R_{\xi_2^-}\rho
	+f_2\circ R_{\xi_2^-}\eps\bigr)
	(\alpha\otimes\beta\otimes\gamma).
	\end{equation}

	As before, we can assume $\gamma_{|S}=\id$, $\alpha_{|V}=\id$ and $\beta\otimes\gamma\neq 0$.
	We put $u_1=\chi(\gamma)(\xi_1^-(-1))$ and $u_2=\chi(\beta)^{-1}(\xi_1^+(1))$.

		\smallskip
	We have

	$$\mathrm{action}\circ (\Xi\otimes f_1)(\alpha\otimes\beta\otimes\gamma)=
	\bigl(\id_V\boxtimes (\alpha_{\xi_2^-(-1)}\cdot [\xi_1^+(1)\to\xi_2^-(-1)])\bigr)\cdot\beta
	\cdot \bigl(\id_S\boxtimes \gamma_{\xi_1^-(-1)}\bigr)$$
	$$=\begin{cases}
		\delta_2(\alpha_{\xi_2^-(-1)}\cdot[\xi_1^-(1)\to\xi_2^-(-1)])
		\boxtimes\beta_{|S} & \text{ if }u_1=u_2 \\
		\delta_3 (\alpha_{\xi_2^-(-1)}\cdot[\xi_1^+(1)\to\xi_2^-(-1)]\cdot\beta_{u_2})\boxtimes
		(\beta_{u_1}\circ\gamma_{\xi_1^-(-1)}) \boxtimes\beta_{|S\setminus\{u_2\}} & \text{ otherwise}
	\end{cases}$$
where 
\begin{itemize}
	\item $\delta_2=1$ if $\gamma_{\xi_1^-(-1)}(1-)=\iota(\beta_{u_2}(0+))$ and $\delta_2=0$ otherwise
	\item $\delta_3=1$ if $\beta_{|U}\cdot(\id_{S}\boxtimes
\gamma_{\xi_1^-(-1)})\neq 0$ and $\delta_3=0$ otherwise.
\end{itemize}

\smallskip
	We have
	$$f_2\circ R_{\xi_2^-}\eps(\alpha\otimes\beta\otimes\gamma)=
	\delta_2 f_2(\alpha\otimes \beta_{|S})=
	\delta_2 (\alpha_{\xi_2^-(-1)}\cdot[\xi_1^-(1)\to\xi_2^-(-1)])\boxtimes \beta_{|S}.
	$$

\smallskip
	We have
	$$f_1\circ R_{\xi_1^-}(\mathrm{mult}\circ\Xi)\circ\sigma L_{\xi_1^+}\circ R_{\xi_2^-}\rho
	(\alpha\otimes\beta\otimes\gamma)=$$
	\begin{align*}
		&=\delta'_3f_1\circ R_{\xi_1^-}(\mathrm{mult}\circ\Xi)\circ\sigma L_{\xi_1^+}\Bigl(\alpha\otimes 
		\bigr(\beta_{|U\setminus\{u_2\}}
	\cdot (\gamma_{\xi_1^-(-1)}\boxtimes\id)\bigl)\otimes (\beta_{u_2}\boxtimes\id)\Bigr)\\
		&=\delta'_3\delta''_3f_1\circ R_{\xi_1^-}(\mathrm{mult}\circ\Xi)\circ\sigma L_{\xi_1^+}\Bigl(\alpha\otimes 
		\bigr(\beta_{|S\setminus\{u_2\}}\boxtimes(\beta_{u_1}\circ\gamma_{\xi_1^-(-1)})
	\bigl)\otimes (\beta_{u_2}\boxtimes\id)\Bigr)\\
		&=\delta'_3\delta''_3f_1\circ R_{\xi_1^-}(\mathrm{mult}\circ\Xi)\Bigl(\bigl((\beta_{u_1}\circ\gamma_{\xi_1^-(-1)})
	\boxtimes\id\bigr)\otimes (\alpha_{\xi_2^-(-1)}\boxtimes\beta_{|S\setminus\{u_2\}})
	\otimes (\beta_{u_2}\boxtimes\id)\Bigr)\\
		&=\delta'_3\delta''_3(\beta_{u_1}\circ\gamma_{\xi_1^-(-1)})
	\boxtimes(\alpha_{\xi_2^-(-1)}\cdot [\xi_1^+(1)\to\xi_2^-(-1)]\cdot\beta_{u_2})
	\boxtimes\beta_{|S\setminus\{u_2\}}
	\end{align*}
	where
	\begin{itemize}
		\item $\delta'_3=1$ if $u_1\neq u_2$ and
	$(\id_{u_1}\boxtimes\beta_{u_2})\cdot(\gamma_{\xi_1^-(-1)}\boxtimes\id_{u_2})\neq 0$
	and $\delta'_3=0$ otherwise
\item
$\delta''_3=1$ if $\beta_{|U\setminus\{u_2\}}\cdot(\id_{S\setminus\{u_2\}}\boxtimes
\gamma_{\xi_1^-(-1)})\neq 0$ and $\delta''_3=0$ otherwise.
	\end{itemize}

	Since $\delta_3=\delta'_3\delta''_3$, we deduce that (\ref{eq:comm2}) holds and the lemma follows.
\end{proof}

\subsubsection{Action of $\tau$}
The action of $\tau$ on $E(T,-)\otimes E(-,S)$ corresponds to an endomorphism
of $R_{\xi_2^-}^2\oplus R_{\xi_2^-}R_{\xi_1^-}\oplus R_{\xi_1^-}R_{\xi_2^-}\oplus
R_{\xi_1^-}^2$ given in (\ref{eq:deftau}).

\begin{lemma}
	We have $\tau\circ ((f_2,f_1)\otimes (f_2,f_1))=((f_2,f_1)\otimes(f_2,f_1))\circ\tau$.
\end{lemma}

\begin{proof}
	Consider $\alpha_i\in R_{\xi_i^-}^\bullet(T,U)$ and 
	$\beta_i\in R_{\xi_i^-}^\bullet(U,S)$. In order to prove that
	the equality of the lemma holds when applied to $((\alpha_2,\alpha_1)\otimes(\beta_2,\beta_1))$,
	we can assume that $(\alpha_i)_{|T}=\id$ and $(\beta_i)_{|S}=\id$, since the morphisms
	involved in the equality commute with the right action of $\CS_M(Z)$.

	We have
	$$\tau\circ (f_i\otimes f_j)(\alpha_i\otimes\beta_j)=$$
	\begin{align*}
		&=
	\tau\Bigl(\bigl((\alpha_i)_{\xi_i^-(-1)}\cdot[\xi_1^-(-1)\to\xi_i^-(-1)]\boxtimes\id\bigr)\otimes
	\bigl(\id \boxtimes(\beta_j)_{\xi_j^-(-1)}\cdot[\xi_1^-(-1)\to\xi_j^-(-1)]\bigr)\Bigr)\\
		&=d_{i,j}\bigl((\beta_j)_{\xi_j^-(-1)}\cdot[\xi_1^-(-1)\to\xi_j^-(-1)]\boxtimes\id\bigr)\otimes
	\bigl(\id \boxtimes(\alpha_i)_{\xi_i^-(-1)}\cdot[\xi_1^-(-1)\to\xi_i^-(-1)]\bigr)
	\end{align*}
	where
	$d_{2,1}=0$, $d_{1,2}=1$ and
	$d_{i,i}=1$ if $\bigl((\alpha_i)_{\xi_i^-(-1)}\cdot [\xi_i^-(-2)\to\xi_i^-(-1)]
	\boxtimes (\beta_i)_{\xi_i^-(-1)}\bigr)\cdot\tau\neq 0$ and $d_{i,i}=0$ otherwise.
	We deduce that the lemma holds when applied to $((\alpha_2,\alpha_1)\otimes(\beta_2,\beta_1))$,
	hence it holds in general.
\end{proof}

\printindex
\printindex[ter]

\end{document}